\documentclass[11pt]{article}

\usepackage[margin=1in]{geometry}
\usepackage{amsmath,amsthm,amssymb,amsfonts,float,bbm}
\usepackage{listings}
\usepackage{pdfpages}
\usepackage{mathtools,slashed}
\usepackage{tikz}
\usetikzlibrary{patterns}
\usepackage{pgfplots}

\makeatletter
\newcommand{\subjclass}[2][2010]{%
  \let\@oldtitle\@title%
  \gdef\@title{\@oldtitle\footnotetext{#1 \emph{Mathematics subject classification.} #2}}%
}
\makeatother

\usepackage[colorlinks=true, pdfstartview=FitV, linkcolor=blue,citecolor=blue, urlcolor=blue]{hyperref}

\usepackage[abbrev,lite,nobysame]{amsrefs}
\usepackage{color}

\usepackage{esint}

\usepackage{mathtools,enumitem,mathrsfs}

\usepackage[compact]{titlesec}

\usepackage{comment}

\newcommand{\ep}{\epsilon}
\newcommand{\eps}{\epsilon}

\newcommand{\grad}{\nabla}

\newcommand{\norm}[1]{\left|\left| #1 \right|\right|}
\newcommand{\abs}[1]{\left| #1 \right|}
\newcommand{\set}[1]{\left\{ #1 \right\}}
\newcommand{\brak}[1]{\left\langle #1 \right\rangle} 

\newcommand{\R}{\mathbb{R}}

\newcommand{\dee}{\mathrm{d}}

\DeclareMathOperator{\supp}{\mathrm{supp}}

\DeclareMathOperator{\curl}{\mathrm{curl}}

\usepackage{bbm}

\newcommand{\p}{\partial}

\newtheorem{theorem}{Theorem}[section]
\newtheorem{proposition}[theorem]{Proposition}
\newtheorem{corollary}[theorem]{Corollary}
\newtheorem{lemma}[theorem]{Lemma}
\newtheorem*{lemma*}{Lemma}
\newtheorem{claim}[theorem]{Claim}

\theoremstyle{definition}
\newtheorem{definition}[theorem]{Definition}
\newtheorem{remark}[theorem]{Remark}

\def\jacob #1{\textcolor{red}{#1}}

\def\nn{\nonumber\\}
\def\paren#1{\left(#1\right)}
\def\abs#1{\left|#1\right|}
\newcommand{\n}{\ensuremath{\nonumber}}
\newcommand{\pa}{\ensuremath{\partial}}

\newcommand{\siming}[1]{\textcolor{blue}{#1}}

\def\bd{\Gamma_{0}}
\def\bb{\bold{a}}

\newcommand{\nnorm}[1]{\left\lVert #1\right\rVert}
\newcommand{\enorm}[1]{\left\lVert #1\right\rVert_{L^2}}
\newcommand{\brk}[1]{\left\langle #1\right\rangle}
\def\ss{{s^{-1}}}
\def\sss{{s}}

\newcommand{\vyi}{{\frac{1}{v_y}}}

\def\gm{{(\gamma)}}
\def\stf{\phi}
\def\stm{\phi}

\newcommand{\pav}{\overline{\partial_v}}
\newcommand{\myb}[1]{\textcolor{blue}{ #1 }}
\newcommand{\myr}[1]{\textcolor{red}{#1}}
\newcommand{\wt}{\widetilde}
\newcommand{\al}{\alpha}

\newcommand{\de}{\Delta}
\newcommand{\lan}{\langle}
\newcommand{\ran}{\rangle}
\newcommand{\lf}{\left}
\newcommand{\rg}{\right}

\def\cd{\mathcal{D}}

\newcommand{\phe}{\phi^{(E)}}

\newcommand{\mf}{\mathfrak}

\def\ztp#1{\mathring{#1}}
\def\bhqn{\ztp{\overline{H}}_{n}}
\def\bhqnn{\ztp{\overline{H}}_{1}}
\def\hqn{\ztp{H}_{n}}

\def\gqn{\ztp{G}_{n}}

\newcommand{\ww}{w}

\allowdisplaybreaks

\numberwithin{equation}{section}

\begin{document}

\title{Uniform Inviscid Damping and Inviscid Limit \\ of the 2D Navier-Stokes equation with \\ Navier Boundary Conditions}
\author{ Jacob Bedrossian\thanks{\footnotesize Department of Mathematics, University of California, Los Angeles, CA 90095, USA \href{mailto:jacob@math.ucla.edu}{\texttt{jacob@math.ucla.edu}}} \and Siming He\thanks{Department of Mathematics, University of South Carolina, Columbia, SC 29208, USA \href{mailto:siming@mailbox.sc.edu}{\texttt{siming@mailbox.sc.edu}}} \and Sameer Iyer\thanks{Department of Mathematics, University of California, Davis, Davis, CA 95616, USA \href{mailto:sameer@math.ucdavis.edu}{\texttt{sameer@math.ucdavis.edu}}} \and Fei Wang\thanks{School of Mathematical Sciences, CMA-Shanghai, Shanghai Jiao Tong University, Shanghai, China \href{mailto:fwang256@sjtu.edu.cn}{\texttt{fwang256@sjtu.edu.cn}}}}

\maketitle

\begin{abstract}
We consider the 2D, incompressible Navier-Stokes equations near the Couette flow, $\omega^{(NS)} = 1 + \eps \omega$, set on the channel $\mathbb{T} \times [-1, 1]$, supplemented with Navier boundary conditions on the perturbation, $\omega|_{y = \pm 1} = 0$. We are simultaneously interested in two asymptotic regimes that are classical in hydrodynamic stability: the long time, $t \rightarrow \infty$, stability of background shear flows, and the inviscid limit, $\nu \rightarrow 0$ in the presence of boundaries. Given small ($\eps \ll 1$, but independent of $\nu$) Gevrey 2- datum, $\omega_0^{(\nu)}(x, y)$, that is supported away from the boundaries $y = \pm 1$, we prove the following results: 
\begin{align*}
&  \|\omega^{(\nu)}(t) - \fint \omega^{(\nu)}(t) \dee x \|_{L^2} \lesssim \eps e^{-\delta \nu^{1/3} t}, & \text{(Enhanced Dissipation)} \\
& \brak{t}\norm{u_1^{(\nu)}(t) - \fint u_1^{(\nu)}(t) \dee x}_{L^2} + \brak{t}^2\norm{u_2^{(\nu)}(t)}_{L^2}  \lesssim \eps e^{-\delta \nu^{1/3} t}, & \text{(Inviscid Damping)} \\ 
&\| \omega^{(\nu)} - \omega^{(0)} \|_{L^\infty} \lesssim \eps \nu t^{3+\eta}, \quad\quad t \lesssim \nu^{-1/(3+\eta)} & \text{(Long-time Inviscid Limit)}
\end{align*}
This is the first nonlinear asymptotic stability result of its type, which combines three important physical phenomena at the nonlinear level: inviscid damping, enhanced dissipation, and long-time inviscid limit in the presence of boundaries. The techniques we develop represent a major departure from prior works on nonlinear inviscid damping as physical space techniques necessarily play a central role. In this paper, we focus on the primary nonlinear result, while tools for handling the linearized parabolic and elliptic equations are developed in our separate, companion work. 
\end{abstract}

\setcounter{tocdepth}{2}
{\small\tableofcontents}


\section{Introduction}

We are considering the 2D Navier-Stokes equations set on the spatial domain $\mathbb{T} \times (-1,1)$:
\begin{subequations} \label{eq:FullNSIntro}
\begin{align} 
  & \partial_t v + (v \cdot \grad)v + \grad p = \nu \Delta v, \\
  & \grad \cdot v = 0,\quad  v_2(t,x,\pm 1) = 0, \quad
   \partial_y v_1(t,x,\pm 1) = 1,\\ 
  &v(t=0,x,y)=v_{in}(x,y). 
\end{align}
\end{subequations}
This problem corresponds to studying a fluid with Navier-type boundary conditions and applying a fixed force (rather than a fixed velocity) which slides the top and bottom boundaries in opposite directions.
See for example \cite{MasmoudiStRaymond03} for derivations of Navier-type boundary conditions from kinetic theory.
Writing the vorticity $\Omega = \partial_y v_1 - \partial_x v_2 = 1 + \omega$, we obtain the perturbation equations
\begin{subequations}
\begin{align} \label{eq1}
&\pa_t \omega + y \pa_x \omega + u \cdot \nabla \omega - \nu \Delta \omega = 0 \qquad (x, y) \in \mathbb{T} \times (-1, 1), \\
& u = \begin{pmatrix} -\partial_y \\ \partial_x \end{pmatrix} \psi \\ & -\Delta \psi = \omega, \;\psi(t,x,\pm 1) = 0, \\ 
&\omega|_{t = 0} = \omega_{in}(x, y), \\ \label{eq1c}
&\omega|_{y = \pm 1} = 0. 
\end{align}
\end{subequations}
We make the following assumptions on the initial datum (see below for details):
\begin{align}
&\text{Support: } \text{supp}(\omega_{in}) = \Omega_0 \subset \mathbb{T} \times (-\frac 1 4, \frac 1 4) \label{cond:supp} \\
&\text{Regularity: } \norm{e^{\Lambda_0 \abs{\grad}^{r}} \omega_{in}}_{L^2} \leq \eps, \quad r > 1/2, \label{cond:reg}
\end{align}
(the choice $1/4$ is irrelevant and could be replaced with any number $< 1$).

We are interested in the simultaneous limits $\nu \to 0$ and $t \to \infty$, and in particular, we are interested
in being able to prove the inviscid limit, i.e. the convergence of viscous solutions to inviscid solutions, uniformly on time-scales which diverge as $\nu \to 0$ (namely $0 < t \lesssim \nu^{-1/3+\delta}$ for any $\delta > 0$, which is nearly optimal).
To our knowledge, this is the first result of its kind for the Navier-Stokes equations even for Navier boundary conditions. Moreover, we capture precise asymptotics in both $t \rightarrow \infty$ and $\nu \rightarrow 0$.  

\vspace{2 mm}

\noindent \textbf{Dynamics as $t \rightarrow \infty$:} The long time behavior of Navier-Stokes near Couette features two important physical phenomena, \textbf{inviscid damping} and \textbf{enhanced dissipation}. While these were known linear stability mechanisms classically by Orr \cite{Orr1907} and Kelvin \cite{Kelvin1887}, proving they persist at the nonlinear level for the Euler and Navier-Stokes equations was achieved only relatively recently. 
The inviscid case ($\nu = 0$) in the channel with no-penetration conditions was studied by Ionescu and Jia  in \cite{HI20}. 
Therein, they prove that for $\eps$ sufficiently small (under conditions \eqref{cond:supp} and \eqref{cond:reg}), the vorticity $\omega$ weakly converges back to an $x$-independent final state and the velocity inviscid damps back to a shear flow as it does in $\mathbb T \times \mathbb R$ (treated earlier in \cite{BM13}). 
On $\mathbb T \times \mathbb R$, this inviscid damping was proved uniformly in $\nu$, along with the additional effect of enhanced dissipation in \cite{BMV14}.
In this paper, we prove this latter result on the channel with Navier boundary conditions and moreover that the viscous solutions converge strongly to the inviscid solutions uniformly over time-scales $0 < t \lesssim \nu^{-1/(3+\eta)}$. 

\vspace{2 mm}

\noindent \textbf{Dynamics as $\nu \rightarrow 0$:} Passage to the \textbf{inviscid limit} as $\nu \rightarrow 0$ for the Navier-Stokes equations in the presence of boundaries with no-slip conditions is notoriously challenging due to the formation of boundary layers as $\nu \to 0$ which involve large concentrations of vorticity in very thin layers along the boundary  (see discussions in e.g. \cite{SammartinoCaflisch98a} and \cite{SammartinoCaflisch98b}).
Navier boundary conditions are easier to treat in that, at worst, any singularities at the boundary should appear only in the higher derivatives of the vorticity, not in the vorticity itself. 
As such, with Navier boundary conditions, the inviscid limit is known on fixed time-scales in bounded domains \cite{MasRou12}.
Nevertheless, the boundary and the viscosity together will pose a significant difficulty for us over longer time-scales $1 \ll t \ll \nu^{-1/3}$ (i.e. not so long that the viscosity becomes dominant), as we explain in more detail below. 

\vspace{2 mm}

\noindent \textbf{Simultaneous Dynamics $t \rightarrow \infty, \nu \rightarrow 0$:}
While viscosity $\nu$ determines the very long-time dynamics, it can act as a \textit{destabilizing effect} in the presence of boundaries, at least with no-slip boundary conditions (referred to as `Heisenberg’s viscous destabilization'). 
Moreover, the methods for the two asymptotic regimes appear to be completely disparate:  inviscid damping and enhanced dissipation require subtle Fourier side analysis in order to capture at the nonlinear level, whereas passing to the inviscid limit with boundaries requires capturing spatially localized quantities. The central effort in our work is to develop a framework to unify these methods. 


\subsection{Statement of Main Theorem}

We now state our main result.
We obtain estimates in much stronger norms than those stated in Theorem \ref{thm:main}, at least for $t \lesssim \nu^{-1/3-\delta}$ for some small $\delta > 0$; see Section \ref{sec:IL} for the stronger versions of the results of Theorem \ref{thm:main}. 
Note that the dynamics of the inviscid $\nu = 0$ problem were characterized in \cite{HI20} (in particular, the inviscid damping). 
Parts (i) and (ii) were proved for $\mathbb T \times \mathbb R$ in \cite{BMV14}, however, Part (iii) is new even without a boundary. 

In what follows, for any function $h \in L^2$, denote the Fourier transform in $x$ as $h_k(t,y) = \fint_{\mathbb T} h(t,x,y)e^{-ikx} \dee x $ and denote the projection to the zero mode $P_0 h := h_0$ and non-zero modes $P_{\neq} h := h - h_0$.


\begin{theorem} \label{thm:main}
  For all  $r \in (\frac{1}{2},1]$ and all $\Lambda_0 > 0 $, there are $\eps_0, \nu_0 > 0$ such that for all $0 < \eps < \eps_0$ and all $\omega_{in}$ satisfying \eqref{cond:supp} and \eqref{cond:reg}, the following holds for all $\nu \in (0,\nu_0)$ (with all constants independent of $\nu, t,$ and $\eps$):
\begin{itemize}
\item[(i)] uniform-in-$\nu$ inviscid damping and enhanced dissipation: $\exists \delta > 0$ (independent of $\nu, t, \eps$) such that  
\begin{align}
  \norm{P_{\neq} \omega(t)}_{L^2} +  \norm{P_{\neq} \brak{t} u_1(t)}_{L^2} + \norm{P_{\neq} \brak{t}^2 u_2(t)}_{L^2} \lesssim \eps e^{-\delta \nu^{1/3} t}; \label{ineq:NeqDecay}
\end{align}
\item[(ii)] long-time convergence back to equilibrium 
\begin{align*}
\norm{P_0 \omega(t)}_{L^2} + \norm{P_0 u_1(t)}_{L^2} \lesssim \eps e^{-\delta \nu t}; 
\end{align*}
\item[(iii)]  for $\nu \geq 0$, if we denote the corresponding solutions $\set{\omega^\nu}_{\nu \geq 0}$, for all $\eta > 0$, $\exists C \geq 1$ such that for $t < \nu^{-1/(3+\eta)}$, there holds 
\begin{align}
\norm{\mathbf{1}_{\abs{y} < 1/2}P_0(\omega^{\nu}(t) - \omega^{0}(t))}_{L^\infty} & \leq C\eps \nu t^{2+\eta}, \\ 
\norm{\mathbf{1}_{\abs{y} < 1/2}(\omega^{\nu}(t) - \omega^{0}(t))}_{L^\infty} & \leq C\eps \nu t^{3+\eta}, \label{ineq:LTIL} \\ 
\norm{\mathbf{1}_{\abs{y} \geq 1/2}(\omega^{\nu}(t) - \omega^{0}(t))}_{L^\infty} &  \leq C\eps e^{-\nu^{-1/8}}.  
\end{align}
\end{itemize} 
\end{theorem}

\begin{remark}
Comparing the enhanced dissipation with the results of \cite{HI20} shows that $t \ll \nu^{-1/3}$ is the optimal range of times one can prove the inviscid limit. 
\end{remark}

\begin{remark}
We obtain the same quantitative inviscid limit in a stronger norm than $L^\infty$ (however the higher regularity is adapted to the transport operator, rather than traditional Sobolev spaces). 
\end{remark}

\subsection{Further Comments on Existing Literature}

We organize the relevant literature in the following categories.

\vspace{2 mm}

\noindent \textbf{Nonlinear Inviscid Damping (and uniform-in-$\nu$ results):} There are only a few works on nonlinear inviscid damping (and uniform in $\nu$ damping), all in recent years. The first such result for the Euler equations $(\nu = 0)$ near the Couette flow on $\mathbb{T} \times \mathbb{R}$ was obtained in \cite{BM13} for Gevrey-2+ perturbations. The work of Ionescu-Jia, \cite{HI20}, treated the Euler equations near Couette on the channel domain $\mathbb{T} \times [-1,1]$, and with the sharp Gevrey-2 datum. The work \cite{BMV14} treated the Navier-Stokes equations and obtained uniform in $\nu$ inviscid damping and enhanced dissipation near Couette on $\mathbb{T} \times \mathbb{R}$. Nonlinear inviscid damping has also been obtained for general monotonic shear flows recently in Ionescu-Jia, \cite{JI20}, and Masmoudi-Zhao, \cite{MZ20} (see e.g. \cite{Zillinger2017,WZZ19,WZZ20,Jia2020,J20,WZZ18,ISJ22,J23,beekie2024uniform,ChenWeiZhang23linear,GNRS20,CZEW20,chen2024enhanced,BH20,WZ19,CLWZ20,almog2021stability} and the references therein for works on the linearized problem). 
The work \cite{DM23} shows that, in some sense at least, the Gevrey-2 regularity is required for these kinds of results. 
Of course, the question of inviscid limit from Navier-Stokes can be most naturally asked for the Couette flow (and the Poiseuille flow) because these flows are solutions to the Navier-Stokes equations as well as Euler.
In all these aforementioned works, Fourier analysis can still play the central role in understanding the nonlinear difficulties due to the fact that the assumptions ensure the perturbation vorticity remains compactly supported away from the boundary.
The previous results on nonlinear inviscid damping all depend heavily on Fourier analysis methods, while here physical space methods must also play a central role in our analysis.

\vspace{2 mm}

\noindent \textbf{Transition Threshold:} For understanding the limit $t \to \infty$, a standard type of dynamical problem which has been recently studied frequently is the `transition' or `stability' threshold, i.e. for a given pair of norms $X,Y$, determine the smallest power of $\gamma$ such that $\norm{u_{in}}_X \lesssim \nu^\gamma$ implies $\norm{u}_Y \ll 1$, where the norm $Y$ is chosen to ensure the dynamics match the linearized dynamics at least to some degree, e.g. enhanced dissipation and inviscid damping estimates such as \eqref{ineq:NeqDecay}. See \cite{BGM_Bull19} and discussions within.
In the case of $\mathbb T \times \mathbb R$ with Sobolev initial data, see \cite{BVW16} which proved first $\gamma \geq 1/2$ and then $\gamma \leq 1/3$ for sufficiently regular data \cite{MW19,WZ23} and $\gamma \leq 1/2$ for rougher data \cite{MW20}; see \cite{LiMasmoudiZhao22b} for some results on sharpness of the lower regularity results. 
For 3D results in Sobolev spaces without boundaries, see for example \cite{BGM15II,BGM15III,WZ21}. 
The work~\cite{BMV14} was able to obtain an uniform in $\nu$ threshold $\gamma=0$ in  Gevrey-$2+$, which is sharp, in some sense, in view of~ \cite{DM23}, consistent with the result of the present paper.

In the presence of a boundary, Fourier analysis cannot be applied, at least not directly, to study the transition threshold problem. Using resolvent estimates, Chen, Li, Wei, and Zhang addressed the problem with Dirichlet (no-slip) boundary conditions in a channel with Sobolev initial data, with the stability threshold $\gamma \leq 1/2$ in 2D \cite{CLWZ20} and $\gamma = 1$ in 3D \cite{ChenWeiZhang20}.
Recently, to extend Theorem~\ref{thm:main} beyond the time interval $[0, \nu^{-1/3-}]$ to infinity, the authors of the present paper developed an energy based method to obtain the stability threshold $\gamma \leq 1/2$ \cite{BHIW23} with Navier boundary conditions, which has subsequently been improved to $\leq 1/3$ \cite{WZ23}. 

\vspace{2 mm}

\noindent \textbf{Inviscid Limit with Boundaries:} For the dynamical problem as $\nu\rightarrow0$ under the Navier boundary condition of general type
$
	\partial_{y} u_1 = \alpha u_1,
$
a derivation of the boundary layer expansion is obtained in~\cite{InfSue11}. In~\cite{MasRou12}, Masmoudi and Rousset verified this boundary layer expansion, proving that the inviscid limit holds for short time in Sobolev spaces. We refer the readers to \cite{Bar72, CloMikrob98, IftPla06, Pad14, Ngu19} for more work in this direction.
In the case of no-slip Dirichlet boundary conditions, due to the more significant mismatch of the boundary conditions between the Navier-Stokes and Euler equations, a stronger boundary layer (of size $\nu^{-1/2}$) appears. In this respect, a classical result of Sammartino and Caflisch \cite{SammartinoCaflisch98a,SammartinoCaflisch98b}
showed the inviscid limit holds for analytic initial data (again on short times). 
 Subsequently, Maekawa \cite{Maekawa14} proved  that the same is true  if the initial data $u_0$ has  vorticity $\omega_0=\curl u_0$ that is compactly supported in the interior of the domain.
The readers may find different proof of these two results in~\cite{NguyenNguyen18, FeiTaoZhang18, WanWanZha17}. Recently, the fourth author and his collaborators bridged the results 
 of Sammartino-Caflisch and Maekawa by proving that the inviscid limit holds for data which are analytic in a vicinity of the boundary, 
 and Sobolev regular otherwise~\cite{KukNguVicWan22,	KukVicWan19,  KukVicWan22,	Wan20}.
 We would also like to mention some important progress  of inviscid limit in Gevrey class~\cite{GerMaeMas18, GerMaeMas2020, CheWuZha22}.
 
 \vspace{2 mm}

\noindent \textbf{Mixing via Vector-Field Method:} It is becoming a somewhat standard practice to apply vector field estimates to derive mixing/inviscid damping in fluid and kinetic systems.
At the simplest level, Coti Zelati introduced in \cite{CotiZelati20} a vector field method merged with hypocoercivity to derive a uniform-in-$\nu$ mixing estimate for passive scalars subject to a strictly monotone shear flow in the high P\'{e}clet number regime.
In the paper \cite{WeiZhangZhu20}, Wei, Zhang, and Zhu applied the vector field method to derive inviscid damping in the $\beta$-plane equations.
Chaturvedi, Luk, and Nguyen applied a vector field method together with hypocoercivity to study the Vlasov-Poisson-Landau equation in the weakly collisional regime, deducing both enhanced dissipation and Landau damping (the analogue of inviscid damping) \cite{ChaturvediLukNguyen23} with a transition threshold of $\gamma \leq 1/3$ (expected to be sharp also in kinetic theory). 
Later, a similar vector field method was successfully applied to the non-cutoff Boltzmann equations on $\mathbb R^d$ \cite{BedrossianCotiZelatiDolce22}.
The present work seems to be the first that uses an nonlinear-adapted vector field to study mixing.

\subsection{Non-technical Overview of Main Ideas} \label{sec:NonTechIntro}

In Section \ref{sec:Outline} we outline the main steps of the proof as lemmas and spend the rest of the paper proving these.
Here we include a relatively non-technical overview of the main ideas in order to ease the reading of the technical details.

The proof is done in two main steps. The first, and primary step, is to reach the time-scales $t \approx \nu^{-1/3-\delta}$ for some small $\delta >0$.
By this time, the enhanced dissipation has eliminated essentially all of the $x$-dependence of the solution and one can apply our recent work \cite{BHIW23} to obtain global-in-time estimates. 
To prove step 1, we employ a nonlinear bootstrap argument which couples several crucial estimates in order to close up to times $t \approx \nu^{-1/3-\delta}$.

\vspace{2 mm}

\noindent \underline{\textsc{Inviscid Damping and Enhanced Dissipation in Interior}} The starting point for our analysis is to recall the mechanisms of inviscid damping and enhanced dissipation on $\mathbb{T} \times \mathbb{R}$ (primarily from \cite{BM13} and \cite{BMV14}, respectively).
This will motivate the types of bounds we can hope to prove in the ``interior" and the methods we employ to handle the nonlinear interactions therein.
One of the primary difficulties in \cite{BM13} is the presence of `nonlinear echo resonances' which necessitates the Gevrey-2 regularity requirement and motivates the delicate Fourier multiplier-based norms employed in \cite{BM13}.
These norms (or slight variations thereof) have appeared in all of the positive results on inviscid damping which do not require smallness with respect to $\nu$ \cite{BM13, HI20, JI20, MZ20, BMV14} and we will employ such norms as well (with some small adjustments). 
As discussed in \cite{BM13,BMV14}, in order to employ such norms, we first need a change-of-coordinates which depends on the solution itself in order to eliminate the background (time-varying) shear flow and make it possible to propagate vital regularity and use the Fourier methods effectively.
Hence we define the new coordinates $(x,y) \mapsto (z,v)$ as (formal for now), 
\begin{subequations}
\begin{align} \label{eq:v}
&z(t,x,y) = x - tv(t,y) \\ 
&\pa_t (t (v - y)) - \nu t \pa_y^2 v = \fint_{\mathbb T} u_1(t,x,  y)\dee x \\ 
&\pa_y v|_{y = \pm 1} = 1  \\ 
& f(t, z(t,x,y), v(t,y)) = \omega(t, x, y); 
\end{align}
\end{subequations}
this is the same change of variables used in \cite{BMV14} (except for the Neumann boundary conditions). 
Due the boundaries, in order to effectively employ a variation of the Fourier multiplier norms, we study the ``interior vorticity profile'' $\chi^I f$ for a Gevrey-smooth cutoff $\chi^I$ supported in the interval $(-\frac{31}{40},\frac{31}{40})$.
The main norm we propagate on the interior is of the form $\norm{\mathfrak{A}(\chi^I f)}_{L^2}$, where $\mathfrak{A}$ is a small adjustment of the multiplier norms used in \cite{BM13, HI20, BMV14} . 
There are also a corresponding set of ``interior'' estimates on the coordinate system itself, performed on energy estimates on $v_y(t,y) := \partial_y v(t,y)$ expressed in the $(z,v)$ variables and several related auxiliary variables (referred to as ``coordinate system variables''). 
See Sections \ref{sec:Outline} and \ref{sec:interior} for more details. 

\vspace{2 mm}

\noindent \underline{\textsc{Strong Localization in the Exterior}}
We now turn to the treatment of the ``exterior'', which comprises most of the new difficulties in this work. 
The first kind of property we will need to quantify is the localization of the vorticity.
Indeed, due to the viscosity, the vorticity will instantly spread everywhere in the channel, however, we will need to very strongly localize it close to the initial support.
We do so with the following (qualitatively optimal) weight
\begin{align}
W(t, y) := \frac{(|y|-1/4-L\ep\arctan (t))_+^2}{K\nu(1+t)},
\end{align}
where $K, L$ are parameters that we choose.
Aside from the $L\eps\arctan(t)$, this Gaussian weight is motivated by the form of the fundamental solution of the heat equation. The extra $L \eps \arctan(t)$ is there to account for the spreading of the vorticity support due to the nonlinear transport; the fact that it remains bounded uniformly in $\nu$ is due to the time-integrable inviscid damping of $u_2(t)$. 

\vspace{2 mm}

\noindent \underline{\textsc{Adapted Vector Fields}}
One can check that that the profile derivative $\p_v f(t, z, v)$ coincides with the following vector field $\Gamma_t$ acting on the vorticity $\omega$, 
\begin{align}
\Gamma_t := \frac{1}{\partial_y v(t,y)} \partial_y + t \partial_x. 
\end{align}
The significance of this derivative $\Gamma_t$ is that it is an \emph{approximate commuting vector field} for the time-dependent advection diffusion associated with the background shear: 
\begin{align}
 [\Gamma_t, \left(\partial_t + (y +  u_1)\partial_x - \nu \Delta\right)] \omega =   {G}\pav \Gamma_t \omega  - \nu \pav(v_y^2-1)\pav^2\omega,\label{eq:GammatComm}
\end{align}
where $\pav = \frac{1}{\p_y v(t, y)} \partial_y$ and where $G = \partial_t v - \nu \partial_{yy} v$ is an auxiliary function that is formally expected to decay like $O(\brak{t}^{-2})$ (although we only prove $O(\brak{t}^{-2+O(\eps)})$) and $P_0 q(t,y) = \fint q(t,x,y) \dee x$; see Section \ref{sec:coord:FEI} for more details. 

While the natural functional setting for the analysis of the interior vorticity is in Gevrey spaces in the $(z, v)$ variable on $f$, we adapt the functional setting for the total vorticity by working in Gevrey spaces of the two commuting vector fields. Moreover, due to the potentially hazardous boundary contributions at $y = \pm 1$, we define the co-normal weight function
\begin{align} 
q(y) =  \begin{cases} 99(y+1), & -1 < y < -1 + \frac{1}{100} \\
1 &  -1 + \frac{1}{50} < y < 1-\frac{1}{50} \\  
99(1-y),  & 1 - \frac{1}{100} < y < 1.
\end{cases}
\end{align}
We quantify the Gevrey regularity in the ``exterior'' using $W q^n \Gamma^n_t \partial_x^m \omega$ (note that this is different from $(q\Gamma_t)^n \partial_x^m$). 

\vspace{2 mm}

\noindent \underline{\textsc{Pseudo-Gevrey Regularity:}} It is useful at this stage to imagine the following domain decomposition, indicated in Figure \ref{fig:regions} which indicates the two primary regimes: the interior region where inviscid damping (and enhanced dissipation) dominates the dynamics and the exterior region where the vorticity should be  small due to the spatial localization. 
\begin{figure}[h] \label{fig:regions}
\centering
\begin{tikzpicture}
\draw[ultra thick, -] (-7,0) -- (7,0);
\node [below] at (5, 0) {$\frac34$};
\node [below] at (6, 0) {$\frac{31}{40}$};
\node [below] at (-5, 0) {$-\frac34$};
\node [below] at (-6, 0) {$-\frac{31}{40}$};
\node [below] at (2, 0) {$\frac{11}{40}$};
\node [below] at (-2, 0) {$-\frac{11}{40}$};
\node [below] at (1, 0) {$\frac{1}{4}$};
\node [below] at (-1, 0) {$-\frac{1}{4}$};
\node [below] at (7.5, 0) {$1$};
\node [below] at (-7.5, 0) {$-1$};
\node [below] at (3, 0) {\textcolor{red}{$\frac38$}};
\node [below] at (4, 0) {\textcolor{red}{$\frac12$}};
\node [below] at (-3, 0) {\textcolor{red}{$-\frac38$}};
\node [below] at (-4, 0) {\textcolor{red}{$-\frac12$}};
\draw[scale=1, domain=-1:1, smooth, variable=\x, thick, blue] plot ({\x}, {0});
\draw[->, scale=1, domain=1:3, smooth, variable=\x, thick, blue] plot ({\x}, {\x-1});
\draw[<-, scale=1, domain=-3:-1, smooth, variable=\x, thick, blue] plot ({\x}, {-\x-1});
\node [below] at (3, 1) {\textcolor{blue}{$(y - \frac14)_+$}};
\draw[color=orange, domain = -5:5, thick] plot (\x,{1}) node[right] {$\chi_I(y)$};
\draw[color=orange, domain = 6:7, thick] plot (\x,{0});
\draw[color=orange, domain = 5:6, thick] plot (\x,{6-\x});
\draw[color=orange, domain = -7:-6, thick] plot (\x,{0});
\draw[color=orange, domain = -6:-5, thick] plot (\x,{6+\x});
\draw[ultra thick, dashed, -, blue] (-7, -1) -- (-2,-1);
\draw[ultra thick, dashed, -, blue] (2, -1) -- (7,-1);
\node[right] at (7,-1) {\textcolor{blue}{weight large}};
\draw[ultra thick, dashed, -, orange] (-5, -1.5) -- (5,-1.5);
\node[right] at (5,-1.5) {\textcolor{orange}{inviscid damping}};
\draw[ultra thick, dashed, -, red] (3, -2) -- (4,-2);
\draw[ultra thick, dashed, -, red] (-4, -2) -- (-3,-2);
\node[right] at (4,-2) {\textcolor{red}{transition region}};
\end{tikzpicture}
\end{figure}
It would seem from this picture that two complementary, overlapping regimes have been identified in which two complementary mechanisms are available to control the nonlinearity, and therefore we might imagine that after estimating suitable commutators with the cutoffs that a scheme built on this picture ``closes". However, this is grossly misleading due to necessary \textit{regularity discrepancies.} Let us explain as follows. There are three apparent constraints on the regularities: 
\begin{itemize}
\item[(1)] Due to commutator terms from $\chi^I$ (schematically of the form $\text{derivatives}(\chi^I) \omega$) which appear as source terms for the ``interior equation" written on $\chi^I \omega$, we must have \textit{at least equal regularity} of $\omega$ in the support of $*\chi^I)'$, namely on $(-\frac{31}{40}, \frac34) \cup (\frac34, \frac{31}{40})$, as we hope to propagate on $\chi^I \omega$ (i.e. the interior).  
\item[(2)] In the exterior, we need to pay regularity in order to obtain enough inviscid damping of the stream function on the support of (at least) $(-\frac{11}{40}, \frac{11}{40})$, we need to propagate \textit{less regularity} in the exterior norm of $\omega$ than on the interior norm for $\chi^I \omega$ on $(-\frac{11}{40}, \frac{11}{40})$.
\item[(3)] Due to the complicated frequency-by-frequency nature of the Fourier multiplier norm employed in the interior (see \eqref{frak:A}) which has no physical-side analogue, it is not possible to work in the \textit{same (or even comparable) regularity} in both the exterior and interior analysis. This means one will always need to lose \emph{Gevrey-scale} regularity when comparing interior and exterior quantities.   
\end{itemize}
These regularity constraints motivate our use of a ``Pseudo-Gevrey regularity"\footnote{The terminology ``Pseudo-Gevrey regularity" arises from the fact that the Gevrey exponent depends on the location of $y$.} which gradually transitions (as a function of $y$) from
$L^2$ regularity on $(-\frac{11}{40},\frac{11}{40})$ to nearly analytic on $\frac{3}{4} < \abs{y} < 1-\frac{1}{100}$ (in particular, a stronger Gevrey regularity than that used on the initial data, and so we are using an instant gain of Gevrey regularity away from the initial support). 
In the region colored red in Figure \ref{fig:regions}, we introduce an infinite cascade of cutoff functions which are tailored specifically to the Gevrey regularity gain we need to obtain.
Hence the interior norm operators at scale $\frac{1}{2} \leq r$ while the exterior norm on $\frac{3}{4} < \abs{y} < 1 - \frac{1}{100}$ provides $r < \frac{1}{s} < 1$ regularity (with $r$ close to $1/2$ and $\frac{1}{s}$ close to $1$).

More specifically, we define a sequence of parameters as follows: 
\begin{figure}[h]
\centering
\begin{tikzpicture}
\draw[ultra thick, -, red, dashed] (-3,0) -- (7,0);
\node[right] at (7, 0) {\textcolor{red}{transition region}}; 
\node[below] at (-3, 0) {$\frac38 = x_1$}; 
\node[below] at (7, 0) {$\frac12$}; 
\node[below] at (-1, 0) {$x_n$};
\node[below] at (0, 0) {$y_n$};
\node[below] at (3, 0) {$x_{n+1}$};
\node[below] at (4, 0) {$y_{n+1}$};
\draw[scale=1, domain=-3:-1, smooth, variable=\x, thick, blue] plot ({\x}, {0});
\draw[scale=1, domain=-1:0, smooth, variable=\x, thick, blue] plot ({\x}, {\x+1});
\draw[scale=1, domain=0:7, smooth, variable=\x, thick, blue] plot ({\x}, {1});
\node[above] at (-.1,.3) {\textcolor{blue}{$\chi_{n}$}};
\draw[scale=1, domain=-3:3, smooth, variable=\x, thick, cyan] plot ({\x}, {0});
\draw[scale=1, domain=3:4, smooth, variable=\x, thick, cyan] plot ({\x}, {\x-3});
\draw[scale=1, domain=4:7, smooth, variable=\x, thick, cyan] plot ({\x}, {1});
\node[above] at (4.1,.3) {\textcolor{cyan}{$\chi_{n+1}$}};
\end{tikzpicture}
\end{figure}
where we choose 
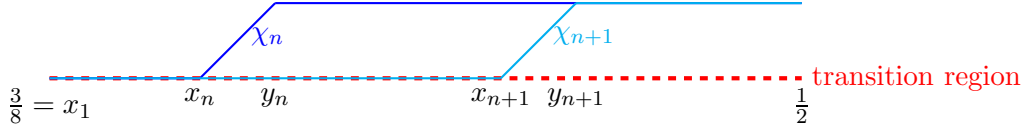
\begin{align}
x_{n+1} = x_n + \frac{c_\sigma}{n^{1+\sigma}}, \qquad y_{n} = x_n + \frac{c_\sigma}{100 n^{1+\sigma}}, \qquad 0 <  \sigma < s - 1.
\end{align}
Then we measure regularity using $W \chi_{n+m} \partial_x^m q^n \Gamma_t^n \omega$, and refer to this as `Pseudo-Gevrey regularity'. 




\vspace{2 mm}

\noindent \underline{\textsc{ Anisotropic Loss of Regularity after $t = \nu^{-\frac13-}$}:} Even with adapted vector fields, the linear commutators and nonlinear effects seem to slowly deplete the available regularity in the exterior, i.e. we cannot propagate a uniformly controlled radius of graded Gevrey regularity.
For this reason, we shed radius of regularity like $O(\brak{t}^{-1})$, however, only in the $y$-direction. 
However, an interpolation argument using the localization gained from the very strong weight $W$ will ensure that, at least for $t \lesssim \nu^{-1/3-\delta}$ for a sufficiently small $\delta$, the scheme closes. We point the reader to Lemma \ref{lem:ExtToInt} for a more precise statement, but in essence, we prove schematically: 
\begin{align*}
\Big(\text{Pseudo-Gevrey $s$} + \text{Weight $W$}  - \text{Anisotropic Shedding}  \Big) \Big|_{\substack{\{(x,  y) \in \text{supp}((\chi^I)')\} \cap \\ \{t \lesssim \nu^{-1/3-\delta}\}}}> \text{Gevrey $\frac1r$}.
\end{align*}
In other words, if we propagate Pseudo-Gevrey regularity close to analytic $(s \sim 1)$ coupled with the large spatial weight, then on timescales $t \lesssim \nu^{-1/3-\delta}$ the anisotropic shedding brings us down to just above Gevrey-2 regularity after we spatially localize to $(\chi^I)'$ (which is enough to merge with the interior inviscid damping analysis): 

\vspace{2 mm}

\noindent \underline{\textsc{Six Families of Energy Functionals:}} The main nonlinear argument to reach the time-scale $\approx \nu^{-1/3-\delta}$ involves six families of vorticity-coordinate system norms (which themselves have several variants and sub-families). These six families can be thought of as belonging to four general groups: interior vorticity norms (these are built from the basic form $\norm{\mathfrak{A}(\chi^I f)}_{L^2}$, where $\mathfrak{A}$ is a multiplier similar to those used in \cite{BM13,BMV14}, interior coordinate system norms, exterior vorticity norms, and exterior coordinate system norms. Out of these, the interior vorticity and interior coordinate system norms are controlled using the techniques from \cite{BM13,BMV14,HI20} and are almost ``black-boxed" from those previous works. The vast majority of the effort comes in controlling the exterior vorticity norms, the exterior coordinate system norms, and the elliptic norms to be discussed further below. Moreover, it is obviously crucial to ensure that all of these aspects work together due to the requirement in this work to switch between frequency space and physical space approaches. We briefly describe some aspects of the various norms now; the reader can turn to Section \ref{sec:norms} for the precise presentation of each of the seven families.


\begin{itemize}

\item[] \underline{Exterior Vorticity Functionals:} First, we wish to control the vorticity (and auxiliary `coordinate' functions arising from the commutators involving $\Gamma_t$) in the exterior in a manner that builds in several of the features we have discussed above. To do so, we design functionals of the general form 
\begin{align*}
\sum_{m = 0}^\infty \sum_{n =0}^\infty \sum_{\iota \in \{0, 1\}} \bold{b}_{m,n}^2 \| e^W \chi_{m + n} q^n (\sqrt{\nu} \nabla)^{\iota} \p_x^m \Gamma_t^n \omega \|_{L^2}^2, \qquad \bold{b}_{m,n} \approx \frac{1}{\brak{t}^{n+1}}\left(\frac{c(t)^{m+n}}{(n+m)!}\right)^s.
\end{align*}
The variants used on the auxiliary coordinate functions are somewhat more subtle; see \eqref{fei:fei:fei} -- \eqref{fei:fei:fei:3} and the discussion thereafter for more details.

\item[] \underline{Sobolev Cloud Functionals:} A major effort in this paper is to control precisely contributions from the nonlinear terms (for example, Section \ref{sec:q} -- \ref{sec:Tri}). It turns out we need a type of intermediary norm that contains weighted information (which are in the exterior norms) but at lower regularities which enable us to avoid inserting the cutoffs $\chi_{n + m}$. As such, these ``cloud" functionals are a type of interpolant between the interior Fourier analytic norms and the exterior norms. To motivate the idea behind controlling this norm, we consider a typical ``quasilinear term" of the form $\nabla^\perp \phi^{(I)}_{\neq 0} \cdot \nabla \omega_{m,n}$, where $\omega_{m,n} = \p_x^m \Gamma^n \omega$, and we have written $\phi^{(I)}$ to denote the contribution of the stream function that comes from the interior (in particular, this is not small in $\nu$). We note that in our scheme, in addition to commuting $\p_x, \Gamma$ derivatives in the Gevrey scale, we also commute one factor of $\sqrt{\nu} \nabla$. We therefore need to consider the following type of trilinear term (among many others):
\begin{align*}
&\text{Trilinear Term of Commutator-Quasilinear Type:= } \\
&\qquad \sum_{m, n} \bold{b}_{m,n}^2 \langle (\sqrt{\nu} \nabla \nabla^\perp \phi^{(I)}) \cdot \nabla \omega_{m,n}, \sqrt{\nu} \nabla \omega_{m,n} \chi_{m+n}^2 q^{2n} e^{2W} \rangle.
\end{align*}
We temporarily restrict $m + n \le 10$, as we are discussing a ``low-frequency" issue. One basic strategy to control this term would be to put $ \sqrt{\nu} \nabla \omega_{m,n}$ in $L^2$, the stream function in $L^\infty$, and $\nabla \omega_{m,n}$ in $L^2$. However, even though we are controlling an $H^1$ based norm (for fixed $m,n$), we cannot put $\nabla \omega_{m,n}$ in $L^2$ due to the lack of $\nu^{\frac12}$ factors. Therefore, we opt to rewrite $\nabla \omega_{m,n}$ as a linear combination of $\omega_{m+1,n}$ and $\omega_{m,n+1}$. 

Normally, the price one has to pay for ``losing a derivative" in this manner is in the Gevrey coefficient $\bold{b}_{m,n}$. However, since we are restricting to low frequencies ($m + n \le 10$) this ends up being a uniform loss. On the other hand, the major loss in our scheme is the required shift of the cutoff function: 
\begin{align*}
\chi_{m + n} \mapsto \chi_{m + n + 1},
\end{align*}
which leaves behind the gap $\chi_{m + n} - \chi_{m + n - 1}$. Now, for $m + n \le 10$, these ``gap-regions" are localized to the interior. Therefore, we need a special ``Cloud" norm\footnote{The terminology ``Cloud" comes from imagining the support of this low-frequency norm as a cloud that encompasses all of these problematic gaps.} which lives at the Sobolev regularity, but importantly covers these ``gap-regions" (also including the substantial weight).  

\item[] \underline{Sobolev boundary functionals:} For several reasons, we need to propagate $H^4$ regularity of the vorticity all the way to the boundary, i.e. without the co-normal weights $q$. The two main motivations are (A) to obtain pointwise estimates on the auxiliary coordinate functions and $P_0\omega$ (i.e. to control the background shear flow); and (B) to obtain the type of $H^{1}$ estimates required to apply our previous work \cite{BHIW23} to extend the results from $\nu^{-1/3-\delta}$ to infinity. These norms require a delicate coupling between Sobolev norms (defined in \eqref{hn:1} -- \eqref{com:bulk:1}) in the bulk and Sobolev trace norms (defined in \eqref{zero:changes:1} -- \eqref{line:90}). Moreover, our trace analysis is quite nontrivial by itself, as it requires deriving PDEs for the vorticity evaluated at the boundaries $\{y = \pm 1\}$, and subsequently extracting maximum regularity bounds on these PDEs as quantified by our ``large" norms, \eqref{zero:changes:10} -- \eqref{zero:changes:large:1}.

\end{itemize}

\vspace{2 mm}

\noindent \underline{\textsc{Four Families of Elliptic Functionals:}}  A crucial, and difficult, part of the proof is obtaining suitable estimates on the streamfunction for the nonlinear terms. Intuitively, we think of the elliptic functionals as the ``arbiters" of the proof: they communicate simultaneously with the interior and exterior mechanisms and couple them together. More precisely, one can think that there are four possible ``directions of communication": 

\begin{itemize}

\item[a.] \textit{Exterior Vorticity-to-Exterior Streamfunction}: This is the most delicate part of the analysis. We introduce a family of quantities ($S\omega$ and $J\psi$) to capture various commutator terms that appear when implementing the Gevrey estimates of the streamfunction. Thanks to this ``ICC scheme" developed in the companion paper \cite{BHIW24b}, one is able to close the estimate. We skip further details here.  

\item[b.] \textit{Interior Vorticity-to-Exterior Streamfunction}:
In our analysis, the interior vorticity is in a low regularity space, and the exterior streamfunction is in a high regularity space. The method to upgrade the regularity is to identify a positive spatial gap between the support of the interior vorticity and the exterior region under consideration. This gap between the ``source" and the ``target" allows the Green's kernel in the Biot-Savart law to exhibit a smoothing effect, enabling us to upgrade the smoothness in the exterior region.  

\item[c.] \textit{Exterior Vorticity-to-Interior Streamfunction}: For the interior streamfunction generated by exterior vorticity, we use the fact that the detailed pseudo-Gevrey estimates yield that the exterior vorticity is small in a high Gevrey space. Hence, a standard elliptic estimate yields that they generate smooth streamfunction in the interior. 

\item[d.] \textit{Interior Vorticity-to-Interior Streamfunction}: For the interior-to-interior elliptic estimates, they are derived in a similar fashion as in the literature \cite{BM13,BMV14, HI20}. We will refer the readers  to these sources for further information. 

\end{itemize}

In particular, in order to execute these elliptic bounds favorably, we actually require \textit{two separate decompositions} of the stream function: 
\begin{align}
\psi = \psi^{(I)} + \psi^{(E)}, \qquad \psi = \phi^{(I)} + \phi^{(E)}.
\end{align}
Roughly, these four quantities above correspond to the four directions of communication in the bullet points above; see \eqref{Intr_psi_I} -- \eqref{ell:E} for precise formulations of these four stream functions. 

\vspace{2 mm}

\noindent \underline{\textsc{Linear Analysis \& the ICC-Method:}}
Clearly, the Navier-Stokes equations will not commute with the weights $q(y)^n \Gamma^n_t \partial_x^m$. 
A technical challenge in our exterior estimates is to analyze and estimate these commutators precisely, as they are (primarily) \textit{linear} terms (and hence the smallness of the data is irrelevant). To treat these contributions we developed a special scheme detailed in our separate, recent paper \cite{BHIW24b}, which we refer to as ``the ICC-Method''. It capitalizes on the precise algebraic structure of these commutators, both for the parabolic vorticity equation and the elliptic equation for the streamfunction (necessary to estimate the nonlinear terms in the exterior). While the implementation of this ICC-Method is the primary content of our companion paper, \cite{BHIW24b}, we quickly indicate the purpose of the technique. Indeed, due to linear commutators between $\nu$, our co-normal weight, $q(y)$, and the adapted vector-field $\Gamma_t$, we need to control linear error terms of the type: 
\begin{align*}
\lf(\frac{m+n}{q}\rg)^a \pa_y^b\pa_x^c \lf(\pa_x^m q^n\Gamma_t^n\omega\rg)\sim S^{(a,b,c)}_{m,n} \omega.
\end{align*}
These types of error terms also appear with certain factors of $\nu$. They represent a delicate linear trade-off between the quantities $(\nu, t,  m,n, q(y)^{-1})$. It turns out that these quantities can be estimated inductively, which is the content of our ICC-Method. However, the scheme is highly nontrivial and, in particular, requires the anisotropic $\varphi(t)^{n+1}$ weight in our energy functionals.   
\vspace{2 mm}

\noindent \underline{\textsc{Analysis of Trilinear Terms and the ``Quasiproduct" Framework}} Upon designing our elaborate functional framework, a major challenge becomes to control the corresponding nonlinear terms. To caricature the analysis, let $U \in \{E, I\}$. We write the nonlinear term $\p_x^m \Gamma^n \{ \nabla^\perp \phi^{(U)}_{\neq 0} \cdot \nabla \omega \}$ as follows
\begin{align} \n
 &\underbrace{\sum_{m' = 0}^{m} \sum_{n' = 0}^n \mathbbm{1}_{n' + m' < n + m} \binom{m}{m'} \binom{n}{n'} ( \Gamma \phi^{(U, \neq 0)}_{m-m', n-n'} \p_x \omega_{m', n'} - \p_x \phi^{(U, \neq 0)}_{m-m', n-n'} \Gamma \omega_{m', n'})}_{\mathcal{T}[\phi^{(U)}, \omega]} \\ \label{joey:2LK}
& +  \nabla^\perp \phi^{(U)}_{\neq 0} \cdot \nabla \p_x^m \Gamma^n \omega. 
\end{align}
We want to pull out the quasilinear term (the second term in \eqref{joey:2LK}, where all the derivatives land on $\nabla \omega$) because normally we integrate this term by parts when paired with $\p_x^m \Gamma^n \omega$ in an inner product. On the other hand, the first term in \eqref{joey:2LK} behaves like a product due to its origin, but it is not formally a product because we've pulled out the quasilinear term. Therefore, we call this bilinear operator $\mathcal{T}[\cdot, \cdot]$ ``quasiproduct", and we prove many types of bilinear estimates for these operators of the type 
\begin{align*}
&\sum_{m = 0}^\infty \sum_{n =0}^\infty \sum_{\iota \in \{0, 1\}} \bold{b}_{m,n}^2 \| e^W \chi_{m + n} q^n (\sqrt{\nu} \nabla)^{\iota} \mathcal{T}[ \phi^{(U)}, \omega ]_{m,n}\|_{L^2}^2 \\
\lesssim &\| \text{Elliptic Norm on $\phi^{(U)}$} \| \times \| \text{Vorticity Norm on $\omega$} \|.
\end{align*}
Such estimates are technically involved: the quadruple infinite sum appearing above requires many decompositions, combinatorial estimates, and various optimal choices of distributing the weights among the trilinear factors $\phi \times \omega \times \omega$ which also depend on the regime $U \in \{E, I\}$ of the streamfunction.  This extensive machinery is developed in Section \ref{sec:q} of our paper and is employed in Section \ref{sec:Tri} to control all trilinear contributions.

\vspace{2 mm}

\noindent \underline{\textsc{Fractional derivative loss and decay in the anisotropic regularity:}} There are three quantities used in this paper to characterize the 
nonlinear change of coordinate system/the adapted vector field, namely $H, \overline H,$ and $G$.
As in \cite{BM13,BMV14,HI20} the PDEs these solve have some hidden regularity loss.
The following toy model captures the nature of the regularity loss involved in how these coefficients couple with the vorticity
\begin{align*}
	\partial_{t} \omega &= H, \\
	\partial_{t} H &= \overline H, \\
	\partial_{t} \overline H &= \partial_{y} \omega.
\end{align*}
Informally, one could expect a $G^3$ loss, which is less than what is needed here (specifically $G^{2+}$).
To simplify the proof, we assume that $H$ has the same regularity as $\omega$, resulting in $\overline H$ having $1/2$-derivative less regularity than $\omega$.
Using a fractional weight  $(n+1)^{2s-2}$ with shifted Gevrey coefficients $\bold{a}_{n+1}$, we account for this regularity discrepancy in physical space. Also because of this discrepancy in regularity and the shedding of Gevrey radius,  the coordinate functional $\overline H$ loses $\brak{t}^{-1/2}$ decay, which would prevent us from obtaining the stability of $H$. A delicate balance among the localization property of the coordinate system, regularity discrepancy, and the fast decay in the interior region is needed to ensure that we achieve the appropriate decay of $\overline H$.

\vspace{2 mm}

\noindent \underline{\textsc{Timescale $t > \nu^{-\frac13-}$:}} As alluded to above, in order to apply the results of our previous work \cite{BHIW23}, we need to obtain some amount of Sobolev regularity up to the boundary, i.e. without co-normal weights. 
In particular, on $P_0 \omega$, we need to demonstrate a uniform estimate in $H^4$ regularity and on $\omega - P_0 \omega$ we need sufficient smallness in $H^1$. 
Neither of these estimates are directly implied by the norms we used thus far.
Hence, we need a new set of estimates specifically to obtain such regularity. The main issue here is ensuring that the boundary conditions on the higher derivatives can be treated and properly taking advantage of the above estimates to obtain the required smallness.

Ultimately, what we prove is the following: at some time $t_\ast = \nu^{-1/3-\zeta}$ for some $0 < \zeta \ll 1$, we have the Sobolev estimates
\begin{align*}
& \norm{\fint \omega(t_\ast,x',\cdot) \dee x'}_{H^4} \lesssim \eps \\
& \norm{\omega(t_\ast) - \fint \omega(t,x',\cdot) \dee x'}_{H^1}  \lesssim \eps e^{-\delta \nu^{-\zeta}}.  
\end{align*}
By using $\omega(t_\ast)$ as the initial data in \cite{BHIW23}, the results therein imply (possibly after adjusting $\eps_0$ in Theorem \ref{thm:main})
\begin{align*}
& \norm{P_0\omega(t)}_{H^4} \lesssim \eps e^{-\delta' \nu t} \\
& \norm{P_{\neq}\omega(t)}_{L^2} \lesssim \eps e^{-\delta \nu^{-\zeta}} e^{-\delta'\nu^{1/3}(t-t_\ast)} \\
& \norm{P_{\neq}u(t)}_{L^2} \lesssim \eps e^{-\delta \nu^{-\zeta}} e^{-\delta'\nu^{1/3}(t-t_\ast)}. 
\end{align*}
Choosing $\delta'' < \frac{1}{2}\min(\delta,\delta')$ and using $e^{-\delta'' \nu^{1/3} t} \lesssim_{N,\zeta} \brak{t}^{-N}$ for all $t > t_\ast$ and $N > 0$, we obtain parts (i) and (ii) in Theorem \ref{thm:main}. 

\vspace{2 mm}

\noindent \underline{\textsc{The Inviscid Limit:}} Part (iii) of Theorem \ref{thm:main} is proved in Section \ref{sec:IL}. 
In order to prove the inviscid limit, we need to estimate $\tilde{\omega} = \omega^{\nu} - \omega^0$.
However, the PDE this quantity satisfies contains not only nonlinear terms but also terms which are basically the Euler equations linearized around $\omega^0$.
Since $\omega^0$ is neither small in $\nu$ nor decaying, these linear terms will destroy any naive attempt to propagate the inviscid limit beyond time-scales like $\abs{\log \eps}^{-1}$.
Instead, a much more subtle approach must be taken, requiring estimates in Gevrey regularity using (again) the full power of the Fourier multiplier methods of \cite{BM13,HI20}.
The regularity is measured in the coordinate system defined by $\omega^0$, however $\omega^{\nu}$ (and hence $\tilde{\omega}$) do not a priori satisfy good estimates in this coordinate system.
In the ``exterior'', the strong weight and flexibility in regularity index is used to obtain such smallness estimates at the cost of restricting the time-scale to $t < \nu^{-1/3+\gamma}$ for an arbitrary $\gamma > 0$.
In the ``interior'', a Gevrey-2 (nonlinear) energy estimate is made which resembles that used in the proof of parts (i) and (ii). This estimate hinges on proving that the background shear $y + P_0 u_1^{\nu} = y + \fint u_1^\nu(t,x,\cdot) \dee x$ converges faster to $y + P_0 u_1^{0}$ (as a function of $\nu$) than the full vorticity using the special structure of the Euler equations (analogous to the structure that ensured $O(\brak{t}^{-2})$ convergence in \cite{BM13,HI20}).
See Section \ref{sec:IL} for more details.


\section{Bootstraps, Norms, and Outline of Main Proof}
\label{sec:Outline}
In this section we provide a technical overview of the proof Theorem \ref{thm:main}.

\subsection{Preparatory Objects \& Definitions}

\vspace{2 mm}

\noindent \underline{\textsc{Coordinate Transform \& Adapted Vector Fields:}}
As discussed above, in \cite{BM13,HI20} and in our interior estimates, the Fourier transform is a necessary tool, and so adapting the coordinate system itself to the background shear flow is a necessity (as opposed to purely using vector fields). 
We  recall how this is done in more detail than above.

In \eqref{eq:v} above we defined the coordinate change $(x,y) \mapsto (z,v)$ and the profile 
\begin{align*}
f(t,z(t,x,y),v(t,y)) = \omega(t,x,y),
\end{align*}
(technically speaking, for now this is formal, as one needs to justify that this coordinate change remains well-defined as one goes forward; see e.g. \cite{BM13,BMV14} for the details of how to do this). Furthermore, we denote the streamfunction profile as $\psi$ and the velocity profile as $U$. 

We now define the following associated quantities, which are taken from \cite{BMV14}:
\begin{subequations}
\begin{align}
C(t, v(t, y)) = &v(t, y) - y ,\\
v'(t, v(t, y)) = & \partial_y v(t, y).
\end{align} 
\end{subequations}
Of paramount importance are the three quantities that we bootstrap control over (introduced in \cite{BM13,BMV14}): 
\begin{subequations}
\begin{align} \label{defn:g}
g(t, v) = & [\pa_t v] - \nu v'' \\ \label{defn:barh}
\overline{h}(t, v) = & v' \pa_v g \\ \label{defn:h}
h(t, v) = & v' - 1,
\end{align}
\end{subequations}
where we are denoting $[\partial_t v](t,v(t,y)) = \partial_tv(t,y)$.
These auxiliary quantities solve PDEs of their own, which we delineate now (see \cite{BMV14} for the derivations):
\begin{subequations}
\begin{align}
&\pa_t g + \frac2t g + g \pa_v g  - \nu |v'|^2 \pa_v^2 g = - \frac{v'}{t} P_0\left( \nabla^\perp \psi_{\neq 0} \cdot \nabla U \right) \\
&\partial_t h + \frac{2}{t} h + g \partial_v h - \nu |v'|^2 \pa_v^2 h = \overline{h} \\
&\partial_t \overline{h} + \frac{2}{t} \overline{h} - \nu |v'|^2 \pa_v^2 \overline{h} = -\frac{v'}{t}P_0\left(U_{\neq} \cdot \grad f\right). 
\end{align}
\end{subequations}
We refer to the auxiliary unknowns $g,h,\bar{h}$ (and their analogues $G,H,\overline{H}$ below) as `coordinate system unknowns' (not to be confused with the complex conjugate $\bar{f}$ of a function ${f}$ also used in this paper). 
The vorticity profile satisfies the PDE
\begin{align*}
\partial_t f + g \partial_v f + v'\grad^\perp \psi_{\neq} \cdot \grad f = \nu |v'|^2 (\pa_v - t\partial_x)^2 f =: \nu \widetilde{\Delta_t} f. 
\end{align*}
See Section \ref{sec:interior} for more information about the stream function profile.

In the $(z,v)$ coordinate system, $\partial_z,\partial_v$ derivatives are the natural way to measure regularity.
These correspond directly to the derivatives
\begin{align*}
\partial_x, \quad \Gamma_t = \frac{1}{\partial_y v(t,y)} \partial_y + t \partial_x,
\end{align*}
and hence regularity in the $(z,v)$ coordinates matches regularity with respect to the vector fields $(\partial_x, \Gamma_t)$. 
As pointed out in \eqref{eq:GammatComm}, the significance of this derivative $\Gamma_t$ is that it is an \emph{approximate commuting vector field} for the time-dependent advection diffusion associated with the background shear. Therefore, the same quantities we need to control in order to use regularity in $(z,v)$ must be estimated in their corresponding $\partial_x,\Gamma_t$-regularity representations. 
With these vector fields, for the exterior fluid, we work in the $(x, y)$ coordinate and rewrite the equation~\eqref{eq1} as 
\begin{align}
	\label{M1a}
	&\pa_t \omega + (y + P_0u_1(t, y))\partial_{x}\omega  =\nu \de \omega-\nabla^{\perp}\stm_{\neq}\cdot\nabla \omega,\\ \n
	&\omega(t,y=\pm 1)=0,\qquad \
	\omega(t=0,y)=\omega_{\text{in}}(y),
	\end{align}
where $u^1_0(t, y)$ is the zero mode of the horizontal velocity. 
The auxiliary quantities $g, \overline h,$ and $h$ are given by $g(t, v(t,y)) = G(t, y)$, $\overline{h}(t, v(t,y)) = \overline{H}(t, y)$, and $h(t, v(t,y)) = H(t, y)$, where
\begin{subequations}
\begin{align} \label{defn:g:cap}
G(t, y)  =&  \pa_t v - \nu \pa_y^2 v = \frac{ P_0u_1(t, y) - (v-y) }{t}\\ \label{defn:barh:cap}
\overline{H}(t, y) = & \pa_y G = \frac{\pa_y P_0u_1(t, y) - (v_y - 1)}{t} = \frac{-P_0\omega(t, y) - H(t, y)}{t} \\ \label{defn:h:cap}
H(t, y) = & \pa_y v(t, y) - 1 = - P_0\omega - t \overline{H}, 
\end{align}
\end{subequations}
where we have used $\omega_0(t, y) =- \pa_y u^1_0(t, y)$.
The PDE that $G$ satisfies is  
\begin{subequations}
	\label{eq:G:main}
\begin{align}
&\pa_t G + \frac2t G - \nu \pa_y^2 G = -\frac{1}{t} \left( u_{\neq} \cdot \nabla u_1 \right)_0 \\ 
&\pa_y G|_{y = \pm 1} = 0.  
\end{align}
\end{subequations}
Above, the Neumann boundary condition on $G$ is forced by the corresponding Neumann boundary condition imposed on $v$, \eqref{eq:v}.
The bulk of our analysis on these quantities will be on $\overline{H} = \pa_y G$ (as in \cite{BM13} as well), on which we can obtain strong localization. This quantity satisfies 
\begin{subequations}
	\label{eq:bar:H:main}
\begin{align} 
&\pa_t \overline{H} + \frac2t \overline{H}  - \nu \pa_y^2 \overline{H}= -\frac{1}{t} \left( u_{\neq} \cdot \nabla \omega \right)_0 \\ 
&\overline{H}(t, \pm 1) =  0,
\end{align}
\end{subequations}
We also need the equation for $H(t, y)$, which is simply 
\begin{subequations}
	\label{eq:H:main}
\begin{align}
&\pa_t H - \nu \pa_y^2 H = \overline{H} \\
&H(t, \pm1) = 0.
\end{align}
\end{subequations}
These seven unknowns, together with the streamfunctions in the two different coordinate systems,  are the quantities studied in the nonlinear bootstrap.

\vspace{2 mm}

\noindent \underline{\textsc{Gevrey Indices:}} We have two essential Gevrey indices which we record here at the outset: 
\begin{align}
\begin{aligned} \label{pgiL1}
\frac12 < r < 1 = & \text{ Interior Gevrey $1/r$ Index} \\
1 < s = &  \text{ Exterior Pseudo-Gevrey $s$ Index.}
\end{aligned}
\end{align}
Our convention of letting $r < 1$ and $s > 1$ may appear strange, but it is motivated by wanting to simplify the notations in the actual analysis: as the interior analysis takes place on the Fourier side, the number $r$ appear most often in the multipliers. On the other hand, the exterior analysis occurs on the physical space side, and therefore the number $s$ most often appears in the estimates. When we compare the two regularities, it will always be either $1/r$ compared to $s$, or $1/s$ compared to $r$. More precisely, it turns out due to Lemma \ref{lem:ExtToInt} below, we will take the regularity gap 
\begin{align}
\frac12 < r < \frac{51}{100}, \qquad \frac{4}{5} < \frac{1}{s} < 1. 
\end{align}
\vspace{2 mm}

\noindent \underline{\textsc{Cutoff Functions:}}
A delicate part of our analysis is the presence of several different localization regions. 

In order to measure the interior profile $f$ (and the coordinate auxiliaries $g,h,\overline{h}$)
appropriately, we define the ``interior" localization using the following cut-off function
\begin{align} \label{chi:I:def}
\chi^I(\zeta) = \begin{cases} 1 \qquad \zeta \in  (-\frac34, \frac34) \\ 0 \qquad \zeta \in (-\frac{31}{40}, \frac{31}{40})^c,  \end{cases} 
\end{align}
and a similar cutoff with slightly expanded support, $\chi^I_e$ which is $\equiv 1$ on the $\abs{\zeta} \leq \frac{31}{40}$ (i.e. the support of $\chi^I$) and vanishes on $\abs{\zeta} \geq \frac{33}{40}$, i.e.
\begin{align}
\chi^I_e(\zeta) := \lf\{\begin{array}{cc}1,\quad&|\zeta|\leq \frac{31}{40},\\ 0,\quad&|\zeta|\in[\frac{33}{40},1].\end{array}\rg.\label{chiIe}
\end{align}


We will now detail the cutoffs used to gradually transition from $L^2$ regularity on the interior to nearly analytic regularity in the transition zone beyond the support of the initial data. Associated to the Gevrey index $s > 1$, we define the following associated parameters, which will arise in our analysis:
\begin{align} \label{s:prime}
0 <  \sigma < s - 1, \qquad \sigma_{\ast} := (s-1) - \sigma > 0. 
\end{align}
Given the parameter $\sigma$ defined above, we design now a sequence of cut-off functions $\chi_n(y)$, $n \in \mathbb{N}$. We first choose the following parameters: 
\begin{align}
x_1 = &\frac38, \\
x_{n+1} = &x_n + \frac{c_{\sigma}}{n^{1+\sigma}}, \qquad n \ge 1, \\
y_n = & x_n + \frac{c_{\sigma}}{100 n^{1+\sigma}}. 
\end{align}
Above, the constant $c_{\sigma}$ is chosen so that $c_{\sigma} \sum_{n \ge 1} \frac{1}{n^{1+\sigma}} < \frac18$. We now define cut-offs, $\chi_n(y)$, adapted to these scales:
\begin{align}\label{chi}
\chi_n(y) = \begin{cases} 0, \qquad - x_n < y < x_n \\ 1, \qquad \{- 1 < y < -y_n\} \cup \{y_n < y < 1\} \end{cases} \qquad n \ge 1.
\end{align}
The following properties follow from our choice of sequences, $\{x_n\}, \{y_n\}$: 
\begin{subequations}
\begin{align} \label{chi:prop:1}
\cap_{n = 1}^{\infty} \{ \chi_n = 1 \} \supset & (-1, -\frac12) \cup (\frac12, 1), \\ \label{chi:prop:2}
\text{supp}(\nabla^k \chi_{n+1}) \subset & \{ \chi_n = 1 \} \qquad \text{ for all } k \in \mathbb{N}, n \in \mathbb{N}, \\ \label{chi:prop:3}
|\pa_y^{j} \chi_n| \lesssim &n^{j (1 + \sigma)} \chi_{n-1}.
\end{align}
\end{subequations}
We also define a Gevrey-$s$ cutoff $\wt \chi_1$ satisfying
\begin{align}\label{wt_chi_intro}
 \wt \chi_1(\xi)= \begin{cases}1,\qquad |\xi|\geq \frac{3}{8}-\frac{1}{80},\\ 
0,\qquad |\xi|\leq \frac{3}{8}-\frac{1}{40}, \\ \text{monotone}, \quad \text{others}, \end{cases} \qquad \wt \chi_1^\mathfrak{c}(\xi)=1-\wt \chi_1(\xi).
\end{align}
This cutoff function is a fattened version $\chi_1$-cutoff \eqref{chi} above.
We highlight that the $\wt \chi_1$-cutoff function is defined in the new coordinate system $v$, so the variable $\xi$ takes values in $v(t,-1)$ and $v(t,1). $ 

There are several additional cutoffs defined in the internal workings of the proof, however all are straightforward variations of these choices. 

\vspace{2 mm}

\noindent \underline{\textsc{Weight Functions:}} We now define our co-normal weight function as a smooth, strictly positive function which satisfies the following,  
\begin{align} \label{q:defn}
q(y) =  \begin{cases} 99(y-1), & -1 < y < -1 + \frac{1}{100} 
\\ 1 &  -1 + \frac{1}{50} < y < 1-\frac{1}{50} \\  
99(1-y),  & 1 - \frac{1}{100} < y < 1.
\end{cases}
\end{align}
We now define our vorticity localization weight-function:
\begin{align} \label{defndW}
W(t, y)=\frac{(|y|-1/4-L\ep\arctan (t))_+^2}{K\nu(1+t)}.
\end{align}
Above, $K,L$ are large but universal constants. It can be checked (for example, see [Lemma 4.12, \cite{BHIW24b}]), that the function $W$ satisfies the following estimate:
\begin{align}\label{W_prop}
    \pa_t W+\frac{1}{8}K\nu|\pa_y W|^2  \leq& -\frac{(|y|-1/4-L\ep\arctan(t))_+^2}{2K\nu(1+t)^2}  -\frac{2L\epsilon(|y|-1/4-L\ep\arctan(t))_+}{K\nu(1+t)(1+t^2)} \leq  0.
\end{align}
Alternatively, there exists a constant $C$, independent of the parameters $K, L$ such that 
\begin{subequations}
\begin{align} \label{wdot:est:a}
\nu |\p_y W|^2 \le & - \frac{C}{K} \pa_t{W}; \\  \label{wdot:est:b}
\frac{\eps}{(1 + t)^2}|\p_y W| \le & - \frac{C}{L}\pa_t{W}.
\end{align}
\end{subequations}

%
We now define a few temporal weights that will serve as base exterior Gevrey radii as follows:
\begin{align} \label{Gev:la}
\lambda(t) := &\lambda_0( 1 + (1 + t)^{-\frac{1}{100}}), \\  
\widehat{\lambda}(t)  := & 2\lambda(t),  \\ \label{varphi}
\varphi(t) := & \frac{1}{(1 + t^2)^{\frac12}}.
\end{align} 
Here $\lambda_0$  is an order one small number chosen depending on the size of the solution ($\ep$) and the Gevrey index $s$. We note that $\lambda(t)$ will appear in the definition of the interior vorticity norm through the $\mathfrak{A}$ multiplier (Section \ref{sec:interior}), as well as to measure the exterior vorticity, \eqref{ef:a}. We have chosen a consistent $\lambda$ for convenience of notation, though it is not required in the proof that the radii of the interior and exterior vorticity functionals are exactly the same.  

The following basic inequalities regarding the above defined functions will be in constant use:  
\begin{align} \label{ineq:varphi}
1 \lesssim & \langle t \rangle \frac{\dot{\varphi}(t)}{\varphi(t)}, \\ \label{lambda:yessir}
 1 \lesssim & \langle t \rangle^2 \frac{\dot{\lambda}(t)}{\lambda(t)}. 
\end{align}
Regarding \eqref{ineq:varphi}, it is important that the weight $\varphi(t)$ does not decay \textit{too fast} so as to force our exterior norms to lose too many derivatives (measured by an appropriate Gevrey index) over the critical time-scale $t \in [0, \nu^{- \frac13-\zeta})$. Therefore the choice of $\varphi(t)$ is not arbitrary: there is a balance between ensuring enough decay to counteract commutators from our adapted vector field $\Gamma$, and having enough non-degeneracy so as to not shed too much regularity (as in the crucial Lemma \ref{lem:ExtToInt}).

Subsequently, we denote our base Gevrey weights:
\begin{subequations}\label{all_B_mn}
\begin{align}\label{B_low}
B_{m,n}^{\mathrm{low}}=&\lf(\frac{2^{-(m+n)}}{(m+n)!}\rg)^{4}, 
\\ \label{Bweight}
B_{m,n}(t) & :=\lf( \frac{\lambda^{n+m}}{(m+n)!} \rg)^{s}, \\
\widehat{B}_{m,n}(t) &:=  \lf( \frac{ \widehat{\lambda}^{m+n}}{(m+n)!} \rg)^{s}.
\end{align}
\end{subequations}
Of these $B_{m,n}$ will be by far the most important: it is used to measure exterior vorticity regularity. On the other hand, $\widehat{B}_{m,n}$ only arises in one elliptic functional, namely \eqref{Friday:1} -- \eqref{Friday:2}. 

We will use the following notations to track Gevrey regularity in the exterior
\begin{align} \label{a:weight}
\bold{a}_{m, n}(t) := & B_{m, n}(t)  \varphi(t)^{1+n}, \\
\bold{a}_{n}(t):  = &\bold{a}_{0, n}, \\ \label{hat_bf_a_intro}
\widehat{\bold{a}}_{m,n}(t) := & \widehat{B}_{m,n}(t) \varphi(t)^{1 + n}.
\end{align}
 
The final (minor) set of weights is designed to slightly adjust the way constants grow as increasing number of derivatives are taken, however, we will want to quantify some of these losses in a way which does not actually lose any radius of Gevrey regularity. 
we will use are defined as follows: let $\{ \theta_n\}_{n = 0}^\infty$ be a sequence of positive numbers such that 
\begin{align}
\begin{aligned}
&\theta_{n + 1} \le \theta_n, \qquad n = 0, 1, \dots \\
&\frac{\theta_{n+1}}{\theta_n} = \delta_{\text{Drop}}, \qquad n < n_\ast \\
&\theta_n = 1, \qquad n \ge n_\ast,
\end{aligned}
\end{align}
where the parameters $n_\ast$ and $\delta_{\text{Drop}}$ will be chosen based on universal constants as part of the linearized theory, Proposition \ref{thm:main:para}. 

\vspace{2 mm}

\subsection{Six Families of Energy Functionals} \label{sec:norms}

To control the vorticity and coordinate systems, we will bootstrap control over six different families of norms:
(1) interior vorticity norms, (2) exterior vorticity norms, (3) exterior coordinate norms, (4) interior coordinate norms, (5) Sobolev boundary norms, and (6) Sobolev ``cloud" norms (explained below).
These quantities measure two of the three main types of quantities that appear in our analysis: the vorticity, $\omega$, and the `coordinate system' auxiliary unknowns $(G, H, \overline{H})$ introduced above (the third being the stream function, $\psi$). 

\vspace{2 mm}

\noindent \textsc{Interior Vorticity:}
To measure the interior vorticity we use 
\begin{align}\mathcal{E}_{\mathrm{Int}}^{\mathrm{low}}(t):=&\sum_{m+n=0}^\infty \lf(B_{m,n}^{\mathrm{low}}\|\pa_x^m \Gamma_k^n(\chi^I \omega_k)\|_{L^2}\rg)^2,\label{E_Int_low}\\
\mathcal{E}_{\text{Int}}(t) := & \norm{\mathfrak{A}(t,\grad)  (\chi^I f) }_{L^2}^2, \label{E_Int}
\end{align}
and dissipation functional 
\begin{align}
\mathcal{D}_{\text{Int}}(t) := &\nu \norm{\grad_L \mathfrak{A}(t,\grad) (\chi^I f) }_{L^2}^2,
\end{align}
and several $\mathcal{CK}$ functionals (terminology we use to denote dissipation-like terms that arise when the time-derivatives land on the norm), 
\begin{align}
\mathcal{CK}_\lambda(t) & := \dot{\lambda} \norm{\abs{\grad}^{s/2} \mathfrak{A} (\chi^I f)}_{L^2}^2 \\
\mathcal{CK}_w(t) & := \norm{\sqrt{\frac{\partial_t w}{w}} \widetilde{\mathfrak{A}} (\chi^I f)  }_{L^2}^2 \\
\mathcal{CK}_M(t) & := \norm{\sqrt{\frac{\partial_t M}{M}} \mathfrak{A} (\chi^I f) }_{L^2}^2 \\
\mathcal{CK}_{\text{Int}} & := \mathcal{CK}_\lambda+ \mathcal{CK}_w + \mathcal{CK}_M, 
\end{align}
where $w,M,\mathfrak{A},\widetilde{\mathfrak{A}}$ are Fourier multipliers defined in Section \ref{sec:interior}. 
We further observe that $\mathcal{E}_{\mathrm{Int}}^{\mathrm{low}}(t)\lesssim \mathcal{E}_{\mathrm{Int}}(t).$
\vspace{2 mm}

\noindent \textsc{Interior Coordinate System:}
We introduce the following energy functionals for the three quantities from our interior coordinate system ($A$ and $A_R$ are defined in Section \ref{sec:interior}):
\begin{align}
\mathcal{E}_{\text{Int,Coord}}^{(h)} := & \norm{A_R (\chi^Ih)}_{L^2}^2,\label{E_IntCh} \\
\mathcal{E}_{\text{Int, Coord}}^{(\overline{h})} := & \brak{t}^{2 + 2r} \norm{\brak{\partial_v}^{-r} A_R (\chi^I \bar{h} )}_{L^2}^2, \\
\mathcal{E}_{\text{Int, Coord}}^{(g)} := & \brak{t}^{4 - 2K \eps} \norm{\brak{\partial_v}^{-5} A (\chi^Ig)}_{L^2}^2, \label{E_IntCg}
\end{align}
for a fixed universal constant $K \geq 1$. 
The nonlinear estimate scheme based on this trio of norms was introduced in \cite{BM13}. 
We have the following dissipation functionals for our interior coordinate system (See Section \ref{sec:interior} for definitions of the multipliers $A$ and $A_R$), 
\begin{align}
\mathcal{D}_{\text{Int,Coord}}^{(h)} := & \nu \norm{\partial_v A_R (\chi^I h) }_{L^2}^2  \\
\mathcal{D}_{\text{Int,Coord}}^{(\overline{h})} := & \nu \brak{t}^{2+2r} \norm{\partial_v A_R (\chi^I \bar{h} ) }_{L^2}^2\\
\mathcal{D}_{\text{Int,Coord}}^{(g)} := & \nu \brak{t}^{4-2K\eps} \norm{\partial_v \brak{\partial_v}^{-5} A (\chi^I g) }_{L^2}^2, 
\end{align}
and the following $\mathcal{CK}$ interior functionals for our coordinate system:
\begin{align}
\mathcal{CK}_{\text{Int,Coord}, \lambda}^{(h)} := &\dot{\lambda} \norm{A_R (\chi^Ih)}_{L^2}^2, \\
\mathcal{CK}_{\text{Int,Coord}, w}^{(h)} := & \norm{ \sqrt{\frac{\p_t w_R}{w_R}} A_R (\chi^Ih)}_{L^2}^2, \\
\mathcal{CK}_{\text{Int, Coord}, \lambda}^{(\overline{h})} := & \dot{\lambda} \brak{t}^{2 + 2r} \norm{\brak{\partial_v}^{-r} A (\chi^I \bar{h} )}_{L^2}^2, \\
\mathcal{CK}_{\text{Int, Coord}, w}^{(\overline{h})} := &  \brak{t}^{2 + 2r} \norm{\sqrt{\frac{\p_t w}{w}}\brak{\partial_v}^{-r} A (\chi^I \bar{h} )}_{L^2}^2, \\
\mathcal{CK}_{\text{Int, Coord}, \lambda}^{(g)} := & \dot{\lambda} \brak{t}^{4 - K \eps} \norm{\brak{\partial_v}^{-5} A (\chi^Ig)}_{L^2}^2, \\
\mathcal{CK}_{\text{Int, Coord}, w}^{(g)} := &  \brak{t}^{4 - 2K \eps} \norm{\sqrt{\frac{\p_t w}{w}}\brak{\partial_v}^{-5} A (\chi^Ig)}_{L^2}^2.
\end{align}
We also define
\begin{align*}
	\mathcal{E}_{\text{Int,Coord}}&= \mathcal{E}_{\text{Int,Coord}}^{(\bar{h})}+ \mathcal{E}_{\text{Int,Coord}}^{(g)} + \frac{1}{M}\mathcal{E}_{\text{Int,Coord}}^{(h)}, \\
	\mathcal{CK}_{\text{Int,Coord}}&= \sum_{\iota \in \{\lambda, w \}}  \mathcal{CK}_{\text{Int,Coord}, \iota}^{(\bar h)}  + \sum_{\iota \in \{\lambda, w \}}  \mathcal{CK}_{\text{Int,Coord}, \iota}^{(g)} + \frac{1}{M} \sum_{\iota \in \{\lambda, w \}}  \mathcal{CK}_{\text{Int,Coord}, \iota}^{(h)}, \\
	\mathcal{D}_{\text{Int,Coord}}^{(h)} &= \mathcal{D}_{\text{Int,Coord}}^{(\bar h)} + \mathcal{D}_{\text{Int,Coord}}^{(g)}+ \frac{1}{M} \mathcal{D}_{\text{Int,Coord}}^{(h)}
\end{align*}
for some $M>0$ sufficiently large.

\vspace{2 mm}

\noindent \textsc{Exterior Vorticity:}
For the exterior vorticity, we will have three families of energy-dissipation-CK functionals: the ``$\gamma$" version corresponding to $L^2$ level of regularity, the ``$\alpha$" version corresponding to $H^1_y$ level of regularity, and the ``$\mu$" version corresponding to $H^1_x$ level of regularity. We note that these $(\gamma, \alpha, \mu)$ functionals are then propagated at the Gevrey level (summing over $(m, n)$ below). We note that it is perhaps counter-intuitive, but our Gevrey spaces \textit{do not} directly imply regularity in the usual $\p_x, \p_y$ derivatives. This is due to a variety of reasons, chief among which is the inclusion of the co-normal weights $q^n$, the pseudo-Gevrey regularity cutoffs $\chi_{m + n}$, and the use of $\Gamma_t$.
It is important to note that simply controlling the $\mathcal{E}^{(\gamma)}$ functional yields essentially no useful $H^1$ information.
First, we define
\begin{align} 
\Gamma & :=\frac{1}{v_y(t,y)}\partial_y + t\partial_x, \quad \omega_{m,n}(t,x,y) := \partial_x^m \Gamma^n \omega(t,x,y),\\ 
\Gamma_k & := \frac{1}{v_y(t,y)}\partial_y + ikt, \quad \omega_{m,n;k}(t,y):= \abs{k}^m\Gamma_{k}^n \omega_k(t,y). 
\end{align}
We now introduce the energy functionals:
\begin{align} \label{ef:a}
\mathcal{E}^{(\gamma)}(t) := & \sum_{m = 0}^\infty \sum_{n =0}^\infty \theta_n^2 \bold{a}_{m,n}^2 \| q^n\omega_{m,n} e^W \chi_{m + n}  \|_{L^2}^2, \\ \label{ef:b}
\mathcal{E}^{(\alpha)}(t) := & \sum_{m = 0}^\infty \sum_{n =0}^\infty \theta_n^2 \bold{a}_{m,n}^2 \| \sqrt{\nu} \p_y (q^n\omega_{m,n}) e^W \chi_{m + n}  \|_{L^2}^2, \\ \label{ef:c}
\mathcal{E}^{(\mu)}(t) := & \sum_{m = 0}^\infty \sum_{n =0}^\infty \theta_n^2 \bold{a}_{m,n}^2 \| \sqrt{\nu} \p_x(q^n \omega_{m,n}) e^W \chi_{m + n}  \|_{L^2}^2.
\end{align}
Correspondingly, we have the dissipation functionals:
\begin{align} \label{df:a}
\mathcal{D}^{(\gamma)}(t) := & \sum_{m = 0}^\infty \sum_{n =0}^\infty \theta_n^2 \bold{a}_{m,n}^2 \| \sqrt{\nu} \nabla (q^n\omega_{m,n}) e^W \chi_{m + n}  \|_{L^2}^2 \\ \label{df:b}
\mathcal{D}^{(\alpha)}(t) := &  \sum_{m = 0}^\infty \sum_{n =0}^\infty \theta_n^2 \bold{a}_{m,n}^2 \| \nu \p_y \nabla (q^n\omega_{m,n}) e^W \chi_{m + n}  \|_{L^2}^2 \\ \label{df:c}
\mathcal{D}^{(\mu)}(t) := & \sum_{m = 0}^\infty \sum_{n =0}^\infty  \theta_n^2 \bold{a}_{m,n}^2 \| \nu \p_x \nabla  (q^n\omega_{m,n}) e^W \chi_{m + n} \|_{L^2}^2
\end{align}
Next, we have the $\mathcal{CK}$ functionals.
In our analysis, we have three decreasing quantities: $\varphi(t)$ defined in \eqref{varphi}, the Gevrey radius $\lambda(t)$ defined in \eqref{Gev:la}, as well as our weight function $W(t, y)$ defined in \eqref{defndW}. Consequently, we have three families of $\mathcal{CK}$ terms which will be used in different ways:   
\begin{align} \label{CKgammavarphi}
\mathcal{CK}^{(\gamma; \varphi)}(t) := &  \sum_{m = 0}^\infty \sum_{n =0}^\infty (1 + n) \theta_n^2 \bold{a}_{m,n}^2 \frac{\dot{\varphi}}{\varphi} \| q^n\omega_{m,n} e^W \chi_{m + n}  \|_{L^2}^2,  \\
\mathcal{CK}^{(\alpha; \varphi)}(t) := & \sum_{m = 0}^\infty \sum_{n =0}^\infty (1 + n) \theta_n^2 \bold{a}_{m,n}^2 \frac{\dot{\varphi}}{\varphi} \| \sqrt{\nu} \p_y  (q^n\omega_{m,n}) e^W \chi_{m + n} \|_{L^2}^2, \\
\mathcal{CK}^{(\mu; \varphi)}(t) := &  \sum_{m = 0}^\infty \sum_{n =0}^\infty (1 + n) \theta_n^2 \bold{a}_{m,n}^2 \frac{\dot{\varphi}}{\varphi} \| \sqrt{\nu} \p_x (q^n\omega_{m,n}) e^W \chi_{m + n}  \|_{L^2}^2, \\
\mathcal{CK}^{(\gamma; \lambda)}(t) := &  \sum_{m = 0}^\infty \sum_{n =0}^\infty (m + n) \theta_n^2 \bold{a}_{m,n}^2 \frac{\dot{\lambda}}{\lambda} \| (q^n\omega_{m,n}) e^W \chi_{m + n}  \|_{L^2}^2,  \\
\mathcal{CK}^{(\alpha; \lambda)}(t) := & \sum_{m = 0}^\infty \sum_{n =0}^\infty (m + n) \theta_n^2 \bold{a}_{m,n}^2 \frac{\dot{\lambda}}{\lambda} \| \sqrt{\nu} \p_y (q^n\omega_{m,n}) e^W \chi_{m + n}  \|_{L^2}^2, \\
\mathcal{CK}^{(\mu; \lambda)}(t) := &  \sum_{m = 0}^\infty \sum_{n =0}^\infty (m + n) \theta_n^2 \bold{a}_{m,n}^2 \frac{\dot{\lambda}}{\lambda} \| \sqrt{\nu} \p_x (q^n\omega_{m,n}) e^W \chi_{m + n}  \|_{L^2}^2, \\
\mathcal{CK}^{(\gamma; W)}(t):=&\sum_{m = 0}^\infty \sum_{n = 0}^\infty \theta_n^2 \bold{a}_{m,n}^2\lf\|\sqrt{-\pa_t W}q^n\omega_{m,n}e^W\chi_{m+n}\rg\|_{L^2}^2.\label{CK_W_ga} \\
\mathcal{CK}^{(\al; W)}(t):=& \sum_{m = 0}^\infty \sum_{n = 0}^\infty \theta_n^2 \bold{a}_{m,n}^2\left\|\sqrt{ -\pa_t W  } \sqrt{\nu}\pa_y (q^n\omega_{m,n})e^W\chi_{m+n}\right\|_{L^2}^2,\label{CK_W_al}\\  
\mathcal{CK}^{(\mu; W)}(t):=& \sum_{m  = 0}^\infty \sum_{n = 0}^\infty \theta_n^2 \bold{a}_{m,n}^2 \lf\|\sqrt{-\pa_t W} \sqrt{\nu}\p_x(q^n\omega_{m,n})e^W\chi_{m+n}\rg\|_{L^2}^2.\label{CK_W_mu}
\end{align}
Our use of these various $\mathcal{CK}$ functionals will be featured extensively in particular when performing trilinear estimates in Section \ref{sec:Tri}. Roughly speaking, the compare $\mathcal{CK}^{(\cdot; \varphi)}$ and $\mathcal{CK}^{(\cdot; \lambda)}$ we see a trade-off between regularity and decay: $(m+n)$ is in general stronger than $(1 + n)$, but $\frac{\dot{\varphi}}{\varphi} \sim \langle t \rangle^{-1}$ which is larger than $\frac{\dot{\lambda}}{\lambda}$. On the other hand, $\mathcal{CK}^{(\cdot, W)}$ encodes more subtle spatial localization information through the quotient $\frac{\p_t W}{W}$ which is absent from both $\mathcal{CK}^{(\cdot; \varphi)}$ and $\mathcal{CK}^{(\cdot; \lambda)}$.

\vspace{2 mm}

\noindent \textsc{Exterior Coordinate System Functionals:}
To control the auxiliary quantities $G,H,\overline{H}$ we need a suitable adaptation of the scheme we applied to estimate $g,h,\bar{h}$, which requires us to introduce some fractional-derivative adjustments to the pseudo-Gevrey norms.
We will also obtain slightly weaker localization estimates on $\overline{H}$, $H$ than we have available on $\omega$ (so that we can gain in nonlinear estimates). Note that localization estimates are not available on $G$. 
For $\iota\in\{\alpha, \gamma\}$, 
we introduce the following energy functionals  on which we will write our parabolic estimates:
\begin{align*}
	\mathcal{E}_{\overline{H}}^{(\iota)} := & 
	\sum_{n\ge0} \theta_{n}^2\mathcal{E}_{\overline{H},n}^{(\iota)}, \\
	\mathcal{E}_{G}^{(\iota)} := & 
	\sum_{n\ge0} \theta_{n}^2\mathcal{E}_{G,n}^{(\iota)}, \\
	\mathcal{E}_{H}^{(\iota)} := & 
	\sum_{n\ge0} \theta_{n}^2\mathcal{E}_{H,n}^{(\iota)},
\end{align*}
where 
\begin{align} \label{fei:fei:fei}
	\mathcal{E}_{\overline{H},n}^{(\iota)} := &\nu^{i_{\iota}/2} (n+1)^{2\sss-2}  \bold{a}_{n+1}^2 \brak{t}^{3+2\ss} \|\partial_y^{i_{\iota}} \bhqn e^{W/2} \chi_n \|_{L^2}^2, \\
	\mathcal{E}_{G,n}^{(\iota)} := & \begin{cases} \langle t \rangle^{4-2K\eps_1} \|\partial_y^{i_{\iota}} G \chi_0 \|_{L^2}^2  \qquad n = 0 \\
		\bold{a}_{n}^2 \brak{t}^{3+2\ss} 
		\| \partial_y^{i_{\iota}} \gqn  e^{W/2} \chi_{n-1+i_\iota}\|_{L^2}^2 \qquad n \ge 1 \end{cases}  \\ \label{fei:fei:fei:3}
	\mathcal{E}_{H,n}^{(\iota)} := & \nu^{i_{\iota}/2} \bold{a}_n^2 \| \partial_y^{i_{\iota}} \hqn  e^{W/2} \chi_n \|_{L^2}^2.
\end{align}
with $K$ given in~\eqref{E_IntCg},  $i_{\iota}$ defined by 
\begin{align*}
	i_{\iota} = \begin{cases}
		1, & \ \ \iota = \alpha \\
		0, & \ \ \iota = \gamma. 
	\end{cases}
\end{align*}
The corresponding dissipation functionals are given as 
\begin{align*}
\mathcal{D}_{\overline{H}}^{(\iota)} & = \sum_{n\ge0} \nu^{1+i_{\iota}/2} \theta_{n}^2 (n+1)^{2\sss-2}  \bold{a}_{n+1}^2 \brak{t}^{3+2\ss}
\|  \partial_y^{1+i_{\iota}} \overline{H}_n e^{W/2} \chi_n \|_{L^2}^2, \\ 
\mathcal{D}_{H}^{(\iota)} & =  \nu \theta_{0}^2 \langle t \rangle^{4-2K\eps}\enorm{\partial_y^{1+i_{\iota}} G\chi_0}^2  + \sum_{n\ge1} 
\nu \theta_{n}^2 \bold{a}_{n}^2 \brak{t}^{3+2\ss} 
\| \partial_y^{1+i_{\iota}} G_n e^{W/2} \chi_{n-1+i_\iota}\|_{L^2}^2,  \\
\mathcal{D}_{G}^{(\iota)} & = \sum_{n\ge0}\theta_{n}^2 \bold{a}_{n}^2 \nu^{1+i_{\iota}/2} \|  \partial_y^{1+i_{\iota}} {H}_n e^{W/2} \chi_n \|_{L^2}^2,
\end{align*}
and the $\mathcal{CK}$ terms are written as
  \begin{align*}
	\mathcal{CK}_{\overline{H}}^{(\iota)} :=& \sum_{n\ge0} \theta_{n}^2 \mathcal{CK}_{\overline{H}, n}^{(\iota)}, \\
	\mathcal{CK}_{G}^{(\iota)} :=& \sum_{n\ge0} \theta_{n}^2 \mathcal{CK}_{G, n}^{(\iota)}, \\
	\mathcal{CK}_{H}^{(\iota)} :=& \sum_{n\ge0} \theta_{n}^2 \mathcal{CK}_{H, n}^{(\iota)},
\end{align*}
where
\begin{align}
	\mathcal{CK}_{\overline{H},n}^{(\iota)} := & \sum_{j = 1}^2  \mathcal{CK}^{(\iota;j)}_{\overline{H}, n}(t), \\
	\mathcal{CK}_{\overline{H},n}^{(\iota;1)}:= &  \nu^{i_{\iota}/2}(n+1)^{2\sss-2} \bold{a}_{n+1}^2 \brak{t}^{3+2\ss} \|\partial_y^{i_{\iota}} \overline{H}_n 
	\sqrt{- \partial_{t}W} e^{W/2} \chi_n \|_{L^2}^2,
	\label{CK1:defi} \\
	\mathcal{CK}_{\overline{H},n}^{(\iota;2)} :=&  -  \nu^{i_{\iota}/2} (n+1)^{2\sss-2} \bold{a}_{n+1}   \dot{\bold{a}}_{n+1} \brak{t}^{3+2\ss}
	\| \partial_y^{i_{\iota}} \overline{H}_n e^{W/2} \chi_n \|_{L^2}^2, \label{CK2:defi} \\
	\mathcal{CK}_{G,n}^{(\iota)} := & \begin{cases}  \paren{4\langle t \rangle^{4-2K\eps_1}t^{-1} - (4-2K\eps_1)t\brak{t}^{2-2\eps}}
		\enorm{\partial_y^{i_{\iota}} G\chi_0}^2 \qquad & n=0
		\\ \sum_{j=1}^{2} \mathcal{CK}_{G, n}^{(\iota; j)}(t)^2 \qquad & n\ge 1, 
	\end{cases}\\
	\mathcal{CK}_{G,n}^{(\iota;1)} :=& \bold{a}_{n}^2 \brak{t}^{3+2\ss} \|\partial_y^{i_{\iota}} G_n \sqrt{- \partial_{t}W} e^{W/2} \chi_{n-1+i_\iota} \|_{L^2}^2,\\
	\mathcal{CK}_{G,n}^{(\iota;2)} :=	&-  \bold{a}_{n} \dot{\bold{a}}_{n} 
	\brak{t}^{3+2\ss}
	\| \partial_y^{i_{\iota}} G_n  \chi_{n-1+i_\iota} \|_{L^2}^2,\\
	\mathcal{CK}_{H,n}^{(\iota)} :=& - \nu^{i_{\iota}/2}  \bold{a}_{n} \dot{\bold{a}}_{n}
	\| \partial_y^{i_{\iota}}{H}_n e^{W/2} \chi_n \|_{L^2}^2
	+ \nu^{i_{\iota}/2}  \bold{a}_{n}^2   \| \partial_y^{i_{\iota}}{H}_n \sqrt{- W_t} e^{W/2} \chi_n \|_{L^2}^2  .
\end{align}
We remind the readers that as in the functionals of $\omega$, we have three kinds of CK terms coming from the functions $\varphi, \lambda,$ and $W$. However, to ease the notation, we do not differentiate them and 
just denote the sum of them as $\mathcal{CK}$ for each coordinate functions $\overline H, H$, and $G$.

We also define the weighted sum of the three families of coordinate functionals.
\begin{align*}
	&\mathcal{E}_{\text{Ext, Coord}}= \sum_{\iota\in\{\gamma, \alpha\}}\paren{\mathcal{E}_{\overline{H}}^{(\iota)} + \mathcal{E}_{G}^{(\iota)} + \frac{1}{M} \mathcal{E}_{H}^{(\iota)}},
	\\&
	\mathcal{D}_{\text{Ext, Coord}} =\sum_{\iota\in\{\gamma, \alpha\}} \paren{\cd_{\overline{H}}^{(\iota)} + \cd_{G}^{(\iota)} + \frac{1}{M} \cd_{{H}}^{(\iota)}},
	\\&
	\mathcal{CK}_{\text{Ext, Coord}} =\sum_{\iota\in\{\gamma, \alpha\}} \paren{\mathcal{CK}_{\overline{H}}^{(\iota)} + \mathcal{CK}_{G}^{(\iota)} 
	+ \frac{1}{M} \mathcal{CK}_{{H}}^{(\iota)}},
\end{align*}
for $M>0$ sufficiently large (independent of $\eps$, $\nu$, $t$).
\vspace{2 mm}

\noindent \textsc{Sobolev Boundary Functionals:}
Due to the degeneracy of the weight, $q$, near the boundaries, we will need to supplement our norms with the following functionals, which do not degenerate at the boundaries, but instead have extra powers of $\nu$. 
We introduce the following Sobolev energy functionals
\begin{align} \label{hn:1}
\mathcal{E}_{sob, 0}(t) := & \sum_{k \in \mathbb{Z}} \| \omega_{k,0} e^W\|_{L^2}^2, \\
\mathcal{D}_{sob, 0}(t) := &\sum_{k \in \mathbb{Z}} \|\nu^{\frac12} \nabla_k \omega_{k,0}e^W \|_{L^2}^2, \\
\mathcal{CK}_{sob,0}(t) := & \sum_{k \in \mathbb{Z}}\|   \omega_{k,0}\sqrt{-\pa_t W}    e^{W} \|_{L^2}^2, \\
\mathcal{E}_{sob, n}(t) := & \nu^{2n} \sum_{k \in \mathbb{Z}} \|\nabla \omega_{k,n-1} e^W\|_{L^2}^2, \qquad \qquad n \ge 1 \\
\mathcal{D}_{sob, n}(t) := & \nu^{2n} \sum_{k \in \mathbb{Z}} \|\nu^{\frac12} \nabla_k^2 \omega_{k,n-1}e^W \|_{L^2}^2, \qquad n \ge 1, \\
\mathcal{CK}_{sob, n}(t) := & \nu^{2n} \sum_{k \in \mathbb{Z}} \|\nabla \omega_{k,n-1} \sqrt{-\partial_t W} e^W\|_{L^2}^2, \qquad \qquad n \ge 1
\end{align}
and finally the complete bulk Sobolev norm we propagate: 
\begin{align} \label{com:bulk:1}
\mathcal{E}_{sob}(t) := \sum_{i = 0}^4 \mathcal{E}_{sob,i}(t), \qquad \mathcal{D}_{sob}(t) := \sum_{i = 0}^4 \mathcal{D}_{sob,i}(t), \qquad \mathcal{CK}_{sob}(t) := \sum_{i = 0}^4 \mathcal{CK}_{sob,i}(t).
\end{align}
It turns out that the bounds on these bulk Sobolev functionals are themselves subtle, and require a delicate interplay between the one-dimensional quantities $\alpha_j^{\pm}(t, x) := \p_y^j \omega(t, x, \pm 1)$ arising through the boundary conditions on higher derivatives on $\omega$.
We introduce functionals to measure these one-dimensional quantities as follows:
\begin{align} \label{zero:changes:1}
\mathcal{E}_{\text{Trace},0}(t) := &\| \nu^{\frac32} \alpha_1 e^W \|_{L^2_x}^2, \qquad \mathcal{D}_{\text{Trace},0}(t) :=  \| \nu^{2} \p_x \alpha_1 e^W\|_{L^2_x}^2, \\
\mathcal{E}_{\text{Trace},1}(t) := &\| \nu^{\frac52} \p_x \alpha_1 e^W \|_{L^2_x}^2, \qquad \mathcal{D}_{\text{Trace},1}(t) :=  \| \nu^2 \p_t \alpha_1 e^W\|_{L^2_x}^2, \\
\mathcal{E}_{\text{Trace},2}(t) := &\| \nu^{\frac72} \p_x^2 \alpha_1 e^W\|_{L^2_x}^2, \qquad \mathcal{D}_{\text{Trace},2}(t) := \| \nu^3 \p_t \p_x \alpha_1 e^W\|_{L^2_x}^2. 
\end{align}
Above, when we use $W$, we mean $W|_{y = \pm 1}(t, x)$. Consequently, we also gain $\mathcal{CK}$ functionals: 
\begin{align}
\mathcal{CK}_{\text{Trace},0}(t) := &\| \sqrt{-\p_t W} \nu^{\frac32} \alpha_1 e^W \|_{L^2_x}^2, \\
\mathcal{CK}_{\text{Trace},1}(t) := &\|\sqrt{-\p_t W} \nu^{\frac52} \p_x \alpha_1 e^W \|_{L^2_x}^2,  \\
\mathcal{CK}_{\text{Trace},2}(t) := &\|\sqrt{-\p_t W} \nu^{\frac72} \p_x^2 \alpha_1 e^W\|_{L^2_x}^2. 
\end{align}
We subsequently also need the following maximal regularity functionals: 
\begin{align}
\mathcal{D}_{\text{Trace, Max},1}(t) &:= \| \nu^3 \p_x^2 \alpha_1 e^W \|_{L^2_x}^2, \\
\mathcal{D}_{\text{Trace, Max},2}(t) &:= \| \nu^4 \p_x^3 \alpha_1 e^W \|_{L^2_x}^2. 
\end{align}
Finally, we will need to measure (in $L^2_t$) the quantities $\alpha_4$ and $\p_t \alpha_4$. Hence, we introduce 
\begin{align}\label{zero:changes:10}
\mathcal{D}_{\text{Trace, Large},0}(t) &:= \| \nu^3 \alpha_4 e^W\|_{L^2_x}^2,\\
\mathcal{D}_{\text{Trace, Large},1}(t) &:= \| \nu^4 \p_x \alpha_4 e^W\|_{L^2_x}^2, \\ \label{zero:changes:large:1}
\mathcal{D}_{\text{Trace, Large},2}(t) &:= \| \nu^4 \p_t \alpha_4 e^W\|_{L^2_x}^2 + \| \nu^5 \p_{xx} \alpha_4 e^W\|_{L^2_x}^2. 
\end{align}
We will have our full trace norms: 
\begin{align}
\mathcal{E}_{\text{Trace}}(t) := & \sum_{i = 0}^2 \mathcal{E}_{\text{Trace},i}(t), \\
\mathcal{D}_{\text{Trace}}(t) := &\sum_{i = 0}^2 \mathcal{D}_{\text{Trace},i}(t) + \sum_{i = 1}^2 \mathcal{D}_{\text{Trace, Max},i}(t) + \sum_{i = 1}^3 \mathcal{D}_{\text{Trace, Large},i}(t),  \\ \label{line:90}
\mathcal{CK}_{\text{Trace}}(t) := &\sum_{i = 0}^2 \mathcal{CK}_{\text{Trace},i}(t).
\end{align}

\vspace{2 mm}

\noindent \textsc{Sobolev Cloud Functionals:} It turns out to close our nonlinear analysis, we need to introduce the family of so-called Sobolev ``cloud" norms. These norms should be interpreted as finite regularity interpolants (even though they cannot be obtained by any interpolation) between interior vorticity and exterior vorticity in the following sense: they retain the weight $e^W$ from the exterior energy functionals, whereas they are localized, due to $\chi^{I}$, to the interior.
In particular, we define: 
\begin{align}
\mathcal{E}_{\text{cloud}}(t) := &\sum_{m + n \le 20}  \| \varphi^{n} \p_x^m \Gamma^n \omega e^W \chi^{I} \|_{L^2}^2, \\
\mathcal{D}_{\text{cloud}}(t) := &\sum_{m + n \le 20} \|\varphi^{n} \sqrt{\nu} \nabla \p_x^m \Gamma^n \omega e^W \chi^{I} \|_{L^2}^2 \\
\mathcal{CK}_{\text{cloud}}^{(W)}(t) := &\sum_{m + n \le 20} \|\varphi^{n} \sqrt{-\dot{W}} \p_x^m \Gamma^n \omega e^W \chi^{I} \|_{L^2}^2 \\
\mathcal{CK}_{\text{cloud}}^{(\varphi)}(t) := &\sum_{m + n \le 20}  \|\varphi^{n} \sqrt{-\dot{\varphi}} \p_x^m \Gamma^n \omega e^W \chi^{I} \|_{L^2}^2.
\end{align}
\vspace{2 mm}

\subsection{Four Families of Elliptic Functionals}

\vspace{2 mm}

\noindent \textsc{Exterior \& Interior $\mapsto$ Interior Streamfunction Decomposition:} There are two separate decompositions of the streamfunction made into `interior' and `exterior' contributions.
In the $(z,v)$ variables, we use the decomposition $\psi = \psi^{(I)} + \psi^{(E)}$ where 
\begin{align}
& -\partial_{xx}\psi^{(I)} -(1+ h^I)^2 (\partial_v-t\partial_z)^2 \psi^{(I)} + (1+h^I)\partial_v h^I (\partial_v - t\partial_z) \psi^{(I)} = f^I,\label{Intr_psi_I} \\
& \psi^{(I)}(t,z,v(t,\pm 1)) = 0,
\end{align}
and
\begin{align}
& -\Delta_t \psi^{(E)} = (1-\chi^I)f - \left( (1+h^I)^2 - (1 +h)^2 \right) (\partial_v-t\partial_z)^2 \psi^{(I)} \\
& \qquad \qquad \;\; -  \left((1+h)\partial_v h - (1+h^I)\partial_v h^I\right)(\partial_v - t\partial_z) \psi^{(I)}, \label{Intr_psi_E} \\ 
& \psi^{(E)}(t, x, \pm 1) = 0,
\end{align}
where here $h^I := h \chi^I$ and $f^I := f \chi^I$.  
For $\psi^{(I)}_k$ in \eqref{Intr_psi_I}, we apply the Ionescu-Jia elliptic estimates from \cite{HI20}; see Section \ref{sec:interior} for more details. 
We consider the following norm for the interior and exterior contribution $\Psi^{(E)}_k(t,x,y)=\psi_k^{(E)}(t,z(t,x,y),v(t,y))$:
\begin{align}
\mathcal{F}_{ell}^{(E)}(t) := &\sum_{k\in\mathbb{Z}\backslash{\{0\}}}\sum_{m+n=0}^\infty\lf(\frac{(2\widetilde{\lambda})^{m+n}}{(m+n)!}\rg)^{2/r}\|\chi_{m+n}|k|^m q^n\Gamma_k^n  \Psi^{{(E)}}_k\|_{L^2}^2.
\end{align}
Estimates on this quantity are obtained in our companion work \cite{BHIW24b}. 
\vspace{2 mm}

\noindent \textsc{Exterior \& Interior $\mapsto$ Exterior Streamfunction Decomposition:} For the exterior elliptic estimates, we use a different decomposition for the stream function 
\begin{subequations}
\begin{align} \label{ell:I}
(\pa_v^2-|k|^2)\ \phi^{(I)}_k(t,v) &=\ \widetilde{\chi}_1^\mathfrak{c}(v)\ \omega_k(t,v)+{\wt \chi_1^\mathfrak{c}(v)(\pa_v^2-(v'\pa_v)^2)\phi_k^{(I)}(t,v)}, \\
\n\qquad &\phi^{(I)}_k(t,v)|_{v=v(t,\pm 1)} \ =\ 0, \\  \label{ell:E}
(\pa_y^2-|k|^2)\ \phi^{(E)}_k(t,y) &=\ \widetilde{\chi}_1\lf(v(y)\rg) \omega_k(t,y)+{\wt\chi_1(v(y))}(\pav^2-\pa_{y}^2)\phi^{(I)}_k(t,v(y)),\\
\n \qquad &\phi^{(E)}_k(t,y)|_{y = \pm 1} = 0. 
\end{align}
\end{subequations}
Here $\pav:=v_y^{-1}\pa_y$ and we will use the $v'$ to denote the function $v_y$ in the $(x,v)$-coordinate. 

To measure these contributions, we introduce a second set of functionals to measure the stream function $\phi^{(I)}$, defined above in \eqref{ell:I}: 
\begin{align} \label{Friday:1}
\mathcal{E}_{\mathrm{ell}}^{(I, out)}(t) := &\sum_{k\in \mathbb{Z}\backslash\{0\}} \sum_{m+n=0}^\infty \widehat{\bold{a}}_{m,n}^2\|\chi_1 |k|^m(\pa_v+ikt)^n\pa_v^\ell \phi^{(I)}_k\|_{L^2_v}^2,\quad \ell\in\{0,1,2\}, \\ \label{Friday:2}
\mathcal{E}_{\mathrm{ell}}^{(I, full)}(t) := & \sum_{k\in \mathbb{Z}\backslash\{0\}}\sum_{m+n=0}^{1000} \widehat{B}_{m,n}^2\| |k|^m(\pa_v+i kt)^n  \phi^{(I)}_k\|_{L^2_v}^2.
\end{align}
We define an auxiliary functional that will be bounded (this helps simplify notations in the more complex trilinear estimates of Section \ref{sec:Tri}): 
\begin{align}
\slashed{\mathcal{E}}_{\mathrm{ell}}^{(I)}(t) := \langle t \rangle^4 \mathcal{E}_{\mathrm{ell}}^{(I, full)} + \langle t \rangle^{1000} \mathcal{E}_{\mathrm{ell}}^{(I, out)}(t). 
\end{align}  

Near the boundary, one needs to control various commutator terms involving $q$. It turns out that the following operations are crucial to express these commutators
\begin{align}
J^{(a,b,c)}_{m,n}\phe_k= &{\lf(\frac{m+n}{q}\rg)^a}\pav^b |k|^{c} ( \chi_{m+n}|k|^m q^n\Gamma_k^n \phe_k) \mathbbm{1}_{\{a+b+c\leq n\text{ or }a=0\}},\quad a,b,c\in \mathbb{N},\quad a+b+c\leq3 ,\label{J_intro}
\end{align}
where $\pav:=v_y^{-1}\pa_y$. We will derive the estimates for the following functionals
\begin{align}
\mathcal{J}_{{\mathrm{ell}}}^{(\ell)}(t):=& \sum_{m,n\geq 0}\ \sum_{a+b+c=\ell}{\bf a}_{m,n}^2 \|J_{m,n}^{(a,b,c)}\phe_{\neq}\|_{L_{x,y}^2}^2 \n\\
=&\sum_{k\in \mathbb{Z}\backslash\{0\}}\ \sum_{m,n\geq 0}\ \sum_{a+b+c=\ell}{\bf a}_{m,n}^2 \|J_{m,n}^{(a,b,c)}\phe_k\|_{L_y^2}^2,\quad \ell\in\{1,2,3\}.\label{Gj_intro}
\end{align}
We highlight that these norms controls the usual $H^{\ell}$ norm of the function. We further highlight that for the $\ell=1,2$ case, the $L_t^\infty$-estimates of $\mathcal{J}_{ell}^{(\ell)}$ are available; whereas for the $\ell=3$ case, one only has $L_t^2$-bound.  

\subsection{Bootstraps and Main Proposition}


We will execute a standard bootstrap argument. Specifically, we assume the following estimates hold on some time interval $[0,T]$ and show that the same inequalities hold with `4' replaced with `2' for all $\eps$ sufficiently small, thus permitting the extension of the desired estimates on the coordinate system and vorticity for all time (note, all quantities will depend continuously on time for $t > 0$). 
\begin{itemize}
\item Interior profile estimates: 
\begin{align}  \label{boot:Intf}
&\sup_{0 \le t \le T} \mathcal{E}_{\text{Int}}(t) + \int_0^T [ \mathcal{CK}_{\lambda}(t) +  \mathcal{CK}_{w}(t) +  \mathcal{CK}_{M}(t) ] dt + \int_0^T \mathcal{D}_{\text{Int}}(t)  dt \leq 4 \eps^2.
\end{align}

\item Interior coordinate system estimates:
	\begin{align} 
			\label{boot:IntH}
		\sup_{0 \le t \le T} \mathcal{E}_{\text{Int,Coord}} +  \int_0^T \mathcal{CK}_{\text{Int,Coord}} dt +  \int_0^T \mathcal{D}_{\text{Int,Coord}} dt \leq 4 \eps^2 .
	\end{align}
\item Exterior vorticity estimates:
\begin{subequations}
 \label{boot:ExtVort}
	\begin{align}
		\sup_{0 \le t \le T} \mathcal{E}^{(\gamma)} + \sum_{\iota \in \{\varphi, W \}} \int_0^T \mathcal{CK}^{(\gamma; \iota)}(t) dt +  \int_0^T \mathcal{D}^{(\gamma)}(t) dt \le 4 \eps^2 \\
		\sup_{0 \le t \le T} \mathcal{E}^{(\alpha)} + \sum_{\iota \in \{\varphi, W \}} \int_0^T \mathcal{CK}^{(\alpha; \iota)}(t) dt +  \int_0^T \mathcal{D}^{(\alpha)}(t) dt \le 4 \eps^2\\
		\sup_{0 \le t \le T} \mathcal{E}^{(\mu)} + \sum_{\iota \in \{\varphi, W \}} \int_0^T \mathcal{CK}^{(\mu; \iota)}(t) dt +  \int_0^T \mathcal{D}^{(\mu)}(t) dt \le 4 \eps^2.
	\end{align}
\end{subequations}
\item Exterior coordinate system estimates: 
\begin{align} \label{boot:H}
\sup_{0 \le t \le T} \mathcal{E}_{\text{Ext, Coord}} + \int_0^T \mathcal{CK}_{\text{Ext, Coord}}(t) dt +  \int_0^T \mathcal{D}_{\text{Ext, Coord}}(t) dt \le 4\eps^2.
\end{align}
\item Sobolev boundary estimates: 
\begin{subequations}
 \label{boot:sob}
 \begin{align}
 	&\sup_{0 \le t \le T} \mathcal{E}_{sob}(t) + \int_0^T \mathcal{CK}_{sob}(t) dt + \int_0^T \mathcal{D}_{sob}(t)  dt \leq 4 \eps^2, \\
 	&\sup_{0 \le t \le T} \mathcal{E}_{Trace}(t) + \int_0^T \mathcal{CK}_{Trace}(t) dt + \int_0^T \mathcal{D}_{Trace}(t)  dt \leq 4 \eps^2.
 \end{align}
\end{subequations}
\item Cloud norm estimates:
\begin{align} \label{boot:cloud}
&\sup_{0 \le t \le T} \mathcal{E}_{\text{cloud}}(t) + \int_0^T \mathcal{CK}_{\text{cloud}}^{(\varphi)}(t) + \mathcal{CK}_{\text{cloud}}^{(W)}(t) dt + \int_0^T \mathcal{D}_{\text{cloud}}(t)  dt \leq 4 \eps^2. 
\end{align}
\end{itemize}

The main step in the proof of Theorem \ref{thm:main} is the following. 
\begin{proposition} \label{prop:boot}
For all $\zeta > 0$ sufficiently small, and all $\lambda_0 > 0$ sufficiently small, there exists an $ \eps_0, \nu_0$ such that for all $\eps \in (0,\eps_0)$ and all $\nu \in (0,\nu_0)$, if the bootstrap hypotheses \eqref{boot:Intf},\eqref{boot:IntH},\eqref{boot:ExtVort}, \eqref{boot:H}, \eqref{boot:sob}, and \eqref{boot:cloud} hold on $[0,T]$ for some $T < \nu^{-1/3-\zeta}$, then the same inequalities in fact hold with `4' replaced with `2'.
\end{proposition}

\begin{remark}
Note that this proposition both takes $r>1/2$ close to $1/2$ and $\lambda_0$ small, mainly for technical convenience. However, this does not change the generality of Theorem \ref{thm:main}. 
\end{remark}

Proposition \ref{prop:boot} along with the standard well-posedness theory implies that the estimates \eqref{boot:Intf}, \eqref{boot:IntH}, \eqref{boot:ExtVort}, and \eqref{boot:H} hold on $[0,\nu^{-1/3-\zeta}]$. Denote $t_\ast = \nu^{-1/3-\zeta}$.  
At $t_\ast$ we then deduce the estimates
\begin{align*}
& \norm{\omega_0(t_\ast)}_{H^4} \lesssim \eps, \\
& \norm{\omega_{\neq}(t_\ast)}_{H^4} \lesssim \eps e^{-\delta' \nu^{1/3} t_\ast},  
\end{align*}
for some small $\delta' > 0$. As discussed in Section \ref{sec:NonTechIntro}, this combines with \cite{BHIW23} and the estimates proved in Proposition \ref{prop:boot} to imply Parts (i) and (ii) in  Theorem \ref{thm:main}. Part (iii) is explained in Section \ref{sec:IL}. 

\subsection{Sub-Propositions and Architecture of the Proof}
In this section we briefly outline the strategy to prove Proposition \ref{prop:boot}. There are really five main kinds of estimates in the proof.
The first kind are those that are done in the $(z,v)$ coordinate system and use the Fourier multiplier methods of \cite{BM13,HI20}; namely, the inequalities \eqref{boot:Intf} and \eqref{boot:IntH}. 
The second kind are high regularity estimates on the exterior vorticity and coordinate system variables which simultaneously quantify the gain in Gevrey index away from the interior, the strong localization away from the boundary, and Gevrey regularity with respect to the almost-commuting vector field $\Gamma_t$ that is adapted to match the $\partial_v$ derivatives. However, these estimates are done with the co-normal weight $q$, and so regularity is not obtained all the way to the boundary.  
These are the family of inequalities \eqref{boot:ExtVort} and \eqref{boot:H}.
The next family are the Sobolev boundary estimates, which propagates strong localization away from the boundary and finite regularity all the way to the boundary.
Hence, these estimates must deal with the boundaries in a more fundamental way, but have the advantage that we only need a few derivatives.
These are \eqref{boot:sob}.
The last family of inequalities are \eqref{boot:cloud}, which are the Sobolev cloud norms. These estimates are used to control some of the most asymmetric quasilinear terms as explained above and bridge the gap between the interior and exterior norms. Crucial to all of these are the four types of elliptic estimates on the streamfunction. 

\subsubsection{Interior estimates}
The main estimate on the interior vorticity is summarized in the following proposition, which is proved in Section \ref{sec:int:vort}.  
\begin{proposition}[Section \ref{sec:int:vort}] \label{prop:MainInterior} 
Under the bootstrap hypotheses, for $\eps$ sufficiently small, $t \lesssim \nu^{-1/3-1/\zeta}$ $0 < \zeta < 1/78$, the following estimates are valid: 
\begin{align*} \n
\frac{1}{2}\frac{d}{dt}\mathcal{E}_{\mathrm{Int}}(t) + \mathcal{CK}_{\mathrm{Int}}(t) +  \mathcal{D}_{\mathrm{Int}}(t) + \delta_I \nu^{1/3} \mathcal{E}_{\mathrm{Int}}(t) &\\
& \hspace{-6cm} \lesssim (\mathcal{E}_{\mathrm{Int}}^{1/2} + (\mathcal{E}_{\mathrm{Int,Coord}}^{(g)})^{1/2}) \left(\mathcal{CK}_{\mathrm{Int}}(t) + \delta_I \nu^{1/3} \mathcal{E}_{\mathrm{Int}}(t)  + \mathcal{D}_{\mathrm{Int}}\right) \\ 
& \hspace{-6cm} \quad + \frac{\mathcal{E}_{\mathrm{Int}}^{1/2} + (\mathcal{E}_{\mathrm{Int,Coord}})^{1/2}}{\brak{t}^{3/2}} \mathcal{E}_{\mathrm{Int}}  + (\mathcal{E}_{\mathrm{Int}}(t))^{1/2}\mathcal{CK}_{\mathrm{Int,Coord}}(t) \\ 
& \hspace{-6cm} \quad + \mathcal{E}_{\mathrm{Ext,Coord}}^{1/2}( \mathcal{CK}_{\mathrm{Int}} +  \frac{1}{\brak{t}^{3/2}}\mathcal{E}_{\mathrm{Int}}) + e^{-\nu^{-1/9}} (\mathcal{E}_H^{(\gamma)})^{1/2} (\mathcal{E}_{\mathrm{Int}}^{1/2} + \mathcal{E}_{\mathrm{Int,Coord}}^{1/2}) \mathcal{D}_{\mathrm{Int}}^{1/2}\\
& \hspace{-6cm} \quad + e^{-\nu^{-1/10}} \left( \mathcal{E}^{(\gamma)}\right)^{1/2} \mathcal{E}_{\mathrm{Int}}^{1/2} \mathcal{D}_{\mathrm{Int}}^{1/2} + \nu^{50} (\mathcal{E}^{(\gamma)}_H)^{1/2} \mathcal{D}_{\mathrm{Int}}^{1/2} \left( \mathcal{D}_{\mathrm{Int}}^{1/2} + \mathcal{D}_{\mathrm{Int,Coord}}^{1/2}\right) \\
& \hspace{-6cm} \quad + \nu^{50}\mathcal{E}_{\mathrm{Int}}^{1/2} (\mathcal{E}^{(\gamma)})^{1/2}(1 + \mathcal{E}_H^{(\gamma)})^{1/2}. 
\end{align*}
\end{proposition}
The first three terms on the RHS arise essentially as in \cite{HI20,BM13} with some slight variations.
The rest arise from the contributions of $\phi^{(E)}$ and commutators with $\chi^I$. 
In order to control $(g,h,\bar{h})$, we use the following, which follow by methods in \cite{HI20,BM13} along with those used to prove Proposition \ref{prop:MainInterior}. 
\begin{proposition}[Section \ref{sec:interior}] \label{prop:MainCoordInt}
Under the bootstrap hypotheses, the following estimates are valid: 
\begin{align*}
\frac{1}{2}\frac{d}{dt}\mathcal{E}_{\mathrm{Int,Coord}}^{(h)} + \mathcal{CK}^{(h)}_{\mathrm{Int,Coord}} +  \mathcal{D}_{\mathrm{Int,Coord}}^{(h)} & \lesssim \\
& \hspace{-6cm} (\mathcal{E}_{\mathrm{Int,Coord}}^{(g)} + \mathcal{E}_{G}^{(\gamma)})^{1/2} \mathcal{CK}_{\mathrm{Int,Coord}}^{(h)} + (\mathcal{E}_{\mathrm{Int,Coord}}^{(h)})^{1/2}( \mathcal{CK}_{\mathrm{Int,Coord}}^{(h)})^{1/2} (\mathcal{CK}_{\mathrm{Int,Coord}}^{(\bar{h})})^{1/2} \\ 
& \hspace{-6cm} \quad + (\mathcal{CK}^{(h)}_{\mathrm{Int,Coord}})^{1/2}(\mathcal{CK}^{(\bar{h})}_{\mathrm{Int,Coord}})^{1/2} \\
& \hspace{-6cm} \quad + \nu (\mathcal{E}_{\mathrm{Int,Coord}}^{(h)})^{1/2} \mathcal{D}_{\mathrm{Int,Coord}}^{(h)} + \nu^{100}(\mathcal{E}_{\mathrm{Int,Coord}}^{(h)})^{1/2}(\mathcal{E}^{(\gamma)}_{H})^{1/2}, \\
\frac{1}{2}\frac{d}{dt}\mathcal{E}_{\mathrm{Int,Coord}}^{(g)}(t) + \mathcal{CK}^{(g)}_{\mathrm{Int,Coord}}(t) +  \mathcal{D}_{\mathrm{Int,Coord}}^{(g)}(t) & \lesssim  \\
& \hspace{-6cm} (\mathcal{E}_{\mathrm{Int,Coord}}^{(g)} + \mathcal{E}_{G}^{(\gamma)})^{1/2} \mathcal{CK}_{\mathrm{Int,Coord}}^{(g)} + (\mathcal{E}_{\mathrm{Int,Coord}}^{(g)})^{1/2}( \mathcal{CK}_{\mathrm{Int,Coord}}^{(g)})^{1/2} (\mathcal{CK}_{\mathrm{Int,Coord}}^{(\bar{h})})^{1/2} \\
& \hspace{-6cm} \quad + \nu (\mathcal{E}_{\mathrm{Int,Coord}}^{(h)})^{1/2} \mathcal{D}_{\mathrm{Int,Coord}}^{(g)} + \nu^{100}(\mathcal{E}_{\mathrm{Int,Coord}}^{(g)})^{1/2}(\mathcal{E}^{(\gamma)}_{G})^{1/2}, \\
& \hspace{-6cm}  \frac{1}{t^{1+r}}\left(\mathcal{E}_{\mathrm{Int}} + \mathcal{E}^{(\gamma)} \right), \\
\frac{1}{2}\frac{d}{dt}\mathcal{E}_{\mathrm{Int,Coord}}^{(\bar{h})}(t) + \mathcal{CK}^{(\bar{h})}_{\mathrm{Int,Coord}}(t) +  \mathcal{D}_{\mathrm{Int,Coord}}^{(\bar{h})}(t) & \lesssim \\  
& \hspace{-6cm} (\mathcal{E}_{\mathrm{Int,Coord}}^{(g)} + \mathcal{E}_{G}^{(\gamma)})^{1/2} \mathcal{CK}_{\mathrm{Int,Coord}}^{(\bar{h})} + (\mathcal{E}_{\mathrm{Int,Coord}}^{(\bar{h})})^{1/2} \mathcal{CK}_{\mathrm{Int,Coord}}^{(\bar{h})} \\
& \hspace{-6cm} \quad + \nu (\mathcal{E}_{\mathrm{Int,Coord}}^{(h)})^{1/2} \mathcal{D}_{\mathrm{Int,Coord}}^{(\bar{h})} + \nu^{100}(\mathcal{E}_{\mathrm{Int,Coord}}^{(\bar{h})})^{1/2}(\mathcal{E}^{(\gamma)}_{\overline{H}})^{1/2}, \\
& \hspace{-6cm} \quad + (\mathcal{CK}_{\mathrm{Int,Coord}})^{1/2}(\mathcal{CK}_{\mathrm{Int}})^{1/2}( \mathcal{E}_{\mathrm{Int}} + \mathcal{E}^{(\gamma)} )^{1/2} \\
& \hspace{-6cm} \quad  + \frac{1}{t^{1+r}}(\mathcal{E}_{\mathrm{Int}}^{1/2} + (\mathcal{E}^{(\gamma)})^{1/2}) \mathcal{E}_{\mathrm{Int,Coord}}^{(\bar{h})}.
\end{align*}
\end{proposition}

\subsubsection{Exterior vorticity estimates} 

We invoke our main parabolic linearized result from our companion paper, \cite{BHIW24b}:
\begin{proposition} \label{thm:main:para} Let the constants $\{\theta_n\}$ appearing in \eqref{ef:a} -- \eqref{CK_W_mu} be chosen as in \cite{BHIW24b}. Under the bootstrap hypotheses, the following energy inequalities are valid 
\begin{align} \n
&\frac{d}{d t} \mathcal{E}^{(\gamma)}[\omega]  + \mathcal{D}^{(\gamma)}[\omega] + \mathcal{CK}^{(\gamma; \varphi)}[\omega] + \mathcal{CK}^{(\gamma; W)}[\omega] 
 \\ \label{lights:on:1}
 & \qquad \lesssim   \sum_{L \in \{I, E\}}  |\mathcal{I}^{(L, \gamma)}_{\mathrm{Trilinear}}| +  \eps \cd_{H}^{(\gamma)} +\eps\cd_{\overline H}^{(\gamma)} + \frac{\eps^3}{\langle t \rangle^2}, \\ \n
 &\frac{d}{d t} \mathcal{E}^{(\alpha)}[\omega]  + \mathcal{D}^{(\alpha)}[\omega] + \mathcal{CK}^{(\alpha; \varphi)}[\omega] + \mathcal{CK}^{(\alpha; W)}[\omega] 
 \\ \label{lights:on:2}
 & \qquad \lesssim  \sum_{\iota \in (\gamma, \alpha, \mu)} \sum_{L \in \{I, E\}}  |\mathcal{I}^{(L, \iota)}_{\mathrm{Trilinear}}|  + \eps \sum_{\iota \in (\alpha, \gamma)} (\mathcal{D}^{(\iota)}_{H} +\mathcal{D}^{(\iota)}_{\overline{H}})  + \eps \cd^{(\gamma)}[\omega]  
 \n\\ &\qquad\quad+ \eps \paren{ \mathcal{CK}^{(\gamma; \varphi)}[\omega] + \mathcal{CK}^{(\gamma; W)}[\omega] }  + \frac{\eps^3}{\langle t \rangle^2},  \\ \n
 &\frac{d}{d t} \mathcal{E}^{(\mu)}[\omega]  + \mathcal{D}^{(\mu)}[\omega] + \mathcal{CK}^{(\mu; \varphi)}[\omega] + \mathcal{CK}^{(\mu; W)}[\omega] 
 \\ \label{lights:on:3}
 & \qquad \lesssim  \sum_{L \in \{I, E\}}  |\mathcal{I}^{(L, \mu)}_{\mathrm{Trilinear}}|  + \eps \cd_{H}^{(\gamma)} +\eps\cd_{\overline H}^{(\gamma)} + \frac{\eps^3}{\langle t \rangle^2}.
\end{align}
The source terms appearing above are defined as follows:
\begin{align}\label{in:n:out:1}
\mathcal{I}^{(L, \gamma)}_{\mathrm{Trilinear}} := & - \sum_{m = 0}^\infty \sum_{n = 0}^\infty \bold{a}_{m,n}^2 \theta_n^2 \mathrm{Re} \langle q^n \p_x^m \Gamma^n (\nabla^\perp \phi^{(L)} \cdot \nabla \omega), \omega_{m,n} \chi_{m + n}^2 e^{2W} \rangle, \\ \label{in:n:out:2}
\mathcal{I}^{(L, \alpha)}_{\mathrm{Trilinear}} := &- \sum_{m = 0}^\infty \sum_{n = 0}^\infty \bold{a}_{m,n}^2 \theta_n^2 \nu \mathrm{Re} \langle q^n \p_x^m \p_y \Gamma^n  (\nabla^\perp \phi^{(L)} \cdot \nabla \omega), \p_y \omega_{m,n} \chi_{m + n}^2 e^{2W} \rangle, \\ \label{in:n:out:3}
\mathcal{I}^{(L, \mu)}_{\mathrm{Trilinear}} := & - \sum_{m = 0}^\infty \sum_{n = 0}^\infty \bold{a}_{m,n}^2 \theta_n^2 \nu \mathrm{Re} \langle q^n \p_x^m \p_x \Gamma^n  (\nabla^\perp \phi^{(L)} \cdot \nabla \omega), \p_x \omega_{m,n} \chi_{m + n}^2 e^{2W} \rangle,
\end{align}
where $\phi^{(I)}, \phi^{(E)}$ are defined above in \eqref{ell:I}, \eqref{ell:E}. 
\end{proposition}

The next proposition provides estimates on the trilinear terms appearing in \eqref{in:n:out:1} -- \eqref{in:n:out:3}: 
\begin{proposition}[Section \ref{sec:Tri}] \label{pro:tri:in} Under the bootstrap hypotheses, the $\gamma$-level trilinear terms obey the following bounds: 
\begin{align} \label{esp:1}
|\mathcal{I}_{\mathrm{Trilinear}}^{(I, \gamma)}| \lesssim & (\mathcal{E}^{(I)}_{\mathrm{ell}})^{\frac12}( \mathcal{CK}^{(\gamma)} + \mathcal{CK}_{\mathrm{Cloud}} + \mathcal{D}^{(\gamma)} ) + \frac{(\mathcal{E}_{\mathrm{ell}}^{(I)})^{\frac12}}{\langle t \rangle^{100}} \mathcal{E}^{(\gamma)}, \\ \label{esp:2}
|\mathcal{I}_{\mathrm{Trilinear}}^{(E, \gamma)}| \lesssim &  \frac{\langle t \rangle^2}{\nu^2} \mathcal{J}^{(2)}_{\mathrm{ell}}(t)^{\frac12} ( \mathcal{D}^{(\gamma)}(t) + \mathcal{E}^{(\gamma)}(t)).
\end{align}
The $\alpha$-level trilinear terms obey the following bounds: 
\begin{align} \label{esp:3}
|\mathcal{I}_{\mathrm{Trilinear}}^{(I, \alpha)}| \lesssim &( \mathcal{E}_{\mathrm{ell}}^{(I)})^{\frac12} ( \mathcal{CK}^{(\gamma)}(t) +  \mathcal{CK}^{(\alpha)}(t) +  \mathcal{CK}_{\mathrm{Cloud}}(t)  ) + \frac{\mathcal{E}^{(I)}_{\mathrm{ell}}(t)^{\frac12}}{\langle t \rangle^{100}} \mathcal{E}^{(\alpha)}(t) , \\ \n
|\mathcal{I}_{\mathrm{Trilinear}}^{(E, \alpha)}| \lesssim & \frac{\langle t \rangle^2}{\nu^2} (\mathcal{J}^{(3)}_{\mathrm{ell}})^{\frac12}[ (\mathcal{E}^{(\alpha)})^{\frac12} (\mathcal{D}^{(\alpha)})^{\frac12} + (\mathcal{E}_{\mathrm{Cloud}})^{\frac12}( \mathcal{E}^{(\alpha)})^{\frac12} ] + \frac{\langle t \rangle^2}{\nu^2} (\mathcal{J}^{(2)}_{\mathrm{ell}})^{\frac12} \mathcal{D}^{(\alpha)} \\ \label{esp:4}
& + e^{-\nu^{-\frac19}} \sqrt{\mathcal{E}_{\mathrm{sob}}} \sqrt{\mathcal{E}^{(\gamma)} \mathcal{D}^{(\alpha)}}.
\end{align}
The $\mu$-level trilinear terms obey the following bounds: 
\begin{align} \label{esp:5}
|\mathcal{I}_{\mathrm{Trilinear}}^{(I, \mu)}| \lesssim &( \mathcal{E}_{\mathrm{ell}}^{(I)})^{\frac12} ( \mathcal{CK}^{(\gamma)}(t) +  \mathcal{CK}^{(\mu)}(t) +  \mathcal{CK}_{\mathrm{Cloud}}(t)  ) + \frac{\mathcal{E}^{(I)}_{\mathrm{ell}}(t)^{\frac12}}{\langle t \rangle^{100}} \mathcal{E}^{(\mu)}(t) \\ \n
|\mathcal{I}_{\mathrm{Trilinear}}^{(E, \mu)}| \lesssim &\frac{\langle t \rangle^2}{\nu^2} (\mathcal{J}^{(3)}_{\mathrm{ell}})^{\frac12}[ (\mathcal{E}^{(\mu)})^{\frac12} (\mathcal{D}^{(\mu)})^{\frac12} + (\mathcal{E}_{\mathrm{Cloud}})^{\frac12}( \mathcal{E}^{(\mu)})^{\frac12} ] + \frac{\langle t \rangle^2}{\nu^2} (\mathcal{J}^{(2)}_{\mathrm{ell}})^{\frac12} \mathcal{D}^{(\mu)}  \\ \label{esp:6}
&  + e^{-\nu^{-\frac19}} \sqrt{\mathcal{E}_{\mathrm{sob}}} \sqrt{\mathcal{E}^{(\gamma)} \mathcal{D}^{(\mu)}}.
\end{align}
\end{proposition}

\subsubsection{Exterior Coordinate Estimates}
\begin{proposition}[Section \ref{sec:coord:FEI}]
	Let $r$ and $s$ be defined as in \eqref{pgiL1}. Under the bootstrap hypotheses, we have
	\begin{align*}
		\frac{1}{2}\frac{d}{dt} \mathcal{E}_{\mathrm{Ext, Coord}}+
		\mathcal{D}_{\mathrm{Ext, Coord}} +
		\mathcal{CK}_{\mathrm{Ext, Coord}}
		\lesssim&  \frac{\eps^3}{\brak{t}^{3+r-2\ss}} + \eps \sum_{\iota\in\{\alpha,\gamma\}} \cd^{(\iota)} .
	\end{align*}
\end{proposition}
The above proposition is  a direct consequence of the proposition below
\begin{proposition}[Section \ref{sec:coord:FEI}] \label{gamm:esti:llp} Let $r$ and $s$ be defined as in \eqref{pgiL1}. Assuming the bootstrap hypotheses, the following estimates are valid for the $\gamma$-level coordinate functionals: 
\begin{align}\n
		\frac{1}{2}\frac{d}{dt} \mathcal{E}_{\overline{H}}^{\gm}+
\cd_{\overline{H}}^{\gm}+
\mathcal{CK}_{\overline{H}}^{\gm}
\lesssim&  
\frac{\eps^3}{\brak{t}^{3+r-2\ss}} +  \eps \paren{\cd^{\gm}+\cd_{H}^{\gm} 
	+\cd_{\overline H}^{\gm}}
\n	\\&\quad 
+\eps\paren{\mathcal{CK}^{\gm}
	+\mathcal{CK}_{H}^{\gm}+\mathcal{CK}_{\overline H}^{\gm}}, \n \\ 
\frac{1}{2}\frac{d}{dt} \mathcal{E}_{G}^{\gm}+
\cd_{G}^{\gm}+
\mathcal{CK}_{G}^{\gm}
\lesssim&  
\frac{\eps^3}{\brak{t}^{3+r-2\ss}} +  \eps \paren{\cd^{\gm}+\cd_{H}^{\gm} 
	+\cd_{\overline H}^{\gm}}
\n \\&\quad 
+\eps\paren{\mathcal{CK}^{\gm}
	+\mathcal{CK}_{H}^{\gm}+\mathcal{CK}_{\overline H}^{\gm}}, \n
\\ 
\frac{1}{2}\frac{d}{dt} \mathcal{E}_{H}^{\gm}+
\cd_{H}^{\gm}+
\mathcal{CK}_{H}^{\gm}
\lesssim &   \mathcal{CK}_{\overline{H}}^{\gm} . \n
\end{align}
The following estimates are valid for the $\alpha$-level coordinate functionals: 
\begin{align} 
	\frac{1}{2}\frac{d}{dt} \mathcal{E}_{\overline{H}}^{(\alpha)}&+
\cd_{\overline{H}}^{(\alpha)}+
\mathcal{CK}_{\overline{H}}^{(\alpha)}
\lesssim \frac{\eps^3}{\brak{t}^{3+r-2\ss}} +  \eps \sum_{\iota\in\{\alpha,\gamma\}} \paren{\cd^{(\iota)}+\cd_{H}^{(\iota)} 
	+\cd_{\overline H}^{(\iota)}}
\n \\&\qquad  \qquad  \qquad  \qquad  \qquad  
+\eps\sum_{\iota\in\{\alpha,\gamma\}} \paren{\mathcal{CK}^{(\iota)}
	+\mathcal{CK}_{H}^{(\iota)}+\mathcal{CK}_{\overline H}^{(\iota)}}, \n
\\ 
\frac{1}{2}\frac{d}{dt} \mathcal{E}_{G}^{(\alpha)} &+
\cd_{G}^{(\alpha)}+
\mathcal{CK}_{G}^{(\alpha)}
\lesssim \frac{\eps^3}{\brak{t}^{3+r-2\ss}} +  \eps \sum_{\iota\in\{\alpha,\gamma\}} \paren{\cd^{(\iota)}+\cd_{H}^{(\iota)} 
	+\cd_{\overline H}^{(\iota)}}
\n	\\&\qquad  \qquad  \qquad  \qquad  \qquad  
+\eps\sum_{\iota\in\{\alpha,\gamma\}} \paren{\mathcal{CK}^{(\iota)}
	+\mathcal{CK}_{H}^{(\iota)}+\mathcal{CK}_{\overline H}^{(\iota)}}, \n 
\\ 
\frac{1}{2}\frac{d}{dt} \mathcal{E}_{H}^{(\alpha)} &+
\cd_{H}^{(\alpha)}+
\mathcal{CK}_{H}^{(\alpha)}
\lesssim \mathcal{CK}_{\overline{H}}^{(\alpha)} . \n
	\end{align}
\end{proposition}

\subsubsection{Sobolev Boundary Norm Estimates} 

\begin{proposition}[Section \ref{sec:SOB:BDRY}] \label{pro:int:sob} The following coupled energy inequality holds for the Sobolev norms under the bootstrap hypotheses,
\begin{align} \label{hazee:a}
\frac{\p_t}{2} \mathcal{E}_{\mathrm{sob}} + \mathcal{CK}_{\mathrm{sob}}^{(W)} + \mathcal{D}_{\mathrm{sob}} \lesssim &\nu^2 \mathcal{D}_{\mathrm{Trace}} + \mathcal{E}_{\mathrm{sob}} \mathcal{D}_{\mathrm{sob}} ,  \\ \n
\frac{\p_t}{2} \mathcal{E}_{\mathrm{Trace}}(t) +   \mathcal{CK}_{\mathrm{Trace}}(t) + \mathcal{D}_{\mathrm{Trace}}(t) \lesssim &(1 +  (\frac{\mathcal{E}_{\mathrm{ell}}^{(I, out)}}{\langle t \rangle^{100}}  + \nu^{100} \mathcal{E}_{\mathrm{sob}})) (1+ (\mathcal{E}_{\mathrm{ell}}^{(I, out)})^{\frac12} +  \mathcal{E}_{\mathrm{sob}}^{\frac12})\mathcal{D}_{\mathrm{sob}} \\ \label{hazee:b}
& + \nu^3  (\frac{\mathcal{E}_{\mathrm{ell}}^{(I, out)}}{\langle t \rangle^{100}}  + \nu^{100} \mathcal{E}_{\mathrm{sob}}) \mathcal{E}_{\mathrm{Trace}}(t). 
\end{align}
\end{proposition}

\subsubsection{Sobolev Cloud Norm Estimates}
\begin{proposition}[Section \ref{sec:CLOUD}] \label{pro:cloud:intro} Let $L \gg 1$ chosen large relative to universal constants. Under the bootstrap hypotheses, the following inequality is valid: 
\begin{align} \n
\frac{\p_t}{2} \mathcal{E}_{\mathrm{cloud}} + \mathcal{D}_{\mathrm{cloud}} + \mathcal{CK}_{\mathrm{cloud}} \lesssim & \frac{\nu^{\frac13 - 2\zeta}}{\langle  t\rangle^{2}}  \mathcal{E}^{(\gamma)} +   \frac{1}{\langle t \rangle^2} \sqrt{ \mathcal{J}_{\mathrm{ell}}^{(2)}(t)} \mathcal{E}_{\mathrm{cloud}}(t) + \frac{1}{\langle t \rangle^2}\sqrt{ \mathcal{J}_{\mathrm{ell}}^{(2)}(t)} \mathcal{E}^{(\gamma)} \\ \label{bhbhy:1:llp}
&+ \frac{1}{L} \mathcal{CK}^{(\gamma, W)} +  \frac{1}{\brak{t}^{2+}} \sqrt{\mathcal{E}_{\overline{H}}^{(\gamma)}}  \mathcal{E}_{\mathrm{cloud}} + \sqrt{ \mathcal{E}_{\overline{H}}^{(\gamma)}}  \mathcal{D}_{\mathrm{cloud}}.
\end{align}
\end{proposition}

\subsubsection{Elliptic estimates}


The following proposition is our main elliptic estimate. 
\begin{proposition}  \label{pro:ell:intro} Assume the bootstrap assumption and recall the definitions \eqref{E_Int}, \eqref{E_IntCh}. The $\mathcal{J}_{ell}$ functionals satisfy the following bounds: 
\begin{align}
\mathcal{J}_{ell}^{(1)} \lesssim& e^{-\nu^{-1/9}}\lf( \mathcal{E}^{(\gamma)} +\lf(\mathcal{E}_{H}^{(\gamma)}+\mathcal{E}_{H}^{(\al)}\rg)\mathcal{E}_{\mathrm{Int}}^{\mathrm{low}}\rg), \label{jell:1}\\
\mathcal{J}_{ell}^{(2)} \lesssim& e^{-\nu^{-1/9}}\lf(\mathcal{E}^{(\gamma)} +\lf(\mathcal{E}_{H}^{(\gamma)}+\mathcal{E}_{H}^{(\al)}\rg)\mathcal{E}_{\mathrm{Int}}^{\mathrm{low}}\rg), \label{jell:2}\\
\mathcal{J}_{ell}^{(3)} \lesssim& e^{-\nu^{-1/9}}  \mathcal{D}^{(\gamma)} +e^{-\nu^{-1/9}}\sum_{\iota\in\{\al,\gamma\}}\lf(\mathcal{E}_{H}^{(\iota)}+\mathcal{D}_{H}^{(\iota)}\rg)\lf(\mathcal{E}^{(\gamma)}
+\mathcal{E}_{\mathrm{Int}}^{\mathrm{low}}\rg). \label{jell:3}
\end{align}
The $\mathcal{E}_{ell}$ functionals satisfy the following bounds: 
\begin{align} 
\mathcal{E}_{ell}^{(I, out)}(t)  \lesssim &  \frac{C_\mathfrak{n}}{ \lan t\ran^2(|k|t)^{2\mathfrak{n}}}\mathcal{E}_{\mathrm{Int}}^{\mathrm{low}}(t);\label{eell:out} \\
\mathcal{E}_{ell}^{(I, full)}(t) \lesssim & \frac{1}{\langle t \rangle^4} \mathcal{E}_{\mathrm{Int}}^{\mathrm{low}}(t) \exp\{-\delta_I\nu^{1/3}t\}.  \label{eell:sob} 
\end{align}
The $\mathcal{F}_{ell}^{(E)}$ functional satisfies the following bounds: 
\begin{align}
\mathcal{F}_{ell}^{(E)}(t) \lesssim &  e^{-\nu^{-1/8}}\lf(\mathcal{E}^{(\gamma)} +\mathcal{E}_{\mathrm{Int}}^{\mathrm{low}} \right) \lf(1+\mathcal{E}_H^{(\al)} +\mathcal{E}_H^{(\gamma)}\rg)^2. \label{fell:e}
\end{align}
\end{proposition}
\begin{remark} The bounds above are split into three groups. Informally, the first, \eqref{jell:1} -- \eqref{jell:3} measure information moving from exterior to exterior. The second, \eqref{eell:out} -- \eqref{eell:sob} captures interior to exterior. The third, \eqref{fell:e} captures exterior to interior.
The interior to interior estimates are essentially in \cite{HI20,BM13}; see Section \ref{sec:interior} for more information. 
\end{remark}

\vspace{2 mm}

\subsection{Proof of Proposition \ref{prop:boot}}

\subsubsection{Proof of Propositions Invoked from \cite{BHIW24b}}

\begin{proof}[Proof of Proposition \ref{thm:main:para}] We apply Theorem 2.1 from \cite{BHIW24b} to the equation~\eqref{M1a} with the forcing term $f=\nabla^{\perp}\stm_{\neq}\cdot\nabla \omega$ to obtain
\begin{align} \n
&\frac{d}{d t} \mathcal{E}^{(\gamma)}[\omega]  + \mathcal{D}^{(\gamma)}[\omega] + \mathcal{CK}^{(\gamma; \varphi)}[\omega] + \mathcal{CK}^{(\gamma; W)}[\omega] 
 \\ \label{lights:on:15}
 & \qquad \lesssim     |\mathcal{I}^{(\gamma)}_{\mathrm{Source}}| +  \eps^3 \nu^{98}
+\eps \paren{\cd^{(\gamma)}[\omega]+\mathcal{CK}^{(\gamma)}[\omega] + \mathcal{D}^{(\gamma)}_{\overline{H} }+ \mathcal{D}^{(\gamma)}_{{H}}}, \\ \n
 &\frac{d}{d t} \mathcal{E}^{(\alpha)}[\omega]  + \mathcal{D}^{(\alpha)}[\omega] + \mathcal{CK}^{(\alpha; \varphi)}[\omega] + \mathcal{CK}^{(\alpha; W)}[\omega] 
 \\ \label{lights:on:25}
 & \qquad \lesssim  \sum_{\iota \in (\gamma, \alpha, \mu)} |\mathcal{I}^{(\iota)}_{\mathrm{Source}}| +  \eps^3 \nu^{98}
+\eps \sum_{\iota \in (\alpha, \gamma)}\paren{\cd^{(\iota)}[\omega]+\mathcal{CK}^{(\iota)}[\omega] + \mathcal{D}^{(\iota)}_{\overline{H} }+ \mathcal{D}^{(\iota)}_{{H}}},  \\ \n
 &\frac{d}{d t} \mathcal{E}^{(\mu)}[\omega]  + \mathcal{D}^{(\mu)}[\omega] + \mathcal{CK}^{(\mu; \varphi)}[\omega] + \mathcal{CK}^{(\mu; W)}[\omega] 
 \\ \label{lights:on:35}
 & \qquad \lesssim  |\mathcal{I}^{(\mu)}_{\mathrm{Source}}| +  \eps^3 \nu^{98}
 +\eps \paren{\cd^{(\mu)}[\omega]+\mathcal{CK}^{(\mu)}[\omega] + \mathcal{D}^{(\gamma)}_{\overline{H} }+ \mathcal{D}^{(\gamma)}_{{H}}},
\end{align}
where the source terms are given as follows:
\begin{align*}
	\mathcal{I}^{(L, \gamma)}_{\mathrm{Trilinear}} := & - \sum_{m = 0}^\infty \sum_{n = 0}^\infty \bold{a}_{m,n}^2 \theta_n^2 \mathrm{Re} \langle q^n \p_x^m \Gamma^n (\nabla^\perp \stm_{\neq} \cdot \nabla \omega), \omega_{m,n} \chi_{m + n}^2 e^{2W} \rangle, \\ 
	\mathcal{I}^{(L, \alpha)}_{\mathrm{Trilinear}} := &- \sum_{m = 0}^\infty \sum_{n = 0}^\infty \bold{a}_{m,n}^2 \theta_n^2 \nu \mathrm{Re} \langle q^n \p_x^m \p_y \Gamma^n  (\nabla^\perp \stm_{\neq} \cdot \nabla \omega), \p_y \omega_{m,n} \chi_{m + n}^2 e^{2W} \rangle, \\ 
	\mathcal{I}^{(L, \mu)}_{\mathrm{Trilinear}} := & - \sum_{m = 0}^\infty \sum_{n = 0}^\infty \bold{a}_{m,n}^2 \theta_n^2 \nu \mathrm{Re} \langle q^n \p_x^m \p_x \Gamma^n  (\nabla^\perp \stm_{\neq} \cdot \nabla \omega), \p_x \omega_{m,n} \chi_{m + n}^2 e^{2W} \rangle.
\end{align*}
Noting that 
\begin{align*}
	\nu \lesssim \frac{1}{\brak{t}^2}
\end{align*}
for $t<\nu^{-1/3-}$ and $\eps\ll 1$, we obtain \eqref{lights:on:1} -- \eqref{lights:on:3}.

\end{proof}

\begin{proof}[Proof of Proposition \ref{pro:ell:intro}] First of all, we recall  that in our companion work \cite{BHIW24b}, it was proven in Theorem 3.1 that under suitable conditions on the coordinate system (covered by the bootstrap hypotheses \eqref{boot:Intf},\eqref{boot:IntH},\eqref{boot:ExtVort}, and \eqref{boot:H}, which are listed in (3.23) in \cite{BHIW24b}) one can obtain a variety of elliptic estimates. More precisely we have the following.   
\begin{theorem} [Theorem 3.1 in \cite{BHIW24b}]
Assume that the bootstrap hypotheses \eqref{boot:Intf},\eqref{boot:IntH},\eqref{boot:ExtVort}, and \eqref{boot:H} hold on $[0,T]$. The $\mathcal{J}_{ell}$ functionals satisfy the following bounds on the same time interval: 
\begin{align*}
\mathcal{J}_{ell}^{(1)} \lesssim& e^{-\nu^{-1/9}}\lf( \mathcal{E}^{(\gamma)}[\ww] +\lf(\mathcal{E}_{H}^{(\gamma)}+\mathcal{E}_{H}^{(\al)}\rg)\mathcal{E}_{\mathrm{Int}}^{\mathrm{low}}[\ww]\rg),\\
\mathcal{J}_{ell}^{(2)} \lesssim& e^{-\nu^{-1/9}}\lf(\mathcal{E}^{(\gamma)}[\ww] +\lf(\mathcal{E}_{H}^{(\gamma)}+\mathcal{E}_{H}^{(\al)}\rg)\mathcal{E}_{\mathrm{Int}}^{\mathrm{low}}[\ww]\rg), \\
\mathcal{J}_{ell}^{(3)} \lesssim& e^{-\nu^{-1/9}}  \mathcal{D}^{(\gamma)}[\ww] +e^{-\nu^{-1/9}}\sum_{\iota\in\{\al,\gamma\}}\lf(\mathcal{E}_{H}^{(\iota)}+\mathcal{D}_{H}^{(\iota)}\rg)\lf(\mathcal{E}^{(\gamma)}[\ww]
+\mathcal{E}_{\mathrm{Int}}^{\mathrm{low}}[\ww]\rg). 
\end{align*}
The $\mathcal{E}_{ell}$ functionals satisfy the following bounds: 
\begin{align*} 
\mathcal{E}_{ell}^{(I, out)}(t)  \lesssim &  \frac{C_\mathfrak{n}}{ \lan t\ran^2(|k|t)^{2\mathfrak{n}}}\mathcal{E}_{\mathrm{Int}}^{\mathrm{low}}[\ww]; \\
\mathcal{E}_{ell}^{(I, full)}(t) \lesssim & \frac{1}{\langle t \rangle^4} \mathcal{E}_{\mathrm{Int}}^{\mathrm{low}}[\ww].  
\end{align*}
Finally, the $\mathcal{F}_{ell}^{(E)}$ functional satisfies the following bound: 
\begin{align*}
\mathcal{F}_{ell}^{(E)}(t) \lesssim &  e^{-\nu^{-1/8}}\lf(\mathcal{E}^{(\gamma)}[\ww] + \mathcal{E}_{\mathrm{Int}}^{\mathrm{low}}[\ww] \right)\lf(1+\mathcal{E}_H^{(\al)}  +\mathcal{E}_H^{(\gamma)}\rg)^2. 
\end{align*}
\end{theorem}

Finally, by setting $w$ in \cite{BHIW24b} to be the $-\omega$ here, we obtain all the bounds in Proposition \ref{pro:ell:intro}. We highlight that the notations in these papers are chosen such that $\mathcal{E}^{(\gamma)}[\ww]\big|_{\ww=\omega},\ \mathcal{D}^{(\gamma)}[\ww]\big|_{\ww= \omega}$ and  $\mathcal{E}^{\text{low}}_{\text{Int}}[\ww]\big|_{w=\omega}$ correspond to $\mathcal{E}^{(\gamma)},\  \mathcal{D}^{(\gamma)}$ and $\mathcal{E}^{\text{low}}_{\text{Int}}$ in this paper.

\end{proof}

\subsubsection{Closing the Main Bootstrap}

We are now ready to bring together all of the above estimates to close the main estimate. 

\begin{proof}[Proof of Proposition \ref{prop:boot}] 
By Proposition~\ref{prop:MainInterior}, using the bootstrap assumptions, we easily get
\begin{align} \n
	\frac{1}{2}\frac{d}{dt}\mathcal{E}_{\mathrm{Int}}(t) + \mathcal{CK}_{\mathrm{Int}}(t) +  \mathcal{D}_{\mathrm{Int}}(t) + \delta \nu^{1/3} \mathcal{E}_{\mathrm{Int}}(t) &\\ \n
	& \hspace{-6cm} \lesssim  \eps^2\mathcal{CK}_{\mathrm{Int,Coord}}(t) + \eps^2\nu^{100},
\end{align}
from where we integrate in time and use the bootstrap assumptions to get inequality~\eqref{boot:Intf} is true with $4\eps^2$ replaced with $2\eps^2$ on the right hand side, where we also used that $\nu\ll 1$.  
In order to prove the right side of~\eqref{boot:IntH} is bounded by $2\eps^2$, we get from Proposition~\ref{prop:MainCoordInt},  \eqref{fell:e}, and the bootstrap assumption that
\begin{align*} 
	\frac{1}{2}\frac{d}{dt} \mathcal{E}_{\text{Int,Coord}} +  \mathcal{CK}_{\text{Int,Coord}} +  \mathcal{D}_{\text{Int,Coord}}  \leq \eps^2\nu^{100} +  \eps\mathcal{CK}_{\text{Int}}
		+ \frac{\eps^3}{\brak{t}^{1 + K\eps}}.
\end{align*}
Again the estimate for the interior coordinate system closes by integrating in time and using that $\nu\ll 1.$ 
We next turn to the estimate of the exterior vorticity. By~Proposition~\ref{thm:main:para} and \ref{pro:tri:in}, we obtain
\begin{align*}
	\frac{d}{d t} \mathcal{E}^{(\gamma)}[\omega]&  + \mathcal{D}^{(\gamma)}[\omega] + \mathcal{CK}^{(\gamma; \varphi)}[\omega] + \mathcal{CK}^{(\gamma; W)}[\omega] 
	 \\& \lesssim   \sum_{L \in \{I, E\}}  |\mathcal{I}^{(L, \gamma)}_{\mathrm{Trilinear}}| +  \eps \cd_{H}^{(\gamma)} +\eps\cd_{\overline H}^{(\gamma)} + \frac{\eps^3}{\langle t \rangle^2}
	 \\&
	 \lesssim \eps^2\nu^{100} +  \eps\mathcal{CK}^{(\gamma)} + \eps\mathcal{CK}_{(\text{cloud)}} + \eps \cd^{(\gamma)}+ \eps \cd_{H}^{(\gamma)} +\eps\cd_{\overline H}^{(\gamma)} + \frac{\eps^3}{\langle t \rangle^2}
\end{align*}
which we integrate in time to get
\begin{align*}
			\sup_{0 \le t \le T} \mathcal{E}^{(\gamma)} + \sum_{\iota \in \{\varphi, W \}} \int_0^T \mathcal{CK}^{(\gamma; \iota)}(t) dt +  \int_0^T \mathcal{D}^{(\gamma)}(t) dt \le 2 \eps^2.
\end{align*}
The $\alpha$ and $\mu$ level estimates follow similarly and we omit further details. Using Proposition~\ref{gamm:esti:llp} and the bootstrap assumptions, we integrate in time to get \eqref{boot:H} holds for an upper bound $2\eps^2$. By Proposition~\eqref{pro:int:sob} and the bootstrap assumptions, we arrive at
\begin{align*}
	\frac{\p_t}{2} \mathcal{E}_{\mathrm{sob}} &+ \mathcal{CK}_{\mathrm{sob}}^{(W)} + \mathcal{D}_{\mathrm{sob}}  + \frac{1}{M} \paren{\frac{\p_t}{2} \mathcal{E}_{\mathrm{Trace}}(t) +   \mathcal{CK}_{\mathrm{Trace}}(t) + \mathcal{D}_{\mathrm{Trace}}(t)}
	\\& \lesssim \frac{\eps^3}{\brak{t}^{100}} + \eps^2\nu^{100}
\end{align*}
for some $M\gg 1.$
Integrating in time gives
\begin{align*}
	\frac{\p_t}{2} \mathcal{E}_{\mathrm{sob}} &+ \mathcal{CK}_{\mathrm{sob}}^{(W)} + \mathcal{D}_{\mathrm{sob}}  + \frac{1}{M} \paren{\frac{\p_t}{2} \mathcal{E}_{\mathrm{Trace}}(t) +   \mathcal{CK}_{\mathrm{Trace}}(t) + \mathcal{D}_{\mathrm{Trace}}(t)}
	\lesssim \eps^3,
\end{align*}
concluding the bootstrap argument for the Sobolev norm.
Similarly, we integrate in time \eqref{bhbhy:1:llp} and get \eqref{boot:cloud} holds with $4\eps^2$ replaced by $2\eps^2$ concluding the proof.
\end{proof}


\section{Interior Estimates}
\label{sec:interior}
In this section we prove Propositions \ref{prop:MainInterior}  and \ref{prop:MainCoordInt}. 
As discussed in Section \ref{sec:Outline}, upon moving to the profile representation in the $(z,v)$ coordinates we have 
\begin{align}
\label{main:equa}
\partial_t f + g \partial_v f + v' \grad^\perp \psi_{k \neq 0} \cdot \grad f = \nu \widetilde{\Delta_{t}} f. 
\end{align}
Defining $f^I = \chi^I f$, 
\begin{align*}
\partial_t f^I + g \partial_v f^I + v' \grad^\perp \psi_{k \neq 0} \cdot \grad f^I - \nu \widetilde{\Delta_t} f^I = \mathcal{C},  
\end{align*}
with the commutator given by 
\begin{align*}
\mathcal{C} =  g \partial_v \chi^I f - v' \partial_v \chi^I  \partial_x \psi_{k \neq 0} f - \nu (v')^2 \partial_{vv}\chi^I f - 2 \nu (v')^2 \partial_v\chi^I (\partial_v-t\partial_z) f. 
\end{align*}
The left-hand side has now been essentially reduced to the profile equation treated in \cite{BM13, HI20}, except for the fact that the Biot-Savart law includes contributions from both $\chi^I f$ and $(1-\chi^I)f$. 
The commutator $\mathcal{C}$ is supported near $\frac{3}{4} \leq \abs{y}  \leq \frac{31}{40}$, which means that the exterior estimate will be extremely strong there due to the weight (see Lemma \ref{lem:ExtToInt} below). For this reason, the extension of the work in \cite{BM13, HI20} to cover the bootstrap estimate on $f^I$ is not too big of a challenge, and hence we will provide only a sketch here, focusing on the places where the arguments used herein depart from those in \cite{BM13, HI20}. 

Next, we further localize the PDEs for the coordinate systems, defining
\begin{align*}
g^I = \chi^I g, \quad h^I  = \chi^I h, \quad \overline{h}^I = \overline{h} \chi^I,
\end{align*}
and the corresponding PDEs become
\begin{align*}
& \partial_t \overline{h}^I + \frac{2}{t} \overline{h}^I + g \partial_v \overline{h}^I + \left(\chi^I v' \grad^\perp \psi_{k \neq 0} \cdot \grad \overline{h}\right)_0 - \nu \widetilde{\Delta_t} \overline{h}^I = \mathcal{C}_{\overline{h}} \\
& \partial_t g^I + \frac{2}{t} g^I + g \partial_v g^I + \left(\chi^I v' \grad^\perp \psi_{k \neq 0} \cdot \grad u\right)_0 - \nu \widetilde{\Delta_t} g^I = \mathcal{C}_g \\
& \partial_t h^I + + g \partial_v h^I - \nu \widetilde{\Delta_t} h^I - \overline{h}^I = \mathcal{C}_{h}, 
\end{align*}
where the commutators above are of the form
\begin{align*}
\mathcal{C}_{\ast} = g \partial_v\chi^I (\ast) - \nu (v')^2 \partial_{vv}\chi^I  (\ast) - 2 \nu (v')^2 \partial_v \chi^I \partial_v (\ast). 
\end{align*}
We subdivide the effective velocity field into two terms 
\begin{align*}
& U^I = \begin{pmatrix} 0 \\ g^I \end{pmatrix} + \chi_e^I(1 + h^I)\grad^\perp \psi^I,
\end{align*}
and 
\begin{align*}
& U^E = \begin{pmatrix} 0 \\ \chi^I_e (1-\chi^I)g \end{pmatrix} + \chi_e^I (1+h^I)\grad^\perp \psi^E + \chi_e^I(1-\chi^I)h\grad^\perp \psi^I,
\end{align*}
where recall the definitions of $\psi^E$ \eqref{Intr_psi_E} and $\psi^I$ \eqref{Intr_psi_I}. 
\subsection{Interior Fourier multiplier norms}
We define the Fourier multipliers $A$ and $A_R$ as in \cite{BM13}; one can also use \cite{HI20} in order to treat the $r=1/2$ case (although we use $w$ to denote the $b$ used therein).
Specifically, these multipliers are of the form
\begin{align*}
A(t,k,\eta) & = e^{\lambda(t)\abs{\grad}^r}\left(\frac{e^{\mu \abs{\eta}^{1/2}}}{w(t,k,\eta)} + e^{\mu \abs{k}^{1/2}}\right) \\
A_R(t,k,\eta) & = e^{\lambda(t)\abs{\grad}^r}\left(\frac{e^{\mu \abs{\eta}^{1/2}}}{w_R(t,k,\eta)} + e^{\mu \abs{k}^{1/2}}\right) \\
\tilde{A}(t,k,\eta) &= e^{\lambda(t)\abs{\grad}^r} \frac{e^{\mu \abs{\eta}^{1/2}}}{w(t,k,\eta)}, 
\end{align*}
where we can use the Fourier multipliers $w,w_R$  as defined in either \cite{BM13} (or \cite{HI20} where their upgraded version is called $b$), and we will choose $\dot{\lambda}(t) \approx -\frac{\eps}{\brak{t}^{1+\delta'}}$ for some small $\delta'>0$.

We also introduce a Fourier multiplier to deduce the enhanced dissipation (as introduced in \cite{BGM15III}), which we define as 
\begin{align*}
M(t,k,\eta) & := \begin{cases} \exp \left( \int_0^t \frac{\nu^{1/3}}{1 + \left(\nu^{1/3} \abs{\eta - k \tau} \right)^{1+\kappa} } d\tau \right) & \quad k \neq 0 \\
1 & \quad k = 0, 
\end{cases}
\end{align*}
where $\kappa> 0$ is a parameter which can be chosen arbitrarily small.

Finally, the main Fourier multiplier norm we apply to the profile is the combination of all the main ideas; for some sufficiently small $\delta_I > 0$, 
\begin{align} \label{frak:A}
\mathfrak{A} := \frac{A}{M} \left(e^{\delta_I \nu^{1/3}t} \mathbf{1}_{k \neq 0} + \mathbf{1}_{k = 0}\right). 
\end{align}

\subsection{Some basic properties of $\mathfrak{A}$}

Aside from the more delicate estimates on $A$ from \cite{BM13,HI20}, we also make use of a few rougher estimates which are convenient for treating the commutator terms.
The proofs are omitted as they are straightforward from the content of \cite{BM13,HI20}. 
\begin{lemma} \label{lem:BasicA}
There holds the pointwise bounds for all $j \geq 0$,
\begin{align}
\brak{k,\eta}^j\left(A(t,k,\eta) + A_R(t,k,\eta)\right)   & \lesssim_{j,\lambda} e^{2\lambda(t) \abs{k,\eta}^r} \\
A_R(t,k,\eta) & \lesssim \brak{t} A(t,k,\eta). \label{ineq:ARtrivBd}
\end{align}
Moreover, for  any sufficiently regular functions $\mathfrak{f}, \mathfrak{g}$, the following product rules hold
\begin{align}
\norm{A(\mathfrak{f} \mathfrak{g})}_{L^2} & \lesssim \brak{t} \norm{A \mathfrak{f}}_{L^2}\norm{A \mathfrak{g}}_{L^2} \label{ineq:Aprod} \\ 
\norm{A_R(\mathfrak{f} \mathfrak{g})}_{L^2} & \lesssim \norm{A_R \mathfrak{f}}_{L^2}\norm{A_R \mathfrak{g}}_{L^2}. \label{ineq:ARprod}
\end{align}
\end{lemma} 
The use of the multiplier $M$ to obtain the exponential decay at $k \neq 0$ slightly departs from the arguments in \cite{BMV14}, and so we will need to provide some additional estimates.
As in \cite{BGM15III}, to obtain the exponential decay estimates using $M$, we apply the following straightforward estimate.
\begin{lemma}
There holds the following uniformly in $t,\eta,\nu$ and $k \neq 0$ 
\begin{align*}
\nu^{1/3} \lesssim_\kappa \frac{\partial_t M(t,k,\eta)}{M(t,k,\eta)} + \nu \abs{k,\eta-kt}^2.  
\end{align*}
\end{lemma}
As far as the rest of the arguments of \cite{BMV14} are concerned, since $M \approx 1$ (a so-called ``ghost multiplier''), the presence of $M$ will mainly only affect the `transport' commutator estimate in [Section 5, \cite{BM13}]; see  Section \ref{sec:NIVE} below. 
We will need the following lemma. 
\begin{lemma} \label{def:ICommM}
Suppose that $\abs{k,\eta} \approx \abs{\ell,\xi}$. Then, 
\begin{align*}
\abs{\frac{M(t,k,\eta)}{M(t,\ell,\xi)} - 1} \lesssim  \brak{k-\ell}\min\left(\frac{t}{\abs{\eta}}, \frac{1}{\abs{k}} \right). 
\end{align*}
\end{lemma}
\begin{proof}
We start by observing, 
\begin{align*}
\abs{\frac{M(t,k,\eta)}{M(t,\ell,\xi)} - 1} \lesssim  \abs{\int_0^t \frac{\nu^{1/3}}{1 + \left(\nu^{1/3} \abs{\eta - k \tau} \right)^{1+\kappa} } d\tau - \int_0^t \frac{\nu^{1/3}}{1 + \left(\nu^{1/3} \abs{\xi - \ell \tau} \right)^{1+\kappa} } d\tau }.  
\end{align*}
For $t < \frac{1}{2}\min( \eta/k , \xi/\ell )$ we can brute force estimate
\begin{align*}
\int_0^t \frac{\nu^{1/3}}{1 + \left(\nu^{1/3} \abs{\eta - k \tau} \right)^{1+\kappa} } d\tau \lesssim \frac{\nu^{1/3} t}{1 + \nu^{1/3}\max(\abs{kt},\abs{\eta})},
\end{align*}
and similarly for $\ell,\xi$.
On the other hand, suppose $t \geq \frac{1}{2}\min( \eta/k , \xi/\ell )$.
Without loss of generality, assume that $\eta/k < \xi/\ell$.
Then, by straightforward integration, 
\begin{align*}
\int_0^t \frac{\nu^{1/3}}{1 + \left(\nu^{1/3} \abs{\eta - k \tau} \right)^{1+\kappa} } d\tau \lesssim \frac{1}{\abs{k}} \lesssim \min \left( \frac{1}{\abs{k}}, \frac{t}{\abs{\eta}} \right), 
\end{align*}
and 
\begin{align*}
\int_0^t \frac{\nu^{1/3}}{1 + \left(\nu^{1/3} \abs{\xi - \ell \tau} \right)^{1+\kappa} } d\tau \lesssim \frac{1}{\abs{\ell}} \lesssim \frac{\brak{k-\ell}}{\brak{k}} \lesssim \brak{k-\ell} \min \left( \frac{1}{\abs{k}}, \frac{t}{\abs{\eta}} \right).   
\end{align*}
This proves the lemma. 
\end{proof} 

\subsection{Exterior-to-interior estimates}

In order to relate the exterior and interior norms, we first record the following lemma which relates the Fourier-side Gevrey regularity norms to the physical-side infinite series norms we use in the exterior. 
\begin{lemma} \label{lem:FourierToPhysical}
Let $r \in (0,1)$ be fixed.
There exists a universal constant $\mathfrak{C}>1$ such that for any function $g(\cdot)\in C_c^\infty(\mathbb{T}\times(-\frac{99}{100},\frac{99}{100}))$  and $\lambda > 0$ there holds
\begin{align}
\| e^{\lambda \abs{\grad}^r} g\|_{L^2}^2
\lesssim \sum_{m+n=0}^\infty \frac{\mathfrak {C}^{m+n} \lambda^{(m+n)/r}}{((m+n)!)^{2/r}}\|\partial_z^m \partial_v^n g\|_{L^2}^2. 
\end{align}
\end{lemma}
\begin{proof}
By rescaling $k$ and $\eta$, we can reduce to the case $\lambda=1$ without loss of generality. 

First of all, we observe the relation $(x+y)^{r}\leq x^{r}+y^{r},\, x, y\geq 0,\, r\in(0,1)$. This is a consequence of the concavity of the function $(\cdot)^{r}$. 
We estimate the left hand side as follows,
\begin{align*}
\|\exp\{(|k|+|\eta|&)^{r}\}\widehat g( k,\eta)\|_{L^2_\eta}^2\leq \|\exp\{|k|^r+|\eta|^r\}\widehat g(k,\eta)\|_{L_\eta^2}^2\\
=&\lf\|\sum_{m=0}^\infty \sum_{n=0}^\infty \frac{|k|^{rm}}{m!}\frac{|\eta|^{rn}}{n!}\widehat{g}(k,\eta)\rg\|_{L_\eta^2}^2\\
\leq&\lf(\sum_{m=0}^\infty \sum_{n=0}^\infty \lf(\frac{2}{ 3 }\rg)^{m+n}\rg) \lf(\sum_{m=0}^\infty \sum_{n=0}^\infty\lf(\frac{3}{2}\rg)^{m+n}\frac{1}{(m!)^2(n!)^2} \lf\||k|^{rm}  |\eta|^{rn} \widehat{g}(k,\eta)\rg\|_{L_\eta^2}^2\rg)\\
\lesssim&  \lf(\sum_{m=0}^\infty \sum_{n=0}^\infty\lf(\frac{3}{2}\rg)^{m+n}\frac{1}{(m!)^2(n!)^2} \lf\||k|^{m}  |\eta|^{n} \widehat{g}(k,\eta)\rg\|_{L_\eta^2}^{2r}\|\widehat{g}(k,\eta)\|_{L_\eta^2}^{2-2r}\rg), 
\end{align*}
where the last line followed by H\"older's inequality. Note that there has been no loss of $(m,n)$ dependent constants.
By H\"older's inequality again we have 
\begin{align*}
\|\exp\{(|k|^2+|\eta|^2&)^{r/2}\}\widehat g( k,\eta)\|_{L_\eta^2}^2\\
\lesssim &\lf(\sum_{m=0}^\infty \sum_{n=0}^\infty \frac{ 3^{(m+n)/r}}{(m!)^{2/r}(n!)^{2/r}} \lf\||k|^{m}  |\eta|^{n} \widehat{g}(k,\eta)\rg\|_2^{2}\rg)^{r}\lf(\sum_{m=0}^\infty\sum_{n=0}^\infty \frac{1}{2^{(m+n)/(1-r)}}\|\widehat{g}(k,\eta)\|_{L_\eta^2}^{2}\rg)^{{1-r}}\\\lesssim &\sum_{m=0}^\infty \sum_{n=0}^\infty \frac{ 3^{(m+n)/r}}{(m+n)!^{2/r}}\underbrace{\frac{(m+n)!^{2/r}}{(m!)^{2/r}(n!)^{2/r}}}_{=\binom{m+n}{n}^{2/r}} \lf\||k|^{m}  |\eta|^{n} \widehat{g}(k,\eta)\rg\|_{L_\eta^2}^{2}.
\end{align*}
Now we note the identity $\sum_{\ell=0}^M \binom{M}{\ell}=2^{M}$, and so $\binom{m+n}{n}\leq 2^{m+n}$. As a result,
\begin{align*}
\|\exp\{(|k|^2+|\eta|^2&)^{r/2}\}\widehat g( k,\eta)\|_{L_\eta^2}^2\lesssim \sum_{m=0}^\infty\sum_{n=0}^\infty \frac{12^{(m+n)/r}}{(m+n)!^{2/r}} \lf\||k|^{m}  |\eta|^{n} \widehat{g}(k,\eta)\rg\|_{L_\eta^2}^{2}, 
\end{align*}
which completes he proof. 
\end{proof}

The next lemma is the fundamental one that relates the exterior estimates to the interior estimates, which shows that in the gluing region (where $\partial_v \chi^I$ is non-vanishing), the exterior norm is much stronger than the interior norm even though the radius of Gevrey regularity of exterior region is smaller for large times due to the large spatial weight $e^W$.
\begin{lemma} \label{lem:ExtToInt}
  Consider $\mathfrak {r}\in( \frac{4}{5},1), \, t\in[0, \nu^{-1/3-\zeta}],\, \zeta \leq \frac{1}{78}$ and $|y|\in[\frac{29}{40},\frac{7}{8}]$. Then, for any $\mathfrak{g}$ sufficiently regular there holds 
\begin{align}
e^{2\delta_I \nu^{1/3} t} \sum_{m+n=0}^\infty \frac{\wt \lambda ^{2(m+n)/\wt{ \mathfrak{r}}}  }{ ((m+n)!)^{2/\wt{ \mathfrak{r}}}}\lf\|\mathbbm{1}_{|y|\in \lf[\frac{29}{40}, \frac{7}{8}\rg]}|k|^m \Gamma_k^n \mathfrak{g}_k \rg\|_{L^2}^2 & \notag \\ & \hspace{-5cm} \lesssim e^{-\nu^{-1/8}} \sum_{m+n=0}^\infty\frac{\lambda ^{2(m+n)/\mathfrak{r}} \varphi^{2n+2}}{((m+n)!)^{2/\mathfrak{r}}}\|e^{W/10}\chi_{m+n}|k|^m q^n \Gamma_k^n \mathfrak{g}_k \|_{L^2}^2,\label{glu_rl}\\
& \hspace{-6cm}  \wt{\mathfrak{r}}=\frac{34\mathfrak{r}}{19(2 +\mathfrak{r})},\quad\wt\lambda(t)=\mathcal{C}(\mathfrak{r},\wt {\mathfrak{r}},K) \lambda(t)^{\wt{\mathfrak{r}}/\mathfrak{r}}.\n
\end{align} 
Here $K$ is the parameter in the definition of $W$ \eqref{defndW} and the implicit constant depends on $K$. 
Moreover, if $\mathfrak{r}\geq 4/5$, we have that  $\wt{\mathfrak{r}}>1/2$. 
\end{lemma}
\begin{proof}
Using that $|y|\in[\frac{29}{40},\frac{7}{8}]$, on the time interval  $t\in[0,\nu^{-1/3-\zeta}]$, we have the following estimate involving the $e^{W}$ weight \eqref{defndW} and $\varphi$ \eqref{varphi}, 
\begin{align*}
\exp&\lf\{-2\delta_I\nu^{1/3}t-\nu^{-1/8}\rg\}\exp\lf\{2\frac{W}{10}\rg\}\varphi^{2n+2}\geq  \exp\lf\{-2\delta_I\nu^{-\zeta}-\nu^{-1/8}+\frac{1}{5K\nu(1+t)}\rg\}\varphi^{2(m+n)+2}\\
\geq& \frac{1}{C}\exp\lf\{\frac{1}{8K\nu (1+t)}\rg\}(1+t^2)^{-m-n-1}\\
\geq & \frac{1}{C} \exp\lf\{\frac{1}{100K\nu(1+t)}\rg\}\ \frac{1}{N!}\frac{1}{(9K)^N\nu^{N} (1+t)^{N}}\  \frac{1}{(1+t^2)^{ m+n }}\ \frac{1}{1+t^2}\\
\geq & \frac{1}{C}\exp\lf\{\frac{1}{200K\nu^{2/3-\zeta}}\rg\}\frac{1}{N!}\frac{1}{\mathcal{G}^N2^{m+n}}\nu^{-\lf(\frac{2}{3}-\zeta \rg)N+\lf(\frac{2}{3}+2\zeta \rg)(m+n) },\qquad\qquad |y|\in\lf[\frac{29}{40},\frac{7}{8}\rg],
\end{align*}
where $\mathcal{G}>1$ is a constant depending on the parameter $K$ in the weight $W$.
By setting $N=\lf\lfloor\frac{2/3+2\zeta}{2/3-\zeta}(m+n)\rg\rfloor$, we have
\begin{align*}
e^{-2\delta_I\nu^{1/3}t-\nu^{-1/8}} e^{W/5}\varphi^{2n+2}
\geq\frac{1}{C\mathcal{G}^{\lf\lfloor\frac{2/3+2\zeta}{2/3-\zeta}(m+n)\rg\rfloor}2^{m+n}\lf\lfloor \frac{2/3+2\zeta}{2/3-\zeta}(m+n)\rg\rfloor!}. 
\end{align*}
By invoking the facts that $\chi_{m+n}\equiv 1,\ \forall \ m+n\in \mathbb{N},  \ |y|\in[29/40,7/8]$ (see the definition in \eqref{chi:prop:1}),
and that the co-normal weight $q\equiv 1,\ \forall |y|\in[29/40,7/8]$, we obtain the following
\begin{align}e^{-2\delta_I\nu^{1/3}t-\nu^{-1/8}}&\sum_{m+n=0}^\infty\frac{\lambda^{2(m+n)/\mathfrak{r}}\varphi^{2n+2}}{((m+n)!)^{2/\mathfrak{r}}}\|e^{W/10}\chi_{m+n}|k|^m q^n \Gamma_k^n \omega_k \|_{L^2}^2\n \\
\gtrsim& \sum_{ m+n=0}^\infty \frac{\lambda^{2(m+n)/\mathfrak{r}} }{\mathcal{G}^{\lf\lfloor\frac{2/3+2\zeta}{2/3-\zeta}(m+n)\rg\rfloor}2^{m+n}((m+n)!)^{2/\mathfrak{r}}\lf\lfloor{\frac{2/3+2\zeta}{2/3-\zeta}(m+n)}\rg\rfloor!}\|\chi_{m+n}|k|^m \Gamma_k^n\omega_k\|_{L^2}^2\n \\
\gtrsim &\sum_{m+n=0}^\infty \frac{\lambda^{2(m+n)/\mathfrak{r}} }{(\sqrt{2}\mathcal{G})^{2(m+n)}\lf(\lf\lfloor\frac{2/3+2\zeta}{2/3-\zeta}(m+n)\rg\rfloor!\rg)^{2/\mathfrak{r}+1}} \|\mathbbm{1}_{|y|\in [29/40,7/8]} |k|^m \Gamma_k^n\omega_k\|_{L^2}^2. \label{Gmm_fnc}
\end{align}

To estimate the right hand side of \eqref{Gmm_fnc}, we use the Gamma function $\Gamma(n)=(n-1)!,\quad n\in \mathbb{N}\backslash \{0\}$ and the log convexity of the Gamma function:
\begin{align*}
\Gamma(\theta x_1+(1-\theta)x_2)\leq \Gamma(x_1)^{\theta}\Gamma(x_2)^{1-\theta},\quad \theta\in[0,1], x_1,x_2>0.
\end{align*}
Since $\zeta \leq \frac{1}{78} $,  so $\frac{2/3+2\zeta}{2/3-\zeta}\leq \frac{18}{17}$, and
\begin{align*}
\left\lfloor \frac{2/3+2\zeta}{2/3-\zeta}(m+n)\right\rfloor!=&\Gamma\lf(\lf\lfloor\frac{18}{17}(m+n)\rg\rfloor+1\rg)\\
\leq& \Gamma(m+n+1)^{\theta}\Gamma(2(m+n)+1)^{1-\theta},\quad \theta = \frac{2(m+n)-\lfloor\frac{18}{17}(m+n)\rfloor}{m+n}\geq\frac{16}{17}.
\end{align*}
Now we estimate the $\Gamma(2(m+n)+1)=(2(m+n))!$:
\begin{align*}
(2(m+n))!=\prod_{\ell_1=1}^{(m+n)}(2\ell_1)\prod_{\ell_2=0}^{(m+n-1)}(2\ell_2+1)\leq 2^{2(m+n)}((m+n)!)^2.
\end{align*}
Hence,
\begin{align*}
\left\lfloor \frac{2/3+2\zeta}{2/3-\zeta}(m+n)\right\rfloor!\leq& ((m+n)!)^{\theta}[(2(m+n))!]^{1-\theta} 
\leq  (m+n)![(2(m+n))!]^{\frac{1}{17}}\\
\leq & 2^{(m+n)\frac{ 2}{17} }((m+n)!)^{\frac{19}{17} }.
\end{align*}
Therefore, the right hand side of  \eqref{Gmm_fnc} has the following lower bound,
\begin{align*}
  \sum_{m+ n=0}^\infty &\frac{\lambda ^{2(m+n)/\mathfrak{r}} }{(\sqrt{2}\mathcal{G})^{2(m+n)}\lf(\lf\lfloor\frac{2/3+2\zeta}{2/3-\zeta}(m+n)\rg\rfloor!\rg)^{2/\mathfrak{r}+1}}
  \|\mathbbm{1}_{|y|\in [29/40,7/8]}|k|^m \Gamma_k^n\omega_k\|_{L^2}^2\\
\gtrsim&\sum_{m+ n=0}^\infty \frac{\lambda^{2(m+n)/\mathfrak{r}} }{(\sqrt{2}\mathcal{G})^{2(m+n)} 2^{\frac{2}{17}(m+n)(\frac{2}{\mathfrak{r}}+1)}((m+n)!)^{\frac{19}{17}(\frac{2}{\mathfrak{r}}+1)}}\|\mathbbm{1}_{|y|\in [29/40, 7/8]} |k|^m \Gamma_k^n\omega_k\|_{L^2}^2\\
\gtrsim&\sum_{ m+n=0}^\infty  \frac{ \lambda^{2(m+n)/\mathfrak{r}}  }{ ({C(\mf r)}\mathcal{G})^{2(m+n)} ((m+n)!)^{\frac{19}{17}(\frac{2}{\mathfrak{r}}+1)}}\|\mathbbm{1}_{|y|\in [29/40, 7/8]} |k|^m \Gamma_k^n\omega_k\|_{L^2}^2.
\end{align*} 
Hence the resulting Gevrey index $\wt{ \mathfrak{r}}$ and radius of analyticity are 
$$\frac{2}{\wt{\mathfrak{r}}}=\frac{19}{17}\lf(\frac{2}{\mathfrak{r}}+1\rg)\quad\Rightarrow\quad \wt {\mathfrak{r}}=\frac{34\mathfrak{r}}{19(2+\mathfrak{r})},\qquad \
\wt \lambda(t)=\frac{\lambda(t)^{\wt {\mathfrak{r}}/\mathfrak{r}}}{(C(\mf r)\mathcal{G})^{\wt{\mathfrak{r}}}}=:\mathcal{C}(\mathfrak{r},\wt {\mathfrak{r}},K) \lambda(t)^{\wt {\mathfrak{r}}/\mathfrak{r}}.$$ 
The last claim of the lemma is a consequence of the following computation, 
\begin{align*}
\mf r\geq 4/5\quad\Rightarrow\quad \wt{\mathfrak{r}} \geq\frac{34}{19(2\frac{5}{4}+{1})}\geq \frac{68}{133}>0.51>1/2.
\end{align*}
This completes the proof. 
\end{proof}

From Lemma \ref{lem:ExtToInt} we deduce the following observation. 
\begin{lemma} \label{lem:ExtToIntf}
For all $j \geq 0$, there holds
\begin{align*}
\norm{\brak{\grad}^j \mathfrak{A} \chi^I_e(1-\chi^I) f}_{L^2} \lesssim e^{-\nu^{-1/8}} (\mathcal{E}^{(\gamma)})^{1/2}.
\end{align*}
\end{lemma}
\begin{proof}
Recall Lemma \ref{lem:BasicA}, which implies
\begin{align*}
\norm{\brak{\grad}^j \mathfrak{A} \chi^I_e(1-\chi^I) f}_{L^2} & \lesssim e^{\delta_I \nu^{1/3}t} \norm{e^{2\lambda(t)\abs{\grad}^{r}} \chi^I_e(1-\chi^I) f_{\neq}}_{L^2} \\
& \quad + \norm{e^{2\lambda(t)\abs{\grad}^{r}} \chi^I_e(1-\chi^I) f_{0}}_{L^2}. 
\end{align*}
The two contributions are treated in the same manner. 
Lemma \ref{lem:FourierToPhysical} implies for some universal constant $\mathfrak{C}$
\begin{align*}
\sum_{k \neq 0} \norm{\brak{\grad}^j \mathfrak{A} \chi^I_e(1-\chi^I) f_k }_{L^2}^2 & \lesssim  \sum_{k \neq 0} \sum_{m+n=0}^\infty \frac{\mathfrak{C}^{m+n} (2\lambda)^{(m+n)/r}}{((m+n)!)^{2/r}}\||k|^m \partial_v^n \chi^I_e(1-\chi^I) f_k \|_{L^2}^2 \\
& \lesssim  \sum_{k \neq 0} \sum_{m+n=0}^\infty \frac{\mathfrak{C}^{m+n} (2\lambda)^{(m+n)/r}}{((m+n)!)^{2/r}}\||k|^m \Gamma^n_k \widetilde{\chi^I_e}(1-\widetilde{\chi^I}) \omega_k \|_{L^2}^2, 
\end{align*}
where in the last line we changed variables $v \mapsto y$, denoting $\widetilde{\chi^I_e} = \chi^I_e \circ v(t,y)$ (similarly for $\chi^I$). 
Therefore, by standard Gevrey product estimates (and the strength of the weight on the support of $\chi_e^I(1-\chi^I)$) and Lemma \ref{lem:ExtToInt}, we have 
\begin{align*}
e^{\delta_I \nu^{1/3}t} \norm{\brak{\grad}^j \mathfrak{A} \chi^I_e(1-\chi^I)f }_{L^2}^2 & \lesssim e^{-\nu^{-1/8}}  \mathcal{E}^{(\gamma)} \sum_{m+n=0}^\infty \frac{\mathfrak{C}^{m+n} (2\lambda)^{(m+n)/r}}{((m+n)!)^{2/r}}\| (v_y^{-1} \partial_y)^n \widetilde{\chi^I_e}(1-\widetilde{\chi^I}) \|_{L^2}^2\\
&= e^{-\nu^{-1/8}} \mathcal{E}^{(\gamma)} \sum_{m+n=0}^\infty \frac{\mathfrak{C}^{m+n} (2\lambda)^{(m+n)/r}}{((m+n)!)^{2/r}}\| \partial_v^n \left(\chi^I_e(1-\chi^I)\right) \|_{L^2}^2 \\ 
&\lesssim e^{-\nu^{-1/8}} \mathcal{E}^{(\gamma)},
\end{align*}
which completes the treatment of the $f_{\neq}$ contribution. The $f_0$ contribution is treated similarly. 
\end{proof}

We have similar estimates on the coordinate system unknowns.
\begin{lemma} \label{lem:ghExtToInt}
For all $j \geq 0$, there holds 
\begin{align*}
\norm{\brak{\grad}^j \mathfrak{A}_R \chi^I_e(1-\chi^I)g}_{L^2} & \lesssim e^{-\nu^{-1/8}}\mathcal{E}_{\overline{H}}^{1/2} + \frac{1}{\brak{t}^2}\mathcal{E}_{G}^{1/2} \\ 
\norm{\brak{\grad}^j \mathfrak{A}_R \chi^I_e(1-\chi^I)h}_{L^2} & \lesssim e^{-\nu^{-1/8}} \mathcal{E}_{H}^{1/2}. 
\end{align*}
\end{lemma}
\begin{proof}
This is a straightforward variation of Lemma \ref{lem:ExtToInt} so we omit the proof for brevity. 
\end{proof}


Before we continue, let us remark a few simple estimates on $\grad^\perp \psi^I$ that are useful for estimating the commutators.
The proof is omitted as it is immediate from the content of \cite{HI20,BM13}. 
\begin{lemma} \label{lem:AphiItriv}
Under the bootstrap hypotheses, for $\eps$ sufficiently small depending only on universal constants, there holds 
\begin{align}
\norm{\mathfrak{A}(\chi_e^I \grad^\perp \psi^I)}_{L^2} & \lesssim \brak{t} \norm{\mathfrak{A} f}_{L^2}(1 + \brak{t}^2 \norm{\mathfrak{A}_R h^I}_{L^2}) \notag \\ 
\norm{\brak{\grad}^{-4}\mathfrak{A}(\chi_e^I \grad^\perp \psi^I)}_{L^2} & \lesssim \brak{t}^{-2} \norm{\mathfrak{A} f}_{L^2}. \label{ineq:LossyInt}
\end{align}
\end{lemma}

Next we give the fundamental estimate which controls the effect of the exterior vorticity through the non-local streamfunction. 
\begin{lemma} \label{lem:phiE}
Under the bootstrap hypotheses, for $\eps$ sufficiently small and $t < \nu^{-1/3 - \zeta}$ for $\zeta < 1/78$ sufficiently small, there holds 
\begin{align}
\norm{\mathfrak{A}_R(\chi_e^I \grad^\perp (\chi^I\psi^E))}_{L^2}^2 & \lesssim e^{-\nu^{-1/9}} \left(\mathcal{E}^{(\gamma)} + \norm{\mathfrak{A}f}_{L^2}^2 \right)^{1/2} \left(1 + \mathcal{E}_H^{(\alpha)} + \mathcal{E}_H^{(\gamma)}\right)^{1/2} \label{ineq:chiIpsiE} \\ 
e^{2\delta_I \nu^{1/3} t} \norm{e^{2\lambda(t)\abs{\grad}^r}(\chi_e^I \grad^\perp( 1-\chi^I) \psi^E)}_{L^2}^2 & \lesssim  e^{-\nu^{-1/8}} \left(\mathcal{E}^{(\gamma)} + \norm{\mathfrak{A}f}_{L^2}^2 \right)^{1/2} \left(1 + \mathcal{E}_H^{(\alpha)} + \mathcal{E}_H^{(\gamma)}\right)^{1/2}
\label{ineq:1mchiIpsiE} 
\end{align}
\end{lemma}
\begin{proof}
Recall that
\begin{align*}
& -\partial_{zz} \psi^E + (v' (\partial_v - t\partial_z))^2 \psi^E = (1-\chi^I)f - \left( (1+h^I)^2 - (1 +h)^2 \right) (\partial_v-t\partial_z)^2 \psi^I \\
& \qquad \qquad -  \left((1+h)\partial_v h - (1+h^I)\partial_v h^I\right)(\partial_v - t\partial_z) \psi^I \\
& \psi^E(t,x,v(t,\pm 1)) = 0. 
\end{align*}
The subtlety in this lemma is that the left-hand side contains $v' = 1+h$ which limits the regularity of $\psi^E$ to $\mathfrak{A}_R$ in the interior (even though the right-hand side has much more regularity through Lemma \ref{lem:ExtToInt}).

The estimate \eqref{ineq:1mchiIpsiE} follows as a corollary of the elliptic estimate \eqref{fell:e} (proved in our companion paper \cite{BHIW24b}). 
Indeed, by standard the Gevrey product rules, we have 
\begin{align*}
\norm{e^{2\lambda(t)\abs{\grad}^r}(\chi_e^I\grad^\perp (1-\chi^I) \psi^E)}^2_{L^2} & \lesssim_\lambda \norm{e^{3\lambda(t)\abs{\grad}^r}( \tilde{\chi} \psi^{(E)}_k}^2_{L^2}, 
\end{align*}
where $\tilde{\chi}$ is a smooth cutoff which is is one on the support of $\chi_e^I(1-\chi^I)$ and supported on a slightly expanded interval. 
Therefore, by Lemma \ref{lem:FourierToPhysical} we have
\begin{align*}
\norm{e^{2\lambda(t)\abs{\grad}^r}(\chi_e^I\grad^\perp (1-\chi^I) \psi^E)}^2_{L^2} & \lesssim \sum_{k \in \mathbb Z} \sum_{n+m =0}^\infty \frac{(4\lambda)^{2(n+m)/r}}{((n+m)!)^{2/r}} \norm{\abs{k}^m \Gamma_k^n (\tilde{\chi} \circ v \Psi^{(E)}_k)}_{L^2}^2 \\
& \lesssim \sum_{k \in \mathbb Z} \sum_{n+m =0}^\infty \frac{(4\lambda)^{2(n+m)/r}}{((n+m)!)^{2/r}} \norm{ \tilde{\tilde{\chi}}\abs{k}^m \Gamma_k^n \Psi^{(E)}_k)}_{L^2}^2, 
\end{align*}
for another smooth cutoff $\tilde{\tilde{\chi}}$ which is is one on the support of $\tilde{\chi}$ and supported on a slightly expanded interval.
Therein we used that for 
\begin{align*}
\sum_{n+m =0}^\infty \frac{(4\lambda)^{2(n)/r}}{((n)!)^{2/r}} \norm{ \overline{\partial_v}^n \tilde{\chi} \circ v }_{L^2}^2 \lesssim 1,
\end{align*}
using the assumption that $\mathcal{E}_H^{(\gamma)} \lesssim 1$ by standard Gevrey composition estimates (see e.g. [Appendix A, \cite{BM13}]. 
Hence, by Lemma \ref{lem:ExtToInt}, we have that \eqref{ineq:1mchiIpsiE} follows from \eqref{fell:e}.

Next, we prove \eqref{ineq:chiIpsiE}. 
Applying $\chi^I$ we have
\begin{align*}
& -\partial_{zz} \chi^I\psi^E + ((1+h^I) (\partial_v - t\partial_z))^2 (\chi^I\psi^E)  = \mathcal{C}_E, \\ 
& \mathcal{C}_E := -[\chi^I,\Delta_t]\psi^E - \left( (1 + h)^2 - (1+h^I)^2 \right)(\partial_v - t\partial_z)^2 (\chi^I\psi^E) \\ & \qquad\quad - \left((1+h)\partial_v h - (1+h^I)\partial_v h^I\right) (\partial_v - t\partial_z) (\chi^I\psi^E). 
\end{align*}
Using the analogue of Lemma \ref{lem:AphiItriv}, we obtain
\begin{align*}
\norm{\mathfrak{A}(\grad^\perp (\chi^I\psi^E)}_{L^2} \lesssim \brak{t} \norm{\mathfrak{A} \mathcal{C}_E}_{L^2}(1 + \brak{t}^2 \norm{\mathfrak{A}_R h^I}_{L^2}). 
\end{align*}
However, since $\mathcal{C}_E$ is supported only in the exterior, we can use Lemmas \ref{lem:FourierToPhysical}, \ref{lem:ExtToInt}, and \eqref{ineq:1mchiIpsiE} to estimate all of the terms (note that $\psi^E$ does not have explicit localization bounds, only smallness).
This yields 
\begin{align*}
\norm{\brak{\grad}^j \mathfrak{A}_R \mathcal{C}_E}_{L^2} \lesssim e^{-\nu^{-1/8}} \brak{t}^2 \left(\mathcal{E}^{(\gamma)} + \norm{\mathfrak{A}f}_{L^2}^2 \right)^{1/2} \left(1 + \mathcal{E}_H^{(\alpha)} + \mathcal{E}_H^{(\gamma)}\right)^{1/2}, 
\end{align*}
which hence implies \eqref{ineq:chiIpsiE}. 
\end{proof}

\subsection{Interior vorticity energy estimate} \label{sec:int:vort}
In this section we prove Proposition \ref{prop:MainInterior}. 
The energy estimate on the vorticity profile begins as follows
\begin{align*}
\frac{1}{2}  \frac{d}{dt} \mathcal{E}_{\text{Int}} & = -\mathcal{CK}_{\mathrm{Int}} - \mathcal{D}_{\mathrm{Int}} + 2 \delta_I \nu^{1/3} \norm{A f^I_{\neq}}^2_{L^2} \\ 
& \quad  - \nu \brak{\mathfrak{A} f, \mathfrak{A} (1 - (v')^2) (\partial_{v} - t\partial_z)^2 f^I } - \brak{\mathfrak{A} f, \mathfrak{A} (U \cdot \grad f^I) } +  \brak{\mathfrak{A} f, \mathfrak{A} \mathcal{C} f}.    
\end{align*}


The nonlinear transport term is divided into exterior and interior contributions: 
\begin{align*}
\brak{\mathfrak{A} f^I, \mathfrak{A} (U \cdot \grad f^I) } = \brak{\mathfrak{A} f^I, \mathfrak{A} (U^I \cdot \grad f^I) } + \brak{\mathfrak{A} f^I, \mathfrak{A} (U^E \cdot \grad f^I) }. 
\end{align*}
These two contributions are treated in the following two subsections.

\subsubsection{Nonlinear interior velocity estimate} \label{sec:NIVE}
We have set up the argument so that the nonlinear term $U^I \cdot \grad f^I$ is controlled almost exactly the same as it was in \cite{HI20,BM13}, the only differences being the presence of the multiplier $M$ and the enhanced dissipation factor $e^{\delta_I \nu^{1/3} t}$, which adds two extra terms in the commutator estimates.
The estimate begins by introducing a commutator,
\begin{align*}
\brak{\mathfrak{A} f^I, \mathfrak{A} (U^I \cdot \grad f^I) } = \brak{\mathfrak{A} f^I, \grad \cdot U^I \mathfrak{A} f^I } + \brak{\mathfrak{A} f^I, [\mathfrak{A}, U^I \cdot \grad ] f^I }. 
\end{align*}
The first term is dealt with easily as in \cite{HI20,BM13} using \eqref{ineq:LossyInt}, yielding the following under the bootstrap hypotheses
\begin{align*}
\abs{\brak{\mathfrak{A} f^I, \grad \cdot U^I \mathfrak{A} f^I}} \lesssim \frac{\mathcal{E}_{\text{Int}}^{1/2} + (\mathcal{E}_{\text{Int,Coord}}^{(g)})^{1/2}}{\brak{t}^{3/2}} \mathcal{E}_{\text{Int}}
\end{align*}
For the latter term, one needs to carefully study
\begin{align}
\mathfrak{A}(t,k,\eta) - \mathfrak{A}(t,k-\ell,\eta-\xi) , \label{eq:AAdiff}
\end{align}
since we can use the Fourier transform to express (ignoring irrelevant factors of $2\pi$)
\begin{align}
\brak{\mathfrak{A} f^I, [\mathfrak{A}, U^I \cdot \grad ] f^I } & \nonumber \\ & \hspace{-3.5cm} = \textup{Re} \sum_{k,\ell} \int_{\mathbb R^2}  \overline{\mathfrak{A}\widehat{f^I_k}(\eta)}\left(\mathfrak{A}(t,k,\eta) - \mathfrak{A}(t,k-\ell,\eta-\xi)\right) \widehat{U^I_{\ell}}(\xi) \cdot i(k-\ell,\eta-\xi) \widehat{f^I_{k-\ell}}(\eta-\xi) \dee \eta \dee \xi. \label{eq:IntComm}
\end{align}
This convolution-in-frequency is split into three regions as in a standard paraproduct, 
\begin{align}
\abs{\ell,\xi} \geq 2 \abs{k-\ell,\eta-\xi}, \quad \abs{k-\ell,\eta-\xi} \leq 2 \abs{\ell,\xi}, \quad \abs{k-\ell,\eta-\xi} \approx \abs{\ell,\xi}, \label{eq:RTRsplit}
\end{align}
respectively called the ``reaction, transport, and remainder'' terms in \cite{BM13}, labeled as $\mathcal{T}_R$, $\mathcal{T}_T$, $\mathcal{T}_{rem}$ respectively. 
The treatment of the reaction and remainder terms are unchanged by the presence of $M$ and the enhanced dissipation factors. These terms are treated as in \cite{HI20,BM13} are hence omitted; these are estimated therein as 
\begin{align*}
\mathcal{T}_R + \mathcal{T}_{rem} \lesssim \mathcal{E}_{\mathrm{Int}}(t)^{1/2} \left(\mathcal{CK}_{\mathrm{Int}}(t) + \mathcal{CK}_{\mathrm{Int,Coord}}(t)\right) + \frac{1}{\brak{t}^{3/2}}(\mathcal{E}_{\text{Int}}^{1/2} + \mathcal{E}_{\text{Int,Coord}}^{1/2}) \mathcal{E}_{\text{Int}}. 
\end{align*}
However, the presence of $M$ and the enhanced dissipation multipliers do introduce some slight additional complications in the `transport' term.
In the `transport' region of frequency, the difference of the $\mathfrak{A}$ multipliers is split as
\begin{align*}
\mathfrak{A}(t,k,\eta) - \mathfrak{A}(t,k-\ell,\eta-\xi) & = \left(A(t,k,\eta) - A(t,k-\ell,\eta-\xi)\right)\frac{1}{M(t,k,\eta)} \left(e^{2\delta_I \nu^{1/3} t} \mathbf{1}_{k \neq 0} + \mathbf{1}_{k = 0} \right) \\ 
& \hspace{-2cm} \quad + A(t,k-\ell,\eta-\xi) \left( \frac{1}{M(t,k,\eta)} - \frac{1}{M(t,k-\ell,\eta-\xi)}\right) \left(e^{2\delta_I \nu^{1/3} t} \mathbf{1}_{k \neq 0} + \mathbf{1}_{k = 0} \right) \\
& \hspace{-2cm} \quad - \frac{A(t,k-\ell,\eta-\xi)}{M(t,k-\ell,\eta-\xi)} \left(\mathbf{1}_{k \neq 0}\mathbf{1}_{k-\ell = 0} - \mathbf{1}_{k=0}\mathbf{1}_{k-\ell \neq 0}\right)\left(e^{2\delta_I \nu^{1/3} t} - 1\right);
\end{align*}
we denote each corresponding contribution in the integral in \eqref{eq:IntComm} as $\mathcal{T}_{T;j}$. 
The first term, $\mathcal{T}_{T;1}$, is treated essentially as in \cite{HI20,BM13} and is omitted for the sake of brevity (note that the $e^{\delta_I \nu^{1/3} t}$ factor can be grouped with either $\widehat{U^I_\ell}(\xi)$ or $\widehat{f^I_{k-\ell}}(\eta-\xi)$, depending on which factor has non-zero $z$ frequency), yielding
\begin{align*}
\mathcal{T}_{T;1} \lesssim (\mathcal{E}_{\text{Int}}^{1/2} + (\mathcal{E}_{\text{Int,Coord}}^{(g)})^{1/2}) \mathcal{CK}_{\text{Int}}. 
\end{align*}
To treat $\mathcal{T}_{T;3}$, we note the simple observation 
\begin{align*}
e^{2\delta_I \nu^{1/3} t} - 1 \lesssim  \nu^{1/3} t e^{2 \delta_I \nu^{1/3}t}. 
\end{align*}
Then, using the frequency localizations, we have 
\begin{align*}
  \mathcal{T}_{T;3} & \lesssim \nu^{1/3} t \norm{\abs{\grad}^{3/4} \mathfrak{A} f_0}_{L^2} \norm{U^I}_{H^4} \norm{\abs{\grad}^{1/4} \mathfrak{A} f}_{L^2} \\
  & \quad +  \nu^{1/3} \norm{U^I}_{H^4} \norm{\abs{\grad}^{1/4} \mathfrak{A} f}_{L^2} \norm{\abs{\partial_v}^{3/4} \mathfrak{A} f_0}_{L^2} \\
& \lesssim \left(\mathcal{E}_{\text{Int}} + \mathcal{E}_{\text{Int,Coord}}^{(g)} \right)^{1/2} \left( \mathcal{CK}_\lambda +  \nu \norm{\partial_v \mathfrak{A} f_0}_{L^2}^2\right),
\end{align*}
where we used the interpolation estimate $\norm{\partial_v^{3/4} f}_{L^2} \lesssim \norm{\partial_v^{1/4} f}_{L^2}^{1/3} \norm{\partial_v f}_{L^2}^{2/3}$ and \eqref{ineq:LossyInt}.

Next, for $\mathcal{T}_{T;2}$ we apply Lemma \ref{def:ICommM},
giving 
\begin{align*}
\abs{\mathcal{T}_{T;2}} & \lesssim \brak{t}^{1-s} \norm{U^I}_{H^2} \norm{\abs{\grad}^{s/2} \mathfrak{A} f^I}_{L^2}^2 \lesssim (\mathcal{E}_{\text{Int}}^{1/2} + (\mathcal{E}_{\text{Int,Coord}}^{(g)})^{1/2}) \mathcal{CK}_\lambda, 
\end{align*}
and hence all contributions can be integrated consistent with Proposition \ref{prop:boot}. 

\subsubsection{Nonlinear exterior velocity estimate} \label{sec:UENL}
This concerns the nonlinear term $\brak{\mathfrak{A} f^I, \mathfrak{A} (U^E \cdot \grad f^I) }$, which is mainly treated via the diffusion (the delicate aspects are the estimates on $\psi^E$, carried out in \eqref{ineq:chiIpsiE} and \eqref{ineq:1mchiIpsiE}) and the strong a priori estimates available on $g$ from Lemma \ref{lem:ExtToInt} and $\mathcal{E}_{\text{Ext,Coord}}$. 
Expanding 
\begin{align*}
\brak{\mathfrak{A} f^I, \mathfrak{A} (U^E \cdot \grad f^I) } & = \brak{\mathfrak{A} f^I, \mathfrak{A} (\chi_e^I (1+h^I)\grad^\perp \psi^E \cdot \grad f^I) } \\
& \quad + \brak{\mathfrak{A} f^I, \mathfrak{A} (\chi^I_e (1-\chi^I)g \partial_v f^I) } + \brak{\mathfrak{A} f^I,\mathfrak{A}(\chi_e^I(1-\chi^I)h\grad^\perp \psi^I, \grad f^I) } \\
& =: T_1 + T_2 + T_3. 
\end{align*}
Consider first $T_2$, which begins by introducing a commutator and integrating by parts
\begin{align*}
T_2 & = -\brak{\mathfrak{A} f^I, \partial_v (\chi^I_e (1-\chi^I) g)  \mathfrak{A} f^I } + \brak{\mathfrak{A} f^I, [\mathfrak{A}, (\chi^I_e (1-\chi^I)g \partial_v] f^I) } \\
& =: T_{21} + T_{22}. 
\end{align*}
The estimate of the first term is immediate using \eqref{lem:ExtToInt}, yielding (note that the term where no derivatives land on $g$ is what implies there is no $\nu$-dependent smallness in this estimate and hence why we cannot directly use the dissipation term), 
\begin{align*}
T_{21} \lesssim \frac{1}{t^{2 - K\eps}} \mathcal{E}_{\text{Ext,Coord}}^{1/2} \mathcal{E}_{\text{Int}}. 
\end{align*}
The latter term, $T_{22}$, we use that $P_0g = g$, which implies the commutator is a little simpler than in the case of $\mathcal{T}$ above in \eqref{eq:IntComm}.
We may apply the same decomposition as in \eqref{eq:RTRsplit}, and we deduce that (using the simpler estimates on $g$ coming from Lemma \ref{lem:ExtToInt})
\begin{align*}
T_{22} \lesssim \mathcal{E}_{\text{Ext,Coord}}^{1/2} (\mathcal{CK}_{\text{Int}} + \frac{1}{t^{2-K\eps}}\mathcal{E}_{\text{Int}}).  
\end{align*}
We can approach $T_3$ more directly, indeed by Lemma \ref{lem:BasicA} 
\begin{align*}
T_3 \lesssim \brak{t}^4 \norm{\mathfrak{A}f^I}_{L^2} \norm{\mathfrak{A}\chi^I_e \grad^\perp \psi^I }_{L^2} \norm{A_R\chi^I_e (1-\chi^I)h}_{L^2} \norm{\mathfrak{A} \grad_L f^I}_{L^2}. 
\end{align*}
Then, using Lemma  \ref{lem:AphiItriv} to estimate $\psi^I$ and Lemma \ref{lem:ghExtToInt} to estimate $(1-\chi^I)h$ we have
\begin{align*}
T_3 \lesssim e^{-\nu^{-1/8}} \brak{t}^6 (\mathcal{E}_{H}^{(\gamma)})^{1/2} (\norm{\mathfrak{A} f}_{L^2} + \norm{\mathfrak{A}_R h^I}_{L^2})\norm{\mathfrak{A}f^I}_{L^2}\norm{\mathfrak{A} \grad_L f^I}_{L^2}, 
\end{align*}
which is again easily integrated using the diffusion to the desired time-scale. 
For $T_1$ we begin again with Lemma \ref{lem:BasicA}, 
\begin{align*}
T_1 \lesssim \brak{t}^2 \norm{\mathfrak{A} f^I}(1 + \norm{A_R h^I}_{L^2}) \norm{\mathfrak{A}_R (\chi_e^I \grad^\perp \psi^E)}_{L^2} \norm{\mathfrak{A} \grad_L f^I}_{L^2}.  
\end{align*}
The middle factor is then estimated via \eqref{ineq:chiIpsiE} and \eqref{ineq:1mchiIpsiE}, yielding a contribution which is again easily integrated using the diffusion to the desired time-scale. 

\subsubsection{Diffusion error term}
The diffusion error term is relatively straightforward.
First note that 
\begin{align*}
\nu \brak{\mathfrak{A} f, \mathfrak{A} (1 - (v')^2) (\partial_{v} - t\partial_z)^2 f^I } = \nu \brak{\mathfrak{A} f, \mathfrak{A} (2h - h^2)  (\partial_{v} - t\partial_z)^2 f^I }. 
\end{align*}
Next, expand the $h$ using the localization, 
\begin{align*}
\nu \brak{\mathfrak{A} f, \mathfrak{A} (1 - (v')^2) (\partial_{v} - t\partial_z)^2 f^I } & = \nu \brak{\mathfrak{A} f, \mathfrak{A} (2h^I - (h^I)^2)  (\partial_{v} - t\partial_z)^2 f^I } \\
& \quad + \nu \brak{\mathfrak{A} f, \mathfrak{A} ( 2(h-h^I) - 2(h-h^I)h - (h-h^I)^2  )  (\partial_{v} - t\partial_z)^2 f^I } \\
& =: \mathcal{D}_1 + \mathcal{D}_2. 
\end{align*}
The former term $\mathcal{D}_1$ is estimated as in \cite{BMV14}, integrating by parts and using the product rules on $A$, yielding
\begin{align*}
\mathcal{D}_1 \lesssim \nu (\mathcal{E}_{\text{Int,Coord}}^{(h)})^{1/2} \norm{\grad_L \mathfrak{A} f}_{L^2}^2 + \sqrt{\nu} \mathcal{E}_{\text{Int}}^{1/2} \norm{\grad_L \mathfrak{A} f}_{L^2}(\mathcal{D}_{\text{Int,Coord}}^{(h)})^{1/2}. 
\end{align*}
The treatment of $\mathcal{D}_2$ is similar, except that on the support of integrand, the exterior norms are used to estimate $h$, easily yielding
\begin{align*}
\mathcal{D}_1 \lesssim \nu^{1000} (\mathcal{E}^{(\gamma)}_{H})^{1/2} \norm{\grad_L \mathfrak{A} f^I}_{L^2}^2.  
\end{align*}

\subsubsection{The cut-off commutators}
In this section we deal with the terms involving $\partial_v\chi^I$. 
First, we treat the commutators coming from the transport term, which we subdivide as before 
\begin{align*}
\mathcal{C}_T := \brak{\mathfrak{A}f ,\mathfrak{A}\left( (U^E_2) + (U^I_2) \right) \partial_v\chi_I f }.
\end{align*}
By the product rules on $\mathfrak{A}$ and Lemma \ref{lem:ExtToIntf} we have 
\begin{align*}
  \brak{\mathfrak{A} f^I,\mathfrak{A} ( U^E_2 \chi_I' f)} & \lesssim \nu^{1000} \norm{\mathfrak{A} f^I}_{L^2} (\mathcal{E}^{(\gamma)})^{1/2} \norm{ \mathfrak{A}_R (\chi_e^I U^E_2)}_{L^2}. 
\end{align*}
Using the estimates we have available on $U^E$ as in Section \ref{sec:UENL} this is easily integrated to the desired time-scale. 
We have (using Lemma \ref{lem:AphiItriv} and \eqref{ineq:ARtrivBd}), 
\begin{align*}
\brak{\mathfrak{A} f^I,\mathfrak{A} (U_2^I \partial_v \chi^I f)} & \lesssim \norm{\mathfrak{A} f^I}_{L^2} \norm{\mathfrak{A}_R U_2^I}_{L^2} \norm{\mathfrak{A}_R (\partial_v\chi^I f)}_{L^2} \\
&  \lesssim \nu^{1000} \brak{t}^2 \norm{\mathfrak{A} f}_{L^2}(1 + \brak{t}^2 \norm{\mathfrak{A}_R h^I}_{L^2}) \norm{\mathfrak{A} f^I}_{L^2}(\mathcal{E}^{(\gamma)})^{1/2}, 
\end{align*}
which easily integrates. 

Next, we treat the commutators coming from the dissipation, namely
\begin{align*}
  \mathcal{C}_D := -\nu \brak{\mathfrak{A}f,\mathfrak{A} \left( (v')^2 \partial_{vv}\chi^I f - 2 (v')^2 \partial_v\chi^I (\partial_v-t\partial_z) f\right)}.  
\end{align*}
Using Lemmas \ref{lem:ExtToIntf} and \ref{lem:ghExtToInt} we have 
\begin{align*}
\mathcal{C}_D \lesssim \nu^{100} \brak{t} \norm{\mathfrak{A} f}_{L^2} (1 + (\mathcal{E}_H^{(\gamma)})^{1/2})^2 (\mathcal{E}^{(\gamma)})^{1/2}, 
\end{align*}
which is again integrated easily on the relevant time-scales. 
This completes the proof of Proposition \ref{prop:MainInterior}. 

\subsection{Interior coordinate energy estimates} 

The proof of Proposition \ref{prop:MainCoordInt} follow easily from the ideas of  \cite{HI20,BM13} combined with the ideas used to prove Proposition \ref{prop:MainInterior} and so we omit the arguments for the sake of brevity.

%

%

\section{Quasiproduct Bounds in ``Sequential" Gevrey Spaces} \label{sec:q}
\subsection{Sequential Gevrey Norms}

Associated to the energy functional, $\mathcal{E}^{(\gamma)}$, we can define the following norm: 
\begin{align} \label{Xgammasp}
\| \omega \|_{X_\gamma}^2 := &  \sum_{n = 0}^{\infty} \sum_{m = 0}^{\infty} \bold{a}_{m, n}^2  \| \p_x^m q^n \Gamma^n \omega e^W \chi_{n+m}  \|_{L^2}^2.
\end{align}
It will be convenient to slightly generalize this notion; instead of thinking of \eqref{Xgammasp} as a norm on a single function, $\omega$, we can think of it as measuring a sequence of functions $\omega_{n,m}$, which in the case of \eqref{Xgammasp} turn out to be related through the relation $\omega_{n,m} := \p_x^m q^n \Gamma^n \omega$. Therefore, given a sequence of functions, $\{ f_{m,n} \}_{m \ge 0, n \ge 0}$, we define the following ``sequential Gevrey" norms
\begin{align} \label{seqXgamma}
\| \{ f_{m,n} \}_{m, n \ge 0} \|_{X_\gamma}^2 := & \sum_{n = 0}^{\infty} \sum_{m = 0}^{\infty} \bold{a}_{m,n}^2\| q^n f_{m,n} e^W \chi_{n + m} \|_{L^2(\mathbb{T} \times [-1,1])}^2, \\ \label{seqXalpha}
\| \{ f_{m,n} \}_{m, n \ge 0} \|_{X_\alpha}^2 :=& \sum_{n = 0}^{\infty} \sum_{m = 0}^{\infty} \nu \bold{a}_{m,n}^2\| q^n \p_y f_{m,n} e^W \chi_{n + m} \|_{L^2(\mathbb{T} \times [-1,1])}^2.
\end{align}
More generally, we will need to keep track of the various multipliers in our norms. Therefore, we introduce the notation which explicitly depicts various parameters in our normed spaces as follows: 
\begin{definition} Let $f_{m,n}$ be a sequence of functions related by the condition $f_{m,n} := \p_x^m \Gamma^n f$. Let $G_{m,n}$ be a sequence of weight functions. Then we define 
\begin{align}
\| \{ f_{m,n} \} \|_{X^{s}_{p}(\{ G_{m,n} \})}^2 := \sum_{n = 0}^{\infty} \sum_{m = 0}^{\infty} \|  f_{m,n} G_{m,n} \|_{W^{s,p}_{xy}(\mathbb{T} \times [-1,1])}^2
\end{align}
\end{definition}
The sequence $G_{m,n} = G_{m,n}(t, y)$ are weights (one can think of the generic scenario from \eqref{seqXgamma} as $G_{m,n}(t, y) := \bold{a}_{m,n} q^n e^W \chi_{n+m}$). We note that, according to the definitions above, the following two relations trivially hold 
\begin{align}
\| \{ f_{m,n} \} \|_{X_\gamma} =  &\| \{ f_{m,n} \}_{m, n \ge 0} \|_{X^{0}_{2}(\bold{a}_{m,n} q^n e^W \chi_{n + m} )} \\
\| \{ f_{m,n} \} \|_{X_\alpha} = &  \| \{ \nu^{\frac12} \p_y f_{m,n} \}_{m, n \ge 0} \|_{X^{0}_{2}(\bold{a}_{m,n} q^n e^W \chi_{n + m} )}
\end{align}

\subsubsection{Definition of Quasiproduct Bilinear Operators} \label{sbg:1}
We examine the trilinear terms appearing from the $\gamma$ estimate. In particular, recalling the stream function decomposition \eqref{ell:I} -- \eqref{ell:E},  we have 
\begin{align} \label{Tmain:k1}
\bold{T}^{(n)}_{main} :=  & \Gamma^n(\nabla^\perp \phi_{\neq 0} \cdot \nabla  \omega) =  \sum_{n' = 0}^n \binom{n}{n'} ( \Gamma^{n-n'} \nabla^\perp \phi_{\neq 0} \cdot \Gamma^{n'} \nabla \omega )
\end{align}
We will split \eqref{Tmain:k1} into 
\begin{align} \label{Tmain:k2}
\bold{T}^{(n)}_{main} :=  \bold{T}^{(n, I)}_{main} + \bold{T}^{(n, E)}_{main}
\end{align}
depending on the corresponding decomposition of the stream function $\phi_{\neq 0}$ into $\phi^{(I)}_{\neq 0}$ and $\phi^{(E)}_{\neq 0}$. We commute $\p_x^m \Gamma^n$, which produces (for $U \in \{E, I\}$):
\begin{align} \n
\bold{T}^{(U, \gamma)}_{m,n} := & \p_x^m \Gamma^n \{ \nabla^\perp \phi^{(U)}_{\neq 0} \cdot \nabla \omega \} \\ \n
= & \sum_{m' = 0}^{m} \sum_{n' = 0}^n \mathbbm{1}_{n' + m' < n + m} \binom{m}{m'} \binom{n}{n'}  (\nabla^\perp \phi^{(U)}_{\neq 0})_{m-m', n-n'} \cdot \nabla  \omega_{m', n'} + \nabla^\perp \phi^{(U)}_{\neq 0} \cdot \nabla \p_x^m \Gamma^n \omega \\ \n
= & \sum_{m' = 0}^{m} \sum_{n' = 0}^n \mathbbm{1}_{n' + m' < n + m} \binom{m}{m'} \binom{n}{n'} ( \Gamma \phi^{(U, \neq 0)}_{m-m', n-n'} \p_x \omega_{m', n'} - \p_x \phi^{(U, \neq 0)}_{m-m', n-n'} \Gamma \omega_{m', n'}) \\ \label{joey:2}
& +  \nabla^\perp \phi^{(U)}_{\neq 0} \cdot \nabla \p_x^m \Gamma^n \omega \\ \n
= & \sum_{i = 1}^2 T_{m,n}^{(i)}[ \{ \phi^{(U, \neq 0)}_{m, n}\}, \{\omega_{m,n}\}] +\nabla^\perp \phi^{(U, \neq 0)} \cdot \nabla \p_x^m \Gamma^n \omega. 
\end{align}
Above, we have used the following identity. First, we define 
\begin{align}
\slashed{\nabla} := (\p_x, \Gamma),
\end{align}
Then, we have for abstract functions $f(t, x, y), g(t, x, y)$
\begin{align} \label{id:gamma:re}
\nabla^\perp g \cdot \nabla f = - \p_x g \p_y f + \p_y g \p_x f = v_y(- \p_x g \Gamma f + \Gamma g \p_x f) = v_y \slashed{\nabla}^\perp g \cdot \slashed{\nabla} f. 
\end{align}
Above, we have defined the following sequential operators: 
\begin{align} \label{odesza:1}
T_{m,n}^{(1)}[ \{b_{m, n}\}, \{\omega_{m,n}\}] := & \sum_{m' = 0}^{m} \sum_{n' = 0}^n \mathbbm{1}_{n' + m' < n + m} \binom{m}{m'} \binom{n}{n'} \Gamma b_{m-m', n-n'} \p_x \omega_{m', n'} \\ \label{odesza:2}
 T_{m,n}^{(2)}[ \{b_{m, n}\}, \{\omega_{m,n}\}] := &-\sum_{m' = 0}^{m} \sum_{n' = 0}^n \mathbbm{1}_{n' + m' < n + m} \binom{m}{m'} \binom{n}{n'} \p_x b_{m-m', n-n'} \Gamma \omega_{m', n'}.
\end{align}
As a matter of notation, we will think of the operators, \eqref{odesza:1} -- \eqref{odesza:2}, as abstract operators. Therefore, the sequence of functions $b_{m,n}$ will denote an abstract sequence of functions related by the relations $b_{m,n} = \p_x^m \Gamma^n b$ and all satisfying the boundary condition $b_{m,n}|_{y = \pm 1} = 0$ (eventually, $b_{m,n}$ will either represent $\phi^{(I)}_{\neq 0}$ or $\phi^{(E)}_{\neq 0}$). 

As a matter of conception, we regard the operators \eqref{odesza:1} -- \eqref{odesza:2} ``quasiproducts". Indeed, $T^{(1)}_{m,n}$ coincides with $\Gamma^n \p_x^m \{ \Gamma b \p_x \omega \}$, with the exception of the term $\Gamma b \p_x^{m+1} \Gamma^n \omega$ (the ``quasilinear" term). 

We now define the following operators
\begin{align} \label{bilinear:Q}
\mathcal{Q}_{m,n}[\{b_{m, n}\}, \{\omega_{m,n}\}] := & \begin{cases} 0 , \qquad m + n \le 4N \\ \sum_{i = 1}^2 T^{(i)}_{m,n}[\{b_{m, n}\}, \{\omega_{m,n}\}], \qquad m + n > 16N  \end{cases} \\ \label{bilinear:R}
\mathcal{R}_{m,n}[\{b_{ m, n}\}, \{\omega_{m,n}\}] := & \begin{cases} \sum_{i = 1}^2 T^{(i)}_{m,n}[\{b_{m, n}\}, \{\omega_{m,n}\}] , \qquad m + n \le 16N \\ 
0 \qquad m + n > 16N, 
\end{cases}
\end{align}
where $N >> 1$ is a large parameter. We will define $\mathcal{Q}^{(i)}_{m,n}[\{b_{m, n}\}, \{\omega_{m,n}\}]$ and $\mathcal{R}^{(i)}_{m,n}[\{b_{m, n}\}, \{\omega_{m,n}\}]$ in the natural manner. 
 
We now define the quantities
\begin{subequations}
\begin{align}   \label{aug5:a}
\bold{T}^{(U, \gamma, 1)}_{m,n} := & \mathcal{R}_{m,n}[\{\phi^{(U)}_{\neq 0, m, n}\}, \{\omega_{m,n}\}],  \\ \label{aug5:b}
\bold{T}^{(U, \gamma, 2)}_{m,n}  := & \mathcal{Q}_{m,n}[\{\phi^{(U)}_{\neq 0, m, n}\}, \{\omega_{m,n}\}], \\ \label{aug5:d}
\bold{T}^{(U, \gamma, 3)}_{m,n} := &v_y \slashed{\nabla}^\perp \phi^{(U)}_{\neq 0} \cdot \slashed{\nabla} \p_x^m \Gamma^n \omega.
\end{align}
\end{subequations}
The following decomposition then is clear:
\begin{align} \label{jalen:brunson:for:three}
\bold{T}^{(U, \gamma)}_{m,n} = \sum_{i = 1}^3 \bold{T}^{(U, \gamma, i)}_{m,n},  \qquad U \in \{I, E \}. 
\end{align}

When we treat the $\alpha$ and $\mu$ estimates, we will need to apply $\nabla$. In the forthcoming, we collect all relevant definitions for this procedure, which will then be used below in Section \ref{jaylen:1}. First, in order to treat $\alpha$ and $\mu$ at once (their estimates are almost fully analogous), we introduce the notation
\begin{align}
\xi \in \{\alpha, \mu\}, \qquad \p_{x_\xi} = \begin{cases} \p_{x} \qquad \xi = \mu \\ \p_y \qquad \xi = \alpha \end{cases}. 
\end{align}
We then define for $\xi \in \{ \alpha, \mu \}$,  
\begin{align} \n
\bold{T}^{(U, \xi)}_{m,n} := \p_{x_\xi} \bold{T}^{(U, \gamma)}_{m,n} = & \mathcal{R}_{m,n}[\{\p_{x_\xi} \phi^{(U)}_{\neq 0, m, n}\}, \{\omega_{m,n}\}] + \mathcal{R}_{m,n}[\{\phi^{(U)}_{\neq 0, m, n}\}, \{\p_{x_\xi} \omega_{m,n}\}] \\ \n
& +  \mathcal{Q}_{m,n}[\{\p_{x_\xi} \phi^{(U)}_{\neq 0, m, n}\}, \{\omega_{m,n}\}] + \mathcal{Q}_{m,n}[\{\phi^{(U)}_{\neq 0, m, n}\}, \{\p_{x_\xi} \omega_{m,n}\}]  \\ \n
& + \p_{x_\xi} \slashed{\nabla}^\perp \phi^{(U)}_{\neq 0} \cdot \slashed{\nabla} \p_x^m \Gamma^n \omega + \slashed{\nabla}^\perp \phi^{(U)}_{\neq 0} \cdot \p_{x_\xi} \slashed{\nabla} \p_x^m \Gamma^n \omega \\ \label{harden:1}
= & \sum_{i = 1}^{6} \bold{T}^{(U, \xi, i)}_{m,n}.
\end{align}
We will need to introduce the Fourier transform for each of these quantities: 
\begin{align}
 \bold{T}^{(U, \gamma)}_{m,n} = \sum_{k \in \mathbb{Z}} e^{ikx} \bold{T}^{(U, \gamma)}_{m,n;k}, \qquad \bold{T}^{(U, \alpha)}_{m,n} = \sum_{k \in \mathbb{Z}} e^{ikx} \bold{T}^{(U, \alpha)}_{m,n;k}, \qquad \bold{T}^{(U, \mu)}_{m,n} = \sum_{k \in \mathbb{Z}} e^{ikx} \bold{T}^{(U, \mu)}_{m,n;k}
\end{align}

\subsection{Estimates on Abstract Bilinear Operators I: interior $\mathcal{R}[\cdot, \cdot]$ bounds}

We first obtain some bounds on the bilinear operator $\mathcal{R}_{m,n}[\cdot, \cdot]$.
\begin{lemma} \label{Lemma:restless:bones} Assume $b_{m,0}|_{y = \pm 1} = 0$. The following bounds are valid on $\mathcal{R}_{m,n}$:
\begin{align} \n
&\| \{  \mathcal{R}_{m,n}[\{b_{m,n}\}, \{\omega_{m,n}\}]  \|_{{X^{0}_2(\bold{a}_{m,n} q^n e^W \chi_{n + m} )}} \\ \n
&\lesssim ( \langle t \rangle \| \{  b_{m, n} \} \|_{X^{0}_\infty(\mathbbm{1}_{m + N < 17N})} +\langle t \rangle^2 \| \{  b_{m, n} \} \|_{X^{0}_\infty(\chi_2 \mathbbm{1}_{m + N < 17N} )}) \\ \label{snow:patrol:1}
& \qquad \| \{ \omega_{m,n} \} \|_{X^{0}_2(\mathbbm{1}_{m + N < 16N} \bold{a}_{m,n} q^n e^W \chi_{n + m})}.
\end{align}
\end{lemma}
\begin{proof} We recall the definition \eqref{bilinear:R}. We first treat the case of $T^{(1)}_{m,n}$. We estimate 
\begin{align} \n
&\| \{  \mathcal{R}^{(1)}_{m,n}[\{b_{m,n}\}, \{\omega_{m,n}\}]  \|_{{X^{0}_2(\bold{a}_{m,n} q^n e^W \chi_{n + m} )}}^2 \\ \n
= &\sum_{n = 0}^\infty \sum_{m = 0}^\infty  \mathbbm{1}_{n + m < 16N} \bold{a}_{n,m}^2 \| \sum_{n' = 0}^n \sum_{m' =0}^m \binom{n}{n'} \binom{m}{m'} \mathbbm{1}_{n' + m' < n + m} \Gamma b_{m - m', n - n'} \p_x \omega_{m', n'} \chi_{n + m} e^W q^n \|_{L^2}^2 \\ \label{kijui:1}
\lesssim & \sum_{n = 0}^\infty \sum_{m = 0}^\infty \mathbbm{1}_{n +m \le 16N}\bold{a}_{n,m}^2 ( \sum_{n' = 0}^n \sum_{m' =0}^m  \| \Gamma b_{m - m', n - n'} \|_{L^\infty} \| \omega_{m' + 1, n'} \chi_{n' + 1 + m'} e^W q^{n'} \|_{L^2} )^2\\ \n
\lesssim & \sum_{n = 0}^\infty \sum_{m = 0}^\infty \mathbbm{1}_{n +m \le 16N} ( \sum_{n' = 0}^n \sum_{m' =0}^m  \| \Gamma b_{m - m', n - n'} \|_{L^\infty} \varphi^{n'+1} \| \omega_{m' + 1, n'} \chi_{n' + 1 + m'} e^W q^{n' } \|_{L^2} )^2 \\ \n
\lesssim & \sum_{n = 0}^\infty \sum_{m = 0}^\infty \mathbbm{1}_{n +m \le 16N} ( \sum_{n' = 0}^n \sum_{m' =0}^m  \| \Gamma b_{m - m', n - n'} \|_{L^\infty}^2 ) (\sum_{n' = 0}^n \sum_{m' =0}^m \varphi^{2(n'+1)} \| \omega_{m' + 1, n'} \chi_{n' + 1 + m'} e^W q^{n' } \|_{L^2}^2) \\ \n
\lesssim & ( \sum_{n + m \le 16N} \| \Gamma b_{m, n } \|_{L^\infty}^2)( \sum_{m + n \le 4N + 1} \varphi^{2n} \| \omega_{m, n} \chi_{n + m} e^W q^{n} \|_{L^2}^2  ) \\ \n
\lesssim & \| \{  b_{ m, n} \} \|_{X^{0}_\infty(\mathbbm{1}_{m + N < 17N})}^2 \| \{ \omega_{m,n} \} \|_{X^{0}_2(\varphi^{n + 1}\mathbbm{1}_{m + n < 4N} q^n e^W \chi_{n + m})}^2
\end{align}
where we have used that $m' + n' + 1 \le m + n$ in order to insert the appropriate cutoff function $\chi_{n + m} \le \chi_{n' + m' + 1}$. Similarly, we have $\varphi^{n} \le \varphi^{n'}$. 

We next treat the case of $T^{(2)}_{m,n}$.  We estimate 
\begin{align} \n
&\| \{  \mathcal{R}^{(2)}_{m,n}[\{b_{m,n}\}, \{\omega_{m,n}\}]  \|_{{X^{0}_2(\bold{a}_{m,n} q^n e^W \chi_{n + m} )}}^2 \\ \n
= &\sum_{n = 0}^\infty \sum_{m = 0}^\infty  \mathbbm{1}_{n + m < 16N} \bold{a}_{n,m}^2 \| \sum_{n' = 0}^n \sum_{m' =0}^m \binom{n}{n'} \binom{m}{m'} \mathbbm{1}_{n' + m' < n + m} \p_x b_{m - m', n - n'} \Gamma \omega_{m', n'} \chi_{n + m} e^W q^n \|_{L^2}^2 \\ \n
\lesssim & \sum_{n = 0}^\infty \sum_{m = 0}^\infty \mathbbm{1}_{n +m \le 16N}\bold{a}_{n,m}^2 ( \sum_{n' = 0}^n \sum_{m' =0}^m  \| \p_x b_{m - m', n - n'} w_{n,n'}(y) \|_{L^\infty} \| \omega_{m', n' + 1} \chi_{n' + 1 + m'} e^W q^{n' + 1} \|_{L^2} )^2\\ \label{sca:1}
\lesssim & \sum_{n = 0}^\infty \sum_{m = 0}^\infty \mathbbm{1}_{n +m \le 16N} ( \sum_{n' = 0}^n \sum_{m' =0}^m  \|  b_{m - m'+1, n - n'} w_{n,n'}(y) \varphi^{-1}\|_{L^\infty} \varphi^{n' + 2} \| \omega_{m' , n'+1} \chi_{n' + 1 + m} e^W q^{n'+1 } \|_{L^2} )^2 \\ \n
\lesssim & \sum_{n = 0}^\infty \sum_{m = 0}^\infty \mathbbm{1}_{n +m \le 16N} ( \sum_{n' = 0}^n \sum_{m' =0}^m  \|  b_{m - m'+1, n - n'} \varphi^{-1} \|_{L^\infty}^2 +  \| \chi_2 \p_y b_{m - m'+1, n - n'} \varphi^{-1} \|_{L^\infty}^2 ) \\ \label{sca:2}
&\qquad \times (\sum_{n' = 0}^n \sum_{m' =0}^m \varphi^{2(1 + n')} \| \omega_{m' , n' + 1} \chi_{n' + 1 + m'} e^W q^{n' + 1 } \|_{L^2}^2) \\ \n
\lesssim & ( \sum_{n + m \le 16N} \| \varphi^{-1} \p_x b_{m, n } \|_{L^\infty}^2 + \|\varphi^{-1} \p_{xy} b_{m, n } \chi_2 \|_{L^\infty}^2)( \sum_{m + n \le 4N + 1} \varphi^{2(1 +  n)} \| \omega_{m, n} \chi_{n + m} e^W q^{n } \|_{L^2}^2  ),
\end{align}
where we have used that $m' + n' + 1 \le m + n$ in order to insert the appropriate cutoff function $\chi_{n + m} \le \chi_{n' + m' + 1}$. Similarly, we have $\varphi^{n} \le \varphi^{n'+ 1} \varphi^{-1}$. Moreover, we have defined 
\begin{align} \label{def:w:n:n'}
w_{n,n'}(y) := \begin{cases} 1 \qquad n' < n \\ \frac{1}{q} \qquad n' = n \end{cases}.
\end{align}
To go from \eqref{sca:1} to \eqref{sca:2}, we have used the following Hardy-type inequality
\begin{align} \n
\| b_{m-m' + 1, n-n'} w_{n, n'} \|_{L^\infty} \lesssim &\| b_{m-m' + 1, n-n'} w_{n, n'} \{1 - \chi_3 \} \|_{L^\infty} + \| b_{m-m' + 1, n-n'} w_{n, n'} \chi_3 \|_{L^\infty} \\ \n
\lesssim & \| b_{m-m' + 1, n-n'}  \{1 - \chi_3 \} \|_{L^\infty} + \mathbbm{1}_{n' = n} \| \p_y b_{m-m' + 1, 0} \chi_3 \|_{L^\infty} \\ \n
&+  \mathbbm{1}_{n' = n} \| b_{m-m' + 1, 0} \chi_3' \|_{L^\infty} \\ \n
\lesssim & \| b_{m-m' + 1, n-n'}  \|_{L^\infty} + \mathbbm{1}_{n' = n} \| \p_y b_{m-m' + 1, 0} \chi_2 \|_{L^\infty} \\
&+  \mathbbm{1}_{n' = n} \| b_{m-m' + 1, 0} \chi_2 \|_{L^\infty},
\end{align}
which is valid due to the hypothesis that $b_{m-m'+1,0}|_{y = \pm 1} = 0$. 

To conclude, we rewrite the term that depends on $b$ from the final line as follows: 
\begin{align*}
&( \sum_{n + m \le 16N} \| \varphi^{-1} \p_x b_{m, n } \|_{L^\infty}^2 + \|\varphi^{-1} \p_{xy} b_{m, n } \chi_2 \|_{L^\infty}^2) \\
\lesssim &\langle t \rangle \sum_{n + m \le 17N}  \| b_{m,n} \|_{L^\infty}^2 + \langle t \rangle \sum_{n + m \le 16N+1}  \| \frac{1}{v_y} \overline{\p}_v b_{m,n} \chi_2 \|_{L^\infty}^2 \\
\lesssim & \langle t \rangle \sum_{n + m \le 17N}  \| b_{m,n} \|_{L^\infty}^2 + \|\frac{1}{v_y}\|_{L^\infty} \langle t \rangle \sum_{n + m \le 16N+1}  \| (\Gamma + t \p_x) b_{m,n} \chi_2 \|_{L^\infty}^2 \\
\lesssim & \langle t \rangle \sum_{n + m \le 17N}  \| b_{m,n} \|_{L^\infty}^2 + \|\frac{1}{v_y}\|_{L^\infty} \langle t \rangle^2 \sum_{n + m \le 17N}  \|  b_{m,n} \chi_2\|_{L^\infty}^2.
\end{align*}
This concludes the proof of the lemma. 
\end{proof} 

We need a variant of this lemma which applies in the more general case when $b_{m,0}|_{y = \pm 1}$ is not assumed to vanish. In this case, we will have to pay derivative on the $\omega_{m,n}$ component. Indeed, we have the following:
\begin{lemma}The following bounds are valid on $\mathcal{R}_{m,n}$:
\begin{align} \n
&\| \{  \mathcal{R}_{m,n}[\{b_{m,n}\}, \{\omega_{m,n}\}]  \|_{{X^{0}_2(\bold{a}_{m,n} q^n e^W \chi_{n + m} )}} \\ \n
\lesssim & \langle t \rangle \| \{  b_{m, n} \} \|_{X^{0}_\infty(\mathbbm{1}_{m + N < 17N})}  \| \{ \omega_{m,n} \} \|_{X^{0}_2(\mathbbm{1}_{m + N < 16N} \bold{a}_{m,n} q^n e^W \chi_{n + m})} \\ \label{office:1}
& +  \| \{  b_{m, n} \} \|_{X^{0}_\infty(\chi_2 \mathbbm{1}_{m + N < 17N})}  \| \{\p_y  \omega_{m,n} \} \|_{X^{0}_2(\mathbbm{1}_{m + N < 16N} \bold{a}_{m,n} q^n e^W \chi_{n + m})}.
\end{align}
\end{lemma}
\begin{proof} First of all, note that the estimate of $\mathcal{R}^{(1)}_{m,n}[\cdot, \cdot]$ from the previous lemma did not invoke the hypothesis $b_{m,0}|_{y = \pm 1} = 0$, and therefore the following estimate still applies in the present lemma: 
\begin{align} \n
&\| \{  \mathcal{R}^{(1)}_{m,n}[\{b_{m,n}\}, \{\omega_{m,n}\}]  \|_{{X^{0}_2(\bold{a}_{m,n} q^n e^W \chi_{n + m} )}}^2 \\ \n
\lesssim & \| \{  b_{ m, n} \} \|_{X^{0}_\infty(\mathbbm{1}_{m + N < 17N})}^2 \| \{ \omega_{m,n} \} \|_{X^{0}_2(\varphi^{n + 1}\mathbbm{1}_{m + n < 4N} q^n e^W \chi_{n + m})}^2.
\end{align}
We therefore estimate $\mathcal{R}^{(2)}_{m,n}[\cdot, \cdot]$, which we decompose as follows:
\begin{align*}
&\| \{  \mathcal{R}^{(2)}_{m,n}[\{b_{m,n}\}, \{\omega_{m,n}\}]  \|_{{X^{0}_2(\bold{a}_{m,n} q^n e^W \chi_{n + m} )}}^2 \\ \n
\le &\sum_{n = 0}^\infty \sum_{m = 0}^\infty  \mathbbm{1}_{n + m < 16N} \bold{a}_{n,m}^2 \| \sum_{n' = 0}^n \sum_{m' =0}^m \binom{n}{n'} \binom{m}{m'} \mathbbm{1}_{n' + m' < n + m} \p_x b_{m - m', n - n'} \Gamma \omega_{m', n'} \chi_{n + m} e^W q^n \chi_I \|_{L^2}^2 \\
& + \sum_{n = 0}^\infty \sum_{m = 0}^\infty  \mathbbm{1}_{n + m < 16N} \bold{a}_{n,m}^2 \| \sum_{n' = 0}^n \sum_{m' =0}^m \binom{n}{n'} \binom{m}{m'} \mathbbm{1}_{n' + m' < n + m} \p_x b_{m - m', n - n'} \Gamma \omega_{m', n'} \chi_{n + m} e^W q^n \\
& \times (1- \chi_I) \|_{L^2}^2 \\
=&: I_{\text{In}} + I_{\text{Out}}.
\end{align*}
The contribution $I_{\text{In}}$ is estimated by following exactly the proof of $\mathcal{R}^{(2)}_{m,n}$ from the previous lemma, without needing to introduce the weight $w_{n,n'}$. As a result, one finds 
\begin{align*}
I_{\text{In}} \lesssim  & \langle t \rangle^2 \| \{  b_{ m, n} \} \|_{X^{0}_\infty(\mathbbm{1}_{m + N < 17N})}^2 \| \{ \omega_{m,n} \} \|_{X^{0}_2(\varphi^{n + 1}\mathbbm{1}_{m + n < 4N} q^n e^W \chi_{n + m})}^2
\end{align*}
We further decompose $I_{\text{Out}}$ depending on the action of the vector-field $\Gamma$ as follows 
\begin{align*}
I_{\text{Out}} \lesssim&  \sum_{n = 0}^\infty \sum_{m = 0}^\infty  \mathbbm{1}_{n + m < 16N} \bold{a}_{n,m}^2 \| \sum_{n' = 0}^n \sum_{m' =0}^m \binom{n}{n'} \binom{m}{m'} \mathbbm{1}_{n' + m' < n + m} \p_x b_{m - m', n - n'} \p_y \omega_{m', n'} \chi_{n + m} e^W q^n \\
& \times (1- \chi_I) \|_{L^2}^2 \\
& + \langle t \rangle^2 \sum_{n = 0}^\infty \sum_{m = 0}^\infty  \mathbbm{1}_{n + m < 16N} \bold{a}_{n,m}^2 \| \sum_{n' = 0}^n \sum_{m' =0}^m \binom{n}{n'} \binom{m}{m'} \mathbbm{1}_{n' + m' < n + m} \p_x b_{m - m', n - n'}  \omega_{m'+1, n'} \chi_{n + m} e^W q^n \\
& \times (1- \chi_I) \|_{L^2}^2 := I_{\text{Out,1}} + I_{\text{Out,2}}.
\end{align*}
The contribution $I_{\text{Out,2}}$ is estimated in the same way as $I_{\text{In}}$, and produces  
\begin{align*}
I_{\text{Out,2}} \lesssim  & \langle t \rangle^2 \| \{  b_{ m, n} \} \|_{X^{0}_\infty(\mathbbm{1}_{m + N < 17N})}^2 \| \{ \omega_{m,n} \} \|_{X^{0}_2(\varphi^{n + 1}\mathbbm{1}_{m + n < 4N} q^n e^W \chi_{n + m})}^2
\end{align*}
The contribution $I_{\text{Out,1}}$ is estimated as follows:  
\begin{align*}
I_{\text{Out,1}} \lesssim  &  \| \{  b_{ m, n} \} \|_{X^{0}_\infty(\chi_2 \mathbbm{1}_{m + N < 17N})}^2 \| \{ \p_y \omega_{m,n} \} \|_{X^{0}_2(\varphi^{n + 1}\mathbbm{1}_{m + n < 4N} q^n e^W \chi_{n + m})}^2.
\end{align*}

\end{proof}

\subsection{Estimates on Abstract Bilinear Operators II: exterior $\mathcal{R}[\cdot, \cdot]$ bounds}

We now provide a second bound on the operator $\mathcal{R}_{m,n}$. 
\begin{lemma} The bilinear operator $\mathcal{R}_{m,n}[\cdot, \cdot]$, defined in \eqref{bilinear:R}, satisfies the following estimates
\begin{align} \n
&\| \mathcal{R}_{m,n}[\{ b_{m,n} \}, \{ \omega_{m,n} \}] \|_{X^0_2(\bold{a}_{m,n} q^{n} e^W \chi_{n+m}^2)} \\ \label{63023:1}
\lesssim & \langle t \rangle \| \{ b_{m,n}\} \|_{X^2_2( \bold{a}_{m,n} q^n \chi_{m+n} )} \| \{ \nabla \omega_{m,n}\} \|_{X^0_2( \bold{a}_{m,n} q^n \chi_{m+n} e^W )}.
\end{align}
\end{lemma}
\begin{proof} We recall the definition \eqref{bilinear:R}. We first treat the case of $T^{(1)}_{m,n}$. We first localize our estimate, as we need to treat the interior region (where issues of cutoff functions become more pronounced) separately from exterior region (where issues of $q$-weights become pronounced). Therefore, we have 
\begin{align} \n
&\| \mathcal{R}_{m,n}[\{ b_{m,n} \}, \{ \omega_{m,n} \}] \|_{X^0_2(\bold{a}_{m,n} q^{n} e^W \chi_{n+m}^2)} \\ \n
\lesssim & \sum_{i = 1}^2 \| \mathcal{R}^{(i)}_{m,n}[\{ b_{m,n} \}, \{ \omega_{m,n} \}] \|_{X^0_2(\bold{a}_{m,n} q^{n} e^W \chi_{n+m}^2 \chi_I)} \\ \n
& +\sum_{i = 1}^2 \| \mathcal{R}^{(i)}_{m,n}[\{ b_{m,n} \}, \{ \omega_{m,n} \}] \|_{X^0_2(\bold{a}_{m,n} q^{n} e^W \chi_{n+m}^2 (1 - \chi_I) )} =: \sum_{i = 1}^2 ( \mathcal{I}^{(i)}_{\text{Bulk}} +   \mathcal{I}^{(i)}_{\text{Bdry}}).
\end{align}
We split the estimates of the above terms into $\mathcal{I}^{(1)}$ and $\mathcal{I}^{(2)}$. 

\vspace{2 mm}

\noindent \underline{Estimate of $\mathcal{I}^{(1)}$:} First, we will treat $\mathcal{I}_{\text{Bulk}}$ for which we used mixed-type norms in order to efficiently distribute the cut-off functions, as follows: 
\begin{align} \label{mast:1}
\mathcal{I}_{\text{Bulk}}^{(1)} = &\| \{\mathbbm{1}_{m + n \le 4N}  T^{(1)}_{m,n}[\{b_{m,n}\}, \{\omega_{m,n}\}]  \|_{{X^{0}_2(\bold{a}_{m,n} q^{n} e^W \chi_{n + m}^2 \chi_I )}}^2 \\ \n
= &\sum_{n = 0}^\infty \sum_{m = 0}^\infty  \mathbbm{1}_{n + m < 4N} \bold{a}_{n,m}^2 \| \sum_{n' = 0}^n \sum_{m' =0}^m \binom{n}{n'} \binom{m}{m'} \mathbbm{1}_{n' + m' < n + m} \Gamma b_{m - m', n - n'} \p_x \omega_{m', n'} \\ \label{mast:2}
& \times \chi_{n + m}^2 e^W q^{n} \chi_I \|_{L^2}^2 \\ \label{mast:3}
\lesssim & \sum_{n = 0}^\infty \sum_{m = 0}^\infty  \mathbbm{1}_{n + m < 4N} \bold{a}_{n,m}^2 \Big(\sum_{n' = 0}^n \sum_{m' =0}^m  \mathbbm{1}_{n' + m' < n + m}\|  \Gamma b_{m - m', n - n'} \p_x \omega_{m', n'} \chi_{n + m}^2 e^W q^{n}  \chi_I \|_{L^2} \Big)^2 \\ \n
\lesssim &  \sum_{n = 0}^\infty \sum_{m = 0}^\infty  \mathbbm{1}_{n + m < 4N} \bold{a}_{n,m}^2 \Big( \sum_{m' = 0}^m \sum_{n' = 0}^n \mathbbm{1}_{n' + m' < n + m} \| \Gamma b_{m-m', n-n'} q^{n-n' - 1} \chi_{(n-n') + (m-m') - 1}^{\frac12} \\ \label{mast:4}
& \qquad \chi_{(n-n') + (m-m')}^{\frac12} \chi_I \|_{L^\infty_x L^2_y} \| \p_x \omega_{m', n'} \chi_{m' + n'}^{\frac12} \chi_{m' + n' + 1}^{\frac12} q^{n' + 1}  \|_{L^2_x L^\infty_y} \Big)^2 \\ \n
\lesssim &  \sum_{n = 0}^\infty \sum_{m = 0}^\infty  \mathbbm{1}_{n + m < 4N} \bold{a}_{n,m}^2 \Big( \sum_{m' = 0}^m \sum_{n' = 0}^n \mathbbm{1}_{n' + m' < n + m} \| \Gamma b_{m-m', n-n'} q^{n-n' - 1} \chi_{(n-n') + (m-m') - 1}^{\frac12} \\ \label{firestone:1}
& \qquad \chi_{(n-n') + (m-m')}^{\frac12}  \chi_I \|_{L^\infty_x L^2_y}^2 \Big) \Big(  \sum_{m' = 0}^m \sum_{n' = 0}^n  \mathbbm{1}_{n + m < 4N} \| \p_x \omega_{m', n'} \chi_{m' + n'}^{\frac12} \chi_{m' + n' + 1}^{\frac12} q^{n' + 1}  \|_{L^2_x L^\infty_y}^2 \Big) \\ \label{defthru:1}
=: & \sum_{n = 0}^\infty \sum_{m = 0}^\infty  \mathbbm{1}_{n + m < 4N} \bold{a}_{n,m}^2 \Big( \sum_{m' = 0}^m \sum_{n' = 0}^n \mathbbm{1}_{n' + m' < n + m}A_{m',n';m,n}^2 \Big) \Big(  \sum_{m' = 0}^m \sum_{n' = 0}^n  \mathbbm{1}_{n + m < 4N}B_{m',n'}^2 \Big). 
\end{align}
Above, to go from \eqref{mast:1} to \eqref{mast:2}, we used the definition \eqref{odesza:1}. To go from \eqref{mast:2} to \eqref{mast:3}, we use the triangle inequality, and the cut-off function $\mathbbm{1}_{m + n} < 4N$ to discard the combinatorial factors. To go from \eqref{mast:3} to \eqref{mast:4}, we have used the following inequality for the sliding cutoff functions: 
\begin{align*}
\chi_{n + m}^2 \le \chi_{n' + m' + 1} \chi_{(n-n') + (m-m')} \le \chi_{n' + m' + 1}^{\frac12} \chi_{(n-n') + (m-m')}^{\frac12} \chi_{n' + m' }^{\frac12} \chi_{(n-n') + (m-m')-1}^{\frac12} 
\end{align*}
which is valid because $n' + m' < n + m$. 

To conclude the estimate, we will provide the following interpolation inequalities: 
\begin{align} \label{oneg:0}
A_{m',n';m,n} = &\|  b_{m-m', n-n'+1} q^{n-n' - 1} \chi_{(n-n') + (m-m') - 1}^{\frac12}  \chi_{(n-n') + (m-m')}^{\frac12}  \chi_I\|_{L^\infty_x L^2_y} \\ \label{oneg:1}
 \lesssim &\|  b_{m-m', n-n'+1} q^{n - n'} \chi_{(n-n') + (m-m') - 1}^{\frac12}  \chi_{(n-n') + (m-m')}^{\frac12}\|_{L^\infty_x L^2_y} \\ \label{oneg:2}
 \lesssim &  \|  b_{m-m', n-n'+1} q^{n-n' } \chi_{(n-n') + (m-m') - 1}\|_{L^2_x L^2_y}^{\frac12}  \|  b_{m-m'+1, n-n'+1} q^{n-n' }   \chi_{(n-n') + (m-m')} \|_{L^2_x L^2_y}^{\frac12} \\ \label{off:the:hook:1}
 \lesssim & \langle t \rangle^2 \| b_{m-m', n-n'} q^{n-n'} \chi_{(m-m') + (n-n')} \|_{H^2_{xy}} 
\end{align}
Above, to go from \eqref{oneg:0} to \eqref{oneg:1}, we use that on the support of $ \chi_I$, the inverse weight function, $q^{-1}$, is bounded. To go from \eqref{oneg:1} to \eqref{oneg:2}, we use that a standard Sobolev interpolation.  

We now estimate the quantity $B_{m',n'}$, defined through \eqref{defthru:1}. To avoid putting derivatives on square-roots of cut-off functions, we square the quantity and estimate it in the following manner: 
\begin{align} \label{tes:1}
 &|\p_x \omega_{m', n'}|^2 \chi_{m' + n'} \chi_{m' + n' + 1} q^{2(n' + 1)} \\ \label{tes:2}
 = & \int_0^y \p_y \{  |\p_x \omega_{m', n'}|^2 \chi_{m' + n'} \chi_{m' + n' + 1} q^{2(n' + 1)}  \} \\ \n
 = & \int_0^y  \p_x \omega_{m', n'}\p_{xy} \omega_{m', n'} \chi_{m' + n'} \chi_{m' + n' + 1} q^{2(n' + 1)}+  \int_0^y  |\p_x \omega_{m', n'}|^2 \chi_{m' + n'}' \chi_{m' + n' + 1} q^{2(n' + 1)} \\ \label{tes:3}
 & +   \int_0^y  |\p_x \omega_{m', n'}|^2 \chi_{m' + n'} \chi_{m' + n' + 1}' q^{2(n' + 1)} +  2(n'+1)  \int_0^y  |\p_x \omega_{m', n'}|^2 \chi_{m' + n'} \chi_{m' + n' + 1}' q^{2(n' + 1)-1} \\ \n
 = &  \int_0^y  \p_x \omega_{m', n'}\p_{xy} \omega_{m', n'} \chi_{m' + n'} \chi_{m' + n' + 1} q^{2(n' + 1)} +   \int_0^y  |\p_x \omega_{m', n'}|^2 \chi_{m' + n'}^2 \frac{ \chi_{m' + n' + 1}'}{\chi_{m' + n'}} q^{2(n' + 1)} \\ \label{tes:4}
 &+  2(n'+1)  \int_0^y  |\p_x \omega_{m', n'}|^2 \chi_{m' + n'} \chi_{m' + n' + 1} q^{2(n' + 1)-1} \\ \label{tes:5}
 \lesssim & \| \p_x \omega_{m',n'} \chi_{m' + n'} q^{n'} \|_{L^2_y} \| \p_x \p_y \omega_{m',n'} \chi_{m' + n'+1} q^{n' + 1} \|_{L^2_y} + \| \p_x \omega_{m',n'} \chi_{m' + n'} q^{n'} \|_{L^2_y}^2.
\end{align}
Above to go from \eqref{tes:1} to \eqref{tes:2}, we use that the factor $\chi_{m' + n' + 1}$ vanishes at $y = 0$. To go from \eqref{tes:3} to \eqref{tes:4}, we use that $\text{supp}(\chi_{m' + n'}') \cap \text{supp}(\chi_{m' + n' + 1}) = \emptyset$, and we also use that $\chi_{m' + n'} = 1$ on the support of $\chi_{m' + n' +1}'$ which allows us to introduce factors of $\chi_{m' + n'}$. According to \eqref{tes:5}, we have
\begin{align} \n
B_{m',n'}^2 \lesssim &  \| \p_x \omega_{m',n'} \chi_{m' + n'} q^{n'} \|_{L^2} \| \p_x \p_y \omega_{m',n'} \chi_{m' + n'+1} q^{n' + 1} \|_{L^2} + \| \p_x \omega_{m',n'} \chi_{m' + n'} q^{n'} \|_{L^2}^2 \\ \n
= & \| \p_x \omega_{m',n'} \chi_{m' + n'} q^{n'} \|_{L^2} \| \p_x\{ \Gamma - t \p_x \} \omega_{m',n'} \chi_{m' + n'+1} q^{n' + 1} \|_{L^2} + \| \p_x \omega_{m',n'} \chi_{m' + n'} q^{n'} \|_{L^2}^2 \\ \label{off:the:hook:2}
\lesssim & \langle t \rangle ( \| \p_x \omega_{m',n'} \chi_{m' + n'} q^{n'} \|_{L^2}^2 + \| \p_x \omega_{m',n'+1} \chi_{m' + n'+1} q^{n'+1} \|_{L^2}^2+  \| \p_x \omega_{m'+1,n'} \chi_{m' + n'+1} q^{n'} \|_{L^2}^2)
\end{align} 
 To conclude the estimate, we insert \eqref{off:the:hook:1} and \eqref{off:the:hook:2} into \eqref{defthru:1}.

We now move to the estimate $\mathcal{I}^{(1)}_{\text{Bdry}}$. 
\begin{align} \label{mmast:1}
\mathcal{I}_{\text{Bdry}}^{(1)} = &\| \{\mathbbm{1}_{m + n \le 4N}  T^{(1)}_{m,n}[\{b_{m,n}\}, \{\omega_{m,n}\}]  \|_{{X^{0}_2(\bold{a}_{m,n} q^{n} e^W \chi_{n + m}^2 (1-\chi_I) )}}^2 \\ \n
= &\sum_{n = 0}^\infty \sum_{m = 0}^\infty  \mathbbm{1}_{n + m < 4N} \bold{a}_{n,m}^2 \| \sum_{n' = 0}^n \sum_{m' =0}^m \binom{n}{n'} \binom{m}{m'} \mathbbm{1}_{n' + m' < n + m} \Gamma b_{m - m', n - n'} \p_x \omega_{m', n'} \\ \label{mmast:2}
& \times \chi_{n + m}^2 e^W q^{n} (1- \chi_I )\|_{L^2}^2 \\ \label{mmast:3}
\lesssim & \sum_{n = 0}^\infty \sum_{m = 0}^\infty  \mathbbm{1}_{n + m < 4N} \bold{a}_{n,m}^2 \Big(\sum_{n' = 0}^n \sum_{m' =0}^m  \mathbbm{1}_{n' + m' < n + m}\|  \Gamma b_{m - m', n - n'} \p_x \omega_{m', n'} \chi_{n + m}^2 e^W q^{n}  (1- \chi_I) \|_{L^2} \Big)^2 \\ \n
\lesssim &  \sum_{n = 0}^\infty \sum_{m = 0}^\infty  \mathbbm{1}_{n + m < 4N} \bold{a}_{n,m}^2 \Big( \sum_{m' = 0}^m \sum_{n' = 0}^n \mathbbm{1}_{n' + m' < n + m} \| \Gamma b_{m-m', n-n'} q^{n-n' } \chi_{(n-n') + (m-m') +10} \|_{L^\infty} \\ \label{mmast:4}
& \qquad \times  \| \p_x \omega_{m', n'} \chi_{m' + n' + 10}q^{n' }  \|_{L^2} \Big)^2 \\ \label{mmast:45}
=: &  \sum_{n = 0}^\infty \sum_{m = 0}^\infty  \mathbbm{1}_{n + m < 4N} \bold{a}_{n,m}^2 \Big( \sum_{m' = 0}^m \sum_{n' = 0}^n \mathbbm{1}_{n' + m' < n + m} B_{m',n';m,n} \| \p_x \omega_{m', n'} \chi_{m' + n' + 10}q^{n' }  \|_{L^2} \Big)^2
\end{align}
To estimate the $L^\infty$ norm above, we have
\begin{align} \label{mntn:1}
B_{m',n';m,n} \lesssim & \|  b_{m-m', n-n'+1} q^{n-n' } \chi_{(n-n') + (m-m') +10} \|_{H^1_x H^1_y} \\ \n
\lesssim & \| \p_y  b_{m-m'+ 1, n-n'+1} q^{n-n' } \chi_{(n-n') + (m-m') +10} \|_{L^2} \\ \n
&+  \|   b_{m-m'+ 1, n-n'+1} \mathbbm{1}_{n' < n} q^{n-n' -1 } \chi_{(n-n') + (m-m') +10} \|_{L^2} \\  \label{mntn:2}
& +  \|   b_{m-m'+ 1, n-n'+1}  q^{n-n'  } \chi'_{(n-n') + (m-m') +10} \|_{L^2} \\  \n
\lesssim & \| \p_y  b_{m-m'+ 1, n-n'+1} q^{n-n' } \chi_{(n-n') + (m-m') +10} \|_{L^2}\\ \label{mntn:3}
& +  \|   b_{m-m'+ 1, n-n'+1}  q^{n-n'  } \chi'_{(n-n') + (m-m') +10} \|_{L^2} \\ \n 
\lesssim & \|   b_{m-m'+ 1, n-n'+2} q^{n-n' } \chi_{(n-n') + (m-m') +10} \|_{L^2} \\ \n
&+ \|   b_{m-m'+ 2, n-n'+1} q^{n-n' } \chi_{(n-n') + (m-m') +10} \|_{L^2} \\  \label{mntn:4}
& +  \|   b_{m-m'+ 1, n-n'+1}  q^{n-n'  } \chi_{(n-n') + (m-m') +9} \|_{L^2}.
\end{align}
 We now insert this estimate \eqref{mmast:45} which concludes the proof for the term $\mathcal{R}^{(1)}$. 
 
\vspace{2 mm}

\noindent \underline{Estimate of $\mathcal{I}^{(2)}$:} First, the bound on $\mathcal{I}^{(2)}_{\text{Bulk}}$ is identical to $\mathcal{I}^{(1)}_{\text{Bulk}}$ due to the presence of $\chi_I$: factors of $q$ do not matter, and only the total derivative count is relevant for the distribution of the cutoff functions. We therefore move to $\mathcal{I}^{(2)}_{\text{Bdry}}$. 
\begin{align} \label{evmmast:1}
\mathcal{I}_{\text{Bdry}}^{(2)} = &\| \{\mathbbm{1}_{m + n \le 4N}  T^{(2)}_{m,n}[\{b_{m,n}\}, \{\omega_{m,n}\}]  \|_{{X^{0}_2(\bold{a}_{m,n} q^{n} e^W \chi_{n + m}^2 (1-\chi_I) )}}^2 \\ \n
= &\sum_{n = 0}^\infty \sum_{m = 0}^\infty  \mathbbm{1}_{n + m < 4N} \bold{a}_{n,m}^2 \| \sum_{n' = 0}^n \sum_{m' =0}^m \binom{n}{n'} \binom{m}{m'} \mathbbm{1}_{n' + m' < n + m} \p_x b_{m - m', n - n'} \Gamma \omega_{m', n'} \\ \label{evmmast:2}
& \times \chi_{n + m}^2 e^W q^{n} (1- \chi_I )\|_{L^2}^2 \\ \label{evmmast:3}
\lesssim & \sum_{n = 0}^\infty \sum_{m = 0}^\infty  \mathbbm{1}_{n + m < 4N} \bold{a}_{n,m}^2 \Big(\sum_{n' = 0}^n \sum_{m' =0}^m  \mathbbm{1}_{n' + m' < n + m}\|  b_{m - m'+1, n - n'}  \omega_{m', n'+1} \chi_{n + m}^2 e^W q^{n}  (1- \chi_I) \|_{L^2} \Big)^2 \\ \n
\lesssim &  \sum_{n = 0}^\infty \sum_{m = 0}^\infty  \mathbbm{1}_{n + m < 4N} \bold{a}_{n,m}^2 \Big( \sum_{m' = 0}^m \sum_{n' = 0}^n \mathbbm{1}_{n' + m' < n + m} \| b_{m-m'+1, n-n'} q^{n-n' } \chi_{(n-n') + (m-m') +10} \|_{L^\infty} \\ \label{evmmast:4}
& \qquad \times  \|  \omega_{m', n'+1} \chi_{m' + n' + 10}q^{n' }  \|_{L^2} \Big)^2 \\ \label{evmmast:45}
\lesssim & \langle t \rangle \sum_{n = 0}^\infty \sum_{m = 0}^\infty  \mathbbm{1}_{n + m < 4N} \bold{a}_{n,m}^2 \Big( \sum_{m' = 0}^m \sum_{n' = 0}^n \mathbbm{1}_{n' + m' < n + m} B_{m',n';m,n} \| \nabla \omega_{m', n'} \chi_{m' + n' + 10}q^{n' }  \|_{L^2} \Big)^2,
\end{align}
where we now define $B_{m',n';m,n}$ through 
\begin{align*}
B_{m',n';m,n} :=  & \|  b_{m-m'+1, n-n'} q^{n-n' } \chi_{(n-n') + (m-m') +10} \|_{L^\infty} \\
\lesssim & \|  b_{m-m'+1, n-n'} q^{n-n' } \chi_{(n-n') + (m-m') +10} \|_{H^1_x H^1_y}.
\end{align*}
From here, the bound works essentially identically to the bounds \eqref{mntn:1} -- \eqref{mntn:4}.  
 
\end{proof}

\subsection{Estimates on Abstract Bilinear Operators III: interior $\mathcal{Q}[\cdot, \cdot]$ bounds}

\begin{lemma} \label{lemma:Q:mn:1} Assume $b_{m,0}|_{y = \pm 1} = 0$. The following bilinear sequential bound holds on the operator $\mathcal{Q}_{m,n}$:
\begin{align} \n
&\| \{ \mathcal{Q}_{m,n}[\{ b_{m,n} \}, \{ \omega_{m,n} \}] \} \|_{X^{0}_2(\frac{\bold{a}_{m,n}}{(m + n)^{\frac12}} q^n e^W \chi_{n+m} )} \\ \label{ewing:1}
 \lesssim & \| \{ \p_y b_{m,n} \} \|_{X^{0}_{\infty}(\bold{a}_{m,n}  \chi_{2} \varphi^{-4}  )} \| \{ \omega_{m,n} \} \|_{X^{0}_2((n + m)^{\frac12}\bold{a}_{m,n} q^n e^W \chi_{n+m} )}.
\end{align}
\end{lemma}
\begin{proof}

The treatment of $\mathcal{Q}_{m,n}^{(i)}$ for $i = 1, 2$ are identical, and so we estimate $\mathcal{Q}^{(1)}_{m,n}[\cdot, \cdot]$. By definition of $\mathcal{Q}_{m,n}$, we need to consider indices $m, n$ such that $m + n \ge 16N$. In particular, this assumption ensures the following identity is valid
\begin{align}
\chi_{m + n} = \chi_{m + n} \chi_2, \qquad m + n \ge 16N. 
\end{align}
This observation will enable us to insert factors of $\chi_2$ in the forthcoming estimates. 
\vspace{2 mm}

\noindent \underline{Case $m \ge 4N$ and $n \ge 4N$} This will be the main case. In this case, we have 
\begin{align} \nonumber
&\| \{ \mathcal{Q}_{m,n}[\{ b_{m,n} \}, \{ \omega_{m,n} \}] \} \|_{X^{0}_2(\frac{\bold{a}_{m,n}}{(m + n)^{\frac12}}, q^n e^W \chi_{n+m} )}^2 \\ \n
=  &\sum_{n = 0}^{\infty} \sum_{m = 0}^{\infty} B_{n,m}^2 \varphi(t)^{2n}  \\  \nonumber
& \times \| (n+m)^{-\frac12} \sum_{n' = 0}^{n} \sum_{m' = 0}^m \mathbbm{1}_{n' + m' < n + m} \binom{n}{n'} \binom{m}{m'} b_{m-m'+1, n-n'} \omega_{m', n' + 1}  e^W q^n \chi_{n + m} \|_{L^2(\mathbb{T} \times [-1,1])}^2 \\  \nonumber
= & \sum_{n = 0}^{\infty} \sum_{m = 0}^{\infty} \Big( \frac{\lambda^{(n+m)}}{(n+m)!} \Big)^{2s'} \varphi(t)^{2n} \\  \label{yonL:1}
& \times \|\frac{1}{(n+m)^{\frac12}} \sum_{n' = 0}^{n} \sum_{m' = 0}^m  \mathbbm{1}_{n' + m' < n + m} \binom{n}{n'} \binom{m}{m'} b_{m-m'+1, n-n'} \omega_{m', n' + 1}  e^W q^n \chi_{n + m} \|_{L^2(\mathbb{T} \times [-1,1])}^2 \\  \n
\le & \sum_{n = 0}^{\infty} \sum_{m = 0}^{\infty}  \Big(  \sum_{n' = 0}^{n} \sum_{m' = 0}^m  \mathbbm{1}_{n' + m' < n + m} \binom{n}{n'} \binom{m}{m'} \frac{1}{(n+m)^{\frac12}}  \Big( \frac{\lambda^{(n+m)}}{(n+m)!} \Big)^{s'} \| b_{m-m'+1, n-n'} w_{n,n'}(y) \varphi^{-1} \chi_2 \|_{L^\infty} \\  \ \label{yonL:2}
& \qquad \qquad \qquad \| \omega_{m', n' + 1}  e^W  \chi_{m' + n' + 1} \varphi^{n'+1} q^{n' + 1}  \|_{L^2} \Big)^2 \\ \n
= & \sum_{n = 0}^{\infty} \sum_{m = 0}^{\infty}  \Big(  \sum_{n' = 0}^{n} \sum_{m' = 0}^m  \mathbbm{1}_{n' + m' < n + m} \binom{n}{n'} \binom{m}{m'}  \frac{1}{(n+m)^{\frac12}} \Big( \frac{(n-n' + m-m')! (m'+n'+1)!}{(n+m)!} \Big)^{s'} \\   \nonumber
& \qquad \qquad \qquad\times \frac{(\lambda^{m-m' + n-n'})^s}{[(n-n' + m-m')!]^s} \|  b_{m-m'+1, n-n'} w_{n,n'}(y) \chi_2   \varphi^{-1}\|_{L^\infty} \\
& \qquad \qquad \qquad \times \frac{ \lambda^{(m' + n' + 1)s} }{[(m' + n' + 1)!]^s}  \| \omega_{m', n' + 1} q^{n' + 1}  e^W \chi_n \varphi^{n'+1} \|_{L^2} \Big)^2 \\  \nonumber
= &  \sum_{n = 0}^{\infty} \sum_{m = 0}^{\infty}  \Big(  \sum_{n' = 0}^{n} \sum_{m' = 0}^m  \mathbbm{1}_{n' + m' < n + m} \bold{X}_{m', n', n, m} \frac{(\lambda^{m-m' + n-n'})^s}{[(n-n' + m-m')!]^s} \| b_{m-m'+1, n-n'} w_{n,n'}(y) \chi_2 \varphi^{-1} \|_{L^\infty} \\ \label{bdfmn1}
& \qquad \qquad \qquad \times \frac{ \lambda^{(m' + n' + 1)s} }{[(m' + n' + 1)!]^s}  \| \omega_{m', n' + 1}  e^W q^{n' + 1} \chi_{m' + n' + 1} \varphi^{n'+1}  \|_{L^2} \Big)^2,
\end{align}
where we define 
\begin{align}
\bold{X}_{m', n', n, m}  := \binom{n}{n'} \binom{m}{m'}  \frac{1}{(n+m)^{\frac12}} \Big( \frac{(n-n' + m-m')! (m'+n'+1)!}{(n+m)!} \Big)^{s},
\end{align}
and we define the weight function 
\begin{align}
w_{n,n'}(y) := \begin{cases} 1 \qquad n' < n \\ \frac{1}{q} \qquad n' = n \end{cases}.
\end{align}
According to this definition, the following inequality is valid (and is used to go from \eqref{yonL:1} to \eqref{yonL:2}):
\begin{align*}
 q^n \le q^{n' + 1}w_{n,n'}(y), \qquad n' \le n.
\end{align*} 
A further consequence of this weight is the following Hardy-type inequality: 
\begin{align*}
\| b_{m-m'+1, n-n'} w_{n,n'}(y) \chi_2 \|_{L^\infty} \lesssim & \mathbbm{1}_{n' = n} ( \| \p_y b_{m-m'+1, n-n'} \chi_2  \|_{L^\infty} + \|  b_{m-m'+1, n-n'} \chi_1  \|_{L^\infty} ) \\
&+ \mathbbm{1}_{n' < n} \| b_{m-m'+1, n-n'}  \chi_2 \|_{L^\infty}.
\end{align*}

We now decompose the double sum above into three components: the ``hi", ``mid", and ``lo" components:
\begin{subequations}
\begin{align}
\mathbb{N}_{m,n}^{(hi)} := & \{ (m', n') \in \mathbb{N}^2: m - N \le m' \le m, n - N \le n' \le n-1 \} \\
\mathbb{N}_{m,n}^{(lo)} := & \{ (m', n') \in \mathbb{N}^2: 0 \le m' \le N, 0 \le n' \le N \} \\
\mathbb{N}_{m, n}^{(mid)} := & \{(m', n') \in \mathbb{N}^2: 0 \le m' \le m, 0 \le n' \le n -1 \} - \mathbb{N}_{m,n}^{(hi)} - \mathbb{N}_{m,n}^{(lo)}. 
\end{align}
\end{subequations}
Next, we define the associated sum as follows 
\begin{align*}
S^{(\iota)} := & \sum_{n = 0}^{\infty} \sum_{m = 0}^{\infty} \Big( \sum_{(m', n') \in \mathbb{N}_{m,n}^{(\iota)}} \bold{X}(m', n', n, m) \frac{(\lambda^{m-m' + n-n'})^s}{[(n-n' + m-m')!]^s} \| \p_y b_{m-m', n-n'} \chi_2 \varphi^{-1}\|_{L^\infty} \\ 
& \qquad \qquad \qquad \times \frac{ \lambda^{(m' + n' + 1)s} }{[(m' + n' + 1)!]^s}  \| \omega_{m', n' + 1}  e^W \chi_n \varphi^{n'+1} \|_{L^2} \Big)^2
\end{align*}
Correspondingly, we have 
\begin{align*}
\eqref{bdfmn1} = \sum_{\iota \in \{ \text{hi, mid, low} \}} S^{(\iota)}.
\end{align*}

\noindent \underline{Bounding $S^{(hi)}$} In this region, we have the following bound on the multiplier: 
\begin{align} \n
\bold{X}(m', n', n, m) \mathbbm{1}_{\mathbb{N}_{m,n}^{(hi)}}(m', n') \lesssim & \binom{n}{n'} \binom{m}{m'} \frac{1}{(n+m)^{\frac12}}  \Big( \frac{(m' + n' + 1)!}{(n+m)!} \Big)^s \\ \n
\lesssim & \frac{n!}{(n')!} \frac{m!}{(m')!} \frac{1}{(n+m)^{\frac12}}\Big( \frac{(m' + n' + 1)!}{(n+m)!} \Big)^s \\ \n
\lesssim & n^{n-n'} m^{m-m'} \frac{1}{(n+m)^{\frac12}} \Big( \frac{1}{(n+m)^{(n-n' - 1 + m - m')}} \Big)^s \\ \n
\lesssim & n n^{n-n' - 1} m^{m-m'}\frac{1}{(n+m)^{\frac12}} \Big( \frac{1}{(n+m)^{(n-n' - 1 + m - m')}} \Big)^s \\ \n
\lesssim & \frac{n}{(n+m)^{\frac12}} (n + m)^{n- n' - 1 + m - m'}\Big( \frac{1}{(n+m)^{(n-n' - 1 + m - m')}} \Big)^s \\ \label{mb1}
\lesssim & (n+m)^{\frac12}.
\end{align}
Due to this multiplier bound, we are led to estimate these bilinear contributions as follows 
\begin{align} \n
|S^{(hi)}| \lesssim &  \sum_{n = 0}^\infty \sum_{m = 0}^\infty \Big(\sum_{(m', n') \in \mathbb{N}^{(hi)}_{m,n}} \frac{(\lambda^{m-m' + n-n'})^{2s}}{[(n-n' + m-m')!]^{2s}} \| \p_y b_{m-m', n-n'} \chi_2  \varphi^{-1} \|_{L^\infty}^2 \Big) \\ \label{sga:1}
 \times & \Big(\sum_{(m', n') \in \mathbb{N}^{(hi)}_{m,n}} (n+m) \frac{ \lambda^{2(m' + n' + 1)s} }{[(m' + n' + 1)!]^{2s}}  \| \omega_{m', n' + 1}  e^W \chi_{m' + n' + 1} q^{n' + 1} \varphi^{n'+1}  \|_{L^2}^2 \Big) \\ \n
 \lesssim & \Big(\sum_{\substack{0 \le m' \le N \\ 0 \le n' \le N} } \frac{(\lambda^{m' + n'})^{2s}}{[(n' + m')!]^{2s}} \|\p_y  b_{m', n'} \chi_2  \varphi^{-1}\|_{L^\infty}^2\Big) \\ \label{sga:2}
 \times & \Big( \sum_{n = 0}^\infty \sum_{m = 0}^\infty \sum_{(m', n') \in \mathbb{N}^{(hi)}_{m,n}} (n+m) \frac{ \lambda^{2(m' + n' + 1)s} }{[(m' + n' + 1)!]^{2s}}  \| \omega_{m', n' + 1}  e^W \chi_n \varphi^{n'+1}  \|_{L^2}^2\Big) \\ \n
  \lesssim & \Big(\sum_{\substack{0 \le m' \le N \\ 0 \le n' \le N} } \frac{(\lambda^{m' + n'})^{2s}}{[(n' + m')!]^{2s}} \|\p_y b_{m', n'} \chi_2 \varphi^{-1}\|_{L^\infty}^2\Big) \\ \label{sga:3}
 \times & \Big( \sum_{n = 0}^\infty \sum_{m = 0}^\infty \sum_{(m', n') \in \mathbb{N}^{(hi)}_{m,n}} (n'+m' + 1) \frac{ \lambda^{2(m' + n' + 1)s} }{[(m' + n' + 1)!]^{2s}}  \| \omega_{m', n' + 1}  e^W \chi_{m' + n' + 1} q^{n' + 1} \varphi^{n'+1} \|_{L^2}^2\Big) \\ \n
   \lesssim & \Big(\sum_{\substack{0 \le m' \le N \\ 0 \le n' \le N} } \frac{(\lambda^{m' + n'})^{2s}}{[(n' + m')!]^{2s}} \|\p_y b_{m', n'} \chi_2  \varphi^{-1}\|_{L^\infty}^2\Big) \\ \label{sga:4}
 \times & \Big( \sum_{n' = 0}^\infty \sum_{m' = 0}^\infty \sum_{ \substack{m' \le m \le m' + N \\ n' \le n \le n' + N }  } (n'+m' + 1) \frac{ \lambda^{2(m' + n' + 1)s} }{[(m' + n' + 1)!]^{2s}}  \| \omega_{m', n' + 1}  e^W \chi_{m' + n' + 1} q^{n' + 1} \varphi^{n'+1}  \|_{L^2}^2\Big) \\ \n
 \lesssim & \Big(\sum_{\substack{0 \le m' \le N \\ 0 \le n' \le N} } \frac{(\lambda^{m' + n'})^{2s}}{[(n' + m')!]^{2s}} \| \p_y b_{m', n'} \chi_2  \varphi^{-1}\|_{L^\infty}^2\Big) \\ \label{sga:5}
 \times & \Big( \sum_{n' = 0}^\infty \sum_{m' = 0}^\infty  (n'+m' ) \frac{ \lambda^{2(m' + n' )s} }{[(m' + n' )!]^{2s}}  \| \omega_{m', n' }  e^W q^{n'} \chi_{n' + m'} \varphi^{n'}  \|_{L^2}^2\Big).
\end{align}
For estimate \eqref{sga:1}, we have used Cauchy-Schwartz as well as the multiplier bound \eqref{mb1}. For estimate \eqref{sga:2}, we have re-indexed the first sum by defining $m' \mapsto m - m'$ and $n' \mapsto n - n'$. For estimate \eqref{sga:3}, we have used that $n + m$ and $n' + m' + 1$ are comparable in the support of $\mathbb{N}^{(hi)}_{n,m}$:
\begin{align*}
c \le \bold{1}_{\mathbb{N}_{n,m}^{(hi)}} \frac{n+m}{n' + m' + 1} \le C.
\end{align*}
For estimate \eqref{sga:4} we have re-indexed the region $m \in \mathbb{N}, m - N \le m' \le m$ into $m' \in \mathbb{N}, m' \le m \le m' + N$ (and similarly for $n, n'$).
\vspace{2 mm}

\noindent \underline{Bounding $S^{(lo)}$:} Here, we have the following multiplier bound:
\begin{align} \n
\bold{X}(m', n', n, m)  \mathbbm{1}_{\mathbb{N}_{m,n}^{(lo)}}(m', n') \lesssim & \binom{n}{n'} \binom{m}{m'} \frac{1}{(n+m)^{\frac12}}  \Big( \frac{(m - m' + n- n')!}{(n+m)!} \Big)^s \\ \n 
\lesssim & \frac{n!}{(n-n')!} \frac{m!}{(m-m')!} \frac{1}{(n+m)^{\frac12}} \Big( \frac{(m - m' + n- n')!}{(n+m)!} \Big)^s \\ \n
\lesssim & n^{n'} m^{m'} \frac{1}{(n+m)^{\frac12}} \Big( \frac{1}{(n+m)^{n' + m'}} \Big)^s \lesssim 1. 
\end{align}
Therefore, we resume as follows
\begin{align*}
|S^{(lo)}| \lesssim & \sum_{n = 0}^\infty \sum_{m = 0}^\infty \Big( \sum_{(m', n') \in \mathbb{N}_{m,n}^{(lo)}} \frac{(\lambda^{m-m' + n - n'})^{2s}}{[(n-n' +m-m')!]^{2s}} \| \p_y b_{m-m', n - n'} \chi_{2} \varphi^{-1} \|_{L^\infty}^2  \Big) \\
& \times \Big( \sum_{(m', n') \in \mathbb{N}^{(lo)}_{m,n}}  \frac{ \lambda^{2(m' + n' + 1)s} }{[(m' + n' + 1)!]^{2s}}  \| \omega_{m', n' + 1}  e^W \chi_{m' + n' + 1} q^{n' + 1} \varphi^{n'+1} \|_{L^2}^2  \Big) \\
\lesssim & \sum_{(m', n') \in \mathbb{N}^{(lo)}_{m,n}}  \frac{ \lambda^{2(m' + n' + 1)s} }{[(m' + n' + 1)!]^{2s}}  \| \omega_{m', n' + 1}  e^W q^{n' + 1} \chi_{m' + n' + 1} \varphi^{n'+1} \|_{L^2}^2 \\
& \times   \sum_{n = 0}^\infty \sum_{m = 0}^\infty  \sum_{m' = m-N}^m  \sum_{n' = n - N}^n \frac{(\lambda^{m' +  n'})^{2s}}{[(n' +m')!]^{2s}} \| \p_y b_{m', n'} \chi_{2} \varphi^{-1} \|_{L^\infty}^2  \Big) \\
\lesssim & \sum_{(m', n') \in \mathbb{N}^{(lo)}_{m,n}}  \frac{ \lambda^{2(m' + n' + 1)s} }{[(m' + n' + 1)!]^{2s}}  \| \omega_{m', n' + 1}  e^W q^{n' + 1} \chi_{m' + n' + 1} \varphi^{n'+1} \|_{L^2}^2 \\
& \times   \sum_{m' = 0}^\infty \sum_{n' = 0}^\infty  \sum_{m = m'}^{m' + N}  \sum_{n = n' }^{n' + N} \frac{(\lambda^{m' +  n'})^{2s}}{[(n' +m')!]^{2s}} \| \p_y b_{m', n'} \chi_{2} \varphi^{-1} \|_{L^\infty}^2  \Big).
\end{align*}
\vspace{2 mm}

\noindent \underline{Bounding $S^{(mid)}$} Here, we have
\begin{align} \n
&\bold{X}(m', n', n, m)   \mathbbm{1}_{\mathbb{N}_{m,n}^{(mid)}}(m', n') \\ \n
= & \mathbbm{1}_{\mathbb{N}_{m,n}^{(mid)}}(m', n') \binom{n}{n'} \binom{m}{m'} \frac{1}{(n+m)^{\frac12}}  \binom{n + m}{n' + m'}^{-s} (1 + n' + m')^s \\
\lesssim & \mathbbm{1}_{\mathbb{N}_{m,n}^{(mid)}}(m', n') \binom{n + m}{n' + m'} \frac{1}{(n+m)^{\frac12}}\binom{n + m}{n' + m'}^{-s} (1 + n' + m')^s \\ \label{62923:1}
\lesssim &  \mathbbm{1}_{\mathbb{N}_{m,n}^{(mid)}}(m', n') \binom{n+m}{n' + m'}^{- (s-1)} \frac{1}{(n+m)^{\frac12}}(n + m)^s \\
\lesssim & (n + m)^{-C(N)}. 
\end{align}
We therefore have 
\begin{align*}
|S^{(mid)}| \lesssim & \sum_{n = 0}^\infty \sum_{m = 0}^\infty (n+m)^{-C(N)} \Big( \sum_{(m', n') \in \mathbb{N}_{m,n}^{(mid)}} \frac{(\lambda^{m-m' + n - n'})^{2s}}{[(n-n' +m-m')!]^{2s}} \| \p_y b_{m-m', n - n'} \chi_{2}  \varphi^{-1}\|_{L^\infty}^2  \Big) \\
& \times \Big( \sum_{(m', n') \in \mathbb{N}^{(mid)}_{m,n}}   \frac{ \lambda^{2(m' + n' + 1)s} }{[(m' + n' + 1)!]^{2s}}  \| \omega_{m', n' + 1}  q^{n' + 1} e^W \chi_{m' + n' + 1} \varphi^{n'+1}  \|_{L^2}^2  \Big) \\
\lesssim &\Big( \sum_{m' = N}^\infty \sum_{n' = N}^\infty \frac{(\lambda^{m' + n'})^{2s}}{[(n' + m')!]^{2s}} \| \p_y b_{m', n'} \chi_{2}  \varphi^{-1}\|_{L^\infty}^2 \Big) \\
& \times \Big(  \sum_{m' = N}^\infty \sum_{n' = N}^\infty   \frac{ \lambda^{2(m' + n' + 1)s} }{[(m' + n' + 1)!]^{2s}}  \| \omega_{m', n' + 1}  e^W q^{n' + 1} \chi_{m' + n' + 1} \varphi^{n'+1} \|_{L^2}^2  \Big) \\
& \times \sum_{n = 0}^\infty \sum_{m = 0}^\infty (n + m)^{-C(N)}.
\end{align*}

\end{proof}

As with the $\mathcal{R}_{m,n}[\cdot, \cdot]$ operators, we need a lemma which applies to the case when $b_{m,0}$ does not vanish at $y = \pm 1$. In this case, it turns out we compensate by losing a derivative on $\omega_{m,n}$.
\begin{lemma} \label{lem:tr:1} The following bilinear sequential bound holds on the operator $\mathcal{Q}_{m,n}$:
\begin{align} \n
&\| \{ \mathcal{Q}_{m,n}[\{ b_{m,n} \}, \{ \omega_{m,n} \}] \} \|_{X^{0}_2(\frac{\bold{a}_{m,n}}{(m + n)^{\frac12}} q^n e^W \chi_{n+m} )} \\ \label{e:money:1}
 \lesssim & \| \{  b_{m,n} \} \|_{X^{0}_{\infty}(\bold{a}_{m,n}  \chi_{2} \varphi^{-4}  )} \| \{ \nabla \omega_{m,n} \} \|_{X^{0}_2((n + m)^{\frac12}\bold{a}_{m,n} q^n e^W \chi_{n+m} )}.
\end{align}
\end{lemma}
\begin{proof} We essentially repeat the calculations from the previous lemma that resulted in \eqref{bdfmn1}, with one change: 
\begin{align*} \nonumber
&\| \{ \mathcal{Q}_{m,n}[\{ b_{m,n} \}, \{ \omega_{m,n} \}] \} \|_{X^{0}_2(\frac{\bold{a}_{m,n}}{(m + n)^{\frac12}}, q^n e^W \chi_{n+m} )}^2 \\ \n
=  &\sum_{n = 0}^{\infty} \sum_{m = 0}^{\infty} B_{n,m}^2 \varphi(t)^{2n}  \\  \nonumber
& \times \| (n+m)^{-\frac12} \sum_{n' = 0}^{n} \sum_{m' = 0}^m \mathbbm{1}_{n' + m' < n + m} \binom{n}{n'} \binom{m}{m'} b_{m-m'+1, n-n'} \omega_{m', n' + 1}  e^W q^n \chi_{n + m} \|_{L^2(\mathbb{T} \times [-1,1])}^2 \\  \nonumber
= & \sum_{n = 0}^{\infty} \sum_{m = 0}^{\infty} \Big( \frac{\lambda^{(n+m)}}{(n+m)!} \Big)^{2s'} \varphi(t)^{2n} \\ 
& \times \|\frac{1}{(n+m)^{\frac12}} \sum_{n' = 0}^{n} \sum_{m' = 0}^m  \mathbbm{1}_{n' + m' < n + m} \binom{n}{n'} \binom{m}{m'} b_{m-m'+1, n-n'} (\frac{1}{v_y} \p_y + t \p_x) \omega_{m', n' }  e^W q^n \chi_{n + m} \|_{L^2(\mathbb{T} \times [-1,1])}^2 \\  \n
\le & \sum_{n = 0}^{\infty} \sum_{m = 0}^{\infty}  \Big(  \sum_{n' = 0}^{n} \sum_{m' = 0}^m  \mathbbm{1}_{n' + m' < n + m} \binom{n}{n'} \binom{m}{m'} \frac{1}{(n+m)^{\frac12}}  \Big( \frac{\lambda^{(n+m)}}{(n+m)!} \Big)^{s'} \| b_{m-m'+1, n-n'} \varphi^{-1} \chi_2 \|_{L^\infty} \\  
& \qquad \qquad \qquad \| \nabla \omega_{m', n'}  e^W  \chi_{m' + n' } \varphi^{n'+1} q^{n'}  \|_{L^2} \Big)^2 \\ \n
= & \sum_{n = 0}^{\infty} \sum_{m = 0}^{\infty}  \Big(  \sum_{n' = 0}^{n} \sum_{m' = 0}^m  \mathbbm{1}_{n' + m' < n + m} \binom{n}{n'} \binom{m}{m'}  \frac{1}{(n+m)^{\frac12}} \Big( \frac{(n-n' + m-m')! (m'+n')!}{(n+m)!} \Big)^{s'} \\   \nonumber
& \qquad \qquad \qquad\times \frac{(\lambda^{m-m' + n-n'})^s}{[(n-n' + m-m')!]^s} \|  b_{m-m'+1, n-n'}  \chi_2   \varphi^{-1}\|_{L^\infty} \\
& \qquad \qquad \qquad \times \frac{ \lambda^{(m' + n' )s} }{[(m' + n' )!]^s}  \| \nabla \omega_{m', n'} q^{n' }  e^W \chi_{m' + n'} \varphi^{n'+1} \|_{L^2} \Big)^2 \\  \nonumber
= &  \sum_{n = 0}^{\infty} \sum_{m = 0}^{\infty}  \Big(  \sum_{n' = 0}^{n} \sum_{m' = 0}^m  \mathbbm{1}_{n' + m' < n + m} \bold{X}_{m', n', n, m} \frac{(\lambda^{m-m' + n-n'})^s}{[(n-n' + m-m')!]^s} \| b_{m-m'+1, n-n'}  \chi_2 \varphi^{-1} \|_{L^\infty} \\ 
& \qquad \qquad \qquad \times \frac{ \lambda^{(m' + n' )s} }{[(m' + n' )!]^s}  \| \omega_{m', n'}  e^W q^{n' } \chi_{m' + n' } \varphi^{n'+1}  \|_{L^2} \Big)^2.
\end{align*}
From here, the bounds follow identically to the previous lemma. 
\end{proof}
%
%

\subsection{Estimates on Abstract Bilinear Operators IV: exterior $\mathcal{Q}[\cdot, \cdot]$ bounds}

We will need to prove a different bound on the bilinear operator $\mathcal{Q}[\cdot, \cdot]$ which we will invoke in the case when $b_{m,n} = \psi^{(E)}_{m,n}$. Here the bounds are dictated by balancing weights of $q$ (which is not so important in the case when $b = \phi^{(E)}$. 

Let us motivate how we are performing the bounds below. Formally performing a derivative count on the trilinear product $\Gamma b_{m-m', n-n'} \p_x \omega_{m', n'} \omega_{m,n}$, we see that there are $1 + (m-m') + (n-n') - 2 + 1 + m' + n' + m + n = 2(m + n)$ derivatives distributed among $\omega^3$, where we subtract $2$ due to the disparity between $\psi$ and $\omega$. On the other hand, there are $q^{2(m+n)}$ weights. This appears to work out exactly in balance, with the only discrepancy being that these considerations are only in $L^2$, whereas in the real estimate we need to put one factor in $L^\infty$. To account for this discrepancy, we need to invoke excess factors of $\nu$. 

\begin{lemma} The following bounds are valid: 
\begin{align} \n
&\| \{ \mathcal{Q}_{m,n}[\{ b_{m,n} \}, \{ \omega_{m,n} \}] \} \|_{X^{0}_2(\frac{\bold{a}_{m,n}}{(m + n)^{s},}  q^{n} e^W \chi_{n+m}^2 )} \\ \label{ewing:2}
 \lesssim & \frac{\langle t \rangle^2}{\nu^2} \| \{b_{m,n} \} \|_{X^{2}_2(\bold{a}_{m,n}, q^n \chi_{m+n})} \| \{ \sqrt{\nu} \nabla  \omega_{m,n} \} \|_{X^{0}_2(\bold{a}_{m,n}, e^W \chi_{m + n} q^n)} 
\end{align}
\end{lemma}
\begin{proof}

\vspace{2 mm}

\noindent \underline{Case $m \ge 4N$ and $n \ge 4N$} This will be the main case. In this case, we have 
\begin{align} \nonumber
&\| \{ \mathcal{Q}_{m,n}[\{ b_{m,n} \}, \{ \omega_{m,n} \}] \} \|_{X^{0}_2(\frac{\bold{a}_{m,n}}{(m + n)^{s},}  q^{n} e^W \chi_{n+m}^2 )}^2 \\ \n
=  &\sum_{n = 0}^{\infty} \sum_{m = 0}^{\infty} B_{n,m}^2 \varphi(t)^{2(n+m)}  \\  \nonumber
& \times \| (n+m)^{-s} \sum_{n' = 0}^{n-1} \sum_{m' = 0}^m \binom{n}{n'} \binom{m}{m'} \slashed{\nabla}^\perp b_{m-m', n-n'} \cdot \slashed{\nabla} \omega_{m', n' } q^{n  } e^W \chi_{n + m}^2 \|_{L^2(\mathbb{T} \times [-1,1])}^2 \\  \nonumber
= & \sum_{n = 0}^{\infty} \sum_{m = 0}^{\infty} \Big( \frac{\lambda^{(n+m)}}{(n+m)!} \Big)^{2s'} \varphi(t)^{2(n+m)}  \\  \nonumber
& \times \|\frac{1}{(n+m)^{s}} \sum_{n' = 0}^{n-1} \sum_{m' = 0}^m \binom{n}{n'} \binom{m}{m'}  \slashed{\nabla}^\perp b_{m-m', n-n'} \cdot \slashed{\nabla} \omega_{m', n' }   e^W q^{n} \chi_{n + m}^2 \|_{L^2(\mathbb{T} \times [-1,1])}^2 \\
= & \sum_{n = 0}^{\infty} \sum_{m = 0}^{\infty}  \Big(  \sum_{n' = 0}^{n-1} \sum_{m' = 0}^m \binom{n}{n'} \binom{m}{m'}  \frac{1}{(n+m)^{s}} \Big( \frac{(n-n' + m-m' + 1)! (m'+n'+1)!}{(n+m)!} \Big)^{s} \\   \nonumber
& \qquad \qquad \qquad \times \frac{(\lambda^{m-m' + n-n' + 1})^s}{[(n-n' + m-m' + 1)!]^s}\frac{ \lambda^{(m' + n' + 1)s} }{[(m' + n' + 1)!]^s} \\
& \qquad \qquad \qquad \times   \|  \slashed{\nabla}^\perp b_{m-m', n-n'} \cdot \slashed{\nabla} \omega_{m', n' }   e^W \chi_{n+m}^2 \varphi^{n+m} q^{n }  \|_{L^2} \Big)^2 \\
= & \sum_{n = 0}^{\infty} \sum_{m = 0}^{\infty}  \Big(  \sum_{n' = 0}^{n-1} \sum_{m' = 0}^m \bold{Y}(m', n', m, n)  \frac{(\lambda^{m-m' + n-n' + 1})^s}{[(n-n' + m-m' + 1)!]^s}\frac{ \lambda^{(m' + n' + 1)s} }{[(m' + n' + 1)!]^s} \\
& \qquad \qquad \qquad \times   \|  \slashed{\nabla}^\perp b_{m-m', n-n'} \cdot \slashed{\nabla} \omega_{m', n' }   e^W \chi_{n+m}^2 \varphi^{n+m} q^{n }  \|_{L^2} \Big)^2 \\
= & S^{(hi)} + S^{(mid)} + S^{(lo)}, 
\end{align}
where we define the multiplier 
\begin{align}
\bold{Y}_{m', n', n, m}  := \binom{n}{n'} \binom{m}{m'}  \frac{1}{(n+m)^{s}} \Big( \frac{(n-n' + m-m' + 1)! (m'+n'+1)!}{(n+m)!} \Big)^{s}.
\end{align}
First, we follow a similar strategy to \eqref{mb1} to obtain the bound 
\begin{align} \n
\bold{Y}(m', n', n, m) \mathbbm{1}_{\mathbb{N}_{m,n}^{(hi)}}(m', n') \lesssim & \binom{n}{n'} \binom{m}{m'} \frac{1}{(n+m)^{s}}  \Big( \frac{(m' + n' + 1)!}{(n+m)!} \Big)^s \\ \n
\lesssim & \frac{n!}{(n')!} \frac{m!}{(m')!} \frac{1}{(n+m)^{s}}\Big( \frac{(m' + n' + 1)!}{(n+m)!} \Big)^s \\ \n
\lesssim & n^{n-n'} m^{m-m'} \frac{1}{(n+m)^{s}} \Big( \frac{1}{(n+m)^{(n-n' - 1 + m - m')}} \Big)^s \\ \n
\lesssim & n^{n-n'} m^{m-m'}  \Big( \frac{1}{(n+m)^{(n-n'  + m - m')}} \Big)^s \\
\lesssim & (n + m)^{-(n-n' + m-m')(s-1)},
\end{align}
which in particular is bounded by $1$. 

Similarly, we can bound 
\begin{align*}
\bold{Y}(m', n', n, m) \mathbbm{1}_{\mathbb{N}_{m,n}^{(lo)}}(m', n') \lesssim & \binom{n}{n'} \binom{m}{m'} \frac{1}{(n+m)^{s}}  \Big( \frac{(m-m' + n-n' + 1)!}{(n+m)!} \Big)^s \\
= &  \binom{n}{n'} \binom{m}{m'}  \Big(\frac{n-n' + m-m' + 1}{n+m}\Big)^s  \Big( \frac{(m-m' + n-n' )!}{(n+m)!} \Big)^s \\
\lesssim & \frac{n!}{(n-n')!} \frac{m!}{(m-m')!}  \Big( \frac{(m-m' + n-n' )!}{(n+m)!} \Big)^s  \\
\lesssim & n^{n'} m^{m'} \Big( \frac{1}{(n+m)^{n' + m'}} \Big)^s \\
\lesssim & 1. 
\end{align*}

We now have the mid range, where we need to extract some decay of the multiplier. We estimate in the same way as \eqref{62923:1} to obtain 
\begin{align*}
\bold{Y}(m', n', n, m) \mathbbm{1}_{\mathbb{N}_{m,n}^{(mid)}}(m', n') \lesssim (m + n)^{-C(N)},
\end{align*}
for a constant $C(N)$ where $C(N) \rightarrow \infty$ as $N \rightarrow \infty$. 

\vspace{2 mm}

\noindent \underline{Bounding $S^{(hi)}$} In the case of $S^{(hi)}$, we want to put the $b$ term in $L^\infty$. We notice that $(m', n') \in \mathbb{N}^{(hi)}_{m,n}$ implies the following inequalities 
\begin{align}
&m' \ge m - N \ge 3N, \qquad m - m' \le N, \qquad \Rightarrow m- m' << m' \\
&n' \ge n - N \ge 3N, \qquad n - m' \le N, \qquad  \Rightarrow n- n' << n'.
\end{align}
This then implies the following inequalities for our sliding cut-off functions
\begin{align*}
\mathbbm{1}_{(m', n') \in \mathbb{N}^{(hi)}_{m,n}} \chi_{m + n}^2 \lesssim \mathbbm{1}_{(m', n') \in \mathbb{N}^{(hi)}_{m,n}} \chi_{m' + n' + 1} \chi_{(m-m') + (n-n') + 2}.
\end{align*}

We now proceed as follows: 
\begin{align*}
|S^{(hi)}| \lesssim &  \sum_{n = 0}^{\infty} \sum_{m = 0}^{\infty}  \Big(  \sum_{n' = 0}^{n-1} \sum_{m' = 0}^m \mathbbm{1}_{(m', n') \in \mathbb{N}^{(hi)}_{m,n}} \frac{(\lambda^{m-m' + n-n'})^s}{[(n-n' + m-m')!]^s}\frac{ \lambda^{(m' + n' + 1)s} }{[(m' + n' + 1)!]^s} \\
& \qquad \qquad \qquad \times  \frac{\langle t \rangle}{\sqrt{\nu}} \| \nabla b_{m-m', n-n'} q^{n-n'} \chi_{m-m' + n-n' + 2} \|_{L^\infty} \\
&\qquad \qquad \qquad  \times \| \sqrt{\nu} \langle t \rangle \nabla \omega_{m', n' }  e^W \chi_{n'+m' + 1} \varphi^{n+m} q^{n'}  \|_{L^2} \Big)^2 \\
\lesssim & \frac{\langle t \rangle^2}{\sqrt{\nu}}  \sum_{n = 0}^{\infty} \sum_{m = 0}^{\infty} S^{(hi, 1)}_{m,n} S^{(hi, 2)}_{m,n},
\end{align*}
where we define the quantities 
\begin{align*}
S^{(hi, 1)}_{m,n} := &  \sum_{n' = 0}^{n-1} \sum_{m' = 0}^m \mathbbm{1}_{(m', n') \in \mathbb{N}^{(hi)}_{m,n}} \frac{(\lambda^{m-m' + n-n' + 1})^{2s}}{[(n-n' + m-m' + 1)!]^{2s}}   \|  \nabla b_{m-m' , n-n'} \varphi^{n-n'+m-m'} \\
& \qquad \times q^{n-n' } \chi_{m-m' + n-n' + 2} \|_{L^\infty}^2, \\
S^{(hi, 2)}_{m,n} := &   \sum_{n' = 0}^{n-1} \sum_{m' = 0}^m \mathbbm{1}_{(m', n') \in \mathbb{N}^{(hi)}_{m,n}} \frac{ \lambda^{2(m' + n' + 1)s} }{[(m' + n' + 1)!]^{2s}}  \| \sqrt{\nu} \nabla \omega_{m', n' }  e^W \chi_{n'+m' + 1} \\
& \qquad \times \varphi^{n'+m'} q^{n'}  \|_{L^2}^2.
\end{align*}
We now provide bounds on each of these quantities. First, in the support of $\mathbb{N}^{(hi)}_{m,n}$, $n-n'$ and $m-m'$ are bounded uniformly. Therefore, we may estimate 
\begin{align} \n
|S^{(hi, 1)}_{m,n}|  \lesssim &  \sum_{n' = 0}^{n-1} \sum_{m' = 0}^m \mathbbm{1}_{(m', n') \in \mathbb{N}^{(hi)}_{m,n}} \frac{(\lambda^{m-m' +1 + n-n'})^{2s}}{[(n-n' + m-m' + 1)!]^{2s}}   \| \frac{ \langle t \rangle }{\nu} \nabla b_{m-m', n-n'} \varphi^{n-n'+m-m' +1}  \\ \label{62903:1}
& \qquad \times q^{n-n' } \chi_{m-m' + n-n' + 2} \|_{L^\infty}^2 \\ \n
\lesssim & \langle t \rangle  \sum_{n' = 0}^{n-1} \sum_{m' = 0}^m \mathbbm{1}_{(m', n') \in \mathbb{N}^{(hi)}_{m,n}} \frac{(\lambda^{m-m' +2 + n-n'})^{2s}}{[(n-n' + m-m' + 2)!]^{2s}}   \| \frac{ \langle t \rangle }{\nu} \nabla b_{m-m', n-n'} \varphi^{n-n'+m-m' +2}  \\ \label{62903:2}
& \qquad \times q^{n-n' } \chi_{m-m' + n-n' + 2} \|_{L^\infty}^2 \\ \n
\lesssim &  \sum_{n' = 0}^{n-1} \sum_{m' = 0}^m \mathbbm{1}_{(m', n') \in \mathbb{N}^{(hi)}_{m,n}} \frac{(\lambda^{m-m' +2 + n-n'})^{2s}}{[(n-n' + m-m' + 2)!]^{2s}}   \| \frac{ \langle t \rangle }{\nu} \nabla b_{m-m', n-n'} \varphi^{n-n'+m-m' +2}  \\  \label{62903:3}
& \qquad \times q^{n-n' } \chi_{m-m' + n-n' + 2} \|_{L^\infty}^2 \\  \label{62903:4}
\lesssim &   \sum_{n' = 0}^{\infty} \sum_{m' = 0}^{\infty} \frac{(\lambda^{m' + n'+1})^{2s}}{[(n' + m'+1)!]^{2s}}   \| \langle t \rangle \nabla  b_{m' , n'} \varphi^{n'+m'+1} q^{n'} \chi_{m' + n'+1} \|_{L^\infty}^2 \\  \label{62903:5}
\lesssim &   \sum_{n' = 0}^{\infty} \sum_{m' = 0}^{\infty} \frac{(\lambda^{m' + n'+1})^{2s}}{[(n' + m'+1)!]^{2s}}   \| \langle t \rangle \nabla  b_{m' , n'} \varphi^{n'+m'+1} q^{n'} \chi_{m' + n' +1} \|_{H^1_x H^1_y}^2 \\  \label{62903:6}
\lesssim & \langle t \rangle^2 \| \{ b_{m,n} \} \|_{X^{2,2}(\bold{a}_{m,n}, q^n \chi_{m+n})}^2.
\end{align}
Above, to go from line \eqref{62903:1} to \eqref{62903:2}, we have used the following bounds: 
\begin{align*}
 \mathbbm{1}_{(m', n') \in \mathbb{N}^{(hi)}_{m,n}} \frac{1}{[(n-n' + m-m' + 1)!]^{2s}} \lesssim &  \mathbbm{1}_{(m', n') \in \mathbb{N}^{(hi)}_{m,n}} \frac{1}{[(n-n' + m-m' + 2)!]^{2s}} , \\
 \varphi^{n-n'+m-m' +1}  \lesssim & \langle t \rangle \varphi^{n-n'+m-m' +2} 
\end{align*}

We next turn our attention to the quantity $|S^{(hi, 2)}_{m,n}|$. In this case, we observe that we can re-index 
\begin{align*}
\sum_{n = N}^\infty \sum_{m = N}^\infty \sum_{n' = n-N}^n \sum_{m' = m- N}^m = \sum_{m' = 0}^\infty \sum_{n' = 0}^\infty \sum_{n = n'}^{n' + N} \sum_{m = m'}^{m' + N}.
\end{align*}
Therefore, we can estimate 
\begin{align*}
 \sum_{n = 0}^{\infty} \sum_{m = 0}^{\infty} S^{(hi, 2)}_{m,n} \lesssim &   \sum_{n' = 0}^{\infty} \sum_{m' = 0}^{\infty} \Big( \sum_{n = n'}^{n' + N} \sum_{m = m'}^{m' + N} \mathbbm{1}_{(m', n') \in \mathbb{N}^{(hi)}_{m,n}} \Big) \frac{ \lambda^{2(m' + n' + 1)s} }{[(m' + n' + 1)!]^{2s}}  \| \sqrt{\nu} \nabla \omega_{m', n' }  \\
& \qquad \times e^W \chi_{n'+m' + 1}  \varphi^{n'+m'} q^{n'} \|_{L^2}^2 \\
\lesssim &   \sum_{n' = 0}^{\infty} \sum_{m' = 0}^{\infty}  \frac{ \lambda^{2(m' + n' + 1 )s} }{[(m' + n'  + 1)!]^{2s}}  \| \sqrt{\nu}\nabla \omega_{m', n' }   e^W \chi_{n'+m' + 1}  \varphi^{n'+m'} q^{n'} \|_{L^2}^2 \\ 
\lesssim &   \sum_{n' = 0}^{\infty} \sum_{m' = 0}^{\infty}  \frac{ \lambda^{2(m' + n'+1 )s} }{[(m' + n' +1)!]^{2s}}  \| \sqrt{\nu} \nabla \omega_{m' + 1, n' }   e^W \chi_{n'+m' + 1}  \varphi^{n'+m'} q^{n'}  \|_{L^2}^2 \\
= & \| \{ \sqrt{\nu} \nabla  \omega_{m,n} \} \|_{X^{0,2}(\bold{a}_{m,n}, e^W \chi_{m + n} q^n)}^2.
\end{align*}

\vspace{2 mm}

\noindent \underline{Bounding $S^{(lo)}$} In the case of $(m', n') \in \mathbb{N}_{m,n}^{(lo)}$, we have the inequalities
\begin{align*}
&m' < N, \qquad m- m' \ge m - N \ge 3N, \Rightarrow m' << m - m', \\
&n' < N, \qquad n- n' \ge n - N \ge 3N, \Rightarrow n' << n - n'. 
\end{align*}
Therefore, the following bounds are valid on our cutoff functions
\begin{align*}
\mathbbm{1}_{(m', n') \in \mathbb{N}^{(lo)}_{m,n}} \chi_{m + n}^2 \le \mathbbm{1}_{(m', n') \in \mathbb{N}^{(lo)}_{m,n}} \chi_{n' + m' + 2} \chi_{(n-n') + (m-m')}
\end{align*}

We have 
\begin{align*}
|S^{(lo)}| \lesssim & \langle t \rangle^3  \sum_{n = 0}^{\infty} \sum_{m = 0}^{\infty}  \Big(  \sum_{n' = 0}^{n-1} \sum_{m' = 0}^m \mathbbm{1}_{(m', n') \in \mathbb{N}^{(lo)}_{m,n}} \frac{(\lambda^{m-m' + n-n' + 1})^s}{[(n-n' + m-m' + 1)!]^s}\frac{ \lambda^{(m' + n' + 1)s} }{[(m' + n' + 1)!]^s} \\
& \qquad \qquad \qquad \times   \| \frac{1}{\nu} \nabla b_{m-m' , n-n'} q^{n-n'-1} \chi_{m-m' + n-n'} \|_{L^2} \\
&\qquad \qquad \qquad  \times \| \sqrt{\nu} \nabla \omega_{m', n' }  e^W \chi_{n'+m'+2} \varphi^{n'+m'+3} q^{n'+1}  \|_{L^\infty} \Big)^2 \\
\lesssim & \frac{\langle t \rangle^3}{\nu^2}  \sum_{n = 0}^{\infty} \sum_{m = 0}^{\infty} S^{(lo, 1)}_{m,n} S^{(lo, 2)}_{m,n},
\end{align*}
where we have defined the quantities 
\begin{align*}
S^{(lo, 1)}_{m,n} := & \sum_{n' = 0}^{n-1} \sum_{m' = 0}^m \mathbbm{1}_{(m', n') \in \mathbb{N}^{(lo)}_{m,n}} \frac{(\lambda^{m-m' + n-n' + 1})^{2s}}{[(n-n' + m-m' + 1)!]^{2s}}   \| \nabla b_{m-m' , n-n'} q^{n-n'-1} \chi_{m-m' + n-n'} \|_{L^2}^2, \\
S^{(lo, 2)}_{m,n} := & \sum_{n' = 0}^{n-1} \sum_{m' = 0}^m \mathbbm{1}_{(m', n') \in \mathbb{N}^{(lo)}_{m,n}} \frac{ \lambda^{(m' + n' +1)s} }{[(m' + n' +1)!]^s} \| \sqrt{\nu} \nabla \omega_{m', n' }  e^W \chi_{n'+m'+2} \varphi^{n'+m'+2} q^{n'+1}  \|_{L^\infty}^2.
\end{align*}
We will first provide bounds on $S^{(lo, 2)}_{m,n}$, as this is the simpler quantity. We have 
\begin{align*}
S^{(lo, 2)}_{m,n} \lesssim & \sum_{n' = 0}^{n-1} \sum_{m' = 0}^m \mathbbm{1}_{(m', n') \in \mathbb{N}^{(lo)}_{m,n}} \frac{ \lambda^{(m' + n' + 3 )s} }{[(m' + n' + 3 )!]^s} \| \sqrt{\nu} \nabla \omega_{m', n' }  e^W \chi_{n'+m'+3} \varphi^{n'+m'+3} q^{n'+1}  \|_{L^\infty}^2 \\
\lesssim &  \sum_{n' = 0}^{\infty} \sum_{m' = 0}^{\infty} \frac{ \lambda^{(m' + n' + 3 )s} }{[(m' + n' + 3 )!]^s} \| \sqrt{\nu} \nabla \omega_{m', n' }  e^W \chi_{n'+m'+3} \varphi^{n'+m'+3} q^{n'+1} \|_{H^1_xH^1_y}^2 \\
\lesssim &  \sum_{n' = 0}^{\infty} \sum_{m' = 0}^{\infty} \frac{ \lambda^{(m' + n'  )s} }{[(m' + n'  )!]^s} \| \sqrt{\nu}  \nabla \omega_{m', n' }  e^W \chi_{n'+m'} \varphi^{n'+m'} q^{n'}  \|_{L^2}^2 \\
\lesssim & \| \{ \sqrt{\nu} \nabla \omega_{m,n}  \} \|_{X^{0,2}(\bold{a}_{m,n}, q^n e^W \chi_{n+m})}^2.
\end{align*}
For the quantity $S^{(lo, 1)}_{m,n}$, we need to estimate separately the cases when $n' = n$ and $n' < n$.  
\begin{align*}
 \| \nabla b_{m-m' , n-n'} q^{n-n'-1} \chi_{m-m' + n-n'} \|_{L^2}^2 \mathbbm{1}_{n' < n} = &  \| \nabla \Gamma b_{m-m' , n-n'-1} q^{n-n'-1} \chi_{m-m' + n-n'} \|_{L^2}^2\mathbbm{1}_{n' < n} \\
 \lesssim & \langle t \rangle \| \nabla^2 b_{m-m'+1 , n-n'-1} q^{n-n'-1} \chi_{m-m' + n-n'} \|_{L^2}^2 \\
 \lesssim & \langle t \rangle \|   b_{m-m'+1 , n-n'-1} q^{n-n'-1} \chi_{m-m' + n-n'} \|_{H^2}^2. 
\end{align*}
We then perform a re-indexing trick to bound it as follows 
\begin{align*}
\sum_{m \ge N} \sum_{n \ge N} |S^{(lo, 1)}_{m,n}| \lesssim & \sum_{m} \sum_n \frac{\lambda^{2(m + n)s}}{((n + m)!)^{2s}} \|  \langle \p_x \rangle^{-1} b_{m,n} q^{n} \chi_{m +n} \|_{H^2}^2 \lesssim  \| \langle \p_x \rangle^{-1} \{b_{m,n} \} \|_{X^{1,2}(\bold{a}_{m,n}, q^n \chi_{m+n})}^2. 
\end{align*}

\vspace{2 mm}

\noindent \underline{Bounding $S^{(mid)}$} In the case of $S^{(mid})$, the main favorable properties of the localized bilinear operator will be the rapid decay of the multiplier as well as the separation of $m', m-m'$ from $m$ and similarly $n', n-n'$ from $n$, which allows for the following identities 
\begin{align*}
\chi_{m + n} = \chi_{m' + n' + 5} \chi_{(m-m') + (n-n') + 5} \chi_{m + n}.
\end{align*}
 We proceed to estimate this contribution as follows 
 \begin{align*}
|S^{(mid)}| \lesssim &  \sum_{n = 0}^{\infty} \sum_{m = 0}^{\infty}  \Big(  \sum_{n' = 0}^{n-1} \sum_{m' = 0}^m \mathbbm{1}_{(m', n') \in \mathbb{N}^{(mid)}_{m,n}} ( n + m)^{-C(N)} \frac{(\lambda^{m-m' + n-n'})^s}{[(n-n' + m-m')!]^s}\frac{ \lambda^{(m' + n' + 1)s} }{[(m' + n' + 1)!]^s} \\
& \qquad \qquad \qquad \times  \frac{1}{\sqrt{\nu}} \| \nabla b_{m-m' , n-n'} q^{n-n'} \chi_{m-m' + n-n' + 2} \|_{L^\infty} \\
&\qquad \qquad \qquad  \times \| \sqrt{\nu} \nabla \omega_{m', n' }  e^W \chi_{n'+m' + 1} \varphi^{n+m} q^{n'}  \|_{L^2} \Big)^2 \\
\lesssim &  \sum_{n = 0}^{\infty} \sum_{m = 0}^{\infty}  \Big(  \sum_{n' = 0}^{n-1} \sum_{m' = 0}^m \mathbbm{1}_{(m', n') \in \mathbb{N}^{(mid)}_{m,n}} ( n + m)^{-\frac{C(N)}{2}} \frac{(\lambda^{m-m' + n-n'})^s}{[(n-n' + m-m' + 5)!]^s}\frac{ \lambda^{(m' + n' + 1)s} }{[(m' + n' + 5)!]^s} \\
& \qquad \qquad \qquad \times  \frac{1}{\sqrt{\nu}} \| \nabla b_{m-m' , n-n'} q^{n-n'} \chi_{m-m' + n-n' + 5} \|_{L^\infty} \\
&\qquad \qquad \qquad  \times \| \sqrt{\nu} \nabla \omega_{m', n' }  e^W \chi_{n'+m' + 5} \varphi^{n+m} q^{n'}  \|_{L^2} \Big)^2,
\end{align*}
and from here the bound follows in essentially the identical fashion to $S^{(hi)}$, the only exception being that the decaying factor of $(n + m)^{-\frac{C(N)}{2}}$ is used to make the summation in $n, m$ finite. A very similar argument applies also to the case when either $m \le 4N$ (which implies that $n \ge 12N$ or $n \le 4N$ (which implies that $m \ge 12N$). 
\end{proof}

\section{Trilinear Bounds}  \label{sec:Tri}
In this section, our aim is to use the abstract bilinear bounds obtained in the previous section in order to control the various trilinear contributions arising on the right-hand sides. Ultimately, this will furnish a proof of Proposition \ref{pro:tri:in}. 

\subsection{$(int, \gamma)$ Inner Products}

We will need to consider the following inner product: 
\begin{align} \n
\mathcal{I}^{(I, \gamma)}(t) := &\sum_{m = 0}^{\infty} \sum_{n = 0}^\infty \bold{a}_{m,n}^2 \langle \bold{T}^{(I, \gamma)}_{m,n} , \omega_{m,n}  q^{2n} e^{2W} \chi_{n + m}^2 \rangle \\ \label{defn:IP:1}
=&  \sum_{i = 1}^3 \sum_{m = 0}^{\infty} \sum_{n = 0}^\infty \bold{a}_{m,n}^2 \langle  \bold{T}^{(I, \gamma, i)}_{m,n} , \omega_{m,n}  q^{2n} e^{2W} \chi_{n + m}^2 \rangle =  \sum_{i = 1}^3 \mathcal{I}^{(I, \gamma, i)}(t).
\end{align}
The main proposition is to obtain estimates on the inner products appearing above. 
\begin{proposition}[$\gamma$ Inner Product Bounds] The inner products defined in \eqref{defn:IP:1} satisfy the following bounds: 
\begin{align} \label{jb:43:a}
| \mathcal{I}^{(I, \gamma, 1)}(t)| \lesssim & \sqrt{\slashed{\mathcal{E}}_{ell}^{(I)}} \mathcal{CK}^{(\gamma, \varphi)}\\ \label{jb:43:b}
| \mathcal{I}^{(I, \gamma, 2)}(t)| \lesssim &  \sqrt{\slashed{\mathcal{E}}_{ell}^{(I)}} \mathcal{CK}^{(\gamma; \lambda)} \\ \label{jb:43:c}
| \mathcal{I}^{(I, \gamma, 3)}(t)| \lesssim & \sqrt{\slashed{\mathcal{E}}_{ell}^{(I)}}( \mathcal{D}^{(\gamma)} +\mathcal{CK}^{(\gamma; W)} + \mathcal{CK}^{(\gamma; \varphi)} + \mathcal{CK}_{\text{Cloud}}) .
\end{align}
\end{proposition}
\begin{proof}[Proof of \eqref{jb:43:a}]We estimate this contribution to the inner-product using Cauchy-Schwartz as follows 
\begin{align} \label{aug5:1}
| \mathcal{I}^{(I, \gamma, 1)}| \lesssim & \Big[ \sum_{m = 0}^{\infty} \sum_{n = 0}^\infty \bold{a}_{m,n}^2 \|  \bold{T}^{(I, \gamma, 1)}_{m,n} q^n e^W \chi_{n+m} \|_{L^2}^2 \Big]^{\frac12} \Big[ \sum_{m = 0}^{\infty} \sum_{n = 0}^\infty \bold{a}_{m,n}^2 \| \omega_{m,n}  q^{n} e^{W} \chi_{n + m}\|_{L^2}^2 \Big]^{\frac12} \\  \label{aug5:2}
\lesssim &  \| \mathcal{R}_{m,n}[ \phi^{(I)}_{\neq 0, m, n}, \omega_{m,n} ]  \|_{X^0_2(\bold{a}_{m,n} q^n e^W \chi_{m + n} )} \sqrt{\mathcal{CK}^{(\gamma; \varphi)}} \langle t \rangle^{\frac12}  \\ \n
\lesssim &  ( \langle t \rangle \| \{  \phi^{(I)}_{\neq 0, m, n}  \} \|_{X^{0}_\infty(\mathbbm{1}_{m + N < 17N})} + \langle t \rangle^2 \| \{ \phi^{(I)}_{\neq 0, m, n} \} \|_{X^{0}_\infty(\chi_2 \mathbbm{1}_{m + N < 17N} )}) \\ \label{snow:p:1}
& \qquad \| \{\omega_{m,n} \} \|_{X^{0}_2(\mathbbm{1}_{m + N < 16N} \bold{a}_{m,n} q^n e^W \chi_{n + m})}\sqrt{ \mathcal{CK}^{(\gamma; \varphi)}} \langle t \rangle^{\frac12}  \\ \label{aug5:3}
\lesssim &  ( (\langle t \rangle^2 \mathcal{E}_{ell}^{(I, \text{Full})})^{\frac12} +  (\langle t \rangle^4 \mathcal{E}_{ell}^{(I, \text{Out})})^{\frac12}) \sqrt{\mathcal{CK}^{(\gamma; \varphi)}} \sqrt{ \mathcal{CK}^{(\gamma; \varphi)}} \langle t \rangle \\  \label{aug5:4}
\lesssim & \sqrt{\slashed{\mathcal{E}}_{ell}^{(I)}}  \mathcal{CK}^{(\gamma, \varphi)}.
\end{align}
Above to go from \eqref{aug5:1} to \eqref{aug5:2}, we have used the definitions \eqref{aug5:a} and \eqref{CKgammavarphi}. To go from \eqref{aug5:2} to \eqref{snow:p:1}, we have used the bilinear estimate \eqref{snow:patrol:1}. To go from \eqref{snow:p:1} to \eqref{aug5:3}, we use Sobolev embedding with the vector-fields $\p_x, \Gamma$, \eqref{Sob:emb}. To go from \eqref{aug5:3} to \eqref{aug5:4}, we use our elliptic bounds. 
\end{proof}
\begin{proof}[Proof of \eqref{jb:43:b}]We estimate this contribution to the inner-product using Cauchy-Schwartz, but after redistributing a weight of $(n + m)^{\frac12}$, as follows: 
\begin{align} \n
| \mathcal{I}^{(I, \gamma, 2)}| \lesssim & \Big[ \sum_{m = 0}^{\infty} \sum_{n = 0}^\infty \frac{\bold{a}_{n,m}^2}{n + m}  \| \bold{T}^{(I, \gamma, 2)}_{m,n} q^n e^W \chi_{n+m} \|_{L^2}^2 \Big]^{\frac12}  \Big[ \sum_{m = 0}^{\infty} \sum_{n = 0}^\infty (n+m)\bold{a}_{n,m}^2  \| \omega_{m,n}  q^{n} e^{W} \chi_{n + m}\|_{L^2}^2 \Big]^{\frac12} \\ \label{some:nights:1}
\lesssim & \| \{ \mathcal{Q}_{m,n}[\{ \phi^{(I)}_{\neq 0, m,n}   \}, \{ \omega_{m,n} \}] \} \|_{X^{0,2}(\frac{\bold{a}_{m,n}}{(m + n)^{\frac12}} q^n e^W \chi_{n+m} )} \sqrt{ \mathcal{CK}^{(\gamma; \lambda)}(t)} \langle t \rangle \\ \label{some:nights:2}
\lesssim &  \| \{ \p_y \phi^{(I)}_{\neq 0, m,n} \} \|_{X^{0,\infty}(\bold{a}_{m,n}  \chi_{2} \varphi^{-4}  )} \| \{ \omega_{m,n} \} \|_{X^{0,2}((n + m)^{\frac12}\bold{a}_{m,n} q^n e^W \chi_{n+m} )} \sqrt{\mathcal{CK}^{(\gamma; \lambda)}} \langle t \rangle\\ \label{some:nights:3}
\lesssim &  \| \{ \p_y \phi^{(I)}_{\neq 0, m,n} \} \|_{X^{0,\infty}(\bold{a}_{m,n}  \chi_{2} \varphi^{-4}  )}  \mathcal{CK}^{(\gamma; \lambda)} \langle t \rangle^2\\ \label{some:nights:4}
\lesssim &(\langle t \rangle^4 \mathcal{E}_{ell}^{(I, out)})^{\frac12} \mathcal{CK}^{(\gamma; \lambda)} \\
\lesssim &\sqrt{\slashed{\mathcal{E}}_{ell}^{(I)}} \mathcal{CK}^{(\gamma; \lambda)}.
\end{align}
Above, to obtain \eqref{some:nights:1}, we use the bound \eqref{lambda:yessir}. To obtain \eqref{some:nights:2}, we use the bilinear estimate \eqref{ewing:1}. 
\end{proof}
\begin{proof}[Proof of \eqref{jb:43:c}] We split the estimate of this contribution into two cases: $n + m \ge 10$ and $n + m < 10$: 
\begin{align}
I^{(I, \gamma, 3)} = I^{(I, \gamma, 3, \ge 10)} + I^{(I, \gamma, 3, < 10)}.
\end{align}

\vspace{1 mm}

\noindent \textit{Case $n + m \ge 10$:} In this case, we want to use the $CK^{(W)}$ term to create a diffusive factor of $\sqrt{\nu}$ to absorb the loss of derivative. We proceed as follows
\begin{align} \n
| \mathcal{I}^{(I, \gamma, 3, \ge 10)}| \lesssim & \Big[ \sum_{m = 0}^{\infty} \sum_{n = 0}^\infty \mathbbm{1}_{m + n \ge 10} \bold{a}_{n,m}^2  \|  \nabla \phi^{(I)}_{\neq} \cdot \nabla \omega_{m,n} q^n e^W \chi_{n+m} \|_{L^2}^2 \Big]^{\frac12} \\ \label{shuy:0}
& \times \Big[ \sum_{m = 0}^{\infty} \sum_{n = 0}^\infty \mathbbm{1}_{m + n \ge 10} \bold{a}_{n,m}^2 \| \omega_{m,n}  q^{n} e^{W} \chi_{n + m}\|_{L^2}^2 \Big]^{\frac12} \\ \n
\lesssim &  \Big[ \sum_{m = 0}^{\infty} \sum_{n = 0}^\infty \bold{a}_{n,m}^2 \mathbbm{1}_{m + n \ge 10}  \|  \nabla \phi^{(I)}_{\neq} \chi_2 \langle t \rangle  \|_{L^\infty} \| \sqrt{\nu} \nabla \omega_{m,n} q^n e^W \chi_{n+m} \|_{L^2}^2 \Big]^{\frac12} \\ \label{shuy:1}
& \times \Big[ \sum_{m = 0}^{\infty} \sum_{n = 0}^\infty \mathbbm{1}_{m + n \ge 10} \bold{a}_{n,m}^2 \|  \frac{d(y)}{\sqrt{\nu} \langle t \rangle} \omega_{m,n}  q^{n} e^{W} \chi_{n + m}\|_{L^2}^2 \Big]^{\frac12} \\ \label{avicii:1}
\lesssim & \sqrt{\langle t \rangle^4 \mathcal{E}_{ell}^{(I, out)}} \sqrt{\mathcal{D}^{(\gamma)}} \sqrt{ \mathcal{CK}^{(\gamma, W)}(t)} \\
\lesssim & \sqrt{\slashed{\mathcal{E}}_{ell}^{(I)}} \sqrt{\mathcal{D}^{(\gamma)}} \sqrt{ \mathcal{CK}^{(\gamma, W)}(t)}.
\end{align} 
Above, to go from \eqref{shuy:0} to \eqref{shuy:1}, we have used that $m + n \ge 10$ to write the identity $\chi_{m + n} = \chi_{m + n} \chi_2$ as well as to insert a factor of $d(y)$ for free, as $1/d(y)$ is bounded on the support of $\chi_{m + n}$. 

\vspace{2 mm}

\noindent \textit{Case $n + m < 10$:} In this case, we cannot invoke the factor of $d(y)$ to generate an extra $\sqrt{\nu}$. Therefore, we need to absorb the loss of derivative (which is finite) into the Cloud norm. First of all, we decompose $\chi_{m + n} = \chi_{m + n + 2} + (\chi_{m + n} - \chi_{m + n+2})$, the latter of which is contained in $\chi_I$. The $\chi_{m + n + 2}$ contribution is estimated identically to the $n + m \ge 10$ case. We proceed to estimate the ``interior" contribution as follows 
\begin{align} \n
| \mathcal{I}^{(I, \gamma, 3, < 10, int)}| \lesssim & \Big[ \sum_{m = 0}^{\infty} \sum_{n = 0}^\infty \mathbbm{1}_{m + n < 10} \bold{a}_{n,m}^2 \| \slashed{\nabla} \phi^{(I)}_{\neq} \cdot \slashed{\nabla} \omega_{m,n} q^n e^W \chi_{I} \|_{L^2}^2 \Big]^{\frac12} \\ \label{shuy:0a}
& \times \Big[ \sum_{m = 0}^{\infty} \sum_{n = 0}^\infty \mathbbm{1}_{m + n < 10} \bold{a}_{n,m}^2 \| \omega_{m,n}  q^{n} e^{W} \chi_{n + m}\|_{L^2}^2 \Big]^{\frac12} \\ \n
\lesssim &  \Big[ \sum_{m = 0}^{\infty} \sum_{n = 0}^\infty \bold{a}_{n+1,m+1}^2 \mathbbm{1}_{m + n < 10}  \|  \slashed{\nabla} \phi^{(I)}_{\neq}  \varphi^{-1} \|_{L^\infty}^2 \|  \slashed{\nabla} \omega_{m,n} q^n e^W \chi_{I} \|_{L^2}^2 \Big]^{\frac12}\mathcal{CK}^{(\gamma)}(t) \\ \label{shuy:1a}\\ 
\lesssim & \langle t \rangle^2 \sqrt{\mathcal{E}_{ell}^{(I, full)}} \sqrt{\mathcal{CK}_{\text{Cloud}}(t)} \sqrt{ \mathcal{CK}^{(\gamma; \varphi)}(t)} \\
\lesssim & \sqrt{\slashed{\mathcal{E}}_{ell}^{(I)}} (\mathcal{CK}_{\text{Cloud}}(t) +  \mathcal{CK}^{(\gamma; \varphi)}(t)), 
\end{align}
which proves the lemma. 
\end{proof}

\subsection{$(int, \alpha/ \mu)$ Inner Products} \label{jaylen:1}

For this section, we want to control the inner products that arise from the decomposition \eqref{harden:1}. We define 
\begin{align} \n
\mathcal{I}^{(I, \xi)}(t) = & \sum_{m = 0}^{\infty} \sum_{n = 0}^\infty \bold{a}_{n,m}^2 \nu \langle \bold{T}^{(I, \xi)}_{m,n} , \p_{x_\xi} \omega_{m,n}  q^{2n} e^{2W} \chi_{n + m}^2 \rangle  \\  \label{rodeo:alpha}
= & \sum_{i = 1}^{6} \sum_{m = 0}^{\infty} \sum_{n = 0}^\infty \bold{a}_{n,m}^2 \nu \langle \bold{T}^{(I, \xi, i)}_{m,n} , \p_{x_\xi} \omega_{m,n}  q^{2n} e^{2W} \chi_{n + m}^2 \rangle  =  \sum_{i = 1}^{6} \mathcal{I}^{(I, \xi, i)}(t).
\end{align}
The main proposition will be the following. 
\begin{proposition} The inner products defined above in \eqref{rodeo:alpha} satisfy the following estimates: 
\begin{align} \label{rodeo:1}
|\mathcal{I}^{(I, \xi, 1)}(t)| \lesssim &  \sqrt{\slashed{\mathcal{E}}_{ell}^{(I)}} (  \nu^{\frac16}\mathcal{CK}^{(\gamma; \varphi)} +  \mathcal{CK}^{(\xi; \varphi)} )  \\ \label{rodeo:2}
|\mathcal{I}^{(I, \xi, 2)}(t)| \lesssim & \sqrt{\slashed{\mathcal{E}}_{ell}^{(I)}}  \mathcal{CK}^{(\xi; \varphi)}(t) \\ \label{rodeo:3}
|\mathcal{I}^{(I, \xi, 3)}(t)| \lesssim &\sqrt{\slashed{\mathcal{E}}_{ell}^{(I)}}(\mathcal{CK}^{(\alpha; \lambda)}(t) + \mathcal{CK}^{(\mu; \lambda)}(t))\\ \label{rodeo:4}
|\mathcal{I}^{(I, \xi, 4)}(t)| \lesssim & \sqrt{\slashed{\mathcal{E}}_{ell}^{(I)}}\mathcal{CK}^{(\xi; \lambda)}(t) \\ \label{rodeo:5}
|\mathcal{I}^{(I, \xi, 5)}(t)| \lesssim &  \sqrt{\slashed{\mathcal{E}}_{ell}^{(I)} } (  \frac{\mathcal{E}^{(\alpha)}(t)}{\langle t \rangle^{90}} +  \frac{\mathcal{E}^{(\mu)}(t)}{\langle t \rangle^{90}} + \nu^{\frac13} \mathcal{CK}_{\text{Cloud}}^{(\varphi)} +  \mathcal{CK}^{(\alpha; \varphi)} ) \\ \label{rodeo:6}
|\mathcal{I}^{(I, \xi, 6)}(t)| \lesssim & \sqrt{\slashed{\mathcal{E}}_{ell}^{(I)} } (\mathcal{CK}_{\text{Cloud}} + \mathcal{CK}^{(\alpha; \varphi)}(t) ).
\end{align}
\end{proposition}
\begin{proof}[Proof of \eqref{rodeo:1}] Due to the definition \eqref{bilinear:R}, we have that $m + n$ is bounded by $16N$. We give the proof of the case $\xi = \alpha$, with the case $\xi = \mu$ actually being simpler. By definition \eqref{rodeo:alpha}, we use Cauchy-Schwartz to control the inner-product as follows 
\begin{align} \n
|\mathcal{I}^{(I, \alpha, 1)}(t)| \lesssim & \Big[ \sum_{m = 0}^\infty \sum_{n = 0}^\infty \bold{a}_{m,n}^2  \nu \|  \bold{T}^{(I,\alpha, 1)} e^W \chi_{n + m} q^n \|_{L^2}^2  \Big]^{\frac12}  \Big[ \sum_{m = 0}^\infty \sum_{n = 0}^\infty \bold{a}_{m,n}^2  \nu  \| \p_y \omega_{m,n} q^n e^W \chi_{m + n} \|_{L^2}^2  \Big]^{\frac12} \\ \n
\lesssim & \langle t \rangle^{\frac12} \nu^{\frac12} \| \{  \mathcal{R}_{m,n}[\{\p_y \phi^{(I)}_{\neq, m,n}\}, \{\omega_{m,n}\}]  \|_{{X^{0}_2(\bold{a}_{m,n} q^n e^W \chi_{n + m} )}} \\
& \times \Big[ \sum_{m = 0}^\infty \sum_{n = 0}^\infty \bold{a}_{m,n}^2 \frac{\dot{\varphi}}{\varphi}  \nu  \| \p_y \omega_{m,n} q^n e^W \chi_{m + n} \|_{L^2}^2  \Big]^{\frac12} \\ \n
\lesssim &\langle t \rangle^{\frac12} \nu^{\frac12} \| \{  \mathcal{R}_{m,n}[\{\p_y \phi^{(I)}_{\neq, m,n}\}, \{\omega_{m,n}\}]  \|_{{X^{0}_2(\bold{a}_{m,n} q^n e^W \chi_{n + m} )}} \sqrt{ \mathcal{CK}^{(\alpha; \varphi)}(t)} \\ \n
\lesssim &\langle t \rangle^{\frac12} \nu^{\frac12} \Big[ \langle t \rangle \| \{  \p_y \phi^{(I)}_{m, n} \} \|_{X^{0}_\infty(\mathbbm{1}_{m + N < 17N})}  \| \{ \omega_{m,n} \} \|_{X^{0}_2(\mathbbm{1}_{m + n < 16N} \bold{a}_{m,n} q^n e^W \chi_{n + m})} \\ \label{SM:1}
& + \| \{  \p_y \phi^{(I)}_{m, n} \} \|_{X^{0}_\infty(\chi_2 \mathbbm{1}_{m + n < 17N})}  \| \{\p_y  \omega_{m,n} \} \|_{X^{0}_2(\mathbbm{1}_{m + N < 16N} \bold{a}_{m,n} q^n e^W \chi_{n + m})} \Big] \\ \n
& \times \sqrt{ \mathcal{CK}^{(\alpha; \varphi)}(t)}=: |\mathcal{I}^{(I, \alpha, 1, 1)}(t)| + |\mathcal{I}^{(I, \alpha, 1, 2)}(t)|.  
\end{align}
To obtain the bound \eqref{SM:1}, we have invoked our bilinear estimate \eqref{office:1}. 

To estimate $|\mathcal{I}^{(I, \alpha, 1, 1)}(t)|$, we will need the following estimate: 
\begin{align}
 & \| \{\nu^{\frac13}\p_y \phi^{(I)}_{\neq 0, m, n} \} \|_{X^{0}_\infty(\mathbbm{1}_{m + n < 17N})} \\
 \lesssim &  \| \{ \nu^{\frac13} (\Gamma - t\p_x) \phi^{(I)}_{\neq 0, m, n} \} \|_{X^{0}_\infty(\mathbbm{1}_{m + n < 17N})} \\
 \lesssim & \| \{ \nu^{\frac13}t \phi^{(I)}_{\neq 0, m, n} \} \|_{X^{0}_\infty(\mathbbm{1}_{m + n < 18N})} \lesssim \| \{e^{-\delta_I  \nu^{\frac13}t} e^{\delta_I \nu^{\frac13}t} \nu^{\frac13}t \phi^{(I)}_{\neq 0, m, n} \} \|_{X^{0}_\infty(\mathbbm{1}_{m + n < 18N})} \\
 \lesssim &  \| \{ e^{\delta_I \nu^{\frac13}t}  \phi^{(I)}_{\neq 0, m, n} \} \|_{X^{0}_\infty(\mathbbm{1}_{m + n < 18N})} \lesssim \Big(\mathcal{E}_{ell}^{(I, full)}\Big)^{\frac12}.
\end{align}
In the final step, to absorb the extra $\nu^{\frac13}t$ weight, we have crucially used the fact that $\delta_I > 0$. Correspondingly, we have 
\begin{align*}
 |\mathcal{I}^{(I, \alpha, 1, 1)}(t)| \lesssim & \nu^{\frac16} \Big[   \langle t \rangle^4 \mathcal{E}_{ell}^{(I, full)} \Big]^{\frac12} \sqrt{\frac{\dot{\varphi}}{\varphi}} \| \{ \omega_{m,n} \} \|_{X^{0}_2(\mathbbm{1}_{m + N < 16N} \bold{a}_{m,n} q^n e^W \chi_{n + m})}  \sqrt{ \mathcal{CK}^{(\alpha; \varphi)}(t)} \\
\lesssim & \nu^{\frac16}  \Big(\langle t \rangle^4 \mathcal{E}_{ell}^{(I, full)} \Big)^{\frac12} \sqrt{\mathcal{CK}^{(\gamma; \varphi)}} \sqrt{ \mathcal{CK}^{(\alpha; \varphi)}(t)} \\
\lesssim & \nu^{\frac16} \sqrt{\slashed{\mathcal{E}}_{ell}^{(I)}} (\mathcal{CK}^{(\gamma; \varphi)} + \mathcal{CK}^{(\alpha; \varphi)}(t)). 
\end{align*}
To estimate the term $ |\mathcal{I}^{(I, \alpha, 1, 2)}(t)|$, we distribute the $\nu^{1/2}$ differently. Indeed we have upon writing $\p_y = \frac{1}{v_y} (\Gamma + t\p_x)$, 
\begin{align*}
 &|\mathcal{I}^{(I, \alpha, 1, 2)}(t)| \\
 \lesssim & \langle t \rangle^{\frac32} \| \phi^{(I)}_{\neq 0, m, n} \|_{X^{0}_\infty(\chi_2 \mathbbm{1}_{m + n < 17N})}  \| \{\nu^{\frac12}\p_y  \omega_{m,n} \} \|_{X^{0}_2(\mathbbm{1}_{m + N < 16N} \bold{a}_{m,n} q^n e^W \chi_{n + m})}\sqrt{ \mathcal{CK}^{(\alpha; \varphi)}(t)} \\
 \lesssim & \langle t \rangle^{2} \| \phi^{(I)}_{\neq 0, m, n} \|_{X^{0}_\infty(\chi_2 \mathbbm{1}_{m + n < 17N})}  \| \{\sqrt{\frac{\dot{\varphi}}{\varphi}}\nu^{\frac12}\p_y  \omega_{m,n} \} \|_{X^{0}_2(\mathbbm{1}_{m + N < 16N} \bold{a}_{m,n} q^n e^W \chi_{n + m})}\sqrt{ \mathcal{CK}^{(\alpha; \varphi)}(t)} \\
 \lesssim &  \sqrt{\slashed{\mathcal{E}}_{ell}^{(I)}} \mathcal{CK}^{(\alpha; \varphi)}(t).
\end{align*}
\end{proof}
\begin{proof}[Proof of \eqref{rodeo:2}]We again use Cauchy-Schwartz to control the inner-product as follows 
\begin{align} \n
|\mathcal{I}^{(I, \alpha, 2)}(t)| \lesssim & \Big[ \sum_{m = 0}^\infty \sum_{n = 0}^\infty \bold{a}_{m,n}^2 \nu \|  \bold{T}^{(I,\alpha, 2)} e^W \chi_{n + m} q^n \|_{L^2}^2  \Big]^{\frac12} \Big[ \sum_{m = 0}^\infty \sum_{n = 0}^\infty \bold{a}_{m,n}^2 \nu  \| \p_y \omega_{m,n} q^n e^W \chi_{m + n}  \|_{L^2}^2  \Big]^{\frac12} \\ \n
\lesssim & \langle t \rangle^{\frac12} \Big[ \sum_{m = 0}^\infty \sum_{n = 0}^\infty \bold{a}_{m,n}^2 \nu \|  \bold{T}^{(I,\alpha, 2)} e^W \chi_{n + m} q^n \|_{L^2}^2  \Big]^{\frac12} \Big[ \sum_{m = 0}^\infty \sum_{n = 0}^\infty \bold{a}_{m,n}^2 \frac{\dot{\varphi}}{\varphi} \nu  \| \p_y \omega_{m,n} q^n e^W \chi_{m + n}  \|_{L^2}^2  \Big]^{\frac12} \\ \n
\lesssim &\langle t \rangle^{\frac12}\Big[ \sum_{m = 0}^\infty \sum_{n = 0}^\infty \bold{a}_{m,n}^2 \nu \|  \mathcal{R}_{m,n}[ \{ \phi^{(I)}_{\neq 0, m, n} \}, \{ \p_y \omega_{m,n}\}  ] e^W \chi_{n + m} q^n \|_{L^2}^2  \Big]^{\frac12} \sqrt{\mathcal{CK}^{(\alpha; \varphi)}(t)} \\ \n
\lesssim &\langle t \rangle^{\frac12} \| \{   \mathcal{R}_{m,n}[\{\phi^{(I)}_{\neq 0, m,n} \}, \{\sqrt{\nu} \p_y \omega_{m,n} \}]  \|_{{X^{0}_2(\bold{a}_{m,n} q^n e^W \chi_{n + m} )}}\sqrt{\mathcal{CK}^{(\alpha; \varphi)}(t)} \\ \n
\lesssim &\langle t \rangle^{\frac12} ( \langle t \rangle \| \{  \phi^{(I)}_{\neq 0, m,n} \} \|_{X^{0}_\infty(\mathbbm{1}_{m + N < 17N})} +\langle t \rangle^2 \| \{ \phi^{(I)}_{\neq 0, m,n} \} \|_{X^{0}_\infty(\chi_2 \mathbbm{1}_{m + N < 17N} )}) \\ \label{snow::1878k}
& \qquad \| \{\sqrt{\nu} \p_y \omega_{m,n} \} \|_{X^{0}_2(\mathbbm{1}_{m + N < 16N} \bold{a}_{m,n} q^n e^W \chi_{n + m})} \sqrt{ \mathcal{CK}^{(\alpha; \varphi)}(t)}  \\ \n
\lesssim &\langle t \rangle ( \langle t \rangle \| \{  \phi^{(I)}_{\neq 0, m,n} \} \|_{X^{0}_\infty(\mathbbm{1}_{m + N < 17N})} +\langle t \rangle^2 \| \{ \phi^{(I)}_{\neq 0, m,n} \} \|_{X^{0}_\infty(\chi_2 \mathbbm{1}_{m + N < 17N} )}) \mathcal{CK}^{(\alpha; \varphi)}(t)  \\  \label{snow::1878k}
\lesssim & \sqrt{\slashed{\mathcal{E}}_{ell}^{(I)}}  \mathcal{CK}^{(\alpha; \varphi)}(t).
\end{align}
\end{proof}
\begin{proof}[Proof of \eqref{rodeo:3}] In this case, we will need to multiply and divide by the weight $(m +n)^{\frac12}$ prior to applying Cauchy-Schwartz, in the following manner: 
\begin{align*}
|\mathcal{I}^{(I, \alpha, 3)}(t)| \lesssim & \Big[ \sum_{m = 0}^\infty \sum_{n = 0}^\infty \frac{\bold{a}_{m,n}^2}{m + n} \nu  \| \bold{T}^{(I,\alpha, 3)} e^W \chi_{n + m} q^n \|_{L^2}^2  \Big]^{\frac12} \\
& \times \Big[ \sum_{m = 0}^\infty \sum_{n = 0}^\infty (m+n) \bold{a}_{m,n}^2  \nu  \| \p_y \omega_{m,n} q^n e^W \chi_{m + n} \|_{L^2}^2  \Big]^{\frac12}  \\
\lesssim & \Big[ \sum_{m = 0}^\infty \sum_{n = 0}^\infty \frac{\bold{a}_{m,n}^2}{m + n} \nu  \| \mathcal{Q}_{m,n}[ \{ \p_y \phi^{(I)}_{\neq 0, m, n} \}, \{ \omega_{m,n} \} ] e^W \chi_{n + m} q^n \|_{L^2}^2  \Big]^{\frac12} \\
& \times \Big[ \sum_{m = 0}^\infty \sum_{n = 0}^\infty (m+n) \bold{a}_{m,n}^2 \nu \| \p_y \omega_{m,n} q^n e^W \chi_{m + n} \|_{L^2}^2  \Big]^{\frac12} \\
\lesssim &\nu^{\frac12} \langle t \rangle \Big[ \sum_{m = 0}^\infty \sum_{n = 0}^\infty \frac{\bold{a}_{m,n}^2}{m + n}   \|  \mathcal{Q}_{m,n}[ \{\p_y \phi^{(I)}_{\neq 0, m, n} \}, \{ \omega_{m,n} \} ] e^W \chi_{n + m} q^n \|_{L^2_x L^2_y}^2  \Big]^{\frac12} \\
& \times \Big[ \sum_{m = 0}^\infty \sum_{n = 0}^\infty (m+n) \bold{a}_{m,n}^2 \frac{\dot{\lambda}}{\lambda}  \| \nu^{\frac12}\p_y \omega_{m,n} q^n e^W \chi_{m + n} \|_{L^2}^2  \Big]^{\frac12} \\
\lesssim & \langle t \rangle\| \{ \p_y \phi^{(I)}_{\neq 0, m,n} \} \|_{X^{0}_\infty(\bold{a}_{m,n}  \chi_{2}  \varphi^{-4} )} \| \{ \nu^{\frac12}\nabla \omega_{m,n}\} \|_{X^{0}_2((n + m)^{\frac12}\bold{a}_{m,n} q^n e^W \chi_{n+m} )} \sqrt{\mathcal{CK}^{(\alpha; \lambda)}(t)} \\
\lesssim &\sqrt{\slashed{\mathcal{E}}_{ell}^{(I)}}(\mathcal{CK}^{(\alpha; \lambda)}(t) + \mathcal{CK}^{(\mu; \lambda)}(t)).
\end{align*}
The lemma is proven. 
\end{proof}
\begin{proof}[Proof of \eqref{rodeo:4}] We follow similarly to the previous calculation, as follows: 
\begin{align*}
&|\mathcal{I}^{(I, \alpha, 4)}(t)| \\
\lesssim & \Big[ \sum_{m = 0}^\infty \sum_{n = 0}^\infty \frac{\bold{a}_{m,n}^2}{m + n} \nu \| \bold{T}^{(I,\alpha, 4)} e^W \chi_{n + m} q^n \|_{L^2}^2  \Big]^{\frac12}  \Big[ \sum_{m = 0}^\infty \sum_{n = 0}^\infty (m+n) \bold{a}_{m,n}^2 \nu   \| \p_y \omega_{m,n} q^n e^W \chi_{m + n} \|_{L^2}^2  \Big]^{\frac12}  \\
\lesssim & \Big[ \sum_{m = 0}^\infty \sum_{n = 0}^\infty \frac{\bold{a}_{m,n}^2}{m + n}  \| \mathcal{Q}_{m,n}[ \{ \phi^{(I)}_{\neq 0, m, n} \}, \{\nu^{\frac13} \p_y \omega_{m,n}  \} ] e^W \chi_{n + m} q^n \|_{L^2}^2  \Big]^{\frac12}\sqrt{ \mathcal{CK}^{(\alpha; \lambda)}(t)} \\
\lesssim &\| \{ \mathcal{Q}_{m,n}[\{\phi^{(I)}_{\neq 0, m,n} \}, \{ \nu^{\frac12} \p_y \omega_{m,n}   \}] \} \|_{X^{0}_2(\frac{\bold{a}_{m,n}}{(m + n)^{\frac12}}, q^n e^W \chi_{n+m} )}\sqrt{ \mathcal{CK}^{(\alpha; \lambda)}(t)} \\
\lesssim &  \| \{   \p_y \phi^{(I)}_{\neq 0, m,n}  \} \|_{X^{0}_{\infty}(\bold{a}_{m,n},  \chi_{2} \varphi^{-4})} \| \{ \nu^{\frac12} \p_y \omega_{m,n}  \} \|_{X^{0}_2((n + m)^{\frac12}\bold{a}_{m,n},  q^n e^W \chi_{n+m} )}\sqrt{ \mathcal{CK}^{(\alpha; \lambda)}(t)} \\
\lesssim & \sqrt{\slashed{\mathcal{E}}_{ell}^{(I)}}\mathcal{CK}^{(\alpha; \lambda)}(t). 
\end{align*}
\end{proof}
\begin{proof}[Proof of \eqref{rodeo:5}] This argument is significantly more involved. We first begin with the following estimate. 
\begin{claim} The quantity $\bold{T}^{(I, \alpha, 5)}_{m,n}$, defined in \eqref{harden:1}, satisfies the following bounds:
\begin{align} \label{morgan:1}
\| \bold{T}^{(I, \alpha, 5)}_{m,n} \|_{X_\alpha}^2 \lesssim &\frac{1}{\langle t \rangle^{90}} \slashed{\mathcal{E}}^{(I)}_{ell}(t) (  \mathcal{E}^{(\alpha)}(t) +  \mathcal{E}^{(\mu)}(t) ) +  \nu^{\frac13} \frac{1}{\langle t \rangle} \slashed{\mathcal{E}}_{ell}^{(I)} \mathcal{CK}_{\text{Cloud}}^{(\varphi)}. 
\end{align}
\end{claim} 
\begin{proof}[Proof of Claim] To establish this bound, we proceed as follows. 
We first treat the case when $m + n \ge 10$. In this case, $\bold{b} = \nabla^\perp \phi^{(I)}_{\neq 0}$ decays rapidly, and so does $\p_y \bold{b}$. Therefore, we may estimate as follows 
\begin{align*}
\|  \bold{T}^{(I, \alpha, 5)}_{m,n} \mathbbm{1}_{m + n \ge 10} \|_{X_{\alpha}}^2 \lesssim & \sum_{n + m \ge 10} \nu \bold{a}_{n,m}^2  \|  q^n \p_y \bold{b} \cdot \nabla \omega_{m,n} e^W \chi_{n + m} \|_{L^2_x L^2_y}^2 \\
\lesssim & \sum_{n + m \ge 10} \nu \bold{a}_{n,m}^2 \| \p_y \bold{b} \chi_2  \|_{L^\infty}^2 \|  q^n \nabla \omega_{m,n} e^W \chi_{n + m} \|_{L^2_x L^2_y}^2 \\
 \lesssim & \frac{1}{\langle t \rangle^{90}} \mathcal{E}^{(I, out)}_{ell}(t)  \sum_{n + m \ge 10} \nu \bold{a}_{n,m}^2  \| q^n \nabla \omega_{m,n} e^W \chi_{n + m} \|_{L^2}^2 \\
\lesssim & \frac{1}{\langle t \rangle^{90}} \mathcal{E}^{(I, out)}_{ell}(t) (  \mathcal{E}^{(\alpha)}(t) +  \mathcal{E}^{(\mu)}(t) ) \\
\lesssim & \frac{1}{\langle t \rangle^{90}} \slashed{\mathcal{E}}^{(I)}_{ell}(t) (  \mathcal{E}^{(\alpha)}(t) +  \mathcal{E}^{(\mu)}(t) ).
\end{align*} 

We next treat the case when $m + n \le 10$. This case is more delicate due to slow decay. First, we localize further as follows 
\begin{align*}
\|  \bold{T}^{(I, \alpha, 5)}_{m,n} \mathbbm{1}_{m + n < 10} \|_{X_{\alpha}}^2 \lesssim & \sum_{n + m < 10} \nu \bold{a}_{n,m}^2  \| q^n \p_y \bold{b} \cdot \nabla \omega_{m,n} e^W \chi_{n + m} (1 - \chi_{10}) \|_{L^2_x L^2_y}^2 \\
&+ \sum_{n + m < 10} \nu^{\frac23} \bold{a}_{n,m}^2  \| q^n \p_y \bold{b} \cdot \nabla \omega_{m,n} e^W \chi_{n + m} \chi_{10} \|_{L^2_x L^2_y}^2 \\
=: &J_{low}^{(I)} + J_{low}^{(E)}.
\end{align*}
The bound on $J_{low}^{(E)}$ is essentially identical to the case when $m + n \ge 10$. We therefore treat $J_{low}^{(I)}$. 
\begin{align*}
|J_{low}^{(I)}| \lesssim &  \sum_{n + m < 10}  \bold{a}_{n+1,m}^2 \| v_y \|_{L^\infty}^2 \|  \nu^{\frac12} \p_y \p_x \phi^{(I)}_{\neq 0} \varphi^{-1} \|_{L^\infty}^2 \| q^{n+1}  e^W \omega_{m,n+1} \chi_{n + m + 1} (1 - \chi_{10})\|_{L^2}^2 \\
& +  \sum_{n + m < 10}  \bold{a}_{n+1,m}^2 \| v_y \|_{L^\infty}^2 \| \nu^{\frac12} \p_y \p_x \phi^{(I)}_{\neq 0} \varphi^{-1} \|_{L^\infty}^2 \| q^n  e^W  \omega_{m,n+1} \{ \chi_{n + m} - \chi_{n + m + 1} \} (1 - \chi_{10})\|_{L^2}^2 \\
= & |J_{low}^{(I, 1)}|  + |J_{low}^{(I, 2)}|.
\end{align*}
%
Above, we have used that $1/q$ is bounded uniformly on the support of $1 - \chi_{10}$ to upgrade $q^n$ to $q^{n+1}$ for the term $J_{low}^{(I, 1)}$. Moreover, we have used the bound $\bold{a}_{n,m} \lesssim \varphi^{-1} \bold{a}_{n+1, m}$. We proceed to estimate 
\begin{align*}
 |J_{low}^{(I, 1)}| \lesssim & \sum_{n + m < 10} \frac{\dot{\varphi}}{\varphi} \bold{a}_{n+1,m}^2 \| v_y \|_{L^\infty}^2 \|  \nu^{\frac12} \p_y \p_x \phi^{(I)}_{\neq 0} \langle t \rangle^{\frac32} \|_{L^\infty}^2 \| q^{n+1}  e^W \omega_{m,n+1} \chi_{n + m + 1} (1 - \chi_{10})\|_{L^2}^2 \\
 \lesssim & \nu^{\frac13} \frac{1}{\langle t \rangle} \slashed{\mathcal{E}}_{ell}^{(I)} \mathcal{CK}^{(\gamma; \varphi)}.
\end{align*}
Similarly, 
\begin{align*}
 |J_{low}^{(I, 2)}| \lesssim & \sum_{n + m < 10} \frac{\dot{\varphi}}{\varphi} \bold{a}_{n+1,m}^2 \| v_y \|_{L^\infty}^2 \|  \nu^{\frac12} \p_y \p_x \phi^{(I)}_{\neq 0} \langle t \rangle^{\frac32} \|_{L^\infty}^2 \| q^{n+1}  e^W \omega_{m,n+1} \chi_I\|_{L^2}^2 \\
 \lesssim & \nu^{\frac13} \frac{1}{\langle t \rangle} \slashed{\mathcal{E}}_{ell}^{(I)} \mathcal{CK}_{\text{Cloud}}^{(\varphi)}.
\end{align*}
The lemma is proven. 
\end{proof}

To now estimate this inner product, we proceed directly as follows
\begin{align*}
|\mathcal{I}^{(I, \alpha, 5)}(t)| \le & \sum_{m = 0}^{\infty} \sum_{n = 0}^\infty\sum_{k \neq 0} \bold{a}_{n,m}^2 \nu |\langle \bold{T}^{(I, \alpha, 5)}_{m,n} ,\p_y \omega_{m,n}  q^{2n} e^{2W} \chi_{n + m}^2 \rangle| \\
\lesssim & \Big[ \sum_{m = 0}^{\infty} \sum_{n = 0}^\infty\bold{a}_{n,m}^2 \nu \|\bold{T}^{(I, \alpha, 5)}_{m,n}  e^W \chi_{n + m} q^n\|_{L^2}^2  \Big]^{\frac12} \Big[ \sum_{m = 0}^{\infty} \sum_{n = 0}^\infty\bold{a}_{n,m}^2 \nu \| \p_y \omega_{m,n,k} e^W \chi_{n + m} q^n \|_{L^2}^2 \Big]^{\frac12} \\
\lesssim & \Big[ \frac{1}{\langle t \rangle^{90}} \slashed{\mathcal{E}}^{(I)}_{ell}(t) (  \mathcal{E}^{(\alpha)}(t) +  \mathcal{E}^{(\mu)}(t) ) +  \nu^{\frac13} \frac{1}{\langle t \rangle} \slashed{\mathcal{E}}_{ell}^{(I)} \mathcal{CK}_{\text{Cloud}}^{(\varphi)} \Big]^{\frac12} \langle t \rangle^{\frac12} \sqrt{\mathcal{CK}^{(\alpha; \varphi)}} \\
\lesssim &  \sqrt{\slashed{\mathcal{E}}_{ell}^{(I)} } (  \frac{\mathcal{E}^{(\alpha)}(t)}{\langle t \rangle^{90}} +  \frac{\mathcal{E}^{(\mu)}(t)}{\langle t \rangle^{90}} + \nu^{\frac13} \mathcal{CK}_{\text{Cloud}}^{(\varphi)} +  \mathcal{CK}^{(\alpha; \varphi)} ).
\end{align*}
The lemma is proven. 
\end{proof}
\begin{proof}[Proof of \eqref{rodeo:6}] We estimate the inner product by splitting into two cases as follows 
\begin{align} 
\mathcal{I}^{(I, \alpha, 6)} = & \sum_{m = 0}^{\infty} \sum_{n = 0}^\infty \bold{a}_{n,m}^2 \nu \langle \bold{T}^{(I, \alpha, 6)}_{m,n} , \p_y \omega_{m,n}  q^{2n} e^{2W} \chi_{n + m}^2 \rangle \\ \n
= & \sum_{m = 0}^{\infty} \sum_{n = 0}^\infty  \mathbbm{1}_{m + n \ge 10} \bold{a}_{n,m}^2 \nu \langle \bold{T}^{(I, \alpha, 6)}_{m,n} , \p_y \omega_{m,n}  q^{2n} e^{2W} \chi_{n + m}^2 \rangle\\
& + \sum_{m = 0}^{\infty} \sum_{n = 0}^\infty \mathbbm{1}_{m + n < 10} \bold{a}_{n,m}^2  \nu \langle \bold{T}^{(I, \alpha, 6)}_{m,n} , \p_y \omega_{m,n}  q^{2n} e^{2W} \chi_{n + m}^2 \rangle\\
= & \mathcal{I}^{(I, \alpha, 6, \ge 10)} + \mathcal{I}^{(I, \alpha, 6, < 10)}.
\end{align}
For the $\ge 10$ case, we want to invoke the $\mathcal{CK}^{(\alpha, W)}$ term to create a diffusive factor. We proceed as follows 
\begin{align} 
&|\mathcal{I}_{\neq}^{(I, \alpha, 6, \ge 10)}| \\ \n
 \lesssim & \Big[ \sum_{m = 0}^{\infty} \sum_{n = 0}^\infty \mathbbm{1}_{m + n \ge 10} \bold{a}_{m,n}^2 \|  \sqrt{\nu} \langle t \rangle \nabla^\perp \phi^{(I)}_{\neq 0} \cdot \nabla (\nu^{\frac12}\p_y \omega_{m,n}) e^W q^n  \chi_{m+n} \|_{L^2}^2 \Big]^{\frac12} \\
& \times \Big[\sum_{m = 0}^{\infty} \sum_{n = 0}^\infty \mathbbm{1}_{m + n \ge 10} \bold{a}_{m,n}^2 \| \frac{d(y)}{\sqrt{\nu}\langle t \rangle} \nu^{\frac12} \p_y \omega_{m,n}^{(\neq)} e^W q^n  \chi_{m+n} \|_{L^2}^2  \Big]^{\frac12} \\ \n
\lesssim & \Big[ \sum_{m = 0}^{\infty} \sum_{n = 0}^\infty \mathbbm{1}_{m + n \ge 10} \bold{a}_{m,n}^2 \|    \langle t \rangle \nabla^\perp \phi^{(I)}_{\neq 0} \chi_2 \|_{L^\infty}^2 \| \sqrt{\nu} \nabla (\nu^{\frac12}\p_y \omega_{m,n}) e^W q^n  \chi_{m+n} \|_{L^2_xL^2_y}^2 \Big]^{\frac12} \\
& \times \Big[\sum_{m = 0}^{\infty} \sum_{n = 0}^\infty \mathbbm{1}_{m + n \ge 10} \bold{a}_{m,n}^2 \| \frac{d(y)}{\sqrt{\nu}\langle t \rangle} \nu^{\frac12} \p_y \omega_{m,n} e^W q^n  \chi_{m+n} \|_{L^2}^2  \Big]^{\frac12} \\
\lesssim & \langle t \rangle\sqrt{ \mathcal{E}_{ell}^{(I, out)}}\sqrt{ \mathcal{D}^{(\alpha)}(t)} \sqrt{\mathcal{CK}^{(\alpha, W)}(t)} \\
\lesssim &   \sqrt{\slashed{\mathcal{E}}_{ell}^{(I)} } ( \mathcal{D}^{(\alpha)}(t)+ \mathcal{CK}^{(\alpha, W)}(t)). 
\end{align}
For the $\le 10$ case, we need to use the cloud norm to absorb the loss of derivative. We rewrite the product as follows
\begin{align} \n
\nabla^\perp \phi^{(I)}_{\neq} \cdot \nabla (\nu^{\frac12} \p_y \omega_{m,n}) = & v_y \slashed{\nabla}^\perp \phi^{(I)}_{\neq} \cdot \slashed{\nabla} ( \nu^{\frac12}\p_y \omega_{m,n} ) =  v_y \slashed{\nabla}^\perp \phi^{(I)}_{\neq} \cdot \slashed{\nabla} ( \nu^{\frac12} \omega_{m,n+1} - \nu^{\frac12}t \omega_{m+1,n} ) \\ \n
= & \nu^{\frac12} v_y \Gamma \phi^{(I)}_{\neq} \omega_{m+1, n+1} -  \nu^{\frac12} v_y \p_x \phi^{(I)}_{\neq} \omega_{m, n+2} \\ \label{four:terms}
& - v_y \nu^{\frac12}t \Gamma \phi^{(I)}_{\neq} \omega_{m+2, n} + v_y \nu^{\frac12}t \p_x \phi^{(I)}_{\neq} \omega_{m+1, n+1}.
\end{align}
Inserting this identity, we then have 
\begin{align} 
&|\mathcal{I}^{(I, \alpha, 6, < 10)}| \\ \n
 \lesssim & \Big[ \sum_{m = 0}^{\infty} \sum_{n = 0}^\infty \mathbbm{1}_{m + n < 10} \bold{a}_{m,n}^2 \|  \slashed{\nabla}^\perp \phi^{(I)}_{\neq} \cdot \slashed{\nabla} ( \nu^{\frac12} \omega_{m,n+1} - \nu^{\frac12}t \omega_{m+1,n} ) e^W q^n  \chi_{m+n} \|_{L^2}^2 \Big]^{\frac12} \\
& \times \Big[\sum_{m = 0}^{\infty} \sum_{n = 0}^\infty \mathbbm{1}_{m + n < 10} \bold{a}_{m,n}^2 \|  \nu^{\frac12} \p_y \omega_{m,n} e^W q^n  \chi_{m+n} \|_{L^2}^2  \Big]^{\frac12} \\ \n 
 \lesssim & \Big[ \sum_{m = 0}^{\infty} \sum_{n = 0}^\infty \mathbbm{1}_{m + n < 10} \bold{a}_{m,n}^2 \|  \slashed{\nabla}^\perp \phi^{(I)}_{\neq} \cdot \slashed{\nabla} ( \nu^{\frac12} \omega_{m,n+1} - \nu^{\frac12}t \omega_{m+1,n} ) e^W q^n  \chi_{m+n+2} \|_{L^2_xL^2_y}^2 \Big]^{\frac12} \\
& \times \Big[\sum_{m = 0}^{\infty} \sum_{n = 0}^\infty \mathbbm{1}_{m + n < 10} \bold{a}_{m,n}^2 \|  \nu^{\frac12} \p_y \omega_{m,n} e^W q^n  \chi_{m+n} \|_{L^2}^2  \Big]^{\frac12} \\ \n
+ &  \Big[ \sum_{m = 0}^{\infty} \sum_{n = 0}^\infty \mathbbm{1}_{m + n < 10} \bold{a}_{m,n}^2 \|  \slashed{\nabla}^\perp \phi^{(I)}_{\neq} \cdot \slashed{\nabla} ( \nu^{\frac12} \omega_{m,n+1} - \nu^{\frac12}t \omega_{m+1,n} ) e^W q^n  \chi_I \|_{L^2_xL^2_y}^2 \Big]^{\frac12} \\
& \times \Big[\sum_{m = 0}^{\infty} \sum_{n = 0}^\infty \mathbbm{1}_{m + n < 10} \bold{a}_{m,n}^2 \|  \nu^{\frac12} \p_y \omega_{m,n}e^W q^n  \chi_{m+n} \|_{L^2}^2  \Big]^{\frac12} \\
=:& |\mathcal{I}^{(I, \alpha, 6, < 10, int)}| + |\mathcal{I}^{(I, \alpha, 6, < 10, ext)}|.
\end{align}
The contribution $|\mathcal{I}^{(I, \alpha, 6, < 10, ext)}|$ can be controlled in an identical manner to the $\ge 10$ case. We further split
\begin{align}
\mathcal{I}^{(I, \alpha, 6, < 10, int)} = \sum_{i = 1}^4 |\mathcal{I}^{(I, \alpha, 6, < 10, int, i)}|
\end{align}
corresponding to the four terms in \eqref{four:terms}. 

We estimate 
\begin{align} \n
| \mathcal{I}_{\neq}^{(I, \alpha, 6, < 10, int, 1)} | \lesssim &  \Big[ \sum_{m = 0}^{\infty} \sum_{n = 0}^\infty \mathbbm{1}_{m + n < 10} \bold{a}_{m,n}^2 \|  \Gamma \phi^{(I)}_{\neq}  \nu^{\frac12} \omega_{m+1,n+1} e^W q^n \chi_{I} \|_{L^2_xL^2_y}^2 \Big]^{\frac12} \\
& \times \Big[\sum_{m = 0}^{\infty} \sum_{n = 0}^\infty \mathbbm{1}_{m + n < 10} \bold{a}_{m,n}^2 \|  \nu^{\frac12} \p_y \omega_{m,n} e^W q^n \chi_{m+n} \|_{L^2}^2  \Big]^{\frac12} \\ \n
\lesssim & \nu^{\frac12} \langle t \rangle^{2}  \Big[ \sum_{m = 0}^{\infty} \sum_{n = 0}^\infty \mathbbm{1}_{m + n < 10} \frac{\dot{\varphi}}{\varphi} \bold{a}_{m+1,n+1}^2 \|  \Gamma \phi^{(I)}_{\neq}   \omega_{m+1,n+1} e^W q^n \chi_{I} \|_{L^2_xL^2_y}^2 \Big]^{\frac12} \\
& \times \Big[\sum_{m = 0}^{\infty} \sum_{n = 0}^\infty \mathbbm{1}_{m + n < 10} \frac{\dot{\varphi}}{\varphi} \bold{a}_{m,n}^2 \|  \nu^{\frac12} \p_y \omega_{m,n} e^W q^n \chi_{m+n} \|_{L^2}^2  \Big]^{\frac12} \\ \n
 \lesssim & \nu^{\frac12} \langle t \rangle^{2} \|  \Gamma \phi^{(I)}_{\neq}   \|_{L^\infty} \sqrt{\mathcal{CK}_{\text{Cloud}}}\sqrt{\mathcal{CK}^{(\alpha; \varphi)}(t)} \\
\lesssim & \nu^{\frac12} \sqrt{\slashed{\mathcal{E}}_{ell}^{(I)} } (\mathcal{CK}_{\text{Cloud}} + \mathcal{CK}^{(\alpha; \varphi)}(t) ).
\end{align}
Next, we estimate 
\begin{align} \n
| \mathcal{I}_{\neq}^{(I, \alpha, 6, < 10, int, 2)} | \lesssim &  \Big[ \sum_{m = 0}^{\infty} \sum_{n = 0}^\infty \mathbbm{1}_{m + n < 10} \bold{a}_{m,n}^2 \|  \Gamma \phi^{(I)}_{\neq}  \nu^{\frac12} \omega_{m,n+2} e^W q^n  \chi_{I} \|_{L^2_xL^2_y}^2 \Big]^{\frac12} \\
& \times \Big[\sum_{m = 0}^{\infty} \sum_{n = 0}^\infty \mathbbm{1}_{m + n < 10} \bold{a}_{m,n}^2 \|  \nu^{\frac12} \p_y \omega_{m,n} e^W q^n \chi_{m+n} \|_{L^2}^2  \Big]^{\frac12} \\ \n
 \lesssim &  \Big[ \sum_{m = 0}^{\infty} \sum_{n = 0}^\infty \mathbbm{1}_{m + n < 10} \bold{a}_{m,n+2}^2 \frac{\dot{\varphi}}{\varphi} \|  \Gamma \phi^{(I)}_{\neq}  \nu^{\frac12} \varphi^2  \|_{L^\infty}^2  \| \omega_{m,n+2} e^W q^{n+2}  \chi_{I} \|_{L^2_xL^2_y}^2 \Big]^{\frac12} \\
& \times \sqrt{\mathcal{CK}^{(\alpha; \varphi)}(t)} \langle t \rangle \\
\lesssim & (\nu^{\frac12} \langle t \rangle) \sqrt{\slashed{\mathcal{E}}_{ell}^{(I)} } (\mathcal{CK}_{\text{Cloud}} + \mathcal{CK}^{(\alpha; \varphi)}(t) ) \\
\lesssim &  \sqrt{\slashed{\mathcal{E}}_{ell}^{(I)} } (\mathcal{CK}_{\text{Cloud}} + \mathcal{CK}^{(\alpha; \varphi)}(t) ). 
\end{align}
Third, we have 
\begin{align} \n
| \mathcal{I}_{\neq}^{(I, \alpha, 6, < 10, int, 3)} | \lesssim &  \Big[ \sum_{m = 0}^{\infty} \sum_{n = 0}^\infty \mathbbm{1}_{m + n < 10} \bold{a}_{m,n}^2 \|  \Gamma \phi^{(I)}_{\neq}  \nu^{\frac12} t \omega_{m+2,n} e^W q^n  \chi_{I} \|_{L^2_xL^2_y}^2 \Big]^{\frac12} \\
& \times \Big[\sum_{m = 0}^{\infty} \sum_{n = 0}^\infty \mathbbm{1}_{m + n < 10} \bold{a}_{m,n}^2 \|  \nu^{\frac12} \p_y \omega_{m,n} e^W q^n  \chi_{m+n} \|_{L^2}^2  \Big]^{\frac12} \\ \n
 \lesssim & \nu^{\frac12} \langle t \rangle^2 \Big[ \sum_{m = 0}^{\infty} \sum_{n = 0}^\infty \mathbbm{1}_{m + n < 10} \bold{a}_{m+2,n}^2 \frac{\dot{\varphi}}{\varphi} \|  \Gamma \phi^{(I)}_{\neq}  \omega_{m+2,n} e^W q^n  \chi_{I} \|_{L^2_xL^2_y}^2 \Big]^{\frac12} \\
& \times \Big[\sum_{m = 0}^{\infty} \sum_{n = 0}^\infty \mathbbm{1}_{m + n < 10} \bold{a}_{m,n}^2 \frac{\dot{\varphi}}{\varphi} \|  \nu^{\frac12} \p_y \omega_{m,n} e^W q^n  \chi_{m+n} \|_{L^2}^2  \Big]^{\frac12} \\
\lesssim &  \nu^{\frac12} \sqrt{\slashed{\mathcal{E}}_{ell}^{(I)} } (\mathcal{CK}_{\text{Cloud}} + \mathcal{CK}^{(\alpha; \varphi)}(t) ).
\end{align}
Finally, we have
\begin{align} \n
| \mathcal{I}_{\neq}^{(I, \alpha, 6, < 10, int, 4)} | \lesssim &  \Big[ \sum_{m = 0}^{\infty} \sum_{n = 0}^\infty \mathbbm{1}_{m + n < 10} \bold{a}_{m,n}^2 \|  \Gamma \phi^{(I)}_{\neq}  \nu^{\frac12}t \omega_{m+1,n+1} e^W q^n  \chi_{I} \|_{L^2_xL^2_y}^2 \Big]^{\frac12} \\
& \times \Big[\sum_{m = 0}^{\infty} \sum_{n = 0}^\infty \mathbbm{1}_{m + n < 10} \bold{a}_{m,n}^2 \|  \nu^{\frac12} \p_y \omega_{m,n} e^W q^n \chi_{m+n} \|_{L^2}^2  \Big]^{\frac12} \\ \n
 \lesssim & \sup_t (\nu^{\frac12}t) \langle t \rangle^2 \Big[ \sum_{m = 0}^{\infty} \sum_{n = 0}^\infty \mathbbm{1}_{m + n < 10} \bold{a}_{m+1,n+1}^2 \frac{\dot{\varphi}}{\varphi} \|  \Gamma \phi^{(I)}_{\neq}   \omega_{m+1,n+1} e^W q^n  \chi_{I} \|_{L^2_xL^2_y}^2 \Big]^{\frac12} \\
& \times \Big[\sum_{m = 0}^{\infty} \sum_{n = 0}^\infty \mathbbm{1}_{m + n < 10} \bold{a}_{m,n}^2 \frac{\dot{\varphi}}{\varphi} \|  \nu^{\frac12} \p_y \omega_{m,n} e^W q^n \chi_{m+n} \|_{L^2}^2  \Big]^{\frac12} \\
\lesssim & \sqrt{\slashed{\mathcal{E}}_{ell}^{(I)} } (\mathcal{CK}_{\text{Cloud}} + \mathcal{CK}^{(\alpha; \varphi)}(t) ).
\end{align}
The lemma is proven. 
\end{proof}

\subsection{$(ext, \gamma)$ Inner Products}

We define the inner-products
\begin{align} \label{def:ip:hgy}
I^{(E, \gamma, i)}(t) := \sum_{m = 0}^\infty \sum_{n = 0}^\infty \bold{a}_{m,n}^2 \langle \bold{T}^{(E, \gamma, i)}_{m,n}, \omega^{(\neq)}_{m,n} \chi_{m + n}^2 e^{2W} q^{2n} \rangle, \qquad i = 1, 2, 3. 
\end{align}

\begin{lemma} The following bounds are valid on the inner products defined in \eqref{def:ip:hgy}:
\begin{align} \label{hold:ya:1}
|I^{(E, \gamma, 1)}(t)| + |I^{(E, \gamma, 2)}(t)| \lesssim \frac{\langle t \rangle^2}{\nu^2} \sqrt{\mathcal{J}_{ell}^{(2)}(t)} \mathcal{D}^{(\gamma)}(t).
\end{align}
\end{lemma}
\begin{proof} The bounds for $I^{(E, \gamma, 1)}(t)$ and $I^{(E, \gamma, 2)}(t)$ follow in an analogous manner. First, we notice that $m + n \ge 1$ due to the definition of $\bold{T}^{(E, \gamma, 1)}_{m,n}$. Therefore, we have the identity 
\begin{align*}
\chi_{m + n}^2 = \chi_{m + n}^2 \chi_{m + n -1}. 
\end{align*}
Inserting this identity, we have the following bound 
\begin{align} \label{my:sbux:1}
&|I^{(E, \gamma, 1)}(t)| \\ \label{my:sbux:2}
\le & \frac{1}{\sqrt{\nu}} \Big[\sum_{m = 0}^\infty \sum_{n = 0}^\infty \bold{a}_{m,n}^2 \|  \bold{T}^{(E, \gamma, 1)}_{m,n} \chi_{m + n}^2 e^W q^n  \|_{L^2}^2 \Big]^{\frac12} \Big[\sum_{m = 0}^\infty \sum_{n = 0}^\infty \bold{a}_{m,n}^2  \| \sqrt{\nu} \omega_{m,n} \chi_{m + n - 1} e^{W} q^{n}\|_{L^2}^2 \Big]^{\frac12} \\ \n
\lesssim  & \frac{1}{\sqrt{\nu}} \Big[\sum_{m = 0}^\infty \sum_{n = 0}^\infty \bold{a}_{m,n}^2 \| \bold{T}^{(E, \gamma, 1)}_{m,n} \chi_{m + n}^2 e^W q^n  \|_{L^2}^2 \Big]^{\frac12} \Big[\sum_{m = 0}^\infty \sum_{n = 0}^\infty \mathbbm{1}_{m + n \ge 1} \mathbbm{1}_{m = 0} \bold{a}_{m,n}^2 \\ \n
& \times \| \sqrt{\nu} \omega_{m,n} \chi_{m + n - 1} e^{W} q^{n} \|_{L^2}^2 \Big]^{\frac12} \\ \n
 + & \frac{1}{\sqrt{\nu}} \Big[\sum_{m = 0}^\infty \sum_{n = 0}^\infty \bold{a}_{m,n}^2 \| \bold{T}^{(E, \gamma, 1)}_{m,n} \chi_{m + n}^2 e^W q^n \|_{L^2}^2 \Big]^{\frac12} \Big[\sum_{m = 0}^\infty \sum_{n = 0}^\infty \mathbbm{1}_{m + n \ge 1} \mathbbm{1}_{m > 0} \bold{a}_{m,n}^2 \\ \label{my:sbux:3}
 & \times  \| \sqrt{\nu} \omega_{m,n} \chi_{m + n - 1} e^{W} q^{n}  \|_{L^2}^2 \Big]^{\frac12} \\ \n
\lesssim  & \frac{1}{\sqrt{\nu}} \Big[\sum_{m = 0}^\infty \sum_{n = 0}^\infty \bold{a}_{m,n}^2 \|  \bold{T}^{(E, \gamma, 1)}_{m,n} \chi_{m + n}^2 e^W q^n \|_{L^2}^2 \Big]^{\frac12} \Big[\sum_{m = 0}^\infty \sum_{n = 0}^\infty \mathbbm{1}_{m + n \ge 1} \mathbbm{1}_{m = 0} \frac{\bold{a}_{m,n}^2}{\bold{a}_{m,n-1}^2} \bold{a}_{m,n-1}^2   \\ \n
& \times \| \sqrt{\nu} (\p_y + t \p_x) \omega_{m,n-1} \chi_{m + n - 1} e^{W} q^{n}\|_{L^2}^2 \Big]^{\frac12} \\ \n
+ & \frac{1}{\sqrt{\nu}} \Big[\sum_{m = 0}^\infty \sum_{n = 0}^\infty \bold{a}_{m,n}^2 \| \bold{T}^{(E, \gamma, 1)}_{m,n} \chi_{m + n}^2 e^W q^n \|_{L^2}^2 \Big]^{\frac12} \Big[\sum_{m = 0}^\infty \sum_{n = 0}^\infty \mathbbm{1}_{m + n \ge 1} \mathbbm{1}_{m > 0} \frac{\bold{a}_{m,n}^2}{\bold{a}_{m-1,n}^2} \bold{a}_{m-1,n}^2   \\ \label{my:sbux:4}
& \times \| \sqrt{\nu} \p_x \omega_{m-1,n} \chi_{m + n - 1} e^{W} q^{n} \|_{L^2}^2 \Big]^{\frac12} \\ \label{my:sbux:5}
\lesssim & \frac{1}{\sqrt{\nu}}\Big[\sum_{m = 0}^\infty \sum_{n = 0}^\infty (n + m)^{-2s} \bold{a}_{m,n}^2 \|\bold{T}^{(E, \gamma, 1)}_{m,n} \chi_{m + n}^2 e^W q^n \|_{L^2}^2 \Big]^{\frac12} \sqrt{\mathcal{D}^{(\gamma)}(t)} \\ \label{my:sbux:6}
\lesssim & \frac{1}{\sqrt{\nu}} \| \{ \bold{T}^{(E, \gamma, 1)}_{m,n} \}\|_{X^0_2(\frac{\bold{a}_{m,n}}{(n + m)^s} q^n e^W \chi_{n + m}^2)}\sqrt{\mathcal{D}^{(\gamma)}(t)}.
\end{align}
Above, to go from \eqref{my:sbux:4} to \eqref{my:sbux:5}, we use the bounds 
\begin{align*}
\frac{\bold{a}_{m,n}}{\bold{a}_{m-1,n}} \lesssim (m + n)^{-s}, \qquad \frac{\bold{a}_{m,n}}{\bold{a}_{m,n-1}} \lesssim (m + n)^{-s} \varphi(t).
\end{align*}
We now recognize 
\begin{align*}
&\| \{ \bold{T}^{(E, \gamma, 1)}_{m,n} \}\|_{X^0_2(\frac{\bold{a}_{m,n}}{(n + m)^s} q^n e^W \chi_{n + m}^2)} =  \| \{ \mathcal{Q}_{m,n}[\{ \phi^{(E)}_{\neq, m,n}  \}, \{ \omega_{m,n} \}] \} \|_{X^{0,2}(\frac{\bold{a}_{m,n}}{(m + n)^{s},}  q^{n} e^W \chi_{n+m}^2 )} \\ 
 \lesssim & \frac{\langle t \rangle^2}{\nu^2} \| \{ \phi^{(E)}_{\neq 0, m,n}  \} \|_{X^{2}_2(\bold{a}_{m,n}, q^n \chi_{m+n})} \| \{ \sqrt{\nu} \nabla \omega_{m,n} \} \|_{X^{0}_2(\bold{a}_{m,n}, e^W \chi_{m + n} q^n)},
\end{align*}
where we have invoked the bound \eqref{ewing:2}. 

To conclude the bound, we invoke our elliptic bootstraps to estimate 
\begin{align*}
|I^{(E, \gamma, 1)}(t)| \lesssim & \frac{\langle t \rangle^2}{\nu^2} \| \{ \phi^{(E)}_{\neq, m,n}  \} \|_{X^{2}_2(\bold{a}_{m,n}, q^n \chi_{m+n})} \| \{ \sqrt{\nu} \nabla \omega_{m,n} \} \|_{X^{0}_2(\bold{a}_{m,n}, e^W \chi_{m + n} q^n)} \sqrt{\mathcal{D}^{(\gamma)}(t)} \\
\lesssim &\frac{\langle t \rangle^2}{\nu^2} \sqrt{\mathcal{J}_{ell}^{(2)}(t)}\mathcal{D}^{(\gamma)}(t).
\end{align*}
To estimate $|I^{(E, \gamma, 2)}(t)|$, we simply need to invoke the bound \eqref{63023:1} instead of \eqref{ewing:2}. 
\end{proof}

\begin{lemma}The following bound is valid on the inner product $I^{(E, \gamma, 3)}(t)$ defined in \eqref{def:ip:hgy}:
\begin{align} \label{hold:ya:2}
|I^{(E, \gamma, 3)}(t) | \lesssim  \frac{1}{\nu}\sqrt{ \mathcal{J}^{(2)}_{ell}(t)} ( \mathcal{D}^{(\gamma)}(t) + \mathcal{E}^{(\gamma)}(t)).
\end{align}
\end{lemma}
\begin{proof} We estimate directly 
\begin{align*}
|I^{(E, \gamma, 3)}(t) | \lesssim & \sum_{m = 0}^\infty \sum_{n = 0}^\infty \bold{a}_{m,n}^2  \| \p_x \phi^{(E)} \|_{L^2_x L^\infty_y} \| \nabla \omega_{m,n} \chi_{m + n} e^W q^n \|_{L^2} \| \omega_{m,n} e^W q^n \chi_{m+n} \|_{L^\infty_x L^2_y} \\
& + \| \p_y \phi^{(E)} \|_{L^\infty_x L^2_y} \|  \nabla \omega_{m,n} \chi_{m + n} e^W q^n \|_{L^2} \| \omega_{m,n} e^W q^n \chi_{m+n} \|_{L^2_x L^\infty_y} \\
\lesssim & \sum_{m = 0}^\infty \sum_{n = 0}^\infty \bold{a}_{m,n}^2 \frac{1}{\nu} \| \phi^{(E)} \|_{H^2} \mathcal{D}^{(\gamma)}_{m,n}(t) (\mathcal{E}_{m,n}^{(\gamma)})^{\frac12} (\mathcal{D}_{m,n}^{(\gamma)})^{\frac12} \\
\lesssim &  \sum_{m = 0}^\infty \sum_{n = 0}^\infty \bold{a}_{m,n}^2 \frac{1}{\nu} \sqrt{\mathcal{J}^{(2)}_{ell}(t)} \sqrt{\mathcal{D}^{(\gamma)}_{m,n}(t)} (\mathcal{E}_{m,n}^{(\gamma)})^{\frac14} (\mathcal{D}_{m,n}^{(\gamma)})^{\frac14} \\
\lesssim &  \frac{1}{\nu} \sqrt{ \mathcal{J}_{ell}^{(2)} } ( \mathcal{E}^{(\gamma)} + \mathcal{D}^{(\gamma)} ). 
\end{align*}
The lemma is proven. 
\end{proof}

\subsection{$(ext, \alpha/ \mu)$ Inner Products}

We recall the definition \eqref{harden:1}. Correspondingly, we define the inner products (for $\xi \in \{\alpha, \mu\}$)
\begin{align}
I^{(E, \xi, i)}(t) := \sum_{m = 0}^\infty \sum_{n = 0}^\infty \bold{a}_{m,n}^2  \nu \langle \bold{T}^{(E, \xi, i)}_{m,n}, \p_{x_\xi} \omega_{m,n} \chi_{m + n}^2 e^{2W} q^{2n} \rangle, \qquad 1 \le i \le 6. 
\end{align}

\begin{lemma} The following bounds are valid: 
\begin{align} \label{063023:1}
|I^{(E, \xi, 1)}(t)| \lesssim & \frac{\langle t \rangle^2}{\nu^2} \sqrt{\mathcal{J}_{ell}^{(3)}}  ( (\mathcal{E}^{(\alpha)})^{\frac12} + (\mathcal{E}^{(\mu)})^{\frac12}) ( \mathcal{D}^{(\xi)})^{\frac12}, \\ \label{063023:2}
|I^{(E, \xi, 2)}(t)| \lesssim &  \frac{\langle t \rangle^2}{\nu^2} \sqrt{\mathcal{J}_{ell}^{(2)}} \mathcal{D}^{(\xi)}.
\end{align}
\end{lemma}
\begin{proof} We first note that the definition of $\bold{T}^{(E, \alpha, 1)}_{m,n}$ implies that $m + n \ge 1$. Therefore, the identity $\chi_{m + n}^2 = \chi_{m + n}^2 \chi_{m + n -1}$ is valid. By an application of Cauchy-Schwartz in $(m, n)$, we obtain 
\begin{align*}
I^{(E, \alpha, 1)}(t) \lesssim & \Big[ \sum_{m = 0}^\infty \sum_{n = 0}^\infty  \bold{a}_{m,n}^2 \nu \| \bold{T}^{(E, \alpha, 1)}_{m,n} e^W q^{n} \chi_{m + n}^2 \|_{L^2}^2 \Big]^{\frac12} \\
& \times \Big[ \sum_{m = 0}^\infty \sum_{n = 0}^\infty  \bold{a}_{m,n}^2 \nu  \|\p_y \omega_{m,n} e^W q^{n} \chi_{m + n-1}  \|_{L^2}^2 \Big]^{\frac12} \\
\lesssim & \Big[ \sum_{m = 0}^\infty \sum_{n = 0}^\infty \frac{ \bold{a}_{m,n}^2}{(m+n)^{2s}} \nu \| \bold{T}^{(E, \alpha, 1)}_{m,n} e^W q^{n} \chi_{m + n}^2  \|_{L^2}^2 \Big]^{\frac12} \\
& \times \Big[ \sum_{m = 0}^\infty \sum_{n = 0}^\infty \mathbbm{1}_{n \ge 1} \bold{a}_{m,n-1}^2 \nu  \|\p_y \nabla \omega_{m,n-1} e^W q^{n} \chi_{m + n-1}  \|_{L^2}^2 \Big]^{\frac12} + \text{symm.} \\
\lesssim & \Big[ \sum_{m = 0}^\infty \sum_{n = 0}^\infty \frac{ \bold{a}_{m,n}^2}{(m+n)^{2s}} \nu \|\bold{T}^{(E, \alpha, 1)}_{m,n} e^W q^{n} \chi_{m + n}^2   \|_{L^2}^2 \Big]^{\frac12} \mathcal{D}^{(\alpha)}(t) + \text{symm.} \\
= &  \| \{ \mathcal{Q}_{m,n}[\{ \nu^{\frac12} \p_y \phi^{(E)}_{\neq, m,n}  \}, \{ \omega_{m,n} \}] \} \|_{X^{0}_2(\frac{\bold{a}_{m,n}}{(m + n)^{s},}  q^{n} e^W \chi_{n+m}^2 )}  \mathcal{D}^{(\alpha)}(t) + \text{symm.} \\
\lesssim &  \frac{\langle t \rangle^2}{\nu^2} \| \{ \p_y \phi^{(E)}_{m,n}  \} \|_{X^{2}_2(\bold{a}_{m,n}, q^n \chi_{m+n})} \| \{ \sqrt{\nu} \nabla \omega_{m,n} \} \|_{X^{0}_2(\bold{a}_{m,n}, e^W \chi_{m + n} q^n)} \\
& \times \mathcal{D}^{(\alpha)}(t)  + \text{symm.} \\
\lesssim & \frac{\langle t \rangle^2}{\nu^2} \sqrt{\mathcal{J}_{ell}^{(3)}(t)} ( \sqrt{\mathcal{E}^{(\alpha)}} +\sqrt{ \mathcal{E}^{(\mu)}})\sqrt{ \mathcal{D}^{(\alpha)}(t)} + \text{symm.}
\end{align*}
Above, we have invoked the bound \eqref{ewing:2}. The proof of the second bound, \eqref{063023:2}, follows in a nearly identical manner. 
\end{proof}

\begin{lemma} The following bounds are valid: 
\begin{align} \label{063023:3}
|I^{(E, \xi, 3)}(t)| \lesssim &  \langle t \rangle \sqrt{ \mathcal{J}^{(3)}_{ell}} ( \mathcal{E}^{(\xi)} \mathcal{D}^{(\xi)})^{\frac12} \\ \label{063023:4}
|I^{(E, \xi, 4)}(t)| \lesssim &  \langle t \rangle\sqrt{ \mathcal{J}^{(2)}_{ell}} \mathcal{D}^{(\xi)}.
\end{align}
\end{lemma}
\begin{proof} We first note that the definition of $\bold{T}^{(E, \alpha, 3)}_{m,n}$ implies that $m + n \ge 1$. Therefore, the identity $\chi_{m + n}^2 = \chi_{m + n}^2 \chi_{m + n -1}$ is valid. By an application of Cauchy-Schwartz in $(m, n)$, we obtain 
\begin{align*}
I^{(E, \alpha, 3)}(t) \lesssim & \Big[ \sum_{m = 0}^\infty \sum_{n = 0}^\infty  \bold{a}_{m,n}^2 \nu \|  \bold{T}^{(E, \alpha, 3)}_{m,n} e^W q^{n} \chi_{m + n}^2   \|_{L^2}^2 \Big]^{\frac12} \\
& \times \Big[ \sum_{m = 0}^\infty \sum_{n = 0}^\infty  \bold{a}_{m,n}^2 \nu  \|\p_y \omega_{m,n} e^W q^{n} \chi_{m + n-1}  \|_{L^2}^2 \Big]^{\frac12} \\
\lesssim & \Big[ \sum_{m = 0}^\infty \sum_{n = 0}^\infty \frac{ \bold{a}_{m,n}^2}{(m+n)^{2s}} \nu \| \bold{T}^{(E, \alpha, 3)}_{m,n} e^W q^{n} \chi_{m + n}^2 \|_{L^2}^2 \Big]^{\frac12} \\
& \times \Big[ \sum_{m = 0}^\infty \sum_{n = 0}^\infty \mathbbm{1}_{n \ge 1} \bold{a}_{m,n-1}^2 \nu^{\frac23}  \|\p_y \nabla \omega_{m,n-1} e^W q^{n} \chi_{m + n-1}  \|_{L^2}^2 \Big]^{\frac12} + \text{symm.} \\
\lesssim & \Big[ \sum_{m = 0}^\infty \sum_{n = 0}^\infty \frac{ \bold{a}_{m,n}^2}{(m+n)^{2s}} \nu \|\bold{T}^{(E, \alpha, 3)}_{m,n} e^W q^{n} \chi_{m + n}^2  \|_{L^2}^2 \Big]^{\frac12}\sqrt{ \mathcal{D}^{(\alpha)}} + \text{symm.} \\
= &  \| \{ \mathcal{R}_{m,n}[\{ \nu^{\frac12} \p_y \phi^{(E)}_{\neq, m,n}  \}, \{ \omega_{m,n} \}] \} \|_{X^{0}_2(\frac{\bold{a}_{m,n}}{(m + n)^{s},}  q^{n} e^W \chi_{n+m}^2 )}  \sqrt{\mathcal{D}^{(\alpha)}} + \text{symm.} \\
\lesssim &\langle t \rangle \| \{  \p_y \phi^{(E)}_{\neq, m,n} \} \|_{X^2_2( \bold{a}_{m,n} q^n \chi_{m+n} )} \| \{\nabla \omega_{m,n}\} \|_{X^0_2( \bold{a}_{m,n} q^n \chi_{m+n} e^W )} \sqrt{ \mathcal{D}^{(\alpha)}} \\
&  + \text{symm.} \\
\lesssim & \langle t \rangle\sqrt{ \mathcal{J}^{(3)}_{ell}(t)} \sqrt{ \mathcal{E}^{(\alpha)}} \sqrt{\mathcal{D}^{(\alpha)}} + \text{symm.}
\end{align*}
Above, we have invoked the bound  \eqref{63023:1}. The second bound, \eqref{063023:4}, follows in a nearly identical manner. 
\end{proof}

\begin{lemma} The following bounds are valid:
\begin{align} \label{063023:5}
|I^{(E, \xi, 5)}(t)| \lesssim &\sqrt{ \mathcal{J}^{(3)}_{ell}(t)}(\sqrt{\mathcal{D}^{(\gamma)}(t)} + \sqrt{\mathcal{D}^{(\alpha)}(t)} + \sqrt{\mathcal{D}^{(\mu)}(t)} )\sqrt{\mathcal{E}^{(\xi)}(t)}.
\end{align}
\end{lemma}
\begin{proof} We estimate as follows 
\begin{align*}
|I^{(E, \alpha, 5)}(t)| \lesssim & \Big[ \sum_{m = 0}^\infty \sum_{n = 0}^\infty \bold{a}_{m,n}^2 \nu \| \p_y \nabla^\perp \phi_{\neq}^{(E)}  \|_{L^\infty}^2 \| \nabla \omega_{m,n} \chi_{m + n} e^{W}\|_{L^2}^2 \Big]^{\frac12} \sqrt{\mathcal{E}^{(\alpha)}(t)} \\
\lesssim & \nu^{\frac12} \| \p_y \nabla^\perp \phi_{\neq}^{(E)}  \|_{L^2_x L^\infty_y} \Big[ \sum_{m = 0}^\infty \sum_{n = 0}^\infty \bold{a}_{m,n}^2  \| \nabla \omega_{m,n} \chi_{m + n} e^{W}\|_{L^\infty_x L^2_y}^2 \Big]^{\frac12}\sqrt{ \mathcal{E}^{(\alpha)}(t)} \\
\lesssim & \nu^{\frac12} \| \phi_{\neq}^{(E)}  \|_{H^3}  \sqrt{\mathcal{D}^{(\alpha)}(t) \mathcal{E}^{(\alpha)}(t)} \\
\lesssim & \sqrt{ \mathcal{J}^{(3)}_{ell}(t)}(\sqrt{\mathcal{D}^{(\gamma)}(t)} + \sqrt{\mathcal{D}^{(\alpha)}(t)} + \sqrt{\mathcal{D}^{(\mu)}(t)} )\sqrt{\mathcal{E}^{(\alpha)}(t)}.
\end{align*}
The lemma is proven. 
\end{proof}

\begin{lemma}The following bounds are valid:
\begin{align} \label{063023:6}
|I^{(E, \xi, 6)}(t)| \lesssim &\sqrt{ \mathcal{J}_{ell}^{(3)}(t) \mathcal{D}^{(\xi)}(t) \mathcal{E}^{(\xi)}(t)} +  \langle t \rangle^4 \nu^{\frac23} \sqrt{\mathcal{J}_{ell}^{(3)}(t) \mathcal{E}_{\text{cloud}}(t)\mathcal{E}^{(\xi)}(t)}.
\end{align}
\end{lemma}
\begin{proof} We estimate as follows 
\begin{align*}
I^{(E, \alpha, 6)}(t)  \lesssim & \Big[ \sum_{m = 0}^\infty \sum_{n = 0}^\infty   \bold{a}_{m,n}^2 \nu \| \bold{T}^{(E, \alpha, 6)}_{m,n} e^W q^{n} \chi_{m + n} \|_{L^2}^2 \Big]^{\frac12} \\
& \times \Big[ \sum_{m = 0}^\infty \sum_{n = 0}^\infty   \bold{a}_{m,n}^2 \nu  \| \p_y \omega_{m,n} e^W q^{n} \chi_{m + n}\|_{ L^2}^2 \Big]^{\frac12} \\
\lesssim & \Big[ \sum_{m = 0}^\infty \sum_{n = 0}^\infty   \bold{a}_{m,n}^2 \nu \|\bold{T}^{(E, \alpha, 6)}_{m,n} e^W q^{n} \chi_{m + n} \|_{L^2}^2 \Big]^{\frac12} \sqrt{\mathcal{E}^{(\alpha)}(t)} \\
\lesssim & \Big[ \sum_{m = 0}^\infty \sum_{n = 0}^\infty  \bold{a}_{m,n}^2 \nu \|\nabla^\perp \phi_{\neq}^{(E)} \cdot \p_y \nabla  \omega_{m,n} e^W q^{n} \chi_{m + n} (1 -\chi_{10}) \|_{L^2}^2 \Big]^{\frac12} \sqrt{\mathcal{E}^{(\alpha)}(t)} \\
& + \Big[ \sum_{m = 0}^\infty \sum_{n = 0}^\infty  \bold{a}_{m,n}^2 \nu \|\nabla^\perp \phi_{\neq}^{(E)} \cdot \p_y \nabla  \omega_{m,n} e^W q^{n} \chi_{m + n} \chi_{10} \|_{L^2}^2 \Big]^{\frac12} \sqrt{ \mathcal{E}^{(\alpha)}(t)} \\
\lesssim & \Big[ \sum_{m = 0}^\infty \sum_{n = 0}^\infty  \bold{a}_{m,n}^2 \nu  \|\nabla^\perp \phi_{\neq}^{(E)} \cdot \p_y \nabla  \omega_{m,n} e^W q^{n} \chi_{m + n} (1 -\chi_{10}) \|_{L^2}^2 \Big]^{\frac12} \sqrt{ \mathcal{E}^{(\alpha)}(t)}  \\
& +  \Big[ \sum_{m = 0}^\infty \sum_{n = 0}^\infty \bold{a}_{m,n}^2 \nu \| \nabla^\perp \phi_{\neq}^{(E)}  \chi_9\|_{L^\infty}\| \p_y \nabla  \omega_{m,n} e^W q^{n} \chi_{m + n} \chi_{10} \|_{L^2}^2 \Big]^{\frac12} \sqrt{\mathcal{E}^{(\alpha)}(t)}\\
= & I^{(E, \alpha, 6, int)}(t) + I^{(E, \alpha, 6, ext)}(t).
\end{align*}
For the ``ext" term, we estimate as follows 
\begin{align*}
| I^{(E, \alpha, 6, ext)}(t)| \lesssim & \|  \nabla^\perp \phi_{\neq}^{(E)}\chi_9 \|_{L^\infty} \sqrt{\mathcal{D}^{(\alpha)}(t) \mathcal{E}^{(\alpha)}(t) }\\
\lesssim &   \|\nabla^\perp \phi_{\neq}^{(E)}   \chi_9\|_{H^1_x H^1_y}\sqrt{ \mathcal{D}^{(\alpha)}(t) \mathcal{E}^{(\alpha)}(t) }\\
\lesssim &\sqrt{ \mathcal{J}_{ell}^{(3)}(t) \mathcal{D}^{(\alpha)}(t) \mathcal{E}^{(\alpha)}(t)}.  
\end{align*}
For the ``int" term, we need to invoke the Cloud norm as follows. First of all, we notice that the presence of $1 - \chi_{10}$ limits the indices to $m + n \le 15$ (clearly, $\chi_{m + n} (1 - \chi_{10}) = 0$ if $m +  n \ge 15$). Indeed, we have 
\begin{align*}
|I^{(E, \alpha, 6, int)}(t)| \lesssim & \Big[ \sum_{m = 0}^\infty \sum_{n = 0}^\infty \mathbbm{1}_{m + n \le 15}  \bold{a}_{m,n}^2 \nu \|\nabla^\perp \phi_{\neq}^{(E)} \cdot \p_y \nabla  \omega_{m,n} e^W q^{n} \chi_I \|_{L^2_x L^2_y}^2 \Big]^{\frac12}\sqrt{ \mathcal{E}^{(\alpha)}(t)} \\
 \lesssim & \langle t \rangle^4 \nu \|\nabla^\perp \phi_{\neq}^{(E)}  \|_{L^\infty} \Big[ \sum_{m = 0}^\infty \sum_{n = 0}^\infty \mathbbm{1}_{m + n \le 20}  \bold{a}_{m,n}^2  \|    \omega_{m,n} e^W q^{n} \chi_I \|_{L^2_x L^2_y}^2 \Big]^{\frac12}\sqrt{ \mathcal{E}^{(\alpha)}(t)} \\
 \lesssim & \langle t \rangle^4 \nu^{\frac23} \sqrt{\mathcal{J}_{ell}^{(3)}(t) \mathcal{E}_{\text{cloud}}(t)\mathcal{E}^{(\alpha)}(t)}.
\end{align*}
The lemma is proven. 
\end{proof}

\subsection{Proof of Proposition \ref{pro:tri:in}}

We are now ready to consolidate all of our trilinear estimates in order to prove Proposition \ref{pro:tri:in}. 

\begin{proof}[Proof of Proposition \ref{pro:tri:in}]To obtain \eqref{esp:1}, we sum together \eqref{jb:43:a} -- \eqref{jb:43:c} to get 
\begin{align*}
|\mathcal{I}^{(I, \gamma)}| \lesssim  &(\mathcal{E}_{ell}^{(I)})^{\frac12}   \mathcal{CK}^{(\gamma)} + \frac{(\mathcal{E}_{ell}^{(I)})^{\frac12}|}{\langle t \rangle^{100}}  ( (\mathcal{E}^{(\gamma)})^{\frac12} + (\mathcal{CK}^{(\gamma)})^{\frac12})(\mathcal{CK}^{(\gamma)})^{\frac12} \\
&+ \frac{(\mathcal{E}^{(I)}_{ell})^{\frac12}}{\langle t \rangle^{88}} (\mathcal{D}^{(\gamma)})^{\frac12} (\mathcal{CK}^{(\gamma, W)})^{\frac12} +  (\mathcal{E}^{(I)}_{ell})^{\frac12} (\mathcal{CK}_{\text{Cloud}})^{\frac12}( \mathcal{CK}^{(\gamma)})^{\frac12} \\
\lesssim & (\mathcal{E}^{(I)}_{ell})^{\frac12}( \mathcal{CK}^{(\gamma)} + \mathcal{CK}_{\text{Cloud}} + \mathcal{D}^{(\gamma)} ) + \frac{(\mathcal{E}_{ell}^{(I)})^{\frac12}}{\langle t \rangle^{100}} \mathcal{E}^{(\gamma)}
\end{align*}
upon using Young's inequality for products. 

The proofs of the remaining bounds in Proposition \ref{pro:tri:in} follows in a similar manner. To obtain \eqref{esp:2}, we sum together \eqref{hold:ya:1} and \eqref{hold:ya:2} and similarly use Young's inequality for products. To obtain \eqref{esp:3}, we sum together \eqref{rodeo:1} -- \eqref{rodeo:6} and split terms appropriately using Young's inequality. Finally to obtain \eqref{esp:4}, we sum together \eqref{063023:1}, \eqref{063023:2}, \eqref{063023:3}, \eqref{063023:4}, \eqref{063023:5}, and \eqref{063023:6} and again use Young's inequality to split the terms appropriately. 

\end{proof}


\section{Coordinate System Estimates} \label{sec:coord:FEI}
In this section, we obtain bounds on the coordinate system quantities $(\overline {H}, H, G)$. 
To make the writing compact, we denote 
\begin{align}
	\ztp{F}_{n} :=q^n F_n = q^n \Gamma_0^n F = q^n \Big( \frac{\pa_y}{v_y} \Big)^n F
\end{align}
for $F\in\{\overline H, H, G\}$.
As before, when the parameter $\iota=\gamma$, it means a $L^2$ based norm and thus we often omit it. For instance, we denote
\begin{align*}
	\mathcal{E}_{\overline{H},n}=\mathcal{E}_{\overline{H},n}^{(\gamma)}.
\end{align*}
 We note that $H$ has the same regularity as $\omega$ while $\overline H$ has $\ss$ derivatives less but with a polynomial decay in time.  To capture this regularity discrepancy, we make the following definition.
 \begin{definition} For $\alpha\in\{0, 1\}$ and $\beta\in\{0, -s\}$, we define the norm $Y_{\alpha,\beta}$ as
 	\label{Y:norm}
\begin{align*}
	&\norm{H}_{Y_{0,0}} = \sqrt{\mathcal{E}_{{H}}^{(\gamma)}} =    \paren{\sum_{n\ge0}  \bold{a}_n^2 \| \hqn  e^{W/2} \chi_n \|_{L^2}^2}^{1/2},
	\\&\norm{\overline H}_{Y_{0,-s}} =  \brak{t}^{-3/2-\ss} \sqrt{\mathcal{E}_{\overline{H}}^{(\gamma)}} =   \paren{\sum_{n\ge0}(n+1)^{2\sss-2}  \bold{a}_{n+1}^2  \|  \bhqn e^{W/2} \chi_n \|_{L^2}^2}^{1/2}, \\
	&\norm{H}_{Y_{1,0}} = \sqrt{\mathcal{E}_{{H}}^{(\alpha)}} = \nu^{1/4}   \paren{\sum_{n\ge0} \bold{a}_n^2 \| \partial_y \hqn  e^{W/2} \chi_n \|_{L^2}^2}^{1/2},
	\\& \norm{\overline H}_{Y_{1, -s}} =  \brak{t}^{-3/2-\ss} \sqrt{\mathcal{E}_{\overline{H}}^{(\alpha)}}
	= \nu^{1/4}    \paren{\sum_{n\ge0}(n+1)^{2\sss-2}  \bold{a}_{n+1}^2  \|  \partial_{y} \bhqn e^{W/2} \chi_n \|_{L^2}^2}^{1/2}.
\end{align*}
 \end{definition}
 \begin{remark}
 	We note that the first sub-index in $Y_{\alpha,\beta}$ is reserved for how many $y$ derivatives we have in the norm, while the second index means regularity loss compared with the Gevrey space $G^{s}$. 
 	Sometimes we also need exact the same norm but without the weight $e^{W/2}$, which are denoted by
 	\begin{align*}
 		\norm{H}_{\overline Y_{0,0}},\ \norm{\overline H}_{\overline Y_{0,-s}}, \ 
 		\norm{H}_{\overline Y_{1,0}} ,\ \norm{\overline H}_{\overline Y_{1, -s}} 
 	\end{align*}
 	respectively.
 \end{remark}
Our main result in this section is the following proposition. 
\begin{proposition}
	We have
	\begin{align*}
		\frac{1}{2}\frac{d}{dt} \mathcal{E}_{\text{Ext, Coord}}+
		 \mathcal{D}_{\text{Ext, Coord}} +
 \mathcal{CK}_{\text{Ext, Coord}}
		\lesssim&  \frac{\eps^3}{\brak{t}^{3+r-2\ss}} + \eps \sum_{\iota\in\{\alpha,\gamma\}} \cd^{(\iota)} .
	\end{align*}
\end{proposition}
\begin{proof}
	It is an direct consequence of the Proposition~\ref{gamm:esti} and \ref{alph:esti}. 
\end{proof}
\begin{proposition}[$\gamma$ estimate] 
	\label{gamm:esti}Let $\overline H, G$, and $H$ be the solutions of the equations~\eqref{eq:bar:H:main}, \eqref{eq:G:main}, and \eqref{eq:H:main} respectively. Assuming the bootstraps,  the following estimates are valid for the $\gamma$-level coordinate functionals: 
\begin{align} \label{apbarH}
\frac{1}{2}\frac{d}{dt} \mathcal{E}_{\overline{H}}^{\gm}+
\cd_{\overline{H}}^{\gm}+
\mathcal{CK}_{\overline{H}}^{\gm}
\lesssim&  
\frac{\eps^3}{\brak{t}^{3+r-2\ss}} +  \eps \paren{\cd^{\gm}+\cd_{H}^{\gm} 
	+\cd_{\overline H}^{\gm}}
\n	\\&\quad 
+\eps\paren{\mathcal{CK}^{\gm}
	+\mathcal{CK}_{H}^{\gm}+\mathcal{CK}_{\overline H}^{\gm}}, \\ \label{apG}
\frac{1}{2}\frac{d}{dt} \mathcal{E}_{G}^{\gm}+
\cd_{G}^{\gm}+
\mathcal{CK}_{G}^{\gm}
\lesssim&  
\frac{\eps^3}{\brak{t}^{3+r-2\ss}} +  \eps \paren{\cd^{\gm}+\cd_{H}^{\gm} 
	+\cd_{\overline H}^{\gm}}
\n \\&\quad 
+\eps\paren{\mathcal{CK}^{\gm}
	+\mathcal{CK}_{H}^{\gm}+\mathcal{CK}_{\overline H}^{\gm}},
 \\ \label{apH}
\frac{1}{2}\frac{d}{dt} \mathcal{E}_{H}^{\gm}+
\cd_{H}^{\gm}+
\mathcal{CK}_{H}^{\gm}
\lesssim &   \mathcal{CK}_{\overline{H}}^{\gm} .
\end{align}
for $r$ and $s$ defined as in \eqref{pgiL1}.
\end{proposition}

\begin{proposition}[$\alpha$ estimate] 
	\label{alph:esti}Under the assumptions of Proposition~\ref{gamm:esti}, the following estimates are valid for the $\alpha$-level coordinate functionals: 
	\begin{align} \label{albarH}
		\frac{1}{2}\frac{d}{dt} \mathcal{E}_{\overline{H}}^{(\alpha)}&+
		\cd_{\overline{H}}^{(\alpha)}+
		\mathcal{CK}_{\overline{H}}^{(\alpha)}
	\lesssim \frac{\eps^3}{\brak{t}^{3+r-2\ss}} +  \eps \sum_{\iota\in\{\alpha,\gamma\}} \paren{\cd^{(\iota)}+\cd_{H}^{(\iota)} 
	+\cd_{\overline H}^{(\iota)}}
	\n \\&\qquad  \qquad  \qquad  \qquad  \qquad  
+\eps\sum_{\iota\in\{\alpha,\gamma\}} \paren{\mathcal{CK}^{(\iota)}
	+\mathcal{CK}_{H}^{(\iota)}+\mathcal{CK}_{\overline H}^{(\iota)}}, 
		 \\ \label{alG}
		\frac{1}{2}\frac{d}{dt} \mathcal{E}_{G}^{(\alpha)} &+
		\cd_{G}^{(\alpha)}+
		\mathcal{CK}_{G}^{(\alpha)}
	\lesssim \frac{\eps^3}{\brak{t}^{3+r-2\ss}} +  \eps \sum_{\iota\in\{\alpha,\gamma\}} \paren{\cd^{(\iota)}+\cd_{H}^{(\iota)} 
		+\cd_{\overline H}^{(\iota)}}
\n	\\&\qquad  \qquad  \qquad  \qquad  \qquad  
	+\eps\sum_{\iota\in\{\alpha,\gamma\}} \paren{\mathcal{CK}^{(\iota)}
		+\mathcal{CK}_{H}^{(\iota)}+\mathcal{CK}_{\overline H}^{(\iota)}}, 
\\ \label{alH}
		\frac{1}{2}\frac{d}{dt} \mathcal{E}_{H}^{(\alpha)} &+
			\cd_{H}^{(\alpha)}+
			\mathcal{CK}_{H}^{(\alpha)}
		\lesssim \mathcal{CK}_{\overline{H}}^{(\alpha)} .
	\end{align}
for $r$ and $s$ defined as in \eqref{pgiL1}.
\end{proposition}


We first record a lemma proved in a companion paper~\cite{BHIW24b}. 
\begin{lemma} \label{lem:com:BA}The following commutator relations hold:
	\begin{subequations}\label{commutators}
		\begin{align}[\pa_{y},\Gamma_k^n ]=&-ikt\sum_{\ell=0}^{n-1}\binom{n}{\ell}\pav^{n-\ell}(v_y-1)\Gamma_k^{\ell}+\sum_{\ell=1}^{n}\binom{n}{\ell}\pav^{n-\ell+1}(v_y-1)\Gamma_k^{\ell}\label{cm_py_G_n}\\
			[\pa_{yy},\Gamma_k^n ]=&-\sum_{\ell=0}^{n-1}\binom{n}{\ell} \left(2\pav^{n-\ell-1}v''\ \pav^2 +\pav^{n-\ell}v'' \ \pav\right)\Gamma_k^{\ell};\label{cm_pyy_G_n}\\
			[\pa_y,q^{n}]f=&q'\lf(\frac{n}{q}q^n f\rg) ;\label{cm_py_qn}\\
			[\pa_{yy},q^{n}]f=&-(q')^2\frac{n^2+n}{n^2}\lf(\frac{n^2}{q^2} q^n f\rg)+2q'\lf(\frac{n }{q}\pa_y(q^n f)\rg)+q''\lf(\frac{n}{q}q^n f\rg);\label{cm_pyy_qn}\\
			\lf[\pav,q^n\rg] f=&\frac{q'}{v_y}\lf(\frac{n}{q}q^n f\rg);\label{cm_pv_qn}\\ \n
			\lf[ \pav^2,q^n\rg]f=&2\frac{q'}{(v_y)^2}\lf(\frac{n}{q}\pa_y\lf(q^n f\rg)\rg)+\frac{q''}{(v_y)^2}\lf(\frac{n}{q}q^n f\rg)-\frac{q'}{(v_y)^2}\lf(\vyi\pa_y v_y\rg)\lf(\frac{n}{q}q^n f\rg)\\ 
			&-\frac{(q')^2}{(v_y)^2}\frac{n^2+n}{n^2}\lf(\frac{n^2}{q^2}q^n f\rg);\label{cm_pvv_qn}\\
			\lf[\pav^2,q\rg]f=&2\frac{q'}{v_y^2}  \pa_y  f  +\frac{q''}{v_y^2} f  -\frac{q'}{ v_y ^3}v _{yy}   f .\label{cm_pvv_q}
		\end{align} 
	\end{subequations}
	Here we recall the notation $\pav=v_y^{-1}\pa_y.$
\end{lemma}

\subsection{$(\overline{H}, \gamma)$ Estimates}
\begin{subequations}
\begin{align} \label{eq:bar:H:L}
&\pa_t \overline{H} + \frac2t \overline{H}  - \nu \pa_y^2 \overline{H}= -\frac{1}{t} \left( u_{\neq} \cdot \nabla \omega \right)_0, \\
&\overline{H}(t, \pm 1) =  0, \\
&\overline{H}|_{t =0} = 0  
\end{align}
\end{subequations}
We commute $q^n \Gamma_0^n$ to \eqref{eq:bar:H:L} to obtain 
\begin{subequations}
	\begin{align} \label{eq:bar:H:L:n}
		&\pa_t \bhqn + \frac2t \bhqn  - \nu \pa_y^2 \bhqn = -  \widetilde{\mathcal{C}^{(n)}_{trans}} -  \widetilde{\mathcal{C}^{(n)}_{visc}} -  \widetilde{\mathcal{C}^{(n)}_{q}}   - q^n \Gamma_0^n \frac{1}{t}       \paren{  \nabla^{\perp} \stf_{\neq 0} \cdot \nabla \omega }_0, \\
		&\bhqn (t, \pm 1) =  0,
	\end{align}
\end{subequations}
where
\begin{align} \label{Cnta}
 \widetilde{\mathcal{C}^{(n)}_{trans}} := & q^n\sum_{m = 1}^n \binom{n}{m} \Gamma_0^m G \Gamma_0^{n-m+1} \overline{H}, \\ \label{Cnva}
 \widetilde{\mathcal{C}^{(n)}_{visc}} := & \nu  q^n\sum_{m = 1}^n \binom{n}{m}  \Gamma_0^m v_y^2 \Gamma_0^{2 + n-m} \overline{H}, \\ \label{tildeCnq}
 \widetilde{C^{(n)}_q} := & - \nu |q'|^2 \frac{n(n-1)}{q^2}  \bhqn  - 2 \nu q'  \frac{n}{q} \pa_y \bhqn - \nu q'' \frac{n}{q} \bhqn .
\end{align}


\begin{proof}[Proof of~\eqref{apbarH}:] 
We compute the (complex) inner product of \eqref{eq:bar:H:L:n} against $\brak{t}^{3+2\ss} \bold{a}_{n+1}^2\bhqn  e^{W} \chi_n^2$ to get 
\begin{align*}
 &   \frac{(n+1)^{2-2\sss}}{2}\frac{d}{dt}  \mathcal{E}_{\overline{H},n}^2- \bold{a}_{n+1}  \frac{d}{dt}\bold{a}_{n+1} \brak{t}^{3+2\ss}
   \| \bhqn  e^{W/2} \chi_{n} \|_{L^2}^2
+ \bold{a}_{n+1}^2  \brak{t}^{3+2\ss} \| \bhqn  \sqrt{- W_t} e^{W/2} \chi_n \|_{L^2}^2 
\\&\qquad\qquad+ \bold{a}_{n+1}^2  \brak{t}^{3+2\ss} \nu \|  \pa_y \bhqn  e^{W/2} \chi_n \|_{L^2}^2 
\\
&\qquad = 
-\bold{a}_{n+1}^2  \brak{t}^{3+2\ss}  \nu \langle \partial_y\bhqn ,
 \bhqn   e^{W}  \pa_y \{\chi_n^2\}  \rangle 
- \bold{a}_{n+1}^2   \brak{t}^{3+2\ss} \nu \langle \partial_y \bhqn , 
\bhqn  \pa_y \{ e^{W}  \} \chi_n^2  \rangle 
\\&\qquad\qquad  +
 \paren{(3/2+\ss)\brak{t}^{1+2\ss}t - \frac{2}{t}\brak{t}^{3+2\ss}}  \bold{a}_{n+1}^2 
\| \bhqn  e^{W/2} \chi_{n} \|_{L^2}^2
\\&\qquad\qquad- \brak{t}^{3+2\ss} \bold{a}_{n+1}^2\Re \brak{\widetilde{\mathcal{C}^{(n)}_{trans}} ,
	 \bhqn  e^{W} \chi_n^2}
-  \brak{t}^{3+2\ss} \bold{a}_{n+1}^2\Re\brak{  \widetilde{\mathcal{C}^{(n)}_{visc}},
	 \bhqn  e^{W} \chi_n^2}
\\&\qquad\qquad- \brak{t}^{3+2\ss} \bold{a}_{n+1}^2\Re \brak{ \widetilde{\mathcal{C}^{(n)}_{q}}, 
	\bhqn  e^{W} \chi_n^2}
\\&\qquad\qquad
- \brak{t}^{3+2\ss} t^{-1} \bold{a}_{n+1}^2 
\Re \brak{    q^n \Gamma_0^n \paren{  \nabla^{\perp} \stf_{\neq 0} \cdot \nabla \omega }_0, 
	\bhqn  e^{W} \chi_n^2}
\\
&\qquad =  -\bold{a}_{n+1}^2  \brak{t}^{3+2\ss}  \nu \langle \partial_y\bhqn ,
\bhqn   e^{W}  \pa_y \{\chi_n^2\}  \rangle 
- \bold{a}_{n+1}^2   \brak{t}^{3+2\ss} \nu \langle \partial_y \bhqn , 
\bhqn  \pa_y \{ e^{W}  \} \chi_n^2  \rangle 
\\&\qquad\qquad  +
\paren{(3/2+\ss)\brak{t}^{1+2\ss}t - \frac{2}{t}\brak{t}^{3+2\ss}} \bold{a}_{n+1}^2 \| \bhqn  e^{W/2} \chi_{n} \|_{L^2}^2
\\&\qquad\qquad- C^{(n)}_{\overline{H},trans} - C^{(n)}_{\overline{H},visc} - C^{(n)}_{\overline{H},q} - N^{(n)}_{\overline{H}}
.
\end{align*}
Next we give lower bounds for the CK terms. Noting
\begin{align*}
		\partial_t W(t, y)=-\frac{(|y|-1/4-L\ep\arctan (t))_+^2}{K\nu(1+t)^2} 
		- \frac{(|y|-1/4-L\ep\arctan (t))_+^2}{K\nu(1+t)}\paren{L\eps \frac{1}{1+t^2}},
\end{align*}
by~\eqref{CK1:defi}, we obtain
\begin{align}
	\label{CK:1}
	\mathcal{CK}^{(1)}_{\overline{H}, n}(t) ^2 \gtrsim  \nu^{-1/3+2\delta}\bold{a}_{n+1}^2 \brak{t}^{3+2\ss} \| \bhqn   e^{W/2} \chi_n \|_{L^2}^2
\end{align}
for $t\lesssim \nu^{-1/3-\delta}$ and $n\ge 1$. In view of~\eqref{Bweight}, \eqref{varphi}, and~\eqref{CK2:defi}, it is easy to check for $n\ge 0$
\begin{align}
	\label{CK:2}
	\mathcal{CK}^{(2)}_{\overline{H}, n}(t)^2 \gtrsim  \frac{n+1}{t} \bold{a}_{n+1}^2 \brak{t}^{3+2\ss}
	\| \bhqn  e^{W/2} \chi_n \langle t \rangle^2\|_{L^2}^2.
\end{align}
Then the desired result is a consequence of lemmas~\ref{nonl:barH}--\ref{easy:term}.
\end{proof}
\subsubsection{Nonlinear contribution}
We first consider the nonlinear estimate. 
\begin{lemma}
	\label{nonl:barH}
\begin{align*}
		\sum_{n}& \theta_{n}^2 (n+1)^{2\sss-2}\abs{N^{(n)}_{\overline{H}}} 
		\\&\lesssim  \paren{\sum_{j=1, 2}\sqrt{\mathcal{J}_{ell}^{(i)}}+ \sqrt{\mathcal{E}_{ell}^{(I, out)}}} \paren{\mathcal{E}_{\overline H}^{(\gamma)} \sqrt{\mathcal{E}_{H}^{(\alpha)}}\brak{t}^{-3-2s^{-1}} + \cd^{\gm}+\cd_{H}^{\gm} 
			+\cd_{\overline H}^{\gm} + \mathcal{CK}^{\gm}
			+\mathcal{CK}_{H}^{\gm}+\mathcal{CK}_{\overline H}^{\gm}}
			\\&\quad+
			\sqrt{\mathcal{E}_{ell}^{(I, full)} + \brak{t}^2 \mathcal{J}_{ell}^{(1)}} \sqrt{\mathcal{E_\mathrm{Int}}} \brak{t}^{-1/2-s^{-1}}\sqrt{\mathcal{E}_{\overline H}^{(\gamma)}}  
			\quad+ 
			\frac{\brak{t}^{2\ss}}{\brak{t}^{1+r}} \sqrt{\mathcal{E}_{ell}^{(I, full)}}  \sqrt{\mathcal{E}_{cloud}} \sqrt{\mathcal{E}^{(\overline h)}_{\mathrm{Int, Coord}}}
			\\&
			\quad+
			\nu^{100}\brak{t}^{2\ss} \paren{\sqrt{\mathcal{E}_{ell}^{(I, full)}} + \sqrt{\brak{t}^2 \mathcal{J}_{ell}^{(1)}}} \mathcal{E}^{(\gamma)} \sqrt{\mathcal{E}_{\overline H}^{(\gamma)}}
			\\&
			\lesssim \frac{\eps^3}{\brak{t}^{3+r-2\ss}} +  \eps \paren{\cd^{\gm}+\cd_{H}^{\gm} 
		+\cd_{\overline H}^{\gm}}
	+\eps\paren{\mathcal{CK}^{\gm}
		+\mathcal{CK}_{H}^{\gm}+\mathcal{CK}_{\overline H}^{\gm}}.
\end{align*}
\end{lemma}
\begin{proof} 
{\bf (1). Case $n\ge2 $:} The most non-trivial case is when $n\ge 2$ which we consider first.
We recall that 
\begin{align*}
	&N^{(n)}_{\overline{H}} 
	= 
    \Big \langle q^n\Gamma_{0}^n \frac{1}{t} \paren{  \nabla^{\perp} \stf_{\neq 0} \cdot \nabla \omega }_0, \brak{t}^{3+2\ss}  \bold{a}_{n+1}^2  \bhqn  e^{W} \chi_n^2 \Big \rangle .
\end{align*}
As in~\eqref{ell:E} and~\eqref{ell:I}, we decompose $\stf_k = \stf^{(I)}_k + \stf^{(E)}_k$, and hence we get the corresponding decomposition of $N^{(n)}_{\overline{H}}$ as
\begin{align*}
N^{(n)}_{\overline{H}} & = 	 N^E +  N^{I} .
\end{align*}
{\bf Exterior part:} We start with the exterior part. 
Applying the product rule, 
we  arrive at
\begin{align*}
  N^E &
  =- \frac{1}{t} \sum_{k} \Big \langle  
    q^n\Gamma_{0}^n \paren{ \partial_y \stf_{k}^{(E)} \overline{ ik \omega_k } }, \brak{t}^{3+2\ss}  \bold{a}_{n+1}^2  \bhqn  e^{W} \chi_n^2 \Big \rangle 
  \\&\quad
  +\frac{1}{t} \sum_{k} \Big \langle   q^n\Gamma_{0}^n \paren{ik \stf_{k}^{(E)} \partial_y \overline{\omega_k}},  \brak{t}^{3+2\ss}  \bold{a}_{n+1}^2  \bhqn  e^{W} \chi_n^2 \Big \rangle 
  \\&
  = \frac{1}{t} \sum_{k}\sum_{m=0}^{n}\binom{n}{m}\Big \langle  q^{n}\Gamma_{k}^{n-m}\partial_y \stf_{k}^{(E)} \Gamma_{k}^{m} ik \overline{\omega_k } , \brak{t}^{3+2\ss}  \bold{a}_{n+1}^2  \bhqn  e^{W} \chi_n^2 \Big \rangle 
    \\&\quad
  +\frac{1}{t} \sum_{k}\sum_{m=0}^{n}\binom{n}{m}\Big \langle  q^{n}\Gamma_{k}^{n-m} ik \stf_{k}^{(E)} \Gamma_{k}^{m} \partial_y \overline{\omega_k} , \brak{t}^{3+2\ss}  \bold{a}_{n+1}^2  \bhqn  e^{W} \chi_n^2 \Big \rangle 
  \\&
  =  N_{1}^E + N_{2}^E.
\end{align*}
We first look at the term $N_{1}^E$ and divide it into two pieces:
\begin{align*}
	N_{1}^{E} &=  \brak{t}^{3+2\ss} t^{-1} \bold{a}_{n+1}^2 \sum_{k}\paren{\sum_{m\le n/2} + \sum_{m> n/2}^{n}}
	\binom{n}{m} \Big \langle  q^{n}\Gamma_{k}^{n-m}\partial_y \stf_{k}^{(E)} \Gamma_{k}^{m} ik \overline{\omega_k } ,   \bhqn  e^{W}\chi_n^2 \Big \rangle 
	\\& = N_{1} ^{HL, E} + N_{1} ^{LH, E}.
\end{align*}
We first consider the $LH$ piece as
\begin{align*}
\abs{	N_{1} ^{LH, E}} \lesssim
 \brak{t}^{3+2\ss} t^{-1} \bold{a}_{n+1}^2 \sum_{k} \sum_{m> n/2}^{n}\binom{n}{m} \norm{q^{n-m}\Gamma_{k}^{n-m}\partial_y \stf_{k}^{(E)}\chi_{n-1}}_{L^\infty}
 \enorm{ q^{m}\Gamma_{k}^{m} ik \overline{\omega_k }e^{W} \chi_n}  \enorm{\bhqn  \chi_n}.
\end{align*}
By~\eqref{cm_py_G_n}, we have
\begin{align*}
	[\pa_{y},\Gamma_k^{n-m}]
	&=
	-ikt\sum_{\ell=0}^{n-m-1}\binom{n-m}{\ell}\lf(\vyi\pa_y\rg)^{n-m-\ell}(v_y-1)\Gamma_k^{\ell}
	\\&\qquad
	+\sum_{\ell=1}^{n}\binom{n-m}{\ell}\lf(\vyi\pa_y\rg)^{n-m-\ell+1}(v_y-1)\Gamma_k^{\ell}.
\end{align*}
Using the formula
\begin{align*}
	[T_1,T_2T_3] = [T_1,T_2]T_3 + T_2[T_1,T_3]
\end{align*}
where $T_1,T_2,T_3$ are linear operators, we further obtain
\begin{align*}
	[\pa_{y},q^{n-m}\Gamma_{k}^{n-m}]& = 	[\pa_{y},q^{n-m}\Gamma_k^{n-m}] 
	= q^{n-m}[\pa_{y},\Gamma_k^{n-m}] +	[\pa_{y},q^{n-m}]\Gamma_k^{n-m} 
	\\&=
	-ikt\sum_{\ell=0}^{n-m-1}\binom{n-m}{\ell}q^{n-m-\ell}\bd^{n-m-\ell}(v_y-1)q^{\ell}\Gamma_k^{\ell}
		\\&\quad
		+\sum_{\ell=1}^{n}\binom{n-m}{\ell}q^{n-m-\ell+1}\bd^{n-m-\ell+1}(v_y-1)q^{\ell-1}\Gamma_k^{\ell}
	+
	q'(n-m)q^{n-m-1}\Gamma_k^{n-m}.
\end{align*}
Note that for the LH piece $m\ge 2$ holds and thus by Sobolev embedding we deduce 
\begin{align}
	\label{com:y:D}
	&\norm{q^{n-m}\Gamma_{k}^{n-m}\partial_y \stf_{k}^{(E)}\chi_{n-1}}_{L^\infty}
	\nn &\qquad\lesssim
	\norm{\partial_y \paren{q^{n-m}\Gamma_{k}^{n-m} \stf_{k}^{(E)}}\chi_{n-m+1}}_{L^\infty}
	\nn &\qquad\quad+ 
	kt\sum_{\ell=0}^{n-m-1}\binom{n-m}{\ell}\norm{q^{n-m-\ell}\bd^{n-m-\ell}(v_y-1)q^{\ell}\Gamma_k^{\ell}\stf_{k}^{(E)} \chi_{n-m+1}}_{L^\infty}
	\nn&\qquad\quad
	+\sum_{\ell=1}^{n}\binom{n-m}{\ell}\norm{\bd^{n-m-\ell+1}(v_y-1)q^{\ell-1}\Gamma_k^{\ell}\stf_{k}^{(E)}\chi_{n-m+1}}_{L^\infty}
	\nn&\qquad\quad
	+ 	
	(n-m)\norm{q^{n-m-1}\Gamma_k^{n-m}\stf_{k}^{(E)}\chi_{n-m+1}}_{L^\infty}
	\nn
&\qquad	\lesssim
	\norm{\partial_y \paren{q^{n-m}\Gamma_{k}^{n-m} \stf_{k}^{(E)}}\chi_{n-m+1}}_{L^2} ^{1/2}
	\Big(	\norm{\partial_y^2\paren{q^{n-m}\Gamma_{k}^{n-m} \stf_{k}^{(E)}}\chi_{n-m+1}}_{L^2}
		\nn&\qquad\quad
	+ \norm{\partial_y \paren{q^{n-m}\Gamma_{k}^{n-m} \stf_{k}^{(E)}}\partial_y\chi_{n-m+1}}_{L^2}\Big)^{1/2}
	\nn&\qquad\quad
	+ 
	kt\sum_{\ell=0}^{n-m-1}\binom{n-m}{\ell}\norm{q^{n-m-\ell}\bd^{n-m-\ell}(v_y-1)q^{\ell}\Gamma_k^{\ell}\stf_{k}^{(E)}\chi_{n-m+1}}_{L^\infty}
\nn &\qquad\quad
	+\sum_{\ell=1}^{n}\binom{n-m}{\ell}\norm{\bd^{n-m-\ell+1}(v_y-1)q^{\ell-1}\Gamma_k^{\ell}\stf_{k}^{(E)}\chi_{n-m+1}}_{L^\infty}
	\nn&\qquad\quad
	+
	(n-m)\norm{q^{n-m-1}\Gamma_k^{n-m}\stf_{k}^{(E)}\chi_{n-m+1}}_{L^\infty}
		\nn &\qquad =
		V^{LH, E}_1+V^{LH, E}_2+V^{LH, E}_3+V^{LH, E}_4.
\end{align}
We define
\begin{align}
	\label{LH:vi}
	N_{1,vi} ^{LH, E}:=\brak{t}^{3+2\ss} t^{-1} \bold{a}_{n+1}^2 \sum_{k} \sum_{m\ge n/2}^{n} V^{LH}_i
	\enorm{ q^{m}\Gamma_{k}^{m} ik \overline{\omega_k }e^{W} \chi_n}  \enorm{\bhqn  \chi_n}
\end{align}
for $i=1,2,3,4$.
Noting 
\begin{align*}
	V^{LH, E}_1 &\lesssim
	\norm{\partial_y \paren{q^{n-m}\Gamma_{k}^{n-m} \stf_{k}^{(E)}}\chi_{n-m+1}}_{L^2} ^{1/2}
	\norm{\partial_y^2\paren{q^{n-m}\Gamma_{k}^{n-m} \stf_{k}^{(E)}}\chi_{n-m+1}}_{L^2}^{1/2}
	\\&\quad
		+(n-m)^{(1+\sigma)/2} \norm{\partial_y \paren{q^{n-m}\Gamma_{k}^{n-m} \stf_{k}^{(E)}}\chi_{n-m}}_{L^2},
\end{align*}
we arrive at
\begin{align*}
	N_{1,v1} ^{LH, E}&:=\brak{t}^{3+2\ss} t^{-1} \bold{a}_{n+1}^2 \sum_{k} \sum_{m\ge n/2}^{n}
	\binom{n}{m} V^{LH, E}_1
	\enorm{ q^{m}\Gamma_{k}^{m} ik \overline{\omega_k }e^{W} \chi_n}  \enorm{\bhqn  \chi_n}
	\\&\lesssim \brak{t}^{3+2\ss} t^{-1} \bold{a}_{n+1}^2 \sum_{k} \sum_{m\ge n/2}^{n}
     \binom{n}{m}\norm{\partial_y \paren{q^{n-m}\Gamma_{k}^{n-m} \stf_{k}^{(E)}} \chi_{n-m+1}}_{L^2} ^{1/2}
          	\\&\quad\times
          \norm{\partial_y^2 \paren{q^{n-m}\Gamma_{k}^{n-m} \stf_{k}^{(E)}}\chi_{n-m+1}}_{L^2}^{1/2}
	\enorm{ q^{m}\Gamma_{k}^{m} ik \overline{\omega_k }e^{W} \chi_n}  \enorm{\bhqn  \chi_n} 
	\\&\quad+
	\brak{t}^{3+2\ss} t^{-1} \bold{a}_{n+1}^2 \sum_{k} \sum_{m\ge n/2}^{n} \binom{n}{m}
	(n-m)^{(1+\sigma)/2} \norm{\partial_y \paren{q^{n-m}\Gamma_{k}^{n-m} \stf_{k}^{(E)}}\chi_{n-m}}_{L^2}
		\\&\quad\times
	\enorm{ q^{m}\Gamma_{k}^{m} ik \overline{\omega_k }e^{W} \chi_n}  \enorm{\bhqn  \chi_n}. 
\end{align*}
Multiplying it by $(n+1)^{2\sss-2}$ and summing in $n$ gives
\begin{align*}
	\sum_{n}& \theta_{n}^2 (n+1)^{2\sss-2}N_{1,v1} ^{LH, E}\\
	&\lesssim
	\sum_{n\ge0}\sum_{m\ge n/2}^{n}\sum_{k}  \bold{a}_{n+1} \binom{n}{m} \bold{a}_{n-m}^{-1} \bold{a}_{1,m}^{-1} \brak{t}^{3/2+\ss}t^{-1/2} 
	\\&\quad \times  \bold{a}_{n-m} \norm{\partial_y \paren{q^{n-m}\Gamma_{k}^{n-m} \stf_{k}^{(E)}}\chi_{n-m+1}}_{L^2} ^{1/2}
	\norm{\partial_y^2\paren{q^{n-m}\Gamma_{k}^{n-m} \stf_{k}^{(E)}}\chi_{n-m+1}}_{L^2}^{1/2}
	\\&\quad \times
	(n+2)^{\sss-1}\bold{a}_{1,m} \enorm{ k q^{m}\Gamma_{k}^{m} \overline{\omega_k }e^{W} \chi_m}  
	\brak{t}^{3/2+\ss} t^{-1/2}(n+2)^{\sss-1}\bold{a}_{n+1}\enorm{\bhqn  \chi_n}
	\\&\quad + 
	\sum_{n\ge0}\sum_{m\ge n/2}^{n}\sum_{k}  \bold{a}_{n+1} \binom{n}{m} \bold{a}_{n-m}^{-1} \bold{a}_{1,m}^{-1}
	\brak{t}^{3/2+\ss}
		t^{-1/2} 
	\bold{a}_{n-m} (n-m)^{(1+\sigma)/2}
			\\&\quad \times
	 \norm{\partial_y \paren{q^{n-m}\Gamma_{k}^{n-m} \stf_{k}^{(E)}}\chi_{n-m}}_{L^2}
	(n+2)^{\sss-1}\bold{a}_{1,m} \enorm{ k q^{m}\Gamma_{k}^{m} \overline{\omega_k }e^{W} \chi_m}  
		\\&\quad \times
		\brak{t}^{3/2+\ss} t^{-1/2}(n+2)^{\sss-1}\bold{a}_{n+1}\enorm{\bhqn  \chi_n}. 
\end{align*}
By Corollary~\ref{comb:boun:vari:1}, the elliptic estimate~\eqref{jell:1}-\eqref{jell:2}, Lemma~\ref{con:no:k}, and that fact that
\begin{align}
	\label{weig:huge}
	\nu^{-100} \lesssim e^{W(t, y)}\ \mbox{for}\ y\in \supp \chi_1,\ t\lesssim \nu^{-1/3-\delta},
\end{align} 
the above quantity is bounded by 
\begin{align*}
	\sum_{n}& \theta_{n}^2 (n+1)^{2\sss-2}  N_{1,v1} ^{LH, E} 
\\&\lesssim
	 \paren{\sum_{j=1, 2}\sqrt{\mathcal{J}_{ell}^{(i)}}}
	\sqrt{\cd^{\gm} }\sqrt{\mathcal{CK}_{\overline{H}}^{\gm}} \lesssim \eps \sqrt{\cd^{\gm}} \sqrt{\mathcal{CK}_{\overline{H}}^{\gm}}
\end{align*}
where we also used H\"older's inequality. 
Next we deal with $N_{1,v2} ^{LH, E}$.
We note that
\begin{align*}
V^{LH, E}_2&=	kt\sum_{\ell=0}^{n-m-1}\binom{n-m}{\ell}\norm{q^{n-m-\ell}\bd^{n-m-\ell}(v_y-1)q^{\ell}\Gamma_k^{\ell}\stf_{k}^{(E)}\chi_{n-1}}_{L^\infty}
\\&\lesssim
kt\sum_{\ell=0}^{n-m-1}\binom{n-m}{\ell}\norm{q^{n-m-\ell}\bd^{n-m-\ell}(v_y-1)\chi_{n-2}}_{L^\infty}\norm{q^{\ell}\Gamma_k^{\ell}\stf_{k}^{(E)}\chi_{n-1}}_{L^\infty}
\\&\lesssim
kt\sum_{\ell=0}^{n-m-1}\binom{n-m}{\ell}\norm{q^{n-m-\ell}\bd^{n-m-\ell}(v_y-1)\chi_{n-2}}_{L^\infty}\norm{q^{\ell}\Gamma_k^{\ell}\stf_{k}^{(E)}\chi_{n-1}}_{L^\infty}
\\&\lesssim
kt\sum_{\ell=0}^{n-m-1}\binom{n-m}{\ell}\norm{q^{n-m-\ell}\bd^{n-m-\ell}(v_y-1)\chi_{n-2}}_{L^2}^{1/2}
\\&\quad\times \norm{\partial_y\paren{q^{n-m-\ell}\bd^{n-m-\ell}(v_y-1)\chi_{n-2}}}_{L^2}^{1/2}
\norm{q^{\ell}\Gamma_k^{\ell}\stf_{k}^{(E)}\chi_{n-1}}_{L^2}^{1/2} \norm{\partial_y\paren{q^{\ell}\Gamma_k^{\ell}\stf_{k}^{(E)}\chi_{n-1}}}_{L^2}^{1/2}
\end{align*}
from where we further arrive at
\begin{align*}
	N_{1,v2} ^{LH, E}&\lesssim\brak{t}^{3+2\ss} t^{-1} \bold{a}_{n+1}^2 \sum_{k} \sum_{m\ge n/2}^{n}
	\binom{n}{m} \enorm{ q^{m}\Gamma_{k}^{m} ik \overline{\omega_k }e^{W} \chi_n}  \enorm{\bhqn  \chi_n}
	\\&\quad\times
	kt\sum_{\ell=0}^{n-m-1}\binom{n-m}{\ell}\norm{q^{n-m-\ell}\bd^{n-m-\ell}(v_y-1)\chi_{n-2}}_{L^2}^{1/2}
	\\&\quad\times \norm{\partial_y\paren{q^{n-m-\ell}\bd^{n-m-\ell}(v_y-1)\chi_{n-2}}}_{L^2}^{1/2}
	\norm{q^{\ell}\Gamma_k^{\ell}\stf_{k}^{(E)}\chi_{n-1}}_{L^2}^{1/2} \norm{\partial_y\paren{q^{\ell}\Gamma_k^{\ell}\stf_{k}^{(E)}\chi_{n-1}}}_{L^2}^{1/2},
\end{align*}
and hence it follows
\begin{align}\label{est:LH:v2}
	\sum_{n}&(n+1)^{2\sss-2}N_{1,v2} ^{LH, E}\nn
	&\lesssim
	\sum_{n} \sum_{k} \sum_{m\ge n/2}^{n} \bold{a}_{n+1} \binom{n}{m} \bold{a}_{1,m}^{-1} 
	\bold{a}_{n-m-\ell}^{-1}\bold{a}_{\ell}^{-1}
	(n+2)^{\sss-1} \bold{a}_{1,m} k\enorm{ q^{m}\Gamma_{k}^{m} \overline{\omega_k }e^{W} \chi_n} 
	\nn&\quad\times 
	\brak{t}^{3/2+\ss} (n+2)^{\sss-1}\bold{a}_{n+1}\enorm{\bhqn  \chi_n}
	\nn&\quad\times
	kt\brak{t}^{3/2+\ss} t^{-1}\sum_{\ell=0}^{n-m-1} \bold{a}_{n-m-\ell}\binom{n-m}{\ell}\norm{q^{n-m-\ell}\bd^{n-m-\ell}(v_y-1)\chi_{n-2}}_{L^2}^{1/2}
	\nn&\quad\times
	\nu^{1/2}\norm{\partial_y\paren{q^{n-m-\ell}\bd^{n-m-\ell}(v_y-1)\chi_{n-2}}}_{L^2}^{1/2}
	\bold{a}_{\ell}\norm{q^{\ell}\Gamma_k^{\ell}\stf_{k}^{(E)}\chi_{n-1}}_{L^2}^{1/2}  
	\nn&\quad\times
           \nu^{1/2}\nu^{-1}\norm{\partial_y\paren{q^{\ell}\Gamma_k^{\ell}\stf_{k}^{(E)}\chi_{n-1}}}_{L^2}^{1/2}.
\end{align}
We use \eqref{weig:huge},  Lemma~\ref{comb:boun:vari:2}, the elliptic estimate~\eqref{jell:1}-\eqref{jell:2},
and Lemma~\ref{con:no:k} to obtain
\begin{align*}
	\sum_{n}& \theta_{n}^2 (n+1)^{2\sss-2} N_{1,v2} ^{LH, E} 
	\lesssim
	\sqrt{\mathcal{CK}^{\gm}} \norm{\overline{H}}_{Y_{0,-s}}
	\sum_{j \ge 0}\sum_{\ell=0}^{j} 2^{-j(\sss-1)}
	\bold{a}_{j-\ell}\norm{\bd^{j-\ell}H\chi_{j-\ell}}_{L^2}^{1/2}
	\\&\quad\times
	\nu^{1/2}\norm{\partial_y\paren{\bd^{j-\ell}H\chi_{j-l}}}_{L^2}^{1/2}
	\nu^{-1}\bold{a}_{\ell}\norm{q^{\ell}\Gamma_k^{\ell}\stf_{k}^{(E)}\chi_{\ell}}_{L^2}^{1/2}             \nu^{1/2}\norm{\partial_y\paren{q^{\ell}\Gamma_k^{\ell}\stf_{k}^{(E)}\chi_{\ell}}}_{L^2}^{1/2}
	\\&\lesssim
	\sqrt{\mathcal{CK}^{\gm}} \norm{\overline{H}}_{Y_{0,-s}} \paren{\sum_{j=1, 2}\sqrt{\mathcal{J}_{ell}^{(i)}}}
	\norm{H}_{\overline{Y}_{1,0}}^{1/2}\paren{\cd_{H}^{\gm}}^{1/4}
	\\&\lesssim
	\sqrt{\mathcal{CK}^{\gm}}\brak{t}^{-3/2-s^{-1}}\sqrt{\mathcal{E}_{\overline H}^{(\gamma)}} \sqrt{\mathcal{E}_{H}^{(\alpha)}}
	\paren{\sum_{j=1, 2}\sqrt{\mathcal{J}_{ell}^{(i)}}} \paren{\cd_{H}^{\gm}}^{1/4}
	\\&\lesssim \eps^3 \brak{t}^{-3/2-s^{-1}} \sqrt{\mathcal{CK}^{\gm}} \paren{\cd_{H}^{\gm}}^{1/4}.
\end{align*}
where we also used H\"older's inequality. 
Next we turn to $N_{1,v3} ^{LH, E}$ and observe that
\begin{align*}
    V^{LH, E}_3 &\lesssim
     \sum_{\ell=1}^{n-m}\binom{n-m}{\ell}\norm{q^{n-m-\ell+1} \bd^{n-m-\ell+1}(v_y-1)}_{L^\infty}
     \norm{q^{\ell-1}\Gamma_k^{\ell}\stf_{k}^{(E)}\chi_{n-1}}_{L^\infty}
     \\&\lesssim
     \sum_{\ell=1}^{n-m}\binom{n-m}{\ell}\norm{q^{n-m-\ell+1}\bd^{n-m-\ell+1}(v_y-1)}_{L^2}^{1/2}
     \\&\quad\times
          \norm{\partial_y \paren{q^{n-m-\ell+1}\bd^{n-m-\ell+1}(v_y-1)}}_{L^2}^{1/2}
          \\&\quad\times
     \norm{q^{\ell-1}\Gamma_k^{\ell}\stf_{k}^{(E)}\chi_{n-1}}_{L^2}^{1/2}
          \norm{\partial_y \paren{q^{\ell-1}\Gamma_k^{\ell}\stf_{k}^{(E)}\chi_{n-1}}}_{L^2}^{1/2}.
\end{align*}
In view of~\eqref{LH:vi}, we arrive at
\begin{align}
	\sum_{n}& \theta_{n}^2 (n+1)^{2\sss-2} N_{1,v3} ^{LH, E}:=\brak{t}^{3+2\ss} t^{-1} \bold{a}_{n+1}^2 \sum_{k} \sum_{m\ge n/2}^{n} 
	\binom{n}{m} V^{LH, E}_3
	\enorm{ q^{m}\Gamma_{k}^{m} ik \overline{\omega_k }e^{W} \chi_n}  \enorm{\bhqn  \chi_n}
	\nn
	&\lesssim
	\sum_{n} \sum_{k} \sum_{m\ge n/2}^{n} \sum_{\ell=1}^{n-m} \bold{a}_{n+1} \binom{n}{m} \bold{a}_{1,m}^{-1} 
	\bold{a}_{n-m-\ell}^{-1}\bold{a}_{\ell}^{-1}
	(n+2)^{\sss-1} \bold{a}_{1,m} k\enorm{ q^{m}\Gamma_{k}^{m} \overline{\omega_k }e^{W} \chi_n} 
	\nn&\quad\times 
	\brak{t}^{3/2+\ss} (n+2)^{\sss-1}\bold{a}_{n+1}\enorm{\bhqn  \chi_n} \brak{t}^{3/2+\ss} t^{-1} \bold{a}_{n-m-\ell}\binom{n-m}{\ell}
	\nn&\quad\times
	\norm{q^{n-m-\ell+1}\bd^{n-m-\ell+1}(v_y-1)}_{L^2}^{1/2} \nu^{1/2} \nu^{-1/2} \norm{\partial_y \paren{q^{n-m-\ell+1}\bd^{n-m-\ell+1}(v_y-1)}}_{L^2}^{1/2}
		\nn&\quad\times
    \bold{a}_{\ell}  \norm{q^{\ell-1}\Gamma_k^{\ell}\stf_{k}^{(E)}\chi_{n-1}}_{L^2}^{1/2}
	 \norm{\partial_y \paren{q^{\ell-1}\Gamma_k^{\ell}\stf_{k}^{(E)}\chi_{n-1}}}_{L^2}^{1/2}.
\end{align}
Similar as before, we use Lemma~\ref{comb:boun:vari:2}, Young's inequality, and \eqref{weig:huge} to obtain
\begin{align}
	\label{est:LH:v3}
	\sum_{n} \theta_{n}^2 (n+1)^{2\sss-2} N_{1,v3} ^{LH, E} 
    &\lesssim
	\sqrt{\cd^{\gm}} \norm{\overline{H}}_{Y_{0,-s}} \norm{H}_{Y_{0,0}}^{1/2}
	\paren{\cd_{H}^{\gm}}^{1/4}
	  \paren{\sum_{l} \enorm{\bold{a}_{\ell} q^{\ell-1}\Gamma_k^{\ell}\stf_{k}^{(E)}\chi_{l}}^2}^{1/4}
	  \nn&\qquad\times
	  \paren{\sum_{l} \enorm{\bold{a}_{\ell} \partial_y \paren{q^{\ell-1}\Gamma_k^{\ell}\stf_{k}^{(E)}\chi_{l}}}^2}^{1/4}.
\end{align}
By an estimate similar to Lemma~5.1 in~\cite{BHIW24b} and the elliptic estimates, we have 
\begin{align}
	\label{dif:typ:1}
	\sum_{l} \enorm{\bold{a}_{\ell} q^{\ell-1}\Gamma_k^{\ell}\stf_{k}^{(E)}\chi_{l}}^2
	&\lesssim
	\sum_{l=0}^{\infty} \sum_{m\le l-1} C^{l-m} \paren{\frac{l!}{m!}}^{-2\sigma}
	\bold{a}_{m}^2 \enorm{\partial_y\paren{q^{m}\Gamma_k^{m}\stf_{k}^{(E)}}\chi_{m}}^2
	\nn&
	\lesssim
\mathcal{J}_{ell}^{(1)},
\end{align}
where we used Fubini in the last step. Similarly, we may have
\begin{align*}
    \sum_{l} \enorm{\bold{a}_{\ell} \partial_y \paren{q^{\ell-1}\Gamma_k^{\ell}\stf_{k}^{(E)}\chi_{l}}}^2
	\lesssim
   \mathcal{J}_{ell}^{(2)}.
\end{align*}
Thus, we may obtain from \eqref{est:LH:v3} 
\begin{align*}
	\sum_{n} \theta_{n}^2 (n+1)^{2\sss-2} N_{1,v3} ^{LH} 
	&\lesssim
	\sqrt{\cd^{\gm}} \norm{\overline{H}}_{Y_{0,-s}} \norm{H}_{Y_{1, 0}}^{1/2}
	\paren{\cd_{H}^{\gm}}^{1/4} \paren{\sum_{j=1, 2}\sqrt{\mathcal{J}_{ell}^{(i)}}}
    \\&\lesssim
    \sqrt{\cd^{\gm}} \brak{t}^{-3/2-s^{-1}}\sqrt{\mathcal{E}_{\overline H}^{(\gamma)}} \sqrt{\mathcal{E}_{H}^{(\alpha)}}
    \paren{\cd_{H}^{\gm}}^{1/4} \paren{\sum_{j=1, 2}\sqrt{\mathcal{J}_{ell}^{(i)}}}
    \\&\lesssim
 \eps^3 \brak{t}^{-3/2-s^{-1}} \sqrt{\cd^{\gm}}
\paren{\cd_{H}^{\gm}}^{1/4}.
\end{align*}
For the last piece, we have
\begin{align*}
    V^{LH, E}_4 &= 
	(n-m)\norm{q^{n-m-1}\Gamma_k^{n-m}\stf_{k}^{(E)}\chi_{n-1}}_{L^\infty}
	\\&\lesssim
	(n-m)\norm{q^{n-m-1}\Gamma_k^{n-m}\stf_{k}^{(E)}\chi_{n-1}}_{L^2}^{1/2}
	\norm{\partial_y\paren{q^{n-m-1}\Gamma_k^{n-m}\stf_{k}^{(E)}\chi_{n-1}}}_{L^2}^{1/2}.
\end{align*}
Actually the treatment of $N_{1,v4} ^{LH, E}$ is just a part  of that of $N_{1,v3} ^{LH}$ and we omit further details.
Next we deal with the $HL$ part as
\begin{align*}
	\abs{N_{1} ^{HL, E}} 
	&
	=  \brak{t}^{3+2\ss} t^{-1} \bold{a}_{n+1}^2 \sum_{k}\sum_{m\le n/2}
	\binom{n}{m} \Big \langle  q^{n-m}\Gamma_{k}^{n-m}\partial_y \stf_{k}^{(E)} q^{m}\Gamma_{k}^{m} ik \overline{\omega_k } ,   \bhqn  e^{W} \chi_n^2 \Big \rangle 
	\\&\lesssim
	\brak{t}^{3+2\ss} t^{-1} \bold{a}_{n+1}^2 \sum_{k}\sum_{m\le n/2} \binom{n}{m} \norm{q^{n-m}\Gamma_{k}^{n-m} k\partial_y \stf_{k}^{(E)}\chi_{n}}_{L^2} \norm{q^{m}\Gamma_{k}^{m} \overline{\omega_k }e^{W} \chi_{n-1} }_{L^\infty}\enorm{\bhqn  \chi_n}
	\\&\lesssim
	\brak{t}^{3+2\ss} t^{-1} \bold{a}_{n+1}^2 \sum_{k}\sum_{m\le n/2} \binom{n}{m} \norm{q^{n-m}\Gamma_{k}^{n-m} k\partial_y \stf_{k}^{(E)}\chi_{n}}_{L^2} \enorm{\bhqn  \chi_n}
	\\&\quad\times
	\norm{q^{m}\Gamma_{k}^{m} \overline{\omega_k }e^{W} \chi_{n-1} }_{L^2}^{1/2}\norm{\partial_y\paren{q^{m}\Gamma_{k}^{m} \overline{\omega_k }e^{W} \chi_{n-1}} }_{L^2}^{1/2}.
\end{align*}
In analogue to~\eqref{com:y:D}, we have
\begin{align*}
	&\norm{q^{n-m}\Gamma_{k}^{n-m}\partial_y \stf_{k}^{(E)}\chi_{n-1}}_{L^2}
		\nn&\qquad \lesssim
	\norm{\partial_y \paren{q^{n-m}\Gamma_{k}^{n-m} \stf_{k}^{(E)}}\chi_{n-1}}_{L^2}
	\nn&\qquad\quad	+ 
	kt\sum_{\ell=0}^{n-m-1}\binom{n-m}{\ell}\norm{q^{n-m-\ell}\bd^{n-m-\ell}(v_y-1)q^{\ell}\Gamma_k^{\ell}\stf_{k}^{(E)}\chi_{n-1}}_{L^2}
	\nn&\qquad\quad
	+\sum_{\ell=1}^{n}\binom{n-m}{\ell}\norm{\bd^{n-m-\ell+1}(v_y-1)q^{\ell-1}\Gamma_k^{\ell}\stf_{k}^{(E)}\chi_{n-1}}_{L^2}
	+ 	
	(n-m)\norm{q^{n-m-1}\Gamma_k^{n-m}\stf_{k}^{(E)}\chi_{n-1}}_{L^2}
	\nn \\&\qquad
	=V^{HL, E}_1+V^{HL, E}_2+V^{HL, E}_3+V^{HL, E}_4.
\end{align*}
From here on, the estimates of $N_{1} ^{HL, E}$ and $N_{1} ^{LH, E}$ are completely parallel and in a similar fashion. Hence, we 
only sample the third piece out of the four to give the readers a flavor and omit further details. 
Denote
\begin{align*}
	N_{1,v3} ^{HL, E}&:=\brak{t}^{3+2\ss} t^{-1} \bold{a}_{n+1}^2 \sum_{k} \sum_{m\le n/2}\binom{n}{m} V^{HL, E}_3
	  \enorm{\bhqn  \chi_n}
	\norm{q^{m}\Gamma_{k}^{m} \overline{\omega_k }e^{W} \chi_{n-1} }_{L^2}^{1/2}
	\\&\qquad\times
	\norm{\partial_y\paren{q^{m}\Gamma_{k}^{m} \overline{\omega_k }e^{W} \chi_{n-1}} }_{L^2}^{1/2}.
\end{align*}
Then by~\eqref{jell:1} and \eqref{jell:2}, analogously to~\eqref{est:LH:v3}, we deduce
\begin{align*}
	\sum_{n}& \theta_{n}^2 (n+1)^{2\sss-2} N_{1,v3} ^{HL, E}
	\\
	&\lesssim
	\sum_{n} \sum_{k} \sum_{m\le n/2} \sum_{\ell=1}^{n-m} \bold{a}_{n+1}\binom{n}{m} \bold{a}_{1,m}^{-1} 
	\bold{a}_{\alpha}^{-1}\bold{a}_{\beta}^{-1}
    \bold{a}_{1,m} k\enorm{ q^{m}\Gamma_{k}^{m} \overline{\omega_k }e^{W} \chi_{m}\widetilde\chi_{ 1}} ^{1/2}
	\\&\quad\times 
	\norm{\partial_y\paren{q^m\Gamma_k^{m} \overline{\omega_k }e^{W} \chi_{n-1}} }_{L^2}^{1/2}
	\brak{t}^{3/2+\ss} (n+2)^{\sss-1}\bold{a}_{n+1}\enorm{\bhqn  \chi_n}
	\\&\quad\times
	\brak{t}^{3/2+\ss} t^{-1} (n+2)^{\sss-1} \bold{a}_{\alpha}\binom{n-m}{\ell}
	\norm{\bd^{n-m-\ell+1}(v_y-1)\chi_{n-m-\ell+1}}_{L^2}
	\\&\quad\times
	\bold{a}_{\beta}  \norm{q^{\ell-1}\Gamma_k^{\ell}\stf_{k}^{(E)}\chi_{\ell}}_{L^2}^{1/2}
	\norm{\partial_y \paren{q^{\ell-1}\Gamma_k^{\ell}\stf_{k}^{(E)}\chi_{\ell}}}_{L^2}^{1/2}
	\\&
	\lesssim
	\paren{ \mathcal{CK}^{\gm}}^{1/4} \paren{\cd^{\gm}}^{1/4} 
	\norm{\overline{H}}_{Y_{0,-s}} \norm{{H}}_{Y_{0,0}} 
\paren{\sum_{j=1, 2}\sqrt{\mathcal{J}_{ell}^{(i)}}}
    \\&
    \lesssim
    \eps^3 \paren{ \mathcal{CK}^{\gm} + \cd^{\gm} 
    	 + 
  \brak{t}^{-3-2s^{-1}}
}
\end{align*}
where we choose $\alpha=n-m-\ell + 2, \beta = \ell-2$ if $n-m-\ell \ge \ell$, and $\alpha=n-m-\ell-1, \beta = \ell+1$ otherwise. 
Next we consider the term
\begin{align*}
	 N_{2}^E =\frac{1}{t} \sum_{k}\sum_{m=0}^{n} \binom{n}{m}\Big \langle  q^{n-m}\Gamma_{k}^{n-m}ik \stf_{k}^{(E)} q^{m}\Gamma_{k}^{m} \partial_y \overline{\omega_k} , \brak{t}^{3+2\ss}  \bold{a}_{n+1}^2  \bhqn  e^{W} \chi_n^2 \Big \rangle 
\end{align*}
in a similar manner and only show details for one sub-term
\begin{align*}
		N_{2,w2} ^{LH, E}:=\brak{t}^{3+2\ss} t^{-1} \bold{a}_{n+1}^2 \sum_{k} \sum_{m\ge n/2}^{n} 
		\norm{q^{n-m}\Gamma_{k}^{n-m}k \stf_{k}^{(E)}}_{L^\infty}\Omega^{LH, E}_2
        \enorm{\bhqn  e^{W/2} \chi_n}
\end{align*}
where
\begin{align*}
     \Omega^{LH, E}_2 = 	
	kt\sum_{\ell=0}^{m-1}\binom{m}{\ell}\norm{\bd^{m-\ell}(v_y-1)q^{\ell}\Gamma_k^{\ell}\omega_{k}e^{W/2}\chi_{n-1}}_{L^2}.
\end{align*}
By Lemma~\ref{comb:boun:vari:2}, the elliptic estimate~\eqref{jell:1}-\eqref{jell:2},
Lemma~\ref{con:no:k}, and $
\brak{t}^2\nu^{-1} \lesssim e^W\ \mbox{for}\ t\lesssim \nu^{-1/3},$ we get
\begin{align*}
	\sum_{n}& \theta_{n}^2 (n+1)^{2\sss-2} N_{2,w2} ^{LH, E}\nn
	&\lesssim
	\sum_{n} \sum_{k} \sum_{m\ge n/2}^{n}\sum_{\ell=0}^{m-1} \bold{a}_{n+1} \binom{n}{m} \bold{a}_{1,n-m}^{-1} 
	\bold{a}_{m-\ell}^{-1}\bold{a}_{\ell}^{-1}
	(n+2)^{\sss-1} \bold{a}_{1,n-m} k\enorm{ q^{n-m}\Gamma_{k}^{n-m}  \stf_k^{(E)} \chi_{n-m+1}}^{1/2} 
	\nn&\quad\times 
		\enorm{\partial_y \paren{q^{n-m}\Gamma_{k}^{n-m}  \stf_k^{(E)}  \chi_{n-m+1}}} ^{1/2}
	\brak{t}^{3/2+\ss} (n+2)^{\sss-1}\bold{a}_{n+1}\enorm{\bhqn  e^{W/2} \chi_n}
	\nn&\quad\times
	kt\brak{t}^{3/2+\ss} t^{-1} \bold{a}_{m-\ell}\binom{m}{\ell}\norm{\bd^{m-\ell}(v_y-1)e^{W/2}\chi_{m-\ell}}_{L^2}
	\nn&\quad\times
	\bold{a}_{\ell}\norm{q^{\ell}\Gamma_k^{\ell}\overline\omega_{k}\chi_{\ell+1}}_{L^2}^{1/2} \norm{\partial_y\paren{q^{\ell}\Gamma_k^{\ell}\overline\omega_{k}\chi_{\ell+1}}}_{L^2}^{1/2}
	\nn&
	\lesssim
\paren{\sum_{j=1, 2}\sqrt{\mathcal{J}_{ell}^{(i)}}}
     	\norm{\overline{H}}_{Y_{0,-s}} \norm{H}_{Y_{0,0}}\paren{\sqrt{\cd^{\gm}} + 
     	\sqrt{\mathcal{CK}^{\gm}}}
     \nn&
	\lesssim \eps^3 \paren{\cd^{\gm} + \mathcal{CK}^{\gm} + \brak{t}^{-3-2s^{-1}}}
\end{align*}
completing the treatment of the exterior part. 
\\{\bf Interior part:} For the interior part, we need slightly different argument, since there is no localization available but instead, fast decay of the 
stream function. We note
\begin{align*}
	N^I &
	= -\frac{1}{t} \sum_{k} \Big \langle  
	q^n\Gamma_{0}^n \paren{ \Gamma_k \stf_{k}^{(I)}\overline{ ik \omega_k } }, \brak{t}^{3+2\ss}  \bold{a}_{n+1}^2  \bhqn  e^{W}\chi_n^2 \Big \rangle 
	\\&\quad
	+\frac{1}{t} \sum_{k} \Big \langle   q^n\Gamma_{0}^n \paren{ik \stf_{k}^{(I)} \Gamma_k \overline{\omega_k}},  \brak{t}^{3+2\ss}  \bold{a}_{n+1}^2  \bhqn  e^{W}\chi_n^2 \Big \rangle 
	\\&
	=  \frac{1}{t} \sum_{k}\sum_{m=0}^{n}\binom{n}{m}\Big \langle  q^{n-m}\Gamma_{k}^{n-m}\Gamma_k \stf_{k}^{(I)} q^{m}\Gamma_{k}^{m} ik \overline{\omega_k } , \brak{t}^{3+2\ss}  \bold{a}_{n+1}^2  \bhqn  e^{W}\chi_n^2 \Big \rangle 
	\\&\quad
	+\frac{1}{t} \sum_{k}\sum_{m=0}^{n}\binom{n}{m}\Big \langle  q^{n-m}\Gamma_{k}^{n-m}ik \stf_{k}^{(I)} q^{m}\Gamma_{k}^{m} \Gamma_k \overline{\omega_k} , \brak{t}^{3+2\ss}  \bold{a}_{n+1}^2  \bhqn  e^{W}\chi_n^2 \Big \rangle 
	\\&
	= N_{1}^I + N_{2}^I.
\end{align*}
As in the exterior case, we divide $N_{1}^I $ into two pieces:
\begin{align*}
	N_{1}^{I} &=  \brak{t}^{3+2\ss} t^{-1} \bold{a}_{n+1}^2 \sum_{k}\paren{\sum_{m\le n/2} + \sum_{m\ge n/2}^{n}}
	\binom{n}{m} \Big \langle  q^{n-m}\Gamma_{k}^{n-m}\Gamma_k \stf_{k}^{(I)} q^{m}\Gamma_{k}^{m} ik \overline{\omega_k } ,   \bhqn  e^{W}\chi_n^2 \Big \rangle 
	\\& = N_{1} ^{HL, I} + N_{1} ^{LH, I}.
\end{align*}
For the second piece, by~\eqref{Sob:emb}, we have
\begin{align*}
	N_{1} ^{LH, I} &\lesssim \brak{t}^{3+2\ss} t^{-1} \bold{a}_{n+1}^2 \sum_{k} \sum_{m\ge n/2}^{n}
	\binom{n}{m}  \norm{ q^{n-m} \Gamma_k^{n-m+1}\stf_{k}^{(I)} \chi_{n-m+1}}_{L^2} ^{1/2}
	\\&\quad\times
	\norm{\Gamma_k \paren{\Gamma_k^{n-m+1} \stf_{k}^{(I)}}\chi_{n-m+1}}_{L^2}^{1/2}
	\enorm{ q^{m}\Gamma_{k}^{m} ik \overline{\omega_k }e^{W} \chi_n}  \enorm{\bhqn  \chi_n} 
	\\&\quad+
	\brak{t}^{3+2\ss} t^{-1} \bold{a}_{n+1}^2 \sum_{k} \sum_{m\ge n/2}^{n} \binom{n}{m}  
	(n-m)^{(1+\sigma)/2} \norm{ q^{n-m}\Gamma_{k}^{n-m}\Gamma_k \stf_{k}^{(I)}\chi_{n-m}}_{L^2}
		\\&\quad\times
	\enorm{ q^{m}\Gamma_{k}^{m} ik \overline{\omega_k }e^{W} \chi_n}  \enorm{\bhqn  \chi_n} .
\end{align*}
Multiplying it by $(n+1)^{2\sss-2}$ and summing in $n$ gives
\begin{align*}
	\sum_{n}& \theta_{n}^2 (n+1)^{2\sss-2} N_{1} ^{LH, I}\\
	&\lesssim
	\sum_{n\ge0}\sum_{m\ge n/2}^{n}\sum_{k}  \bold{a}_{n+1} \binom{n}{m} \bold{a}_{n-m}^{-1} \bold{a}_{1,m}^{-1} 
	\\&\quad \times \brak{t}^{3/2+\ss}  \bold{a}_{n-m} \norm{ \Gamma_k^{n-m+1}\stf_{k}^{(I)} \chi_{n-m+1}}_{L^2} ^{1/2}
	\norm{\Gamma_k \paren{\Gamma_k^{n-m+1} \stf_{k}^{(I)}}\chi_{n-m+1}}_{L^2}^{1/2}
	\\&\quad \times
	(n+2)^{\sss-3/2}\bold{a}_{1,m}t^{-1/2} \enorm{ k q^{m}\Gamma_{k}^{m} \overline{\omega_k }e^{W} \chi_m}  
	\brak{t}^{3/2+\ss} t^{-1/2}(n+2)^{\sss-1/2}\bold{a}_{n+1}\enorm{\bhqn  \chi_n}
	\\&\quad + 
	\sum_{n\ge0}\sum_{m\ge n/2}^{n}\sum_{k}  \bold{a}_{n+1} \binom{n}{m} \bold{a}_{n-m}^{-1} \bold{a}_{1,m}^{-1}
		\\&\quad\times 
		\brak{t}^{3/2+\ss}
	\bold{a}_{n-m} 
	 (n-m)^{(1+\sigma)/2} \norm{ q^{n-m}\Gamma_{k}^{n-m}\Gamma_k \stf_{k}^{(I)}\chi_{n-m}}_{L^2}
	\\&\quad \times
	(n+2)^{\sss-3/2}\bold{a}_{1,m} t^{-1/2} \enorm{ k q^{m}\Gamma_{k}^{m} \overline{\omega_k }e^{W} \chi_m}  
	\brak{t}^{3/2+\ss} t^{-1/2}(n+2)^{\sss-1/2}\bold{a}_{n+1}\enorm{\bhqn  \chi_n} .
\end{align*}
By Corollary~\ref{comb:boun:vari:1}, the elliptic estimate~\eqref{eell:out}, and Lemma~\ref{con:no:k},
the above quantity could be easily bounded by 
\begin{align*}
	\sum_{n}& \theta_{n}^2 (n+1)^{2\sss-2}  N_{1} ^{LH, I} 
	\lesssim 
	\sqrt{\mathcal{E}_{ell}^{(I, out)}} 
	\sqrt{\mathcal{CK}^{\gm}} \sqrt{\mathcal{CK}_{\overline{H}}^{\gm}}
	\lesssim \eps 
	\sqrt{\mathcal{CK}^{\gm}} \sqrt{\mathcal{CK}_{\overline{H}}^{\gm}} .
\end{align*}
Next, for the $HL$ part, we similarly deduce
\begin{align*}
	\sum_{n}& \theta_{n}^2 (n+1)^{2\sss-2} N_{1} ^{HL, I}\\
	&\lesssim
	\sum_{n\ge0}\sum_{m\le n/2}\sum_{k}  \bold{a}_{n+1} \binom{n}{m} \bold{a}_{n-m+2}^{-1} \bold{a}_{1,m-2}^{-1}  (n+2)^{\sss-3/2}
	\\&\quad \times \brak{t}^{3/2+\ss}  \bold{a}_{n-m+2} \norm{ \Gamma_k^{n-m+1}\stf_{k}^{(I)} \chi_{n-m+1}}_{L^2} ^{1/2}
	\norm{\Gamma_k\paren{\Gamma_k^{n-m+1} \stf_{k}^{(I)}}\chi_{n-m+1}}_{L^2}^{1/2}
	\\&\quad \times
	\bold{a}_{1,m-2}t^{-1/2} \enorm{ k q^{m}\Gamma_{k}^{m} \overline{\omega_k }e^{W} \chi_m}  
	\brak{t}^{3/2+\ss} t^{-1/2}(n+2)^{\sss-1/2}\bold{a}_{n+1}\enorm{\bhqn  \chi_n}
	\\&\quad + 
	\sum_{n\ge0}\sum_{m\ge n/2}^{n}\sum_{k}  \bold{a}_{n+1} \binom{n}{m} \bold{a}_{n-m+1}^{-1} \bold{a}_{1,m-1}^{-1}
	\brak{t}^{3/2+\ss}
	\bold{a}_{n-m+1} (n-m)^{(1+\sigma)/2}
	\\&\quad \times
	 \norm{ q^{n-m}\Gamma_{k}^{n-m}\Gamma_k \stf_{k}^{(I)}\chi_{n-m}}_{L^2}
	(n+2)^{\sss-3/2}\bold{a}_{1,m-1} t^{-1/2} \enorm{ k q^{m}\Gamma_{k}^{m} \overline{\omega_k }e^{W} \chi_m}  
	\\&\quad \times 
	\brak{t}^{3/2+\ss} t^{-1/2}(n+2)^{\sss-1/2}\bold{a}_{n+1}\enorm{\bhqn  \chi_n},
\end{align*}
from where we again arrive at
\begin{align*}
	\sum_{n}& \theta_{n}^2 (n+1)^{2\sss-2}  N_{1} ^{HL, I} 
	\lesssim \sqrt{\mathcal{E}_{ell}^{(I, out)}} 
	\sqrt{\mathcal{CK}^{\gm}} \sqrt{\mathcal{CK}_{\overline{H}}^{\gm}}
	\lesssim \eps
	\sqrt{\mathcal{CK}^{\gm}} \sqrt{\mathcal{CK}_{\overline{H}}^{\gm}}.
\end{align*}
For the second piece $N_{2}^I$ we have
\begin{align*}
		N_{2}^I &= 
		\frac{1}{t} \sum_{k}\sum_{m=0}^{n}\binom{n}{m}\Big \langle  q^{n-m}\Gamma_{k}^{n-m}ik \stf_{k}^{(I)} q^{m}\Gamma_{k}^{m} \Gamma_k \overline{\omega_k} , 
		\brak{t}^{3+2\ss}  \bold{a}_{n+1}^2  \bhqn  e^{W}\chi_n^2 \Big \rangle 
	\\& =
	\frac{1}{t} \sum_{k}    	\paren{\sum_{m\le n/2} + \sum_{m\ge n/2}^{n}}  \binom{n}{m}\Big \langle  q^{n-m}\Gamma_{k}^{n-m}ik \stf_{k}^{(I)} q^{m}\Gamma_{k}^{m} \Gamma_k \overline{\omega_k} , 
	\brak{t}^{3+2\ss}  \bold{a}_{n+1}^2  \bhqn  e^{W}\chi_n^2 \Big \rangle 
	\\& = N_{2}^{HL, I} +N_{2}^{LH, I} 
\end{align*}
of which in the above only the second term needs new estimates, which we show below:
\begin{align*}
	N_{2}^{LH, I}  =  	\frac{1}{t} \sum_{k}    	\sum_{n\ge m\ge n/2}  \binom{n}{m}\Big \langle  q^{n-m}\Gamma_{k}^{n-m}ik \stf_{k}^{(I)} q^{m}\Gamma_{k}^{m} \Gamma_k \overline{\omega_k} , 
	\brak{t}^{3+2\ss}  \bold{a}_{n+1}^2  \bhqn  e^{W}\chi_n^2 \Big \rangle .
\end{align*}
When $m<n$, the  estimates are relatively straightforward 
\begin{align*}
	\sum_{n}& \theta_{n}^2 (n+1)^{2\sss-2}	\frac{1}{t} \sum_{k}	\sum_{n > m\ge n/2}  \binom{n}{m}\Big \langle  q^{n-m}\Gamma_{k}^{n-m}ik \stf_{k}^{(I)} q^{m}\Gamma_{k}^{m} \Gamma_k \overline{\omega_k} , 
	\brak{t}^{3+2\ss}  \bold{a}_{n+1}^2  \bhqn  e^{W}\chi_n^2 \Big \rangle 
	\\&
	\lesssim \sum_{n} \theta_{n}^2 (n+1)^{2\sss-2}
   \frac{1}{t} \brak{t}^{3+2\ss} \sum_{k}  \sum_{n > m\ge n/2}  \binom{n}{m}   \bold{a}_{n+1}^2  
    \norm{q^{n-m-1}\Gamma_k^{n-m-1}ik \Gamma_k \stf_{k}^{(I)}\chi_{n-m}}_{L^\infty} 
    \\&\qquad \times \enorm{q^{m+1}\Gamma_k^{m+1}  \overline{\omega_k} e^{W} \chi_n}
   \enorm{\bhqn  \chi_n }
      \\&\lesssim 
      \sqrt{\mathcal{E}_{ell}^{(I, out)}} \sqrt{\mathcal{E}^{(\gamma)}} \sqrt{\mathcal{E}_{H}^{(\gamma)}}
   \lesssim \frac{\eps^3}{\brak{t}^{90}}.
\end{align*}
While $m=n$
\begin{align*}
	\sum_{n}& \theta_{n}^2 (n+1)^{2\sss-2} \frac{1}{t} \sum_{k}    \Big \langle  ik \stf_{k}^{(I)} q^n \Gamma_k^{n} \Gamma_k \overline{\omega_k} , 
	\brak{t}^{3+2\ss}  \bold{a}_{n+1}^2  \bhqn  \chi_n^2 \Big \rangle 
	\\&
	\lesssim \sum_{n} \theta_{n}^2 (n+1)^{2\sss-2}
	\frac{1}{t} \brak{t}^{3+2\ss} \sum_{k}  \sum_{n > m\ge n/2}  \binom{n}{m}   \bold{a}_{n+1}^2  
	\norm{ik \stf_{k}^{(I)}\chi_{n-1}}_{L^\infty} 
	\\&\qquad \times \enorm{q^n\Gamma_k^{n} \Gamma_k \overline{\omega_k} e^{W} \chi_n}
	\enorm{\bhqn  \chi_n }
	\\&
	\lesssim 
	\sqrt{\mathcal{E}_{ell}^{(I, out)}} \sqrt{\mathcal{D}^{(\gamma)}} \sqrt{\mathcal{E}_{H}^{(\gamma)}}
	\lesssim \frac{\eps^2}{\brak{t}^{90}}\sqrt{\cd^{\gm}} \lesssim 
        \frac{\eps^3}{\brak{t}^{180}} + \eps \cd^{\gm}
\end{align*}
concluding the proof for $n\ge2$.\\
{\bf (2). Case $n=0$:} At last we consider the case $n=0$, where we note that
\begin{align*}
	&N^{(0)}_{\overline{H}} 
	= 
	\Big \langle \frac{1}{t} \paren{  \nabla^{\perp} \stf_{\neq 0} \cdot \nabla \omega }_0, \brak{t}^{3+2\ss}  \bold{a}_{1}^2  \overline{H} e^{W}  \Big \rangle .
\end{align*}
The main difference between the high frequency part with the low ones is that the localization property is not 
guaranteed over all the domain of the integral, but only near the boundary. Therefore, we may further divide the term $N_1$ into two parts: the inner and outer part,
\begin{align*}
N^{(0)}_{\overline{H}}  &= 
	\Big \langle \frac{1}{t} \paren{  \nabla^{\perp} \stf_{\neq 0} \cdot \nabla \omega }_0, \brak{t}^{3+2\ss}  \bold{a}_{1}^2  \overline{H} e^{W} \chi_2^2 \Big \rangle 
      + 
		\Big \langle \frac{1}{t} \paren{  \nabla^{\perp} \stf_{\neq 0} \cdot \nabla \omega }_0, \brak{t}^{3+2\ss}  \bold{a}_{1}^2  \overline{H} e^{W} (1- \chi_2^2 )\Big \rangle 
	\\&=: N^{E, 0} + N^{I, 0}.
\end{align*}
For the exterior part, the treatment is similar to the case $n\ge 2$ and we only focus on the bound of $N_1^{I,0}.$ Noting
\begin{align*}
	N^{I, 0}  &= \Big \langle \frac{1}{t} \paren{  \nabla^{\perp} \stf_{\neq 0} \cdot \nabla \omega }_0, \brak{t}^{3+2\ss}  \bold{a}_{1}^2  \overline{H} e^{W} (1- \chi_2^2 )\Big \rangle 
	\\&
	= 
	-\sum_{k} \Big \langle \frac{1}{t}   \Gamma_k \stf_{k}  \overline {ik  \omega_k} , \brak{t}^{3+2\ss}  \bold{a}_{1}^2  \overline{H} e^{W} (1- \chi_2^2 )\Big \rangle 
	+
	\sum_{k} \Big \langle \frac{1}{t}  ik \stf_{k}  \Gamma_k \omega_k , \brak{t}^{3+2\ss}  \bold{a}_{1}^2  \overline{H} e^{W} (1- \chi_2^2 )\Big \rangle 
	\\&=
	N_{1}^{I, 0} + N_{2}^{I, 0} 
\end{align*}
of which we only show details for the first term since the estimate of the second one is completely parallel.
We proceed as
\begin{align*}
	|N_{1}^{I, 0}| & \le \bold{a}_{1}^2 \brak{t}
	\sum_{k} \brak{t}^{2\ss-r} \norm{\Gamma_k \stf_{k}(1-\chi_{ 3})}_{L^\infty} \norm{k \omega_k e^W\sqrt{1- \chi_2^2}}_{L^2}
	   \brak{t}^{1+r}\enorm{\overline{H} \sqrt{1- \chi_2^2}}
	   \\& \lesssim 
	   \bold{a}_{1}^2\brak{t}
	   \sum_{k} \brak{t}^{2\ss-r} \enorm{\Gamma_k \stf_{k}(1-\chi_{ 3})}^{1/2}  \enorm{\Gamma_k\paren{\Gamma_k \stf_{k}(1-\chi_{ 3})}}^{1/2}  \norm{k \omega_k e^W\sqrt{1- \chi_2^2}}_{L^2}
	   \\&\qquad \times 
	   \brak{t}^{1+r}\enorm{\overline{H} \sqrt{1- \chi_2^2}}
\end{align*}
where $r$ is the radius of Gevrey regularity in the bulk flow. 
By the elliptic estimate~\eqref{eell:sob} and \eqref{fell:e}, 
one easily gets
\begin{align*}
	\sum_{k} \enorm{\Gamma_k \stf_{k}(1-\chi_{ 3})} ^2
	&\le \sum_{k}\enorm{\Gamma_k \stf_{k}^{(I)}(1-\chi_{ 3})}^2 + \sum_{k} \enorm{\Gamma_k \stf_{k}^{(E)}(1-\chi_{ 3})} ^2
	\\& \lesssim  \mathcal{E}_{ell}^{(I, full)} + \brak{t}^2\mathcal{J}_{ell}^{(1)}
	{\lesssim \frac{\eps^2}{\brak{t}^4} + \eps^2e^{-\nu^{-1/10}}}
	\\
	\sum_{k} \enorm{\Gamma_k\paren{\Gamma_k \stf_{k}(1-\chi_{ 3})}} ^2& \lesssim {\frac{\eps^2}{\brak{t}^4}}
\end{align*}
for $t\lesssim \nu^{-1/3-\delta}$. Therefore, using H\"older's inequality and the bootstrap assumptions, we obtain
\begin{align*}
	|N_{1}^{I, 0}| &\lesssim \sqrt{\mathcal{E}_{ell}^{(I, full)} +\brak{t}^2 \mathcal{J}_{ell}^{(1)}}  \sqrt{\mathcal{E}_{cloud}} \sqrt{\mathcal{E}^{(\overline h)}_{\mathrm{Int, Coord}}}  \frac{\brak{t}^{2\ss}}{\brak{t}^{1+r}} 
	  \lesssim 
	  \frac{\eps^3}{\brak{t}^{3+r-2\ss}}.
\end{align*}
We turn to $N_{2}^{I, 0} $ next. Similarly we have
\begin{align*}
	|N_{2}^{I, 0}| & \le \frac{1}{t}\bold{a}_{1}^2  
	\sum_{k} \brak{t}^{3+2\ss} \norm{ k \stf_{k}(1-\chi_{ 3})}_{L^\infty} \norm{ \Gamma_k \omega_k e^W\chi^I_{\omega}}_{L^2}
    \enorm{\overline{H} (1- \chi_2^2 ) }
	\\& \lesssim \frac{1}{t}
	\bold{a}_{1}^2  
	\sum_{k} \brak{t}^{3+2\ss} \enorm{\Gamma_k \stf_{k}(1-\chi_{ 3})}^{1/2}  \enorm{\Gamma_k\paren{k \stf_{k}(1-\chi_{ 3})}}^{1/2}  \norm{\Gamma_k \omega_k e^W\chi^I_{\omega}}_{L^2}
	\\&\quad \times \enorm{\overline{H} (1- \chi_2^2 )}
	\\&
	\lesssim \frac{\brak{t}^{2\ss}}{\brak{t}^{1+r}} \sqrt{\mathcal{E}_{ell}^{(I, full)}}  \sqrt{\mathcal{E}_{cloud}} \sqrt{\mathcal{E}^{(\overline h)}_{\mathrm{Int, Coord}}}
	\lesssim 
	\frac{\eps^3}{\brak{t}^{3+r-2\ss}}.
\end{align*}
{\bf (3). Case $n=1$:} We note 
\begin{align*}
	&N^{(1)}_{\overline{H}} 
	= 
	\Big \langle q\Gamma_0 \frac{1}{t} \paren{  \nabla^{\perp} \stf_{\neq 0} \cdot \nabla \omega }_0, \brak{t}^{3+2\ss}  \bold{a}_{2}^2  \bhqnn  e^{W}\chi_1^2 \Big \rangle .
\end{align*}
Like in the case $n=0$, we split the term $N^{(1)}_{\overline{H}}$ into interior and exterior part, focusing on the first one
\begin{align*}
	N_{\overline H}^{(I, 1)} &= 
		\Big \langle q\Gamma_0 \frac{1}{t} \paren{  \nabla^{\perp} \stf_{\neq 0} \cdot \nabla \omega }_0, 
		\brak{t}^{3+2\ss}  \bold{a}_{2}^2  \bhqnn  e^{W} \chi_1^2 (1-\chi_2^2)\Big \rangle 
		\\& = \sum_{k}
		-\Big \langle q\Gamma_0 \paren{\frac{1}{t}  \Gamma_k \stf_{k} \overline { ik \omega_k} }, 
		\brak{t}^{3+2\ss}  \bold{a}_{2}^2  \bhqnn  e^{W} \chi_1^2 (1-\chi_2^2)\Big \rangle 
		\\&\qquad + \sum_{k}
		\Big \langle q\Gamma_0 \frac{1}{t} \paren{ ik \stf_{k} \Gamma_k \omega_k }, 
		\brak{t}^{3+2\ss}  \bold{a}_{2}^2  \bhqnn  e^{W} \chi_1^2 (1-\chi_2^2)\Big \rangle 
		\\&
		= 	N_{1}^{(I, 1)} + N_{2}^{(I, 1)} .
\end{align*}
For the first piece $N_{1}^{(I, 1)}$, we use Leibniz's rule to obtain
\begin{align*}
	N_{1}^{(I, 1)}  = &
	-\frac{1}{t} 	\sum_{k}\Big \langle q\Gamma_0   \Gamma_k \stf_{k} \overline {ik \omega_k }, 
	\brak{t}^{3+2\ss}  \bold{a}_{2}^2  \bhqnn  e^{W} \chi_1^2 (1-\chi_2^2)\Big \rangle 
	\\&\quad
	-\frac{1}{t} 	\sum_{k}\Big \langle  \Gamma_k \stf_{k} \overline { ik q\Gamma_0  \omega_k} , 
	\brak{t}^{3+2\ss}  \bold{a}_{2}^2  \bhqnn  e^{W} \chi_1^2 (1-\chi_2^2)\Big \rangle 
	\\=&	N_{11}^{(I, 1)} + 	N_{12}^{(I, 1)} .
\end{align*}
Using $\nu^{-100} \lesssim e^{W/2}$, $\Gamma_k = \Gamma_0 + ikt$, by H\"older's inequality and Sobolev embedding, we arrive at
\begin{align*}
	|N_{11}^{(I, 1)}| &\le 	\frac{1}{t} 	\brak{t}^{3+2\ss}  \bold{a}_{2}^2
	\sum_{k}     \norm{q\Gamma_0   \Gamma_k \stf_{k}(1-\chi_2^2) }_{L^\infty} 
	\enorm{ik \omega_k e^{W} \chi_1}
    \enorm{ \bhqnn  \chi_1} 	
    \\& \lesssim \nu^{100}\brak{t}^{2\ss} \paren{\sqrt{\mathcal{E}_{ell}^{(I, full)}} + \sqrt{ \brak{t}^2  \mathcal{J}_{ell}^{(1)}}} \mathcal{E}^{(\gamma)} \sqrt{\mathcal{E}_{\overline H}^{(\gamma)}} 
    \lesssim \nu^{90} \ep^3.
\end{align*}
In the same manner, we also have
\begin{align*}
		|N_{12}^{(I, 1)}| &\le 	\frac{1}{t} 	\brak{t}^{3+2\ss}  \bold{a}_{2}^2
	\sum_{k}     \norm{  \Gamma_k \stf_{k}(1-\chi_2^2) }_{L^\infty} 
	\enorm{ik q\Gamma_0  \omega_k e^{W} \chi_1}
	\enorm{ \bhqnn  \chi_1} 	
	\\& \lesssim \nu^{90} \ep^3.
\end{align*}
The above estimates similarly work for $	N_{2}^{(I, 1)}$ and we only show one term where the derivative falls on $\omega_{k}$
\begin{align*}
	\sum_{k} &
	\Big \langle \frac{1}{t} ik \stf_{k} q\Gamma_0 \Gamma_k \omega_k , 
	\brak{t}^{3+2\ss}  \bold{a}_{2}^2  \bhqnn  e^{W}\chi_1^2 (1-\chi_2^2)\Big \rangle 
	\\&
	\lesssim
	\sum_{k} 
	\frac{1}{t}\brak{t}^{2+\ss}  \bold{a}_{2}^2   \norm{k \stf_{k} \sqrt{(1-\chi_2^2)}}_{L^\infty} \enorm{q\Gamma_0 \Gamma_k \omega_k e^{W} \chi_1 \sqrt{(1-\chi_2^2)}}
	   \brak{t}^{1+\ss} \enorm{\mathring{\overline{H}}_1  \chi_1 (1-\chi_3^2)}
	  \\&
\lesssim \nu^{90} \ep^3,
\end{align*}
where we used that $e^{W/2}$ is huge near the boundary, concluding the proof of this lemma.
\end{proof}
\subsubsection{Viscous commutator}
There are several commutator terms to be addressed. We first focus on the viscous commutator
\begin{align*}
	C^{(n)}_{\overline{H},visc} &= \brak{t}^{3+2\ss} \bold{a}_{n+1}^2\Re\brak{  \widetilde{\mathcal{C}^{(n)}_{visc}}, \bhqn  e^{W} \chi_n^2}
	\nn&
	=
	\nu \brak{t}^{3+2\ss} \bold{a}_{n+1}^2\Re\brak{   q^n\sum_{m = 1}^n \binom{n}{m}  \Gamma_0^m v_y^2 \Gamma_0^{2 + n-m} \overline{H}, \bhqn  e^{W} \chi_n^2}.
\end{align*}
\begin{lemma}
	\label{visc:comm:barH}
	It follows
	\begin{align}
	\label{visc:comm:esti:barH}
	&\sum_{n} \theta_{n}^2 (n+1)^{2\sss-2}  C^{(n)}_{\overline{H},visc}
	\lesssim
       \sqrt{\mathcal{E}_{\overline H}^{(\alpha)}}
	\sqrt{\cd_{\overline H}^{\gm}} \paren{\sqrt{\mathcal{CK}_{\overline{H}}^{\gm}} + \sqrt{\cd_{\overline H}^{\gm}}}
	\lesssim
	\eps
	\sqrt{\cd_{\overline H}^{\gm}} \paren{\sqrt{\mathcal{CK}_{\overline{H}}^{\gm}} + \sqrt{\cd_{\overline H}^{\gm}}}.
\end{align}
\begin{proof}
The proof follows almost line by line as that of Lemma~5.6 in \cite{BHIW24b} except a slight difference in 
	the binomial coefficients and we only show one term here to give the readers a
	flavor. 
{	As in Lemma~5.6 in \cite{BHIW24b}, we arrive at}
	\begin{align}
		\label{visc:comm:barH:eq01}
		C^{(n)}_{\overline{H},visc}
		&=
		\nu \brak{t}^{3+2\ss}\sum_{l = 1}^n  \bold{a}_{n+1}^2 \binom{n}{l}
		\Re\brak{   q^n   \Gamma_0^l \paren{v_y^2-1} \Gamma_0^{2 + n-l} \overline{H},
		 \bhqn  e^{W} \chi_n^2}
		\nn&=
		-\nu \brak{t}^{3+2\ss}\sum_{l = 1}^n  \bold{a}_{n+1}^2 \binom{n}{l}
		\Re\brak{   q^n   \Gamma_0^l \paren{v_y^2-1} \Gamma_0^{1 + n-l} \overline{H},
		  \Gamma_0 	\bhqn  e^{W} \chi_n^2}
		\nn&\quad
		-\nu \brak{t}^{3+2\ss}\sum_{l = 1}^n  \bold{a}_{n+1}^2 \binom{n}{l}
		\Re\brak{   q^n   \Gamma_0^l \paren{v_y^2-1} \Gamma_0^{1 + n-l} \overline{H},
		\bhqn  2\Gamma_0W e^{W}\chi_n^2}
		\nn&\quad 
		-\nu \brak{t}^{3+2\ss}\sum_{l = 1}^n  \bold{a}_{n+1}^2 \binom{n}{l}
		\Re\brak{   \Gamma_0\paren{q^l   \Gamma_0^l \paren{v_y^2-1}} q^{n-l}\Gamma_0^{1 + n-l} \overline{H},
			\bhqn  e^{W} \chi_n^2}
		\nn&\quad
		-\nu \brak{t}^{3+2\ss}\sum_{l = 1}^n  \bold{a}_{n+1}^2 \binom{n}{l}
		\Re\brak{   q^l   \Gamma_0^l \paren{v_y^2-1} \Gamma_0 q^{n-l}\Gamma_0^{1 + n-l} \overline{H},
			\bhqn  e^{W} \chi_n^2}
		\nn&\quad
		-\nu \brak{t}^{3+2\ss}\sum_{l = 1}^n  \bold{a}_{n+1}^2 \binom{n}{l}
		\Re\brak{   q^n   \Gamma_0^l \paren{v_y^2-1} \Gamma_0^{1 + n-l} \overline{H},
			\bhqn 2	\Gamma_0 \chi_n e^{W} \chi_n}
		\nn&\quad
		-\nu \brak{t}^{3+2\ss}\sum_{l = 1}^n  \bold{a}_{n+1}^2 \binom{n}{l}
		\Re\brak{   q^n   \Gamma_0^l \paren{v_y^2-1} \Gamma_0^{1 + n-l} \overline{H},
			\bhqn  e^{W} \partial_y\paren{\frac{1}{v_y}} \chi_n^2}
		\nn&= 	V_{\overline H, 1}^{(n)}+V_{\overline H, 2}^{(n)}+V_{\overline H, 3}^{(n)}+V_{\overline H, 4}^{(n)}+V_{\overline H, 5}^{(n)}+V_{\overline H, 6}^{(n)}.
	\end{align}
Note in this lemma, we automatically have $n\ge1$ and the localization property is available. 
We proceed to rewrite the term involving $V_{\overline H, 1}^{(n)}$ as
\begin{align*}
	\sum_{n} \theta_{n}^2 (n+1)^{2\sss-2}  V_{\overline H, 1}^{(n)} &= 
	\sum_{n}  \paren{\sum_{l\le n/2} + \sum_{n/2<l\le n} } 
	(n+1)^{2\sss-2}\nu \brak{t}^{3+2\ss} \bold{a}_{n+1}^2 \binom{n}{l}
	\nn&\qquad\times
	\Re\brak{   q^n   \Gamma_0^l \paren{v_y^2-1} \Gamma_0^{2 + n-l} \overline{H},
		\bhqn  e^{W} \chi_n^2}
	\nn&=
	\sum_{n} \theta_{n}^2 (n+1)^{2\sss-2}  V_{\overline H, 1}^{LH}	+ 	\sum_{n} \theta_{n}^2 (n+1)^{2\sss-2}  V_{\overline H, 1}^{HL}
\end{align*}
where we only give details of the treatment of the second piece. By H\"older's inequality, one gets
\begin{align*}
		\sum_{n} \theta_{n}^2 (n+1)^{2\sss-2}  V_{\overline H, 1}^{HL}
		&\lesssim
		\sum_{n}  \sum_{n/2<l\le n}  
		(n+1)^{2\sss-2}\nu \brak{t}^{3+2\ss} \bold{a}_{n+1}^2 \binom{n}{l}
		   \enorm{q^{l-1}\Gamma_0^l \paren{v_y^2-1} e^{W/2}\chi_{l-1}\widetilde\chi_1}
		 		\nn&\qquad\times 
		\norm{q^{n-l+1}\Gamma_0^{1 + n-l} \overline{H} \chi_{n-l+1}}_{L^\infty}
			\enorm{\frac{\partial_y}{v_y}\bhqn  e^{W/2} \chi_n}
\end{align*}
where $\widetilde\chi_1$ is a fattened version of $\chi_1$ and the localization property still holds in the support of it. 
By Sobolev embedding and Poincar\'e's inequality, it follows
\begin{align*}
	&\norm{q^{n-l+1}\Gamma_0^{1 + n-l} \overline{H} \chi_{n-l+1}}_{L^\infty}
\nn&\qquad\qquad
\lesssim
	(1+(n-l)^{1+\sigma})\norm{q^{n-l+1}\Gamma_0^{1 + n-l} \overline{H} \chi_{n-l}\widetilde\chi_1}_{L^2}
		\nn&\qquad\qquad\qquad+
	\norm{\partial_y\paren{q^{n-l+1}\Gamma_0^{1 + n-l} \overline{H}} \chi_{n-l+1}}_{L^2}
\nn&\qquad\qquad
\lesssim
(1+(n-l)^{1+\sigma})\norm{ \ztp{\overline{H}}_{n-l+1} e^{W/2} \chi_{n-l+1}}_{L^2}
+\nu^{1/2}
\norm{\partial_y\ztp{\overline{H}}_{n-l+1} e^{W/2}\chi_{n-l+1}}_{L^2},
\end{align*}
where we also used that
\begin{align*}
	\partial_{y} \chi_{m+1} \lesssim m^{1+\sigma} \chi_{m},\ \mbox{for}\ m\ge0.
\end{align*}
Using $\abs{v_y}\gtrsim 1$, we get
\begin{align*}
	\enorm{\frac{\partial_y}{v_y}\bhqn  e^{W/2} \chi_n} 
	\lesssim \enorm{\partial_y\bhqn  e^{W/2} \chi_n}.
\end{align*}
Therefore, we further obtain
\begin{align*}
		&\sum_{n} \theta_{n}^2 (n+1)^{2\sss-2}  V_{\overline H, 1}^{HL}
	\lesssim
	\sum_{n}  \sum_{n/2<l\le n}  
	(n+1)^{2\sss-2}\nu \brak{t}^{3+2\ss} \bold{a}_{n+1}^2 \binom{n}{l} \paren{\bb_{l-1}\bb_{n-l}}^{-1} (n-l)^{-\sss+1}
	\bb_{l-1}
		\nn&\qquad\times \enorm{q^{l-1}\Gamma_0^l \paren{v_y^2-1} e^{W/2}\chi_n}
\bb_{n-l}(n-l)^{\sss+\sigma}\norm{ \ztp{\overline{H}}_{n-l+1} e^{W/2} \chi_{n-l+1}}_{L^2}
	\enorm{\partial_y\bhqn  e^{W/2} \chi_n}
	\nn	&\qquad+
	\sum_{n}  \sum_{n/2<l\le n}  
(n+1)^{2\sss-2}\nu \brak{t}^{3+2\ss} \bold{a}_{n+1}^2 \binom{n}{l} \paren{\bb_{l-1}\bb_{n-l}}^{-1} (n-l)^{-\sss+1}
\bb_{l-1}
\nn&\qquad\times \enorm{q^{l-1}\Gamma_0^l \paren{v_y^2-1} e^{W/2}\chi_n}
\bb_{n-l}(n-l)^{\sss+\sigma}	\nu^{1/2}\norm{\partial_y\ztp{\overline{H}}_{n-l+1} e^{W/2}\chi_{n-l+1}}_{L^2}
\enorm{\partial_y\bhqn  e^{W/2} \chi_n}.
\end{align*}
By Lemma~\ref{comb:boun}, H\"older's inequality, and Young's inequality, it follows
\begin{align*}
	\sum_{n} \theta_{n}^2 (n+1)^{2\sss-2}  V_{\overline H, 1}^{HL} \lesssim &
	\nu^{100}	\sqrt{\cd_{\overline H}^{\gm}} \brak{t}^{3/2+\ss}\paren{\norm{\overline{H}}_{Y_{0,-s}} + \nu^{1/2}\norm{\overline{H}}_{Y_{1,-s}} }
	\\&\quad\times
	\paren{\sum_{l\ge1} \enorm{\bold{a}_{\ell} q^{\ell-1}\Gamma_0^{\ell} \paren{v_y^2-1} e^{W/2} \chi_{l}}^2}^{1/2}.
\end{align*}
A similar argument as  {Lemma~5.1 in~\cite{BHIW24b}}, we deduce
\begin{align*}
	\paren{\sum_{l\ge1} \enorm{\bold{a}_{\ell} q^{\ell-1}\Gamma_0^{\ell} \paren{v_y^2-1} e^{W/2} \chi_{l}}^2}^{1/2}
	\lesssim
	\sum_{m} \bb_m \enorm{\partial_y \bd^{m}\paren{v_y^2-1} e^{W/2}\chi_n}
	\lesssim \norm{v_y^2-1}_{Y_{1, 0}}.
\end{align*}
On the other hand, for the $LH$ piece, we proceed as
\begin{align*}
	\sum_{n} \theta_{n}^2 (n+1)^{2\sss-2}  V_{\overline H, 1}^{LH}
	&\lesssim
	\sum_{n}  \sum_{n/2<l\le n}  
	(n+1)^{2\sss-2}\nu \brak{t}^{3+2\ss} \bold{a}_{n+1}^2 \binom{n}{l}
	\norm{q^{l}\Gamma_0^l \paren{v_y^2-1} e^{W/2}\chi_n}_{L^\infty}
	\nn&\qquad\times 
	\enorm{q^{n-l}\Gamma_0^{1 + n-l} \overline{H} \chi_{n-l+1}}
	\enorm{\frac{\partial_y}{v_y}\bhqn  e^{W/2} \chi_n},
\end{align*}
 {from where using a similar argument as in Lemma~5.6 in \cite{BHIW24b}, we obtain}
\begin{align*}
	\sum_{n} \theta_{n}^2 (n+1)^{2\sss-2}  V_{\overline H, 1}^{LH} &\lesssim
	\nu^{100}
		\sqrt{\cd_{\overline H}^{\gm}} 
      \nu^{1/2}\norm{v_y^2-1}_{Y_{1, 0}} 
	\norm{\overline{H}}_{Y_{1,-s}} 
	\\&\lesssim
	\nu^{100}
	\sqrt{\cd_{\overline H}^{\gm}} 
	\nu^{1/2}\norm{H}_{Y_{1, 0}} \norm{H}_{\overline Y_{1, 0}} 
	\norm{\overline{H}}_{Y_{1,-s}} 
	\\&\lesssim
	\nu^{100}
	\sqrt{\cd_{\overline H}^{\gm}} 
	\nu^{1/2} \mathcal{E}_{H}^{(\alpha)} 
	\brak{t}^{-3/2-s^{-1}}\sqrt{\mathcal{E}_{\overline H}^{(\alpha)}}
		\\&\lesssim
	\eps \sqrt{\cd_{\overline H}^{\gm}} \sqrt{\cd_{ H}^{\gm}} 
\end{align*}
and finish the treatment of $V_{\overline H, 1}^{(n)}$, where we also use Lemma~\ref{pro:1}, Poincar\'e's inequality, and the bootstrap assumptions. 
 {The rest of the proof is parallel to that of Lemma 5.6 in~\cite{BHIW24b} and 
we omit further details for the sake of brevity.}
\end{proof}
\end{lemma}
\subsubsection{Transport commutator}
For the transport commutator part, we note
\begin{align}
	C^{(n)}_{\overline{H},trans} &= 
	\brak{t}^{3+2\ss} \bold{a}_{n+1}^2\Re \brak{\widetilde{\mathcal{C}^{(n)}_{trans}} , \bhqn  e^{W} \chi_n^2}
	\nn&
	=
	\brak{t}^{3+2\ss} \bold{a}_{n+1}^2
	\Re \brak{q^n\sum_{m = 1}^n \binom{n}{m} \Gamma_0^m G \Gamma_0^{n-m+1} \overline{H}, \bhqn  e^{W} \chi_n^2}.
\end{align}
 {By an argument similar to Lemma 5.9 in~\cite{BHIW24b}, we obtain the following lemma. }
\begin{lemma}
	\label{tran:comm:barH}
       It holds
\begin{align}
	\label{tran:comm:esti:barH}
	\sum_{n} \theta_{n}^2 (n+1)^{2\sss-2}  C^{(n)}_{\overline{H},trans}
	\lesssim \nu^{90}\frac{\paren{\mathcal{E}_{\overline H}^{(\gamma)}}^2}{\brak{t}^2}  + \mathcal{E}_{\overline H}^{(\gamma)} 
	\paren{ \cd_{\overline H}^{\gm} + \cd_{H}^{\gm}}
	\lesssim \frac{\eps^4}{\brak{t}^2} + 
	\eps \paren{ \cd_{\overline H}^{\gm} + \cd_{H}^{\gm}}.
\end{align}
\end{lemma}
\begin{proof}
	As before, we split the sum into two pieces
		\begin{align*}
C^{(n)}_{\overline{H},trans} &
=\paren{\sum_{l\le n/2} + \sum_{n/2<l\le n-1} }\binom{n}{l}
\brak{t}^{3+2\ss} \bold{a}_{n+1}^2
\Re \brak{q^n \Gamma_0^l \frac{\overline H}{v_y} \Gamma_0^{n-l} \overline{H}, \bhqn  e^{W} \chi_n^2}
	\nn&	=
		T_{\overline H,1}^{\gm} +T_{\overline H,2}^{\gm}.
	\end{align*}
 {By exact the same argument as in the proof of Lemma~5.9 in~\cite{BHIW24b}, we have}
	\begin{align*}
	\sum_{n} & \theta_{n}^2 (n+1)^{2\sss-2} T_{\overline H,1}^{\gm}\\
	&\lesssim
	\sum_{n\ge0}
	\sum_{l\le n/2} | \bold{a}_{n+1}|^2 \brak{t}^{3+2\ss}
	\binom{n}{l}  (n+1)^{2\sss-2}
	\nnorm{  \overline \hqn  e^{W}\chi_n}_{L^2} \bold{a}_{n-l+1}^{-1}\bb_{l+1}^{-1}(l+1)^{2-2\sss}
	\nn&\quad
	\times
	 \bold{a}_{n-l+1}(n-l+1)^{\sss-1}\enorm{ \ztp{\overline{H}}_{n-l} e^{W}\chi_n}
	\bb_{l+1}(l+1)^{2\sss-2}\paren{ l^{1+\sigma} + 1}\norm{\bd^{l} \frac{\overline{H}}{v_y}\chi_{(l-1)_+}\widetilde\chi_1}_{L^2}
	\nn&\quad+
\sum_{n\ge0}
\sum_{l\le n/2} | \bold{a}_{n+1}|^2 \brak{t}^{3+2\ss}
\binom{n}{l}  (n+1)^{2\sss-2}
	\nnorm{ \bhqn e^{W}\chi_n}_{L^2} \bold{a}_{n-l+1}^{-1}\bb_{l+1}^{-1}(l+1)^{2-2\sss}
	\nn&\quad
	\times
	 \bold{a}_{n-l+1}(n-l+1)^{\sss-1}\enorm{ \ztp{\overline{H}}_{n-l} e^{W}\chi_n}
	\bb_{l+1}(l+1)^{2\sss-2}\nu^{1/2} \norm{\partial_y\paren{ \bd^{l} \frac{\overline{H}}{v_y}} \chi_{l}\widetilde\chi_1}_{L^2}
	\\&\lesssim \nu^{100}
	\brak{t}^{3+2\ss}
	\norm{\overline{H}}_{Y_{0,-s}}^2
	\paren{\norm{\frac{\overline{H}}{v_y}}_{Y_{0,-s}} 
		+\nu^{1/2}\norm{\frac{\overline{H}}{v_y}}_{Y_{1,-s}}}.
\end{align*}
Using Lemma~\ref{pro:-s} and the identity
\begin{align}
	\label{1onvy}
	\frac{1}{v_y} = 	\frac{1}{1 + H} = \sum_{l=0}^{\infty} (-1)^l H^l,
\end{align}
 we further arrive at
	\begin{align*}
	\sum_{n} & \theta_{n}^2 (n+1)^{2\sss-2} T_{\overline H,1}^{\gm}
	\\
	&\lesssim \nu^{100}
	\brak{t}^{3+2\ss}
	\norm{\overline{H}}_{Y_{0,-s}}^2
     \left(\norm{\overline{H}}_{Y_{0,-s}} \paren{\norm{\frac{1}{v_y}}_{\overline Y_{1, 0}} + \enorm{\frac{1}{v_y} }} + \paren{\norm{\overline{H}}_{Y_{1,-s}} +\enorm{\overline H}} \norm{\frac{1}{v_y}}_{\overline Y_{0, 0}}\right.
	\\&\quad \left.
	+\nu^{1/2}\norm{\overline{H}}_{Y_{1,-s}} \paren{\norm{\frac{1}{v_y}}_{\overline Y_{1, 0}} + \enorm{\frac{1}{v_y} }} + \paren{\norm{\overline{H}}_{Y_{1,-s}} +\enorm{\overline H}} \norm{\frac{1}{v_y}}_{\overline Y_{1, 0}}\right)
	\\
	&\lesssim \nu^{100}
	\brak{t}^{3+2\ss}
	\norm{\overline{H}}_{Y_{0,-s}}^2
	\left( \norm{\overline{H}}_{Y_{0,-s}} \paren{\norm{\frac{1}{v_y}}_{\overline Y_{1, 0}} + \enorm{\frac{1}{v_y} }} + \paren{\norm{\overline{H}}_{Y_{1,-s}} +\enorm{\overline H}} \norm{\frac{1}{v_y}}_{\overline Y_{0, 0}}\right.
	\\&\quad \left.
	+\nu^{1/2}\norm{\overline{H}}_{Y_{1,-s}} \paren{\norm{\frac{1}{v_y}}_{\overline Y_{1, 0}} + \enorm{\frac{1}{v_y} }} + \nu^{1/2} \paren{\norm{\overline{H}}_{Y_{1,-s}} +\enorm{\overline H}} \norm{\frac{1}{v_y}}_{\overline Y_{1, 0}} \right)
	\\
	&\lesssim \nu^{100}
	\brak{t}^{3+2\ss}
	\norm{\overline{H}}_{Y_{0,-s}}^2
	\left( \norm{\overline{H}}_{Y_{0,-s}} \paren{\frac{1}{1-\norm{H}_{\overline Y_{1. 0}}-\enorm{H}} 
		+ 1 + \norm{H}_{L^\infty}} \right.
	\\&\quad+ \left. \paren{\norm{\overline{H}}_{Y_{1,-s}} +\enorm{\overline H}}\paren{1+\frac{\norm{H}_{Y_{0, 0}}}{1-\norm{H}_{\overline Y_{1, 0}} - \enorm{H}}}\right.
	\\&\quad \left.
	+\nu^{1/2}\norm{\overline{H}}_{Y_{1,-s}} \paren{\frac{\norm{H}_{\overline Y_{1. 0}}}{1-\norm{H}_{\overline Y_{1. 0}}}  + 1 + \norm{H}_{L^\infty}} 
	\right. \\&
	\quad \left. + \nu^{1/2} \paren{\norm{\overline{H}}_{Y_{1,-s}} +\enorm{\overline H}} \frac{\norm{H}_{\overline Y_{1. 0}}}{1-\norm{H}_{\overline Y_{1. 0}}}   \right)
	\\&\lesssim
	\nu^{90}\frac{\paren{\mathcal{E}_{\overline H}^{(\gamma)}}^2}{\brak{t}^2}  + \mathcal{E}_{\overline H}^{(\gamma)} 
	\paren{ \cd_{\overline H}^{\gm} + \cd_{H}^{\gm}}
	\lesssim \frac{\eps^4}{\brak{t}^2} + 
	\eps \paren{ \cd_{\overline H}^{\gm} + \cd_{H}^{\gm}}
\end{align*}
where we also used Remark~\ref{Y:norm}, H\"older's inequality, and the bootstrap assumptions.
The estimates of the term $T_{\overline H,2}^{\gm}$ follows  {similarly as the second part of Lemma~5.9 in \cite{BHIW24b} and} we omit further details to conclude the proof.
\end{proof}
\subsubsection{Inductive terms}
We recall that
\begin{align*}
	C^{(n)}_{\overline{H},q}&=- \brak{t}^{3+2\ss} \bold{a}_{n+1}^2\Re \brak{ \widetilde{\mathcal{C}^{(n)}_{q}}, 
		\bhqn  e^{W} \chi_n^2}
	\\&=
	 \nu\brak{t}^{3+2\ss} \bold{a}_{n+1}^2\Re \brak{|q'|^2 \frac{n(n-1)}{q^2}  \bhqn  , 
		\bhqn  e^{W} \chi_n^2}
	+
	 \nu\brak{t}^{3+2\ss} \bold{a}_{n+1}^2\Re \brak{ 2 q'  \frac{n}{q} \pa_y \bhqn, 
		\bhqn  e^{W} \chi_n^2}
	\\&\quad
	+
	 \nu\brak{t}^{3+2\ss} \bold{a}_{n+1}^2\Re \brak{ q'' \frac{n}{q} \bhqn , 
		\bhqn  e^{W} \chi_n^2}.
\end{align*}
Note that this term is actually a linear term, and
 we obtain
\begin{lemma}
	We have
	\label{C:n:q}
	\begin{align*}
	\sum_{n} \theta_{n}^2 (n+1)^{2\sss-2} 	C^{(n)}_{\overline{H},q} \lesssim \eps_1 \cd_{\overline H}^{\gm}
	\end{align*}
	for some small universal constant $\eps_1>0$.
\end{lemma}
\begin{proof}
	By a similar argument as  {Lemma~5.1 in~\cite{BHIW24b}}, we deduce that
		\begin{align*}
		\sum_{n} &\theta_{n}^2 (n+1)^{2\sss-2} 	C^{(n)}_{\overline{H},q} \le \sum_{n} \theta_{n}^2\sum_{0\le m\le n-1} C^{n-m}
		\paren{\frac{n!}{m!}}^{-\sigma} D_{\overline H,m}^{\gm} 
		\\& \le C \min\left\{\frac{1}{n_{*}^{\sigma/2}}, \delta_{\text{Drop}}^2\right\} \cd_{\overline H}^{\gm} \le \eps_1 \cd_{\overline H}^{\gm}
	\end{align*}
provided $n_{*} $ is large enough and $\delta_{\text{Drop}}$ is sufficiently small, 
where we used Fubini's theorem 
and 
\begin{align*}
	\frac{\theta_{n+1}}{\theta_{n}} = \delta_{\text{Drop}}
\end{align*}
for $n\le n_{*}$,
 concluding the proof.
\end{proof}
\subsubsection{Easy terms}
Next we bound the easy terms at once.
\begin{lemma}
	\label{easy:term}
	It holds
	\begin{align*}
		&\sum_{n} \theta_{n}^2 (n+1)^{2\sss-2}  \bold{a}_{n+1}^2  \brak{t}^{3+2\ss}  \nu 
		\langle \partial_y\bhqn , \bhqn   e^{W}  \pa_y \{\chi_n^2\} \rangle 
		\lesssim
		\eps_1 \cd_{\overline H}^{\gm};
		\nn&
		\sum_{n} \theta_{n}^2 (n+1)^{2\sss-2}  \bold{a}_{n+1}^2   \brak{t}^{3+2\ss} \nu \langle \partial_y \bhqn , 
		\bhqn  \pa_y^2 \{ e^{W}  \} \chi_n^2  \rangle
		\lesssim \eps_1 \sqrt{\mathcal{CK}_{\overline{H}}^{\gm}}\sqrt{\cd_{\overline H}^{\gm}};
		\nn&
		\sum_{n} \theta_{n}^2 (n+1)^{2\sss-2} \paren{(3/2+\ss)\brak{t}^{1+2\ss}t - \frac{2}{t}\brak{t}^{3+2\ss}} \bold{a}_{n+1}^2 \| \bhqn  e^{W/2} \chi_{n} \|_{L^2}^2
	\le	\frac{1}{2} \mathcal{CK}_{\overline{H}}^{\gm}.
	\end{align*} 
for some small $\eps_1>0$.
\end{lemma}
\begin{proof}
In view of the fact
\begin{align*}
	\abs{\partial_y \chi_n} \lesssim n^{1+\sigma}
\end{align*}
with $\sss> 1 + \sigma$, we have
\begin{align*}
	&\sum_{n} \theta_{n}^2 (n+1)^{2\sss-2}  \bold{a}_{n+1}^2  \brak{t}^{3+2\ss}  \nu
	\langle \partial_y\bhqn , \bhqn   e^{W}  \pa_y \{\chi_n^2\} \rangle 
	\nn&\qquad\le 
	\sum_{n} \theta_{n}^2 (n+1)^{2\sss-2}  \bold{a}_{n+1}^2 n^{1+\sigma} \brak{t}^{3+2\ss}  \nu
	\enorm{q^{n}\frac{\partial_y}{v_y}\Gamma_0^{n-1}\overline{H}  e^{W/2} \chi_{n-1}\tilde{\chi}_1}
	\enorm{\partial_y \bhqn   e^{W/2} \chi_n}.
\end{align*}
Using
\begin{equation}
	\label{comm:q:deri}
	[q^\ell, \partial_y] = - \ell q^{\ell-1}\partial_yq
\end{equation}
 { and Lemma~6.20 in \cite{BHIW24b}, we have the further bound}
\begin{align*}
	\sum_{n} &\theta_{n}^2 (n+1)^{2\sss-2}  \bold{a}_{n+1}^2 n^{1+\sigma} \brak{t}^{3+2\ss}  \nu
	\paren{ \enorm{\partial_yq\Gamma_0^{n-1}\overline{H}  e^{W/2} \chi_{n-1}\tilde{\chi}_1 }
	+ 	n \enorm{q^{n-1}\Gamma_{0}^{n-1}\overline{H}  e^{W/2} \chi_{n-1}\tilde{\chi}_1 } }
	\nn&\qquad \times \enorm{\partial_y\bhqn   e^{W/2} \chi_n}
		\nn&
	\le C \sum_{n\ge1} \theta_{n}^2  \frac{1}{n^{\sigma^*}} \sqrt{\cd_{\overline H, n}} \sqrt{\cd_{\overline H, n-1}}
	+  C \sum_{n\ge1}\theta_{n}^2   \sqrt{\sum_{\ell=0}^{n-1} (\mathfrak C\lambda)^{2(n-\ell)}\lf(\frac{\ell!}{n!}\rg)^{2\sigma}  
		\cd_{\overline H, \ell}}  \sqrt{\cd_{\overline H, n}}
	\nn&
	\le  C \min\left\{\frac{1}{n_{*}^{\sigma^*}}, \delta_{\text{Drop}}\right\} \cd_{\overline H}^{\gm}
	+ C \min\left\{\frac{1}{n_{*}^{\sigma/2}}, \delta_{\text{Drop}}\right\} \cd_{\overline H}^{\gm}
	 \le \eps_1 \cd_{\overline H}^{\gm}
\end{align*}
for $\delta_{\text{Drop}}$ sufficiently small and $n_*$ large, 
where we also used that 
$\sss > 1+\sigma +\sigma^*$ for some $\sigma^*>0$.
Note by~\eqref{defndW} 
\begin{align*}
	\abs{\partial_yW(t, y)}\lesssim \frac{(|y|-1/4-L\ep\arctan (t))_+}{K\nu(1+t)}.
\end{align*}
Therefore, we may get
\begin{align*}
	&\sum_{n} \theta_{n}^2 (n+1)^{2\sss-2}  \bold{a}_{n+1}^2   \brak{t}^{3+2\ss} \nu\langle \partial_y\bhqn , 
	\bhqn  \pa_y \{ e^{W}  \} \chi_n^2  \rangle
	\nn&\qquad\lesssim
	\sum_{n} \theta_{n}^2 (n+1)^{2\sss-2} \bold{a}_{n+1}^2   \brak{t}^{3+2\ss}  
	\enorm{\nu^{1/2}\frac{(|y|-1/4-L\ep\arctan (t))_+}{K\nu(1+t)}\bhqn  e^{W/2}  \chi_n}
	\nn&\qquad\quad\times
	\nu^{1/2}\enorm{\partial_y\bhqn  e^{W/2}  \chi_n} 
	\nn&\qquad\lesssim \eps_1 \sqrt{\mathcal{CK}_{\overline{H}}^{\gm}}\sqrt{\cd_{\overline H}^{\gm}}
\end{align*}
where we used that $1\gg \nu$ and $K$ is large.
For the last inequality, we simply note that 
\begin{align*}
	\paren{(3/2+\ss)\brak{t}^{1+2\ss}t - \frac{2}{t}\brak{t}^{3+2\ss}} \le \frac{1}{2} \brak{t}^{1+2\ss}t
\end{align*}
and hence
\begin{align*}
	&\paren{(3/2+\ss)\brak{t}^{1+2\ss}t - \frac{2}{t}\brak{t}^{3+2\ss}} \bold{a}_{n+1}^2 \| \bhqn  e^{W/2} \chi_{n} \|_{L^2}^2
	\le  \frac{1}{2} \mathcal{CK}_{\overline{H}}^{\gm},
\end{align*}
where we used the $\mathcal{CK}$ terms due to $d\bold{a}_n/dt$ in the last step,
concluding the proof.
\end{proof}

\subsection{$(\overline {H}, \alpha)$ Estimates}
In this subsection, we give estimates of the higher regularity part, i.e., $\mathcal{E}_{\overline{H},n}^{(\alpha)}$, which we call the $\alpha$
part. Since most of the estimates are similar to that of the $\gamma$ part, we only sketch the proof with the main focus on essentially 
different terms. 
\begin{proof}[Proof of~\eqref{albarH}:] 
	We apply a $y$ derivative to \eqref{eq:bar:H:L:n} and compute the (complex) inner product of it against $\brak{t}^{3+2\ss} \bold{a}_{n+1}^2\partial_y \bhqn  e^{W} \chi_n^2$ to get 
	\begin{align*}
		&  \frac{\nu^{-1/2}}{2}\frac{d}{dt} \mathcal{E}_{\overline{H},n}^{(\alpha)}- \bold{a}_{n+1}  \frac{d}{dt}\paren{ \bold{a}_{n+1}} \brak{t}^{3+2\ss}
		\| \partial_y\bhqn  e^{W/2} \chi_n\|_{L^2}^2
		+ \bold{a}_{n+1}^2  \brak{t}^{3+2\ss} \| \partial_y\bhqn  \sqrt{- W_t} e^{W/2} \chi_n \|_{L^2}^2 
		\\&\qquad\qquad+ \bold{a}_{n+1}^2  \brak{t}^{3+2\ss} \nu \|  \pa_y^2 \bhqn  e^{W/2} \chi_n \|_{L^2}^2 
		\\
		&\qquad = -\bold{a}_{n+1}^2  \brak{t}^{3+2\ss}  \nu \langle \partial_y^2\bhqn ,
		\partial_y \bhqn   e^{W}  \pa_y \{\chi_n^2\}  \rangle 
		- \bold{a}_{n+1}^2   \brak{t}^{3+2\ss} \nu \langle \partial_y^2 \bhqn , 
		\partial_y\bhqn  \pa_y \{ e^{W}  \} \chi_n^2  \rangle 
		\\&\qquad\qquad  +
		\paren{(3/2+\ss)\brak{t}^{1+2\ss}t - \frac{2}{t}\brak{t}^{3+2\ss}}\bold{a}_{n+1}^2
		\| \partial_y \bhqn  e^{W/2} \chi_{n} \|_{L^2}^2
		\\&\qquad\qquad-  \brak{t}^{3+2\ss} \bold{a}_{n+1}^2\Re \brak{\partial_y \widetilde{\mathcal{C}^{(n)}_{trans}} ,
			\partial_y\bhqn  e^{W} \chi_n^2}
		-   \brak{t}^{3+2\ss} \bold{a}_{n+1}^2\Re\brak{  \partial_y\widetilde{\mathcal{C}^{(n)}_{visc}},
			\partial_y\bhqn  e^{W} \chi_n^2}
		\\&\qquad\qquad-  \brak{t}^{3+2\ss} \bold{a}_{n+1}^2\Re \brak{ \partial_y\widetilde{\mathcal{C}^{(n)}_{q}}, 
			\partial_y\bhqn  e^{W} \chi_n^2}
		\\&\qquad\qquad-  \brak{t}^{3+2\ss} t^{-1} \bold{a}_{n+1}^2 
		\Re \brak{   \partial_y \paren{q^n \Gamma_0^n \paren{  \nabla^{\perp} \stf_{\neq 0} \cdot \nabla \omega }_0}, 
			\partial_y\bhqn  e^{W} \chi_n^2}
		\\
		&\qquad =  -\bold{a}_{n+1}^2  \brak{t}^{3+2\ss}   \nu \langle \partial_y^2\bhqn ,
		\partial_y\bhqn   e^{W}  \pa_y \{\chi_n^2\}  \rangle 
		- \bold{a}_{n+1}^2   \brak{t}^{3+2\ss}  \nu \langle \partial_y^2 \bhqn , 
		\partial_y\bhqn  \pa_y \{ e^{W}  \} \chi_n^2  \rangle 
		\\&\qquad\qquad  +
		\paren{(3/2+\ss)\brak{t}^{1+2\ss}t - \frac{2}{t}\brak{t}^{3+2\ss}} \brak{t}^2\bold{a}_{n+1}^2
		\| \partial_y \bhqn  e^{W/2} \chi_{n} \|_{L^2}^2
		\\&\qquad\qquad- C^{(\alpha; n)}_{\overline{H},trans} - C^{(\alpha; n)}_{\overline{H},visc} - C^{(\alpha; n)}_{\overline{H}, q} - N^{(\alpha; n)}_{\overline{H}}
		.
	\end{align*}
	Similar as~\eqref{CK:1} and \eqref{CK:2}, we obtain
	\begin{align}
		\label{CK:1:1}
		\nu^{-1/2}\paren{\mathcal{CK}_{\overline{H},n}^{(\alpha; 1)}}^2 \gtrsim  \nu^{-1/3+2\delta} \bold{a}_{n+1}^2 \brak{t}^{3+2\ss} 
		\| \partial_y\bhqn   e^{W/2} \chi_n \|_{L^2}^2
	\end{align}
	for $t\lesssim \nu^{-1/3-\delta}$ with
	\begin{align}
		\label{CK:2:1}
				\nu^{-1/2}\paren{\mathcal{CK}_{\overline{H},n}^{(\alpha; 2)}}^2 \gtrsim  \frac{n+1}{t} \bold{a}_{n+1}^2 \brak{t}^{3+2\ss}
		\| \partial_y\bhqn  e^{W/2} \chi_n \langle t \rangle^2\|_{L^2}^2.
	\end{align}
	Then the desired result is a consequence of lemmas~\ref{nonl:barH:1}--\ref{easy:term:1}.
\end{proof}
\subsubsection{Nonlinear contribution}
We first consider the nonlinear estimate. 
\begin{lemma}
	\label{nonl:barH:1}
	\begin{align*}
				\sum_{n} &\theta_{n}^2 (n+1)^{2\sss-2} \abs{N^{(\alpha; n)}_{\overline{ H}}}
				\\&\lesssim  \paren{\sum_{j=1, 2}\sqrt{\mathcal{J}_{ell}^{(i)}}+ \sqrt{\mathcal{E}_{ell}^{(I, out)}}} \left(
				\cd^{\gm}+\cd_{H}^{\gm} 
			+\cd_{\overline H}^{\gm} + \mathcal{CK}^{\gm} 
			+\mathcal{CK}_{H}^{\gm}+\mathcal{CK}_{\overline H}^{\gm} +  \mathcal{J}_{ell}^{(3)} \right)
		\\&\quad 
       +
		\sqrt{\mathcal{E}_{ell}^{(I, full)} + \brak{t}^2 \mathcal{J}_{ell}^{(2)}} \sqrt{\mathcal{E_\mathrm{Int}}} \brak{t}^{1/2-s^{-1}}   \sqrt{\mathcal{E}_{\overline H}^{(\alpha)}}  
	\\&\quad    + 
		\frac{\brak{t}^{2\ss}}{\brak{t}^{r}} \sqrt{\mathcal{E}_{ell}^{(I, full)}}  \sqrt{\mathcal{E}_{cloud}} \sqrt{\mathcal{E}^{(\overline h)}_{\mathrm{Int, Coord}}}
	\\&\quad 
+
	\nu^{100}\brak{t}^{2\ss} \paren{\sqrt{\mathcal{E}_{ell}^{(I, full)}} + \sqrt{\brak{t}^2 \mathcal{J}_{ell}^{(2)}}} \sum_{\iota\in\{\alpha,\gamma\}} \mathcal{E}^{(\iota)} 
    \sqrt{\mathcal{E}_{\overline H}^{(\alpha)}}
		\\&
		\lesssim \frac{\eps^3}{\brak{t}^{3+r-2\ss}} +  \eps \sum_{\iota\in\{\alpha,\gamma\}} \paren{\cd^{(\iota)}+\cd_{H}^{(\iota)} 
			+\cd_{\overline H}^{(\iota)}}
		+\eps\sum_{\iota\in\{\alpha,\gamma\}} \paren{\mathcal{CK}^{(\iota)}
			+\mathcal{CK}_{H}^{(\iota)}+\mathcal{CK}_{\overline H}^{(\iota)}}.
	\end{align*}
\end{lemma}
\begin{proof} {\bf (1). Case $n\ge 2$:}
We rewrite 	as in Lemma~\ref{nonl:barH}
	\begin{align*}
		N^{(\alpha; n)}_{\overline{H}} 
		&= 
		 \Big \langle \partial_y\paren{q^n\Gamma_{0}^n \frac{1}{t} \paren{  \nabla^{\perp} \stf_{\neq 0} \cdot \nabla \omega }_0}, \brak{t}^{3+2\ss}  \bold{a}_{n+1}^2  \partial_y\bhqn  e^{W} \chi_n^2 \Big \rangle 
		 \\& =
		 N^E+N^I.
	\end{align*}
{\bf Exterior part: } Applying the product rule, a similar calculation as in Lemma~\ref{nonl:barH} gives
	\begin{align*}
		N^{E} =-N_{1}^{E} + N_{2}^{E}
	\end{align*}
where
\begin{align*}
	N_{1}^{E}&= \frac{1}{t} \sum_{k}\sum_{m=0}^{n} \binom{n}{m} \Big \langle \partial_y \paren{q^{n-m}\Gamma_{k}^{n-m}\partial_y \stf_{k}^{(E)}  \overline{q^{m}\Gamma_{k}^{m} ik \omega_k }} , \brak{t}^{3+2\ss}  \bold{a}_{n+1}^2  \partial_y\bhqn  e^{W} \chi_n^2 \Big \rangle ,
	\\
	 N_{2}^{E}&=\frac{1}{t} \sum_{k}\sum_{m=0}^{n} \binom{n}{m} \Big \langle \partial_y\paren{ q^{n-m}\Gamma_{k}^{n-m}ik \stf_{k}^{(E)}  \overline{q^{m}\Gamma_{k}^{m} \partial_y \omega_k} }, \brak{t}^{3+2\ss}  \bold{a}_{n+1}^2  \partial_y\bhqn  e^{W} \chi_n^2 \Big \rangle. 
\end{align*}
Since the proof here is completely parallel to that of Lemma~\ref{nonl:barH}, we only show details for the first term $N_{1}^E$.
	We first divide the sum into two pieces
	\begin{align*}
		N_{1}^{E} &=  \brak{t}^{3+2\ss} t^{-1} \bold{a}_{n+1}^2 \sum_{k}\paren{\sum_{m\le n/2} + \sum_{m\ge n/2}^{n}}
		\binom{n}{m} \Big \langle  \partial_y\paren{q^{n-m}\Gamma_{k}^{n-m}\partial_y \stf_{k}^{(E)}  \overline{q^{m}\Gamma_{k}^{m} ik \omega_k }} ,  \partial_y \bhqn  e^{W} \chi_n^2 \Big \rangle 
		\\& = N_{1} ^{HL, E} + N_{1} ^{LH, E}.
	\end{align*}
	We then bound the $LH$ piece as
	\begin{align*}
		\abs{	N_{1} ^{LH, E}} \lesssim&
		\brak{t}^{3+2\ss} t^{-1} \bold{a}_{n+1}^2 \sum_{k} \sum_{m\ge n/2}^{n}\binom{n}{m} \norm{\partial_y \paren{q^{n-m}\Gamma_{k}^{n-m}\partial_y \stf_{k}^{(E)}}\chi_{n-1}}_{L^\infty}
		\nn&\quad \times\enorm{  \overline{q^{m}\Gamma_{k}^{m} ik \omega_k }e^{W} \chi_n}  \enorm{\partial_y\bhqn  \chi_n}
		\nn&\quad
		+ 
		\brak{t}^{3+2\ss} t^{-1} \bold{a}_{n+1}^2 \sum_{k} \sum_{m\ge n/2}^{n}\binom{n}{m} \norm{ q^{n-m}\Gamma_{k}^{n-m}\partial_y \stf_{k}^{(E)}\chi_{n-1}}_{L^\infty}
		\nn&\quad \times
		\enorm{ \partial_y \paren{\overline{q^{m}\Gamma_{k}^{m} ik \omega_k }}e^{W} \chi_n}  \enorm{\partial_y\bhqn  \chi_n}
		\nn&=
		N_{11} ^{LH, E} + N_{12} ^{LH, E}.
	\end{align*}
	By the commutator equality~\eqref{cm_py_G_n}, we deduce 
	\begin{align*}
		&\norm{\partial_y\paren{q^{n-m}\Gamma_{k}^{n-m}\partial_y \stf_{k}^{(E)}}\chi_{n-m}}_{L^\infty}
		\lesssim
		\norm{\partial_y^2 \paren{q^{n-m}\Gamma_{k}^{n-m} \stf_{k}^{(E)}}\chi_{n-m}}_{L^\infty}
		\nn \\&\qquad\quad + 
		kt\sum_{\ell=0}^{n-m-1}\binom{n-m}{\ell}\norm{\partial_y\paren{q^{n-m-\ell}\bd^{n-m-\ell}(v_y-1)q^{\ell}\Gamma_k^{\ell}\stf_{k}^{(E)}}\chi_{n-m}}_{L^\infty}
		\nn&\qquad\quad
		+\sum_{\ell=1}^{n}\binom{n-m}{\ell}\norm{\partial_y\paren{\bd^{n-m-\ell+1}(v_y-1)q^{\ell-1}\Gamma_k^{\ell}\stf_{k}^{(E)}}\chi_{n-m}}_{L^\infty}
		\\&\qquad\quad
		+ 	
		(n-m)\norm{\partial_y \paren{q^{n-m-1}\Gamma_k^{n-m}\stf_{k}^{(E)}}\chi_{n-m}}_{L^\infty}
		\nn&\qquad
		\lesssim
		\norm{\partial_y^2 \paren{q^{n-m}\Gamma_{k}^{n-m} \stf_{k}^{(E)}}\chi_{n-m}}_{L^2} ^{1/2}
				\nn \\&\qquad\quad\times
		\paren{	\norm{\partial_y\paren{\partial_y^2 \paren{q^{n-m}\Gamma_{k}^{n-m} \stf_{k}^{(E)}}\chi_{n-m}}}_{L^2}
			+ \norm{\partial_y^2 \paren{q^{n-m}\Gamma_{k}^{n-m} \stf_{k}^{(E)}}\partial_y\chi_{n-m}}_{L^2}}^{1/2}
		\\&\qquad\quad
		+ 
		kt\sum_{\ell=0}^{n-m-1}\binom{n-m}{\ell}\norm{\partial_y\paren{q^{n-m-\ell}\bd^{n-m-\ell}(v_y-1)q^{\ell}\Gamma_k^{\ell}\stf_{k}^{(E)}}\chi_{n-m}}_{L^\infty}
		\nn \\&\qquad\quad
		+\sum_{\ell=1}^{n}\binom{n-m}{\ell}\norm{\partial_y\paren{\bd^{n-m-\ell+1}(v_y-1)q^{\ell-1}\Gamma_k^{\ell}\stf_{k}^{(E)}}\chi_{n-m}}_{L^\infty}
		\\&\qquad\quad+
		(n-m)\norm{\partial_y\paren{q^{n-m-1}\Gamma_k^{n-m}\stf_{k}^{(E)}}\chi_{n-m}}_{L^\infty}
		\nn \\&\qquad
		=V^{LH, E}_1+V^{LH, E}_2+V^{LH, E}_3+V^{LH, E}_4.
	\end{align*}
	We define
	\begin{align*}
		N_{11,vi} ^{LH, E}:=\brak{t}^{3+2\ss} t^{-1} \bold{a}_{n+1}^2 \sum_{k} \sum_{m\ge n/2}^{n}\binom{n}{m} V^{LH, E}_i
		\enorm{  \overline{q^{m}\Gamma_{k}^{m} ik \omega_k }e^{W} \chi_n}  \enorm{\partial_y\bhqn  \chi_n}
	\end{align*}
	for $i=1,2,3,4$.
Thus by exactly the same argument as in Lemma~\ref{nonl:barH} we arrive at
	\begin{align*}
		\sum_{n\ge2}(n+1)^{2\sss-2} N_{11,v1} ^{LH, E} &\lesssim
		 \paren{\mathcal{J}_{ell}^{(2)}}^{1/4} 		 \paren{\mathcal{J}_{ell}^{(3)}}^{1/4} \sqrt{\cd^{\gm}} \paren{\mathcal{CK}_{\overline{H}}^{(\alpha)}}^{1/4}   \paren{\mathcal{E}_{\overline{H}}^{(\alpha)}}^{1/4} 
		 \\&
		 \lesssim \eps \paren{\cd^{\gm}}^{3/4} \paren{\mathcal{CK}_{\overline{H}}^{(\alpha)}}^{1/4} + \eps^3\brak{t}^{-2},
	\end{align*}
	where we also used H\"older's inequality and the bootstrap assumptions. 
	Next we deal with $N_{11,v2} ^{LH, E}$.
	We note that
	\begin{align*}
		V^{LH, E}_2&=	kt\sum_{\ell=0}^{n-m-1}\binom{n-m}{\ell}
		\norm{\partial_y\paren{q^{n-m-\ell}\bd^{n-m-\ell}(v_y-1)q^{\ell}\Gamma_k^{\ell}\stf_{k}^{(E)}}\chi_{n-1}}_{L^\infty}
		\\&\lesssim
		kt\sum_{\ell=0}^{n-m-1}\binom{n-m}{\ell}
		\norm{\partial_y\paren{q^{n-m-\ell}\bd^{n-m-\ell}(v_y-1)}\chi_{n-m-\ell+1}}_{L^\infty}\norm{q^{\ell}\Gamma_k^{\ell}\stf_{k}^{(E)}\chi_{\ell+1}}_{L^\infty}
		\\&\quad+
		kt\sum_{\ell=0}^{n-m-1}\binom{n-m}{\ell}
		\norm{q^{n-m-\ell}\bd^{n-m-\ell}(v_y-1)\chi_{n-m-\ell+1}}_{L^\infty}\norm{\partial_y\paren{q^{\ell}\Gamma_k^{\ell}\stf_{k}^{(E)}}\chi_{\ell+1}}_{L^\infty}
		\\&\lesssim
		kt\sum_{\ell=0}^{n-m-1}\binom{n-m}{\ell}\norm{\partial_y\paren{q^{n-m-\ell}\bd^{n-m-\ell}(v_y-1)}\chi_{n-m-\ell+1}}_{L^2}^{1/2}
				\\&\quad\times
		\norm{\partial_y\paren{\partial_y\paren{q^{n-m-\ell}\bd^{n-m-\ell}(v_y-1)}\chi_{n-m-\ell+1}}}_{L^2}^{1/2}
		\\&\quad\times
		\norm{q^{\ell}\Gamma_k^{\ell}\stf_{k}^{(E)}\chi_{\ell+1}}_{L^2}^{1/2} \norm{\partial_y\paren{q^{\ell}\Gamma_k^{\ell}\stf_{k}^{(E)}\chi_{\ell+1}}}_{L^2}^{1/2}
		\\&\quad+
		kt\sum_{\ell=0}^{n-m-1}\binom{n-m}{\ell}\norm{q^{n-m-\ell}\bd^{n-m-\ell}(v_y-1)\chi_{n-m-\ell+1}}_{L^2}^{1/2}
				\\&\quad\times
		\norm{\partial_y\paren{q^{n-m-\ell}\bd^{n-m-\ell}(v_y-1)\chi_{n-m-\ell+1}}}_{L^2}^{1/2}
		\\&\quad\times
		\norm{\partial_y\paren{q^{\ell}\Gamma_k^{\ell}\stf_{k}^{(E)}}\chi_{\ell+1}}_{L^2}^{1/2} \norm{\partial_y\paren{\partial_y\paren{q^{\ell}\Gamma_k^{\ell}\stf_{k}^{(E)}}\chi_{\ell+1}}}_{L^2}^{1/2}
		\\&=: V^{LH, E}_{2,1} + V^{LH, E}_{2,2}
	\end{align*}
	from where we further arrive at
	\begin{align*}
		N_{11,v2} ^{LH, E}&\lesssim\brak{t}^{3+2\ss} t^{-1} \bold{a}_{n+1}^2 \sum_{k} \sum_{m\ge n/2}^{n}\binom{n}{m}
		\enorm{  \overline{q^{m}\Gamma_{k}^{m} ik \omega_k }e^{W} \chi_n}  \enorm{\partial_y\bhqn  \chi_n}
		\paren{ V^{LH, E}_{2,1} + V^{LH, E}_{2,2}}
		\\&=: 	N_{11,v2,1} ^{LH, E} + N_{11,v2,2} ^{LH, E}.
	\end{align*}
The treatment of the two pieces are in similar manner and we only show details for
	\begin{align}
		\sum_{n\ge2}&(n+1)^{2\sss-2}N_{11,v2,1} ^{LH, E}\nn
		&\lesssim
		\sum_{n\ge2} \sum_{k} \sum_{m\ge n/2}^{n}\binom{n}{m} \bold{a}_{n+1} \bold{a}_{1,m}^{-1} 
		\bold{a}_{n-m-\ell}^{-1}\bold{a}_{\ell}^{-1}
		(n+2)^{\sss-1} \bold{a}_{1,m} k\enorm{ q^{m}\Gamma_{k}^{m} \overline{\omega_k }e^{W} \chi_{m}} 
		\nn&\quad\times 
		\brak{t}^{3/2+\ss} (n+2)^{\sss-1}\bold{a}_{n+1}\enorm{\partial_y\bhqn  \chi_n}
		\nn&\quad\times
		kt\brak{t}^{3/2+\ss} t^{-1}\sum_{\ell=0}^{n-m-1} \bold{a}_{n-m-\ell}\binom{n-m}{\ell}\norm{\partial_y\paren{q^{n-m-\ell}\bd^{n-m-\ell}(v_y-1)}\chi_{n-m-\ell+1}}_{L^2}^{1/2}
		\nn&\quad\times
		\nu^{1/2}\norm{\partial_y\paren{\partial_y\paren{q^{n-m-\ell}\bd^{n-m-\ell}(v_y-1)}\chi_{n-m-\ell+1}}}_{L^2}^{1/2}
		\bold{a}_{\ell}\norm{q^{\ell}\Gamma_k^{\ell}\stf_{k}^{(E)}\chi_{\ell+1}}_{L^2}^{1/2}  
		\nn&\quad\times
		\nu^{1/2}\nu^{-1}\norm{\partial_y\paren{q^{\ell}\Gamma_k^{\ell}\stf_{k}^{(E)}\chi_{\ell+1}}}_{L^2}^{1/2}.
	\end{align}
	Following line by line as in Lemma~\ref{nonl:barH}, we use Lemma~\ref{comb:boun}, 
   Young's inequality, and $
	\brak{t}^2\nu^{-1} \lesssim e^W\ \mbox{for}\ t\lesssim \nu^{-1/3-\delta},$ to obtain
	\begin{align*}
			\sum_{n\ge2}(n+1)^{2\sss-2}N_{11,v2,1} ^{LH, E} 
			&\lesssim \nu^{10}
			\sqrt{\cd^{\gm}} \norm{\overline{H}}_{Y_{1,-s}} 
				\norm{H}_{\overline{Y}_{1,0}}^{1/2}\paren{\cd_{H}^{(\alpha)}}^{1/4}
				\paren{\sum_{j=1, 2}\sqrt{\mathcal{J}_{ell}^{(i)}}}
			\\&
		\lesssim \eps \sqrt{\cd^{\gm}}  	\paren{\mathcal{CK}_{\overline{H}}^{(\alpha)}}^{1/4}\paren{\cd_{H}^{(\alpha)}}^{1/4},
	\end{align*}
	where we also used H\"older's inequality. The treatment of $N_{11,v2,2} ^{LH, E}$ is in the same philosophy. Moreover,
	proceeding step by step as in Lemma~\ref{nonl:barH} with slight modification, we could bound $N_{11,v3} ^{LH, E}$ and 
	$N_{11,v4} ^{LH, E}$. Therefore, we omit further 
	details to turn to $N_{12} ^{LH, E}$. Observe that
\begin{align*}
	\abs{	N_{12} ^{LH, E}} &\lesssim
	\brak{t}^{3+2\ss} t^{-1} \bold{a}_{n+1}^2 \sum_{k} \sum_{m\ge n/2}^{n}\binom{n}{m} 
	\norm{q^{n-m}\Gamma_{k}^{n-m}\partial_y \stf_{k}^{(E)}\chi_{n-1}}_{L^\infty}
	\nn&\quad\times
	\abs{k}\enorm{ \partial_y \paren{q^{m}\Gamma_{k}^{m}  \overline{\omega_k }}e^{W} \chi_n}  \enorm{\partial_y\bhqn  \chi_n}.
\end{align*}
Compared to the corresponding $	N_{1} ^{LH, E}$ term in Lemma~\ref{nonl:barH}, the only difference is one extra $y$ derivative 
in the $\omega$ involving term. Hence, we may apply exactly the same procedure to estimate $N_{12} ^{LH, E}$, except replacing 
$\mathcal{E}^{(\gamma)}$, $\mathcal{CK}^{\gm}$, and $\cd^{\gm}$ by $\mathcal{E}^{(\alpha)}$, $\mathcal{CK}_\omega^{(\alpha)}$, and $\cd_\omega^{(\alpha)}$ respectively.  
	Next we deal with the $HL$ part as
	\begin{align*}
		\abs{N_{11} ^{HL, E}} 
		&
		=  \brak{t}^{3+2\ss} t^{-1} \bold{a}_{n+1}^2 \sum_{k}\sum_{m\le n/2}\binom{n}{m} 
		\Big \langle \partial_y \paren{q^{n-m}\Gamma_{k}^{n-m}\partial_y \stf_{k}^{(E)}  \overline{q^{m}\Gamma_{k}^{m} ik \omega_k }} ,  
		\partial_y \bhqn  \chi_n^2 \Big \rangle 
		\\&\lesssim
		\brak{t}^{3+2\ss} t^{-1} \bold{a}_{n+1}^2 \sum_{k}\sum_{m\le n/2}\binom{n}{m}  
		\norm{\partial_y\paren{q^{n-m}\Gamma_{k}^{n-m} k\partial_y \stf_{k}^{(E)}}\chi_{n}}_{L^2} 
		\\&\quad\times
		\norm{q^{m}\Gamma_{k}^{m} \overline{\omega_k }e^{W} \chi_{n-1} }_{L^\infty}\enorm{\partial_y\bhqn  \chi_n}
		\\&\quad+
		\brak{t}^{3+2\ss} t^{-1} \bold{a}_{n+1}^2 \sum_{k}\sum_{m\le n/2}\binom{n}{m}  \norm{q^{n-m}\Gamma_{k}^{n-m} k\partial_y \stf_{k}^{(E)}\chi_{n}}_{L^2} 
		\\&\quad\times
		\norm{\partial_y \paren{q^{m}\Gamma_{k}^{m} \overline{\omega_k }}e^{W} \chi_{n-1} }_{L^\infty}\enorm{\partial_y\bhqn  \chi_n}
		\\&\lesssim
		\brak{t}^{3+2\ss} t^{-1} \bold{a}_{n+1}^2 \sum_{k}\sum_{m\le n/2}\binom{n}{m}  \norm{\partial_y \paren{q^{n-m}\Gamma_{k}^{n-m} k\partial_y \stf_{k}^{(E)}}\chi_{n}}_{L^2} \enorm{\partial_y\bhqn  \chi_n}
		\\&\quad\times
		\norm{q^{m}\Gamma_{k}^{m} \overline{\omega_k }e^{W} \chi_{m+1} }_{L^2}^{1/2}\norm{\partial_y\paren{q^{m}\Gamma_{k}^{m} \overline{\omega_k }e^{W} \chi_{m+1}} }_{L^2}^{1/2}
		\\&\quad+
		\brak{t}^{3+2\ss} t^{-1} \bold{a}_{n+1}^2 \sum_{k}\sum_{m\le n/2}\binom{n}{m}  \norm{q^{n-m}\Gamma_{k}^{n-m} k\partial_y \stf_{k}^{(E)}\chi_{n}}_{L^2} \enorm{\partial_y\bhqn  \chi_n}
		\\&\quad\times
		\norm{\partial_y \paren{q^{m}\Gamma_{k}^{m} \overline{\omega_k }}e^{W} \chi_{m+1} }_{L^2}^{1/2}
		\norm{\partial_y\paren{\partial_y \paren{q^{m}\Gamma_{k}^{m} \overline{\omega_k }}e^{W} \chi_{m+1}} }_{L^2}^{1/2}.
	\end{align*}
	From here on, the estimates of the ``High-Low" interaction are in a similar fashion and completely parallel with the corresponding part in Lemma~\ref{nonl:barH}, except some relevant parts upgrading to one derivative higher regularity, i.e.,  $\gamma$
	 terms replaced by $\alpha$ terms, of which further details are omitted.
The term $N_{2}^{E}$ could be dealt with essentially the same way as $N_{1}^{E}$ except some minor technicality differences. For instance, we may encounter a dangerous term  
where all $y$ derivatives fall on 
$\omega$, i.e., $\partial_y^3 \omega_{k, n}$. However, using 
\begin{align*}
	v_y^{-1}  \partial_{y} = \Gamma_k - ikt,
\end{align*}
we may easily convert this term into a situation treated above 
and we conclude the Exterior part without showing more details.
\\{\bf Interior part: } We turn to the interior part next. We note
\begin{align*}
	N^I &
	= -\frac{1}{t} \sum_{k} \Big \langle  
	\partial_y \paren{q^n\Gamma_{0}^n \paren{ \Gamma_k \stf_{k}^{(I)} \overline{ik \omega_k } }}, \brak{t}^{3+2\ss}  \bold{a}_{n+1}^2  \partial_y \bhqn  e^{W} \chi_n^2 \Big \rangle 
	\\&\quad
	+\frac{1}{t} \sum_{k} \Big \langle  \partial_y \paren{q^n\Gamma_{0}^n \paren{ik \stf_{k}^{(I)} \Gamma_k \overline{\omega_k}}}, 
	 \brak{t}^{3+2\ss}  \bold{a}_{n+1}^2  \partial_y \bhqn  e^{W} \chi_n^2 \Big \rangle 
	\\&
	=  \frac{1}{t} \sum_{k}\sum_{m=0}^{n}\binom{n}{m}\Big \langle \partial_y
	\paren{ q^{n-m}\Gamma_{k}^{n-m}\Gamma_k \stf_{k}^{(I)}  \overline{q^{m}\Gamma_{k}^{m} ik \omega_k }} , \brak{t}^{3+2\ss}  \bold{a}_{n+1}^2  \partial_y \bhqn  e^{W} \chi_n^2 \Big \rangle 
	\\&\quad
	+\frac{1}{t} \sum_{k}\sum_{m=0}^{n}\binom{n}{m}\Big \langle  \partial_y 
	\paren{q^{n-m}\Gamma_{k}^{n-m}ik \stf_{k}^{(I)} q^{m}\Gamma_{k}^{m} \Gamma_k \overline{\omega_k}} , 
	\brak{t}^{3+2\ss}  \bold{a}_{n+1}^2 \partial_y \bhqn  e^{W} \chi_n^2 \Big \rangle 
	\\&
	= N_{1}^I + N_{2}^I.
\end{align*}
As in the exterior case, we divide $N_{1}^I $ into two pieces:
\begin{align*}
	N_{1}^{I} &=  \brak{t}^{3+2\ss} t^{-1} \bold{a}_{n+1}^2 \sum_{k}\paren{\sum_{m\le n/2} + \sum_{m\ge n/2}^{n}}
	\binom{n}{m} \Big \langle \partial_y \paren{q^{n-m}\Gamma_{k}^{n-m}\Gamma_k \stf_{k}^{(I)}}  \overline{q^{m}\Gamma_{k}^{m} ik \omega_k } ,   \partial_y\bhqn  e^{W} \chi_n^2 \Big \rangle 
	\\&\quad +  \brak{t}^{3+2\ss} t^{-1} \bold{a}_{n+1}^2 \sum_{k}\paren{\sum_{m\le n/2} + \sum_{m\ge n/2}^{n}}
	\binom{n}{m} \Big \langle q^{n-m}\Gamma_{k}^{n-m}\Gamma_k \stf_{k}^{(I)} \partial_y  \paren{\overline{q^{m}\Gamma_{k}^{m} ik \omega_k} } ,   \partial_y\bhqn  e^{W} \chi_n^2 \Big \rangle
	\\& = N_{11} ^{HL, I} + N_{11} ^{LH, I} + N_{12} ^{HL, I} + N_{12} ^{LH, I}.
\end{align*}
Denote
\begin{align*}
	\norm{f}_{H^1_{\mu}} = \enorm{f \mu} + \enorm{\partial_{y} f \mu}.
\end{align*}
Then for the LH part in the first piece, we have
\begin{align*}
	N_{11} ^{LH, I} &\lesssim \brak{t}^{3+2\ss} t^{-1} \bold{a}_{n+1}^2 \sum_{k} \sum_{m\ge n/2}^{n}
	\binom{n}{m}     \norm{ q^{n-m}\Gamma_k^{n-m+1}\stf_{k}^{(I)}}_{H^1_{\chi_{n-m+1}}} ^{1/2}
	\\&\quad\times
	\norm{q^{n-m}\Gamma_k^{n-m+1} \stf_{k}^{(I)}}_{H^2_{\chi_{n-m+1}}}^{1/2}
	\enorm{  \overline{q^{m}\Gamma_{k}^{m} ik \omega_k }e^{W} \chi_n}  \enorm{\partial_{y} \bhqn  \chi_n} 
	\\&\quad+
	\brak{t}^{3+2\ss} t^{-1} \bold{a}_{n+1}^2 \sum_{k} \sum_{m\ge n/2}^{n} \binom{n}{m}   
	(n-m)^{(1+\sigma)/2} \norm{q^{n-m}\Gamma_k^{n-m+1} \stf_{k}^{(I)}}_{H^1_{\chi_{n-m+1}}}
	\\&\quad\times
	\enorm{  \overline{q^{m}\Gamma_{k}^{m} ik \omega_k }e^{W} \chi_n}  \enorm{\partial_{y} \bhqn  \chi_n} .
\end{align*}
Multiplying it by $(n+1)^{2\sss-2}$ and summing in $n$ gives
\begin{align*}
	\sum_{n\ge2}&(n+1)^{2\sss-2}N_{11} ^{LH, I}\\
	&\lesssim
	\sum_{n\ge0}\sum_{m\ge n/2}^{n}\sum_{k}  \bold{a}_{n+1} \binom{n}{m} \bold{a}_{n-m}^{-1} \bold{a}_{1,m}^{-1}  
    \brak{t}^{3/2+\ss}  \bold{a}_{n-m} 
    	\\&\quad \times   \norm{ q^{n-m}\Gamma_k^{n-m+1}\stf_{k}^{(I)}}_{H^1_{\chi_{n-m+1}}} ^{1/2}
    	\norm{q^{n-m}\Gamma_k^{n-m+1} \stf_{k}^{(I)}}_{H^2_{\chi_{n-m+1}}}^{1/2}
	\\&\quad \times
	(n+2)^{\sss-3/2}\bold{a}_{1,m}t^{-1/2} \enorm{ k q^{m}\Gamma_{k}^{m} \overline{\omega_k }e^{W} \chi_m}  
	\brak{t}^{3/2+\ss} t^{-1/2}(n+2)^{\sss-1/2}\bold{a}_{n+1}\enorm{\partial_{y} \bhqn  \chi_n}
	\\&\quad + 
	\sum_{n\ge0}\sum_{m\ge n/2}^{n}\sum_{k}  \bold{a}_{n+1} \binom{n}{m} \bold{a}_{n-m}^{-1} \bold{a}_{1,m}^{-1}
	\brak{t}^{3/2+\ss}
	\bold{a}_{n-m}   (n-m)^{(1+\sigma)/2} \norm{ q^{n-m}\Gamma_k^{n-m+1}\stf_{k}^{(I)}}_{H^1_{\chi_{n-m+1}}}
	\\&\quad \times
	(n+2)^{\sss-3/2}\bold{a}_{1,m} t^{-1/2} \enorm{ k q^{m}\Gamma_{k}^{m} \overline{\omega_k }e^{W} \chi_m}  
	\brak{t}^{3/2+\ss} t^{-1/2}(n+2)^{\sss-1/2}\bold{a}_{n+1}\enorm{\partial_{y} \bhqn  \chi_n} .
\end{align*}
By Corollary~\ref{comb:boun:vari:1}, the elliptic estimate~\eqref{eell:out}, and Lemma~\ref{con:no:k},
the above quantity could be easily bounded by 
\begin{align*}
	\sum_{n\ge2}&(n+1)^{2\sss-2} N_{11} ^{LH, I} 
	\lesssim 
	\nu ^{99}\sqrt{\mathcal{E}_{ell}^{(I, out)}} 
	\sqrt{\mathcal{CK}^{\gm}} \nu^{-1/2} \sqrt{\mathcal{CK}_{\overline{H}}^{(\alpha)}}
	\lesssim \eps 
	\sqrt{\mathcal{CK}^{\gm}} \sqrt{\mathcal{CK}_{\overline{H}}^{(\alpha)}}.
\end{align*}
Next, for the $HL$ part, we similarly deduce
\begin{align*}
	\sum_{n\ge2}&(n+1)^{2\sss-2}N_{11} ^{HL, I}\\
	&\lesssim
	\sum_{n\ge0}\sum_{m\le n/2}\sum_{k}  \bold{a}_{n+1} \binom{n}{m} \bold{a}_{n-m+2}^{-1} \bold{a}_{1,m-2}^{-1}   (n+2)^{\sss-3/2}
	\\&\quad \times \brak{t}^{3/2+\ss}  \bold{a}_{n-m+2} \norm{q^{n-m} \Gamma_k^{n-m+1}\stf_{k}^{(I)}}_{H^1_{\chi_{n-m+1}}} ^{1/2}
	\norm{q^{n-m}\Gamma_k^{n-m+1} \stf_{k}^{(I)}}_{H^2_{\chi_{n-m+1}}} ^{1/2}
	\\&\quad \times
	\bold{a}_{1,m-2}t^{-1/2} \enorm{ k q^{m}\Gamma_{k}^{m} \overline{\omega_k }e^{W} \chi_m}  
	\brak{t}^{3/2+\ss} t^{-1/2}(n+2)^{\sss-1/2}\bold{a}_{n+1}\enorm{\partial_{y} \bhqn  \chi_n}
	\\&\quad + 
	\sum_{n\ge0}\sum_{m\ge n/2}^{n}\sum_{k}  \bold{a}_{n+1} \binom{n}{m} \bold{a}_{n-m+1}^{-1} \bold{a}_{1,m-1}^{-1}
	\brak{t}^{3/2+\ss}
	\bold{a}_{n-m+1}   (n-m)^{(1+\sigma)/2} \norm{ q^{n-m}\Gamma_k^{n-m+1} \stf_{k}^{(I)}}_{H^1_{\chi_{n-m}}} 
	\\&\quad \times
	(n+2)^{\sss-3/2}\bold{a}_{1,m-1} t^{-1/2} \enorm{ k q^{m}\Gamma_{k}^{m} \overline{\omega_k }e^{W} \chi_m}  
	\brak{t}^{3/2+\ss} t^{-1/2}(n+2)^{\sss-1/2}\bold{a}_{n+1}\enorm{\partial_{y} \bhqn  \chi_n} .
\end{align*}
from where we again arrive at
\begin{align*}
	\sum_{n\ge2}&(n+1)^{2\sss-2} N_{11} ^{HL, I} 
	\lesssim \nu^{100} \sqrt{\mathcal{E}_{ell}^{(I, out)}} 
	\sqrt{\mathcal{CK}^{\gm}} \nu^{-1/2} \sqrt{\mathcal{CK}_{\overline{H}}^{(\alpha)}}
	\lesssim \eps
	\sqrt{\mathcal{CK}^{\gm}} \sqrt{\mathcal{CK}_{\overline{H}}^{(\alpha)}}.
\end{align*}
The treatment of $N_{12} ^{HL, I}$ and $N_{12} ^{LH, I}$ is the same as that of the interior part of Lemma~\ref{nonl:barH} except upgrading to one derivative higher regularity, 
\begin{align*}
	\sum_{n\ge2}&(n+1)^{2\sss-2} N_{12} ^{HL, I} 
	\lesssim \nu^{100} \sqrt{\mathcal{E}_{ell}^{(I, out)}}  \nu^{-1}
	\sqrt{\mathcal{CK}^{(\alpha)}}  \sqrt{\mathcal{CK}_{\overline{H}}^{(\alpha)}}
	\lesssim \eps 
	\sqrt{\mathcal{CK}^{(\alpha)}}  \sqrt{\mathcal{CK}_{\overline{H}}^{(\alpha)}}
	\\
	\sum_{n\ge2}&(n+1)^{2\sss-2} N_{12} ^{LH, I} 
	\lesssim \nu^{100}
  \sqrt{\mathcal{E}_{ell}^{(I, out)}}  \nu^{-1}
	\sqrt{\mathcal{CK}^{(\alpha)}} \sqrt{\mathcal{CK}_{\overline{H}}^{(\alpha)}}
	\lesssim \eps 
	\sqrt{\mathcal{CK}^{(\alpha)}} \sqrt{\mathcal{CK}_{\overline{H}}^{(\alpha)}}.
\end{align*}
 of which further details are omitted. 
\\{\bf (2). Case $n\le  1$:} We only show the estimate for $N^{(1)}_{\overline{H},1} $. As in Lemma~\ref{nonl:barH}, we divide it into
 the interior and exterior parts (note here the decomposition is different from the above situation in this lemma, where $\stf$ is divided into the interior and exterior parts)
\begin{align*}
	N^{(1)} &= N_1^{E, 1} + N_1^{I,1},
\end{align*}
where the estimate of $N_1^{E, 1}$ is standard using the localization property of both $\omega, \overline H$ and hence we only focus on the interior one
\begin{align*}
		N_1^{(I, 1)} &= \frac{1}{t} \brak{t}^{3+2\ss}  \bold{a}_{2}^2 
		\Big \langle \partial_{y} (q\Gamma_0) \paren{  \nabla^{\perp} \stf_{\neq 0} \cdot \nabla \omega }_0, 
		 \partial_{y} \mathring{\overline{H}}_1 e^{W} \chi_1^2 (1-\chi_2^2)\Big \rangle 
		\\& = 
		-\sum_{k} \frac{1}{t} \brak{t}^{3+2\ss}  \bold{a}_{2}^2  \Big \langle \partial_{y} (q\Gamma_0) \paren{ \Gamma_k \stf_{k} \overline {ik \omega_k }}, 
      \partial_{y}  \mathring{\overline{H}}_1 e^{W} \chi_1^2 (1-\chi_2^2)\Big \rangle 
		\\&\qquad + \sum_{k} \frac{1}{t} \brak{t}^{3+2\ss}  \bold{a}_{2}^2 
		\Big \langle \partial_{y} (q\Gamma_0)  \paren{ ik \stf_{k} \Gamma_k \omega_k }, 
       \partial_{y} \mathring{\overline{H}}_1 e^{W} \chi_1^2 (1-\chi_2^2)\Big \rangle 
		\\&
		= 	- N_{1}^{(I, 1)} + N_{2}^{(I, 1)} .
\end{align*}
Distributing the derivatives, we have the further decomposition
\begin{align*}
	N_{1}^{(I, 1)} =& 	-\sum_{k} \frac{1}{t} \brak{t}^{3+2\ss}  \bold{a}_{2}^2  \Big \langle \partial_{y} (q\Gamma_k) \Gamma_k \stf_{k} \overline{ik \omega_k}, 
	\partial_{y}  \mathring{\overline{H}}_1 e^{W} \chi_1^2 (1-\chi_2^2)\Big \rangle \\&
		-\sum_{k} \frac{1}{t} \brak{t}^{3+2\ss}  \bold{a}_{2}^2  \Big \langle \partial_{y}  \Gamma_k \stf_{k} \overline{ik (q\Gamma_k) \omega_k} , 
	\partial_{y}  \mathring{\overline{H}}_1 e^{W} \chi_1^2 (1-\chi_2^2)\Big \rangle 	
	\\&-\sum_{k} \frac{1}{t} \brak{t}^{3+2\ss}  \bold{a}_{2}^2  \Big \langle  (q\Gamma_k)  \Gamma_k \stf_{k}\overline{ ik \partial_{y}\omega_k }, 
	\partial_{y}  \mathring{\overline{H}}_1 e^{W} \chi_1^2 (1-\chi_2^2)\Big \rangle 	
	\\&-\sum_{k} \frac{1}{t} \brak{t}^{3+2\ss}  \bold{a}_{2}^2  \Big \langle  \Gamma_k \stf_{k} \overline{ik \partial_{y} (q\Gamma_k)\omega_k} , 
	\partial_{y}  \mathring{\overline{H}}_1 e^{W} \chi_1^2 (1-\chi_2^2)\Big \rangle 
	\\=& N_{11}^{(I, 1)} + N_{12}^{(I, 1)} +N_{13}^{(I, 1)} +N_{14}^{(I, 1)} .
\end{align*}
The first term is easily treated as 
\begin{align*}
		|N_{11}^{(I, 1)}| &\le 	\frac{1}{t} 	\brak{t}^{3+2\ss}  \bold{a}_{2}^2
		\sum_{k}     \norm{ \partial_{y} (q\Gamma_k)   \Gamma_k \stf_{k}\sqrt{(1-\chi_2^2)} }_{L^\infty} 
		\enorm{ik \omega_k e^{W} \chi_1}
		\enorm{\partial_{y} \mathring{\overline{H}}_1 e^{W} \chi_1}
		\\&\lesssim 	\frac{1}{t} 	\brak{t}^{3+2\ss}  \bold{a}_{2}^2
		 \sum_{k}     \norm{ \partial_{y} (q\Gamma_k)   \Gamma_k \stf_{k} \sqrt{(1-\chi_2^2)}  }_{L^2}^{1/2}  \norm{ \Gamma_k \partial_{y} (q\Gamma_k)   \Gamma_k \stf_{k} \sqrt{(1-\chi_2^2)}  }_{L^2}^{1/2} 
      \\&\qquad \times
		 \enorm{ik \omega_k e^{W} \chi_1}
		 \enorm{\partial_{y} \mathring{\overline{H}}_1  \chi_1  \sqrt{(1-\chi_2^2)} }.	
\end{align*}
Using $\bold{a}_2 \lesssim \bold{a}_1 \frac{1}{\brak{t}}$ and the elliptic estimate~\eqref{eell:sob}, we further arrive at
\begin{align*}
	|N_{11}^{(I, 1)}| 
	\lesssim \frac{\eps}{\brak{t}^{4+r-2\ss}} \bold{a}_{1}	\enorm{\partial_{x} \omega e^{W} \chi_1}
		\langle t \rangle^{1+\ss} \enorm{\partial_{y} \mathring{\overline{H}}_1  \sqrt{(1-\chi_2^2)}  \chi_1}  \lesssim \frac{\eps^3}{\brak{t}^{4+r-2\ss}}.
\end{align*}
The second term is treated as 
\begin{align*}
		|N_{12}^{(I, 1)}| &=\sum_{k} \frac{1}{t} \brak{t}^{3+2\ss}  \bold{a}_{2}^2  \Big|\Big \langle \partial_{y} \Gamma_k \stf_{k} \overline{ik (q\Gamma_k) \omega_k}, 
	\partial_{y}  \mathring{\overline{H}}_1e^{W} \chi_1^2 (1-\chi_2^2)\Big \rangle \Big|
	\\& \lesssim 	\frac{1}{t} 	\brak{t}^{3+2\ss}  \bold{a}_{2}^2
	\sum_{k}     \norm{ \partial_{y} ik  \Gamma_k \stf_{k}  \sqrt{(1-\chi_2^2)}  }_{L^\infty} 
	\enorm{ (q\Gamma_k)  \omega_k e^{W} \chi_1}
	\enorm{\partial_{y} \mathring{\overline{H}}_1  \sqrt{(1-\chi_2^2)}  \chi_1} 
	\\& \lesssim 	\frac{\eps}{\brak{t}^{4+r-2\ss}}
	\bold{a}_{1} \enorm{ (q\Gamma_k)  \omega e^{W} \chi_1}
	 	\langle t \rangle^{1+\ss} \bold{a}_{2}  \enorm{\partial_{y} \mathring{\overline{H}}_1  \sqrt{(1-\chi_2^2)}  \chi_1} 	
	 \\& \lesssim \frac{\eps^3}{\brak{t}^{4+r-2\ss}}.
\end{align*}
 Similar argument gives
 \begin{align*}
 		|N_{13}^{(I, 1)}| =&	
 		\sum_{k} \frac{1}{t} \brak{t}^{3+2\ss}  \bold{a}_{2}^2 \Big| \Big \langle  (q\Gamma_k)  \Gamma_k \stf_{k} \overline{ik \partial_{y}\omega_k }, 
 		\partial_{y}  \mathring{\overline{H}}_1 e^{W} \chi_1^2 (1-\chi_2^2)\Big \rangle 	\Big|
 		\\\lesssim &
 			 			\nu^{-1/2} \sum_{k} \frac{1}{t} \brak{t}^{3+2\ss}  \bold{a}_{2}^2 \norm{  |k|^{1/3}   ik \Gamma_k ^2 \stf_{k}  \sqrt{(1-\chi_2^2)}  }_{L^\infty} 
 			\nu^{1/3} |k|^{-1/3} \enorm{ \partial_{y}  \omega_k e^{W} \chi_1}
 			\\&\times
 			\enorm{\partial_{y} \mathring{\overline{H}}_1  \sqrt{(1-\chi_2^2)}  \chi_1} 
 		\\\lesssim & \nu^{-1/2}  \frac{\eps^3}{\brak{t}^{4+r-2\ss}}.
 \end{align*}
and 
 \begin{align*}
	|N_{14}^{(I, 1)}| =&	
	\sum_{k} \frac{1}{t} \brak{t}^{3+2\ss}  \bold{a}_{2}^2 \Big| \Big \langle  (q\Gamma_k)  \Gamma_k \stf_{k}\overline{ ik \partial_{y}\omega_k }, 
	\partial_{y}  \mathring{\overline{H}}_1 e^{W} \chi_1^2 (1-\chi_2^2)\Big \rangle 	\Big|
	\\\lesssim &
	\nu^{-1/2} \sum_{k} \frac{1}{t} \brak{t}^{3+2\ss}  \bold{a}_{2}^2 \norm{  |k|^{1/3} \Gamma_k \stf_{k}(1-\chi_2^2) }_{L^\infty} 
	\nu^{1/3} |k|^{-1/3} \enorm{ \partial_{y}  ik  (q\Gamma_0)  \omega_k e^{W} \chi_1}
	 			\\&\times
	\enorm{\partial_{y} \mathring{\overline{H}}_1  \sqrt{(1-\chi_2^2)}  \chi_1} 
	\\\lesssim & \nu^{-1/2}  \frac{\eps^3}{\brak{t}^{4+r-2\ss}}.
\end{align*}
Next we turn to $N_{2}^{(I, 1)} $,  which may be rewritten as
\begin{align*}
	N_{2}^{(I, 1)} =& 	-\sum_{k} \frac{1}{t} \brak{t}^{3+2\ss}  \bold{a}_{2}^2  \Big \langle \partial_{y} (q\Gamma_k) ik \stf_{k} \overline{\Gamma_k \omega_k}, 
	\partial_{y}  \mathring{\overline{H}}_1 e^{W} \chi_1^2 (1-\chi_2^2)\Big \rangle \\&
	-\sum_{k} \frac{1}{t} \brak{t}^{3+2\ss}  \bold{a}_{2}^2  \Big \langle \partial_{y}  ik \stf_{k} \overline{ (q\Gamma_k) \Gamma_k \omega_k} , 
	\partial_{y}  \mathring{\overline{H}}_1 e^{W} \chi_1^2 (1-\chi_2^2)\Big \rangle 	
	\\&-\sum_{k} \frac{1}{t} \brak{t}^{3+2\ss}  \bold{a}_{2}^2  \Big \langle  (q\Gamma_k)  ik \stf_{k}\overline{  \partial_{y}\Gamma_k \omega_k }, 
	\partial_{y}  \mathring{\overline{H}}_1 e^{W} \chi_1^2 (1-\chi_2^2)\Big \rangle 	
	\\&-\sum_{k} \frac{1}{t} \brak{t}^{3+2\ss}  \bold{a}_{2}^2  \Big \langle ik \stf_{k} \overline{ \partial_{y} (q\Gamma_k) \Gamma_k \omega_k} , 
	\partial_{y}  \mathring{\overline{H}}_1 e^{W} \chi_1^2 (1-\chi_2^2)\Big \rangle 
	\\=& N_{21}^{(I, 1)} + N_{22}^{(I, 1)} +N_{23}^{(I, 1)} +N_{24}^{(I, 1)} .
\end{align*}
By H\"older's inequality, one has
\begin{align*}
	|N_{21}^{(I, 1)}| &\le 	\frac{1}{t} 	\brak{t}^{3+2\ss}  \bold{a}_{2}^2
	\sum_{k}     \norm{ \partial_{y} (q\Gamma_k)  k \stf_{k} \sqrt{(1-\chi_2^2)}  }_{L^\infty} 
	\enorm{\Gamma_k \omega_k e^{W} \chi_1}
	\enorm{\partial_{y} \mathring{\overline{H}}_1  \sqrt{(1-\chi_2^2)}  \chi_1}
	\\&\lesssim 	\frac{1}{t} 	\brak{t}^{3+2\ss}  \bold{a}_{2}^2
	\sum_{k}     \norm{ k \partial_{y} (q\Gamma_k)  \stf_{k}(1-\chi_2^2) }_{L^2}^{1/2}  \norm{k \Gamma_k  \partial_{y} (q\Gamma_k)  \stf_{k}(1-\chi_2^2) }_{L^2}^{1/2} 
	\\&\qquad \times
	\enorm{\partial_{y} \omega_k e^{W} \chi_1}
	\enorm{\partial_{y} \mathring{\overline{H}}_1  \sqrt{(1-\chi_2^2)}  \chi_1}
	\\&\quad 
	+ \frac{1}{t} 	\brak{t}^{3+2\ss}  \bold{a}_{2}^2
	\sum_{k}  |k|t   \norm{ k \partial_{y} (q\Gamma_k)  \stf_{k}(1-\chi_2^2) }_{L^2}^{1/2}  \norm{ k \Gamma_k \partial_{y} (q\Gamma_k)   \stf_{k}(1-\chi_2^2) }_{L^2}^{1/2} 
	\\&\qquad \times
	\enorm{\omega_k e^{W} \chi_1}
	\enorm{\partial_{y} \mathring{\overline{H}}_1  \sqrt{(1-\chi_2^2)}  \chi_1}.	
\end{align*}
Using $\bold{a}_2\lesssim \bold{a}_0 \brak{t}^{-2}$, inviscid damping estimate~\eqref{eell:sob} and \eqref{fell:e}, and the bootstrap assumption, we deduce that
\begin{align*}
	|N_{21}^{(I, 1)}| &\le \nu^{-1/2}  \frac{\eps^3}{\brak{t}^{4+r-2\ss}}.
\end{align*}
The second term is already a bit tricky and we proceed as
\begin{align*}
	|N_{22}^{(I, 1)}| &\le 	\frac{1}{t} 	\brak{t}^{3+2\ss}  \bold{a}_{2}^2
	\sum_{k}     \norm{k \partial_{y}  \stf_{k}  \sqrt{(1-\chi_2^2)}  }_{L^\infty} 
	\enorm{\Gamma_k(q\Gamma_k)  \omega_k e^{W} \chi_1}
	\enorm{\partial_{y} \mathring{\overline{H}}_1  \sqrt{(1-\chi_2^2)}  \chi_1}
	\\&\quad + 	\frac{1}{t} 	\brak{t}^{3+2\ss}  \bold{a}_{2}^2
	\sum_{k}     \norm{k \partial_{y} \stf_{k}  \sqrt{(1-\chi_2^2)}  }_{L^\infty} 
	\enorm{\Gamma_k \omega_k e^{W} \chi_1}
	\enorm{\partial_{y} \mathring{\overline{H}}_1  \sqrt{(1-\chi_2^2)}  \chi_1}
\end{align*}
where we used \eqref{cm_pv_qn} for $n=1$.
Note the second term in the above formula can be treated as $N_{21}^{(I, 1)}$ and we only show how to 
handle the first one. We have
\begin{align*}
  	\frac{1}{t} 	\brak{t}^{3+2\ss}&  \bold{a}_{2}^2
	\sum_{k}     \norm{k \partial_{y}  \stf_{k}  \sqrt{(1-\chi_2^2)}  }_{L^\infty} 
	\enorm{\Gamma_k(q\Gamma_k)  \omega_k e^{W} \chi_1}
	\enorm{\partial_{y} \mathring{\overline{H}}_1  \sqrt{(1-\chi_2^2)} \chi_1}
	\\&\lesssim 	\frac{1}{t} 	\brak{t}^{3+2\ss}  \bold{a}_{2}
	\sum_{k}     \norm{ k \partial_{y} \stf_{k}  \sqrt{(1-\chi_2^2)} }_{L^2}^{1/2}  \norm{k \Gamma_k  \partial_{y}  \stf_{k}  \sqrt{(1-\chi_2^2)}  }_{L^2}^{1/2} 
	\\&\qquad \times
	\bold{a}_{2}\enorm{\Gamma_k (q\Gamma_k)\omega_k e^{W} \chi_1}
	\enorm{\partial_{y} \mathring{\overline{H}}_1  \sqrt{(1-\chi_2^2)}  \chi_1}
	\\&\lesssim \nu^{-1/2} \nu^{1/2}  \frac{\eps^3}{\brak{t}^{3+r-2\ss}} 
	\lesssim \nu^{-1/2}  \frac{\eps^3}{\brak{t}^{4+r-2\ss+\delta}}
\end{align*}
for some tiny parameter $\delta>0$, where we also used $t\lesssim \nu^{-1/3-\gamma}$.
Next we treat the most dangerous term $N_{24}^{(I, 1)}$ and the estimate for $N_{23}^{(I, 1)}$ is strictly simpler and essentially included in it. Noting
\begin{align*}
	|N_{24}^{(I, 1)}| & \le \sum_{k} \frac{1}{t} \brak{t}^{3+2\ss}  \bold{a}_{2}^2 \Big| \Big \langle ik \stf_{k} \overline{ \partial_{y} (q\Gamma_k) \Gamma_k \omega_k}, \partial_{y}  \mathring{\overline{H}}_1 e^{W} \chi_1^2 (1-\chi_2^2)\Big \rangle \Big|,
\end{align*}
we obtain
\begin{align*}
	 \sum_{k} \frac{1}{t}& \brak{t}^{3+2\ss}  \bold{a}_{2}^2 \Big| \Big \langle ik \stf_{k} \overline{ \partial_{y} (q\Gamma_k) \Gamma_k \omega_k}, \partial_{y}  \mathring{\overline{H}}_1 e^{W} \chi_1^2 (1-\chi_2^2)\Big \rangle \Big|\\
	&\le 	\frac{1}{t} 	\brak{t}^{3+2\ss}  \bold{a}_{2}^2
	\sum_{k}     \norm{ k \stf_{k}\sqrt{1-\chi_2^2} }_{L^\infty} 
	\enorm{ \partial_y (q\Gamma_k) \Gamma_k \omega_k e^{W} (1-\chi_3^2)\chi_2}
	\enorm{\partial_{y} \mathring{\overline{H}}_1  \sqrt{(1-\chi_2^2)}  \chi_2}
	\\&\lesssim  \frac{\eps^3}{\brak{t}^{4+r-2\ss}}.
\end{align*}
We remark here that due to the localization $\chi_n$, we are not able to use the $\mathcal{CK}$ terms to recover the derivative lose in the interior and we use the ``cloud" estimates instead. The treatment of $N^{(0)}_{\overline{H},1} $ is very similar to that of $N^{(1)}_{\overline{H},1} $ except we get a slower decay rate
\begin{align*}
	N^{(0)}_{\overline{H},1}  \lesssim \frac{\eps^3}{\brak{t}^{3+r-2\ss}}
\end{align*}
due to the fact that $\bold{a}_1 \sim t \bold{a}_2$, and we omit further details to conclude the proof.  
\end{proof}

\subsubsection{Viscous commutator}
Next we focus on the viscous commutator
\begin{align*}
	C^{(\alpha; n)}_{\overline{H}, visc} &= \brak{t}^{3+2\ss} \bold{a}_{n+1}^2\Re\brak{  \partial_y \widetilde{\mathcal{C}^{(n)}_{visc}}, \partial_y \bhqn  e^{W} \chi_n^2}
	\nn&
	=
	\nu \brak{t}^{3+2\ss} \bold{a}_{n+1}^2\Re\brak{ \partial_y  \paren{q^n\sum_{m = 1}^n \binom{n}{m}  \Gamma_0^m v_y^2 \Gamma_0^{2 + n-m} \overline{H}}, \partial_y\bhqn  e^{W} \chi_n^2}.
\end{align*}
\begin{lemma}
	\label{visc:comm:barH:1}
	It follows
	\begin{align}
		\label{visc:comm:esti:barH:1}
		&\sum_{n} \theta_{n}^2 (n+1)^{2\sss-2}  C^{(\alpha; n)}_{\overline{H}, visc}
		\lesssim \eps
		\sqrt{\cd_{\overline H}^{(\alpha)}} \paren{\sqrt{\mathcal{CK}_{\overline{H}}^{(\alpha)}} + \sqrt{\cd_{\overline H}^{(\alpha)}}}
		.
	\end{align}
	\begin{proof}
 {		As stated in Lemma~\ref{visc:comm:barH}, the proof is similar to that of Lemma 5.6 in \cite{BHIW24b} except a slight difference in }
		the binomial coefficients and we again only sample one term here to show the readers how it works.
Integration by parts gives
		\begin{align*}
			C^{(\alpha; n)}_{\overline{H}, visc}
			&=
			- \nu \brak{t}^{3+2\ss}\sum_{l = 1}^n  \bold{a}_{n+1}^2 \binom{n}{l}
			\Re\brak{   q^n   \Gamma_0^l \paren{v_y^2-1} \Gamma_0^{2 + n-l} \overline{H},
			\partial_y\paren{\partial_y	\bhqn  e^{W} \chi_n^2}}
			\nn&=
			-\nu \brak{t}^{3+2\ss}\sum_{l = 1}^n  \bold{a}_{n+1}^2 \binom{n}{l}
			\Re\brak{   q^n   \Gamma_0^l \paren{v_y^2-1} 
				\Gamma_0^{2+ n-l} \overline{H},
			\partial_y^2	\bhqn  e^{W} \chi_n^2}
			\nn&\quad
			+\nu \brak{t}^{3+2\ss}\sum_{l = 1}^n  \bold{a}_{n+1}^2 \binom{n}{l}
			\Re\brak{   q^n   \Gamma_0^l \paren{v_y^2-1} \Gamma_0^{2+ n-l} \overline{H},
			\partial_y	\bhqn  \partial_y \paren{e^{W}\chi_n^2}}
			\nn&= 	V_{\overline H, 1}^{(n)}+V_{\overline H, 2}^{(n)}.
		\end{align*}
		For the term involving $V_{\overline H, 1}^{(n)}$, we split the sum as
		\begin{align*}
			\sum_{n} \theta_{n}^2 (n+1)^{2\sss-2}  V_{\overline H, 1}^{(n)} &= 
			\sum_{n}  \paren{\sum_{l\le n/2} + \sum_{n/2<l\le n} } 
			(n+1)^{2\sss-2}\nu \brak{t}^{3+2\ss} \bold{a}_{n+1}^2 \binom{n}{l}
			\nn&\qquad\times
			\Re\brak{   q^n   \Gamma_0^l \paren{v_y^2-1} 
				\Gamma_0^{2+ n-l} \overline{H},
				\partial_y^2	\bhqn  e^{W} \chi_n^2}
			\nn&=
			\sum_{n} \theta_{n}^2 (n+1)^{2\sss-2}  V_{\overline H, 1}^{LH}	+ 	\sum_{n} \theta_{n}^2 (n+1)^{2\sss-2}  V_{\overline H, 1}^{HL}.
		\end{align*}
	 Since the ``High-Low" interaction is a bit tricky, we deal with it first. By H\"older's inequality, one easily obtain
		\begin{align*}
			\sum_{n} \theta_{n}^2 (n+1)^{2\sss-2}  V_{\overline H, 1}^{HL}
			&\lesssim
			\sum_{n}  \sum_{n/2<l\le n}  
			(n+1)^{2\sss-2}\nu \brak{t}^{3+2\ss} \bold{a}_{n+1}^2 \binom{n}{l}
			\enorm{q^{l-1}\Gamma_0^l \paren{v_y^2-1} e^{W/2}\chi_{l-1}\widetilde \chi_1}
			\nn&\qquad\times 
			\norm{q^{n-l+1}\Gamma_0^{2+ n-l} \overline{H} \chi_{n-l+1}}_{L^\infty}
			\enorm{\partial_y^2\bhqn  e^{W/2} \chi_n}
		\end{align*}
	where $\widetilde\chi_1$ is a fattened version of $\chi_{ 1}$ while $W$ is still huge in $\supp \widetilde\chi_{ 1}$.
		By Sobolev embedding, it follows
		\begin{align*}
			&\norm{q^{n-l+1}\Gamma_0^{2 + n-l} \overline{H} \chi_{n-l+1}}_{L^\infty}
			\lesssim
			(n-l)^{1+\sigma}\norm{q^{n-l+1}\Gamma_0^{2 + n-l} \overline{H} \chi_{n-l}\widetilde\chi_1}_{L^2}
			\nn&\qquad+ 	\norm{q^{n-l+1}\Gamma_0^{2 + n-l} \overline{H} \chi_{n-l+1}}_{L^2}^{1/2}
			\norm{\partial_y\paren{q^{n-l+1}\Gamma_0^{2 + n-l} \overline{H}} \chi_{n-l+1}}_{L^2}^{1/2}.
		\end{align*}
		Therefore, we further obtain
		\begin{align*}
			&\sum_{n} \theta_{n}^2 (n+1)^{2\sss-2}  V_{\overline H, 1}^{HL}
			\lesssim
			\sum_{n}  \sum_{n/2<l\le n}  
			(n+1)^{2\sss-2}\nu \brak{t}^{3+2\ss} \bold{a}_{n+1}^2 \binom{n}{l} \paren{\bb_{l-1}\bb_{n-l+2}}^{-1} (n-l)^{-\sss+1}
			\bb_{l-1}
			\nn&\qquad\times \enorm{q^{l-1}\Gamma_0^l \paren{v_y^2-1} e^{W}\chi_{l-1}\widetilde \chi_1}
			\bb_{n-l+2}(n-l+2)^{\sss+\sigma}\norm{q^{n-l+1}\Gamma_0^{2 + n-l} \overline{H} \chi_{n-l}\widetilde\chi_1}_{L^2}
			\nn	&\qquad\times			\enorm{\partial_y^2\bhqn  e^{W} \chi_n}
			\nn	&\qquad+
			\sum_{n}  \sum_{n/2<l\le n}  
			(n+1)^{2\sss-2}\nu \brak{t}^{3+2\ss} \bold{a}_{n+1}^2 \binom{n}{l} \paren{\bb_{l-1}\bb_{n-l+2}}^{-1} (n-l)^{-\sss+1}
			\bb_{l-1}
			\nn&\qquad\times \enorm{q^{l-1}\Gamma_0^l \paren{v_y^2-1} e^{W}\chi_{l-1}\widetilde \chi_1}
			\bb_{n-l+2}(n-l)^{\sss+\sigma}	\norm{q^{n-l+1}\Gamma_0^{2 + n-l} \overline{H} \chi_{n-l+1}}_{L^2}^{1/2}
			\nn	&\qquad\times		
			\norm{\partial_y\paren{q^{n-l+1}\Gamma_0^{2 + n-l} \overline{H}} \chi_{n-l+1}}_{L^2}^{1/2}
			\enorm{\partial_y^2\bhqn  e^{W} \chi_n}.
		\end{align*}
		By Lemma~\ref{comb:boun}, H\"older's inequality, and Young's inequality, it follows
		\begin{align*}
			\sum_{n} \theta_{n}^2 (n+1)^{2\sss-2} & V_{\overline H, 1}^{HL} \lesssim
			\nu^{1/2}	\sqrt{\cd_{\overline H}^{(\alpha)}} \langle t \rangle^{1+\ss}
			\paren{\paren{\sum_{l\ge0} \norm{\bb_{l+1}(l+1)^{\sss+\sigma} q^{l}\Gamma_0^{l+1} 
						\overline{H} \chi_{n-l+1}}_{L^2}^2}^{1/2}
			\right.\nn&\quad \left.
			+\paren{\sum_{l\ge0} \norm{\bb_{l+1}(l+1)^{\sss+\sigma} q^{l}\Gamma_0^{l+1} 
						\overline{H} \chi_{n-l+1}}_{L^2}^2}^{1/4}
		 			\right.\nn&\quad \left.\times
			\paren{\sum_{l\ge0} \norm{\bb_{l+1}(l+1)^{\sss+\sigma} \partial_y\paren{q^{l}\Gamma_0^{l+1} 
				 			\overline{H}} \chi_{n-l+1}}^2_{L^2}}^{1/4} }
			\nn&\quad \times
			\paren{\sum_{l\ge1} \enorm{\bold{a}_{\ell} q^{\ell-1}\Gamma_0^{\ell} \paren{v_y^2-1} e^{W} 
						\chi_{l-1}\widetilde \chi_1}^2}^{1/2}
			\nn&\lesssim \nu^{100}
			\sqrt{\cd_{\overline H}^{(\alpha)}} \paren{\nu^{1/2}\norm{\overline{H}}_{Y_{1,-s}} + \nu^{1/4}\norm{\overline{H}}_{Y_{1,-s}}^{1/2}\paren{\cd_{\overline H}^{(\alpha)}}^{1/4}}
			\norm{v_y^2-1}_{Y_{1, 0}},
		\end{align*}
	where an argument like~\eqref{dif:typ:1} is used. 
		On the other hand, for the $LH$ piece, we proceed as
		\begin{align*}
			\sum_{n} \theta_{n}^2 (n+1)^{2\sss-2}  V_{\overline H, 1}^{LH}
			&\lesssim
			\sum_{n}  \sum_{n/2<l\le n}  
			(n+1)^{2\sss-2}\nu \brak{t}^{3+2\ss} \bold{a}_{n+1}^2 \binom{n}{l}
			\norm{q^{l}\Gamma_0^l \paren{v_y^2-1} e^{W}\chi_n}_{L^2}
			\nn&\qquad\times 
			\norm{q^{n-l}\Gamma_0^{2 + n-l} \overline{H} \chi_{n-l+1}}_{L^\infty}
			\enorm{\partial_y^2\bhqn  e^{W} \chi_n},
		\end{align*}
		from where using a similar argument as in Lemma~\ref{visc:comm:barH}, we obtain
		\begin{align*}
			\sum_{n} \theta_{n}^2 (n+1)^{2\sss-2}  V_{\overline H, 1}^{LH} \lesssim
			\nu^{100}	\cd_{\overline H}^{(\alpha)} 
              \norm{v_y^2-1}_{Y_{0, 0}}^{1/2}\norm{v_y^2-1}_{Y_{1, 0}}^{1/2} 
		\end{align*}
		and finish the treatment of $V_{\overline H, 1}^{(n)}$ using Lemma \ref{pro:0} and \ref{pro:1}. While the term involving $V_{\overline H, 2}^{(n)}$ is relatively easy, the proof
 {		of which is essentially covered in the treatment of viscous commutator of the nonlinear term in Subsections 5.3 in \cite{BHIW24b}.}
	\end{proof}
\end{lemma}
\subsubsection{Transport commutator}
For the transport commutator part, we note
\begin{align*}
	C^{(\alpha; n)}_{\overline{H}, trans} &= 
	\brak{t}^{3+2\ss} \bold{a}_{n+1}^2\Re \brak{\partial_y \widetilde{\mathcal{C}^{(n)}_{trans}} , \partial_y\bhqn  e^{W} \chi_n^2}
	\nn&
	=
	\brak{t}^{3+2\ss} \bold{a}_{n+1}^2
	\Re \brak{\partial_y \paren{q^n\sum_{m = 1}^n \binom{n}{m} \Gamma_0^m G \Gamma_0^{n-m+1} 
			\overline{H}}, \partial_y \bhqn  e^{W} \chi_n^2}.
\end{align*}
 {By an argument similar to Lemma 5.9 in \cite{BHIW24b}, we obtain the following lemma.}
\begin{lemma}
	\label{tran:comm:barH:1}
	It holds
	\begin{align}
		\label{tran:comm:esti:barH:1}
		\sum_{n} \theta_{n}^2 (n+1)^{2\sss-2}  C^{(\alpha; n)}_{\overline{H}, trans}
		\lesssim
		\eps \paren{\cd_{\overline H}^{(\alpha)} + \cd_{H}^{(\alpha)} }.
	\end{align}
\end{lemma}
\begin{proof}
	After integration by parts, we split the sum into two pieces
	\begin{align*}
		C^{(\alpha;  n)}_{\overline{H}, trans} 
		=&-\paren{\sum_{l\le n/2} + \sum_{n/2<l\le n-1} }\binom{n}{l}
		\brak{t}^{3+2\ss} \bold{a}_{n+1}^2
		\Re \brak{q^n \Gamma_0^l \frac{\overline H}{v_y} \Gamma_0^{n-l} \overline{H}, \partial_y^2\bhqn  e^{W} \chi_n^2}
		\nn&	
		-\paren{\sum_{l\le n/2} + \sum_{n/2<l\le n-1} }\binom{n}{l}
		\brak{t}^{3+2\ss} \bold{a}_{n+1}^2
		\Re \brak{q^n \Gamma_0^l \frac{\overline H}{v_y} \Gamma_0^{n-l} \overline{H}, \partial_y\bhqn  \partial_y \paren{e^{W} \chi_n^2}}
		\nn=&
		T_{\overline H,1}^{(\alpha)} +T_{\overline H,2}^{(\alpha)} +T_{\overline H,3}^{(\alpha)} +T_{\overline H,4}^{(\alpha)}.
	\end{align*}
 {	By the same argument as in the proof of Lemma 5.9 in \cite{BHIW24b}, we obtain}
	\begin{align*}
		\sum_{n} & \theta_{n}^2 (n+1)^{2\sss-2} T_{\overline H,1}^{\gm}\\
		&\lesssim
		\sum_{n\ge1}
		\sum_{l\le n/2} | \bold{a}_{n+1}|^2 \brak{t}^{3+2\ss}
		\binom{n}{l}  (n+1)^{2\sss-2}
		\nnorm{  \partial_y^2 \bhqn  e^{W}\chi_n}_{L^2} \bold{a}_{n-l+1}^{-1}\bb_{l+1}^{-1}(l+1)^{2-2\sss}
		\nn&\quad
		\times
		 \bold{a}_{n-l+1}(n-l+1)^{\sss-1}\enorm{ \ztp{\overline{H}}_{n-l} e^{W}\chi_{n-l}\widetilde\chi_{ 1}}
		\bb_{l+1}(l+1)^{2\sss-2}\paren{ l^{1+\sigma} + 1}\norm{\bd^{l} \frac{\overline{H}}{v_y}\chi_{(l-1)_+}\widetilde\chi_1}_{L^2}
		\nn&\quad+
		\sum_{n\ge0}
		\sum_{l\le n/2} | \bold{a}_{n+1}|^2 \brak{t}^{3+2\ss}
		\binom{n}{l}  (n+1)^{2\sss-2}
		\nnorm{ \partial_y^2\bhqn e^{W}\chi_n}_{L^2} \bold{a}_{n-l+1}^{-1}\bb_{l+1}^{-1}(l+1)^{2-2\sss}
		\nn&\quad
		\times
		 \bold{a}_{n-l+1}(n-l+1)^{\sss-1}\enorm{ \ztp{\overline{H}}_{n-l} e^{W}\chi_{n-l}\widetilde\chi_{ 1}}
		\bb_{l+1}(l+1)^{2\sss-2}\nu^{1/2} \norm{\partial_y\paren{ \bd^{l} \frac{\overline{H}}{v_y}} \chi_{ l}\widetilde\chi_1}_{L^2}
		\\&\lesssim \nu^{100}
	   \sqrt{\cd_{\overline H}^{(\alpha)}}	\langle t \rangle^{1+\ss}
		\norm{\overline{H}}_{Y_{0,-s}}
		\paren{\norm{\frac{\overline{H}}{v_y}}_{Y_{0,-s}}+\nu^{1/2}\norm{\frac{\overline{H}}{v_y}}_{Y_{1,-s}}}
	\end{align*}
where we used $\nu^{-100}\lesssim e^W$.
 {	The estimates of the term $T_{\overline H,2}^{(\alpha)}$ follows similarly as the second part in Lemma 5.9 in \cite{BHIW24b}.}
	While the terms $T_{\overline H,3}^{(\alpha)}$ with $T_{\overline H,4}^{(\alpha)}$ are relatively easy, and again more details are omitted, concluding the proof.
\end{proof}
\subsubsection{Inductive terms}
We recall that
\begin{align*}
	C^{(\alpha; n)}_{\overline{H}, q}&=- \brak{t}^{3+2\ss} \bold{a}_{n+1}^2\Re \brak{ \partial_{y} \widetilde{\mathcal{C}^{(n)}_{q}}, 
		\partial_{y} \bhqn  e^{W} \chi_n^2}.
\end{align*}
Noting that 
\begin{align}
	\label{py:C:q}
        -\nu^{-1}\partial_{y} \widetilde{\mathcal{C}^{(n)}_{q}} = &
         \frac{n(n-1)(n-2) - n^2(n-1)}{q^2} (q')^3 \overline \hqn   + \frac{2n(n-1)-n}{q^2}q'q'' \overline \hqn   
        \nn&\quad 
        +
        \frac{n}{q} q'''\overline \hqn    + \frac{n(n-1)-2n}{q^2}(q')^2\partial_{y} \overline \hqn 
        + \frac{3n}{q^2} q''\partial_{y}\overline \hqn 
         +  \frac{n}{q} q'\partial_{y}^2\overline \hqn ,
\end{align}
we may further write
\begin{align*}
	C^{(\alpha; n)}_{\overline{H}, q}&= \sum_{i=1}^{6} 	C^{(\alpha; n)}_{\overline{H}, q, i}
\end{align*}
where $C^{(\alpha; n)}_{\overline{H}, q, i}$ is corresponding to the $i$-th term on the right of~\eqref{py:C:q}.

\begin{lemma}
	We have
	\label{C:n:q:1}
	\begin{align*}
		\sum_{n} \theta_{n}^2 (n+1)^{2\sss-2} 	C^{(\alpha; n)}_{\overline{H},q} \lesssim \eps_1  \nu^{-1/2}\cd_{\overline H}^{(\alpha)}
	\end{align*}
where $\eps_1$ is a small parameter depending on $n_*$ and $\delta_{\text{Drop}}$. 
\end{lemma}
\begin{proof}
	As in the proof of Lemma~\ref{C:n:q}, we deduce that
	\begin{align*}
		\sum_{n} &\theta_{n}^2 (n+1)^{2\sss-2} 	C^{(\alpha; n)}_{\overline{H},q} \le \sum_{n} \theta_{n}^2\sum_{0\le m\le n-1} C^{n-m}
		\paren{\frac{n!}{m!}}^{-\sigma} \nu^{-1/2} D_{\overline H,m}^{(\alpha)} 
		\\& \le C \min\left\{\frac{1}{n_{*}^{\sigma/2}}, \delta_{\text{Drop}}^2\right\} \nu^{-1/2} \cd_{\overline H}^{(\alpha)} \le \eps_1 \nu^{-1/2} \cd_{\overline H}^{(\alpha)}
	\end{align*}
	provided $n_{*} $ is large enough and $\delta_{\text{Drop}}$ is sufficiently small.
\end{proof}

\subsubsection{Easy terms}
Next we bound the easy terms for the $\alpha$ part.
\begin{lemma}
	\label{easy:term:1}
	It holds
	\begin{align*}
		&\sum_{n} \theta_{n}^2 (n+1)^{2\sss-2}  \bold{a}_{n+1}^2  \brak{t}^{3+2\ss}  \nu 
		\langle \partial_y^2\bhqn , \partial_y\bhqn   e^{W}  \pa_y \{\chi_n^2\} \rangle 
		\lesssim
		\eps_1 \nu^{-1/2}	\cd_{\overline H}^{(\alpha)} ;
		\nn&
		\sum_{n} \theta_{n}^2 (n+1)^{2\sss-2}  \bold{a}_{n+1}^2   \brak{t}^{3+2\ss} \nu \langle \partial_y^2 \bhqn , 
		\partial_y \bhqn  \pa_y \{ e^{W}  \} \chi_n^2  \rangle
		\lesssim \eps_1 \nu^{-1/2}\sqrt{\mathcal{CK}_{\overline{H}}^{(\alpha)}}\sqrt{\cd_{\overline H} ^{(\alpha)}};
				\nn&
		\paren{(3/2+\ss)\brak{t}^{1+2\ss}t - \frac{2}{t}\brak{t}^{3+2\ss}} \bold{a}_{n+1}^2 \| \partial_y \bhqn  e^{W/2} \chi_{n} \|_{L^2}^2
		\le 
		\frac{1}{2} \nu^{-1/2} \mathcal{CK}_{\overline{H}}^{(\alpha)}.
	\end{align*}
	for some small $\eps_1>0$.
\end{lemma}
\begin{proof}
	Noting that
	\begin{align*}
		\abs{\partial_y \chi_n} \lesssim n^{1+\sigma},
	\end{align*}
	with $\sss> 1 + \sigma$ we have
	\begin{align*}
		&\sum_{n} \theta_{n}^2 (n+1)^{2\sss-2}  \bold{a}_{n+1}^2  \brak{t}^{3+2\ss}  \nu
		\langle \partial_y^2\bhqn , \partial_y\bhqn   e^{W/2}  \pa_y \{\chi_n^2\} \rangle 
		\nn&\qquad\lesssim
		\sum_{n} \theta_{n}^2 (n+1)^{2\sss-2}  \bold{a}_{n+1}^2 n^{1+\sigma} \brak{t}^{3+2\ss}  \nu
		\enorm{\partial_y\paren{q^{n}\frac{\partial_y}{v_y}\Gamma_0^{n-1}\overline{H}}  e^{W/2} \chi_{n-1}}
		\enorm{\partial_y^2 \bhqn   e^{W/2} \chi_n}.
	\end{align*}
 {	Similar as in Lemma~\ref{easy:term}, using~\eqref{comm:q:deri}, Lemma 6.20 in \cite{BHIW24b},} $q\in C^\infty[-1, 1]$, and the fact $v_y \sim 1$, $|\partial_{y}^2v| \lesssim 1$ (see Corollary~\ref{v:3:deri:l}), the above quantity may be further bounded by 
	\begin{align*}
		\sum_{n} & \theta_{n}^2 (n+1)^{2\sss-2}  \bold{a}_{n+1}^2 n^{1+\sigma} \brak{t}^{3+2\ss}  \nu
		\paren{ \enorm{\partial_y^2(q^{n-1}\Gamma_0^{n-1})\overline{H}  e^{W/2} \chi_{n-1}}
			+ 	n \enorm{\partial_y(q^{n-1}\Gamma_0^{n-1})\overline{H}  e^{W/2} \chi_{n-1}}}
		\nn&\qquad\times
		\enorm{\partial_y^2\bhqn   e^{W/2} \chi_n}\nn&
			\le C  \nu^{-1/2} \sum_{n\ge1} \theta_{n}^2  \frac{1}{n^{\sigma^*}} \sqrt{\cd_{\overline H, n}^{(\alpha)}} \sqrt{\cd_{\overline H, n-1}^{(\alpha)}}
		+  C \nu^{-1/2} \sum_{n\ge1}\theta_{n}^2   \sqrt{\sum_{\ell=0}^{n-1} (\mathfrak C\lambda)^{2(n-\ell)}\lf(\frac{\ell!}{n!}\rg)^{2\sigma}  
			\cd_{\overline H, \ell}^{(\alpha)}}  \sqrt{\cd_{\overline H, n}^{(\alpha)}}
		\nn&
		\le  C \nu^{-1/2} \min\left\{\frac{1}{n_{*}^{\sigma^*}}, \delta_{\text{Drop}}\right\} \cd_{\overline H}^{\gm}
		+ C  \nu^{-1/2}\min\left\{\frac{1}{n_{*}^{\sigma/2}}, \delta_{\text{Drop}}\right\} \cd_{\overline H}^{\gm}
		\le \eps_1  \nu^{-1/2} \cd_{\overline H}^{\gm}
	\end{align*}
	where we also used that $\sss > 1+\sigma +\sigma^*$ for some $\sigma^*>0$.
	Note by~\eqref{defndW} 
	\begin{align*}
		\abs{\partial_yW(t, y)}\lesssim \frac{(|y|-1/4-L\ep\arctan (t))_+}{K\nu(1+t)}.
	\end{align*}
	Therefore, we may get
	\begin{align*}
		&\sum_{n} \theta_{n}^2 (n+1)^{2\sss-2}  \bold{a}_{n+1}^2   \brak{t}^{3+2\ss} \nu\langle \partial_y^2\bhqn , 
		\partial_y\bhqn  \pa_y \{ e^{W}  \} \chi_n^2  \rangle
		\nn&\qquad\lesssim
		\sum_{n} \theta_{n}^2 (n+1)^{2\sss-2} \bold{a}_{n+1}^2   \brak{t}^{3+2\ss}  
		\enorm{\nu^{1/2}\frac{(|y|-1/4-L\ep\arctan (t))_+}{K\nu(1+t)}\partial_y\bhqn  e^{W/2}  \chi_n}
		\nn&\qquad\qquad \times		
		\nu^{1/2}\enorm{\partial_y^2\bhqn  e^{W/2}  \chi_n} 
		\nn&\qquad\lesssim \eps_1 \nu^{-1/2} \sqrt{\mathcal{CK}_{\overline{H}}^{(\alpha)}}\sqrt{\cd_{\overline H}^{(\alpha)}}
	\end{align*}
	provided $K$ is sufficiently large. The third inequality of this lemma
	follows exactly the same as that for the $\gamma$ part, which is omitted, concluding the proof.
\end{proof}

\subsection{$(G, \gamma)$ Estimates}

We recall the following Neumann problem for $G$:
\begin{subequations}
\begin{align} \label{G:eq:2}
	&\pa_t G + \frac2t G - \nu \pa_y^2 G = -\frac{1}{t} \left( u_{\neq} \cdot \nabla u^x \right)_0    \\
	&\pa_y G|_{y = \pm 1} = 0.
\end{align}
\end{subequations}
We recall the relations
\begin{align} \label{GbarH}
\pa_y G = \overline{H} \Rightarrow \gqn  = q \paren{\frac{\overline{H}}{v_y}}_{n-1}.
\end{align}
Our main purpose in this subsection is to obtain the \textit{a-priori} estimate \eqref{apG}. 
When $n\ge1$, using \eqref{1onvy}, the bootstrap assumption about $H$, localization property of $H$ and Lemma~\ref{pro:-s},  we obtain
\begin{align*}
	\mathcal{E}_{G}^{(\gamma)} - 	\theta_0^2\mathcal{E}_{G, 0}^{(\gamma)} 
	\lesssim
	\brak{t}^{3-2\ss}\norm{\frac{\overline{H}}{v_y}}_{Y_{0, -s}}^2
	\lesssim 
	(1+\eps) \paren{	\mathcal{E}_{\overline H}^{(\gamma)} + 	\mathcal{E}_{\overline H}^{(\alpha)}}. 
\end{align*}
Therefore, the only nontrivial case is $n = 0$, to deal with which we have the following lemma.
\begin{lemma}
	It holds
	\begin{align}
		\label{G:0:esti}
			\frac{1}{2}\frac{d}{dt}\mathcal{E}_{G, 0} + \mathcal{CK}_{G, 0} + \cd_{G, 0} \le \frac{\eps^3}{t^2}.
	\end{align}
\end{lemma}
\begin{proof} 
	Noting $\chi_0=1$, we compute the inner-product of \eqref{G:eq:2} with $\overline{G} \langle t \rangle^{4-2K\eps_1}\chi_0^2$ 
	to get
\begin{align*}
	\frac{1}{2}\frac{d}{dt} \mathcal{E}_{G, 0}& + \paren{4\langle t \rangle^{4-2K\eps_1}t^{-1} - (4-2K\eps_1)t\brak{t}^{2-2\eps_1}}\enorm{G\chi_0}^2 +
	\nu \langle t \rangle^{4-2K\eps_1}\enorm{\partial_y G\chi_0}^2
	\nn&=
	- \langle t \rangle^{4-2K\eps_1}t^{-1} \R\int  \langle \nabla^\perp \stf_{\neq 0} \cdot \nabla u^x \rangle \overline G  \chi_0^2 \,dy
\end{align*}
For the nonlinear term, we notice
\begin{align*}
	- \frac{1}{t} \langle \nabla^\perp \stf_{\neq 0} \cdot \nabla u^x \rangle_x 
	=
	 \frac{1}{t} \langle \Gamma \stf_{\neq 0}\partial_x u^x\rangle_x - \langle\partial_x \stf_{\neq 0} \Gamma u^x \rangle_x.
\end{align*}
For the first piece, we write
\begin{align*}
	\langle \Gamma \stf_{\neq 0}\partial_x u^x\rangle_x &=
	\langle \Gamma \stf_{\neq 0}^E\partial_x \paren{u^x}^E\rangle_x + \langle \Gamma \stf_{\neq 0}^I\partial_x \paren{u^x}^E\rangle_x
	+ \langle \Gamma \stf_{\neq 0}^E\partial_x \paren{u^x}^I\rangle_x + \langle \Gamma \stf_{\neq 0}^I\partial_x \paren{u^x}^I\rangle_x
	\\& =\sum_{i=1}^{4} N_i
\end{align*}
 {For the first term $N_1$, by Lemma~6.12 in \cite{BHIW24b}, we have}
\begin{align*}
	\langle t \rangle^{1-K\eps_1}\enorm{N_1} &\le \langle t \rangle^{1-K\eps_1}\sum_{k\neq 0}\enorm{\Gamma \stf_{k}^{(I)} k \paren{u^x}_k^E}
	\le \langle t \rangle^{1-K\eps_1}\sum_{k\neq 0}\norm{\Gamma \stf_{k}^{(E)}}_{L^\infty}\enorm{ k \paren{u^x}_k^E}
	\\ &\le \langle t \rangle^{1-K\eps_1}\sum_{k\neq 0}\enorm{\partial_{y}\Gamma \stf_{k}^{(E)}}^{1/2}\enorm{\Gamma \stf_{k}^{(E)}}^{1/2}
	\enorm{ k \paren{u^x}_k^E} \le \eps^2 \frac{1}{t^2} .
\end{align*}
 {Still using Lemma~6.12 in \cite{BHIW24b}, we note}
\begin{align*}
	\langle t \rangle^{1-K\eps_1}\enorm{N_2} &\le \langle t \rangle^{1-K\eps_1}\sum_{k\neq 0}\enorm{\Gamma \stf_{k}^{(E)} k \paren{u^x}_k^E}
	\le \langle t \rangle^{1-K\eps_1}\sum_{k\neq 0}\norm{\Gamma \stf_{k}^{(I)}}_{L^\infty}\enorm{ k \paren{u^x}_k^E}
	\\ &\le \langle t \rangle^{1-K\eps_1}\sum_{k\neq 0}\enorm{\partial_{y} \Gamma \stf_{k}^{(I)}}^{1/2}\enorm{\Gamma \stf_{k}^{(I)}}^{1/2}
	\enorm{ k \paren{u^x}_k^E} 
	\\&\le  \frac{1}{t^2} \paren{\enorm{\omega^I} + \enorm{\Gamma \omega^I}} \enorm{\omega^E e^W} \lesssim \frac{\eps^2}{t^2}
\end{align*}
where we also used the inviscid damping estimate of the interior part.
Similarly, for the third piece, we obtain
\begin{align*}
	\langle t \rangle^{1-K\eps_1}\enorm{N_3}  \lesssim \frac{\eps^2}{t^2}.
\end{align*}
The last piece $N_4$ is a bit tricky and we proceed as
\begin{align*}
	\langle t \rangle^{1-K\eps_1}\enorm{N_4} 
	&\le \langle t \rangle^{1-K\eps_1}\sum_{k\neq 0}\enorm{\Gamma \stf_{k}^{(I)} k \paren{u^x}_k^I}
	\le \langle t \rangle^{1-K\eps_1}\sum_{k\neq 0}\norm{\Gamma \stf_{k}^{(I)}}_{L^\infty}\enorm{ k \paren{u^x}_k^I}
	\\ &\le \langle t \rangle^{1-K\eps_1}\sum_{k\neq 0}\enorm{\Gamma^2 \stf_{k}^{(I)}}^{1/2}\enorm{\Gamma \stf_{k}^{(I)}}^{1/2}
	\enorm{ k \paren{u^x}_k^I} 
	\\&\le  \frac{1}{t} \paren{\enorm{\omega^I} + \enorm{\Gamma \omega^I}}\sum_{k\neq 0}
	\enorm{\Gamma^2 \stf_{k}^{(I)}}^{1/2}\enorm{\Gamma \stf_{k}^{(I)}}^{1/2} \lesssim \frac{\eps^2}{t^2}
\end{align*}
where we used \eqref{eell:out}.
Collecting the above estimate on the nonlinear term, we arrive at
\begin{align*}
	&|\langle  \frac{1}{t} \langle \Gamma \stf_{\neq 0}\partial_x u^x\rangle, 
	\overline{G} \langle t \rangle^{4-2K\eps_1}\chi_{0}^2 \rangle|
	 \lesssim   
	\frac{\eps}{t^2}
	 \| G \langle t \rangle^{2-\eps} \|_{L^2}
	 \lesssim
	 \eps \mathcal{CK}_{G, 0} + 
	 \frac{\eps^2}{t^3}.
\end{align*}
The second piece $\langle\partial_x \stf_{\neq 0} \Gamma u^x \rangle_x$ is treated similarly and we omit further details to conclude the proof.
\end{proof}

\subsection{$(G, \alpha)$ Estimates}

We recall the relations
\begin{align*} 
	\pa_y G = \overline{H}  \Rightarrow \partial_{y} \gqn  = \partial_{y} \paren{q \paren{\frac{\overline H}{v_y}}_{n-1}} = \partial_{y}q \paren{\frac{\overline H}{v_y}}_{n-1} + q \partial_{y} \paren{\frac{\overline H}{v_y}}_{n-1}.
\end{align*}
Therefore, the $\alpha$ estimate of $G$ is essentially done in the previous subsections. In fact, using \eqref{1onvy}, the bootstrap assumption about $H$, localization property of $H$ and Lemma~\ref{pro:-s},  we obtain
\begin{align*}
	\mathcal{E}_{G}^{(\alpha)} - 	\theta_0^2\mathcal{E}_{G, 0}^{(\alpha)} 
	\lesssim
	\brak{t}^{3-2\ss} \paren{\norm{\frac{\overline{H}}{v_y}}_{Y_{0, -s}}^2 + \norm{\frac{\overline{H}}{v_y}}_{Y_{1, -s}}^2}
	\lesssim 
	(1+\eps) \paren{	\mathcal{E}_{\overline H}^{(\gamma)} + 	\mathcal{E}_{\overline H}^{(\alpha)}}. 
\end{align*} To complete the proof, we only need to show the case $n=0$ in detail, which is the result of the lemma below.
\begin{lemma}
	It holds
	\begin{align}
		\label{G:0:esti:1}
		\frac{1}{2}\frac{d}{dt}\mathcal{E}_{\overline{H}, 0}  + \mathcal{CK}_{\overline H, 0}
		+ \cd_{\overline H, 0} \le \frac{\eps^3}{t^2}.
	\end{align}
\end{lemma}
\begin{proof} We compute the inner-product of \eqref{eq:bar:H:L} with $\overline{H} \langle t \rangle^{4-2K\eps_1}\chi_0^2$ 
	to get
	\begin{align*}
		\frac{1}{2}\frac{d}{dt} \mathcal{E}_{\overline{H}, 0}& + \paren{4\langle t \rangle^{4-2K\eps_1}t^{-1} - (4-2K\eps_1)t\brak{t}^{2-2\eps_1}}\enorm{\overline{H}\chi_0}^2 +
		\nu \langle t \rangle^{4-2K\eps_1}\enorm{\partial_y \overline{H}\chi_0}^2
		\nn&=
		- \langle t \rangle^{4-2K\eps_1}t^{-1}v_y \R\int  \langle \nabla^\perp \stf_{\neq 0} \cdot \nabla \omega \rangle \overline{H}  \chi_0^2 \,dy.
	\end{align*}
The only difference between the equations of $\overline{H}$ and $G$ is the nonlinearity, which we treat as
	\begin{align*}
      \langle \nabla^\perp \stf_{\neq 0} \cdot \nabla \omega \rangle 
		=
	    \langle \Gamma \stf_{\neq 0}\partial_x \omega\rangle 
		- \langle\partial_x \stf_{\neq 0} \Gamma \omega \rangle.
	\end{align*}
Like in the previous lemma, we divide the first term in the nonlinearity into four pieces:
\begin{align*}
	 \langle \Gamma \stf_{\neq 0}\partial_x \omega\rangle 
	= \sum_{i=1}^{4} N_i
\end{align*}
where the treatment of the first three pieces are exactly the same as in the previous lemma. Hence we only focus on the last term
\begin{align*}
	\langle t \rangle^{1-K\eps_1}\enorm{N_4} 
	&\le \langle t \rangle^{1-K\eps_1}\sum_{k\neq 0}\enorm{\Gamma_k \stf_{k}^{(I)} k \omega^I_k}
	\le \langle t \rangle^{1-K\eps_1}\sum_{k\neq 0}\norm{\Gamma_k \stf_{k}^{(I)}}_{L^2}\norm{ k \omega^I_k}_{L^\infty}
	\\ &\le \langle t \rangle^{1-K\eps_1}\sum_{k\neq 0}\enorm{\Gamma \stf_{k}^{(I)}}
\enorm{ k \omega^I_k}^{1/2} \enorm{ k \Gamma_k\omega^I_k}^{1/2} 
	\\& \lesssim \frac{\eps^2}{t^{1+\eps}}.
\end{align*}
And hence we have
	\begin{align*}
		|\langle  \frac{1}{t} \langle \Gamma \stf_{\neq 0}\partial_x \omega\rangle, 
		\overline{H} \langle t \rangle^{4-2K\eps_1}\chi_{0}^2 \rangle|
		&\lesssim   
		 \frac{\eps^2}{t^{1+\eps}}
		\| \overline{H} \langle t \rangle^{2-\eps} \chi_{0}\|_{L^2}
		\lesssim \eps \mathcal{CK}_{\overline H, 0}+ \frac{\eps^2}{t^{1+\eps}}.
	\end{align*}
  The second piece of the nonlinear term is treated similarly and we omit further 
	details to conclude the proof.
\end{proof}

\subsection{$(H, \gamma)$ Estimates}

	We recall $H$ satisfies the following system
\begin{align}
	&\pa_t H - \nu \pa_y^2 H = \overline{H}, \label{eq:H}\\
	&H(t, \pm1) = 0 \label{eq:H:bou}.
\end{align}
The main purpose for this section is to prove~\eqref{apH}, for which it suffices to show the following.
\begin{lemma}
	\label{H:esti}
	It follows
	\begin{align*}
	   \frac{1}{2}\frac{d}{dt} \mathcal{E}_{H}^{\gm} + \mathcal{CK}_{H}^{\gm} + \cd_{H}^{\gm} \lesssim  \mathcal{CK}_{\overline{H}}^{\gm} .
	\end{align*}
\end{lemma}
\begin{proof}
	We compute the (complex) inner product of \eqref{eq:H} against $\bold{a}_{n}^2 \hqn  e^{W} \chi_n^2$ and sum in $n$ to get 
	\begin{align*}
		&  \frac{1}{2}\frac{d}{dt} \mathcal{E}_{H}^{\gm} 
		+ \paren{\mathcal{CK}_{{H}}^{\gm}}^2		+ \paren{\cd_{{H}}^{\gm}}^2\\
		&\qquad = -\sum_{n} \bold{a}_{n}^2  \nu \langle \partial_y\hqn ,
		\hqn   e^{W}  \pa_y \{\chi_n^2\}  \rangle 
		- \sum_{n}  \bold{a}_{n}^2   \nu \langle \partial_y \hqn , 
		\hqn  \pa_y \{ e^{W}  \} \chi_n^2  \rangle 
		\\&\qquad\quad- \sum_{n}  \bold{a}_{n}^2\Re \brak{\bhqn ,
			\hqn  e^{W} \chi_n^2}
		.
	\end{align*}
The CK terms are treated similarly as for the $\overline H$ term and 
exact the same argument as in Lemma~\ref{easy:term}, we obtain
\begin{align}
	&\sum_{n\ge1}  \bold{a}_{n}^2 \nu 
	\langle \partial_y \hqn , \hqn   e^{W}  \pa_y \{\chi_n^2\} \rangle 
	\lesssim
	\eps_1\paren{	\sqrt{\cd_{H}^{\gm}} + \sqrt{\mathcal{CK}_{H}^{\gm}}} \sqrt{\cd_{H}^{\gm}}, \label{H:com:1}
	\\&
	\sum_{n\ge1}  \bold{a}_{n}^2   \nu\langle \partial_y \hqn , 
	\hqn  \pa_y \{ e^{W}  \} \chi_n^2  \rangle
	\lesssim \eps_1 \sqrt{\mathcal{CK}_{H}^{\gm}}\sqrt{\cd_{H}^{\gm}}. \label{H:com:2}
\end{align}
At last, we focus on the source term and notice
\begin{align*}
	 \sum_{n\ge1} \bold{a}_{n}^2\Re \brak{\bhqn ,
		\hqn  e^{W} \chi_n^2} \lesssim
	\bold{a}_{n}^2 \enorm{\overline \hqn  e^{W} \chi_n}\enorm{\hqn  e^{W} \chi_n}
	\lesssim
  \sqrt{\mathcal{CK}_{H}^{\gm}}\sqrt{\mathcal{CK}_{\overline H}^{(\gm)}},
\end{align*}
concluding the proof, where we used that
\begin{align}
	\sqrt{- \bold{a}_{n+1} \frac{d}{dt}( \bold{a}_{n+1})} n^{\sss-1} \sqrt{- \bold{a}_{n} \frac{d}{dt}( \bold{a}_{n})}
	\gtrsim
     \bold{a}_{n}^2 
     \label{H:for}
\end{align}
in the support of the integrand in the last step,
concluding the proof.
\end{proof}

\subsection{$(H, \alpha)$ Estimates}

\begin{subequations}
	The main purpose for this section is to prove~\eqref{apH}, for which it suffices to show the following.
	\begin{lemma}
		\label{H:esti:1}
		It follows
	\begin{align*}
	\frac{1}{2}\frac{d}{dt} \mathcal{E}_{H}^{(\alpha)} + \mathcal{CK}_{H}^{(\alpha)} + \cd_{H}^{(\alpha)} \lesssim \mathcal{CK}_{\overline{H}}^{(\alpha)}.
    \end{align*}
	\end{lemma}
	\begin{proof}
		We apply $\partial_y$ to equation~\eqref{eq:H} and compute the (complex) inner against $ \bold{a}_{n}^2 \partial_y\hqn  e^{W} \chi_n^2$ and sum in $n$ to get 
		\begin{align*}
			& \nu^{-1/2}\paren{\frac{1}{2}\frac{d}{dt} \mathcal{E}_{H}^{(\alpha)} 
				+ \paren{\mathcal{CK}_{{H}}^{(\alpha)}}^2		+ \paren{\cd_{{H}}^{(\alpha)}}^2}\\
			&\qquad = -\sum_{n}\bold{a}_{n}^2  \nu \langle \partial_y^2\hqn ,
			\partial_y\hqn   e^{W}  \pa_y \{\chi_n^2\}  \rangle 
			- \sum_{n} \bold{a}_{n}^2   \nu \langle \partial_y^2 \hqn , 
			\partial_y\hqn  \pa_y \{ e^{W}  \} \chi_n^2  \rangle 
			\\&\qquad\qquad- \sum_{n} \bold{a}_{n}^2\Re \brak{\partial_y\bhqn,
			\partial_y	\hqn  e^{W} \chi_n^2}
			.
		\end{align*}
Then the proof goes exactly the same as in Lemma~\ref{H:esti} except that the index $\gamma$ is replaced by $\alpha$ everywhere and we conclude the lemma without showing more details.
	\end{proof}
\end{subequations}
\subsection{Sobolev estimate for $H$}
\begin{corollary}
	\label{v:3:deri:l}
	Let $H$ be a solution to~\eqref{eq:H}--\eqref{eq:H:bou}, then we have that
		\begin{align}
		\label{v:3:deri}
		\norm{ H}_{W^{3, \infty}} &\lesssim \eps. 
	\end{align}
\end{corollary}
\begin{proof}[Proof of Corollary~\ref{v:3:deri:l}]
Applying $\partial_{y}^2$ to equation~\eqref{eq:H}, multiplying it by $\partial_{y}^2H$, and integrating in space variables gives
\begin{align}
	\label{H:H2}
	\frac{1}{2} \frac{d}{dt} \enorm{\partial_{y}^2 H}^2 + \nu \enorm{\partial_{y}^3 H}^2 = \int \partial_{y}^2 H \partial_{y}^2 \overline H \,dy,
\end{align}
where we used that $\partial_{y}^2H|_{y=\pm 1}=0$ from equation~\eqref{eq:H}. 
Noting by H\"older's inequality and Poincar\'e's inequality 
\begin{align*}
	\int \partial_{y}^2 H \partial_{y}^2 \overline H \,dy \le \enorm{\partial_{y}^2 H}\enorm{\partial_{y}^2 \overline H} \le \frac\nu 2 \enorm{\partial_{y}^3 H}^2 + C\nu^{-1} \enorm{\partial_{y}^2 \overline H}^2.
\end{align*}
Since $e^{W/2} \ge 1$, 
noting 
\begin{align*}
	\int_{0}^{\nu^{-1/3-\delta}} \enorm{\partial_{y}^2 \overline H}^2\,dt \lesssim \eps^2 \nu^{-3/2}
\end{align*}
by Proposition~\ref{alph:esti} ($n=0$ level estimate), 
we further arrive at
\begin{align*}
	\enorm{\partial_{y}^2 H} \lesssim \eps  \nu^{-5/4}.
\end{align*}
Next we show that $H$ is $L^\infty$ bounded. Without loss of generality, we assume $y\in [-1, 0]$and the case $y\in[0,1]$ is similarly treated. Noting
\begin{align*}
	H = \int_{-1}^{y} \partial_{y} H \,d\bar{y} &= \paren{\int_{-1}^{-1+\nu^{1/2}} + \int_{-1+\nu^{1/2}}^{-1 + 1/4} +  \int_{-1 + 1/4}^{y} }    \partial_{y} H \,d\bar{y}
	\\& = I_1 + I_2 + I_3.
\end{align*}
The first piece is treated easily as
\begin{align*}
	I_1 \le \nu^{1/4} \enorm{\partial_{y} H} \lesssim \eps.
\end{align*}
For t he second integral, noting
\begin{equation*}
	\nu^{-10} \lesssim e^{W/2}\ \ \ \ \mbox{for} \ \ \ \ y\in[-1, -3/4],
\end{equation*}
we have 
\begin{align*}
	I_2 \lesssim \nu^{10} \enorm{q\partial_{y} H e^{W/2}} \lesssim \eps.
\end{align*}
For the third piece, we use H\"older's inequality and the interior regularity of $H$ to deduce
\begin{align*}
	I_3 \lesssim \norm{\partial_{y} H}_{L^2{([-3/4, 1])}} \lesssim \mathcal{E}^{(h)}_{\text{Int, Coord}} \lesssim \eps.
\end{align*}
Collecting the estimates for $I_1, I_2, I_3$, we arrive at
\begin{align*}
		\norm{ H}_{L^\infty} \lesssim \eps. 
\end{align*}
By the fundamental theorem of calculus, one may get
\begin{align*}
	\partial_{y} H = \partial_{y} H(0) - \int_{y}^{0} \partial_{y}^2H
\,d\bar{y}.
\end{align*}
The interior regularity of $H$ implies
\begin{align*}
	\partial_{y}H(0) \lesssim \eps.
\end{align*}
By exactly the same argument as for $\norm{H}_{L^\infty}$, adjusting the intervals of integral accordingly, we may prove 
\begin{align*}
	\int_{y}^{0} \partial_{y}^2H \lesssim \eps.
\end{align*}
And hence we get
\begin{align*}
	\norm{\partial_{y} H}_{L^\infty} \lesssim \eps. 
\end{align*}
Next we improve the $H^2$ estimate to get a bound on the $\norm{\partial_{y}^2H}_{L^\infty} $, for which we need an estimate of $\enorm{\partial_{y}^3H}$. Noting from the equation for $\overline H, H$, we have
	\begin{align*}
		\partial_{y}^2\overline H |_{y=\pm 1} = \partial_{y}^2 H |_{y=\pm 1} = 0.
	\end{align*}
	Applying $\partial_{y}^2$ to \eqref{eq:H} and testing it against $\partial_{y}^4 H$, we get
	\begin{align*}
			\frac{1}{2} \frac{d}{dt} \enorm{\partial_{y}^3 H}^2 + \nu \enorm{\partial_{y}^4 H}^2 = - \int \partial_{y}^4 H \partial_{y}^2 \overline H \,dy,
	\end{align*}
	from where it follows 
	\begin{align*}
		\frac{d}{dt} \enorm{\partial_{y}^3 H}^2 + \nu \enorm{\partial_{y}^4 H}^2 \le 2\nu^{-1}\enorm{\partial_{y}^2 \overline H}.
	\end{align*}
	Using $	\int_{0}^{\nu^{-1/3-\delta}} \enorm{\partial_{y}^2 \overline H}^2\,dt \lesssim \nu^{-3/2}$, we arrive at
	\begin{align*}
		\enorm{\partial_{y}^3 H} \lesssim \eps\nu^{-5/4},
	\end{align*}
	from where we deduce that 
		\begin{align*}
		\norm{\partial_{y}^2H}_{L^\infty} \lesssim \eps
	\end{align*}
	 by a similar argument as the $L^\infty$ estimate of $H$. 
Finally, we 	apply $\partial_{y}^4$ to \eqref{eq:H} and testing it against $\partial_{y}^4 H$, arriving at
	 	\begin{align*}
	 	\frac{1}{2} \frac{d}{dt} \enorm{\partial_{y}^4 H}^2 + \nu \enorm{\partial_{y}^5 H}^2 =  \int \partial_{y}^4 H \partial_{y}^4 \overline H \,dy = -\int \partial_{y}^5 H \partial_{y}^3 \overline H \,dy
	 \end{align*}
	 where we used the fact that 
	 \begin{align*}
	 	\partial_{y}^4 H |_{y=\pm 1} = 0
	 \end{align*}
	 due to~\eqref{eq:H}. For the forcing term, using H\"older's inequality gives
	 \begin{align*}
	 	-\int \partial_{y}^5 H \partial_{y}^3 \overline H \,dy \lesssim \frac{\nu}{2} \enorm{\partial_{y}^5 H}^2 + \frac{\nu^{-1}}{2} \enorm{\partial_{y}^3 \overline H}^2.
	 \end{align*}
	 Using~\eqref{defn:barh:cap} and \eqref{H:H2}, we obtain
	 \begin{align*}
	 	\int_{0}^{\nu^{-1/3-\zeta} } \enorm{\partial_{y}^3 \overline H}^2 \,dt \lesssim \eps \nu^{-7/2-\delta}
	 \end{align*}
	 for some $\delta>0.$
	 Therefore, it holds
	 \begin{align*}
	 	 \enorm{\partial_{y}^4 H} \lesssim \eps \nu^{-9/4-\delta/2}
	 \end{align*}
	 which is sufficient for $\dot{W}^{3, \infty}$ bound of $H$, concluding the proof.
	\end{proof}

\subsection{Properties of the Gevrey Spaces $Y_{\alpha,\beta}$}

\subsubsection{Binomial coefficients lemmas}
We need the following lemma about the binomial coefficients bound.

\begin{lemma}
	\label{comb:boun:vari:2}
	For $l\le n/2$ and $n\ge 5$, we have
	\begin{align}
		\label{comb:boun:vari:2:est}
	\brak{t}^{-2}	\bold{a}_{m,n+1} \binom{n}{l}\bold{a}_{m+1,l}^{-1}\binom{n-l}{j}\bold{a}_{0,j}^{-1}\bold{a}_{0,n-l-j}^{-1}
		\le \paren{\frac{1}{2}}^{l\paren{\sss-1}}.
	\end{align}
	Moreover, if $m\le c_0n$ for some $c_0>0$, then it holds
	\begin{align}
		\label{comb:boun:est:vari:2:refi}
	\brak{t}^{-2}	\bold{a}_{m,n+1} \binom{n}{l}\bold{a}_{m+1,l}^{-1}\binom{n-l}{j}\bold{a}_{0,j}^{-1}\bold{a}_{0,n-l-j}^{-1}
		\le \paren{\frac{1}{2}}^{l\paren{\sss-1}}
		\paren{\frac{c_0+1/2}{c_0+1}}^{m\sss}.
	\end{align}
	If $m\ge C_0n$, then we have
	\begin{align}
		\label{comb:boun:est:vari:2:refi1}
		\brak{t}^{-2} \bold{a}_{m,n+1}& \binom{n}{l}\bold{a}_{m+1,l}^{-1}\binom{n-l}{j}\bold{a}_{0,j}^{-1}\bold{a}_{0,n-l-j}^{-1} \le
		\paren{\frac{1}{2^l}}^{\sss-1}\paren{\frac{1}{C_0+1}}^{(n-l)\sss}.
	\end{align}
\end{lemma}
\begin{proof}
	By the definition of $\bold{a}_{m,n}$, we unfold the expression to be bounded in~\eqref{comb:boun:vari:2:est} as
	\begin{align*}
		\brak{t}^{-2} &\bold{a}_{m,n+1} \binom{n}{l}\bold{a}_{m+1,l}^{-1}\binom{n-l}{j}\bold{a}_{0,j}^{-1}\bold{a}_{0,n-l-j}^{-1}
		\\&=
		\Big( \frac{(\ss)^{m+n+1}}{(m+n+1)!} \Big)^{\sss}\paren{\frac{1}{C\brk{t}}}^{m+n+1}
		\frac{n!}{l!(n-l)!}
		\Big( \frac{(\ss)^{m+l+1}}{(m+l+1)!} \Big)^{-\sss}\paren{\frac{1}{C\brk{t}}}^{-(m+l+1)}
		\\&\quad\times \frac{(n-l)!}{j!(n-l-j)!}
		\Big( \frac{(\ss)^{j}}{j!} \Big)^{-\sss}\paren{\frac{1}{C\brk{t}}}^{-j}
		\Big( \frac{(\ss)^{n-l-j}}{(n-l-j)!} \Big)^{-\sss}\paren{\frac{1}{C\brk{t}}}^{-(n-l-j)}
		\\&=
		\Big( \frac{(m+l+1)!j!(n-l-j)!}{(m+n+1)!} \Big)^{\sss}	\frac{n!}{l!j!(n-l-j)!}	
		\\&\le
		\Big( \frac{(m+l+1)!}{(n-l+1)(n-l+2)\cdots (m+n+1)} \Big)^{\sss}	\frac{(n-l+1)\times\cdots \times n}{l!}	
		\\&=
		\Big( \frac{l!}{(n-l+1)(n-l+2)\cdots n} \Big)^{\sss-1}
		\Big( \frac{(l+1)\cdots (m+l+1)}{(n+1)(n + 2)\cdots (m+n+1)} \Big)^{\sss},
	\end{align*}
	from where, an easy algebraic computation gives
	\begin{align*}
\brak{t}^{-2}		\bold{a}_{m,n+1}& \binom{n}{l}\bold{a}_{m+1,l}^{-1}\binom{n-l}{j}\bold{a}_{0,j}^{-1}\bold{a}_{0,n-l-j}^{-1} \le \paren{\frac{1}{2^l}}^{\sss-1}.
	\end{align*}
	Actually, for $m\le c_0 n$, we have
	\begin{align*}
		\brak{t}^{-2} \bold{a}_{m,n+1}& \binom{n}{l}\bold{a}_{m+1,l}^{-1}\binom{n-l}{j}\bold{a}_{0,j}^{-1}\bold{a}_{0,n-l-j}^{-1} \le \paren{\frac{1}{2^l}}^{\sss-1}\paren{\frac{c_0+1/2}{c_0+1}}^{m\sss},
	\end{align*}
	proving~\eqref{comb:boun:est:vari:2:refi}. 
	While $m\ge C_0n$, then it is easy to obtain 
	\begin{align*}
		\brak{t}^{-2} \bold{a}_{m,n+1}& \binom{n}{l}\bold{a}_{m+1,l}^{-1}\binom{n-l}{j}\bold{a}_{0,j}^{-1}\bold{a}_{0,n-l-j}^{-1} \le
		\paren{\frac{1}{2^l}}^{\sss-1}\paren{\frac{n}{m+n}}^{(n-l)\sss}
		\\&
		\le \paren{\frac{1}{2^l}}^{\sss-1}\paren{\frac{1}{C_0+1}}^{(n-l)\sss},
	\end{align*}
	concluding the proof.
\end{proof}
\begin{remark}
	Note that in the above theorem, for $\bold{a}_{\alpha,\beta}$, it is essentially the sum $\alpha+\beta$ that matters. Hence, we are able to generalize it to a more general version:
	\begin{align}
		\label{comb:boun:rem}
		\brak{t}^{-2} \bold{a}_{m,n+1} \binom{n}{l}\bold{a}_{\alpha_1, \beta_1}^{-1}\binom{n-l}{j}\bold{a}_{\alpha_2, \beta_2}^{-1}\bold{a}_{\alpha_3, \beta_3}^{-1}
		\le \paren{\frac{1}{2}}^{l\paren{\sss-1}}.
	\end{align}
	with 
	\begin{align*}
		\alpha_1+\beta_1=m+l+1,\    \alpha_2+\beta_2=j,\     \alpha_3+\beta_3=n-l-j.
	\end{align*}
	Of course, the corresponding inequalities for \eqref{comb:boun:est:vari:2:refi} and \eqref{comb:boun:est:vari:2:refi1} also holds.
\end{remark}

\begin{remark}
	From Theorem~\ref{comb:boun:vari:2}, one may also obtain
	\begin{align}
		\label{comb:boun:rem:1}
		\brak{t}^{-2} \bold{a}_{m,n+1} \binom{n}{l}\bold{a}_{m+1,l}^{-1}\binom{n-l}{j}\bold{a}_{0,j}^{-1}\bold{a}_{0,n-l-j}^{-1} m^a
		\le C \paren{\frac{1}{2}}^{l\paren{\sss-1}} 
	\end{align}
	for a parameter $a\le n/4$, where the constant $C$ depends on $a$.
\end{remark}

Two direct corollaries are the following.
\begin{corollary}
	\label{comb:boun}
	For $l\le n/2$ and $n\ge 5$, we have
	\begin{align}
		\label{comb:boun:est}
	\brak{t}^{-1}	\bold{a}_{m,n} \binom{n}{l}\bold{a}_{m,l}^{-1}\bold{a}_{0,n-l}^{-1} \le \paren{\frac{1}{2}}^{l\paren{\sss-1}}.
	\end{align}
	Moreover, if $m\le c_0n$ for some $c_0>0$, then it holds
	\begin{align}
		\label{comb:boun:est:refi}
	\brak{t}^{-1}	\bold{a}_{m,n} \binom{n}{l}\bold{a}_{m,l}^{-1}\bold{a}_{0,n-l}^{-1} \le \paren{\frac{1}{2}}^{l\paren{\sss-1}}
		\paren{\frac{c_0+1/2}{c_0+1}}^{m\sss}.
	\end{align}
	If $m\ge C_0n$, then we have
	\begin{align}
		\label{comb:boun:est:refi1}
		\brak{t}^{-1} \bold{a}_{m,n} \binom{n}{l}\bold{a}_{m,l}^{-1}\bold{a}_{0,n-l}^{-1} \le
		\paren{\frac{1}{2^l}}^{\sss-1}\paren{\frac{1}{C_0+1}}^{(n-l)\sss}.
	\end{align}
\end{corollary}
\begin{corollary}
	\label{comb:boun:vari:1}
	For $l\le n/2$ and $n\ge 5$, we have
	\begin{align}
		\label{comb:boun:vari:1:est}
		\brak{t}^{-1}\bold{a}_{m,n+1} \binom{n}{l}\bold{a}_{m+1,l}^{-1}\bold{a}_{0,n-l}^{-1} \le \paren{\frac{1}{2}}^{l\paren{\sss-1}}.
	\end{align}
	Moreover, if $m\le c_0n$ for some $c_0>0$, then it holds
	\begin{align}
		\label{comb:boun:est:vari:1:refi}
		\brak{t}^{-1}\bold{a}_{m,n+1} \binom{n}{l}\bold{a}_{m+1,l}^{-1}\bold{a}_{0,n-l}^{-1} \le \paren{\frac{1}{2}}^{l\paren{\sss-1}}
		\paren{\frac{c_0+1/2}{c_0+1}}^{m\sss}.
	\end{align}
	If $m\ge C_0n$, then we have
	\begin{align}
		\label{comb:boun:est:vari:1:refi1}
		\brak{t}^{-1}\bold{a}_{m,n+1} \binom{n}{l}\bold{a}_{m+1,l}^{-1}\bold{a}_{0,n-l}^{-1}
		\le
		\paren{\frac{1}{2^l}}^{\sss-1}\paren{\frac{1}{C_0+1}}^{(n-l)\sss}.
	\end{align}
\end{corollary}
\subsubsection{Product rule}
We need the following product lemmas.
\begin{lemma}
	\label{pro:0}
	For $f, g\in Y_{0,0}$ with $f|_{y=\pm1}=g|_{y=\pm1}=0$, we have
	\begin{align*}
		\norm{fg}_{Y_{0,0}} \lesssim 
		\norm{ g}_{Y_{0,0}}
		\norm{f}_{\overline Y_{1,0}} + \norm{f}_{Y_{0,0}}
		\norm{g}_{\overline Y_{1,0}}.
	\end{align*}
	For $f, g$ which do not satisfy the homogeneous boundary condition $f|_{y=\pm1}=g|_{y=\pm1}=0$, it holds
		\begin{align*}
		\norm{fg}_{Y_{0,0}} \lesssim 
		\norm{ f}_{Y_{0,0}} \paren{\norm{ g}_{\overline Y_{1,0}}+ \enorm{g}}
		 + \norm{ g}_{Y_{0,0}}
		 \paren{\norm{ f}_{\overline Y_{1,0}} + \enorm{f} }.
	\end{align*}
\end{lemma}
\begin{proof}
	We recall that
	\begin{align*}
		\norm{fg}_{Y_{0,0}}^2 &= \sum_{n \ge 0}   \| \bb_{n}\bd^{n} (fg) e^{W} \chi_n \|_{L^2_y}^2
		\le 
		\sum_{n \ge 0}\sum_{l=0}^{n}  \bb_{n}^2   \binom{n}{l}^2 \| \bd^{n-l} f \bd^{l} g e^{W} \chi_n \|_{L^2_y}^2
		\\& = P^{HL} + P^{LH}
	\end{align*}
	where
	\begin{align*}
		P^{HL} = \sum_{n \ge 0}\sum_{l\le n/2}   \bb_{n}^2  \binom{n}{l}^2 \| \bd^{n-l} f \bd^{l} g e^{W} \chi_n \|_{L^2_y}^2
	\end{align*}
	and
	\begin{align*}
		P^{LH} = \sum_{n \ge 0}\sum_{n/2<l\le n}   \bb_{n}^2  \binom{n}{l}^2 \| \bd^{n-l} f \bd^{l} g e^{W} \chi_n \|_{L^2_y}^2.
	\end{align*}
	Since the real difficult part is when $n$ is huge, we assume $n\ge 5$.
	By H\'older's inequality and Sobolev embedding, it follows easily
	\begin{align*}
		P^{HL} &\le  \sum_{n \ge 0}\sum_{l\le n/2}    \bb_{n}^2  \binom{n}{l}^2 \| \bd^{n-l} f  e^{W} \chi_n \|_{L^2_y}^2 
		\| \bd^{l} g \chi_{l+1} \|_{L^\infty_y}^2
		\\&\lesssim  \sum_{n \ge 0}\sum_{l\le n/2}    \bb_{n}^2  \binom{n}{l}^2 \| \bd^{n-l} f  e^{W} \chi_n \|_{L^2_y}^2 
		\paren{ \| \partial_y \bd^{l} g \chi_{l+1} \|_{L^2_y}^2 + (1+l)^2 \| \bd^{l} g \chi_{l} \|_{L^2_y}^2}
		\\&\lesssim  	\norm{ f}_{Y_{0,0}}^2
		\sum_{l \ge 0}    \bb_{l}^2  
		\paren{ \| \partial_y \bd^{l} g \chi_{l} \|_{L^2_y}^2 +   \| \bd^{l} g \chi_{l} \|_{L^2_y}^2}
		\lesssim  	\norm{ f}_{Y_{0,0}}^2 \norm{ g}_{\overline Y_{1,0}}^2
	\end{align*}
	where we used Lemma~\ref{comb:boun}, Young's inequality, and Poincar\'e's inequality in the last line.
	Similarly, for the $LH$ part, we have
	\begin{align*}
		P^{LH} 
		\lesssim  	\norm{ g}_{Y_{0,0}}^2
		\norm{ f}_{\overline Y_{1,0}}^2.
	\end{align*}
	On the other hand, 	without requiring the homogeneous boundary condition $f|_{y=\pm1}=g|_{y=\pm1}=0$, then we have
	\begin{align*}
		P^{HL}
		\lesssim  	\norm{ f}_{Y_{0,0}}^2 \paren{\norm{ g}_{\overline Y_{1,0}}^2+ \enorm{g}^2}
	\end{align*}
	and 
		\begin{align*}
		P^{LH} 
		\lesssim  	\norm{ g}_{Y_{0,0}}^2
		\paren{\norm{ f}_{\overline Y_{1,0}}^2 + \enorm{f}^2}
	\end{align*}
	concluding the proof.
\end{proof}

\begin{lemma}
	\label{pro:-s}
	For $f|_{y=\pm1}=g|_{y=\pm1}=0$, we have
	\begin{align*}
		\norm{fg}_{Y_{0,-s}} \lesssim 
		\norm{ f}_{Y_{0,-s}}
		\norm{g}_{\overline Y_{1,0}} + \norm{f}_{Y_{1,-s}}
		\norm{g}_{\overline Y_{0,0}}.
	\end{align*}
	Without the homogeneous boundary condition, it holds
	\begin{align*}
			\norm{fg}_{Y_{0,-s}} \lesssim 
		\norm{ f}_{Y_{0,-s}}
		\paren{\norm{g}_{\overline Y_{1,0}} + \enorm{g}} + \paren{\norm{f}_{Y_{1,-s}} + \enorm{f}}
		\norm{g}_{\overline Y_{0,0}}.
	\end{align*}
\end{lemma}
\begin{proof}
	As in the previous lemma, we have
	\begin{align*}
		\norm{fg}_{Y_{0,-s}}^2 &= \sum_{n \ge 0}n^{2\sss-2}   \| \bb_{n+1}\bd^{n} (fg) e^{W} \chi_n \|_{L^2_y}^2
		\\&\le 
		\sum_{n \ge 0}\sum_{l=0}^{n} n^{2\sss-2}    \bb_{n+1}^2  \binom{n}{l}^2 \| \bd^{n-l} f \bd^{l} g e^{W} \chi_n \|_{L^2_y}^2
		\\& = P^{HL} + P^{LH}
	\end{align*}
	where
	\begin{align*}
		P^{HL} = \sum_{n \ge 0}\sum_{l\le n/2}n^{2\sss-2}   \bb_{n+1}^2  \binom{n}{l}^2 \| \bd^{n-l} f \bd^{l} g e^{W} \chi_n \|_{L^2_y}^2
	\end{align*}
	and
	\begin{align*}
		P^{LH} = \sum_{n \ge 0}\sum_{n/2<l\le n}n^{2\sss-2}   \bb_{n+1}^2  \binom{n}{l}^2 \| \bd^{n-l} f \bd^{l} g e^{W} \chi_n \|_{L^2_y}^2.
	\end{align*}
	The rest follows exactly as in Lemma~\ref{pro:0} once we notice
	\begin{align*}
		n^{\sss-1}  \bb_{n+1} \lesssim   \bb_{n}
	\end{align*}
	and hence we	conclude the proof.
\end{proof}

\begin{lemma}
	\label{pro:1}
	Assume $f|_{y=\pm1}=g|_{y=\pm1}=0$, we have
	\begin{align*}
		\norm{fg}_{Y_{1,0}} \lesssim 
		\norm{f}_{Y_{1,0}}
		\norm{g}_{\overline Y_{1,0}}.
	\end{align*}
	Without the boundary condition requirement, it holds
	\begin{align*}
		\norm{fg}_{Y_{1,0}} \lesssim 
		\paren{\norm{f}_{Y_{1,0}} + \enorm{f}}
		\paren{\norm{g}_{\overline Y_{1,0}} + \enorm{g}}.
	\end{align*}
\end{lemma}

\begin{proof}
	We first recall
	\begin{align*}
		\norm{fg}_{Y_{1,0}}^2 &= \sum_{n \ge 0}   \| \bb_{n}\partial_y\bd^{n} (fg) e^{W} \chi_n \|_{L^2_y}^2
		\le 
		\sum_{n \ge 0}\sum_{l=0}^{n}   \bb_{n}^2  \binom{n}{l}^2 \| \partial_y\paren{\bd^{n-l} f \bd^{l} g }e^{W} \chi_n \|_{L^2_y}^2
		\\& = \sum_{n \ge 0}\sum_{l=0}^{n}   \bb_{n}^2  \binom{n}{l}^2 \| \partial_y\bd^{n-l} f \bd^{l} ge^{W} \chi_n \|_{L^2_y}^2
		+
		\sum_{n \ge 0}\sum_{l=0}^{n}   \bb_{n}^2  \binom{n}{l}^2 \| \bd^{n-l} f \partial_y\bd^{l} ge^{W} \chi_n \|_{L^2_y}^2
		\\& =  P_1+P_2.
	\end{align*}
	$P_1$ and $P_2$ are treated similarly and we only show details for $P_1$. 
	Note that 
	\begin{align*}
		P_1 = P^{HL} + P^{LH}
	\end{align*}
	where
	\begin{align*}
		P^{HL} = \sum_{n \ge 0}\sum_{l\le n/2}   \bb_{n}^2  \binom{n}{l}^2 \| \partial_y\bd^{n-l} f \bd^{l} g e^{W} \chi_n \|_{L^2_y}^2
	\end{align*}
	and
	\begin{align*}
		P^{LH} = \sum_{n \ge 0}\sum_{n/2<l\le n}   \bb_{n}^2  \binom{n}{l}^2 \| \partial_y\bd^{n-l} f \bd^{l} g e^{W} \chi_n \|_{L^2_y}^2.
	\end{align*}
	Again we assume $n\ge 5$. The treatment of the term $P^{HL}$ is exactly the same as in Lemma~\ref{pro:0} and we next give the bound for $P^{LH}$.
	By the definition of $\bd$, H\'older's inequality, Poincar\'e's inequality, and Sobolev embedding, it follows easily
	\begin{align*}
		P^{LH} &\le  \sum_{n \ge 0}\sum_{n/2<l\le n}   \bb_{n}^2  \binom{n}{l}^2 \|  \bd^{n-l+1} f  e^{W} \chi_{n-l+2} \|_{L^\infty_y}^2 
		\| v_y q^{l-1}\Gamma_0^{l}  g \chi_{n} \|_{L^2_y}^2
		\\&\lesssim  \sum_{n \ge 0}\sum_{n/2<l\le n}   \bb_{n}^2  \binom{n}{l}^2
		\paren{ \|  \partial_y\bd^{n-l+1} f   \chi_{n-l+2} \|_{L^2_y}^2 + (n-l+2)^2\|  \bd^{n-l+1} f  \chi_{n-l+2} \|_{L^2_y}^2 }
		\\&\quad\times	\| v_y q^{l-1}\Gamma_0^{l}  g e^{W} \chi_{n} \|_{L^2_y}^2
		\\&\lesssim 	\norm{ f}_{Y_{1,0}}^2 \norm{ g}_{\overline Y_{1,0}}^2
	\end{align*}
	where we used Lemma~\ref{comb:boun}, Young's inequality, and an argument similar to  {Lemma~5.1 in~\cite{BHIW24b}} in the last line.
	The proof is completed.
\end{proof}
Similarly we have the lemma below.

\begin{lemma}
	\label{pro:1:-s}
	Suppose $f|_{y=\pm1}=g|_{y=\pm1}=0$, we have
	\begin{align*}
		\norm{fg}_{Y_{1,-s}} \lesssim 
		\norm{f}_{Y_{1,-s}}
		\norm{g}_{\overline Y_{1,0}}.
	\end{align*}
	Without boundary value requirement, it holds
	\begin{align*}
		\norm{fg}_{Y_{1,-s}} \lesssim 
		\paren{\norm{f}_{Y_{1,-s}} + \enorm{f}}
		\paren{\norm{g}_{\overline Y_{1,0}} + \enorm{g}}.
	\end{align*}
\end{lemma}
\begin{proof}
	The proof is essentially covered in the proof of Lemma~\ref{pro:0}--\ref{pro:1} and is omitted.
\end{proof}
\begin{remark}
	\label{bar:nor}
	We note that in the proof of the above product lemmas, the localization weight $e^{W/2}$ does not play a role. Hence we remark that all the product inequalities in this subsection hold for $\overline Y$ norms, i.e., replacing $\norm{\cdot}_{Y_{\alpha, \beta }}$ by $\norm{\cdot}_{\overline  Y_{\alpha, \beta }}$ everywhere. 
\end{remark}

\subsubsection{Convolution type lemmas}
\begin{lemma}
	\label{con:no:k}
	It holds
	\begin{align*}
		\sum_{n\ge 0}\sum_{0\le m\le n}\sum_{k\in\mathbb{Z}} f_{n-m,k}g_{m,k}h_n
		\le
		\sum_{n\ge 0} \paren{\sum_{k\in\mathbb{Z}}|f_{n,k}|^2}^{1/2} 
		\paren{\sum_{n\ge 0}\sum_{k\in\mathbb{Z}}|g_{n,k}|^2}^{1/2}
		\paren{\sum_{n\ge 0}|h_{n}|^2}^{1/2}.
	\end{align*}
\end{lemma}
\begin{proof}
	The inequality is a result of H\'older's and Young's inequality:
	\begin{align*}
		\sum_{n\ge 0}\sum_{0\le m\le n}\sum_{k\in\mathbb{Z}} f_{n-m,k}g_{m,k}h_n
		&\le \sum_{n\ge 0}\sum_{0\le m\le n}\paren{\sum_{k\in\mathbb{Z}}|f_{n,k}|^2}^{1/2}
		\paren{\sum_{k\in\mathbb{Z}}|g_{m,k}|^2}^{1/2}h_n
		\\&\le 
		\paren{\sum_{n\ge 0}\paren{\sum_{0\le m\le n}\paren{\sum_{k\in\mathbb{Z}}|f_{n,k}|^2}^{1/2}
				\paren{\sum_{k\in\mathbb{Z}}|g_{m,k}|^2}^{1/2}}^2}^{1/2}\paren{\sum_{n\ge 0}|h_{n}|^2}^{1/2}
		\\&\le
		\sum_{n\ge 0} \paren{\sum_{k\in\mathbb{Z}}|f_{n,k}|^2}^{1/2} 
		\paren{\sum_{n\ge 0}\sum_{k\in\mathbb{Z}}|g_{n,k}|^2}^{1/2}
		\paren{\sum_{n\ge 0}|h_{n}|^2}^{1/2}.
	\end{align*}
\end{proof}
We also need the following version of the convolution type lemma.
\begin{lemma}
	\label{con:k}
	It holds
	\begin{align*}
		\sum_{n\ge 0}\sum_{0\le m\le n}\sum_{k\in\mathbb{Z}}\sum_{l\in\mathbb{Z}} f_{n-m,k-l}g_{m,l}h_{n,k}
		\le
		\paren{\sum_{n\ge 0}\sum_{k\in\mathbb{Z}}|f_{n,k}|}
		\paren{\sum_{n\ge 0}\sum_{k\in\mathbb{Z}}|g_{n,k}|^2}^{1/2}
		\paren{\sum_{n\ge 0}\sum_{k\in\mathbb{Z}}|h_{n,k}|^2}^{1/2}.
	\end{align*}
\end{lemma}
\begin{proof}
	Again by H\'older's and Young's inequality:
	\begin{align*}
		\sum_{n\ge 0}&\sum_{0\le m\le n}\sum_{k,l\in\mathbb{Z}} f_{n-m,k-l}g_{m,l}h_{n,k}
		\le 
		\paren{\sum_{n}\sum_{k}\paren{\sum_{ m}\sum_{l}
				|f_{n-m,k-l}g_{m,l}|}^2}^{1/2}
		\paren{\sum_{n\ge 0}\sum_{k\in\mathbb{Z}}|h_{n,k}|^2}^{1/2}
		\\&\le 
		\paren{\sum_{n\ge 0}\sum_{k\in\mathbb{Z}}|f_{n,k}|}
		\paren{\sum_{n\ge 0}\sum_{k\in\mathbb{Z}}|g_{n,k}|^2}^{1/2}
		\paren{\sum_{n\ge 0}\sum_{k\in\mathbb{Z}}|h_{n,k}|^2}^{1/2}.
	\end{align*}
\end{proof}

\subsubsection{Sobolev embedding}
In this section, we establish the Sobolev embedding in term of $\Gamma_k$ derivative. Indeed, we have the following simple lemma.
\begin{lemma}
	With $\Gamma_k = \partial_{y} + v_yikt$, we have inequality
	\begin{align}
		\label{Sob:emb}
		\norm{f_k}_{L^\infty_y} \lesssim \enorm{f_k}^{1/2}\enorm{\Gamma_k f_k}^{1/2}.
	\end{align}
\end{lemma}
\begin{proof}
	This lemma is a direct consequence of Sobolev embedding lemma. In fact note that 
	\begin{align*}
		\norm{f_k}_{L^\infty_y} =	\norm{f_k e^{vikt}}_{L^\infty_y}	
		\lesssim \norm{f_k e^{vikt}}_{L^2_y}^{1/2}	\norm{\partial_{y} \paren{f_k e^{vikt}}}_{L^2_y}^{1/2}
		= \enorm{f_k}^{1/2}\enorm{\Gamma_k f_k}^{1/2}
	\end{align*}	
	concluding the proof.
\end{proof}

\subsection{Proof of Proposition \ref{gamm:esti:llp}}

We are now ready to consolidate the bounds established in this section in order to prove Proposition \ref{gamm:esti:llp}.
\begin{proof}[Proof of Proposition \ref{gamm:esti:llp}]
It is a direct consequence of Proposition~\ref{gamm:esti} and Proposition~\ref{alph:esti} in this section.
\end{proof}
%

\section{Sobolev ``Cloud" Norms} \label{sec:CLOUD}
\subsection{Norms and Energy Functionals}

We set up our norms and energy functionals for our ``Cloud" norms. For this purpose, we first recall the definition of the interior cutoff, $\chi^I$, from \eqref{chi:I:def}, and notice that 
\begin{align} \label{shgy:1}
\text{supp}(\nabla^k \chi^I \subset \cap_{n = 1}^\infty \{ \chi_n = 1  \} = \{ \chi_E = 1 \}.
\end{align}
 We define 
\begin{align}
\mathcal{E}_{\text{cloud}}^{(m, n)} := &\| \varphi^{n+1} \p_x^m \Gamma^n \omega e^W \chi^I \|_{L^2}^2, \\
\mathcal{D}_{\text{cloud}}^{(m, n)} := &\|\varphi^{n+1} \sqrt{\nu} \nabla \p_x^m \Gamma^n \omega e^W \chi^I \|_{L^2}^2 \\
\mathcal{CK}_{\text{cloud, W}}^{(m, n)} := &\|\varphi^{n+1} \sqrt{-\dot{W}} \p_x^m \Gamma^n \omega e^W \chi^I \|_{L^2}^2 \\
\mathcal{CK}_{\text{cloud}, \varphi}^{(m, n)} := &\|\varphi^{n+1} \sqrt{-\dot{\varphi}} \p_x^m \Gamma^n \omega e^W \chi^I \|_{L^2}^2
\end{align} 
 
We now write the equation: 
\begin{align} \label{eq:sa:1}
\p_t \omega_{m,n} + (y + U^0(t, y)) \p_x \omega_{m,n} - \nu \Delta \omega_{m,n} = F_{m,n},
\end{align}
where the source term 
\begin{align}
F_{m,n} = \bold{T}^{(\gamma)}_{m,n} + \bold C_{\mathrm{trans},k}^{(n)}+\bold C_{\mathrm{visc},k}^{(n)},  
\end{align}
where the quadratic terms are defined as follows 
\begin{align} \label{wh:1}
 \bold{T}^{(\gamma)}_{m,n} := \sum_{n' = 0}^{n} \sum_{m' = 0}^{m} \mathbbm{1}_{m' + n' < m + n} \nabla^\perp \psi^{(\neq 0)}_{m - m',n - n'} \cdot \nabla \omega_{m', n'} + \nabla^\perp \psi^{(\neq 0)} \cdot \nabla \omega_{m,n},
\end{align}
and where the linear commutators $\bold C_{\mathrm{trans},k}^{(n)}$ and $\bold C_{\mathrm{visc},k}^{(n)}$ are defined in [Lemma 4.2; \cite{BHIW24b}] 

\subsection{Cloud Energy Identity}

We derive the following energy identity 
\begin{lemma} The following energy identity holds: 
\begin{align} \label{bhbhy:1}
\frac{\p_t}{2} \mathcal{E}_{\text{cloud}}^{(m, n)} + \mathcal{D}_{\text{cloud}}^{(m, n)} + \mathcal{CK}_{\text{cloud}}^{(m, n)} \lesssim \frac{\nu^{\frac13 - 2\zeta}}{\langle  t\rangle^{2}}  \mathcal{E}^{(\gamma)} + \sum_{i = 1}^3 \text{Err}^{(m, n)}_{\text{cloud}, i},
\end{align}
where we define the error terms on the right-hand side as follows 
\begin{subequations}
\begin{align}
\text{Err}^{(m, n)}_{\text{cloud}, 1} := & |\langle \bold{T}^{(\gamma)}_{m,n}, \omega_{m,n} e^{2W} \varphi^{2n+ 2} (\chi^I)^2 \rangle|, \\
\text{Err}^{(m, n)}_{\text{cloud}, 2} := & |\langle \bold C_{\mathrm{trans},k}^{(n)}, \omega_{m,n} e^{2W} \varphi^{2n + 2} (\chi^I)^2 \rangle|, \\
\text{Err}^{(m, n)}_{\text{cloud}, 3} := &\nu |\langle \partial_x^m\bold C_{\mathrm{visc},k}^{(n)}, \omega_{m,n} e^{2W} \varphi^{2n + 2} (\chi^I)^2 \rangle|. 
\end{align}
\end{subequations}
\end{lemma}
\begin{proof} We apply the multiplier $\omega_{m,n} e^{2W} \varphi^{2n+2} (\chi^I)^2$ to equation \eqref{eq:sa:1}, which produces then the following identity 
\begin{align}
&\langle F_{m,n}, \omega_{m,n} (\chi^I)^2 e^{2W} \varphi^{2n+2} \rangle \\
= & \langle \p_t \omega_{m,n} + (y + U^0(t, y)) \p_x \omega_{m,n} - \nu \Delta \omega_{m,n}, \omega_{m,n} (\chi^I)^2 e^{2W} \varphi^{2n + 2} \rangle \\
= & \frac{\p_t}{2} \| \omega_{m,n} e^W \chi^I \varphi^{n+1} \|_{L^2}^2 + (1 + n) \| \sqrt{-\frac{\dot{\varphi}}{\varphi}}  \omega_{m,n} e^W \chi^I \varphi^{n+1} \|_{L^2}^2 \\
& +  \| \sqrt{-\dot{W}}  \omega_{m,n} e^W \chi^I \varphi^{n+1} \|_{L^2}^2 + \| \sqrt{\nu} \nabla \omega_{m,n} e^W \chi^I \varphi^{n + 1} \|_{L^2}^2 + E_{D,1} + E_{D,2}, 
\end{align}
where we define the commutator terms above via 
\begin{align}
E_{D,1} := &\langle \nu \p_y \omega_{m,n}, \omega_{m,n} \p_y \{ e^W \} (\chi^I)^2 \varphi^{2(n+ 1)} \rangle \\
E_{D,2} := &  \langle \nu \p_y \omega_{m,n}, \omega_{m,n}  e^W \p_y\{ (\chi^I)^2\} \varphi^{2(n+ 1)} \rangle.
\end{align}
We estimate the error term $E_{D,2}$ as follows, upon using \eqref{shgy:1}
\begin{align} \n
|E_{D,2}| \lesssim & \| \sqrt{\nu} \p_y \omega_{m,n} e^W \chi^{(n+m)}_{\text{cloud}} \varphi^{n+m} \|_{L^2} \| \sqrt{\nu} \omega_{m,n} \chi_{E} \varphi^{2(n+1)} e^W \|_{L^2} \\ \n
\lesssim &  \sqrt{\mathcal{D}_{\text{cloud}}^{(n+m)}} \frac{1}{\langle t \rangle^{1}} (\sup_{t \le \nu^{-\frac13 - \zeta}} |\sqrt{\nu} t |)  \|  \omega_{m,n} \chi_E \varphi^{(n+1)} e^W \|_{L^2} \\  \n
\lesssim & \frac{\nu^{\frac16 - \zeta}}{\langle t \rangle^{1}} \sqrt{ \mathcal{D}_{\text{cloud}}^{(n+m)}} \sqrt{ \mathcal{E}^{(\gamma)}(t)}. 
\end{align}
We estimate the error term $E_{D,1}$ using an identical calculation to \eqref{I_21}, which results in the bound $|E_{D,1}| \le \frac14 |\mathcal{D}_{\text{cloud}}^{(m, n)}|^2 + \frac{1}{8}| \mathcal{CK}_{\text{cloud}}^{(m, n)}|^2$. These contributions are then absorbed into the left-hand side of \eqref{bhbhy:1}.
\end{proof}

\subsection{Cloud Error Terms Bounds}

We now proceed to control the various terms appearing on the right-hand side of \eqref{bhbhy:1}. 
\begin{lemma} For any $m, n$ satisfying $0 \le m + n \le 20$, the following bound holds 
\begin{align} \label{merced:1}
|\text{Err}^{(m, n)}_{\text{cloud}, 1}| \lesssim \frac{1}{\langle t \rangle^2} (\sqrt{\slashed{\mathcal{E}}_{ell}} + \sqrt{ \mathcal{J}_{ell}^{(2)}(t)} )\mathcal{E}_{\text{cloud}}(t) + \frac{1}{\langle t \rangle^2}  (\sqrt{\slashed{\mathcal{E}}_{ell}} + \sqrt{ \mathcal{J}_{ell}^{(2)}(t)} )\mathcal{E}^{(\gamma)} + \frac{1}{L} \mathcal{CK}^{(\gamma, W)}.
\end{align}
\end{lemma}
\begin{proof} We use the decomposition shown in \eqref{wh:1} to obtain 
\begin{align} \n
&|\langle \bold{T}^{(\gamma)}_{m,n}, \omega_{m,n} e^{2W} \varphi^{2(n+1)} (\chi^I)^2 \rangle| \\ \n
\le & \langle \sum_{n' = 0}^{n} \sum_{m' = 0}^{m} \mathbbm{1}_{m' + n' < m + n} \nabla^\perp \psi^{(\neq 0)}_{m - m',n - n'} \cdot \nabla \omega_{m', n'} , \omega_{m,n} e^{2W} \varphi^{2(n+1)}(\chi^I)^2 \rangle \\ \n
& + \langle \nabla^\perp \psi^{(\neq 0)} \cdot \nabla \omega_{m,n}, \omega_{m,n} e^{2W} \varphi^{2(n+1)} (\chi^I)^2 \rangle \\ \n
= & \langle \sum_{n' = 0}^{n} \sum_{m' = 0}^{m} \mathbbm{1}_{m' + n' < m + n} v_y (- \psi^{(m-m' + 1, n - n')} \omega^{m', n' + 1)} + \psi^{m-m', n-n' + 1} \omega^{m' + 1, n'}  ) \\ 
& , \omega_{m,n} e^{2W} \varphi^{2(n+1)} (\chi^I)^2 \rangle + \langle \nabla^\perp \psi^{(\neq 0)} \cdot \nabla \omega_{m,n}, \omega_{m,n} e^{2W} \varphi^{2(n+1)} |(\chi^I)^2 \rangle \\ \n
 = & \text{Err}_1 + \text{Err}_2. 
\end{align}
For $\text{Err}_1$, we crucially use the ``closure" property of the set $m + n \le 20$ as follows: 
\begin{align*}
|\text{Err}_1| \lesssim & \| \psi^{21, 21} q^{200} \|_{L^\infty} \Big( \sum_{m' + n' \le 20} \| \omega^{(m', n')} \varphi^{1 + n'} e^W \chi^I \|_{L^2}\Big) \| \omega^{(m, n)} e^W \chi_{n + m} \varphi^{n +1} \|_{L^2} \\
\lesssim & \frac{1}{\langle t \rangle}(\sqrt{ \slashed{\mathcal{E}}_{ell}} + \sqrt{\mathcal{J}_{ell}^{(2)}}) \mathcal{CK}_{\text{cloud}}
\end{align*}
Above, we have used the support of $\chi^I$ to insert factors of $q^{200}$ into the $\psi$ term. We have also used that $\varphi^{2n + 2} \le \varphi^{2 + n'} \varphi^{1 + n} \varphi^{-1}$. 

For $\text{Err}_2$, we integrate by parts as follows:
\begin{align*}
\text{Err}_2 = & - \frac12 \langle \p_x \psi_{\neq 0} \omega_{m,n}, \omega_{m,n} \p_y \{ e^{2W} \} \varphi^{2(n+m)} (\chi^I)^2 \rangle \\
 & - \frac12 \langle \p_x \psi_{\neq 0} \omega_{m,n}, \omega_{m,n} e^{2W}  \varphi^{2(n+m)} \p_y \{ (\chi^I)^2 \} \rangle = \text{Err}_{2, 1} + \text{Err}_{2, 2},
\end{align*}
and we subsequently estimate 
\begin{align} \n
|\text{Err}_{2,1}| \lesssim & \| \p_x \psi_{\neq 0} \|_{L^\infty} \| \omega_{m,n} \sqrt{|\p_y W|} e^W \varphi^{m + n} \chi_E \|_{L^2}^2 \lesssim \frac{1}{\langle t \rangle^2} \mathcal{E}_{ell}(t)| \| \omega_{m,n} \sqrt{|\p_y W|} e^W \varphi^{m + n} \chi_E \|_{L^2}^2 \\ \label{cloud:ibp:1}
\le & \frac{C}{L} \| \sqrt{-\dot{W}} \omega_{m,n} e^W \varphi^{m + n} \chi_E \|_{L^2}^2 \le \frac{C}{L} \mathcal{CK}^{(\gamma, W)} \\ \label{cloud:ibp:2}
|\text{Err}_{2,2}| \lesssim & \| \p_x \psi_{\neq 0} \|_{L^\infty} \| \omega_{m,n} e^W \varphi^{m + n} \chi_{E} \|_{L^2}^2 \lesssim \frac{1}{\langle t \rangle^2} (\sqrt{ \slashed{\mathcal{E}}_{ell}(t)} + \sqrt{\mathcal{J}^{(2)}_{ell}}) \mathcal{E}^{(\gamma)},
\end{align}
where for the term $\text{Err}_{2,1}$, we have invoked the pointwise bound \eqref{wdot:est:b}. 
\end{proof}
\begin{lemma} For any $m, n$ satisfying $0 \le m + n \le 20$, the following bound holds 
\begin{align} \label{merced:2}
|\text{Err}^{(m, n)}_{\text{cloud}, 2}| \lesssim \frac{\eps}{\brak{t}^{1+s^{-1}}} \mathcal{E}_{cloud}.
\end{align}
\end{lemma}
\begin{proof} 
Recall that 
\begin{align*}
	\text{Err}^{(m, n)}_{\text{cloud}, 2} = & |\langle \partial_{x}^m \bold C_{\mathrm{trans},k}^{(n)}, \omega_{m,n} e^{2W} \varphi^{2n + 2} (\chi^I)^2 \rangle|
	\\=& \abs{\left\langle \sum_{l = 1}^n \binom{n}{l} (\frac{\pa_y}{v_y})^{l-1}
	\paren{\frac{\overline{H}}{v_y}}  \omega_{m, n-l+1}, \omega_{m,n} e^{2W} \varphi^{2n + 2} (\chi^I)^2 \right\rangle}.
\end{align*}
Using H\"older's inequality and Sobolev embedding, we arrive at
\begin{align*}
	\text{Err}^{(m, n)}_{\text{cloud}, 2} \lesssim &  \sum_{l = 1}^n  \norm{\paren{\frac{\pa_y}{v_y}}^{l-1} \paren{\frac{\overline{H}}{v_y}} \varphi^{l-1} \tilde \chi^I}_{L^\infty_y} 	\enorm{\omega_{m, n-l+1} e^{W} \varphi^{n -l+ 2} \chi^I}
		\\&\quad \times 
	\enorm{ \omega_{m,n} e^{W} \varphi^{n + 1} \chi^I}
	\\ \lesssim &  \sum_{l = 1}^n  \enorm{\paren{\frac{\pa_y}{v_y}}^{l-1} \paren{\frac{\overline{H}}{v_y}} \varphi^{l-1} \tilde \chi^I}^{1/2} 	
	\enorm{\partial_{y}\paren{\paren{\frac{\pa_y}{v_y}}^{l-1} \paren{\frac{\overline{H}}{v_y}} \varphi^{l-1} \tilde \chi^I}}^{1/2}
		\\&\quad \times 
	\enorm{\omega_{m, n-l+1} e^{W} \varphi^{n -l+ 2} \chi^I}
	\enorm{ \omega_{m,n} e^{W} \varphi^{n + 1} \chi^I}
		\\ \lesssim &  \sum_{l = 1}^n \frac{1}{\brak{t}^{1+s^{-1}}}\sqrt{{\mathcal{E}^{(\overline h)}_{\text{Int,Coord}} + \mathcal{E}_{\overline H}}}
	\enorm{\omega_{m, n-l+1} e^{W} \varphi^{n -l+ 2} \chi^I}
	\enorm{ \omega_{m,n} e^{W} \varphi^{n + 1} \chi^I}
	\\\lesssim & \frac{\eps}{\brak{t}^{1+s^{-1}}} \mathcal{E}_{cloud}
\end{align*}
where $\tilde \chi^I$ is again a fattened version of $\chi^I$ such that 
$\supp \tilde \chi^I \subset \{\chi^{ I}=1\}$, concluding the proof.
\end{proof}
\begin{lemma} For any $m, n$ satisfying $0 \le m + n \le 20$, the following bound holds 
\begin{align} \label{merced:3}
|\text{Err}^{(m, n)}_{\text{cloud}, 3}| \lesssim \eps \frac{1}{\brak{t}^{2+}}  \mathcal{E}_{\text{cloud}} + \eps  \mathcal{D}_{\text{cloud}}.
\end{align}
\end{lemma}
\begin{proof} 
	Recall that
	\begin{align*}
		\text{Err}^{(m, n)}_{\text{cloud}, 3} = &\nu \left|\left\langle \partial_x^m\bold C_{\mathrm{visc}}^{(n)}, \omega_{m,n} e^{2W} \varphi^{2n + 2} (\chi^I)^2 \right\rangle\right|
		\\ = &
        \nu \left|\left\langle \sum_{l = 1}^n \binom{n}{l}  \paren{\frac{\pa_y}{v_y}}^l \paren{v_y^2-1} \paren{\frac{\pa_y}{v_y}}^2 \partial_x^m \Gamma^{n-l} \omega, \omega_{m,n} e^{2W} \varphi^{2n + 2} (\chi^I)^2 \right\rangle\right|. 
	\end{align*}
By the definition of $\Gamma$
and
$t\lesssim \nu^{-1/3-\delta}$, we easily see that
\begin{align*}
	\text{Err}^{(m, n)}_{\text{cloud}, 3} \lesssim &
	\nu \left|\left\langle \sum_{l= 1}^n \binom{n}{l}  \paren{\frac{\pa_y}{v_y}}^l \paren{v_y^2-1} \paren{\frac{\pa_y}{v_y}}\partial_x^m \Gamma^{n-l+1} \omega, \omega_{m,n} e^{2W} \varphi^{2n + 2} (\chi^I)^2 \right\rangle\right|
	\\& 
	+ \nu t\left|\left\langle \sum_{l = 1}^n \binom{n}{l}  \paren{\frac{\pa_y}{v_y}}^l \paren{v_y^2-1} \paren{\frac{\pa_y}{v_y}} \partial_x^{m+1}\Gamma^{n-l} \omega, \omega_{m,n} e^{2W} \varphi^{2n + 2} (\chi^I)^2 \right\rangle\right|
	\\\lesssim &
	\nu \sum_{l = 1}^n \binom{n}{l}  \norm{\paren{\frac{\pa_y}{v_y}}^l \paren{v_y^2-1} \tilde \chi^I}_{L^\infty} 
	\enorm{ \omega_{m,n} e^{W} \varphi^{n + 1} \chi^I}\\
	&\times
	\paren{\enorm{\paren{\frac{\pa_y}{v_y}}\partial_x^m \Gamma^{n-l+1} e^{W}\omega\varphi^{n + 1} \chi^I} + \enorm{\paren{\frac{\pa_y}{v_y}}\partial_x^{m+1} \Gamma^{n-l} \omega e^{W} \varphi^{n + 1} \chi^I}}
	\\\lesssim &
	\sqrt{\nu} \sqrt{{\mathcal{E}^{( h)}_{\text{Int,Coord}} + \mathcal{E}_{ H}}}\sqrt{\mathcal{E}_{\text{cloud}}} \sqrt{ \mathcal{D}_{\text{cloud}}}
	\lesssim \eps \frac{1}{\brak{t}^{1+}} \sqrt{ \mathcal{E}_{\text{cloud}}} \sqrt{ \mathcal{D}_{\text{cloud}}}
	\lesssim \eps \frac{1}{\brak{t}^{2+}}  \mathcal{E}_{\text{cloud}} + \eps  \mathcal{D}_{\text{cloud}}
\end{align*}
where $\tilde \chi^I$ is defined in the previous lemma, concluding the proof.
\end{proof}

\subsection{Proof of Proposition \ref{pro:cloud:intro}}

We are now ready to prove the main proposition of this section. 
\begin{proof}[Proof of Proposition \ref{pro:cloud:intro}] We bring together bounds \eqref{bhbhy:1}, \eqref{merced:1}, \eqref{merced:2}, and \eqref{merced:3} which immediately yields the desired result. 
\end{proof}


\section{Sobolev Boundary Norms} \label{sec:SOB:BDRY}
The main proposition of this section is Proposition \ref{pro:int:sob}.

\subsection{Setup of Equations, Norms, and Lift Functions}

We take the Fourier transform in $x$ of equation \eqref{M1a} to arrive at
\begin{align*}
	\pa_t \omega_k + (y + U^x_0(t, y))(ik)\omega_k  =\nu \de \omega_k-\paren{\nabla^{\perp}\stm_{\neq}\cdot\nabla \omega}_k.
\end{align*}
Introducing the notations 
\begin{align}
f_{k,n} := \p_y^n \omega_k, \qquad \bar{U}(t, y) := y + U^x_0(t, y) 
\end{align}
and applying $\p_y^n$, we further obtain 
\begin{align} \label{M1a:fin}
&\pa_t f_{k,n} + ik \bar{U} f_{k,n} - \nu \Delta_k f_{k,n} = A_{k,n} + B_{k,n} + C_{k,n}  \qquad (k, y) \in \mathbb{Z} \times (-1,1).
\end{align}
We will denote the one-dimensional quantities 
\begin{align}
\alpha_{n, \pm}(t) := \p_y^n \omega(t, \pm 1). 
\end{align}
We define the error terms 
\begin{align} \label{Akn}
A_{k,n} := &- \sum_{n' = 0}^{n-1} \binom{n}{n'} \p_y^{n-n'} \bar{U} (ik) f_{k,n'} \\ \label{Bkn}
B_{k,n} := &-  \sum_{n' = 0}^{n-1} \binom{n}{n'} (\nabla^\perp \p_y^{n-n'} \psi_{\neq 0} \cdot \nabla \p_y^{n'} \omega)_k \\ \label{Ckn}
C_{k,n} := & -( \nabla^\perp  \psi_{\neq 0} \cdot \nabla \p_y^{n} \omega)_k.
\end{align}
We evaluate \eqref{M1a:fin} at $y = \pm 1$. For $n = 0, 1, 2, 3, 4$, we have the following boundary conditions: 
\begin{align}
f_{k,0}|_{y = \pm 1} = 0, \\
\p_y f_{k,1}|_{y = \pm 1} = 0, \\
f_{k,2}|_{y = \pm 1} = 0, \\
\p_y f_{k,3}|_{y = \pm 1} = \alpha_{k, 4, \pm}(t), \\
f_{k,4}|_{y = \pm 1} = \alpha_{k,4,\pm}(t),
\end{align}
where the function $\alpha_{k,4,\pm}(t)$ is defined as follows 
\begin{align}
\alpha_{k,4,\pm}(t) := - \frac{2}{\nu} \bar{U}' (ik) f_{k,1}|_{y = \pm 1} - \frac{2}{\nu} (u^{\neq 0}_y \omega_{xy})_k - \frac{1}{\nu} (v_{yy} \omega_y)_{k}.
\end{align}

We now introduce a lift function to homogenize the boundary condition $\alpha_{k,4,\pm}(t)$. Correspondingly, we define 
\begin{align} \label{small:fh}
h_{k}(t, y) := \nu \sum_{\iota \in \pm} \alpha_{k,4,\iota}(t) \varphi_{\iota}(\frac{y}{\nu}), \qquad \varphi_{\iota}'(\iota 1) = 1.  
\end{align}
We subsequently consider
\begin{align} \label{bigFH}
F_{k,3} := f_{k,3} - h_{k}(t, y),
\end{align}
which satisfies the equation 
\begin{align} \label{M1a:fin:b}
&\pa_t F_{k,3} + ik \bar{U} F_{k,3} - \nu \Delta_k F_{k,3} = A_{k,3} + B_{k,3} + C_{k,3} + H_{k,3} \qquad (k, y) \in \mathbb{Z} \times (-1,1), \\ \label{M1b:fin:b}
&\p_y^{ \iota_n }F_{k,3}|_{y = \pm 1} = 0, 
\end{align}
where the contributed source term $H_{k,3}$ is defined as follows 
\begin{align} \label{def:big:H}
H_{k,3} := - \p_t h_{k,3} - i k \bar{U} h_{k,3} - \nu \Delta_k h_{k,3}.
\end{align}

\subsection{Control of $\mathcal{E}_{sob, n}, \mathcal{D}_{sob,n}$}

First, we have 
\begin{lemma}[$\mathcal{E}_{sob, 0}, \mathcal{D}_{sob,0}$] There exists a universal constant $\delta_0 > 0$ so that the following bound holds
\begin{align} \label{Ebd0}
\frac{\p_t}{2} \mathcal{E}_{sob,0} + \delta_0  \mathcal{CK}^{(W)}_{sob,0} + \mathcal{D}_{sob,0} \le & 0. 
\end{align}
\end{lemma}
\begin{proof} We take the (real) inner product of equation \eqref{M1a:fin} ($n = 0$) with the quantity $\overline{f}_{k,0} e^{2W}$ which produces the identity $\text{LHS} = \text{RHS}$, where we define 
\begin{align*}
\text{LHS} := & \langle \pa_t f_{k,0} + ik \bar{U} f_{k,0} - \nu \Delta_k f_{k,0} , \overline{f}_{k,0} e^{2W}  \rangle\\
\text{RHS} := & \langle A_{k,0} + B_{k,0} + C_{k,0} , \overline{f}_{k,0} e^{2W} \rangle .
\end{align*}
We integrate by parts to get 
\begin{align}
\text{LHS}= &\frac{\p_t}{2} \| f_{k,0} e^{W} \|_{L^2}^2 - \Re\langle f_{k,0}, \nu \pa_y^2 \overline{f}_{k,3} e^{2W} \rangle +  \nu  |k|^2 \|f_{k,0} e^W \|_{L^2}^2  -  \Re \langle f_{k,0}, f_{k,0} \pa_t (e^{2W} )\rangle \\ \n
= & \frac{\p_t}{2} \| f_{k,0} e^{W} \|_{L^2}^2 + \nu \| \nabla_k f_{k,0} e^W  \|_{L^2}^2+ 2\|   f_{k,0}\sqrt{-\pa_t W}    e^{W} \|_{L^2}^2  - \Re \langle  f_{k,0},    \nu \pa_y \overline{f}_{k,0}  \pa_y ( e^{2W} ) \rangle.
\end{align}  
The final term above can be estimated with \eqref{W_prop} as follows
\begin{align}
|\Re \langle  f_{k,0},    \nu \pa_y \overline{f}_{k,0}  \pa_y ( e^{2W} ) \rangle| \leq &\nu\|| \pa_y W |\ f_{k,0} e^W \|_2\|\pa_yf_{k,0}e^W\|_2\\
\leq&\frac{1}{4}  \nu \|\pa_yf_{k,0}e^W\|_2^2+\frac{1}{8}\lf\|   f_{k,0}  \sqrt{-\pa_t W}    e^{W} \rg\|_2^2.\label{I_21}
\end{align}
Here we have chosen $K$ large enough in \eqref{W_prop}. 

Proceeding to the terms on the RHS, we first of all notice that $A_{k,0} = B_{k,0} = 0$. It remains only to treat the quantity $|\langle C_{k,0} , \overline{f}_{k,0} e^{2W} \rangle|$. For this, an essentially identical integration by parts to \eqref{cloud:ibp:1} proves that 
\begin{align*}
|\langle C_{k,0} , \overline{f}_{k,0} e^{2W} \rangle| \le \frac{C}{L} \mathcal{CK}_{sob, 0}^{(W)}.
\end{align*}

\end{proof}

\begin{lemma}[$\mathcal{E}_{sob, j}, \mathcal{D}_{sob,j}$, where $j = 1, 3$] The following bounds are valid: 
\begin{align} \label{twp:1}
\frac{\p_t}{2} \mathcal{E}_{sob,1}  + \mathcal{CK}_{sob,1}^{(W)} + \mathcal{D}_{sob,1} \lesssim &\| \nu^{\frac12} C_{k,0} e^W \|_{L^2}^2, \\ \label{twp:2}
\frac{\p_t}{2} \mathcal{E}_{sob,3} + \mathcal{CK}_{sob,3}^{(W)} + \mathcal{D}_{sob,3} \lesssim &\sum_{\mathcal{K} \in \{ A, B, C\}} \| \nu^{\frac52} \mathcal{K}_{k,2} e^W \|_{L^2}^2.
\end{align}
\end{lemma}
\begin{proof} We prove only \eqref{twp:2}, as \eqref{twp:1} is obtained in a similar manner. We take inner product of \eqref{M1a:fin} ($n= 2$) against the quantity $- \nu^6 \Delta_k \overline{f}_{k,2}e^{2W}$.  This produces the identity (where $\langle \cdot, \cdot \rangle$ corresponds to a real inner product) $\text{LHS} = \text{RHS}$, where  
\begin{align} 
\text{LHS} := & \Re \langle \pa_t f_{k,2} + ik \bar{U} f_{k,2} - \nu \Delta_k f_{k,2}  , - \nu^6 \Delta_k \overline{f}_{k,2} e^{2W} \rangle \\
\text{RHS} := & \Re \langle A_{k,2} + B_{k,2} + C_{k,2}, - \nu^6 \Delta_k \overline{f}_{k,2} e^{2W}\rangle .
\end{align}
Integration by parts produces the following identity 
\begin{align*}
\text{LHS} = & \frac{\p_t}{2} \| \nu^3 \nabla_k f_{k,2} e^{W} \|^2 - \frac12 \Re \langle \nu^6 \nabla_k f_{k,2}, \nabla_k \overline{f}_{k,2} \p_t \{ e^{2W} \} \rangle \\ \n
& + \Re \langle \nu^6 \p_t f_{k,2}, \p_y f_{k,2} \p_y \{ e^{2W} \} \rangle -2 \nu^7 \langle |k|^2 f_{k,2}, \p_y f_{k,2} \p_y \{ e^{2W} \} \rangle\\
& +\nu^6 \| \nu^{\frac12} \nabla_k f_{k,3} e^{W} \|^2 + \nu^6 \| \nu^{\frac12} \nabla_k |k| f_{k,2} e^{W} \|^2 + \nu^6\Re \langle ik \p_y \bar{U} f_{k,2}, \p_y \overline{f}_{k,2} e^{2W} \rangle \\
&+ \nu^6\Re \langle ik \bar{U} f_{k,2}, \p_y \overline{f}_{k,2} \p_y \{ e^{2W} \}  \rangle \\
= &  \frac{\p_t}{2} \| \nu^3 \nabla_k f_{k,2} e^{W} \|^2 + \| \nu^{3 + \frac12} \nabla_k f_{k,3} e^{W} \|^2 + \| \nu^{3 + \frac12} \nabla_k |k| f_{k,2} e^{W} \|^2 \\
&+ \frac12 \Re \langle \nu^4 \p_y f_{k,2}, \p_y \overline{f}_{k,2} - \p_t \{ e^{2W} \} \rangle + \text{Err}_{LHS},
\end{align*}
where we define 
\begin{align*}
\text{Err}_{LHS} := & \Re \langle \nu^6 \p_t f_{k,2}, \p_y f_{k,2} \p_y \{ e^{2W} \} \rangle-2 \nu^7 \langle |k|^2 f_{k,2}, \p_y f_{k,2} \p_y \{ e^{2W} \} \rangle \\
& + \nu^6\Re \langle ik \p_y \bar{U} f_{k,2}, \p_y \overline{f}_{k,2} e^{2W} \rangle + \nu^6\Re \langle ik \bar{U} f_{k,2}, \p_y \overline{f}_{k,2} \p_y \{ e^{2W} \}  \rangle =  \sum_{i = 1}^4 \text{Err}_{LHS}^{(i)}.  
\end{align*}
We now need to estimate each of the terms appearing above. First, to estimate $\text{Err}_{LHS}^{(1)}$, we use the equation \eqref{M1a:fin} (with $n = 2$) to generate 
\begin{align} \n
\text{Err}_{LHS}^{(1)} = &  \Re \langle \nu^6 \p_t f_{k,2}, \p_y f_{k,2} \p_y \{ e^{2W} \} \rangle =2 \Re \langle \nu^6 \p_t f_{k,2}, \p_y f_{k,2} \p_y W e^{2W} \} \rangle \\ \label{cku:1}
\lesssim & \| \nu^{\frac{5}{2}} \p_t f_{k,2} e^W \|_{L^2} \| \nu^{3} \p_y f_{k,2} (\sqrt{\nu} \p_y W) e^W \|_{L^2} \\ \label{cku:2}
\le & \frac{C}{K} \| \nu^{\frac52} \p_t f_{k,2} e^W \|_{L^2} \| \nu^{3} \p_y f_{k,2} \sqrt{- \p_t W} e^W \|_{L^2} \\ \label{cku:3}
\le & \frac{C}{K} \| \nu^{\frac52} \p_t f_{k,2} e^W \|_{L^2} \sqrt{ \mathcal{CK}_{sob, 3}^{(W)}}
\end{align}
where $K >> 1$ is a large parameter. Above, to go from \eqref{cku:1} to \eqref{cku:2}, we have used \eqref{wdot:est:a}. To conclude the bound on this term, we need to provide an estimate on the first term on the right-hand side of \eqref{cku:2}. For this purpose, we use the equation \eqref{M1a:fin} with $n = 2$, as follows: 
\begin{align*}
 \| \nu^{\frac52} \p_t f_{k,2} e^W \|_{L^2} \lesssim & \|  \nu^{\frac52} \bar{U} k f_{k,2}  e^W \|_{L^2} +    \| \nu^{\frac52} \nu \p_y f_{k,3}  e^W \|_{L^2} +  \| \nu^{\frac52} \nu |k|^2 f_{k,2}  e^W \|_{L^2}  \\
 &+  \| \nu^{\frac52} A_{k,2} e^W \|_{L^2} +  \| \nu^{\frac52} B_{k,2} e^W \|_{L^2} +  \| \nu^{\frac52} C_{k,2} e^W \|_{L^2} \\
 \lesssim & \sqrt{\mathcal{D}_{sob,2}} + \sqrt{ \mathcal{D}_{sob,3}} +  \| \nu^{\frac52} A_{k,2} e^W \|_{L^2} +  \| \nu^{\frac52} B_{k,2} e^W \|_{L^2} +  \| \nu^{\frac52} C_{k,2} e^W \|_{L^2}.
 \end{align*}
 This conclude the bound on $\text{Err}_{LHS}^{(1)}$. We now turn to the bounds of $\text{Err}_{LHS}^{(2)}$ and $\text{Err}_{LHS}^{(4)}$, for which we have 
 \begin{align*}
 |\text{Err}_{LHS}^{(2)}| \lesssim &  \| \nu^{\frac72}|k|^2 f_{k,2}  e^W \|_{L^2} \| \nu^{3} \p_y f_{k,2} (\sqrt{\nu} \p_y W) e^W \|_{L^2} \le  \frac{C}{K} \sqrt{ \mathcal{D}_{sob,3} \mathcal{CK}_{sob, 3}^{(W)}}, \\
 |\text{Err}_{LHS}^{(4)}| \lesssim & \| U \|_{L^\infty}  \| \nu^{\frac52}|k| f_{k,2}  e^W \|_{L^2} \| \nu^{3} \p_y f_{k,2} (\sqrt{\nu} \p_y W) e^W \|_{L^2} \le  \frac{C}{K} \sqrt{ \mathcal{D}_{sob,3} \mathcal{CK}_{sob, 3}^{(W)}}.
 \end{align*}
Next, we have 
\begin{align*}
|\text{Err}_{LHS}^{(3)}| \lesssim & \nu^6|\Re \langle ik \p_y \bar{U} f_{k,2}, \p_y \overline{f}_{k,2} e^{2W} \rangle| \lesssim \nu \| \p_y \bar{U} \|_{\infty} \| \nu^{\frac52}|k|  f_{k,2} e^W \|_{L^2} \| \nu^{\frac52} \p_y f_{k,2} e^W \|_{L^2} \\
\lesssim & \nu \mathcal{D}_{sob,2}^{\frac12} \mathcal{D}_{sob,2}^{\frac12},
\end{align*} 
concluding the proof.
\end{proof}

We now perform our estimate on the left-hand side of \eqref{M1a:fin}. 

\begin{lemma}[$\mathcal{E}_{sob, j}, \mathcal{D}_{sob,j}$, $j = 2, 4$] The following bounds are valid: 
\begin{align}\label{Ebd2}
\frac{\p_t}{2} \mathcal{E}_{sob,2}+ \mathcal{CK}_{sob, 2}^{(W)} + \mathcal{D}_{sob,2} \lesssim &\sum_{\mathcal{K} \in \{ A, B, C\}} \| \nu^{\frac32} \mathcal{K}_{k,1} e^W \|_{L^2}^2, \\ \label{enhy:1}
\frac{\p_t}{2} \mathcal{E}_{sob,4} + \mathcal{CK}_{sob, 4}^{(W)} + \mathcal{D}_{sob,4} \lesssim & \sum_{\mathcal{K} \in \{ A, B, C, H\}} \| \nu^{\frac72} \mathcal{K}_{k,3} e^W \|_{L^2}^2.
\end{align}
\end{lemma}
\begin{proof} We take the inner product of \eqref{M1a:fin:b} with $\nu^8 \Delta_k \overline{F}_{k,3} e^{2W}$ and integrate by parts. 

\vspace{2 mm} 

\noindent \textit{Energy Identity:} This produces the following identity 
\begin{align} \n
\text{LHS} = & \frac{\p_t}{2} \| \nu^4 \nabla_k F_{k,3} e^{W} \|^2 - \frac12 \Re \langle \nu^8 \nabla_k F_{k,3}, \nabla_k \overline{F}_{k,3} \p_t \{ e^{2W} \} \rangle \\ \n
& + \Re \langle \nu^8 \p_t F_{k,3}, \p_y F_{k,3} \p_y \{ e^{2W} \} \rangle -2 \nu^9 \langle |k|^2 F_{k,3}, \p_y F_{k,3} \p_y \{ e^{2W} \} \rangle\\ \n
& +\nu^8 \| \nu^{\frac12} \nabla_k \p_y F_{k,3} e^{W} \|^2 + \nu^8 \| \nu^{\frac12} \nabla_k |k| F_{k,3} e^{W} \|^2 + \nu^8\Re \langle ik \p_y \bar{U} F_{k,3}, \p_y \overline{F}_{k,3} e^{2W} \rangle \\ \n
&+ \nu^8\Re \langle ik \bar{U} F_{k,3}, \p_y \overline{F}_{k,3} \p_y \{ e^{2W} \}  \rangle \\ \n
= &  \frac{\p_t}{2} \| \nu^4 \nabla_k F_{k,3} e^{W} \|^2 + \| \nu^{4 + \frac12} \nabla_k \p_y F_{k,3} e^{W} \|^2 + \| \nu^{4 + \frac12} \nabla_k |k| F_{k,3} e^{W} \|^2 \\ \label{deekay:1}
&+ \frac12 \Re \langle \nu^8 \nabla_k F_{k,3}, \nabla_k \overline{F}_{k,3} - \p_t \{ e^{2W} \} \rangle + \text{Err}_{LHS},
\end{align}
where we define 
\begin{align*}
\text{Err}_{LHS} := & \Re \langle \nu^8 \p_t F_{k,3}, \p_y F_{k,3} \p_y \{ e^{2W} \} \rangle-2 \nu^9 \langle |k|^2 F_{k,3}, \p_y F_{k,3} \p_y \{ e^{2W} \} \rangle \\
& + \nu^8\Re \langle ik \p_y \bar{U} F_{k,3}, \p_y \overline{F}_{k,3} e^{2W} \rangle + \nu^8\Re \langle ik \bar{U} F_{k,3}, \p_y \overline{F}_{k,3} \p_y \{ e^{2W} \}  \rangle =  \sum_{i = 1}^4 \text{Err}_{LHS}^{(i)}.  
\end{align*}

\vspace{2 mm} 

\noindent \textit{Estimation of Error Terms:} To estimate $\text{Err}_{LHS}^{(1)}$, we follow similarly to \eqref{cku:3}, which gives 
\begin{align} \label{nd:1}
|\text{Err}_{LHS}^{(1)}| \le & \frac{C}{K} \| \nu^{\frac72} \p_t F_{k,3} e^W \|_{L^2} \sqrt{ \mathcal{CK}_{sob, 4}^{(W)}}.
\end{align}
We subsequently use \eqref{M1a:fin:b} to estimate
\begin{align*}
\| \nu^{\frac72} \p_t F_{k,3} e^W \|_{L^2} \lesssim &\underbrace{ \| \nu^{\frac72} |k| F_{k,3} e^W \|_{L^2}}_{O_1} +  \underbrace{ \| \nu^{\frac72} \nu \Delta_k F_{k,3} e^W \|_{L^2} }_{O_2} + \underbrace{ \sum_{\mathcal{K} \in \{ A, B, C, H\}} \| \nu^{\frac72} \mathcal{K}_{k,3} e^W \|_{L^2}}_{O_3}.
\end{align*}
We first estimate $O_1$ using \eqref{bigFH} and \eqref{small:fh} as follows
\begin{align*}
O_1 \le  & \| \nu^{\frac72} |k| f_{k,3} e^W \|_{L^2} +  \| \nu^{\frac72} |k| h_{k,3} e^W \|_{L^2} \\
\lesssim & \| \nu^{\frac72} |k| f_{k,3} e^W \|_{L^2} +  \| \nu^{\frac92} |k| \alpha_{k,4} \varphi(\frac{\cdot}{\nu}) e^W \|_{L^2} \\
\lesssim & \nu \mathcal{D}_{sob,2}^{\frac12} + \nu \mathcal{D}_{\text{Trace, Large,1}}^{\frac12}
\end{align*}
Next we clearly have $\mathcal{O}_2$ is bounded by the diffusive term in $\text{LHS}$, and will thus be absorbed upon choosing $K$ large enough. Similarly, we have $\mathcal{O}_3$ is appearing on the right-hand side of the estimate. Therefore, we get upon inserting into \eqref{nd:1}, 
\begin{align*}
|\text{Err}_{LHS}^{(1)}| \le & \frac{C}{K} (O_1 + O_2 + O_3)\sqrt{ \mathcal{CK}_{sob, 4}^{(W)}} \\
\le & \frac{C}{K} (\nu \mathcal{D}_{sob,2}^{\frac12} + \nu \mathcal{D}_{\text{Trace, Large,1}}^{\frac12} + O_2 + O_3)\sqrt{ \mathcal{CK}_{sob, 4}^{(W)}} \\
\le & \nu \mathcal{D}_{sob,2} + \nu \mathcal{D}_{\text{Trace, Large, 1}} + \frac{1}{\sqrt{K}} \mathcal{CK}_{sob,4}^{(W)} + \frac{1}{\sqrt{K}}( \| \nu^{4 + \frac12} \Delta F_{k,3} e^{W} \|^2) \\
& +  \sum_{\mathcal{K} \in \{ A, B, C, H\}} \| \nu^{\frac72} \mathcal{K}_{k,3} e^W \|_{L^2}^2.
\end{align*}
The term $\text{Err}_{LHS}^{(2)}$ is identical to $O_2$ above and $\text{Err}_{LHS}^{(4)}$ to $O_1$ above. This just leaves $\text{Err}_{LHS}^{(3)}$, which we estimate as follows:
\begin{align*}
\text{Err}_{LHS}^{(3)} \lesssim \| \nu^4 \nabla_k F_{k,3} e^W \|_{L^2}^2 \lesssim & \nu \mathcal{D}_{sob,3} + \| \nu^5 |k| \alpha_{k,4}  \varphi(\frac{\cdot}{\nu}) e^W \|_{L^2}^2 +  \| \nu^4 \alpha_{k,4}  \varphi'(\frac{\cdot}{\nu}) e^W \|_{L^2}^2 \\
\lesssim & \nu \mathcal{D}_{sob,3} +  \nu^3 \mathcal{D}_{\text{Trace, Large, 1}} + \nu^3 \mathcal{D}_{\text{Trace, Large, 0}}. 
\end{align*}

\vspace{2 mm} 

\noindent \textit{Pullback of Functionals:}  To complete the proof, we need to pullback the bounds on $F_{k,3}$ into those on $f_{k,3}$ using the expression \eqref{bigFH}. Indeed, we begin with the energetic term from \eqref{deekay:1}, for which we have 
\begin{align*}
 \| \nu^4 \nabla_k F_{k,3} e^{W} \|^2 \gtrsim & \mathcal{E}_{sob,4} -  \| \nu^4 \nabla_k h_{k,3} e^{W} \|^2 \\
 \gtrsim &  \mathcal{E}_{sob,4} -  \| \nu^5 |k| \alpha_{k,4} \varphi(\frac{\cdot}{\nu}) e^{W} \|^2 -  \| \nu^4 \alpha_{k,4} \varphi'(\frac{\cdot}{\nu}) e^{W} \|^2 \\
 \gtrsim & \mathcal{E}_{sob, 4} - \nu^2 \mathcal{E}_{\text{Trace},2} - \nu^2 \mathcal{E}_{\text{Trace}, 1}.
\end{align*}
We estimate the diffusion term as follows:
\begin{align*}
\| \nu^{\frac92} \Delta F_{k,3} e^W \|_{L^2}^2 \gtrsim & \| \nu^{\frac92} \Delta f_{k,3} e^W \|_{L^2}^2 - \| \nu^{\frac{11}{2}} |k|^2 \alpha_{k,4} \varphi(\frac{\cdot}{\nu}) e^W \|_{L^2}^2 - \| \nu^{\frac{7}{2}}  \alpha_{k,4} \varphi''(\frac{\cdot}{\nu}) e^W \|_{L^2}^2 \\
\gtrsim & \mathcal{D}_{sob,4} - \nu^2 \mathcal{D}_{\text{Trace, Large, 2}} - \nu^2 \mathcal{D}_{\text{Trace, Large, 0}}. 
\end{align*}
Finally, we need to estimate the $CK$ term as follows:
\begin{align*}
\Re \langle \nu^6 \p_y F_{k,3}, \p_y \overline{F}_{k,3} - \p_t \{ e^{2W} \} \rangle \gtrsim &  \mathcal{CK}_{sob,4} - \nu^2 \mathcal{CK}_{\text{Trace, 2}} - \nu^2 \mathcal{CK}_{\text{Trace, 0}}
\end{align*}
finishing the proof.
\end{proof}

\subsection{Bounds on $\mathcal{J}_{sob}$}

It turns out to be convenient to decompose the stream function, $\psi$, into $\psi = \psi_{sob}^{(E)} + \psi_{sob}^{(I)}$, where 
\begin{align}
&\Delta \psi^{(E)}_{sob} = \omega \chi_2, \\
&\psi^{(E)}_{sob}|_{y = \pm 1} = 0, \\
&\omega|_{y = \pm 1} = 0,
\end{align}
and
\begin{align}
&\Delta \psi^{(I)}_{sob} = \omega (1 - \chi_2), \\
&\psi^{(I)}_{sob}|_{y = \pm 1} = 0, \\
&\omega|_{y = \pm 1} = 0.
\end{align}
We will here define the functionals $\mathcal{J}_{sob}$:
\begin{align} \label{mvp:1}
\mathcal{J}_{sob} := & \mathcal{J}_{sob}^{(i,i)} + \mathcal{J}_{sob}^{(o,o)} + \mathcal{J}^{(i,o)}_{sob} \\ \label{mvp:2}
\mathcal{J}_{sob}^{(i,i)} := & \sum_{m + n \le 10} \| \p_x^m (\nu^{\frac12} \p_y)^{n} \psi_{sob}^{(I)} \|_{H^2}^2, \\ \label{mvp:3}
\mathcal{J}_{sob}^{(o,o)}:=& \frac{1}{\nu^{100}}( \| \psi_{sob}^{(E)} \|_{H^2}^2 + \sum_{i = 1}^4 \| \nabla \p_y^{i-1} \psi_{sob}^{(E)}  \|_{H^2}^2 ) \\
\mathcal{J}^{(i,o)}_{sob} :=& \| \psi_{sob}^{(I)} \chi_{9} \|_{H^{50}}^2
\end{align}

\begin{lemma} The following estimate holds:
\begin{align}
\mathcal{J}_{sob} \lesssim \mathcal{E}_{sob} + \mathcal{E}_{\mathrm{Int}}. 
\end{align}
\end{lemma}
\begin{proof}The argument for \eqref{mvp:2} -- \eqref{mvp:3} is standard elliptic regularity, and we omit repeating many of the details. To prove \eqref{mvp:1}, we apply similar trick as in the proof of  \eqref{ext_dif}. Thanks to the Fa\`a di Bruno's formula, the $y$ derivative of the function $\omega(t,x,y)=\Omega(t,x,v(y))$ can be expressed as follows ($j\leq 10$)
\begin{align*}\pa_{y}^{j} \omega(y)=\pa_y^j(\Omega(v(y)))
=&\sum_{(m_1,m_2,\cdots,m_i)\in S_{j}}\frac{{j}!}{m_1!m_2!\cdots m_{j}!}(\pa_{v}^{(m_1+m_2+\cdots+m_{j})}\Omega)\bigg|_{v=v(t,y)}\prod_{\ell=1}^{j}\lf(\frac{\pa_{y}^{\ell }v(t,y)}{\ell!}\rg)^{m_\ell}.
\end{align*}
Since we consider finite regularity $j\leq 10$, and assume that $\|v_y^{-1}\|_{L^\infty}\leq C$ and the coordinate $h$ \eqref{E_IntCh} is well bounded \eqref{boot:IntH} on the support of $1-\chi_2$, we have that the combinatorial constants are bounded and the $\|(1-\chi_2)v_y\|_{H_y^{10}}$ is bounded $\lesssim \ep$. As a result, we have that 
\begin{align*}
\sum_{j=0}^{10}&\|(1-\chi_2)(\nu^{1/3+}\pa_y)^j\omega(t,x,y)\|_{L_{x,y}^2}^2\\
\lesssim& \sum_{j=0}^{10}\|\chi^I(\nu^{1/3+}\pa_v)^j\Omega(t,x,v)\|_{L_{x,v}^2}^2\lesssim\sum_{j=0}^{10}\|\chi^I(\nu^{1/3+}(\pa_v+t\pa_x -t\pa_x))^j\Omega(t,x,v)\|_{L_{x,v}^2}^2\\
\lesssim& \sum_{m+n=0}^{10}\nu^{\frac{1+}{3}(m+n)}t^m\|\chi^I \pa_x^m(\pa_v+t\pa_x)^n\Omega(t,x,v)\|_{L_{x,v}^2}^2\lesssim \mathcal E_{\rm Int}^{\rm low}.
\end{align*}
\end{proof}

The following corollary will be how we use these elliptic bounds in the future.
\begin{corollary} The stream function, $\psi$, satisfies the following bounds 
\begin{align}
\| \psi \|_{L^\infty}^2 + \sum_{i = 1}^4 \nu^{i} \| \nabla \p_y^{i-1} \psi  \|_{L^\infty}^2 \lesssim \mathcal{J}_{sob}.
\end{align}
\end{corollary}
\begin{proof} We simply decompose $\psi = \psi^{(I)}_{sob} + \psi^{(E)}_{sob}$, apply estimates \eqref{mvp:1} and \eqref{mvp:2}, and perform the usual $H^2 \hookrightarrow L^\infty$ Sobolev embedding. 
\end{proof}

\subsection{Bounds on $A_n, B_n, C_n, H_n$}

We need to estimate the quantities $\sum_{i = 1}^3 \| \nu^{\frac12 + i} A_{k,i} e^W \|_{L^2}^2$. 
\begin{lemma} The following bound is valid:
\begin{align}
\sum_{i = 1}^3 \| \nu^{\frac12 + i} A_{k,i} e^W \|_{L^2}^2 \lesssim & \nu(1 + \mathcal{J}_{sob}) \mathcal{D}_{sob}.
\end{align}
\end{lemma}
\begin{proof} We start with the case of $i = 1$. Here, we have $A_{k,1} = \p_y \bar{U} ik f_{k,0}$, after which we get 
\begin{align*}
\| \nu^{\frac32} \p_y \bar{U} ik f_{k,0} e^W\|_{L^2} \lesssim \nu \| \p_y \bar{U} \|_{L^\infty} \| \sqrt{\nu} |k| f_{k,0} e^W\|_{L^2} \lesssim \nu \| \p_y \bar{U} \|_{L^\infty} \mathcal{D}_{sob,0}^{\frac12}.
\end{align*}
Next we move to $A_{k,2}$. We estimate 
\begin{align*}
\| \nu^{\frac52} \p_y \bar{U} ik f_{k,1} e^W\|_{L^2} \lesssim &\nu \| \p_y \bar{U} \|_{L^\infty} \| \nu^{\frac32} |k| f_{k,1} e^W\|_{L^2} \lesssim \nu \| \p_y \bar{U} \|_{L^\infty} \mathcal{D}_{sob,1}^{\frac12}  \\
\| \nu^{\frac52} \p_y^2 \bar{U} ik f_{k,0} e^W\|_{L^2} \lesssim & \nu^2 \| \p_y^2 \bar{U} \|_{L^\infty} \| \sqrt{\nu} |k| f_{k,0} e^W\|_{L^2} \lesssim \nu^2 \| \p_y^2 \bar{U} \|_{L^\infty} \mathcal{D}_{sob,0}^{\frac12}.
\end{align*}
Finally, we move to $A_{k,3}$ and get
\begin{align*}
\| \nu^{\frac72} \p_y \bar{U} ik f_{k,2} e^W\|_{L^2} \lesssim &  \nu \| \p_y \bar{U} \|_{L^\infty} \| \nu^{\frac52} |k| f_{k,2} e^W\|_{L^2} \lesssim  \nu \| \p_y \bar{U} \|_{L^\infty}\mathcal{D}_{sob,2}^{\frac12} \\
\| \nu^{\frac72} \p_y^2 \bar{U} ik f_{k,1} e^W\|_{L^2} \lesssim &  \nu^2 \| \p_y^2 \bar{U} \|_{L^\infty} \| \nu^{\frac32} |k| f_{k,1} e^W\|_{L^2} \lesssim \nu^2 \| \p_y^2 \bar{U} \|_{L^\infty} \mathcal{D}_{sob,1}^{\frac12}\\
\| \nu^{\frac72} \p_y^3 \bar{U} ik f_{k,0} e^W\|_{L^2} \lesssim &  \nu^3 \| \p_y^3 \bar{U} \|_{L^\infty} \| \sqrt{\nu} |k| f_{k,0} e^W\|_{L^2} \lesssim \nu^3 \| \p_y^3 \bar{U} \|_{L^\infty} \mathcal{D}_{sob,0}^{\frac12}.
\end{align*}
\end{proof}


The following lemma gives a bound of the quantities $\sum_{i = 1}^3 \| \nu^{\frac12 + i} B_{k,i} e^W \|_{L^2}^2$. 
\begin{lemma}The following estimates hold:
\begin{align}
\sum_{i = 1}^3 \| \nu^{\frac12 + i} B_{i} e^W \|_{L^2} \lesssim & \mathcal{J}_{sob}^{\frac12} \mathcal{D}_{sob}^{\frac12}.
\end{align}
\end{lemma}
\begin{proof} 
Noting
\begin{align*}
	B_1 = \p_y^2 \psi \omega_x - \p_{xy} \psi \omega_y,
\end{align*}
we have 
\begin{align*}
\| \nu^{\frac32} B_1 e^W \|_{L^2} \lesssim \nu \| \psi \|_{W^{2,\infty}} \| \sqrt{\nu} \nabla \omega e^W \|_{L^2} \lesssim \nu \mathcal{J}_{sob}^{\frac12} \mathcal{D}_{sob,0}^{\frac12}
\end{align*}
proving the inequality in the case $i = 1$. For the case $n =2$, we have 
\begin{align} \n
	B_{2} = & - \p_y^3 \psi \p_x \omega + \p_y^2 \p_x \psi \p_y \omega - 2 \p_y^2 \psi \p_{xy} \omega + 2 \p_{xy} \psi \p_y^2 \omega,
\end{align}
for which we obtain the successive bounds: 
\begin{align*}
	\|  \nu^{\frac52} B_{2,1} e^W \|_{L^2} \lesssim & \sqrt{\nu} \| \nu^{\frac32} \p_y^3 \psi \|_{L^\infty} \| \sqrt{\nu} \p_x \omega e^W \|_{L^2} \lesssim \mathcal{J}_{sob}^{\frac12} \mathcal{D}_{sob,0}^{\frac12} \\
	\|  \nu^{\frac52} B_{2,2} e^W \|_{L^2} \lesssim & \sqrt{\nu} \| \nu^{\frac32} \p_y^2 \p_x \psi \|_{L^\infty} \| \sqrt{\nu} \p_y \omega e^W \|_{L^2} \lesssim \mathcal{J}_{sob}^{\frac12} \mathcal{D}_{sob,0}^{\frac12} \\
	\|  \nu^{\frac52} B_{2,3} e^W \|_{L^2} \lesssim & \| \nu \p_y^2 \psi \|_{L^\infty} \| \nu^{\frac32} \p_{xy} \omega e^W \|_{L^2} \lesssim \mathcal{J}_{sob}^{\frac12} \mathcal{D}_{sob,1}^{\frac12} \\
	\|  \nu^{\frac52} B_{2,4} e^W \|_{L^2} \lesssim & \| \nu \p_{xy} \psi \|_{L^\infty} \| \nu^{\frac32} \p_{yy} \omega e^W \|_{L^2} \lesssim \mathcal{J}_{sob}^{\frac12} \mathcal{D}_{sob,1}^{\frac12},
\end{align*}
proving the second inequality. In view of 
\begin{align}
	B_3 = & - \p_y^4 \psi \omega_x - 3 \psi_{yyy} \omega_{xy} - 2 \psi_{yy} \omega_{xyy} + \psi_{yyyx} \omega_y + 3 \psi_{xyy} \omega_{yy} + 2 \psi_{xy} \omega_{yyy},
\end{align}
we estimate these terms successively as follows: 
\begin{align*}
	\| \nu^{\frac72} B_{3,1} e^W \|_{L^2} \lesssim & \nu \| \nu^2 \p_y^4 \psi \|_{L^\infty} \| \sqrt{\nu} \omega_x  e^W \|_{L^2} \lesssim \mathcal{J}_{sob}^{\frac12} \mathcal{D}_{sob,0}^{\frac12}, \\
	\| \nu^{\frac72} B_{3,2} e^W \|_{L^2} \lesssim & \nu^{\frac12} \| \nu^{\frac32} \p_y^3 \psi \|_{L^\infty} \| \nu^{\frac32} \omega_{xy}  e^W \|_{L^2} \lesssim \mathcal{J}_{sob}^{\frac12} \mathcal{D}_{sob,1}^{\frac12}, \\
	\| \nu^{\frac72} B_{3,3} e^W \|_{L^2} \lesssim &  \| \nu \p_y^2 \psi \|_{L^\infty} \| \nu^{\frac52} \omega_{xyy}  e^W \|_{L^2} \lesssim \mathcal{J}_{sob}^{\frac12} \mathcal{D}_{sob,2}^{\frac12}, \\
	\| \nu^{\frac72} B_{3,4} e^W \|_{L^2} \lesssim & \nu \| \nu^2  \psi_{yyyx} \|_{L^\infty} \| \nu^{\frac12} \omega_{y}  e^W \|_{L^2} \lesssim \mathcal{J}_{sob}^{\frac12} \mathcal{D}_{sob,0}^{\frac12}, \\
	\| \nu^{\frac72} B_{3,5} e^W \|_{L^2} \lesssim & \nu^{\frac12} \|  \nu^{\frac32} \psi_{yyx} \|_{L^\infty} \| \nu^{\frac32} \omega_{yy}  e^W \|_{L^2} \lesssim \mathcal{J}_{sob}^{\frac12} \mathcal{D}_{sob,1}^{\frac12}, \\
	\| \nu^{\frac72} B_{3,6} e^W \|_{L^2} \lesssim & \| \nu \psi_{yx} \|_{L^\infty} \| \nu^{\frac52} \omega_{yyy}  e^W \|_{L^2} \lesssim \mathcal{J}_{sob}^{\frac12} \mathcal{D}_{sob,2}^{\frac12},
\end{align*}
which combined together give the third inequality, concluding the proof.
\end{proof}


\begin{lemma} The following bound is valid:
\begin{align}
\sum_{i = 0}^3 \| \nu^{\frac12 + i} C_{i} e^W \|_{L^2} \lesssim \mathcal{J}_{sob}^{\frac12} \mathcal{D}_{sob}^{\frac12}.
\end{align}
\end{lemma}
\begin{proof} A straightforward estimate gives
\begin{align*}
\sum_{i = 0}^3 \| \nu^{\frac12 + i} C_{i} e^W \|_{L^2} \lesssim \|  \nabla \psi \|_{L^\infty}( \sum_{i = 0}^3 \| \nu^{\frac12 + i} \nabla \p_y^i \omega \|_{L^2}  ) \lesssim \mathcal{J}_{sob}^{\frac12} \sum_{i = 0}^3 \mathcal{D}_{sob,i}^{\frac12}
\end{align*}
concluding the proof. 

\end{proof}

Finally, we provide our bounds on $H_3$:
\begin{lemma} It holds
\begin{align}
\| \nu^{\frac72} H_{k,3} e^W \|_{L^2_{xy}}^2 \lesssim & \nu^2 \mathcal{D}_{\text{Trace}}(t).
\end{align}
\end{lemma}
\begin{proof} Based on the definition \eqref{def:big:H}, we have
\begin{align*}
\|  \nu^{\frac72} H_{k,3} e^W \|_{L^2} \lesssim & \| \nu^{\frac72} \p_t h_{k,3} e^W \|_{L^2} +  \| \nu^{\frac72}  k \bar{U} h_{k,3} e^W \|_{L^2} + \| \nu^{\frac92} \Delta_k h_{k,3} e^W \|_{L^2} \\ \n
\lesssim & \| \nu^{\frac92} \p_t \alpha_4 \varphi(\frac{y}{\nu}) e^W \|_{L^2} + \| \nu^{\frac92} \p_x \alpha_4 \varphi(\frac{y}{\nu}) e^W \|_{L^2} + \| \nu^{\frac{11}{2}} \p_{xx} \alpha_4 \varphi(\frac{y}{\nu}) e^W \|_{L^2} \\
&+ \| \nu^{\frac{7}{2}}  \alpha_4 \varphi''(\frac{y}{\nu}) e^W \|_{L^2} \\
\lesssim & \| \nu^{\frac92} \p_t \alpha_4 \varphi(\frac{y}{\nu}) e^W|_{y = \pm 1} \|_{L^2} + \| \nu^{\frac92} \p_x \alpha_4 \varphi(\frac{y}{\nu}) e^W|_{y = \pm 1} \|_{L^2} + \| \nu^{\frac{11}{2}} \p_{xx} \alpha_4 \varphi(\frac{y}{\nu}) e^W|_{y = \pm 1} \|_{L^2} \\
&+ \| \nu^{\frac{7}{2}}  \alpha_4 \varphi''(\frac{y}{\nu}) e^W|_{y = \pm 1} \|_{L^2} \\
\lesssim & \nu \mathcal{D}_{\text{Trace, Large},2}^{\frac12} + \nu \mathcal{D}_{\text{Trace, Large},1}^{\frac12} +  \nu \mathcal{D}_{\text{Trace, Large},2}^{\frac12} +  \nu \mathcal{D}_{\text{Trace, Large},0}^{\frac12} \\
\lesssim & \nu  \mathcal{D}_{\text{Trace}}(t)^{\frac12}.
\end{align*}
Above, we have used that the weight function $W$ is increasing towards the boundaries $y = \pm 1$. 
\end{proof}

\subsection{Bounds on Trace Functionals, $\alpha_n(t)$}

Due to its appearance on the right-hand side of \eqref{enhy:1}, we need to prove a bound on $\| \nu^{\frac72} H_3 e^W \|_{L^2_t L^2_{xy}}$. We proceed in three steps: first we control $\alpha_1(t, x)$, then we control $\alpha_4(t, x)$, and then we use these bounds to control $H_3$. 

To provide bounds on $\alpha_1(t, x)$, we recall the equation 
\begin{align} \label{eq:alpha:1}
\p_t \alpha_1 - \nu \p_x^2 \alpha_1 = &\nu \alpha_3 - \psi_y(t, x, 0) \p_x \alpha_1 + \psi_{xy}(t, x, 0) \alpha_1, \\ \label{eq:alpha:2}
\p_t \p_x \alpha_1 - \nu \p_x^2 \p_x \alpha_1 =& \nu \p_x \alpha_3 - \psi_y(t, x,0) \p_x^2 \alpha_1 + \psi_{xxy}(t,x,0) \alpha_1 . 
\end{align} 
We will also need the equation: 
\begin{align} \label{sr:a}
\alpha_{4}(t) = & \frac{2}{\nu} \bar{U}' \p_x \alpha_1, \\  \label{sr:b}
\p_t \alpha_{4}(t) = & \frac{2}{\nu} \bar{U}' \p_x \p_t \alpha_1.
\end{align}
It will turn out to be convenient to introduce the following notation: 
\begin{align}
\mathcal{S}_{\le n}(t) := \sum_{n' = 0}^n (\mathcal{D}_{sob,n'}(t) + \mathcal{CK}_{sob,n'}(t)). 
\end{align}

We then have
\begin{lemma} The trace quantity $\alpha_1(t, x)$ satisfies 
\begin{align}
\frac{\p_t}{2} \mathcal{E}_{\text{Trace},0}(t) +   \mathcal{CK}_{\text{Trace},0}(t) + \mathcal{D}_{\text{Trace},0}(t) \lesssim (1+  \mathcal{J}_{sob}^{\frac12}) \mathcal{S}_{\le 3}.
\end{align}
\end{lemma}
\begin{proof} We apply the multiplier $\nu^3 \alpha_1 e^{2W}$ to \eqref{eq:alpha:1}, which produces the identity
\begin{align} \n
\frac{\p_t}{2} \mathcal{E}_{\text{Trace},0}(t) +   \mathcal{CK}_{\text{Trace},0}(t) + \mathcal{D}_{\text{Trace},0}(t) = & \langle \nu^4 \alpha_3, \alpha_1 e^{2W} \rangle + \frac32 \nu^3 \langle \widetilde{\psi}_{xy} \alpha_1, \alpha_1 e^{2W} \rangle \\
= & \text{Err}^{(T)}_{1} + \text{Err}^{(T)}_2.
\end{align}
We now use trace inequalities to estimate the terms above. First, we have.
\begin{align*}
|\text{Err}^{(T)}_{1}| \lesssim & \| \nu^{3 + \frac12} \p_y^4 \omega e^W \|_{L^2}^{\frac12} \| \nu^{2 + \frac12} \p_y^3 \omega e^W \|_{L^2}^{\frac12} \| \nu^{1 + \frac12} \p_y^2 \omega e^W \|_{L^2}^{\frac12} \|  \nu^{\frac12}\p_y \omega e^W \|_{L^2}^{\frac12} \\
& + \| \nu^{3} \p_y^3 \omega (\sqrt{\nu} \p_y W) e^W \|_{L^2}^{\frac12} \| \nu^{2 + \frac12} \p_y^3 \omega e^W \|_{L^2}^{\frac12} \| \nu^{1 + \frac12} \p_y^2 \omega e^W \|_{L^2}^{\frac12} \| \nu^{\frac12} \p_y \omega e^W \|_{L^2}^{\frac12} \\
& +  \| \nu^{3 + \frac12} \p_y^4 \omega e^W \|_{L^2}^{\frac12} \| \nu^{2 + \frac12} \p_y^3 \omega e^W \|_{L^2}^{\frac12} \| \nu \p_y \omega (\sqrt{\nu}\p_y W) e^W \|_{L^2}^{\frac12} \| \sqrt{\nu} \p_y \omega e^W \|_{L^2}^{\frac12} \\
& + \|\nu^3 \p_y^3 \omega (\sqrt{\nu} \p_y W) e^W \|_{L^2}^{\frac12} \| \nu^{2 + \frac12} \p_y^3 \omega e^W \|_{L^2}^{\frac12} \|\nu \p_y \omega (\sqrt{\nu}\p_y W) e^W \|_{L^2}^{\frac12} \| \sqrt{\nu} \p_y \omega e^W \|_{L^2}^{\frac12} \\
\lesssim & \mathcal{D}_{sob,3}^{\frac14}\mathcal{D}_{sob,2}^{\frac14} \mathcal{D}_{sob,1}^{\frac14} \mathcal{D}_{sob,0}^{\frac14} +  \mathcal{CK}_{sob,3}^{\frac14}\mathcal{D}_{sob,2}^{\frac14} \mathcal{D}_{sob,1}^{\frac14} \mathcal{D}_{sob,0}^{\frac14} +  \mathcal{D}_{sob,3}^{\frac14}\mathcal{D}_{sob,2}^{\frac14} \mathcal{CK}_{sob,1}^{\frac14} \mathcal{D}_{sob,0}^{\frac14} \\
& +  \mathcal{CK}_{sob,3}^{\frac14}\mathcal{D}_{sob,2}^{\frac14} \mathcal{CK}_{sob,1}^{\frac14} \mathcal{D}_{sob,0}^{\frac14} \\
\lesssim & \mathcal{S}_{\le 3}(t).
\end{align*}
Second, we have 
\begin{align*}
|\text{Err}^{(T)}_2| \lesssim & \nu^3 |\langle \psi_{xyy} \omega_y, \omega_y e^{2W} \rangle| +  |\langle \psi_{xy} \omega_{yy}, \omega_y e^{2W} \rangle| +  |\langle \psi_{xy} \omega_y, \omega_{yy} e^{2W} \rangle| \\
&+ \nu^3 |\langle \psi_{xy} \omega_y, \omega_y 2\p_y\{W\} e^{2W} \rangle|  \\
\lesssim & \nu^3 \| \psi_{xyy} \|_{L^2_x L^\infty_y} \| \omega_y e^W \|_{L^\infty_x L^2_y} \| \omega_y e^W \|_{L^2} + \nu^3 \| \psi_{xy} \|_{L^\infty_x L^2_y} \| \omega_{yy} e^W \|_{L^2_x L^\infty_y} \| \omega_y e^W \|_{L^2} \\
& + \nu^3 \| \psi_{xy} \|_{L^\infty_x L^2_y} \| \omega_y W_y e^W  \|_{L^2} \| \omega_y e^W \|_{L^2_x L^\infty_y} \\
\lesssim &\| \nu \psi_{xyy} \|_{L^2}^{\frac12}  \| \nu^2 \psi_{xyyy} \|_{L^2}^{\frac12} \| \nu^{\frac12} \omega_y e^W \|_{L^2}^{\frac12} \| \nu^{1 + \frac12} \omega_{xy} e^W \|_{L^2}^{\frac12} \| \nu^{\frac12} \omega_y e^W \|_{L^2} \\
& +  \| \psi_{xy} \|_{L^2}^{\frac12} \| \nu \psi_{xxy} \|_{L^2}^{\frac12} \| \nu^{1 + \frac12} \omega_{yy} e^W \|_{L^2}^{\frac12} \| \nu^{2 + \frac12} \omega_{yyy} e^W \|_{L^2}^{\frac12} \| \nu^{\frac12} \omega_y e^W \|_{L^2} \\
& +   \| \psi_{xy} \|_{L^2}^{\frac12} \| \nu \psi_{xxy} \|_{L^2}^{\frac12} \| \nu^{1 + \frac12} \omega_{yy} e^W \|_{L^2}^{\frac12} \| \nu^2 \omega_{yy} (\nu^{\frac12} W_y) e^W \|_{L^2}^{\frac12} \| \nu^{\frac12} \omega_y e^W \|_{L^2} \\
& +   \| \psi_{xy} \|_{L^2}^{\frac12} \| \nu \psi_{xxy} \|_{L^2}^{\frac12} \| \nu \omega_y (\nu^{\frac12} W_y) e^W  \|_{L^2} \| \nu^{\frac12} \omega_y e^W \|_{L^2}^{\frac12}\| \nu^{1 + \frac12} \omega_{yy} e^W \|_{L^2}^{\frac12} \\
& +  \| \psi_{xy} \|_{L^2}^{\frac12} \| \nu \psi_{xxy} \|_{L^2}^{\frac12} \| \nu \omega_y (\sqrt{\nu} W_y) e^W  \|_{L^2}  \| \nu^{\frac12} \omega_y e^W \|_{L^2}^{\frac12}\| \nu \omega_{y} (\nu^{\frac12} W_y) e^W \|_{L^2}^{\frac12} \\
\lesssim & \mathcal{J}_{sob}^{\frac12} \mathcal{S}_{\le 3}(t).
\end{align*}
\end{proof}

\begin{lemma}The trace quantity $\alpha_1(t, x)$ satisfies 
\begin{align} \label{chris:day:1}
\frac{\p_t}{2} \mathcal{E}_{\text{Trace},1}(t) +   \mathcal{CK}_{\text{Trace},1}(t) + \mathcal{D}_{\text{Trace},1}(t) \lesssim & (1+(\mathcal{E}_{ell}^{(I, out)})^{\frac12} +  \mathcal{J}_{sob}^{\frac12}) \mathcal{S}_{\le 3}, \\ \label{chris:day:2}
\mathcal{D}_{\text{Trace, Max},1}(t) \lesssim & (1+ (\mathcal{E}_{ell}^{(I, out)})^{\frac12} +  \mathcal{J}_{sob}^{\frac12}) \mathcal{S}_{\le 3}.
\end{align}
\end{lemma}
\begin{proof}[Proof of \eqref{chris:day:1}] We first prove the energy inequality, \eqref{chris:day:1}. To do so, we multiply \eqref{eq:alpha:1} by $\nu^4 \p_t \alpha_1(t, x) e^{2W}$, integrate over $x \in \mathbb{T}$, and subsequently use Cauchy-Schwartz and Young's inequality for products which produces the bound: 
\begin{align} \n
\p_t  \mathcal{E}_{\text{Trace},1}(t) + \mathcal{D}_{\text{Trace},1}(t) +  \mathcal{CK}_{\text{Trace},1}(t) \le& | \langle \nu \alpha_3 - \widetilde{\psi}_y \p_x \alpha_1 + \psi_{xy}(t, x, 0) \alpha_1, \nu^4 \p_t \alpha_1 e^{2W} \rangle| \\ \n
\le & \frac12 \mathcal{D}_{\text{Trace},1}(t) +C \| \nu^3 \alpha_3 e^W \|_{L^2_x}^2 +C \| \nu^2 \widetilde{\psi}_y \p_x \alpha_1 e^W \|_{L^2_x}^2 \\ \n
&+C  \| \nu^2 \widetilde{\psi}_{xy}  \alpha_1 e^W \|_{L^2_x}^2 \\
=:&\frac12 \mathcal{D}_{\text{Trace},1}(t) + C \sum_{i = 1}^3 \text{Err}^{(T)}_i(t). 
\end{align}
We now successively estimate the error terms appearing above. First, we have 
\begin{align*}
|\text{Err}^{(T)}_1(t)| \lesssim & \nu^6 \| \p_y^3 \omega e^W \|_{L^2} \| \p_y^4 \omega e^W \|_{L^2} + \nu^6 \| \p_y^3 \omega e^W \|_{L^2} \| \p_y^3 \omega e^W (W_y) \|_{L^2} \\
\lesssim & \| \nu^{2 + \frac12} \p_y^3 \omega e^W \|_{L^2} \| \nu^{3 + \frac12} \p_y^4 \omega e^W \|_{L^2} + \| \nu^{2 + \frac12} \p_y^3 \omega e^W \|_{L^2} \| \nu^3 \p_y^3 \omega e^W (\sqrt{\nu} W_y) \|_{L^2} \\
\lesssim & \mathcal{D}_{sob,2}^{\frac12} \mathcal{D}_{sob,3}^{\frac12} + \mathcal{D}_{sob,2}^{\frac12} \mathcal{CK}_{sob,3}^{\frac12} \\
\lesssim & \mathcal{S}_{\le 3}(t). 
\end{align*}
Next, we have by noticing that $\chi_{10} = 1$ at $y = \pm 1$,  
\begin{align*}
|\text{Err}^{(T)}_2(t)| = &  \| \nu^2 \widetilde{\psi}_y \p_x \alpha_1 e^W \chi_{10} \|_{L^2_x}^2 \\
\lesssim & \nu^4 \| \psi_y \omega_{xy} e^W \chi_{10} \|_{L^2}( \| \psi_{yy} \omega_{xy} e^W \chi_{10} \|_{L^2} +\| \psi_{y} \omega_{xyy} e^W \chi_{10} \|_{L^2} + \| \psi_{y} \omega_{xy} e^W W_y \chi_{10} \|_{L^2} \\
&+ \| \psi_{y} \omega_{xy} e^W \chi_{10}' \|_{L^2}  ) \\
\lesssim & \nu^4 (\| \psi_y \chi_{10} \|_{L^\infty} + \| \psi_{yy} \chi_{10} \|_{L^\infty}) \| \omega_{xy} e^W  \|_{L^2}( \| \omega_{xy} e^W \|_{L^2} +\| \omega_{xyy} e^W  \|_{L^2} + \| \omega_{xy} e^W W_y  \|_{L^2}   ) \\
\lesssim & [(\mathcal{E}_{ell}^{(I, out)})^{\frac12} + (\mathcal{J}_{sob})^{\frac12}] \| \nu^{1 + \frac12} \omega_{xy} e^W \|_{L^2}( \| \nu^{1 + \frac12} \omega_{xy} e^W \|_{L^2} +\| \nu^{2 + \frac12} \omega_{xyy} e^W  \|_{L^2} \\
&+ \| \nu^2 \omega_{xy} e^W (\nu^{\frac12} W_y)  \|_{L^2}  ) \\
\lesssim & [(\mathcal{E}_{ell}^{(I, out)})^{\frac12} + (\mathcal{J}_{sob})^{\frac12}] \mathcal{D}_{sob,1}^{\frac12}( \mathcal{D}_{sob,1}^{\frac12}+ \mathcal{D}_{sob,2}^{\frac12} + \mathcal{CK}_{sob,2}^{\frac12} ) \\
\lesssim & [(\mathcal{E}_{ell}^{(I, out)})^{\frac12} + (\mathcal{J}_{sob})^{\frac12}]\mathcal{S}_{\le 3}(t).
\end{align*}
The last term is treated in essentially an identical manner, and we omit it. 
\end{proof}
\begin{proof}[Proof of \eqref{chris:day:2}] We use equation \eqref{eq:alpha:1} by $\nu^2$ to get 
\begin{align*}
\mathcal{D}_{\text{Trace, Max},1}(t) \lesssim \mathcal{D}_{\text{Trace},1}(t) +  \| \nu^3 \alpha_3 e^W \|_{L^2_x}^2 + \| \nu^2 \widetilde{\psi}_y \p_x \alpha_1 e^W \|_{L^2_x}^2 +  \| \nu^2 \widetilde{\psi}_{xy}  \alpha_1 e^W \|_{L^2_x}^2,
\end{align*}
and since all of the terms on the right-hand side above have been already estimated, the proof is concluded. 
\end{proof}

\begin{lemma}The trace quantity $\alpha_1(t, x)$ satisfies 
\begin{align} \label{beed:beed:1}
\frac{\p_t}{2} \mathcal{E}_{\text{Trace},2}(t) +   \mathcal{CK}_{\text{Trace},2}(t) + \mathcal{D}_{\text{Trace},2}(t) \lesssim & \mathcal{S}_{\le 3}^{\frac12} \mathcal{S}_{\le 4}^{\frac12} + (\frac{\mathcal{E}_{ell}^{(I, out)}}{\langle t \rangle^{100}}  +\nu^{100} \mathcal{J}_{sob})(\mathcal{D}_{\text{Trace, Max},1} + \nu^{3}\mathcal{E}_{\text{Trace},0}), \\ \label{beed:beed:2}
\mathcal{D}_{\text{Trace, Max},2}(t) \lesssim &  \mathcal{S}_{\le 3}^{\frac12} \mathcal{S}_{\le 4}^{\frac12} + (\frac{\mathcal{E}_{ell}^{(I, out)}}{\langle t \rangle^{100}}  + \nu^{100}\mathcal{J}_{sob})(\mathcal{D}_{\text{Trace, Max},1} + \nu^{3}\mathcal{E}_{\text{Trace},0}).
\end{align}
\end{lemma}
\begin{proof}[Proof of \eqref{beed:beed:1}] We take inner product of \eqref{eq:alpha:2} with $\nu^6 \p_t \p_x \alpha_1 e^{2W}$: 
\begin{align*}
&\p_t  \mathcal{E}_{\text{Trace},2}(t) + \mathcal{D}_{\text{Trace},2}(t) +  \mathcal{CK}_{\text{Trace},2}(t) \\
\le& | \langle \nu \p_x \alpha_3 - \psi_y(t, x,0) \p_x^2 \alpha_1 + \psi_{xxy}(t,x,0) \alpha_1 , \nu^6 \p_t \p_x \alpha_1 e^{2W}\rangle | \\
\le & \frac12 \mathcal{D}_{\text{Trace},2}(t) + C \| \nu^4 \p_x \alpha_3 e^W \|_{L^2_x}^2 +C \| \nu^3 \widetilde{\psi}_y \p_x^2 \alpha_1 e^W \|_{L^2_x}^2 + C \| \nu^3 \widetilde{\psi}_{xxy} \alpha_1 e^W \|_{L^2_x}^2 \\
=: &  \frac12 \mathcal{D}_{\text{Trace},2}(t) + C \sum_{i = 1}^3 \text{Err}^{(T)}_i(t). 
\end{align*}
We will now estimate the error terms appearing above. We first have 
\begin{align*}
|\text{Err}^{(T)}_1|  \lesssim & \| \nu^{3 + \frac12} \p_x \p_y^3 \omega e^W \|_{L^2} \| \nu^{4 + \frac12} \p_x \p_y^4 \omega e^W \|_{L^2} +  \| \nu^{3 + \frac12} \p_x \p_y^3 \omega e^W \|_{L^2} \| \nu^4 \p_x \p_y^3 \omega e^W (\sqrt{\nu} W_y) \|_{L^2} \\
\lesssim & \mathcal{D}_{sob,3}^{\frac12} \mathcal{D}_{sob,4}^{\frac12} +  \mathcal{D}_{sob,3}^{\frac12}  \mathcal{CK}_{sob,4}^{\frac12}.
\end{align*}
Next, we have  
\begin{align*}
|\text{Err}^{(T)}_2(t)| \lesssim &  \| \widetilde{\psi}_y \|_{L^\infty_x}^2 \| \nu^3 \p_x^2 \alpha_1 e^W \|_{L^2_x}^2 \lesssim  \| \widetilde{\psi}_y \|_{L^\infty_x}^2 \mathcal{D}_{\text{Trace, Max},1} \\
\lesssim & (\frac{\mathcal{E}_{ell}^{(I, out)}}{\langle t \rangle^{100}}  + \mathcal{J}_{sob})\mathcal{D}_{\text{Trace, Max},1}. 
\end{align*}
Last, we have 
\begin{align*}
|\text{Err}^{(T)}_3(t)| \lesssim & \nu^{3}  \|  \widetilde{\psi}_{xxy} \|_{L^\infty_x}^2 \mathcal{E}_{\text{Trace},0} \lesssim \nu^3 (\frac{\mathcal{E}_{ell}^{(I, out)}}{\langle t \rangle^{100}}  + \mathcal{J}_{sob}) \mathcal{E}_{\text{Trace},0}
\end{align*}
\end{proof}
\begin{proof}[Proof of \eqref{beed:beed:2}] Using equation \eqref{eq:alpha:2}, we get 
\begin{align*}
\mathcal{D}_{\text{Trace,Max},2} \lesssim \mathcal{D}_{\text{Trace},2} +   \| \nu^4 \p_x \alpha_3 e^W \|_{L^2_x}^2 + \| \nu^3 \widetilde{\psi}_y \p_x^2 \alpha_1 e^W \|_{L^2_x}^2 +  \| \nu^3 \widetilde{\psi}_{xxy} \alpha_1 e^W \|_{L^2_x}^2.
\end{align*}
Upon noting that all quantities on the right-hand side have been estimated in the previous estimate, our claim follows.  

\end{proof}

We are now ready to establish:
%
\begin{lemma}[$\alpha_4(t, x)$ bounds] The trace quantity $\alpha_4(t)$ satisfies the following bounds:  
\begin{align} \label{invoke:me:1}
\mathcal{D}_{\text{Trace, Large},0}(t) \lesssim &(1+  \mathcal{J}_{sob}^{\frac12}) \mathcal{S}_{\le 3}, \\  \label{invoke:me:2}
\mathcal{D}_{\text{Trace, Large},1}(t) \lesssim &  (1+ (\mathcal{E}_{ell}^{(I, out)})^{\frac12} +  \mathcal{J}_{sob}^{\frac12}) \mathcal{S}_{\le 3},\\  \label{invoke:me:3}
\mathcal{D}_{\text{Trace, Large},2}(t) \lesssim &  \mathcal{S}_{\le 3}^{\frac12} \mathcal{S}_{\le 4}^{\frac12} + (\frac{\mathcal{E}_{ell}^{(I, out)}}{\langle t \rangle^{100}}  + \nu^{100} \mathcal{J}_{sob})(\mathcal{D}_{\text{Trace, Max},1} + \nu^{3}\mathcal{E}_{\text{Trace},0}).
\end{align}
\end{lemma}
\begin{proof} First, we have upon invoking the equation \eqref{sr:a}, the bounds
\begin{align*}
\mathcal{D}_{\text{Trace, Large},0}(t) \lesssim &  \| \bar{U}' \|_{L^\infty_x}^2 \| \nu^2 \p_x \alpha_1 e^W \|_{L^2_x}^2 \lesssim \mathcal{D}_{\text{Trace},0}(t), \\
\mathcal{D}_{\text{Trace, Large},1}(t) \lesssim &  \| \bar{U}' \|_{L^\infty_x}^2 \| \nu^3 \p_x^2 \alpha_1 e^W \|_{L^2_x}^2 \lesssim \mathcal{D}_{\text{Trace, Max},1}(t),
\end{align*}
and next, upon invoking the equation \eqref{sr:b}, the bound 
\begin{align*}
\mathcal{D}_{\text{Trace, Large},1}(t) \lesssim &  \| \bar{U}' \|_{L^\infty_x}^2( \| \nu^3 \p_x \p_t \alpha_1 e^W \|_{L^2_x}^2 +  \| \nu^4 \p_x^3 \alpha_1 e^W \|_{L^2_x}^2) \lesssim \mathcal{D}_{\text{Trace},2}(t) + \mathcal{D}_{\text{Max, Trace},2}(t).
\end{align*}
From here, the desired bounds follow from the previously established bounds on the diffusion terms appearing on the right-hand sides above. 
\end{proof}

We have thus proven the following
\begin{proposition} The one-dimensional traces satisfy the following energy-CK-dissipation estimate: 
\begin{align} \n
\frac{\p_t}{2} \mathcal{E}_{\text{Trace}}(t) +   \mathcal{CK}_{\text{Trace}}(t) + \mathcal{D}_{\text{Trace}}(t) \lesssim &(1 +  (\frac{\mathcal{E}_{ell}^{(I, out)}}{\langle t \rangle^{100}}  + \nu^{100} \mathcal{J}_{sob})) (1+ (\mathcal{E}_{ell}^{(I, out)})^{\frac12} +  \mathcal{J}_{sob}^{\frac12}) \mathcal{S}_{\le 3} \\
&+  \mathcal{S}_{\le 3}^{\frac12} \mathcal{S}_{\le 4}^{\frac12} + \nu^3  (\frac{\mathcal{E}_{ell}^{(I, out)}}{\langle t \rangle^{100}}  + \nu^{100} \mathcal{J}_{sob}) \mathcal{E}_{\text{Trace}}(t). 
\end{align}
\end{proposition}

\subsection{Proof of Proposition \ref{pro:int:sob}}

\begin{proof}[Proof of Proposition \ref{pro:int:sob}] We have proven the following estimates: 
\begin{subequations}
\begin{align}\label{jayz:a}
&\frac{\p_t}{2} \mathcal{E}_{sob} + \mathcal{CK}_{sob}^{(W)} + \mathcal{D}_{sob} \lesssim \| \nu^{\frac12} C_{k,0} e^W \|_{L^2}^2 +\sum_{i = 1}^3 \sum_{\mathcal{K} \in \{ A, B, C\}} \| \nu^{i + \frac12} \mathcal{K}_{k,i} e^W \|_{L^2}^2   +  \| \nu^{\frac72} \mathcal{H}_{k,3} e^W \|_{L^2}^2,  \\ \n
&\frac{\p_t}{2} \mathcal{E}_{\text{Trace}}(t) +   \mathcal{CK}_{\text{Trace}}(t) + \mathcal{D}_{\text{Trace}}(t) \lesssim (1 +  (\frac{\mathcal{E}_{ell}^{(I, out)}}{\langle t \rangle^{100}}  + \nu^{100} \mathcal{J}_{sob})) (1+ (\mathcal{E}_{ell}^{(I, out)})^{\frac12} +  \mathcal{J}_{sob}^{\frac12}) \mathcal{S}_{\le 3} \\ \label{jayz:b}
&\qquad +  \mathcal{S}_{\le 3}^{\frac12} \mathcal{S}_{\le 4}^{\frac12} + \nu^3  (\frac{\mathcal{E}_{ell}^{(I, out)}}{\langle t \rangle^{100}}  + \nu^{100} \mathcal{J}_{sob}) \mathcal{E}_{\text{Trace}}(t), \\ \label{jayz:c}
&\mathcal{J}_{sob} \lesssim  \mathcal{E}_{sob},\\ \label{jayz:d}
&\| \nu^{\frac12} C_{k,0} e^W \|_{L^2}^2 +\sum_{i = 1}^3 \sum_{\mathcal{K} \in \{ A, B, C\}} \| \nu^{i + \frac12} \mathcal{K}_{k,i} e^W \|_{L^2}^2  \lesssim (\nu + \mathcal{J}_{sob}) \mathcal{D}_{sob}, \\ \label{jayz:e}
&\| \nu^{\frac72} \mathcal{H}_{k,3} e^W \|_{L^2}^2 \lesssim  \nu^2 \mathcal{D}_{\text{Trace}}.
\end{align}
\end{subequations}
We insert the bounds \eqref{jayz:d} -- \eqref{jayz:e} into \eqref{jayz:a} -- \eqref{jayz:c} to produce our main estimates 
\begin{subequations}
\begin{align} \label{hayz:a}
\frac{\p_t}{2} \mathcal{E}_{sob} + \mathcal{CK}_{sob}^{(W)} + \mathcal{D}_{sob} \lesssim &\nu^2 \mathcal{D}_{\text{Trace}} +(\nu + \mathcal{J}_{sob}) \mathcal{D}_{sob} ,  \\ \n
\frac{\p_t}{2} \mathcal{E}_{\text{Trace}}(t) +   \mathcal{CK}_{\text{Trace}}(t) + \mathcal{D}_{\text{Trace}}(t) \lesssim &(1 +  (\frac{\mathcal{E}_{ell}^{(I, out)}}{\langle t \rangle^{100}}  + \nu^{100} \mathcal{J}_{sob})) (1+ (\mathcal{E}_{ell}^{(I, out)})^{\frac12} +  \mathcal{J}_{sob}^{\frac12}) \mathcal{S}_{\le 3} \\ \label{hayz:b}
&+  \mathcal{S}_{\le 3}^{\frac12} \mathcal{S}_{\le 4}^{\frac12} + \nu^3  (\frac{\mathcal{E}_{ell}^{(I, out)}}{\langle t \rangle^{100}}  + \nu^{100} \mathcal{J}_{sob}) \mathcal{E}_{\text{Trace}}(t), \\ \label{hayz:c}
\mathcal{J}_{sob} \lesssim & \mathcal{E}_{sob}.
\end{align}
\end{subequations}
These bounds evidently close to provide the result of the proposition.
\end{proof}

\section{Inviscid Limit}  \label{sec:IL}
In this section we indicate how to prove the inviscid limit stated in Theorem \ref{thm:main} part (iii). 
The proof also applies on $\mathbb T \times \mathbb R$ (a result that the work \cite{BMV14} did not carry out). 
For a given $\omega_{in}$ satisfying the hypotheses of Theorem \ref{thm:main}, for $\nu \geq 0$ let $\set{\omega^\nu}_{\nu \geq 0}$ be the solutions studied in Theorem \ref{thm:main} (the case $\nu = 0$ being covered by \cite{HI20}).

The start to the inviscid limit begins by studying the PDE solved by $\tilde{\omega}^\nu := \omega^{\nu} - \omega^{(0)}$, which involves terms which are linear in $\tilde{\omega}^\nu$ (the Euler nonlinearity linearized around $\omega^{(0)}$) and a term which is nonlinear $\tilde{\omega}^\nu$. 
However, one sees immediately that in order to reach time-scales much longer than e.g. $\abs{\log \eps}^{-1}$, one needs to be able to treat these linear terms rather carefully, and if one wants to get time-scales which are basically independent of $\eps$, one will also need to carefully control the nonlinear terms as well.
In particular, it becomes clear that one will need to re-do a potentially significant portion of the kind of estimates we carried out in Proposition \ref{prop:boot}.
For this reason, we must carry out the inviscid limit in Gevrey spaces similar to $\mathcal{E}_{\text{Int}}$ and $\mathcal{E}^{(\gamma)}$. 
The difficulty with this is that the coordinate system/adapted vector fields themselves depend on $\nu$, which makes the inviscid limit a significantly more subtle proposition than one may guess at first. 

Below, we denote by
\begin{align}\label{Ga_t_nu_0}
\Gamma_{t;\nu} = \frac{1}{\pa_y v^{\nu}} \partial_y + t\partial_x,\quad \Gamma_{t;0}=\frac{1}{\pa_y v^0}\pa_y+t\pa_x, 
\end{align}
the vector field derivatives adapted to the viscous ($\nu>0$) and inviscid ($\nu=0$) problems. 
We will work in the \emph{inviscid} coordinate system, i.e. the coordinate system for $\nu=0$.
The main step is the following inviscid limit. 
\begin{theorem}[Long-time inviscid limit]
For $\nu \geq 0$ and let $\omega^\nu$ be the solutions considered in Theorem \ref{thm:main}. 
Then for any $\zeta > 0$, $t \lesssim \nu^{-1/(3+\zeta)}$, and all $\lambda_\ast < \frac{1}{4}\lambda$, there holds  
\begin{align}
\sum_{n,m} \left(\frac{\lambda^{n+m}_\ast}{((n+m)!)^{1/r}}\right)^2 \norm{ e^{W/10} \partial_x^{m} q^{2n} \Gamma_{t;0}^n(\omega^\nu - \omega^0)}_{L^2}^2 & \lesssim (\eps \nu t^{3+\zeta})^2, \label{ineq:InLim} \\
\sum_{0 \leq n+m \leq 4}\norm{e^{W/10}\partial_x^n (\nu^{-1}\partial_y)^m(\omega^\nu - \omega^0)}_{L^2} & \lesssim \eps \nu t^{3+\zeta}, \\
\sum_{0 \leq m \leq 4}\norm{e^{W/10}\partial_y^m P_0(\omega^\nu - \omega^0)}_{L^2} & \lesssim \eps \nu t^{2+\zeta}. 
\end{align}
\end{theorem}



\subsection{Inviscid limit in the exterior}

We first show that in the exterior, we can pass to the inviscid limit in the pseudo-Gevrey regularity spaces adapted to $\Gamma_{t;0}$ (as well as standard $H^4$ regularity). 
Note that in the exterior, $\omega^{0}$ vanishes, so the main challenge here is to translate the existing estimates with $\Gamma_{t;\nu}$ into estimates in terms of $\Gamma_{t;0}$.

\begin{theorem}[Long-time inviscid limit]\label{thm_bdy_lim}
Let $\omega^\nu$ be the Navier-Stokes solution subject to viscosity $0<\nu\ll 1$ and $\omega^0$ be the Euler solution ($\nu=0$). Fix two arbitrary parameters $r\in(0,1),\, s_0\in(s,5/4)$, where $s$ is the exterior pseudo-Gevrey index specified in \eqref{pgiL1}. 
Then for all $t \leq \nu^{-1/3-\zeta}$, and for all $\lambda_\ast$ chosen small enough compared to universal constants and exterior Gevrey radius $\lambda$ \eqref{Gev:la}, there exists a Gevrey radius parameter $\lambda_1=C(r,s_0,K)\lambda_\ast^{\frac{1}{r s_0}}$ (with $K$ being defined in \eqref{defndW}) such that the following estimate holds 
\begin{align}\n
\sum_{ m+n=0 }^\infty &\frac{\lambda^{2(m+n)/r}_\ast}{((m+n)!)^{2/r}} \norm{ \chi^E(v^\nu(t,\cdot)) \mathbbm{1}_{\abs{y} \leq 7/8} e^{W/8} \partial_x^{m}  \Gamma_{t;0}^n(\omega^\nu - \omega^0)}_{L_{x,y}^2}^2\\
 &\hspace{-1cm}\lesssim \sum_{ m+n=0 }^\infty \frac{\lambda_1^{2s_0(m+n)} \varphi^{2n}}{((m+n)!)^{2s_0}} \norm{ \chi^E(v^\nu(t,\cdot)) e^{W/2} \partial_x^{m} q^{2n} \Gamma_{t;0}^n (\omega^\nu -\omega^0)}_{L_{x,y}^2}^2\n \\
 &\hspace{-1cm}\lesssim \exp\{-\nu^{-1/8}\}\mathcal{E}^{(\gamma)}\lesssim \exp\{-\nu^{-1/8}\}\ep^2.  \label{ineq:InLim}
\end{align}
We recall that $\chi^{E}=(1-\chi^I)$ \eqref{chi:I:def}, and $\mathcal{E}^{(\gamma)}$ is defined in \eqref{ef:a}. Here the Gevrey index $r$ can be chosen to be  strictly greater than $1/2$ if the solution $\omega^\nu$ is in the pseudo-Gevrey-$s$ regularity space ($s<5/4$) in the exterior region.  
 
Moreover, the inviscid limit holds in the Sobolev space in the exterior
\begin{align}
\sum_{0 \leq m+n \leq 4}\norm{e^{W/9}\ \chi^E\ \partial_x^n \partial_y^m\ (\omega^\nu - \omega^0)}_{L^2} & \lesssim \exp\{-\nu^{-1/8}\}\eps . \label{ineq:H4IL}
\end{align}
\end{theorem}
\begin{remark}
Here, $\lambda_\ast$ determines the space in which the vanishing viscosity limit holds. The parameters $s_0,\, \lambda_1$ will not be used elsewhere.  
\end{remark}

\begin{proof}
First, notice that \eqref{ineq:H4IL} follows from the estimates on $\mathcal{E}_{sob}$ in Proposition \ref{prop:boot} just using that the weight $e^{W/8}$ is large in the exterior region on the time interval $[0,\nu^{-1/3-\zeta}]$.
We next focus on \eqref{ineq:InLim}, which is significantly more complicated. We organize the proof in several steps. In {\bf Step \#~1}, we simplify the problem and introduce the necessary setups. In {\bf Step \#~2}, we discuss the related coordinate systems in the problem and use the Fa\`a di Bruno's formula to set up the main expression to estimate. We will see that the essential difficulties are controlling certain combinatorial objects. In {\bf Step \#~3}, we apply combinatorial arguments to bound these combinatorial objects. In {\bf Step \#~4}, we collect all the information and finish the proof.  

\noindent
\textbf{Step \# 1: Setting up.}
First, we identify the key difficulties in the proof. We observe that the first inequality in \eqref{ineq:InLim} is a variant of Lemma \ref{lem:ExtToInt}. By setting $\widetilde{\mathfrak{r}}=r,\quad \mathfrak{r}=1/s_0$ in the lemma, and introducing the $e^{W/8}$ weight in the proof, one can derive the result. We omit further details for the sake of brevity. Furthermore, the third inequality in \eqref{ineq:InLim} is a direct consequence of the bootstrap hypothesis \eqref{boot:ExtVort}. Hence, to prove the theorem, it remains to derive the second inequality in \eqref{ineq:InLim}. 

We further simplify our task by observing that the Euler vorticity $\omega^0$ is   compactly supported in the interior of the channel, i.e., $\text{support}\{\omega^0\} \subset (-1/2,1/2)$, which yields that
\begin{align*}
\sum_{m+n=0}^\infty \frac{\lambda_1^{2s_0(m+n)}\varphi^{2n}}{((m+n)!)^{2s_0}} \norm{\chi^E e^{W/2} \partial_x^{m} q^{2n} \Gamma_{t;0}^n(\omega^\nu - \omega^0)}_{L_{x,y}^2}^2 & \\
&  \hspace{-4cm}= \sum_{m+n=0}^\infty \frac{\lambda_1^{2s_0(m+n)}\varphi^{2n}}{((m+n)!)^{2s_0}} \norm{\chi^E e^{W/2} \partial_x^{m} q^{2n} \Gamma_{t;0}^n\omega^\nu }_{L_{x,y}^2}^2. 
\end{align*}
Here we highlight that since we only care about the estimate in the exterior region, i.e.,  the support of $\chi^E$, we have that the $\Gamma_{t;0}$-derivative simplifies, i.e.,
\begin{align*}
\Gamma_{t;0}f(t,x,y)=(\pa_y+t\pa_x)f(t,x,y),\quad \forall y\in \bigcup_{0\leq \nu\ll 1}\text{support}_y\{\chi^E(v^\nu(t,\cdot))\}\subset [-1,1]\backslash(-1/2,1/2).
\end{align*}
This is a consequence of the fact that $v^0=y+$constant in the region $(-1/2,1/2)^c$. For further discussions of this fact, we refer the readers to \cite{HI20}.   
Later on in the proof, we will consider the following coordinate system
\begin{align}\label{coord_x_v_nu}
(x,v^\nu(t,y)), \quad 0<\nu \ll 1 
\end{align}
with $v^\nu$ being the coordinate systems defined by \eqref{eq:v}, here now explicitly parameterized by $\nu$. This coordinate system is natural in the sense that most of the estimates ($\mathcal{E}^{(\gamma)}$, etc) developed in this paper have direct counterparts in this coordinate system (due to the choice of $\Gamma_{t;\nu}$). Within the proof of Theorem \ref{thm_bdy_lim}, we apply the following notation to represent the vorticity in $(x,v^\nu)$-coordinate,
\begin{align*}
\Omega^\nu(t,x,v^\nu(t,y))=\omega^\nu(t,x, y),\quad \Omega_k^\nu(t,v^\nu(t,y))=\omega_k^\nu(t,y).
\end{align*}
As above, the partial derivative $\pa_{v^\nu}$ has the following expression in the $(x,y)$-coordinate:
\begin{align} \pa_{v^\nu}\Omega_k^\nu=\overline{\pa}_{v^\nu}\omega_k^\nu=(\pa_{v^\nu} y)\pa_y\omega_k^\nu =\frac{1}{\pa_{y}v^\nu} \pa_y\omega_k^\nu.
\end{align}
Throughout the proof, we always use $\overline{\partial}$ to denote the representation of corresponding derivative in the $(x,y)$-coordinate.  

The final observations leading to simplification is that the definitions \eqref{chi:I:def}, \eqref{chi} and \eqref{defndW} yield the following pointwise estimate for all $t\in(0,\nu^{-1/3-\zeta}]$,
\begin{align*}
&\text{support}_y \{\chi^E(v^\nu(t,\cdot))\}\subset \bigcap_{n=0}^\infty \text{support}_y \{\chi_n\}, \quad\varphi^{-2} e^{W/2}\leq C e^{-\nu^{-1/8}}e^{W}\\
&\hspace{2cm}\Rightarrow \chi^E \varphi^{-1}e^{W/2}\leq Ce^{-\nu^{-1/8}}e^{W}\chi^E\chi_n.
\end{align*} 
This observation, together with the definition of $\mathcal{E}^{(\gamma)}$ \eqref{ef:a} yields that the following estimate guarantees the second inequality in \eqref{ineq:InLim} and hence implies the theorem:
\begin{align}
&\sum_{m+n=0}^\infty \frac{\lambda_1^{2s_0(m+n)}\varphi^{2n}}{((m+n)!)^{2s_0}} \norm{\chi^E (v^\nu(t,\cdot)) e^{W/2} \partial_{x}^{m} q^{2n} (\pa_y+t\pa_x)^n\omega^\nu}_{L_{x,y}^2}^2\n \\
&\lesssim \sum_{m+n=0}^\infty \frac{\lambda^{2s(m+n)}\varphi^{2n}}{((m+n)!)^{2s}} \norm{\chi^E {(v^\nu(t,\cdot))} e^{W/2} \partial_{x}^{m} q^{n} \Gamma_{t;\nu}^n \omega^\nu }_{L_{x,y}^2}^2\lesssim e^{-\nu^{-1/8}}\mathcal{E}^{(\gamma)},\quad \forall t \in\lf[0, \nu^{-1/3-\zeta}\rg].\label{ext_dif}
\end{align}
Here $\lambda$ is the radius of Gevrey regularity in exterior region and the $\Gamma_{t;\nu}$ vector field is defined in \eqref{Ga_t_nu_0}. Moreover, we highlight that $\Gamma_{t;\nu}$ is the $\Gamma$'s we applied throughout this paper. Hence, the estimate \eqref{ext_dif} can be viewed as a translation between the $(\pa_y+t\pa_x)$-derivative in the $(x,y)$-coordinate to the $(\pa_{v^\nu}+t\pa_x)$-derivative in the $(x,v^\nu)$-coordinate (with $\Gamma_{t;\nu}=\overline{\pa}_{v^\nu}+t\pa_x$ being its $(x,y)$-coordinate representation). The classical way to do this transition in Gevrey spaces is to use the Fa\`a di Bruno's formula which can easily switch $\pa_y$ to $\pa_{v^\nu}$ by paying some regularity on $v^\nu$. However, the extra $t\pa_x$ component complicates the business. With this comment, we conclude {\bf Step \# 1}.

\ifx
\jacob{Hmm I am a little confused...why do we want to put the $v^\nu$ in there? Is it clear that \eqref{ext_dif} yields the estimate we actually want? ie. why can we/why do we want to, commute these compositions past the derivatives?
Is this some trickery because directly estimating is too hard? I would have thought we need to estimate the following:}\siming{Answer: Because the solutions $\omega^\nu$ in the exterior regions are essentially measured in the $(x,v^\nu)$ coordinate systems. When we take the $\pav^n$ in the exterior regions, it is like the $\pa_{v^\nu}$ in the $(x,v^\nu)$ coordinate. I just try to highlight that here. }

\begin{align*}
\sum_k\sum_{m+n=0}^\infty \frac{\lambda_1^{2s_0(m+n)}}{(m+n)!^{2s}}\lf\|\varphi^{2n}e^{W/2}\chi^E\sum_{j_1=0}^nq^{2n-j_1} \binom{n}{j_1}(ik)^{m} \ q^{j_1} \left(\frac{\partial_y v^\nu}{ \partial_y v^0}\partial_{v^\nu}\right)^{j_1}
(ikt)^{n-j_1} \omega_k^{\nu} \rg\|_{L_y^2}^2 .
\end{align*}
and then expanding
\begin{align*}
\left(\frac{v^\nu_y}{ v^0_y} \partial_{v^\nu}\right)^{j_1} &= \sum_{\ell=0}^{j_1} \mathfrak{G}_{\ell,j_1} \partial_{v^\nu}^\ell  \\
&= \sum_{\ell=0}^{j_1} \mathfrak{G}_{\ell,j_1} \sum \frac{\ell!}{\ell'!(\ell-\ell')!} \Gamma_{t;\nu}^{\ell'} (ikt)^{\ell-\ell'}, 
\end{align*}
where the $\mathfrak{G}_{\ell,j_1}$ is something complicated based on iterated commutators. 
Oh yeah proceed by induction.
Suppose I want to commute $(aX)^j$ for a scalar field times a vector field. 
Suppose you already know how to re-write it as a polynomial in $a_{j;n} X^j$ for some coefficients for $(aX)^{n}$. Then
\begin{align*}
(aX)^{n+1} = \sum a_{j;n} X^j aX & = \sum a_{j;n}a X^{j+1} +  a_{j;n} [X^j,a] X \\
& = \sum_{j=0}^n a_{j;n}a X^{j+1} +  \sum_{p=0}^j a_{j;n} \frac{j!}{p! (j-p)!} \mathrm{ad}_X^{j-p}(a) X^{p+1} \\
& = \sum_{p=0}^n a_{p;n}a X^{p+1} +  \sum_{p=0}^n (\sum_{j=p}^n a_{j;n} \frac{j!}{p! (j-p)!} \mathrm{ad}_X^{j-p}(a)) X^{p+1} \\
\end{align*}
This technically gives a recursion relation
\begin{align*}
a_{p+1;n+1} = a a_{p;n} + \sum_{j=p}^n a_{j;n} \frac{j!}{p! (j-p)!} \mathrm{ad}_X^{j-p}(a).   
\end{align*}
That's a pretty awful formula.  
This means we would get something like
\begin{align*}
\mathfrak{G}_{\ell,j_1} = \frac{v^\nu_y}{ v^0_y} \mathfrak{G}_{\ell,j_1-1} + \sum_{j=p}^{j_1-1} \mathfrak{G}_{j;j_1-1} \frac{j!}{p! (j-p)!} \mathrm{ad}_{\partial_{v^\nu}}^{j-p}( \frac{v^\nu_y}{ v^0_y} ).   
\end{align*}
Hmm yeah that's a pretty bad nightmare too...
\fi

\noindent
{\bf Step \# 2: Change of coordinate.} To implement the switching between the derivatives, we apply the Fourier transform in $x$ and expand the expression on the left hand side of \eqref{ext_dif} as follows
\begin{align}
\n \eqref{ext_dif}_{\text{L.H.S}}=&\sum_{k\in \mathbb{Z}}\sum_{m+n=0}^\infty \frac{\lambda_1^{2s_0(m+n)}}{(m+n)!^{2s_0}}\varphi^{2n}\|e^{W/2}\chi^E(ik)^{m} q^{2n} (\pa_y+ikt )^n \omega_k^{\nu} \|_{L_y^2}^2 \\ 
=&\sum_{k\in \mathbb{Z}}\sum_{m+n=0}^\infty \frac{\lambda_1^{2s_0(m+n)}}{(m+n)!^{2s_0}}\bigg\|e^{W/2}\chi^E\underbrace{\lf(\varphi^{n}\sum_{\mf{j}=0}^nq^{2n-\mf{j}} \binom{n}{\mf{j}}(ik)^{m} \ q^{\mf{j}} \pa_y^{\mf{j}}(ikt)^{n-\mf{j}} \omega_k^{\nu}\rg)}_{=:\mathcal T}\bigg\|_{L_y^2}^2 . \label{Step_1}
\end{align}
Now we can see that the $\pa_y^\mf j$ derivatives are singled out and susceptible to the Fa\`a di Bruno's formula, which  we recall here:
\begin{align}
\frac{d^n}{dx^n}f(g(x))=&\sum_ {(m_1,m_2,\cdots, m_n)\in S_n}\frac{n!}{m_1!(1!)^{m_1}m_2! (2!)^{m_2}...m_n! (n!)^{m_n}} f^{(m_1+\cdots+m_n)}(g(x))\cdot \prod_{j=1}^n(g^{(j)}(x))^{m_j},\label{FaadiBruno}\\
\n &\quad S_n:=\left\{(m_1,m_2,\cdots,m_n)\in \mathbb{N}^n\bigg|\sum_{\ell=1}^n \ell m_\ell=n\right\}. 
\end{align} A direct application of the formula \eqref{FaadiBruno} on the coordinate systems $(x,y)$ and $(x,v^\nu)$ yields the following expression for $y\in\text{support}\ \chi^E(t,v^\nu(t,\cdot))$, 
\begin{align*} 
\pa_{y}^{\mf{j}} \omega_k^\nu(t,y)
&=\pa_{y}^{\mf{j}}(\Omega_k^\nu(t,v^\nu(t,y))) \\
&=\sum_{(m_1,m_2,\cdots,m_i)\in S_{\mf{j}}}\frac{{\mf{j}}!}{m_1!m_2!\cdots m_{\mf{j}}!}\lf(\pa_{v^\nu}^{(m_1+m_2+\cdots+m_{\mf{j}})}\Omega_k^\nu\rg)(v^\nu(t,y))\prod_{\ell=1}^{\mf{j}}\lf(\frac{\pa_{y}^{\ell }v^\nu(t,y)}{\ell!}\rg)^{m_\ell}\\
&=\sum_{(m_1,m_2,\cdots,m_i)\in S_{\mf{j}}}\frac{{\mf{j}}!}{m_1!m_2!\cdots m_{\mf{j}}!}\lf(\overline{\pa}_{v^\nu}^{(m_1+m_2+\cdots+m_{\mf{j}})}\omega_k^\nu\rg)(t,y)\prod_{\ell=1}^{\mf{j}}\lf(\frac{\pa_{y}^{\ell }v^\nu(t,y)}{\ell!}\rg)^{m_\ell}\\
&=\sum_{(m_1,m_2,\cdots,m_i)\in S_{\mf{j}}}\frac{{\mf{j}}!}{m_1!m_2!\cdots m_{\mf{j}}!}\lf((\Gamma_{t;\nu}-ikt)^{(m_1+m_2+\cdots+m_{\mf{j}})}\omega_k^\nu\rg)(t,y)\prod_{\ell=1}^{\mf{j}}\lf(\frac{\pa_{y}^{\ell }v^\nu(t,y)}{\ell!}\rg)^{m_\ell}.
\end{align*}
Here in the last line, we use the relation $\overline{\pa}_{v^{\nu}}\omega_k^\nu=({\pa_y v^\nu})^{-1}\pa_y\omega_k^\nu=(\Gamma_{t;\nu}-ikt)\omega_k^\nu$. 

By the above discussion, we have that the $\mathcal T$ in \eqref{Step_1} can be expanded as follows 
\begin{align*}\mathcal T
&=\varphi^{n}\sum_{\mf{j}=0}^n\binom{n}{\mf{j}}(ik)^{m} q^{2n} (ikt)^{n-\mf{j}} \sum_{(m_1,m_2,\cdots,m_i)\in S_{\mf{j}}}\frac{{\mf{j}}!}{m_1!m_2!\cdots m_{\mf{j}}!}\prod_{\ell=1}^{\mf{j}}\lf(\frac{\pa_{y}^{\ell }v^\nu}{\ell!}\rg)^{m_\ell}\\
&\hspace{8cm}\quad\times {(\Gamma_{t;\nu}-ikt)}^{(m_1+m_2+\cdots+m_{\mf{j}})}\omega_k^{\nu}\\
&=\varphi^{n}\sum_{\mf{j}=0}^n\binom{n}{\mf{j}}(ik)^{m} q^{2n} (ikt)^{n-\mf{j}} \sum_{(m_1,m_2,\cdots,m_{\mf{j}})\in S_{\mf{j}}}\frac{{\mf{j}}!}{m_1!m_2!\cdots m_{\mf{j}}!}\prod_{\ell=1}^{\mf{j}}\lf(\frac{\pa_{y}^{\ell}v^\nu}{\ell!}\rg)^{m_\ell}\\
&\quad\times{\sum_{p=0}^{m_1+\cdots+m_{\mf{j}}}\binom{m_1+\cdots+ m_{\mf{j}}}{p} (-ikt)^{(m_1+m_2+\cdots+m_{\mf{j}})-p}\ \ \Gamma_{t;\nu}^{p}\omega_k^{\nu}}\\
  &=\varphi^{n}\sum_{\mf{j}=0}^n\binom{n}{\mf{j}}(ik)^{m} q^{2n} (ikt)^{n-\mf{j}} \sum_{(m_1,m_2,\cdots,m_{\mf{j}})\in S_{\mf{j}}}{\frac{{\mf{j}}!(m_1+\cdots+ m_{\mf{j}})!}{m_1!m_2!\cdots m_{\mf{j}}!}}\prod_{\ell=1}^{\mf{j}}\lf(\frac{\pa_{y}^{\ell}v^\nu}{\ell!}\rg)^{m_\ell}\\
&\quad\times\sum_{p=0}^{m_1+\cdots+m_{\mf{j}}}\frac{1}{(m_1+\cdots+ m_{\mf{j}}-p)!p!} (-ikt)^{(m_1+m_2+\cdots+m_{\mf{j}})-p}\ \ \Gamma_{t;\nu}^{p}\omega_k^{\nu}.
\end{align*}
Now we have that 
\begin{align*}
\mathcal T&=\sum_{\mf{j}=0}^n\binom{n}{\mf{j}}(ik)^{m}  (ikt)^{n-\mf{j}} \sum_{\sum_{\ell=1}^{\mf{j}}\ell m_\ell=\mf{j}}\frac{{\mf{j}}!(m_1+\cdots+ m_{\mf{j}})!}{m_1!m_2!\cdots m_{\mf{j}}!}\prod_{\ell=1}^{\mf{j}}\lf(\frac{{\varphi^{\ell-1} q^\ell} \pa_{y }^{\ell }v^\nu(y )}{\ell!}\rg)^{m_\ell}\\
&\quad\times\sum_{p=0}^{m_1+\cdots+m_{\mf{j}}}\frac{1}{(m_1+\cdots+ m_{\mf{j}}-p)!p!} (-ikt)^{(m_1+m_2+\cdots+m_{\mf{j}})-p}\ \ ({\varphi^p q^p}\Gamma_{t;\nu}^{p}\omega_k^{\nu})\times {\underbrace{q^{2n-\mf{j}-p}}_{\leq 1}\varphi^{{n-p-\mf{j}+\sum_{\ell}m_\ell}}}.
\end{align*}
We observe that $\mf{j}\leq n,\ p\leq m_1+\cdots+m_{\mf{j}}\leq n$, so $n-p-\mf{j}+\sum_{\ell}m_\ell\geq 0.$ With the expansion above, we can carry out the estimate on each Fourier $k$-mode:
\begin{align}\n
 &\sum_{m+n=0}^\infty\frac{\lambda_1^{2s_0(m+n)}}{(m+n)!^{2{s_0}}}\varphi^{2n}\|e^{W/2}\chi^E |k|^m  q^{2n}\Gamma_{t;0}^n \omega_k^{\nu}\|_2^2\\ \n
&\leq \sum_{m+n\geq 0}\lf(\frac{\lambda_1^{s_0(m+n)}}{(m+n)!^{{s_0}}}\sum_{\mf{j}=0}^n\binom{n}{\mf{j}}|k|^{m+n-\mf{j}}\sum_{\sum_{\ell=1}^{\mf{j}} \ell m_\ell={\mf{j}}}\frac{{\mf{j}}!(\sum_{\ell}m_\ell)!}{m_1! m_2!\cdots m_{\mf{j}}!}\prod _{\ell=1}^{\mf{j}}\left\|\frac{{ \varphi^{{\ell-1}} q^\ell}\pa_{y }^{\ell} v^\nu}{\ell!}\right\|_{L^\infty({\rm supp}\{\chi^E\})}^{m_\ell} \rg.\\ \n
&\quad\lf.\times\sum_{p=0}^{m_1+\cdots+m_{\mf{j}}}\frac{|k|^{(m_1+m_2+\cdots+m_{\mf{j}})-p}}{(m_1+\cdots+ m_{\mf{j}}-p)!p!} \ \|e^{W/2} \chi^E {\varphi^p q^p}\Gamma_{t;\nu}^{p}\omega_k^{\nu}\|_{L^2} \underbrace{ \varphi^{n-p-\mf{j}+\sum_\ell m_\ell}t^{n-\mf{j}+\sum_{\ell}m_\ell-p}}_{\leq1	}\rg)^2\\ \n
&\leq
\sum_{m+n\geq 0}\bigg({(\lambda_1  \lambda^{-1}\Lambda^{-1})^{s_0(m+n)}}\sum_{\mf{j}=0}^n\sum_{\sum_{\ell=1}^{\mf{j}} \ell m_\ell ={\mf{j}}}\sum_{p=0}^{\sum_{\ell=1}^{\mf{j}}m_\ell}{\underbrace{\frac{n!}{(n-\mf{j})!}\frac{1}{(\sum m_\ell -p)!p! \prod_\ell (\ell!)^{m_\ell}}}_{=:F_1}}\\ \n
&\hspace{0.2 cm}\times{\underbrace{\lf(\frac{\lf(m+n-\mf{j}+\sum_{\ell=1}^{\mf{j}}m_\ell\rg)! \prod_{\ell=1}^{\mf{j}}(\ell!)^{m_\ell}}{(m+n)!}\rg)^{{s_0}}}_{=:F_2^{{s_0}}}}\underbrace{\frac{(\sum_\ell  m_\ell )!}{m_1! m_2!\cdots m_{\mf{j}}!}  \prod_{\ell=1}^{\mf{j}} \lf(\frac{\|\varphi^{{\ell-1}} q^\ell \pa_{y }^{\ell}v^{\nu}\|_{L^\infty({\rm supp}\{\chi^E\})}\Lambda^{s_0\ell}}{(\ell !)^{{s_0}}}\rg)^{m_\ell }}_{=:F_3}\\
&\hspace{0.2 cm}\times  \frac{ \lambda^{s_0(m+n)}\|e^{W/2}\chi^E|k|^{m+n-\mf{j}+\sum_{\ell=1}^{\mf{j}}m_\ell-p}\varphi^p q^p\Gamma_{t;\nu}^{p}\omega_{k}^{\nu} \|_{L^2 }}{(m+n-\mf{j}+\sum_{\ell=1}^{\mf{j}}m_\ell)!^{{s_0}}}\bigg)^2.\label{F123}
\end{align}
Here we recall the radius of convergence $\lambda$ for the vorticity $\omega_k^\nu$ in the exterior region and define $\Lambda$ as the radius of convergence for $v^\nu$ in the $y $-coordinate.   
Here the second inequality is organized according to the multinomial formula
\begin{align}
(x_1+x_2+\cdots+x_m)^n=\sum_{\substack{k_1+k_2+\cdots+k_m=n;\\ k_1,k_2,\cdots, k_m\geq 0}}\frac{n!}{k_1!k_2!\cdot\cdot\cdot k_m!}\prod_{t=1}^m x_t^{k_t}.
\end{align}
With this expansion, we concludes 
{\bf Step \# 2}.

\noindent
{\bf Step \# 3: Controlling combinatorial factors $F_1, F_2, F_3$.} 
Now we observe the following computation fact of $F_2$ from \eqref{combn}
\begin{align*}
F_2&=\frac{\lf(m+n-(\mf{j}-\sum_{\ell=1}^{\mf{j}}m_\ell)\rg)! \prod_{\ell=1}^{\mf{j}}(\ell!)^{m_\ell}}{(m+n)!}\\ &\leq \frac{\lf(m+n-(\mf{j}-\sum_{\ell=1}^{\mf{j}}m_\ell)\rg)! \mf{j}!}{(m+n)!(\sum_\ell m_\ell)!}\\
&=\frac{\lf(m+n-(\mf{j}-\sum_{\ell=1}^{\mf{j}}m_\ell)\rg)! (\mf{j}-\sum_\ell m_\ell)!}{(m+n)!}\binom{\mf{j}}{\sum_\ell m_\ell}\\
&=\binom{m+n}{\mf{j}-\sum_\ell m_\ell}^{-1}\binom{\mf{j}}{\sum_\ell m_\ell}=\binom{m+n}{\mf{j}-\sum_\ell m_\ell}^{-1}\binom{\mf{j}}{\mf{j}-\sum_\ell m_\ell}\leq1.
\end{align*}
Hence we have that for ${s_0}>1$, 
\begin{align}\n
F_1 F_2^{{s_0}}\leq& F_1 F_2=\frac{n!}{(n-\mf{j})!\mf{j}!}\frac{\mf{j}!}{(\sum m_\ell -p)!p! \prod_\ell (\ell!)^{m_\ell}}\frac{\lf(m+n-(\mf{j}-\sum_{\ell=1}^{\mf{j}}m_\ell)\rg)! \prod_{\ell=1}^{\mf{j}}(\ell!)^{m_\ell}}{(m+n)!}\\
\leq&\frac{\lf(m+n-(\mf{j}-\sum_{\ell=1}^{\mf{j}}m_\ell)\rg)! }{(n-\mf{j})!(\sum m_\ell -p)!p!m!}=\binom{m+n-\mf{j}+\sum_{\ell=1}^{\mf{j}}m_\ell}{n-\mf{j},\sum m_\ell -p,p,m}\leq 4^{m+n}.\label{F12}
\end{align}
Here in the last line, we used the result that the multinomial coefficient $\binom{N}{N_1,N_2,N_3,N_4}\leq 4^N$.  

Finally, we estimate the $F_3$ factor. Through explicit estimate in the same spirit as \cite{Yamanaka89} (also see  \cite{J20}), we obtain that the following Gevrey inversion estimate holds for all $ \Lambda\in [1,2]$ and an associated $\wt \Lambda=\wt\Lambda(\Lambda,s_0)$, 
\begin{align}
 \sum_{\ell=1}^\infty \frac{\|\varphi^{\ell} q^\ell \pa_{y }^{\ell }v^{\nu} \|_{L^\infty({\rm supp}\{ \chi^E\})}\Lambda^{s_0\ell}}{(\ell !)^{{s_0}}} 
{\lesssim\sum_{\ell=0}^\infty \frac{\|\varphi^{\ell} q^\ell \pa_{v^\nu }^{\ell }y \|_{L^\infty({\rm supp}\{ \chi^E\})}\wt\Lambda^{s_0\ell}}{(\ell !)^{{s_0}}} }\lesssim  \lf(\sum_{\ell=0}^\infty \frac{\|\varphi^{\ell} q^\ell \overline{\pa}_{v^\nu}^{\ell }H\|_{L^2(\text{supp} \{\chi^E\})}^2\lambda^{2s\ell }}{(\ell !)^{2s}}\rg)^{1/2}.
\end{align}
Here we have used the assumption that $s<s_0.$ 
 Thanks to the multi-nomial formula and the Gevrey inversion formula,  
\ifx 
\footnote{\myb{HS: This is the bound that we will work on. According to Fei, on the support of $1-\chi_I$, the $v^0(t,y)=y+c$. This is the result from the Hao-Ionescu paper. To derive it, one will incorporate the fact that the vorticity is zero near the boundary and the solution $\psi$ is harmonic. So the estimate that we really need is $\sum_{\ell=1}^\infty \frac{\|\varphi^\ell q^\ell \pa_{y}^{\ell }v^{\nu}\|_{L^\infty(\text{supp}(1-\chi_I))}\lambda^{\ell/r}}{(\ell !)^{1 /r}}$. We have $\sum_{\ell=1}^\infty \frac{\|\varphi^\ell q^\ell \pa_{v^\nu}^{\ell }y\|_{H^1(\text{supp}(1-\chi_I))}\lambda^{\ell/r}}{(\ell !)^{1 /r}}\leq 1$ thanks to the $\mathcal{E}_H^\al$ and the weight within. Certain version of Gevrey inverse estimate should work?
}}\fi
\begin{align}F_3 \lesssim& \lf(\sum_{\ell=1}^\infty \frac{\|\varphi^{{\ell-1}} q^\ell \pa_{y }^{\ell }v^{\nu}\|_{L^\infty({\rm supp}\{ \chi^E\})}\Lambda^{s_0\ell}}{(\ell !)^{{s_0}}}\rg)^{\sum_{\ell} m_\ell}\lesssim \lf(\lan t\ran\sum_{\ell=1}^\infty \frac{\|\varphi^{\ell} q^\ell \pa_{y }^{\ell }v^{\nu} \|_{L^\infty({\rm supp}\{ \chi^E\})}\Lambda^{s_0\ell}}{(\ell !)^{{s_0}}}\rg)^{\sum_{\ell} m_\ell}\n \\
\lesssim& \lf( \sum_{n=1}^\infty \frac{\|\varphi^{\ell} q^\ell \pa_{v^\nu}^{\ell }H e^{W/2}\|_{L^2({\rm supp}\{ \chi^E\})}^2\lambda^{s\ell }}{(\ell !)^{2s}}\rg)^{\frac{1}{2}\sum_{\ell} m_\ell}\lesssim 2^{m+n-\mf{j}+\sum_{\ell=1}^{\mf{j}}m_\ell}\lesssim 2^{m+n}.\label{F_3}
\end{align}
Here we used the relation $s<s_0.$   
With this, we concludes 
{\bf Step \# 3}.

\noindent
{\bf Step \# 4: Conclusion.}
Combining all the estimates developed so far, we obtain that
\begin{align*}
&\sum_{m+n=0}^\infty \frac{\lambda_1^{2s_0(m+n)}}{(m+n)!^{2{s_0}}}\varphi^{2n}\|e^{W/2}\chi^E |k|^m  q^{2n}\Gamma_{t;0}^n \omega_k^{\nu}\|_2^2\\
&\lesssim\sum_{m+n\geq 0}\lf({\lf(\frac{8\lambda_1}{ \lambda}\rg)^{s_0(m+n)}}{\sum_{\mf{j}=0}^n\sum_{\sum_{\ell=1}^{\mf{j}} \ell m_\ell ={\mf{j}}}}\lf(\sum_{p=0}^{\sum_{\ell=1}^{\mf{j}}m_\ell} 
\frac{ (\lambda/2)^{s_0(m+n)}\|e^{W/2} \chi^E|k|^{m+n-\mf{j}+\sum_{\ell=1}^{\mf{j}}m_\ell-p}\varphi^p q^p\Gamma_{t;\nu}^{p}\omega_{k}^{\nu} \|_{L^2 }}{(m+n-\mf{j}+\sum_{\ell=1}^{\mf{j}}m_\ell)!^{s_0}}\rg)\rg)^2\\
&\lesssim\sum_{m+n\geq 0}\bigg(\lf(\frac{8\lambda_1}{ \lambda}\rg)^{s_0(m+n)}{\sum_{\mf{j}=0}^n\sum_{\sum_{\ell=1}^{\mf{j}} \ell m_\ell =\mf{j}}}\underbrace{ \lf(\sum_{m'+n'=0}^{\infty} 
\frac{ \lambda^{2s_0(m'+n')}\|e^{W/2}\chi^E|k|^{m'}\varphi^{n'} q^{n'}\Gamma_{t;\nu}^{n'}\omega_{k}^{\nu} \|_{L^2}^2}{(m'+n')!^{2{s_0}}}\rg)^{1/2}}_{\mathcal{I}}\bigg)^2.
\end{align*}
Here we observe that the $\mathcal{I}$ can be easily bounded by $\sqrt{\mathcal{E}^{(\gamma)}}$. However, we are summing a lot of copies of $\mathcal I$. Here we need to estimate the total number of tuples $(m_1,m_2,\cdots, m_\ell)$ such that $\sum_\ell \ell m_\ell = \mf{j}$. If this number is exponentially large in terms of $m+n$, we can pick $\lambda_1$ small enough to compete with it. The following simple combinatorial argument tells us that  the total number of tuples is  $\lesssim 5^{\mf{j}}$.\footnote{
The author learned this combinatorics argument from Ruth Luo.
} To begin with, we observe the total number of valid tuples is less than the number of tuples such that $\sum_{\ell=1}^{\mf{j}}m_\ell\leq \mf{j}$. Assume that the sum $\sum_{\ell=1}^{\mf{j}}m_\ell=\widetilde j(\leq \mf{j})$. We can imagine that there are $\widetilde{j}$ `stars' and we are trying to separate them with $\mf{j}-1$ `bars'. This is called the `star and bar' problem in combinatorics. To solve it, we observe that there are $\widetilde{j}\text{ (`star') }+(\mf{j}-1)\text{ (`bar') }$ objects in total. As long as we determine the positions for the $\mf{j}-1$  `bars' among these $\widetilde{j}+\mf{j}-1$ objects, the tuple $(m_1,m_2,\cdots, m_\ell)$ is determined and vice versa. Hence, the problem is equivalent to picking $\mf{j}-1$ objects from $\widetilde{j}+\mf{j}-1$ objects. 
Hence the total number of tuples is $\binom{\widetilde j+\mf{j}-1}{\mf{j}-1}\leq 2^{\mf{j}+\wt j}$. Since the number $\widetilde j$ is bounded by $\mf{j}$, we have at most $\mf{j} 4^{\mf{j}}\leq C 5^{\mf{j}}$ tuples. Hence,  there exists a constant $\mathfrak C$ such that the above sum is bounded by  
\begin{align*}
\sum_{m+n=0}^\infty &\frac{\lambda_1^{2s_0(m+n)}}{(m+n)!^{2{s_0}}}\varphi^{2n}\|e^{W/2}\chi^E |k|^m  q^{2n}\Gamma_{t;0}^n \omega_k^{\nu}\|_2^2\\
\lesssim &\sum_{m+n\geq 0}\bigg((\mathfrak C\lambda_1 \lambda^{-1})^{s_0(m+n)}\lf(\sum_{m'+n'=0}^{\infty} 
\frac{ \lambda^{2s_0(m'+n')}\|e^{W/2}\chi^E|k|^{m'}\varphi^{n'} q^{n'}\Gamma_{t;\nu}^{n'}\omega_{k}^{\nu} \|_{L^2}^2}{(m'+n')!^{2{s_0}}}\rg)^{1/2}\bigg)^2.
\end{align*}
Hence if  $\lambda_1=C\lambda_\ast^{\frac{1}{rs_0}}$ is small enough, we have
\begin{align*}
\sum_{m+n\geq 0}&\frac{\lambda_1^{2s_0(m+n)}}{(m+n)!^{2{s_0}}}\varphi^{2n}\|e^{W/2}\chi^E |k|^m  q^{2n}\Gamma_{t;0}^n \omega_k^{\nu}\|_2^2\lesssim\sum_{m+n\geq0}\frac{\lambda^{2s(m+n)} \|\chi^E|k|^{m}\varphi^n q^n\Gamma_{t;\nu}^{n}\omega_{k}^{\nu}e^{W/2} \|_{L^2 }^2}{(m+n)!^{2{s}}}.
\end{align*} 
This implies \eqref{ext_dif}. Hence the proof is finished.
\end{proof}

The following combinatorial lemma is needed when computing the Gevrey norm:
\begin{lemma}The following bound holds
\begin{align}\label{combn}
\frac{(1!)^{m_1}(2!)^{m_2}\cdots(n!)^{m_n}(\sum_{j=1}^n m_j)!}{(\sum_{j=1}^n j m_j)!}\leq 1.
\end{align}
Here $\sum_{j=1}^n jm_j=n$. 
\end{lemma} 
\begin{proof}
We start by estimating the left hand side of  \eqref{combn}. The expression can be reorganized  in the following way:
\begin{align}\n
\frac{(1!)^{m_1} m_1!}{m_1!}\times& \frac{(2!)^{m_2}(m_1+m_2)!/m_1!}{(m_1+2m_2)!/m_1!}\times \frac{(3!)^{m_3}(m_1+m_2+m_3)!/(m_1+m_2)!}{(m_1+2m_2+3m_3)!/(m_1+2m_2)!}\times \cdots\\
&\times \frac{(j!)^{m_j}(\sum_{i=1}^jm_i)!/(\sum_{i=1}^{j-1}m_i)!}{(\sum_{i=1}^j im_i)!/(\sum_{i=1}^{j-1}im_i)!}\cdots =:\prod_{j=1}^n \mathfrak{F}_j.\label{dfn_Fj}
\end{align}
Now we prove that the factors $\mathfrak F_j$ are less than $1$.
For general $\mathfrak F_k, k\in\{1,2,\cdots n\}$,
\begin{align}\n
\mathfrak F_k&=\frac{\prod_{j=1}^{m_k}( k! {(\sum_{i=1}^{k-1}m_i+j)})}{\prod_{j=1}^{m_k}\lf[{(\sum_{i=1}^{k-1}i m_i+k(j-1)+1)}{(\sum_{i=1}^{k-1}im_i+k(j-1)+2)}\cdots(\sum_{i=1}^{k-1}im_i+kj)\rg]}\\
&\leq\prod_{j=1}^{m_k} \frac{2
\times 3\times \cdots  \times k }{{(\sum_{i=1}^{k-1}im_i+k(j-1)+2)}\times(\sum_{i=1}^{k-1}im_i+k(j-1)+3)\times\cdots\times(\sum_{i=1}^{k-1}im_i+kj)}.\label{F_k_dcmp}
\end{align}
We further observe that 
\begin{align*}
\quad \ell\leq& \lf(\sum_{i=1}^{k-1}im_i+k(j-1)+\ell\rg),\quad j\geq 1,\,  k\geq 0.
\end{align*}
Hence all the factors in $\mathfrak F_k$ \eqref{F_k_dcmp} is less than $1$ and $\mathfrak F_k\leq 1$. This, when combined with \eqref{dfn_Fj}, yields the result \eqref{combn}. 

\ifx
\siming{ To motivate the ideas, we focus on the third factor
{\scriptsize\begin{align*} 
&F_3=\frac{\prod_{j=1}^{m_3} 3! {(m_1+m_2+j)}}{\prod_{j=1}^{m_3}{(m_1+2m_2+3(j-1)+1)}{(m_1+2m_2+3(j-1)+2)}{(m_1+2m_2+3j)}}\\
& =
\frac{(3\cdot 2\cdot 1) {(m_1+m_2+1)}(3\cdot 2\cdot 1) \cdots(3\cdot 2\cdot 1){(m_1+m_2+j)}\cdots}{{(m_1+2m_2+1)} {(m_1+2m_2+2)} {(m_1+2m_2+3)}\cdots{(m_1+2m_2+3(j-1)+1)}{(m_1+2m_2+3(j-1)+2)}{(m_1+2m_2+3j)}\cdots},\\
&\qquad\qquad\qquad j\in\{1,2,\cdots,m_3\}.
\end{align*} }
Since we have
\begin{align*}
\frac{m_1+m_2+j}{m_1+2m_2+3(j-1)+1}\leq 1,\quad j\geq 1,
\end{align*}
the factor $\mathfrak F_3$ can be estimated as follows
\begin{align*}
\mathfrak F_3\leq \frac{\prod_{j=1}^{m_3} 3\cdot 2 }{\prod_{j=1}^{m_3}{(m_1+2m_2+3(j-1)+2)}{(m_1+2m_2+3j)}}.
\end{align*} 
Now we have that 
\begin{align*}
2\leq (m_1+2m_2+3(j-1)+2),\quad 3\leq (m_1+2m_2+3j), \quad j\geq1.
\end{align*}
As a result, $\mathfrak F_3\leq1.$}
\fi
\end{proof}

\ifx
\begin{lemma} \label{lem:ExtToInt_2}
Consider $s>1,\, r<1,\ t\in[0, \nu^{-1/3+\eta}]$. Then, for $t=\nu^{-1/3+\eta}$  
\begin{align} \sum_{m+n=0}^\infty& \frac{ \lambda_0^{2(m+n) }  }{ ((m+n)!)^{2/r}}\lf\|e^{W/8}\siming{q^{2n}}
|k|^m \Gamma_{t;0}^n f_k\rg\|_{L^2}^2 \lesssim_{s,r} \sum_{m+n=0}^\infty\frac{ \lambda^{2(m+n)} \varphi^{4n}}{((m+n)!)^{2s}}\|e^{W/4}\chi_{E}|k|^m q^{2n} \Gamma_{t;0}^n f_k \|_{L^2}^2.\label{glu_rl} 
\end{align} 
 Here $r>\frac{1}{2}$ if $s-1<\frac{3\eta}{2/3+\eta}$. Moreover, $\lambda_0$ depends on $\lambda_\ast. $ 
\end{lemma}
\begin{proof}
On the time interval  $t\in[0,\nu^{-1/3+\eta}]$, we have the following estimate involving $e^{W}$ weight \eqref{defndW} and $\varphi$ \eqref{varphi}:
\begin{align*}
e^{W/4}\varphi^{4n}\geq &e^{W/4}\varphi^{4(m+n)}\geq \frac{1}{C}e^{W/8}e^{\frac{1}{80K\nu (1+t)}}(1+t^2)^{-2m-2n}\\
\geq & \frac{e^{W/8}}{C} e^{\frac{1}{100K\nu(1+t)}}\ \frac{1}{N!}\frac{1}{(\mathcal G' K)^N\nu^{N} (1+t)^{N}}\  \frac{1}{(1+t^2)^{ 2m+2n }}\\
\geq & \frac{e^{W/8}}{C}e^{\frac{1}{200K\nu^{2/3+\eta}}}\frac{1}{N!}\frac{1}{\mathcal{G}^N2^{m+n}}\nu^{-\lf(\frac{2}{3}+\eta\rg)N+\lf(\frac{4}{3}-4\eta\rg)(m+n) },
\end{align*}
where $\mathcal{G}>1$ is a constant depending on the parameter $K$ in the weight $W$.
By setting $N=\lf\lfloor\frac{4/3-4\eta}{2/3+\eta}(m+n)\rg\rfloor$, we have
\begin{align*}
e^{W/4}\varphi^{4(m+n)}
\geq\frac{e^{W/8}}{C\mathcal{G}^{\lf\lfloor\frac{4/3-4\eta}{2/3+\eta}(m+n)\rg\rfloor}2^{m+n}\lf\lfloor \frac{4/3-4\eta}{2/3+\eta}(m+n)\rg\rfloor!}. 
\end{align*}
Hence,\begin{align}\sum_{m+n=0}^\infty&\frac{\lambda^{2(m+n)}\varphi^{4n}}{((m+n)!)^{2s}}\|e^{W}\chi_E|k|^m q^n \Gamma_{t;0}^n f_k \|_{L^2}^2\n \\
\gtrsim& \sum_{ m+n=0}^\infty \frac{\lambda^{2(m+n)} }{\mathcal{G}^{\lf\lfloor\frac{4/3-4\eta}{2/3+\eta}(m+n)\rg\rfloor}2^{m+n}((m+n)!)^{2s}\lf\lfloor{\frac{4/3-4\eta}{2/3+\eta}(m+n)}\rg\rfloor!}\|\chi_{E}|k|^m q^{2n} \Gamma_{t;0}^n f_k\|_{L^2}^2. \label{Gmm_fnc}
\end{align}\myr{
The Gevrey index $r$ is roughly:
\begin{align}
\frac{1}{r}=s+\frac{2/3-2\eta}{2/3+\eta}\Rightarrow r=\frac{1}{2+(s-1-\frac{3\eta}{2/3+\eta})}.
\end{align}
So $r>1/2$ if 
\begin{align}\label{f}
s-  1<\frac{3\eta}{2/3+\eta}.
\end{align}}



To estimate the right hand side of \eqref{Gmm_fnc}, we use the Gamma function $\Gamma(n)=(n-1)!,\quad n\in \mathbb{N}\backslash \{0\}$ and the log convexity of the Gamma function:
\begin{align*}
\Gamma(\theta x_1+(1-\theta)x_2)\leq \Gamma(x_1)^{\theta}\Gamma(x_2)^{1-\theta},\quad \theta\in[0,1], x_1,x_2>0.
\end{align*}
We have that
\begin{align}
\lf\lfloor\frac{4/3-4\eta}{2/3+\eta}(m+n)\rg\rfloor=\theta(m+n)+2(1-\theta)(m+n)\quad\Rightarrow\quad \theta=\frac{1}{m+n}\lf\lceil\frac{6\eta}{2/3+\eta}(m+n)\rg\rceil.
\end{align}
\begin{align*}
\left\lfloor \frac{4/3-4\eta}{2/3+\eta}(m+n)\right\rfloor!=&\Gamma\lf(\lf\lfloor \frac{4/3-4\eta}{2/3+\eta}(m+n)\rg\rfloor+1\rg)\\
\leq& \Gamma(m+n+1)^{\theta}\Gamma(2(m+n)+1)^{1-\theta},\quad \theta = \frac{1}{m+n}\lf\lceil\frac{6\eta}{2/3+\eta}(m+n)\rg\rceil.
\end{align*}
Now we estimate the $\Gamma(2(m+n)+1)=(2(m+n))!$:
\begin{align*}
(2(m+n))!=\prod_{\ell_1=1}^{(m+n)}(2\ell_1)\prod_{\ell_2=0}^{(m+n-1)}(2\ell_2+1)\leq 2^{2(m+n)}((m+n)!)^2.
\end{align*}
Hence,
\begin{align*}
\left\lfloor \frac{4/3-4\eta}{2/3+\eta}(m+n)\right\rfloor!
\leq & 2^{2(m+n)(1-\theta)}((m+n)!)^{2-\theta}.
\end{align*}
Therefore, the right hand side of  \eqref{Gmm_fnc} has the following lower bound,
\begin{align*}
  \sum_{ m+n=0}^\infty &\frac{\lambda^{2(m+n)} }{\mathcal{K}^{2(m+n)}\lf((m+n)!\rg)^{2s+2-\theta}}\||k|^m q^{2n}\Gamma_{t;0}^nf_k\|_{L^2}^2.
\end{align*}
Hence the resulting Gevrey index $r$ is 
$$\frac{2}{r}=\lf(2s+2-\frac{1}{m+n}\lf\lceil\frac{6\eta}{2/3+\eta}(m+n)\rg\rceil\rg)\underbrace{\Rightarrow}_\eqref{f} r>1/2,!
$$ 
which completes the proof. 
\end{proof}

\fi

\subsection{Inviscid limit in the interior}

%
In what follows, denote
\begin{align*}
\mathcal{E}^{il}_{\mathrm{ext}} := \sum_{ m+n=0 }^\infty \frac{\lambda_1^{2s_0(m+n)} \varphi^{2n}}{((m+n)!)^{2s_0}} \norm{ \chi^E e^{W/2} \partial_x^{m} q^{2n} \Gamma_{t;0}^n (\omega^\nu -\omega^0)}_{L_{x,y}^2}^2\n
\end{align*}
which we know is a priori small from Theorem \ref{thm_bdy_lim}.

We next concern ourselves with estimating the vorticity in the interior. 
For the vorticity profile written in the $(z,v^0)$ coordinates (where we are denoting $v' := 1+h^0$) we have 
\begin{align*}
& \partial_t f^0 + g^0 \partial_v f^0 + v' \grad^\perp \psi_{\neq}^0 \cdot \grad f^0 = 0 \\
& \partial_t f^\nu + g^0 \partial_v f^\nu + v'\partial_v (\psi_{0}^0 - \psi_0^\nu) \partial_z f^\nu +  v' \grad^\perp \psi_{\neq}^\nu \cdot \grad f^\nu = \nu \Delta_t f^\nu,
\end{align*}
and hence for $\tilde{f} = f^\nu - f^0$ we have (denoting $\tilde{f}^I := \chi^I \tilde{f}$)  
\begin{align*}
& \partial_t \tilde{f}^I + g^0 \partial_v \tilde{f}^I + v' \chi^I \grad^\perp \tilde{\phi}_{\neq} \cdot \grad f^0 + v' \grad^\perp \psi^0_{\neq} \cdot \grad \tilde{f}^I + v' \grad^\perp \tilde{\psi}_{\neq} \cdot \grad \tilde{f}^I \\ 
& \quad  = \nu \Delta_t \tilde{f}^I + \nu \chi^I \Delta_t f^0 - v'\partial_v (\psi_{0}^0 - \psi_0^\nu) \partial_z (\tilde{f}^I + \chi^I f^0) + \mathcal{C}_f, 
\end{align*}
where
\begin{align*}
\mathcal{C}_f & = -g^0 \partial_v\chi^I \tilde{f} + v' \partial_z \psi^0_{\neq} \partial_v\chi^I \tilde{f} + v' \partial_z \tilde{\psi}_{\neq} \partial_v\chi^I \tilde{f} + \nu [\chi^I, \Delta_t] \tilde{f}. 
\end{align*}
An important detail is that we have good a priori estimates on $f^0$ in this coordinate system, but we are actually lacking good a priori estimates on $f^\nu$. 

We now re-define $A$ to be slightly weaker
\begin{align*}
A^{il}(t,k,\eta) & := e^{\tilde{\lambda}(t)\abs{\grad}^s}\left(\frac{e^{\mu \abs{\eta}^{1/2}} }{w(t,k,\eta)} + e^{\mu \abs{k}^{1/2}}\right), 
\end{align*}
where (assuming without loss of generality that we will halt the estimates at $\nu t^3 \leq K$), 
\begin{align*}
\tilde{\lambda}(t) = \left(\frac{1}{4} + \frac{1}{4\brak{t}^{\alpha}} - \frac{1}{8 K}\nu t^3\right)\lambda_\infty. 
\end{align*}
We then define a new ``interior'' energy  for some $\delta > 0$
\begin{align*}
\mathcal{E}^{il}(t) := t^{-4-2\delta} \norm{\brak{\partial_v}^{-1} A^{il} \tilde{f}^I_0}^2_{L^2} + t^{-6-2\delta} \norm{A^{il} \tilde{f}^I}_{L^2}^2. 
\end{align*}
\begin{lemma}
For $\nu t^{3+\delta} \leq 1$, there holds the uniform estimate
\begin{align*}
\mathcal{E}^{il}(t) \lesssim_\delta \nu^2 \eps^2
\end{align*}
yielding the a priori estimates
\begin{align*}
\norm{A^{il} \tilde{f}^I}_{L^2} & \lesssim \nu \eps t^{3+\delta} \\
\norm{A^{il}\tilde{U}^I_0}_{L^2} & \lesssim \nu \eps t^{2+\delta}.
\end{align*}
\end{lemma}
Without loss of generality we can assume $t \geq 1$ (see \cite{MasRou12} for short time inviscid limits with Navier boundary conditions and \cite{BM13} for discussions on the initial time layer in the ensuing nonlinear energy estimates).

We argue as a bootstrap, that is we assume that over $t \in [1,T_\ast]$ there holds 
\begin{align*}
\mathcal{E}^{il}(t) \leq 4K_0\nu^2 \eps^2, 
\end{align*}
for some suitable $K_0 \geq 1$ (independent of $T_\ast$). We then show that for $\eps$ sufficiently small, $\nu t^{3+\delta} \leq 1$, and $K_0$ sufficiently large (independent of $t,\nu$), there holds over the same time-interval
\begin{align*}
\mathcal{E}^{il}(t) \leq 2K_0\nu^2 \eps^2, 
\end{align*}
which by continuity in time (due to the qualitative a priori regularity) implies the desired result. 

We begin the proof as follows
\begin{align*}
\frac{d}{dt} \mathcal{E}^{il} & = -\frac{4 + 2\delta}{t} t^{-4-2\delta} \norm{\brak{\partial_v}^{-1} A^{il} \tilde{f}^I_0}^2 - \frac{6+2\delta}{t}t^{-6-2\delta} \norm{A^{il} \tilde{f}^I}_{L^2}^2 + CK_\lambda + CK_w \\ 
& \quad - \nu t^{-4-2\delta}\norm{\grad_L \brak{\partial_v}^{-1} A^{il} \tilde{f}^I_0}_{L^2}^2 - \nu t^{-6-2\delta}\norm{\grad_L A^{il} \tilde{f}^I}_{L^2}^2 \\ 
& \quad + \mathcal{D} - L_{\mathcal{S}} + L_{\neq} + NL_{\neq} + L_0 + NL_0 \\ 
& \quad + t^{-6-2\delta}\brak{A^{il}\tilde{f}^I, A^{il}\nu \chi^I \Delta_t f^0 } + t^{-4-2\delta}\brak{\brak{\partial_v}^{-1} A^{il}\tilde{f}^I_0, \brak{\partial_v}^{-1} A^{il }\nu \chi^I (v' \partial_v)^2 f^0_0}.  
\end{align*}
where
\begin{align*}
CK_\lambda & = t^{-6-2\delta} \norm{\sqrt{\frac{\partial_t w}{w}} \widetilde{A}^{il} \tilde{f}^I }_{L^2}^2 + t^{-4-2\delta} \norm{\sqrt{\frac{\partial_t w}{w}} \brak{\partial_v}^{-1} A^{il} \tilde{f}^I_0 }_{L^2}^2 \\
CK_\lambda & = t^{-6-2\delta} \dot{\tilde{\lambda}} \norm{\abs{\grad}^{r/2} A^{il} \tilde{f}^I }_{L^2}^2 + t^{-4-2\delta} \dot{\tilde{\lambda}} \norm{\abs{\partial_v}^{r/2} \brak{\partial_v}^{-1} A^{il} \tilde{f}^I_0 }_{L^2}^2 \\
\mathcal{D} & = \nu t^{-6-2\delta}\brak{A^{il} \tilde{f}^I, A^{il} ((\Delta_t - \Delta_L) \tilde{f}^I)} +  \nu t^{-4-2\delta}\brak{A^{il} \brak{\partial_v}^{-1}\tilde{f}^I, \brak{\partial_v}^{-1} A^{il} ((\Delta_t - \Delta_L) \tilde{f}^I_0)} \\
\mathcal{S} & =  t^{-6-2\delta} \brak{A^{il} \tilde{f}^I, A^{il}( \chi^I v'\partial_v\tilde{\psi}_0\partial_z f^0) } + t^{-6-2\delta} \brak{A^{il} \tilde{f}^I, A^{il}( v'\partial_v\tilde{\psi}_0 \partial_z \tilde{f}^I) } \\ 
L_{\neq} & = t^{-6-2\delta}\brak{A^{il} \tilde{f}^I, A^{il} \left( v' \chi^I \grad^\perp \tilde{\psi}_{\neq} \cdot \grad f^0 + v' \grad^\perp \psi^0_{\neq} \cdot \grad \tilde{f}^I\right) } \\
NL_{\neq} & = t^{-6-2\delta}\brak{A^{il} \tilde{f}^I, A^{il} \left( v' \grad^\perp \tilde{\psi}_{\neq} \cdot \grad \tilde{f}^I\right)} \\
L_0 & = t^{-4-2\delta}\brak{\brak{\partial_v}^{-1} A^{il} \tilde{f}^I_0, \brak{\partial_v}^{-1} A^{il} P_0 \left( v' \chi^I \grad^\perp \tilde{\psi}_{\neq} \cdot \grad f^0 + v' \grad^\perp \psi^0_{\neq} \cdot \grad \tilde{f}^I\right) } \\
NL_0 & = t^{-4-2\delta}\brak{\brak{\partial_v}^{-1} A^{il} \tilde{f}^I, \brak{\partial_v}^{-1} A^{il} P_0 \left( v' \grad^\perp \tilde{\psi}_{\neq} \cdot \grad \tilde{f}^I\right)} \\
\mathcal{C} & = t^{-4-2\delta}\brak{ \brak{\partial_v}^{-1} A^{il} \tilde{f}^I, \brak{\partial_v}^{-1} A^{il} P_0\mathcal{C}_f } + t^{-6-2\delta}\brak{ A^{il} \tilde{f}^I, A^{il}\mathcal{C}_f }.
\end{align*}
The first baseline growth estimates will come from the source terms. Due to the a priori estimates available from \cite{HI20}, for any $\eta > 0$, there is a constant $C(\eta)$ such that 
\begin{align*}
t^{-6-2\delta}\brak{A^{il}\tilde{f}^I, \nu A^{il} \chi^I \Delta_t f^0 } & \lesssim t^{-6-2\delta} \norm{A^{il} \tilde{f}^I}_{L^2} \nu t^2 \norm{\chi_e^I f^0}_{\mathcal{G}^{2\tilde{\lambda},r}} \\ 
& \leq \frac{\eta}{100} t^{-7-2\delta} \norm{A^{il} \tilde{f}^I}_{L^2}^2 + \frac{1}{t^{1+2\delta}}C(\eta) \nu^2 \eps^2 \\ 
t^{-4-2\delta}\brak{\brak{\partial_v}^{-1} A^{il}\tilde{f}^I_0 , \nu \brak{\partial_v}^{-1} A^{il} \chi^I (v' \partial_v)^2 f^0_0} & \lesssim \nu t^{-4-2\delta} \norm{\brak{\partial_v}^{-1} A^{il}\tilde{f}^I_0}_{L^2} \norm{\chi_e^I f^0}_{\mathcal{G}^{2\tilde{\lambda},r}} \\
& \leq \frac{\eta}{100} t^{-5-2\delta} \norm{\brak{\partial_v}^{-1} A^{il}\tilde{f}^I_0}_{L^2}^2 + \frac{C(\eta)}{t^{3+2\delta}} \nu^2 \eps^2. 
\end{align*}
Note that the second estimate formally suggests an error estimate of the form $\norm{\brak{\partial_v}^{-1}A^{il} \tilde{f}^I_0}_{L^2} \lesssim \nu \eps t $ might be available, however, it turns out that the linear terms coupling $\tilde{f}_0$ and $\tilde{f}_{\neq}$ seem to dictate a faster growth for the $\tilde{f}^I_0$ estimate. 

Consider next the linear term in $\mathcal{S}$, which involves the `shear' part of the velocity field. 
As in the methods of \cite{BM13,HI20}, we need to treat the fact that $\tilde{U}$ is measured without resonant regularity corrections and this introduces losses into the product rule which must be absorbed by the $CK$ terms along with the loss of a power of $t$: 
\begin{align*}
t^{-6-2\delta} \abs{\brak{A^{il} \tilde{f}^I , A^{il}( \chi^I v'\partial_v\tilde{\psi}_0 \partial_z f^0) }} & \lesssim t^{-2-\delta}\norm{\chi^I_e f^0}_{\mathcal{G}^{2\tilde\lambda,r}} (CK_w + CK_\lambda)^{1/2} \\ & \quad\quad  \times\left( \norm{ \sqrt{\frac{\partial_t w}{w}} A^{il} \partial_v\tilde{\psi}_0 }_{L^2}^2 - \dot{\tilde{\lambda}}\norm{ \abs{\grad}^{r/2} A^{il} \partial_v\tilde{\psi}_0 }_{L^2}^2\right)^{1/2} \\ & \quad + t^{-6-2\delta}\norm{\chi^I_e f^0}_{\mathcal{G}^{2\tilde\lambda,r}} \norm{A^{il} \tilde{f}^I}_{L^2} \norm{ A^{il} \partial_v\tilde{\psi}_0 }_{L^2}. 
\end{align*}
By an elliptic regularity estimate similar to those done in \cite{HI20} and Lemma \ref{lem:phiE}, we obtain 
\begin{align*}
t^{-6-2\delta} \abs{\brak{A^{il} \tilde{f}^I , A^{il}( \chi^I v'\partial_v\tilde{\psi}_0 \partial_z f^0) }} & \lesssim \eps CK^{1/2} (CK + \nu^{1000} \mathcal{E}^{il}_{\mathrm{ext}})^{1/2} \\ & \quad + \frac{\eps}{t^{7+2\delta}}\norm{A^{il} \tilde{f}^I}_{L^2}^2 + \frac{\eps }{t^{5+2\delta}} \left( \norm{\brak{\partial_v}^{-1} A^{il} \tilde{f}^I_0}_{L^2}^2 + \nu^{1000} \mathcal{E}^{il}_{\mathrm{ext}} \right), 
\end{align*}
which are then absorbed by the negative-definite terms in the energy estimate for $\eps$ sufficiently small or simply integrated. 
Note in the above we also used the a priori estimate $\norm{\chi_e^I h^0}_{\mathcal{G}^{2\tilde\lambda,r}} \lesssim \eps$. 

For the nonlinear variation of this term we need to introduce a commutator
\begin{align}
t^{-6-2\delta} \brak{A^{il} \tilde{f}^I, A^{il}( v'\partial_v\tilde{\psi}_0 \partial_z \tilde{f}^I) } & = t^{-6-2\delta} \brak{A^{il} \tilde{f}^I, [A^{il}, \chi_e^I v'\partial_v\tilde{\psi}_0 \partial_z] \tilde{f}^I) }. \label{eq:CommAil} 
\end{align}
The commutator in \eqref{eq:CommAil} needs to be split into `transport, reaction, and remainder' terms as in \cite{BM13}.
Note that since $\tilde{U}_0$ is zero-frequency in $z$, the main difficulty in the `transport' contribution does not appear.
The main loss that could appear is in the `reaction' contribution, when the leading factor of $A^{il}\tilde{f}$ is at `resonant' regularity (a loss because $\partial_v \psi_0$ is always in `non-resonant' regularity). At this point a power of $t$ is lost.
Reasoning as in \cite{BM13,HI20} therefore yields the following estimate for some $c \in (0,1)$, 
\begin{align*}
t^{-6-2\delta} \brak{A^{il} \tilde{f}^I, [A^{il}, \chi_e^I v'\partial_v \tilde{\psi}_0 \partial_z] \tilde{f}^I) } & \lesssim t^{-6-2\delta} \norm{\chi_e^I \tilde{U}_0}_{\mathcal{G}^{c\tilde{\lambda},s}} \norm{\abs{\grad}^{r/2} A^{il} \tilde{f}^I}_{L^2}^2 \\
& \quad + \norm{\tilde{f}}_{\mathcal{G}^{c\tilde{\lambda},s}} (CK_w + CK_\lambda + \nu^{1000} \mathcal{E}_{\mathrm{ext}}^{il}) \\ 
& \quad + t^{-6-2\delta} \norm{\chi_e^I \tilde{f}}_{\mathcal{G}^{c\tilde{\lambda},s}} \norm{A^{il}\tilde{f}}_{L^2} \norm{A^{il}\partial_v \tilde{\psi}_0}_{L^2} \\
& \lesssim (\eps \nu t^{3+\delta}(CK_w + CK_\lambda + \nu^{1000} \mathcal{E}_{\mathrm{ext}}^{il}) \\  
& \quad + t^{-6-2\delta} (\eps \nu t^{3+\delta}) \norm{A^{il}\tilde{f}}_{L^2} \left( \norm{\brak{\partial_v}^{-1}A^{il}\tilde{f}_0}_{L^2} + \nu^{1000}(\mathcal{E}_{\mathrm{ext}}^{il})^{1/2} \right), 
\end{align*}
where in the last line we used the bootstrap hypotheses. 

There are two other types of linear terms, $L_{\neq}$ and $L_0$.
One of the crucial steps for making the self-consistent argument work is $L_0$, given by
\begin{align*}
L_0 & = t^{-4-2\delta}\brak{\brak{\partial_v}^{-1} A^{il} \tilde{f}^I_0, \brak{\partial_v}^{-1} A^{il} P_0 \left( v' \chi_I \grad^\perp \tilde{\psi}_{\neq} \cdot \grad f^0 + v' \grad^\perp \psi^0_{\neq} \cdot \grad \tilde{f}^I\right) } \\
& =: L_{0,1} + L_{0,2}. 
\end{align*}
Turn first to the estimate of $L_{0,1}$. 
First, we have that 
\begin{align*}
L_{0,1} & = t^{-4-2\delta} \brak{\brak{\partial_v}^{-1} A^{il}\tilde{f}^I_0, \brak{\partial_v}^{-1} A^{il} (1 + h^0)\left(\chi^I \grad^\perp \tilde{\psi}_{\neq} \cdot \grad f^0 \right)_0} \\ 
& = L_{0,1}^{'} + L_{0,1}^{h}.
\end{align*}
As in \cite{BM13,HI20}, the latter term $L_{0,1}^h$ is easily dealt with provided $L_{0,1}'$ can be treated (especially since $h^0$ has a significant amount of extra regularity available), and so we only deal with $L_{0,1}'$. 
A key detail here is that $\tilde{f}^I_0$ is always in non-resonant regularity.  
Subdividing in frequency,
\begin{align*}
L_{0,1}' & = t^{-4-2\delta} \mathrm{Re} \int_{\mathbb R} \int_{\mathbb R} \left(\mathbf{1}_{\abs{\eta-\xi} < \frac{\abs{\xi}}{2}} + \mathbf{1}_{\abs{\xi} < \frac{\abs{\eta-\xi}}{2}} + \mathbf{1}_{\abs{\eta-\xi} \approx \abs{\xi} } \right)  \brak{\eta}^{-1} A^{il} \widehat{\tilde{f}^I_0}(t,\eta) \\ & \qquad\qquad\qquad \times \brak{\eta}^{-1} A^{il}(t,\eta)i\eta \left( \widehat{\chi^I \grad^\perp \tilde\psi_{k}}(t,\xi) \widehat{\chi^I_e f^0}(-k,\eta-\xi)\right) \dee \eta \dee \eta \\
& =: L_{0,1}^{HL}  + L_{0,1}^{LH} + L_{0,1}^{R}. 
\end{align*}
From the methods of \cite{HI20}, it is straightforward to obtain estimates on the latter two terms due to the extra available regularity on $f^0$, leading to 
\begin{align*}
L_{U,1}^{LH} + L_{U,1}^{R} & \lesssim \frac{1}{\brak{t}^{6+2\delta}} \left(\norm{A^{il} \tilde{f}^I}_{L^2}  + \nu^{1000}(\mathcal{E}^{il}_{\mathrm{ext}})^{1/2} \right) \norm{A^{il} \tilde{U}}_{L^2}\norm{P_{\neq}f^0}_{\mathcal{G}^{2\tilde{\lambda},r}} \lesssim \frac{\eps}{t} \mathcal{E}^{il},
\end{align*}
which is absorbed by the negative definite term associated with the time-growth for $\eps \ll \delta$.
Note we used that the estimates on $\tilde{f}_0$ and $\tilde{f}_{\neq}$ are separated by at most one power of $t$. 

Turning to $L_{U,1}^{HL}$, we observe that that when the $z$-frequency of $\partial_z\tilde{\psi}_{\neq}$ is not too large, the situation is somewhat similar to the main `reaction' terms in \cite{BM13,HI20}, wherein we need to gain powers of $t$ from a combination of the design of $A$ (especially the $w$) where the ellipticity of $\Delta_t \tilde{\psi}$ fails (i.e. the critical time), while we need to gain optimal powers of $t$ from the ellipticity when it is available.
By carrying out essentially the same argument as in \cite{BM13,HI20}, we obtain a gain in $t^{-1}$ near the critical time in exchange for introducing $CK_w$ factors, while away from the critical time, we gain $t^{-2}$ making it possible to integrate using the negative-definite term associated with the time-growth (recall $CK_\lambda$ terms are also introduced due to errors and overlap regions in the frequency decompositions as in \cite{BM13,HI20}).
Ultimately, this yields
\begin{align*}
L_{U,1}^{HL} \lesssim \norm{f^0}_{\mathcal{G}^{2\tilde{\lambda},r}}\left((CK_w + CK_\lambda) + \frac{1}{\brak{t}}\mathcal{E}^{il} + \nu^{1000} \mathcal{E}^{il}_{\mathrm{ext}}\right);
\end{align*}
where the $\mathcal{E}^{il}_{ext}$ contribution arises through the Biot-Savart law. 
Next, it is straightforward to obtain the following estimate (note the regularity imbalance), 
\begin{align*}
L_{0,1} & = t^{-4-2\delta}\brak{\brak{\partial_v}^{-1} A^{il} \tilde{f}^I_0, \brak{\partial_v}^{-1} A^{il} P_0 \left( \chi^I v' \grad^\perp \psi^0_{\neq} \cdot \grad \tilde{f}^I\right) } \\
& \lesssim \frac{\eps}{t^{6+2\delta}}\norm{\brak{\partial_v}^{-1} A^{il} \tilde{f}^I_0}_{L^2} \norm{A^{il} \tilde{f}^I}_{L^2} \lesssim \frac{\eps}{t} \mathcal{E}^{il},
\end{align*}
which is then absorbed by the negative-definite term in the energy. 

Turn next to $L_f$. 
\begin{align*}
L_f & = t^{-6-2\delta} \brak{A^{il}\tilde{f}^I, A^{il} \chi^I v' \grad^\perp \tilde{\psi}_{\neq} \cdot \grad f^0} + t^{-6-2\delta} \brak{A^{il} \tilde{f}^I, A^{il} \chi^I v'\grad^\perp \psi^0_{\neq} \cdot \grad \tilde{f}}. 
\end{align*}
The first term is treated essentially like the `reaction' term in the nonlinear arguments of \cite{HI20,BM13}, whereas the latter term is treated as in the `transport' terms; we omit the treatment for brevity and simply provide the answer 
\begin{align*}
L_{f} \lesssim \eps(CK_w + CK_\lambda) + \nu^{1000}\mathcal{E}^{il}_{\mathrm{ext}} + \eps t^{-8-2\delta}\norm{A \tilde{f}}_{L^2}^2. 
\end{align*} 
Next, turn to the commutators $\mathcal{C}_f$. These are treated easily using Lemmas \ref{lem:ExtToInt} and \ref{lem:FourierToPhysical} (and the strong localization of the vorticity inherent in $\mathcal{E}^{il}_{\mathrm{ext}}$). 
First, let us consider $\mathcal{C}_0$, which gives the following (using the product rule in Lemma \ref{lem:BasicA})
\begin{align*}
t^{-4-2\delta}\brak{ \brak{\partial_v}^{-1} A^{il} \tilde{f}^I_0,  \brak{\partial_v}^{-1} A^{il}\left(v'\partial_z \psi_{\neq}^0 (\partial_v \chi^I) \tilde{f}^0 \right)_0} & \lesssim
\\
& \hspace{-7cm} t^{-4-2\delta}\norm{\brak{\partial_v}^{-1} A^{il} \tilde{f}^I_0}_{L^2} \norm{\chi^I_e \abs{\chi_I'}^{1/2} \partial_z \psi_{\neq}^0 }_{\mathcal{G}^{2\tilde{\lambda},r}} \norm{A^{il} \chi_e^I \abs{\chi_I'}^{1/2} \tilde{f}_{\neq} }_{L^2} (1 + \norm{\chi_e^I h^0}_{\mathcal{G}^{2\tilde{\lambda},r}}). 
\end{align*}
Since $\omega^0$ is supported away from the support of $\abs{\partial_v \chi^I}^{1/2}$, one has for all $N$,
\begin{align*}
\norm{\chi^I_e \abs{\chi_I'}^{1/2} \partial_z \psi_{\neq}^0 }_{\mathcal{G}^{2\tilde{\lambda},r}} \lesssim_N \frac{1}{\brak{t}^N} \eps. 
\end{align*}
While Lemmas \ref{lem:ExtToInt} and \ref{lem:FourierToPhysical} gives
\begin{align}
\norm{\brak{\partial_v}^{-1} A^{il} \chi_e^I \abs{\partial_v\chi^I}^{1/2} \tilde{f}_{\neq} }_{L^2} \lesssim t^{3+\delta} (\mathcal{E}^{il})^{1/2} + \nu^{1000}(\mathcal{E}^{il}_{\mathrm{ext}})^{1/2}. \label{ineq:TildeUneqLoc}
\end{align}
Hence, due to the additional inviscid damping, the contributions are easily integrated.
For the nonlinear version of this term, note that it is the second component of the velocity field that is arising (and hence the inviscid damping is faster and there is one less power of $t$ lost at the critical times), 
\begin{align*}
t^{-4-2\delta}\brak{\brak{\partial_v}^{-1} A^{il} \tilde{f}^I_0, \brak{\partial_v}^{-1} A^{il}\left(\partial_z\tilde{\psi}_{\neq} (\partial_v \chi^I) \tilde{f}_{\neq} \right)_0} & \lesssim \\ & \hspace{-6cm} t^{-4-2\delta} \norm{\brak{\partial_v}^{-1} A^{il} \tilde{f}^I_0}_{L^2} \norm{\brak{\partial_v}^{-1} A^{il} \chi_e^I \abs{\partial_v \chi^I}^{1/2} \partial_z \tilde{\psi}_{\neq}}_{L^2} \\  & \hspace{-6cm} \quad \times \norm{A^{il} \chi_e^I \abs{\partial_v \chi^I}^{1/2} \tilde{f}_{\neq}}_{L^2} (1 + \norm{\chi_e^I h^0}_{\mathcal{G}^{2\tilde{\lambda},r}}).  
\end{align*}
By a straightforward variation of Lemma \ref{lem:phiE} and Lemma \ref{lem:ExtToInt} (along with the elliptic estimates in \cite{HI20}) we have for any $N$, 
\begin{align*}
\norm{A^{il} \chi_e^I \abs{\partial_v \chi^I}^{1/2} \partial_z\tilde{\psi}_{\neq}}_{L^2} \lesssim \frac{1}{\brak{t}^N} (\mathcal{E}^{il})^{1/2} + \nu^{1000} (\mathcal{E}^{il}_{\mathrm{ext}})^{1/2}. 
\end{align*}
For the other factor, we may use \eqref{ineq:TildeUneqLoc} again.
The resulting terms are then integrated.
For the next term in the commutator we use that $A^{il}$ satisfies an algebra property when applied to functions with all the same frequency in $z$
\begin{align*}
t^{-4-2\delta}\brak{\brak{\partial_v}^{-1} A^{il} \tilde{f}^I_0, \brak{\partial_v}^{-1} A^{il}(g^0 (\partial_v\chi^I) \tilde{f}^I_0) } & \lesssim t^{-4-2\delta}\norm{\brak{\partial_v}^{-1} A^{il} \tilde{f}^I_0}\norm{\chi_e^I g^0}_{\mathcal{G}^{2\tilde{\lambda},r}} \\ & \quad \times  \norm{\brak{\partial_v}^{-1} A^{il} \chi_e^I \abs{\chi_I'}^{1/2} \tilde{f}_{0} }_{L^2} \\
& \lesssim \frac{\eps}{t^{2 - K\eps}} (\mathcal{E}^{il} + \nu^{1000}\mathcal{E}^{il}_{\mathrm{ext}}). 
\end{align*}
The corresponding term in the $\tilde{f}_{\neq}$ equation is treated similarly. 
It remains to treat the commutators with $\Delta_t$, for example the term (using all of the a priori estimates available on $v'$), 
\begin{align*}
\nu t^{-6-2\delta}\brak{ A^{il} \tilde{f}^I, A^{il} [\chi^I,\Delta_t] \tilde{f}^I) } & \lesssim \nu t^{-6-2\delta}\norm{A^{il} \tilde{f}^I} \left( \norm{ A^{il} \chi_e^I \abs{\chi_I''}^{1/2} \tilde{f}_{0} }_{L^2} + \norm{ A^{il} \chi_e^I \abs{\chi_I'}^{1/2} \grad_L \tilde{f}_{0} }_{L^2} \right) \\
& \lesssim \nu \mathcal{E}^{il} + \mathcal{D} + \nu^{1000} \mathcal{E}_{\mathrm{ext}}^{il};
\end{align*}
this completes the treatment of the commutator terms. 
Only the fully nonlinear terms $NL_{\neq}$ remain, and these are treated as in \cite{HI20,BM13} noting that the a priori estimates provide uniform estimates over $\nu t^{3 + \delta} \leq 1$.  

This completes the proof of the estimate $\mathcal{E}^{il} \lesssim \nu^2 \eps^2$, which together with Theorem \ref{thm_bdy_lim}, completes the proof of Part (iii) in Theorem \ref{thm:main}.


%
%
%
%
%
%




%
%
%
%
%
%


\vspace{2 mm}

\noindent \textbf{Acknowledgments:} JB was supported by NSF Award DMS-2108633. JB would also like to thank Ryan Arbon for helpful discussions. The work of  SH is supported in part by NSF grants DMS 2006660, DMS 2304392, DMS 2006372. SH would like to thank Ruth Luo for suggestions on combinatorics. The work of SI is supported by NSF DMS-2306528 and a UC Davis Society of Hellman Fellowship award. The work of FW is supported by the National Natural Science Foundation of China (No. 12101396, 12161141004, and 12331008).

\addcontentsline{toc}{section}{References}
\bibliographystyle{abbrv}
\bibliography{bibliography}

\end{document}